%% file: DissertationMarkfelderMain.tex
\documentclass[12pt,a4paper]{book}

\usepackage{graphicx}
\usepackage{subfig}
\usepackage{amsmath,amsthm,amssymb}
\usepackage{esint}
\usepackage{mathrsfs}

\usepackage{hyperref}

\usepackage[left=2.2cm,right=2.2cm,top=3cm,bottom=3.5cm]{geometry}

\usepackage{color}
\usepackage{array}

\theoremstyle{plain}
\newtheorem{thm}{Theorem}[section]
\newtheorem{prop}[thm]{Proposition}
\newtheorem{lemma}[thm]{Lemma}
\newtheorem{cor}[thm]{Corollary}
\theoremstyle{definition}
\newtheorem{defn}[thm]{Definition}
\newtheorem{ex}[thm]{Example}
\theoremstyle{remark}
\newtheorem*{rem}{Remark}

\numberwithin{equation}{chapter}
\numberwithin{figure}{chapter}

\usepackage{csquotes}
\usepackage[backend=bibtex8, style=numeric-verb, maxnames=5, sorting=nyt]{biblatex}
\renewcommand\labelenumi{(\alph{enumi})}
\renewcommand\theenumi\labelenumi

\newcommand{\R}{\mathbb{R}} 
\newcommand{\N}{\mathbb{N}}
\newcommand{\Z}{\mathbb{Z}} 
\newcommand{\C}{\mathbb{C}} 
\newcommand{\Q}{\mathbb{Q}}
\newcommand{\sphere}{\mathcal{S}} 

\newcommand{\Grad}{\nabla}  
\newcommand{\Hess}{\Grad^2}
\newcommand{\Div}{{\rm div}\,} 
\newcommand{\Divtx}{{\rm div}_{t,\vx}\,} 
\newcommand{\Divtxrot}{{\rm div}_{t^\rot,\vx^\rot}\,} 
\newcommand{\Lap}{\Delta} 

\newcommand{\parthree}[3]{\partial_{#1#2#3}^3}
\newcommand{\parfour}[4]{\partial_{#1#2#3#4}^4}

\newcommand{\dx}{\,{\rm d}\vx}
\newcommand{\dxcomp}{\,{\rm d}x}

\newcommand{\dt}{\,{\rm d}t}  
\newcommand{\ds}{\,{\rm d}s} 
\newcommand{\dsigma}{\,{\rm d}\sigma} 
\newcommand{\dtau}{\,{\rm d}\tau} 
\newcommand{\dr}{\,{\rm d}r} 
\newcommand{\dS}{\,{\rm d}S}

\renewcommand{\vec}[1]{{\bf #1}}
\newcommand{\vx}{\vec{x}}
\newcommand{\vy}{\vec{y}} 
\newcommand{\vz}{\vec{0}}
\newcommand{\vn}{\vec{n}}
\newcommand{\va}{\vec{a}}
\newcommand{\vb}{\vec{b}} 
\newcommand{\vd}{\vec{d}}
\newcommand{\ve}{\vec{e}}
\newcommand{\vf}{\vec{f}}
 
\newcommand{\vp}{\vec{p}}
\newcommand{\vq}{\vec{q}}
\newcommand{\vs}{\vec{s}}
\newcommand{\vnu}{\boldsymbol{\nu}}
\newcommand{\veta}{\boldsymbol{\eta}}
\newcommand{\valpha}{\boldsymbol{\alpha}}
\newcommand{\vR}{\vec{R}}
\newcommand{\vF}{\vec{F}}

\newcommand{\vu}{\vec{u}} 
\newcommand{\vm}{\vec{m}} 
\newcommand{\vv}{\vec{v}} 
\newcommand{\vU}{\vec{U}} 
\newcommand{\vw}{\vec{w}} 
\newcommand{\vphi}{\boldsymbol{\varphi}} 
\newcommand{\vpsi}{\boldsymbol{\psi}}

\newcommand{\mU}{\mathbb{U}} 
\newcommand{\mF}{\mathbb{F}}
\newcommand{\id}{\mathbb{I}}
\newcommand{\mA}{\mathbb{A}}
\newcommand{\mB}{\mathbb{B}}
\newcommand{\mM}{\mathbb{M}}
\newcommand{\mT}{\mathbb{T}}
\newcommand{\mZ}{\mathbb{O}} 

\newcommand{\gl}[1]{{\rm GL}(#1)}
\newcommand{\orth}[1]{{\rm O}(#1)}
\newcommand{\sym}[1]{{\rm Sym}(#1)}
\newcommand{\symz}[1]{{\rm Sym}_0(#1)}
\newcommand{\szn}{\symz{n}}
\newcommand{\sztwo}{\symz{2}}

\newcommand{\tr}{{\rm tr}\,} 
\newcommand{\trans}{\top} 

\newcommand{\ec}{\, ,}
\newcommand{\ed}{\, .}
\newcommand{\es}{\, ;}

\newcommand{\closure}[1]{\ov{#1}} 
\newcommand{\interior}[1]{{#1}^\circ} 

\renewcommand{\rho}{\varrho}
\renewcommand{\theta}{\vartheta}
\newcommand{\half}{\frac{1}{2}}  
\newcommand{\supp}{{\rm supp}} 
\newcommand{\loc}{{\rm loc}} 
\newcommand{\Cc}{C^\infty_{\rm c}} 
\newcommand{\ov}[1]{\overline{#1}}  
\newcommand{\un}[1]{\underline{#1}}  
\newcommand{\til}[1]{\widetilde{#1}}  
\renewcommand{\hat}[1]{\widehat{#1}}
\newcommand{\ep}{\varepsilon} 
\newcommand{\rot}{{\rm rot}}
\newcommand{\const}{{\rm const}} 
\newcommand{\ext}{{\rm ext}} 
\newcommand{\co}{{\rm co}} 
\newcommand{\pert}{{\rm pert}}
\renewcommand{\log}{{\rm log}}
\let\subsetnotused\subset 
\renewcommand{\subset}{\subseteq}
\renewcommand{\supset}{\supseteq}
\newcommand{\sO}{\mathcal{O}}
\newcommand{\sU}{\mathcal{U}}
\renewcommand{\setminus}{\smallsetminus}
\newcommand{\opL}{\mathscr{L}}
\newcommand{\subsetcomp}{\subsetnotused\joinrel\subsetnotused}
\newcommand{\ts}{{\rm s}}
\newcommand{\tf}{{\rm f}}
\newcommand{\oscA}{{\rm A}}
\newcommand{\oscB}{{\rm B}}
\renewcommand{\emptyset}{\varnothing}
\newcommand{\name}[1]{\textsc{#1}} 
\newcommand{\charf}{\mathbf{1}} 
\newcommand{\init}{{\rm init}}

\begin{document}
\frontmatter

\begin{titlepage}
	
	\begin{center}
		\vspace*{5cm}
		\LARGE
		
		\textsc{Convex Integration\\Applied to the Multi-Dimensional Compressible Euler Equations}\\ 
		
		\vspace{1.5cm}
		
		\Large Simon Markfelder
		
	\end{center}
\end{titlepage}

\chapter*{Preface} 

In this book we shall deal with both the barotropic and the full compressible Euler system in multiple space dimensions. Both systems are particular examples of hyperbolic conservation laws. Whereas for scalar conservation laws there exists a well-known complete well-posedness theory, and for one-dimensional systems one also has achieved several results on existence and uniqueness, in the case of multi-dimensional systems there are even negative results regarding uniqueness: With the so-called \emph{convex integration} method it is possible to show that there exist initial data for which the compressible Euler equations in multiple space dimensions admit infinitely many solutions. The convex integration technique was originally developed in the context of differential inclusions and has later been applied in groundbreaking papers by \name{De~Lellis} and \name{Sz{\'e}kelyhidi} to the incompressible Euler equations which led to infinitely many solutions. In the literature this result has been refined in order to obtain solutions for the compressible Euler system as well. The common feature of all of these non-uniqueness results for compressible Euler is an ansatz which reduces the compressible Euler equations to some kind of ``incompressible system'' for which a slight modification of the incompressible theory can be applied. In this book we present a first result of a direct application of convex integration to the barotropic compressible Euler equations. With the help of this result we will show existence of initial data for which there are infinitely many solutions both for the  barotropic and full Euler system.

We begin by discussing in detail the notion of a solution which is commonly used in the literature and also in this book, namely the \emph{admissible weak solution}, as it is essential for results on well- or ill-posedness to work with a proper notion of a solution. 

In the main part of this book we show how convex integration can be applied to the barotropic Euler system. More precisely we show that under the assumption of existence of a so-called \emph{subsolution} there exist infinitely many functions which solve the barotropic Euler equations in a particular sense. 

In view of some results that are available in the literature and which are originally shown by reducing the compressible Euler equations to some kind of ``incompressible system'', we discuss how our convex integration result can be improved in order to reprove those results.

We finish by the consideration of a particular type of initial data both for the isentropic and the full Euler system, which is inspired by one-dimensional Riemann problems. For some of those initial data we prove existence of infinitely many admissible weak solutions by applying the convex integration result. To this end it will suffice to show existence of a suitable subsolution. In addition to that we collect further results in this context from the literature.

In the appendix we briefly explain the notation, and furthermore collect well-known lemmas and theorems which are used in this book.

I would like to express my deep gratitude to my PhD-advisors Prof. Christian Klingenberg from the University of W\"urzburg and Prof. Eduard Feireisl from the Czech Academy of Sciences in Prague for their support and encouragement. Furthermore I wish to thank Ond{\v r}ej Kreml, V{\'a}clav M{\'a}cha and Hind Al Baba from the Czech Academy of Sciences in Prague for a fruitful collaboration, parts of which are presented in this book in Chapter~\ref{chap:appl-riemann}. Many thanks to Till Heller and Eva Horlebein for reading parts of this book. I am very grateful for their helpful suggestions and comments. 

\bigskip

\hfill \textit{Simon Markfelder} $\qquad$

\tableofcontents

\mainmatter

\input{DissertationCh1}

\input{DissertationCh2}

\input{DissertationCh3}

\input{DissertationCh4}

\input{DissertationCh5}

\input{DissertationCh6}

\appendix

\input{DissertationCh0}

\backmatter

\printbibliography[heading=bibintoc]

\end{document}

%% file: DissertationCh1.tex
\chapter{Introduction} \label{chap:intro} 

\section{The Euler Equations} \label{sec:intro-euler}

One of the oldest system of partial differential equations (PDEs) is the system which describes the motion of fluids formulated by \name{L. Euler} in the 18th century, see e.g. \cite{Euler08} for an English translation of one of \name{Euler}'s original works. This system of PDEs, which is nowadays known as the \emph{Euler system} or the \emph{Euler equations}, shall be investigated regarding the question of uniqueness of its solutions in this book. 

In order to formulate his model, \name{Euler} assumes that the state of a fluid is described by the three quantities
\begin{itemize}
	\item (mass) density $\rho$,
	\item velocity $\vu$ and
	\item pressure $p$, 
\end{itemize} 
all of which are functions of time $t$ and position in space $\vx$. There is a relation between the pressure $p$ and the density $\rho$ which is nowadays referred to as the \emph{equation of state}. Furthermore the state of the fluid at a certain moment in time, which we call the \emph{initial time}, is known and denoted by the \emph{initial state}. Usually the initial time is set to be zero. 

Under this assumptions \name{Euler} derives the following system of PDEs, see \cite{Euler08}, which we call the \emph{barotropic compressible Euler system}:

\begin{align}
	\partial_t \rho + \Div (\rho \vu) & = 0 \es \label{eq:baro-euler-pv-dens} \\
	\partial_t (\rho\vu) + \Div (\rho \vu\otimes\vu) + \Grad p(\rho) & = \vz \ed \label{eq:baro-euler-pv-mom}
\end{align}

This system can be viewed as a system of conservation laws, where \eqref{eq:baro-euler-pv-dens} describes the conservation of mass and \eqref{eq:baro-euler-pv-mom} the conservation of momentum. Both of those conservation principles are fundamental in physics. In other words \eqref{eq:baro-euler-pv-mom} represents \name{Newton}'s second law without external force. Note that the pressure acts on the fluid as an internal force, which is the reason for the pressure term in \eqref{eq:baro-euler-pv-mom}.

We derive the Euler equations \eqref{eq:baro-euler-pv-dens}, \eqref{eq:baro-euler-pv-mom} in a more general and modern way in Chapter~\ref{chap:conslaws}. More precisely we show there, how the concept of conservation is related to PDEs of the form \eqref{eq:baro-euler-pv-dens}, \eqref{eq:baro-euler-pv-mom}. Roughly speaking this is based on the formulation of the concept of conservation as integral equations, together with an application of the Divergence Theorem.

Let us become a bit more precise in our formulation of the Euler equations. We consider a fluid confined to a spatial domain $\Omega\subset \R^n$, where in this book we only consider the multi-dimensional case, i.e. $n=2$ or $n=3$. Note that $\Omega$ could also be the whole space $\R^n$. We study the behaviour of the fluid on a certain time interval $[0,T)$ with $T\in\R^+\cup\{\infty\}$. The initial state of the fluid is determined by given functions\footnote{In the literature it is also common to write $\rho_0,\vu_0$ for the initial data. However we will use the index ``0'' for something different, see e.g. Theorem~\ref{thm:convint}.} $\rho_\init$, $\vu_\init$ of $\vx\in \Omega$. The density takes values in $\R^+$ whereas the velocity is vector-valued, i.e. $\vu(t,\vx)\in\R^n$. In particular we exclude the vacuum case $\rho=0$ in this book. We look for solutions $\rho,\vu$ of system \eqref{eq:baro-euler-pv-dens}, \eqref{eq:baro-euler-pv-mom} which additionally satisfy the initial condition 
\begin{equation} \label{eq:baro-initial}
	\rho(0,\cdot) = \rho_\init \ec \qquad \vu(0,\cdot) = \vu_\init \ed
\end{equation}

As indicated above, in this book the pressure $p$ is a given function of $\rho$ only. We assume that this function satisfies
$$
p\in C^1(\R^+_0)\ec \quad p(0)=0\ec\quad p'(\rho)>0 \ \forall \rho\in\R^+\quad \text{and}\quad p\text{ convex}\ed
$$

In Chapter~\ref{chap:appl-riemann} we consider the special \emph{isentropic} pressure law 
\begin{equation} \label{eq:isentropic-EOS}
	p(\rho)= a \rho^\gamma \ed
\end{equation}
Here $a>0$ and $\gamma> 1$ are constants, where $\gamma$ is called \emph{adiabatic coefficient}. In this case the barotropic Euler equations \eqref{eq:baro-euler-pv-dens}, \eqref{eq:baro-euler-pv-mom} are called \emph{isentropic} Euler equations.

The equation of state makes the equations \eqref{eq:baro-euler-pv-dens}, \eqref{eq:baro-euler-pv-mom} a system of $1+n$ equations and $1+n$ unknowns $\rho,\vu$. Hence the system is formally well-determined, i.e. the number of unknowns equals the number of equations. 

Finally, if the boundary of the domain $\Omega$ is non-empty, one typically imposes boundary conditions. Only the following two cases appear in this book. Either $\Omega$ is bounded (see Chapter~\ref{chap:appl-ibvp}) and in this case we consider the \emph{impermeability boundary condition}
\begin{equation} \label{eq:impermeability}
	\vu\cdot \vn\big|_{\partial\Omega} = 0 \ec
\end{equation}
or $\Omega = \R^n$ (see Chapter~\ref{chap:appl-riemann}), i.e. $\partial\Omega = \emptyset$ and a boundary condition is not applicable. Note that it is also possible to consider many other cases, e.g. a periodic domain. Then $\Omega$ is a cube and the periodic boundary condition, where the outflow through one face of the cube flows into the cube through the opposite face, is studied. However in this book we stick to the two cases above. 

In this book we focus on the barotropic Euler equations \eqref{eq:baro-euler-pv-dens}, \eqref{eq:baro-euler-pv-mom}. However we also deal with a slightly different system, which has been established as a consequence of the development of thermodynamics in the 19th centure. The Euler system \eqref{eq:baro-euler-pv-dens}, \eqref{eq:baro-euler-pv-mom} has been supplemented by another equation which represents the conservation of energy\footnote{Note that the equations \eqref{eq:baro-euler-pv-dens}, \eqref{eq:baro-euler-pv-mom} derived by \name{Euler} in \cite{Euler08}, already allow for the pressure to depend additionally on a temparature field. However this temparatur is given a priori (for all times $t$ and all positions~$\vx$) and not considered as an additional unknown.}. We call the resulting system the \emph{full\footnote{Some authors write ``\emph{complete} compressible Euler system'' instead.} compressible Euler system}, which reads as follows:
\begin{align}
	\partial_t \rho + \Div (\rho \vu) & = 0 \es \label{eq:full-euler-pv-dens} \\
	\partial_t (\rho\vu) + \Div (\rho \vu\otimes\vu) + \Grad p & = \vz \es \label{eq:full-euler-pv-mom} \\
	\partial_t \left(\half \rho|\vu|^2 + \rho e(\rho,p)\right) + \Div \left(\left(\half \rho|\vu|^2 + \rho e(\rho,p) + p\right)\vu\right) & = 0 \ed \label{eq:full-euler-pv-en}
\end{align}

Here $e$ denotes the \emph{(specific) internal energy}. As the barotropic Euler system, the full Euler system \eqref{eq:full-euler-pv-dens} - \eqref{eq:full-euler-pv-en} is completed by an equation of state, i.e. a relation between the pressure $p$ and the internal energy $e$. The unknowns of the full Euler system are the density $\rho$, the velocity $\vu$ and the pressure $p$, taking values in $\R^+$, $\R^n$ and $\R^+$, respectively. Note that in contrast to the barotropic equations \eqref{eq:baro-euler-pv-dens}, \eqref{eq:baro-euler-pv-mom}, where the pressure $p$ is related to the density $\rho$ via the equation of state, the pressure in the full system \eqref{eq:full-euler-pv-dens} - \eqref{eq:full-euler-pv-en} is another unknown. Again the unknowns are functions of time $t\in[0,T)$ and $\vx\in \Omega\subset \R^n$. The initial state is given by a triple $\rho_\init,\vu_\init,p_\init$, i.e. the corresponding initial condition reads
\begin{equation} \label{eq:full-initial}
	\rho(0,\cdot) = \rho_\init\ec \qquad \vu(0,\cdot) = \vu_\init \ec \qquad p(0,\cdot) = p_\init \ed
\end{equation}

As indicated above, an equation of state is needed to complement the full Euler system \eqref{eq:full-euler-pv-dens} - \eqref{eq:full-euler-pv-en}. Typically in fluid mechanics another unknown -- the temparature $\theta$ -- is considered, and the equation of state consists in fact of two equations which interrelate the three unknowns $p$, $e$ and $\theta$: The \emph{thermal} equation of state $p=p(\rho,\theta)$ and the \emph{caloric} equation of state $e=e(\rho,\theta)$, see e.g. \name{Feireisl} and \name{Novotn{\'y}} \cite[Section 1.4.2]{FeiNov}\footnote{We mention here that the terms ``thermal'' and ``caloric'' are not used uniquely in the literature.}. While the consideration of the temparature $\theta$ is unavoidable for many systems in fluid mechanics (e.g. the Navier-Stokes-Fourier system, a viscous version of full Euler), there is no need to introduce a temparature $\theta$ as long as the Euler equations \eqref{eq:full-euler-pv-dens} - \eqref{eq:full-euler-pv-en} are considered. It suffices to state an \emph{incomplete} equation of state, which interrelates $p$ and $e$. More precisely it reads
\begin{equation} \label{eq:incomp-EOS-pv}
	e(\rho,p) = \frac{p}{(\gamma-1) \rho} \ed 
\end{equation} 
Again the adiabatic coefficient $\gamma>1$ is a constant. Furthermore we want to mention, that many authors still introduce a temparature $\theta$ when considering the Euler equations. This can be viewed as a special case of our general setting via the corresponding (thermal and caloric) equations of state, as long as the latter guarantee that \eqref{eq:incomp-EOS-pv} holds. For example the \emph{ideal gas} is considered very often in the literature. Here the temparature $\theta$ is related to the pressure $p$ and the internal energy $e$ via the following thermal and caloric equations of state:
\begin{align}
p(\rho,\theta) &= \rho\theta \es \label{eq:id-gas-p} \\
e(\rho,\theta) &= \frac{\theta}{\gamma - 1} \ed \label{eq:id-gas-e} 
\end{align}
It is simple to check that \eqref{eq:id-gas-p} and \eqref{eq:id-gas-e} imply the incomplete equation of state \eqref{eq:incomp-EOS-pv}. Hence the ideal gas is indeed a special case of our general setting.


The equation of state makes the full Euler system \eqref{eq:full-euler-pv-dens} - \eqref{eq:full-euler-pv-en} a formally well-determined system, as it is a system of $n+2$ equations for $n+2$ unknowns. 

In this book the full Euler system is only considered on the whole space, i.e. $\Omega=\R^n$, see Chapter~\ref{chap:appl-riemann}. Hence boundary conditions are not applicable. Of course in other contexts, when the full Euler system is confined to a bounded domain, one imposes a boundary condition.

In this book we will drop the word ``compressible''. This differs from the habit used in the literature, where the term ``Euler equations'' usually denotes the \emph{incompressible} Euler equations. When we mean the latter we will always indicate this explicitely by writing ``incompressible''.

\section{Weak Solutions and Admissibility} \label{sec:intro-sol}

In order to solve the initial (boundary) value problem for the barotropic Euler system \eqref{eq:baro-euler-pv-dens}, \eqref{eq:baro-euler-pv-mom} or for the full Euler system \eqref{eq:full-euler-pv-dens} - \eqref{eq:full-euler-pv-en}, we have to specify what we mean by a solution. At a first view one would define a solution as a pair $(\rho,\vu)$ or a triple $(\rho,\vu,p)$ of differentiable functions satisfying the PDEs \eqref{eq:baro-euler-pv-dens}, \eqref{eq:baro-euler-pv-mom} or \eqref{eq:full-euler-pv-dens} - \eqref{eq:full-euler-pv-en} pointwise in addition to the initial and, if applicable, boundary conditions. However it is well-known that even for scalar conservation laws, strong solutions\footnote{See Definition~\ref{defn:strongsol}.} of the corresponding initial value problem do not exist globally in time, no matter how smooth or small the initial data are. To overcome this problem one typically considers weak solutions. This step is also justified by keeping in mind that the concept of conservation is actually modelled by integral equations rather than PDEs like \eqref{eq:baro-euler-pv-dens}, \eqref{eq:baro-euler-pv-mom} or \eqref{eq:full-euler-pv-dens} - \eqref{eq:full-euler-pv-en}, see Chapter~\ref{chap:conslaws}. Because of many simple examples of initial data for which there are inifinitely many weak solutions, one imposes admissibility criteria to single out the physically relevant weak solutions. For scalar conservation laws it turned out that the \emph{entropy criterion} is a satisfying criterion, in the sense that weak solutions fulfilling the entropy condition, exist, are unique and depend continuously on the initial data, see \name{Kru\v{z}kov} \cite{Kruzkov70}. A weak solution is hence called \emph{admissible} or \emph{entropy solution} if it satisfies the entropy condition.

For this reason we supplement the Euler systems \eqref{eq:baro-euler-pv-dens}, \eqref{eq:baro-euler-pv-mom} and \eqref{eq:full-euler-pv-dens} - \eqref{eq:full-euler-pv-en} by the entropy condition, see Section~\ref{sec:conslaws-admissibility}. We will see in Chapter~\ref{chap:euler} that for both systems there are just few entropies\footnote{See Definition~\ref{defn:comp-entropy}~\ref{item:entropy} for a definition of the ``mathematical'' notion of entropy.}. For barotropic Euler in multiple space dimenstions, we prove that there is in principal just one entropy, namely the physical energy, see Subsection~\ref{subsec:euler-baro-entropies}. Hence we complement the barotropic Euler system \eqref{eq:baro-euler-pv-dens}, \eqref{eq:baro-euler-pv-mom} with the energy inequality 
\begin{equation} \label{eq:baro-euler-adm(energy)}
	\partial_t \left(\half\rho|\vu|^2 + P(\rho)\right) + \Div\left[\left(\half\rho|\vu|^2 + P(\rho) + p(\rho)\right)\vu\right] \leq 0\ed
\end{equation}
Here $P$ is the \emph{pressure potential} given by 
\begin{equation} \label{eq:pressure-potential}
P(\rho) = \rho \int_{\rho_0}^\rho \frac{p(r)}{r^2} \dr \ec
\end{equation}
where $\rho_0$ is arbitrary. In particular for the isentropic equation of state \eqref{eq:isentropic-EOS} one obtains 
\begin{equation} \label{eq:isentropic-P}
P(\rho) = \frac{a}{\gamma - 1} \rho^\gamma \ed
\end{equation}
Consistently a weak solution $(\rho,\vu)$ of the barotropic Euler system \eqref{eq:baro-euler-pv-dens}, \eqref{eq:baro-euler-pv-mom} is called \emph{admissible} or \emph{entropy solution} if the energy inequality \eqref{eq:baro-euler-adm(energy)} holds in the weak sense. As indicated above we show in Subsection~\ref{subsec:euler-baro-entropies} that the energy inequality \eqref{eq:baro-euler-adm(energy)} is nothing but the entropy condition in the sense of Section~\ref{sec:conslaws-admissibility}. One could establish the energy inequality \eqref{eq:baro-euler-adm(energy)} from a physical point of view as well since prohibiting energy to be produced is a fundamental principle in physics.

For the full Euler system \eqref{eq:full-euler-pv-dens} - \eqref{eq:full-euler-pv-en} in multiple space dimensions, we obtain the following entropy criterion: A weak solution $(\rho,\vu,p)$ is called \emph{admissible} if 
\begin{equation} \label{eq:full-euler-adm(entropy)} 
	\partial_t \Big(\rho Z\big(s(\rho,p)\big) \Big) + \Div \Big(\rho Z\big(s(\rho,p)\big) \vu \Big) \geq 0
\end{equation}
holds for all $Z\in C^\infty(\R)$ with $Z'\geq 0$, where 
\begin{equation} \label{eq:s}
s(\rho,p) := \frac{1}{\gamma-1} \Big( \log p - \gamma \log \rho \Big) 
\end{equation}
is the ``physical'' \emph{(specific) entropy}\footnote{One should notice that there is a difference between the ``mathematical'' entropies introduced in Section~\ref{sec:conslaws-entropies} and the ``physical'' entropy $s$. A priori those are two different notions. However we show in Subsection~\ref{subsec:euler-full-entropies} that in principal all (mathematical) entropies $\eta$ in the context of the full Euler equations are of the form $\eta=-\rho Z(s)$.}. Again our derivation of the condition above is build upon the consideration of (mathematical) entropies in the sense of Sections \ref{sec:conslaws-entropies} and \ref{sec:conslaws-admissibility}. Condition \eqref{eq:full-euler-adm(entropy)} can be also justified from a physical point of view as \eqref{eq:full-euler-adm(entropy)} represents the second law of thermodynamics.

\section{Overview on Well-Posedness Results} \label{sec:intro-well-posedness} 

A similar well-posedness theory as obtained by \name{Kru\v{z}kov} \cite{Kruzkov70} for scalar conservation laws, is far from being reached for the initial boundary value problem for systems of conservation laws, e.g. the Euler systems \eqref{eq:baro-euler-pv-dens}, \eqref{eq:baro-euler-pv-mom} or \eqref{eq:full-euler-pv-dens} - \eqref{eq:full-euler-pv-en}. On the one hand, some is known in one space dimension: Existence of weak entropy solutions for sufficiently small initial data was shown by \name{Glimm} \cite{Glimm65} in 1965. Here the standard approach was used, where one selects an appropriate sequence of approximate solutions, for which one is able to show that a subsequence converges and that the limit is in fact a desired weak solution. Later on in 1983 \name{DiPerna} \cite{DiPerna83} used a different approach to prove existence of solutions to the isentropic Euler equations. He considered approximate solutions which converge to a so-called \emph{measure-valued solution} which can be viewed as a ``very weak'' notion of a solution. With the technique of compensated compactness, he could prove that the limit is not just a measure but actually a weak solution.

In addition to these existence results, \name{Bressan et al.} \cite{BreCraPic00} proved uniqueness of \name{Glimm}'s solutions and continuous dependence on the initial data, again in the case of one-dimensional systems of conservation laws. 

Let us now turn our attention towards multiple space dimensions. Existence of admissible weak solutions to the compressible Euler equations is still open. For this reason one also studies a more general notion of a solution, namely the above-mentioned measure-valued solutions. Just like for the weak solutions, one also introduces notions of \emph{admissibility} for the measure-valued solutions. For this ``admissible'' measure-valued solutions one can a least prove existence together with some other properties like weak-strong uniqueness, see e.g. \name{Gwiazda}, \name{{\'S}wierczewska-Gwiazda} and \name{Wiedemann} \cite{GwiWie15} for isentropic Euler or \name{B{\v r}ezina} and \name{Feireisl} \cite{BreFei18} for full Euler. However in this book we stick to weak solutions. 

Regarding uniqueness, negative results are known, some of which are the principal content of this book. Before detailing these non-uniqueness results, let us begin by introducing the \emph{incompressible} Euler equations, even though they play only a minor role in this book:
\begin{align} 
	\Div \vv &=0\es \label{eq:incomp-euler-divfree}\\
	\partial_t \vv + \Div (\vv\otimes\vv) + \Grad p &=\vz\ed \label{eq:incomp-euler-velocity}
\end{align}
The unknowns are the velocity $\vv = \vv(t,\vx)$ and the pressure $p=p(t,\vx)$, which take values in $\R^n$ and $\R^+$, respectively. 

The first non-uniqueness results for the incompressible Euler equations go back to \name{Scheffer} \cite{Scheffer93} in 1993 and \name{Shnirelman} \cite{Shnirelman00} in 2000. Later in 2009 and 2010 \name{De~Lellis} and \name{Sz{\'e}kelyhidi} \cite{DelSze09}, \cite{DelSze10} continued this study with the aim to get a better understanding of turbulence and to tackle a proof of Onsager's conjecture. This conjecture claims that weak solutions of \eqref{eq:incomp-euler-divfree}, \eqref{eq:incomp-euler-velocity} which are H{\"o}lder continuous with exponent $\alpha>1/3$ conserve the energy $\half |\vv|^2$, and furthermore the threshold of $1/3$ is sharp in the sense that up to the H\"older exponent $\alpha<1/3$ there exist H\"older continuous solutions of the incompressible Euler system which dissipate energy. The first part of Onsager's conjecture had been proven already in 1994 by \name{Constantin}, \name{E} and \name{Titi} \cite{ConETit94}, whereas the other part remained open until recently. A series of papers \cite{DelSze13}, \cite{IsettPHD}, \cite{DelSze14}, \cite{BuckmasterPHD}, \cite{Buckmaster15}, \cite{BDIS15}, \cite{DanSze17}, all of which are built upon the technique formed by \name{De~Lellis} and \name{Sz{\'e}kelyhidi} in \cite{DelSze09} and \cite{DelSze10}, proved existence of energy dissipative solutions whose regularity improved from result to result. Finally the remaining part of Onsager's conjecture was proved in 2018 by \name{Isett} \cite{Isett18} and \name{Buckmaster}, \name{De~Lellis}, \name{Sz\'ekelyhidi} and \name{Vicol} \cite{BDSV18}. 

The technique used in \cite{DelSze09} and \cite{DelSze10} is called \emph{convex integration} and was developed by \name{Gromov}, see e.g. \cite{Gromov}, in the context of partial differential relations. The series of papers mentioned above can be viewed as a refinement of this technique. The solutions which are obtainted by convex integration are sometimes called ``\emph{wild}''.

The paper \cite{DelSze10} by \name{De~Lellis} and \name{Sz{\'e}kelyhidi} additionally considered both the compressible barotropic Euler equations \eqref{eq:baro-euler-pv-dens}, \eqref{eq:baro-euler-pv-mom} and the full Euler system \eqref{eq:full-euler-pv-dens} - \eqref{eq:full-euler-pv-en} and hence this paper is the cornerstone of our studies. Here it was shown that for the barotropic Euler equations \eqref{eq:baro-euler-pv-dens}, \eqref{eq:baro-euler-pv-mom} there exist bounded initial data for which there are infinitely many weak solutions, all of which fulfill the energy inequality \eqref{eq:baro-euler-adm(energy)}. A similar statement can be shown for the full Euler system \eqref{eq:full-euler-pv-dens} - \eqref{eq:full-euler-pv-en}. This non-uniqueness result for compressible Euler is derived from the results for the incompressible Euler system. More precisely \name{De~Lellis} and \name{Sz{\'e}kelyhidi} proceed as follows. In \cite{DelSze10} existence of bounded initial data to which there are infinitely many weak solutions of the initial value problem to \eqref{eq:incomp-euler-divfree}, \eqref{eq:incomp-euler-velocity} is shown. Here it is possible to prescribe the kinetic energy $\half|\vv|^2$ for all times $t$ and almost all positions $\vx$. If one requires the kinetic energy to be constant, then one can achieve infinitely many weak solutions to \eqref{eq:incomp-euler-divfree}, \eqref{eq:incomp-euler-velocity} with constant pressure\footnote{Here \emph{constant} means constant in time and space.}. It is easy to verify that $(\rho,\vu)$ definied by $\rho\equiv 1$ and $\vu=\vv$, where $\vv$ is one of the infinitely many solutions of the incompressible Euler system \eqref{eq:incomp-euler-divfree}, \eqref{eq:incomp-euler-velocity} with constant pressure, solves the barotropic compressible Euler system \eqref{eq:baro-euler-pv-dens}, \eqref{eq:baro-euler-pv-mom}. Since the density $\rho$, $|\vu|$ and the pressure are constant, and in addition to that $\Div\vu=\Div\vv=0$ due to \eqref{eq:incomp-euler-divfree}, $(\rho,\vu)$ even solves the energy inequality \eqref{eq:baro-euler-adm(energy)} as equation. Defining additionally $p\equiv 1$, similar arguments show that $(\rho,\vu,p)$ is an admissible weak solution of the full Euler system \eqref{eq:full-euler-pv-dens} - \eqref{eq:full-euler-pv-en}, where again the entropy inequality \eqref{eq:full-euler-adm(entropy)} holds as equation. 

This first result on non-uniqueness for the compressible Euler equations has been further improved in the case of the barotropic Euler system \eqref{eq:baro-euler-pv-dens}, \eqref{eq:baro-euler-pv-mom}. For a periodic domain \name{Chiodaroli} \cite{Chiodaroli14} as well as \name{Feireisl} \cite{Feireisl14} showed that for any continuously differentiable initial density $\rho_\init$ there exists a bounded initial velocity $\vu_\init$ for which there are infinitely many admissible weak solutions to the initial boundary value problem.

\name{Chiodaroli}'s idea \cite{Chiodaroli14} is to look for solutions with constant-in-time density, which means $\rho(t,\cdot)=\rho_\init$ for all $t$. This leads to some kind of ``incompressible system'' for the momentum $\rho\vu$ which includes the prescribed initial density $\rho_\init$. A slight modification of \name{De~Lellis}' and \name{Sz{\'e}kelyhidi}'s convex integration then yields the result. This result has been refined recently for other boundary conditions or the whole space by \name{Akramov} and \name{Wiede\-mann} \cite{AkrWie21}.

\name{Feireisl}'s approach \cite{Feireisl14} is even more general. His idea is to apply Helmholtz decomposition to the momentum, i.e. write $\rho\vu=\vv + \Grad\Phi$ where $\vv$ is div-free and $\Phi$ is a scalar field. He then prescribes $\rho$ and $\Phi$ in such a way that they are compatible with the conservation of mass \eqref{eq:baro-euler-pv-dens}, i.e. $\partial_t\rho + \Lap \Phi = 0$. Similarly to \name{Chiodaroli}'s work sketched above, one finally ends up with some kind of ``incompressible system'' for $\vv$, to which one applies a modified version of \name{De~Lellis}' and \name{Sz{\'e}kelyhidi}'s convex integration.

For the full Euler system \name{Feireisl}, the author and others \cite{FKKM20} prove that for all piecewise constant initial densities $\rho_\init$ and piecewise initial pressures $p_\init$ there exists a bounded initial velocity $\vu_\init$ for which there are infinitely many admissible weak solutions. This result is built upon a simple observation by \name{Luo}, \name{Xie} and \name{Xin} \cite{LuoXieXin16} which allows to apply convex integration in pairwise disjoint pieces of the domain and merge the solutions to obtain an object defined on the whole domain. Again the actual convex integration is merely a version of method developed by \name{De~Lellis} and \name{Sz{\'e}kelyhidi}.

In addition to that, the literature provides non-uniqueness results for the initial value problem to compressible Euler with a special type of initial data that are inspired by one-dimensional Riemann problems, see papers by \name{Chiodaroli}, \name{Kreml}, the author and collaborators \cite{ChiDelKre15}, \cite{ChiKre14}, \cite{KliMar18_1}, \cite{ChiKre18}, \cite{KliMar18_2}, \cite{BreChiKre18}, \cite{BreKreMac18}, \cite{CKMS19}, \cite{AKKMM20} \cite{KKMM20}. These solutions are constructed with piecewise constant densities and again a slight modification of \name{De~Lellis}' and \name{Sz{\'e}kelyhidi}'s convex integration. We come back to these results in Chapter~\ref{chap:appl-riemann}.

The reduction of the problem to some kind of ``incompressible system'' for which a slight modification of \name{De~Lellis}' and \name{Sz{\'e}kelyhidi}'s convex integration can be applied, is common to all non-uniqueness results for compressible Euler available in the literature. Apart from a general presentation of non-uniqueness results for the compressible Euler equations, the objective of this book is to apply convex integration directly to the barotropic compressible Euler system \eqref{eq:baro-euler-pv-dens}, \eqref{eq:baro-euler-pv-mom}.

\section{Structure of this Book}

This book is organized as follows. In Chapter~\ref{chap:conslaws} we present basic facts on hyperbolic conservation laws. We discuss some notions of solutions and introduce entropies. In Chapter~\ref{chap:euler} we show that the barotropic Euler equations \eqref{eq:baro-euler-pv-dens}, \eqref{eq:baro-euler-pv-mom} and the full Euler system \eqref{eq:full-euler-pv-dens} - \eqref{eq:full-euler-pv-en} are indeed systems of hyperbolic conservation laws. Furthermore -- and this is more important -- we derive the entropy conditions \eqref{eq:baro-euler-adm(energy)} and \eqref{eq:full-euler-adm(entropy)} for both systems. Chapter~\ref{chap:convint} is the main part of this book. Here convex integration is implemented directly for barotropic Euler, i.e. not via an ansatz to obtain some kind of ``incompressible system''. For completeness and in view of the application in Chapter~\ref{chap:appl-riemann}, we prove with slight modifications a fixed-density-version of our convex integration, the result of which coincides in some sense with the incompressible convex integration by \name{De~Lellis} and \name{Sz{\'e}kelyhidi}. As an application of our convex integration result, we first look at general initial boundary value problems for the barotropic Euler equations and prove a simple statement on weak solutions, see Chapter~\ref{chap:appl-ibvp}. In view of the results by Chiodaroli \cite{Chiodaroli14} and Feireisl \cite{Feireisl14} mentioned above, we explain how our convex integration could be further improved, especially in order to obtain \emph{admissible} solutions. In Chapter~\ref{chap:appl-riemann} we discuss a second application of the convex integration result presented in this book. More precisely we consider the above mentioned special type of initial data that are inspired by one-dimensional Riemann problems. We summarize results regarding non-/uniqueness of admissible weak solutions and show exemplarily how non-uniqueness is proven, where we have to use the fixed-density-version of our convex integration in order to obtain infinitely many \emph{admissible} solutions. 

\bigskip

Let us finish our introduction by quoting \name{Euler} \cite[Paragraph 68]{Euler08}:
\begin{quotation}
	\textit{``This makes it quite clear how far removed we are from a complete understanding of the motion of fluids and that my exposition is no more than a mere beginning. Nevertheless, everything that the Theory of Fluids contains is embodied in the two equations formulated above [...], so that it is not the laws of Mechanics that we lack in order to pursue this research but only the Analysis, which has not yet been sufficiently developed for this purpose. It is therefore clearly apparent what discoveries we still need to make in this branch of Science before we can arrive at a more perfect Theory of the motion of fluids.''}
\end{quotation}

To summarize the introduction and addressing to the quote above, let us note that even if some of the mathematical difficulties described by \name{Euler} have been overcome, for example by considering weaker notions of solutions and admissibility conditions, a complete understanding of the Euler equations has still not been reached.

%% file: DissertationCh2.tex
\chapter{Hyperbolic Conservation Laws} 
\label{chap:conslaws}

The theory of \emph{hyperbolic\footnote{Note that the notion of \emph{hyperbolicity} as defined in Section~\ref{sec:conslaws-hyper} can be viewed as a side note in this book because it is not really important here.} conservation laws} is a very important field in mathematics. Their formulation is highly inspired by natural processes. We observe in Chapter~\ref{chap:euler} that the compressible Euler systems \eqref{eq:baro-euler-pv-dens}, \eqref{eq:baro-euler-pv-mom} and \eqref{eq:full-euler-pv-dens} - \eqref{eq:full-euler-pv-en} are particular examples of hyperbolic conservation laws. In this chapter we deal with hyperbolic conservation laws in general. More details can be found in textbooks, e.g. by \name{Dafermos} \cite{Dafermos}.

\section{Formulation of a Conservation Law} \label{sec:conslaws-formulation} 
Let us consider a set of $m\in\N$ conserved quantities which are functions of time $t\in[0,T)$ and space $\vx\in\Omega$. The conservation process is studied on a time interval $[0,T)\subset\R$ with maximal time $T\in\R^+\cup\{\infty\}$. Furthermore the space-variable $\vx$ shall lie in some spatial Lipschitz domain $\Omega\subset\R^n$, where $n$ is the dimension of the physical space\footnote{In this book we focus on the case $n=2$ or $n=3$, but this restriction is not used in the current chapter.} and $\Omega$ is not necessarily bounded. 

For any $(t,\vx)\in [0,T)\times \Omega$ we collect the spatial densities of those conserved quantities in an $m$-dimensional vector $\vU(t,\vx)$, the \emph{state vector}. Furthermore we require the state vector to lie in some open set $\sO\subset \R^m$, which is also called \emph{phase space}. Hence 
$$
	\vU:[0,T)\times\Omega \to \sO
$$ 
is a function which maps $(t,\vx)\in [0,T)\times\Omega$ to something that lies in $\sO$. With these definitions the integral 
$$
\int_D \vU(t,\vx) \dx
$$ 
for $t\in[0,T)$ and a bounded subset $D\subset\Omega$ gives total ``mass'' of the conserved quantities in the subset $D$ at time $t$.

In order to formulate the conservation law, we have to introduce the matrix-valued \emph{flux function} 
$$
\mF: \sO \to \R^{m\times n},\ \vU\mapsto \mF(\vU) \ed
$$
In the sequel we assume that $\mF\in C^1(\sO;\R^{m\times n})$.

\begin{rem} 
	Note that since $\vU$ is a function of $(t,\vx)$, $\mF$ can be viewed as function of $(t,\vx)$ as well. We will do this very often in the following. More generally and not in this book, one could also consider the flux $\mF$ as an immediate function of $(t,\vx)$ rather than via $\vU$, see e.g. the textbook by \name{Dafermos} \cite[Chapter I]{Dafermos}. However in many cases and in particular always in this book, there is a relation between $\mF$ and $\vU$, which is sometimes called \emph{constitutive relation}.
\end{rem} 

Now we are ready to formulate the conservation law: We say that the quantities $\vU:[0,T)\times \Omega\to \sO$ are \emph{conserved} or obey a \emph{conservation law} if 
\begin{equation} \label{eq:cons-law-integral}
	\int_D \vU(t_1,\vx) \dx - \int_D \vU(t_0,\vx) \dx + \int_{t_0}^{t_1} \int_{\partial D} \mF\big(\vU(t,\vx)\big) \cdot \vn(\vx) \dS_\vx \dt = \vz
\end{equation}
holds for all times $0 \leq t_0 \leq t_1< T$ and all bounded Lipschitz domains $D\subset \Omega$. 

\begin{rem} 
	For the time being we do not specify in which function space $\vU$ lies. In particular we do not ask for regularity of $\vU$. However we tacitly assume that the integrals in \eqref{eq:cons-law-integral} exist. We will later determine that $\vU\in L^\infty\big((0,T)\times \Omega;\sO\big)$. Then we need to say how \eqref{eq:cons-law-integral} has to be understood since an $L^\infty$ function is only defined almost everywhere and the sets $\{t_0\}\times D$ or $[t_0,t_1]\times \partial D$ are sets of zero measure (with respect to the $(1+n)$-dimensional Lebesgue measure).  
\end{rem}

\begin{rem}
	If there is only one conserved quantity, i.e. $m=1$, we call the conservation law \emph{scalar}, whereas we speak of a \emph{system} of conservation laws if $m\geq 2$.
\end{rem}

In the sequel we are going to understand the meaning of \eqref{eq:cons-law-integral}. From a physical point of view the integral 
\begin{equation} \label{eq:1-temp-conslaws}
	\int_{t_0}^{t_1} \int_{\partial D} \mF\big(\vU(t,\vx)\big) \cdot \vn(\vx) \dS_\vx \dt
\end{equation}
represents the mass which flows through the boundary $\partial D$ out\footnote{The fact that it represents the \emph{outward} flowing mass is due to the fact that $\vn$ is the \emph{outward} pointing normal vector on $\partial D$.} of the domain $D$ in the time interval $[t_0,t_1]$. Note that the same expression \eqref{eq:1-temp-conslaws} with a minus sign in front accordingly represents the mass that flows \emph{into} the domain $D$ in the time interval $[t_0,t_1]$. In order to understand the physical meaning of \eqref{eq:cons-law-integral} let us rewrite it as 
\begin{equation} \label{eq:cons-law-integral-B}
\int_D \vU(t_1,\vx) \dx = \int_D \vU(t_0,\vx) \dx -  \int_{t_0}^{t_1} \int_{\partial D} \mF\big(\vU(t,\vx)\big) \cdot \vn(\vx) \dS_\vx \dt\ed
\end{equation} 
With the above-mentioned physical interpretations of each of the terms, equation \eqref{eq:cons-law-integral-B} can be understood as follows. The total mass of the conserved quantities in the domain $D$ at time $t_1$ is equal to the total mass in the same domain $D$ at the earlier time $t_0$ added to the amount of mass that flew into the domain $D$ during $[t_0,t_1]$. Therefore the conservation law as formulated in \eqref{eq:cons-law-integral-B} or \eqref{eq:cons-law-integral} is precisely a mathematical formulation of the physical concept of \emph{conservation}. Hence this \emph{integral formulation} is the most natural way how to write down a conservation law. However in the literature a conservation law is commonly written in its \emph{weak formulation} rather than \eqref{eq:cons-law-integral}: The equation 
\begin{equation} \label{eq:cons-law-weak}
	\int_0^T \int_{\Omega} \Big(\vU(t,\vx) \cdot \partial_t\vpsi(t,\vx) + \mF\big(\vU(t,\vx)\big) : \Grad \vpsi(t,\vx) \Big)\dx \dt= 0
\end{equation}
holds for all test functions $\vpsi \in \Cc\big((0,T)\times\Omega;\R^m\big)$.

Indeed one can show that the integral formulation \eqref{eq:cons-law-integral} and the weak formulation \eqref{eq:cons-law-weak} are equivalent. However this requires some effort. A rigorous proof can be found in the textbook by \name{Dafermos} \cite[Section 1.1 - 1.3]{Dafermos}. In this book conservation laws are introduced via measures on the boundary of Lipschitz space-time domains, which can be viewed as a generalization of \eqref{eq:cons-law-integral} where such a space-time domain is $(t_0,t_1)\times D$. Such a measure is induced by a measurable function on the boundary. Then there exists $\vU\in L^\infty\big((0,T)\times \Omega;\sO\big)$ which satisfies \eqref{eq:cons-law-weak} for all $\vpsi \in \Cc\big((0,T)\times\Omega;\R^m\big)$ as proven in \cite[Theorem 1.2.1]{Dafermos}. The converse is shown in \cite[Theorem 1.3.4]{Dafermos}: Let $\vU\in L^\infty\big((0,T)\times \Omega;\sO\big)$ fulfill \eqref{eq:cons-law-weak} for all $\vpsi \in \Cc\big((0,T)\times\Omega;\R^m\big)$. Then for each Lipschitz space-time domain there exists a measurable function defined on the boundary\footnote{Such a function is sometimes called \emph{(normal) trace} of $\vU$.} such that the conservation law \eqref{eq:cons-law-integral} is satisfied. 

For this reason from now on a function which obeys a conservation law is a function $\vU\in L^\infty\big((0,T)\times \Omega;\sO\big)$ which obeys the conservation law in its weak formulation, i.e. $\vU$ fulfills \eqref{eq:cons-law-weak} for all $\vpsi \in \Cc\big((0,T)\times\Omega;\R^m\big)$.

Let us finish this section with the \emph{strong}\footnote{Other authors use the term \emph{classical} instead of \emph{strong}.} formulation of the conservation law. Assume that $\vU\in C^1\big((0,T)\times \Omega;\sO\big)$ satisfies the conservation law, i.e. it satisfies \eqref{eq:cons-law-weak} for all $\vpsi \in \Cc\big((0,T)\times\Omega;\R^m\big)$. Then it follows from the Divergence Theorem (Proposition~\ref{prop:not-divergence}) that 
\begin{align}
	0 &= \int_0^T \int_{\Omega} \Big(\vU(t,\vx) \cdot \partial_t\vpsi(t,\vx) + \mF\big(\vU(t,\vx)\big) : \Grad \vpsi(t,\vx) \Big)\dx \dt \notag\\
	&= - \int_0^T \int_{\Omega} \Big(\partial_t \vU(t,\vx) + \Div\mF\big(\vU(t,\vx)\big) \Big) \cdot \vpsi(t,\vx)\dx \dt \label{eq:2-temp-conslaws}
\end{align}
for all $\vpsi \in \Cc\big((0,T)\times\Omega;\R^m\big)$. Note that the fact that $\Omega$ might be unbounded seems to cause problems when applying the Divergence Theorem. However since $\vpsi$ has compact support, we can integrate over a \emph{bounded} (and still Lipschitz) domain instead of $\Omega$ in \eqref{eq:2-temp-conslaws}. Furthermore the boundary terms in \eqref{eq:2-temp-conslaws} disappear because the test functions $\vpsi$ have compact support in $(0,T)\times \Omega$. 

Since \eqref{eq:2-temp-conslaws} holds for all test functions $\vpsi \in \Cc\big((0,T)\times\Omega;\R^m\big)$, we deduce that 
\begin{equation} \label{eq:cons-law-strong}
	\partial_t \vU(t,\vx) + \Div\mF\big(\vU(t,\vx)\big)=\vz
\end{equation}
must hold pointwise for all $(t,\vx)\in(0,T)\times\Omega$. This is the strong formulation. 

The converse of what is proven above is also true. Let us record this by the following proposition.
\begin{prop} \label{prop:weak-vs-strong}
	A function $\vU\in C^1\big((0,T)\times \Omega;\sO\big)$ fulfills \eqref{eq:cons-law-weak} for all test functions $\vpsi \in \Cc\big((0,T)\times\Omega;\R^m\big)$ if and only if it fulfills \eqref{eq:cons-law-strong} for all $(t,\vx)\in(0,T)\times\Omega$.
\end{prop}

\begin{proof}
	One implication has been showed above. The other implication follows by applying similar arguments in reverse order.
\end{proof}

\begin{rem} 
	Proposition~\ref{prop:weak-vs-strong} says that we can use the strong formulation of the conservation law as long as differentiable functions $\vU$ are considered. However for $L^\infty$ functions we must to stick to the weak formulation. 
\end{rem}

\begin{rem} 
	One says that a function $\vU\in L^\infty\big((0,T)\times \Omega;\sO\big)$ satisfies \eqref{eq:cons-law-strong} in the \emph{sense of distributions} if \eqref{eq:cons-law-weak} holds for all test functions $\vpsi \in \Cc\big((0,T)\times\Omega;\R^m\big)$. In the same manner one commonly uses the strong formulation when writing down a conservation law even if one looks for $L^\infty$-functions which fulfill the conservation law in the weak sense or, equivalently, in the sense of distributions.
\end{rem}

\section{Initial Boundary Value Problem} \label{sec:conslaws-ibvp} 

One typically views a conservation law as a description of the time evolution of a set of conserved quantities, i.e. a state vector $\vU$. Hence one aims to find $\vU(t,\cdot)$ for all times $t\in(0,T)$ where one knows the initial state, given by a function $\vU_\init\in L^\infty(\Omega;\sO)$. In other words the $\vU$ we are looking for not only obeys conservation, but also the \emph{initial condition} 
\begin{equation} \label{eq:conslaws-ic}
\vU(0,\cdot) = \vU_\init \ed
\end{equation}
If $\Omega \subsetneq \R^n$, one usually wants $\vU$ to satisfy a \emph{boundary condition}\footnote{Actually the boundary condition is a requirement on $\mF$. However this can be translated into a requirement on $\vU$ through the constitutive relation.}, i.e. one determines 
\begin{equation} \label{eq:conslaws-bc}
\mF\cdot \vn\big|_{\partial\Omega} = \vF_{\partial \Omega} \ec
\end{equation}
where $\vF_{\partial \Omega}\in L^\infty\big((0,T)\times \partial\Omega;\R^m\big)$ is a given function.

Let us now introduce two notions of solutions to the initial (boundary) value problem\footnote{We write the word ``boundary'' in brackets since the boundary condition is only applicable if $\partial\Omega\neq \emptyset$.} \eqref{eq:cons-law-strong}\footnote{As mentioned in a remark above, we write \eqref{eq:cons-law-strong} even if we deal with weak solutions in Definition~\ref{defn:weaksol}.}, \eqref{eq:conslaws-ic} and, if applicable, \eqref{eq:conslaws-bc}. It is simple to define strong solutions because for those it is clear in which sense the initial and the boundary condition hold.

\begin{defn} \label{defn:strongsol}
	We call a function $\vU\in C^1\big([0,T)\times \closure{\Omega};\sO\big)$ a \emph{strong solution} of the initial (boundary) value problem \eqref{eq:cons-law-strong}, \eqref{eq:conslaws-ic}, \eqref{eq:conslaws-bc} if
	\begin{itemize}
		\item $\vU$ obeys the conservation law strongly, i.e. it solves \eqref{eq:cons-law-strong} for all $(t,\vx)\in (0,T)\times \Omega$,
		\item the initial condition \eqref{eq:conslaws-ic} holds pointwise a.e. on $\Omega$ and
		\item if applicable, the boundary condition \eqref{eq:conslaws-bc} is fulfilled pointwise a.e. on $(0,T)\times \partial\Omega$.
	\end{itemize}
\end{defn} 

More interesting are the weak solutions. Here we have to clearify in which sense the initial and the boundary condition hold.

\begin{defn} \label{defn:weaksol}
	We call a function $\vU\in L^\infty\big((0,T)\times \Omega;\sO\big)$ a \emph{weak solution} of the initial (boundary) value problem \eqref{eq:cons-law-strong}, \eqref{eq:conslaws-ic}, \eqref{eq:conslaws-bc} if the equation
	\begin{align} 
		\int_0^T \int_{\Omega} \Big(\vU \cdot \partial_t\vpsi + \mF(\vU) : \Grad \vpsi \Big)\dx \dt + \int_\Omega \vU_\init \cdot \vpsi(0,\cdot) \dx \qquad & \notag \\
		- \int_0^T \int_{\partial\Omega} \vpsi \cdot \vF_{\partial \Omega} \dS_\vx \dt &= 0 \label{eq:conslaws-ivp-weak} 
	\end{align}
	holds for all test functions $\vpsi \in \Cc\big([0,T)\times\closure{\Omega};\R^m\big)$. If $\Omega=\R^n$, then the boundary term
	\begin{equation} \label{eq:boundary-term}
		\int_0^T \int_{\partial\Omega} \vpsi \cdot \vF_{\partial \Omega} \dS_\vx \dt
	\end{equation}
	disappears. Furthermore if $\Omega\subsetneq\R^n$ and one does not ask for a boundary condition, then one usually adjusts the support of the test function in order to make the boundary term \eqref{eq:boundary-term} vanish. 
\end{defn} 

\begin{rem}
	Note that in contrast to \eqref{eq:cons-law-weak} the test functions $\vpsi$ in \eqref{eq:conslaws-ivp-weak} might not vanish for $t=0$ and on the boundary $\partial\Omega$. It is a matter of taste if one writes $\Cc\big([0,T)\times\closure{\Omega};\R^m\big)$ or $\Cc\big([0,T)\times\R^n;\R^m\big)$ since each $\vpsi \in \Cc\big([0,T)\times\closure{\Omega};\R^m\big)$ can be extended to a function in $\Cc\big([0,T)\times\R^n;\R^m\big)$ and conversely the restriction $\vpsi|_{\closure{\Omega}} \in \Cc\big([0,T)\times\closure{\Omega};\R^m\big)$ for each $\vpsi \in \Cc\big([0,T)\times\R^n;\R^m\big)$.
\end{rem}

The initial condition \eqref{eq:conslaws-ic} is represented in \eqref{eq:conslaws-ivp-weak} by the term 
$$
	\int_\Omega \vU_\init \cdot \vpsi(0,\cdot) \dx 
$$
whereas the boundary condition \eqref{eq:conslaws-bc} is represented by the term \eqref{eq:boundary-term}.

The following proposition is an analogon to Propostion \ref{prop:weak-vs-strong}.

\begin{prop} \label{prop:weak-vs-strong2}
	A function $\vU\in C^1\big([0,T)\times \closure{\Omega};\sO\big)$ is a strong solution of the initial (boundary) value problem \eqref{eq:cons-law-strong}, \eqref{eq:conslaws-ic}, \eqref{eq:conslaws-bc} if and only if it is a weak solution.
\end{prop}

\begin{proof}
	The proof is works analogously to the proof of Propostion \ref{prop:weak-vs-strong}. The key is to apply the Divergence Theorem (Propostion \ref{prop:not-divergence}) where now the boundary terms do not vanish as the test functions $\vpsi \in \Cc\big([0,T)\times\closure{\Omega};\R^m\big)$ do not vanish for $t=0$ and on $\partial \Omega$. The details are left to the reader.
\end{proof}

\section{Hyperbolicity} \label{sec:conslaws-hyper}

The notion of \emph{hyperbolicity} is only a side note in this book. For details we refer e.g. to the textbook by \name{Dafermos} \cite[Section 3.1]{Dafermos}.

\begin{defn} \label{defn:hyperbolicity}
	The system of conservation laws \eqref{eq:cons-law-strong} is called \emph{hyperbolic} if for any $\vU\in\sO$ and $\vnu\in\sphere^{n-1}$ there exist $m$ linearly independent eigenvectors $\vR_i(\vnu;\vU)\in\R^m$ ($i=1,...,m$) of the matrix $\sum_{k=1}^n \nu_k \Grad_\vU \vF_k$. The corresponding eigenvalues are denoted by $\lambda_i(\vnu;\vU)$. 
\end{defn}

\section{Companion Laws and Entropies} \label{sec:conslaws-entropies}

In order to select a physically relevant weak solution, we have to study entropies.

\begin{defn} \label{defn:comp-entropy}
	\begin{enumerate}
		\item \label{item:companion} A function $\eta\in C^\infty(\sO)$ is called \emph{companion} if there exists a function $\vq\in C^\infty(\sO;\R^n)$ with 
		\begin{equation} \label{eq:companion-cond}
			\Grad_\vU q_k(\vU) = \Grad_\vU \eta(\vU) \cdot \Grad_\vU \vF_k (\vU)
		\end{equation}
		for all $k=1,...,n$ and all $\vU\in\sO$.
		\item \label{item:entropy} A convex campanion is called \emph{entropy}. In this case the corresponding $\vq$ is called \emph{entropy flux} and the pair $(\eta,\vq)$ \emph{entropy-entropy flux-pair} or simply \emph{entropy pair}. 
	\end{enumerate} 
\end{defn}

\begin{rem}
	Condition \eqref{eq:companion-cond} means that 
	$$
	\partial_{U_j} q_k(\vU) = \sum_{i=1}^m \partial_{U_i} \eta(\vU) \ \partial_{U_j} F_{ik} (\vU)
	$$
	holds for all $k=1,...,n$, all $i,j=1,...,m$ and all $\vU\in\sO$.
\end{rem}

\begin{rem}
	For each entropy, the corresponding entropy flux is unique up to an additive constant.
\end{rem}

\begin{rem}
	Convexity of a $C^\infty$ function is equivalent to the fact that its Hessian is positive semi-definite. Hence a companion is an entropy if and only if its Hessian is positive semi-definite. 
\end{rem}

Let us consider the following example which yields entropies that are not very useful, see Proposition~\ref{prop:weak-linentropies} below.

\begin{ex}\label{ex:entropies}
	Let $\va\in\R^m$ and $b\in \R$. Then $\eta\in C^\infty (\sO)$ defined by $\eta(\vU) = \va\cdot \vU + b$ is an entropy where the corresponding entropy flux $\vq\in C^\infty(\sO;\R^n)$ reads $q_k(\vU)= \va\cdot \vF_k(\vU)$ for $k=1,...,n$.
\end{ex} 

\begin{proof}
	 One easily verifies that $\Grad_\vU \eta (\vU)= \va^\trans$ and hence $\Hess_\vU\eta(\vU) = \mZ_m$, i.e. $\eta$ is convex. Furthermore we have $\Grad_\vU q_k(\vU) = \va^\trans \cdot \Grad_\vU \vF_k(\vU)$. Hence we deduce that \eqref{eq:companion-cond} holds.
\end{proof}

\begin{prop}\label{prop:companion-cons-law}
	Every strong solution fulfills the companion conservation law
	\begin{equation}
		\partial_t \eta(\vU) + \Div \vq(\vU) = 0 \label{eq:companion-cons-law}
	\end{equation} 
	for all $(t,\vx)\in(0,T)\times \Omega$.
\end{prop}

\begin{proof}
	The claim simply follows from the chain rule and condition \eqref{eq:companion-cond}. Indeed we have
	\begin{align*} 
		\partial_t \eta(\vU) + \Div \vq(\vU) &= \Grad_{\vU} \eta(\vU) \cdot \partial_t \vU + \sum_{k=1}^n \Grad_{\vU} q_k(\vU) \cdot \partial_k \vU \\
		&= \Grad_{\vU} \eta(\vU) \cdot \partial_t \vU + \sum_{k=1}^n \Grad_{\vU} \eta(\vU) \cdot \Grad_{\vU} \vF_k(\vU) \cdot \partial_k \vU \\
		&= \Grad_{\vU} \eta(\vU) \cdot \Big( \partial_t \vU +\Div \mF(\vU)\Big) \\
		&= 0 \ed
	\end{align*} 
\end{proof} 

From the integrability conditions we obtain an equivalent condition for $\eta$ to be a companion. To use the integrability conditions we need to assume that $\sO$ is simply connected and $\mF\in C^2(\sO;\R^{m\times n})$. Note that this is not a severe restriction as the sets $\sO$ for the particular examples of conservation laws that appear in this book -- i.e. the isentropic and full compressible Euler equations -- are simply connected, see \eqref{eq:baro-O} and \eqref{eq:full-O}.

\begin{prop} \label{prop:companion-int-bed}
	Let $\sO$ be simply connected and $\mF\in C^2(\sO;\R^{m\times n})$. The function $\eta\in C^\infty(\sO)$ is a companion if and only if the matrices\footnote{We will later on drop ``$(\vU)$'' for convenience.} $$\Hess_\vU \eta(\vU) \cdot \Grad_\vU \vF_k(\vU)$$ are symmetric for all $k=1,...,n$ and all $\vU\in \sO$.
\end{prop}

\begin{proof}
	The integrability conditions (Proposition~\ref{prop:not-int-cond}) state that the existence of each $q_k\in C^\infty(\sO)$ is equivalent to 
	$$
		\partial_{U_i} [\Grad_\vU\eta \cdot \Grad_\vU \vF_k]_j = \partial_{U_j} [\Grad_\vU\eta \cdot \Grad_\vU \vF_k]_i
	$$
	for all $i,j=1,...,m$ and all $\vU\in \sO$. Computing the derivatives on both sides we obtain
	$$
		\sum_{\ell=1}^m \partial_{U_i} \partial_{U_\ell}\eta \ \ \partial_{U_j} F_{\ell k} + \sum_{\ell=1}^m \partial_{U_\ell}\eta \ \ \partial_{U_i} \partial_{U_j} F_{\ell k} = \sum_{\ell=1}^m \partial_{U_j} \partial_{U_\ell}\eta \ \ \partial_{U_i} F_{\ell k} + \sum_{\ell=1}^m \partial_{U_\ell}\eta \ \ \partial_{U_j} \partial_{U_i} F_{\ell k}\ed
	$$
	The terms containing the second order derivatives of $\mF$ cancel and we end up with
	$$
		\sum_{\ell=1}^m [\Hess_\vU\eta]_{i\ell}\  [\Grad_\vU \vF_k]_{\ell j} = \sum_{\ell=1}^m 	[\Hess_\vU\eta]_{j\ell}\  [\Grad_\vU \vF_k]_{\ell i} \ed
	$$
	Hence $\eta\in C^\infty(\sO)$ is a companion if and only if 
	$$
		[\Hess_\vU\eta \cdot \Grad_\vU \vF_k]_{i j} = [\Hess_\vU\eta \cdot \Grad_\vU \vF_k]_{j i} \ed
	$$
	for all $i,j=1,...,m$, all $k=1,...,n$ and all $\vU\in \sO$. 
\end{proof}

\begin{rem}
	Note that Proposition~\ref{prop:companion-int-bed} yields $\frac{m(m-1)}{2}$ conditions for each $k=1,...,n$ since this is the number of components in the upper or lower triangular of an $(m\times m)$-matrix. Hence altogether there are $\frac{nm(m-1)}{2}$ conditions. If $m=1$ and $n$ arbitrary, i.e. we deal with a scalar conservation law in arbitrary space dimensions, the number of conditions is 0. This can be also observed as a $(1\times 1)$-matrix is always symmetric. Thus every $\eta\in C^\infty(\sO)$ is an entropy. However if we deal with systems in several space dimensions, e.g. the compressible Euler equations \eqref{eq:baro-euler-pv-dens}, \eqref{eq:baro-euler-pv-mom} or \eqref{eq:full-euler-pv-dens} - \eqref{eq:full-euler-pv-en}, one can check that there are more conditions than unknowns. In other words the problem of finding entropies is formally overdetermined. We will however show that there still exist entropies albeit there are not many.
\end{rem}

\section{Admissible Weak Solutions} \label{sec:conslaws-admissibility}

We know from simple examples of hyperbolic conservation laws\footnote{E.g. the \emph{Burgers' equation}, a scalar conservation law in one space dimension.} that even for smooth initial data, strong solutions do not exist for all times. Thus we must stick to looking for weak solutions. On the other hand again simple examples of hyperbolic conservation laws show that there might be infinitely many weak solutions. To overcome this issue on non-uniqueness one imposes admissibility criteria to single out the \emph{relevant} weak solutions. What is meant by a relevant solution depends on the particular process in nature which is modelled by the conservation law. 

In the sequel we propose a heuristical derivation of a universal and commonly used admissibility criterion. Assume that a natural process is modelled by a conservation law. Then effects that can be modelled by higher order partial derivatives are neglected. This makes sense if those effects play only a minor role. In the case of the Euler equations one could think of viscosity as such a neglected effect. By considering this neglect as a continuous process, the conservation law can be viewed as a ``limit'' of a higher order partial differential equation, e.g. 
\begin{equation} \label{eq:conslaws-viscous}
	\partial_t \vU + \Div \mF(\vU) = \ep \Lap \vU \ed
\end{equation} 

Assume that for all $\ep>0$ there exists a function $\vU^\ep \in C^2\big((0,T)\times \Omega;\R^m\big)$ which solves \eqref{eq:conslaws-viscous} and suppose furthermore that the $\vU^\ep$ converge\footnote{Assume for simplicity that $\vU^\ep \to \vU$ strongly in $L^\infty$. Note again that we are merely justifying an admissibility criterion \emph{heuristically}.} to some $\vU$ as $\ep\to 0$. From \eqref{eq:conslaws-viscous} we deduce for any entropy pair $(\eta,\vq)$ that
\begin{align*}
	\partial_t \eta(\vU^\ep) + \Div \vq(\vU^\ep) &= \Grad_\vU \eta(\vU^\ep) \cdot \partial_t\vU^\ep + \sum_{k=1}^n \Grad_\vU q_k(\vU^\ep) \cdot \partial_k \vU^\ep \\
	&= \Grad_\vU \eta(\vU^\ep) \cdot \left( \partial_t\vU^\ep + \sum_{k=1}^n \Grad_\vU \vF_k(\vU^\ep) \cdot \partial_k \vU^\ep \right) \\
	&= \Grad_\vU \eta(\vU^\ep) \cdot \Big( \partial_t\vU^\ep + \Div \mF(\vU^\ep) \Big) \\
	&= \ep \Grad_\vU \eta(\vU^\ep) \cdot \Lap\vU^\ep \ec
\end{align*}
where we have used \eqref{eq:companion-cond}. Since 
\begin{align*}
	\Lap \eta(\vU^\ep) &= \sum_{k=1}^n \sum_{i=1}^m \partial_k \Big(\partial_{U_i} \eta(\vU^\ep)\ \partial_k U_i^\ep\Big) \\
	&= \sum_{k=1}^n \sum_{i,j=1}^m \partial_{U_i}\partial_{U_j} \eta(\vU^\ep)\ \partial_k U_i^\ep \ \partial_k U_j^\ep + \Grad_\vU \eta(\vU^\ep) \cdot \Lap \vU^\ep \\
	&= \tr\Big(\Grad^\trans \vU^\ep \cdot \Hess_\vU\eta(\vU^\ep) \cdot \Grad \vU^\ep\Big) +  \Grad_\vU \eta(\vU^\ep) \cdot \Lap \vU^\ep \ec 
\end{align*}
this implies 
\begin{equation*}
	\partial_t \eta(\vU^\ep) + \Div \vq(\vU^\ep) = \ep \Lap \eta(\vU^\ep) - \ep\, \tr \Big(\Grad^\trans \vU^\ep \cdot \Hess_\vU \eta(\vU^\ep) \cdot \Grad \vU^\ep\Big) \ed
\end{equation*}
Let us multiply the latter equation with a non-negative test function $\varphi\in \Cc\big((0,T)\times \Omega;\R^+_0\big)$ and apply the Divergence Theorem (Proposition~\ref{prop:not-divergence}), to obtain the weak formulation
\begin{align}
	&\int_0^T \int_{\Omega} \Big(\eta(\vU^\ep) \,\partial_t\varphi + \vq(\vU^\ep) \cdot \Grad \varphi \Big)\dx \dt \notag\\
	&= \ep \int_0^T \int_{\Omega} \tr\Big(\Grad^\trans \vU^\ep \cdot \Hess_\vU\eta(\vU^\ep) \cdot \Grad \vU^\ep\Big)\varphi \dx \dt - \ep \int_0^T \int_{\Omega} \eta(\vU^\ep)\ \Lap\varphi \dx \dt \ed \label{eq:temp-visc}
\end{align}
Note that the boundary terms disappear due to the compact support of the test function $\varphi$. 

The aim is now to consider the limit of \eqref{eq:temp-visc} as $\ep\to 0$. The term
$$
\ep \int_0^T \int_{\Omega} \tr\Big(\Grad^\trans \vU^\ep \cdot \Hess_\vU\eta(\vU^\ep) \cdot \Grad \vU^\ep\Big)\varphi \dx \dt
$$
causes troubles as the derivative $\Grad \vU^\ep$ may blow up as $\ep\to 0$. However we get rid of this problematic term via the estimate
$$
\ep \int_0^T \int_{\Omega} \tr\Big(\Grad^\trans \vU^\ep \cdot \Hess_\vU\eta(\vU^\ep) \cdot \Grad \vU^\ep\Big)\varphi \dx \dt \geq 0\ec
$$
which holds due to the convexity of $\eta$, i.e. the Hessian of $\eta$ is positive semi-definite, and the fact that $\varphi$ takes non-negative values. 

Hence we end up with 
$$
	\int_0^T \int_{\Omega} \Big(\eta(\vU^\ep) \,\partial_t\varphi + \vq(\vU^\ep) \cdot \Grad \varphi \Big)\dx \dt \geq - \ep \int_0^T \int_{\Omega} \eta(\vU^\ep)\ \Lap\varphi \dx \dt  
$$
whose limit as $\ep\to 0$ is 
$$
	\int_0^T \int_{\Omega} \Big(\eta(\vU) \,\partial_t\varphi + \vq(\vU) \cdot \Grad \varphi \Big)\dx \dt \geq 0 \ed
$$

This motivates the following definition. 

\begin{defn} \label{defn:admissibility} 
	A weak solution $\vU\in L^\infty\big((0,T)\times \Omega;\sO\big)$ of the initial (boundary) value problem \eqref{eq:cons-law-strong}, \eqref{eq:conslaws-ic}, \eqref{eq:conslaws-bc} is called \emph{admissible weak solution} or \emph{entropy solution} if
	\begin{align} 
		\int_0^T \int_{\Omega} \Big(\eta(\vU) \  \partial_t\varphi + \vq(\vU) \cdot \Grad \varphi \Big)\dx \dt + \int_\Omega \eta(\vU_\init) \varphi(0,\cdot) \dx \qquad & \notag \\
		-  \int_0^T \int_{\partial\Omega} q_{\partial \Omega} \varphi \dS_\vx \dt &\geq 0   \label{eq:admissibility} 
	\end{align}
	holds for all entropy pairs $(\eta,\vq)$ and all non-negative test functions $\varphi\in \Cc\big([0,T)\times \closure{\Omega};\R^+_0\big)$.
	Here the boundary values $q_{\partial\Omega} = \vq \cdot \vn\big|_{\partial\Omega}$ are given by the boundary condition \eqref{eq:conslaws-bc}. Note that again the boundary term 
	$$
	\int_0^T \int_{\partial\Omega} q_{\partial \Omega} \varphi \dS_\vx \dt
	$$
	disappears if $\Omega=\R^n$.
\end{defn}

As in the context of weak solutions, one also says that 
\begin{equation} \label{eq:admissibility-strong}
	\partial_t \eta(\vU) + \Div \vq(\vU) \leq 0
\end{equation}
holds in the sense of distributions if \eqref{eq:admissibility} holds for all non-negative test functions $\varphi\in \Cc\big([0,T)\times \closure{\Omega};\R^+_0\big)$. The inequality \eqref{eq:admissibility-strong} or \eqref{eq:admissibility} is called \emph{entropy inequality}. If the entropy inequality holds for all entropy pairs $(\eta,\vq)$, we say that the \emph{entropy criterion} is satisfied.

\begin{prop}
	Every strong solution $\vU\in C^1\big([0,T)\times \closure{\Omega};\sO\big)$ of the initial (boundary) value problem \eqref{eq:cons-law-strong}, \eqref{eq:conslaws-ic}, \eqref{eq:conslaws-bc} is an admissible weak solution.
\end{prop}

\begin{proof}
	We already proved in Proposition~\ref{prop:weak-vs-strong2} that $\vU$ is a weak solution. Furthermore \eqref{eq:admissibility} holds as equality for all entropy pairs $(\eta,\vq)$ and all non-negative test functions $\varphi\in \Cc\big([0,T)\times \closure{\Omega};\R^+_0\big)$. Indeed this can be simply proven using Proposition~\ref{prop:companion-cons-law} and the Divergence Theorem (Propostion \ref{prop:not-divergence}).
\end{proof}

\begin{rem} 
	The process we described in the beginning of this section is called \emph{vanishing viscosity method}. We used it in order to motivate the entropy criterion: Under very strong assumptions we proved that the solution obtained as the vanishing viscosity limit\footnote{The existence of this limit was one of the strong assumptions.} satisfies the entropy criterion. The vanishing viscosity method itself could be also used as a criterion to select relevant solutions. This is in fact done for scalar conservation laws where \name{Kru\v{z}kov}~\cite{Kruzkov70} showed that entropy solutions exist, are unique and coincide with the solution obtained by the vanishing viscosity method. However it should be noticed that as soon as systems of conservation laws are considered, at the current state of research the assumptions used above are in general out of reach. 
\end{rem}

\begin{rem} 
	In Chapter~\ref{chap:euler} we compute the entropies for Euler and we observe that the corresponding entropy inequalities coincide with general principles in thermodynamics. This can be viewed as a second explanation why the solutions ruled out by the entropy condition are physically irrelevant. 
\end{rem}

When solving the initial (boundary) value problem \eqref{eq:cons-law-strong}, \eqref{eq:conslaws-ic}, \eqref{eq:conslaws-bc}, our objective is to find admissible weak solutions in the sense of Definitions \ref{defn:weaksol} and \ref{defn:admissibility}. In order to search for admissible weak solutions, one has to find all entropy pairs $(\eta,\vq)$, since one needs to check if the entropy inequality holds for all entropy pairs.

Let us finish this chapter with the following propostion.

\begin{prop} \label{prop:weak-linentropies} 
	If a function $\vU\in L^\infty\big((0,T)\times \Omega;\sO\big)$ is a weak solution of the initial (boundary) value problem \eqref{eq:cons-law-strong}, \eqref{eq:conslaws-ic}, \eqref{eq:conslaws-bc}, then the entropy inequality \eqref{eq:admissibility-strong} holds for all entropies of the form considered in Example~\ref{ex:entropies} in the sense of distributions.
\end{prop}

\begin{proof}
Let $\vU\in L^\infty\big((0,T)\times \Omega;\sO\big)$ be a weak solution of the initial (boundary) value problem \eqref{eq:cons-law-strong}, \eqref{eq:conslaws-ic}, \eqref{eq:conslaws-bc} and consider an entropy pair of the form
\begin{align*}
	\eta(\vU) &= \va\cdot \vU + b\ec \\
	q_k(\vU) &= \va\cdot \vF_k(\vU) \qquad \text{ for }k=1,...,n \ed
\end{align*}
Let furthermore $\varphi\in \Cc\big([0,T)\times \closure{\Omega};\R^+_0\big)$ be an arbitrary non-negative test function. Our objective is to show that \eqref{eq:admissibility} holds for $(\eta,\vq)$ and $\varphi$. Setting $\vpsi := \varphi \va$ in \eqref{eq:conslaws-ivp-weak} we obtain
\begin{align*} 
	0 &=\int_0^T \int_{\Omega} \Big(\va\cdot \vU \, \partial_t\varphi + \sum_{k=1}^n \big(\va \cdot\vF_k(\vU)\big)  \partial_k \varphi \Big)\dx \dt + \int_\Omega \va\cdot\vU_\init \, \varphi(0,\cdot) \dx   \\
	&\qquad - \int_0^T \int_{\partial\Omega} \va \cdot \vF_{\partial \Omega} \,\varphi \dS_\vx \dt \\
	&= \int_0^T \int_{\Omega} \Big(\eta(\vU) \  \partial_t\varphi + \vq(\vU) \cdot \Grad \varphi \Big)\dx \dt + \int_\Omega \eta(\vU_\init) \varphi(0,\cdot) \dx \\
	& \qquad  -  \int_0^T \int_{\partial\Omega} q_{\partial \Omega} \varphi \dS_\vx \dt  \ec
\end{align*}
where we have used the Divergence Theorem (Proposition~\ref{prop:not-divergence}), to show that 
$$
b \int_0^T \int_\Omega \partial_t\varphi \dx\dt + b \int_\Omega \varphi(0,\cdot)\dx = 0 \ed
$$
Hence \eqref{eq:admissibility} is satisfied as equality.

\end{proof}

\begin{rem} 
	The converse of the statement in Proposition~\ref{prop:weak-linentropies} is also true but not needed in this book.
\end{rem}

\begin{rem}
	From Proposition~\ref{prop:weak-linentropies} we deduce that the entropies considered in Example~\ref{ex:entropies} are not useful as they do not rule out any non-physical weak solution.
\end{rem}

%% file: DissertationCh3.tex
\chapter{The Euler Equations as a Hyperbolic System of Conservation Laws} \label{chap:euler} 
\chaptermark{The Euler Equations as a Conservation Law} 


In this book we deal with two examples of systems of conservation laws, namely the barotropic compressible Euler equations \eqref{eq:baro-euler-pv-dens}, \eqref{eq:baro-euler-pv-mom} and the full compressible Euler equations \eqref{eq:full-euler-pv-dens} - \eqref{eq:full-euler-pv-en}, both of which are introduced in Chapter~\ref{chap:intro}. In this chapter we show that those two systems are indeed hyperbolic conservation laws as treated in Chapter~\ref{chap:conslaws}. Furthermore we compute the entropies\footnote{Here we mean ``mathematical'' entropies in the sense of Definition~\ref{defn:comp-entropy}~\ref{item:entropy}.} for both systems, where we present many details as we couldn't find these calculations in the literature. It turns out that the only relevant entropies are the energy (in the case of barotropic Euler) and the physical entropy (for full Euler). Finally we recap the definition of an admissible weak solution for both systems. 

The appropriate formulation of the Euler systems \eqref{eq:baro-euler-pv-dens}, \eqref{eq:baro-euler-pv-mom} and \eqref{eq:full-euler-pv-dens} - \eqref{eq:full-euler-pv-en} is the one using conserved variables, i.e. density $\rho$, momentum\footnote{More precisely $\vm$ denotes the \emph{momentum density} and $E$ the \emph{energy density}.} $\vm$ and energy $E$, instead of the primitive variables, i.e. $\rho$, velocity $\vu$ and pressure $p$. Hence the barotropic Euler system \eqref{eq:baro-euler-pv-dens}, \eqref{eq:baro-euler-pv-mom} is written as
\begin{align}
	\partial_t \rho + \Div \vm & = 0 \ec \label{eq:baro-euler-cv-dens} \\
	\partial_t \vm + \Div \left(\frac{\vm\otimes\vm}{\rho}\right) + \Grad p(\rho) & = \vz \ec \label{eq:baro-euler-cv-mom}
\end{align}
where now the unknowns are the density $\rho$ and the momentum $\vm$, both of which are functions of time $t\in [0,T)$ and space $\vx\in \Omega$ and take values in $\R^+$ and $\R^n$ respectively.
Note that since we exclude vacuum, $\rho$ in the denominator cannot cause any problems. 

Similarly, the full Euler system \eqref{eq:full-euler-pv-dens} - \eqref{eq:full-euler-pv-en} is rewritten as follows:
\begin{align}
	\partial_t \rho + \Div \vm & = 0 \es \label{eq:full-euler-cv-dens} \\
	\partial_t \vm + \Div \left(\frac{\vm\otimes\vm}{\rho}\right) + \Grad p(\rho,\vm,E) & = \vz \es \label{eq:full-euler-cv-mom} \\
	\partial_t E + \Div \left(\Big(E+ p(\rho,\vm,E)\Big)\frac{\vm}{\rho}\right) & = 0 \ed \label{eq:full-euler-cv-en}
\end{align} 
Now the unknowns are the density $\rho$, the momentum $\vm$ and the energy $E$, which are functions of $(t,\vx)$ and take values in $\R^+$, $\R^n$ and $\R^+$, respectively. In this notation the incomplete equation of state \eqref{eq:incomp-EOS-pv} turns into
\begin{equation} \label{eq:incomp-EOS-cv}
	p(\rho,\vm,E) = (\gamma-1)\left( E - \half \frac{|\vm|^2}{\rho} \right) \ed
\end{equation}

\section{Barotropic Euler System} \label{sec:euler-barotropic} 

The barotropic Euler system \eqref{eq:baro-euler-cv-dens}, \eqref{eq:baro-euler-cv-mom} can be written in the manner of Chapter~\ref{chap:conslaws}, where 
\begin{align}
	\sO &= \R^+ \times \R^n \ec \label{eq:baro-O}\\
	\vU &= \left(\begin{array}{c}
		\rho \\ \vm
	\end{array}\right)\ \in\ \R^{1+n} \ec \label{eq:baro-U} \\
	\mF(\vU) &= \left(\begin{array}{c}
		\vm^\trans \\ \frac{\vm\otimes\vm}{\rho} + p(\rho) \id 
	\end{array}\right) \ \in\ \R^{(1+n)\times n} \ed \label{eq:baro-F}
\end{align}
The number of unknowns is $m=n+1$.

A straighforward computation yields  
\begin{equation} \label{eq:baro-flux-jac}
	\Grad_\vU \vF_k = \left(\begin{array}{cc}
		0 & \ve_k^\trans \\
		-\frac{m_k\,\vm}{\rho^2} + p'(\rho)\ve_k & \frac{\vm}{\rho} \ve_k^\trans + \frac{m_k}{\rho} \id
	\end{array}\right)
\end{equation}
for $k=1,...,n$, where $\ve_k$ denotes the $k$-th standard basis vector, see Section \ref{sec:not-vecmat}.

\subsection{Hyperbolicity} 

As already mentioned, the hyperbolicity of the Euler system \eqref{eq:baro-euler-cv-dens}, \eqref{eq:baro-euler-cv-mom} is not really needed in this book. However for completeness let us indicate how hyperbolicity can be shown. 

Let $\vnu\in\sphere^{n-1}$ be arbitrary. According to \eqref{eq:baro-flux-jac} the matrix we are interested in reads
\begin{equation} \label{eq:baro-flux-jac-sum}
\sum_{k=1}^n \nu_k \Grad_{\vU}\vF_k = \left(\begin{array}{cc}
0 & \vnu^\trans \\
-\frac{(\vm\cdot\vnu)\vm}{\rho^2} + p'(\rho)\vnu & \frac{\vm}{\rho} \vnu^\trans + \frac{\vm\cdot\vnu}{\rho} \id
\end{array}\right) \ed
\end{equation}
Let $\va_1,...,\va_{n-1}\in\sphere^{n-1}$ be $n-1$ linearly independent vectors which are perpendicular to $\vnu$. It is a matter of straightforward calculation that the $1+n$ vectors
$$
\left(\begin{array}{c}
	1 \\ \frac{\vm}{\rho} - \sqrt{p'(\rho)} \vnu 
\end{array}\right) \ec  \left(\begin{array}{c}
	0 \\ \va_1 
\end{array}\right) \ec  ... \ec \left(\begin{array}{c}
	0 \\ \va_{n-1} 
\end{array}\right)\ec \left(\begin{array}{c}
	1 \\ \frac{\vm}{\rho} + \sqrt{p'(\rho)} \vnu 
\end{array}\right)
$$
are eigenvectors of the matrix \eqref{eq:baro-flux-jac-sum}, where the corresponding eigenvalues read 
$$
	\frac{\vm\cdot \vnu}{\rho} - \sqrt{p'(\rho)} \ \ec\  \frac{\vm\cdot \vnu}{\rho} \ec\  ... \ \ec\  \frac{\vm\cdot \vnu}{\rho} \ \ec\  \frac{\vm\cdot \vnu}{\rho} + \sqrt{p'(\rho)}
$$
respectively. Since $p'(\rho)>0$ for all $\rho\in \R^+$, the eigenvectors written above are linearly independent. Hence the barotropic Euler system \eqref{eq:baro-euler-cv-dens}, \eqref{eq:baro-euler-cv-mom} is hyperbolic in the sense of Definition~\ref{defn:hyperbolicity}.

\subsection{Entropies} \label{subsec:euler-baro-entropies}
Our goal is now to find the entropies for the barotropic Euler system \eqref{eq:baro-euler-cv-dens}, \eqref{eq:baro-euler-cv-mom}. To this end we make use of Proposition~\ref{prop:companion-int-bed}. To avoid problems, we assume here that $p\in C^\infty(\R^+_0)$. This is true for an isentropic flow where the pressure is given by \eqref{eq:isentropic-EOS}. Note that whenever the entropy condition is relevant (only in Chapter~\ref{chap:appl-riemann}), we consider the isentropic pressure law, which justifies the assumption above. From $p\in C^\infty$ we deduce $\mF\in C^\infty(\sO;\R^{m\times n})$, see \eqref{eq:baro-F}, in particular $\mF\in C^2$. Furthermore $\sO$ is simply connected, see \eqref{eq:baro-O}.

\begin{prop} \label{prop:companion-baro}
	The function $\eta\in C^\infty(\sO)$ is a companion for barotropic Euler \eqref{eq:baro-euler-cv-dens} - \eqref{eq:baro-euler-cv-mom} if and only if its Hessian takes the form 
	\begin{equation} \label{eq:baro-hess-eta}
		\Hess_\vU \eta(\rho,\vm) = f(\rho,\vm) \left(\begin{array}{cc}
			p'(\rho) + \frac{|\vm|^2}{\rho^2} & -\frac{\vm^\trans}{\rho} \\
			-\frac{\vm}{\rho} & \id_n
		\end{array}\right)
	\end{equation}
	with a function $f\in C^\infty(\sO)$.
\end{prop}

\begin{proof}
With \eqref{eq:baro-flux-jac} we compute 
\begin{align*}
	&\Hess_\vU \eta \cdot \Grad_\vU \vF_k \\
	&= \left(\begin{array}{cc}
		\partial_{\rho}\partial_{\rho}\eta & \Grad_{\vm}\partial_{\rho}\eta \\
		\Grad_{\vm}^\trans\partial_{\rho}\eta & \Hess_\vm \eta 
	\end{array}\right) \cdot \left(\begin{array}{cc}
		0 & \ve_k^\trans \\
		-\frac{m_k\,\vm}{\rho^2} + p'(\rho)\ve_k & \frac{\vm}{\rho} \ve_k^\trans + \frac{m_k}{\rho} \id
	\end{array}\right) \\
	&= \left(\begin{array}{cc}
		- \Grad_{\vm}\partial_{\rho}\eta \cdot \frac{m_k\,\vm}{\rho^2}+ \partial_{m_k}\partial_{\rho}\eta \  p'(\rho) & \left(\partial_{\rho}\partial_{\rho} \eta +\Grad_{\vm}\partial_{\rho}\eta\cdot\frac{\vm}{\rho}\right)\ve_k^\trans  + \Grad_{\vm}\partial_{\rho}\eta \ \frac{m_k}{\rho}\\
		- \Hess_\vm\eta \cdot \frac{m_k\,\vm}{\rho^2} + \Grad_{\vm}^\trans\partial_{m_k} \eta \ p'(\rho) & \Grad_{\vm}^\trans\partial_{\rho}\eta \cdot \ve_k^\trans + \Hess_\vm \eta \cdot \frac{\vm}{\rho} \ve_k^\trans + \Hess_\vm \eta \ \frac{m_k}{\rho} 
	\end{array}\right) \ed
\end{align*}
Hence according to Proposition~\ref{prop:companion-int-bed}, $\eta$ is a companion if and only if both
$$
	\left[\left(\partial_{\rho}\partial_{\rho} \eta +\Grad_{\vm}\partial_{\rho}\eta\cdot\frac{\vm}{\rho}\right)\ve_k^\trans  + \Grad_{\vm}\partial_{\rho}\eta \ \frac{m_k}{\rho}\right]_i = \left[ - \Hess_\vm\eta \cdot \frac{m_k\,\vm}{\rho^2} + \Grad_{\vm}^\trans\partial_{m_k} \eta \ p'(\rho) \right]_i 
$$
for all $i,k=1,...,n$ and 
$$
	\left[\Grad_{\vm}^\trans\partial_{\rho}\eta \cdot \ve_k^\trans + \Hess_\vm \eta \cdot \frac{\vm}{\rho} \ve_k^\trans + \Hess_\vm \eta \ \frac{m_k}{\rho}\right]_{ij} = \left[\Grad_{\vm}^\trans\partial_{\rho}\eta \cdot \ve_k^\trans + \Hess_\vm \eta \cdot \frac{\vm}{\rho} \ve_k^\trans + \Hess_\vm \eta \ \frac{m_k}{\rho} \right]_{ji} 
$$
for all $i,j,k=1,...,n$.

This can be translated into
\begin{equation} \label{eq:baro-entropy-A}
	\left(\partial_{\rho}\partial_{\rho} \eta +\Grad_{\vm}\partial_{\rho}\eta\cdot\frac{\vm}{\rho}\right) \delta_{ik}  + \partial_{m_i}\partial_{\rho}\eta \ \frac{m_k}{\rho} = - \Grad_{\vm}\partial_{m_i}\eta \cdot \frac{m_k\,\vm}{\rho^2} + \partial_{m_i}\partial_{m_k} \eta \ p'(\rho)
\end{equation}
for all $i,k=1,...,n$ and 
\begin{equation} \label{eq:baro-entropy-B}
	\partial_{m_i}\partial_{\rho}\eta \ \delta_{jk} + \Grad_\vm\partial_{m_i} \eta \cdot \frac{\vm}{\rho} \ \delta_{jk} + \partial_{m_i}\partial_{m_j} \eta \ \frac{m_k}{\rho} = \partial_{m_j}\partial_{\rho}\eta \ \delta_{ik} + \Grad_\vm\partial_{m_j} \eta \cdot \frac{\vm}{\rho} \delta_{ik} + \partial_{m_i}\partial_{m_j} \eta \ \frac{m_k}{\rho}
\end{equation}
for all $i,j,k=1,...,n$.

On the one hand, if $\Hess_\vU \eta$ is of the form \eqref{eq:baro-hess-eta} then simple calculations show that \eqref{eq:baro-entropy-A} and \eqref{eq:baro-entropy-B} are true.

For the converse, assume that \eqref{eq:baro-entropy-A} and \eqref{eq:baro-entropy-B} hold. From \eqref{eq:baro-entropy-A} with $i\neq k$ we obtain
$$
\partial_{m_i}\partial_{\rho}\eta \ \frac{m_k}{\rho} = - \Grad_{\vm}\partial_{m_i}\eta \cdot \frac{m_k\,\vm}{\rho^2} + \partial_{m_i}\partial_{m_k} \eta \ p'(\rho) \qquad \text{ for }i\neq k \ec
$$
and assuming $i\neq j=k$ we get from \eqref{eq:baro-entropy-B}
\begin{equation} \label{eq:baro-entropy-temp1}
	\partial_{m_i}\partial_{\rho}\eta + \Grad_\vm\partial_{m_i} \eta \cdot \frac{\vm}{\rho} = 0  \qquad \text{ for all }i=1,...,n\ed
\end{equation}
Hence we find
$$
\partial_{m_i}\partial_{m_k} \eta \ p'(\rho) = \left(\partial_{m_i}\partial_{\rho}\eta + \Grad_\vm\partial_{m_i} \eta \cdot \frac{\vm}{\rho}\right) \frac{m_k}{\rho} = 0
$$
for $i\neq k$ and therefore
\begin{equation} \label{eq:baro-entropy-mimk}
	\partial_{m_i}\partial_{m_k} \eta = 0 \qquad \text{ for }i\neq k\ed
\end{equation}

Next setting $i=k$ in \eqref{eq:baro-entropy-A} and using \eqref{eq:baro-entropy-temp1} we end up with
\begin{align*}
	\partial_{m_i}\partial_{m_i} \eta \ p'(\rho) &= \partial_{\rho}\partial_{\rho} \eta +\Grad_{\vm}\partial_{\rho}\eta\cdot\frac{\vm}{\rho} + \left(\partial_{m_i}\partial_{\rho}\eta + \Grad_{\vm}\partial_{m_i}\eta \cdot \frac{\vm}{\rho} \right) \frac{m_i}{\rho} \\
	&= \partial_{\rho}\partial_{\rho} \eta +\Grad_{\vm}\partial_{\rho}\eta\cdot\frac{\vm}{\rho}\ed
\end{align*}
From this we observe that $\partial_{m_i}\partial_{m_i} \eta$ is independent on $i$ and we set 
\begin{equation} \label{eq:baro-entropy-mimi}
f(\rho,\vm):=\partial_{m_i}\partial_{m_i} \eta(\rho,\vm)
\end{equation}
with any $i=1,...,n$. Furthermore we get 
\begin{equation} \label{eq:baro-entropy-temp3}
\partial_{\rho}\partial_{\rho} \eta = f \, p'(\rho) - \Grad_{\vm}\partial_{\rho}\eta\cdot\frac{\vm}{\rho} \ec
\end{equation}
and, using \eqref{eq:baro-entropy-mimk} in \eqref{eq:baro-entropy-temp1},
\begin{equation} \label{eq:baro-entropy-mirho}
\partial_{m_i}\partial_{\rho}\eta  = -f \, \frac{m_i}{\rho} \qquad \text{ for all }i=1,...,n\ed
\end{equation}
Plugging this into \eqref{eq:baro-entropy-temp3} we find
\begin{equation} \label{eq:baro-entropy-rhorho}
\partial_{\rho}\partial_{\rho} \eta = f \left( p'(\rho) + \frac{|\vm|^2}{\rho^2} \right) \ed
\end{equation}
Collecting \eqref{eq:baro-entropy-rhorho}, \eqref{eq:baro-entropy-mirho}, \eqref{eq:baro-entropy-mimi} and \eqref{eq:baro-entropy-mimk} we conclude that the Hessian of $\eta$ is of the form \eqref{eq:baro-hess-eta}. 
\end{proof}

Proposition~\ref{prop:companion-baro} allows to find the entropies for the barotropic Euler system. Keep in mind that we assume in this subsection that $p \in C^\infty(\R^+_0)$. Hence the pressure potential $P$ as introduced in \eqref{eq:pressure-potential} is also of class $C^\infty$.

\begin{prop} \label{prop:baro-entropies}
	The entropies for the barotropic Euler system  \eqref{eq:baro-euler-cv-dens}, \eqref{eq:baro-euler-cv-mom} are all functions of the form
	\begin{equation} \label{eq:baro-entropies}
		\eta(\rho,\vm) = a \left(\frac{|\vm|^2}{2\rho} + P(\rho)\right) + \vm\cdot \vb + c\rho + d  \ec
	\end{equation}
	with $a\geq 0$, $\vb\in \R^n$ and $c,d\in \R$. The corresponding entropy fluxes are given by 
	\begin{equation} \label{eq:baro-entr-fluxes}
		\vq (\rho,\vm) = a \left( \frac{|\vm|^2}{2\rho} + P(\rho) + p(\rho)\right) \frac{\vm}{\rho} + (\vm\cdot \vb) \frac{\vm}{\rho} + p(\rho)\vb  + c\vm \ed
	\end{equation}
\end{prop}

\begin{proof} 
	A straightforward computation shows that $(\eta,\vq)$ as given by \eqref{eq:baro-entropies} and \eqref{eq:baro-entr-fluxes} satisfy \eqref{eq:companion-cond} for all $a,c,d\in \R$ and $\vb\in \R^n$, where one has to use that \eqref{eq:pressure-potential} implies 
	$$
	\rho P'(\rho)=P(\rho) + p(\rho) \ed
	$$
	Hence $\eta$ is a companion and $\vq$ the corresponding flux. Another straightforward computation proves that the Hessian of $\eta$ satisfies \eqref{eq:baro-hess-eta} with $f(\rho,\vm)=\frac{a}{\rho}$. Here one makes use of $P''(\rho)=\frac{p'(\rho)}{\rho}$ which is another consequence of \eqref{eq:pressure-potential}. Now we are ready to prove that $\eta$ is convex, i.e. the Hessian of $\eta$ is positive semi-definite. Indeed for all $\vw\in \R^{1+n}$ we have 
	\begin{align}
		\vw^\trans \cdot \Hess_\vU \eta \cdot \vw &= \vw^\trans \cdot \frac{a}{\rho} \left(\begin{array}{cc}
			p'(\rho) + \frac{|\vm|^2}{\rho^2} & -\frac{\vm^\trans}{\rho} \\
			-\frac{\vm}{\rho} & \id_n
		\end{array}\right) \cdot \vw \notag \\
		&= \frac{a}{\rho} \left( w_t^2\left( p'(\rho) + \frac{|\vm|^2}{\rho^2}\right) - 2 \frac{w_t}{\rho} \vm\cdot \vw_\vx + |\vw_\vx|^2\right) \notag \\
		& = \frac{a}{\rho} w_t^2 p'(\rho) + \frac{a}{\rho} \left| w_t \frac{\vm}{\rho} - \vw_\vx \right|^2 \ \geq \ 0\ec \label{eq:1-temp-euler}
	\end{align}
	because $a\geq 0$. Hence the functions $\eta$ given by \eqref{eq:baro-entropies} are indeed entropies.
	
	It remains to show that all entropies are of the form \eqref{eq:baro-entropies}. Let $\eta$ be an entropy. From Proposition~\ref{prop:companion-baro} we know that there exists a function $f\in C^\infty (\sO)$ such that \eqref{eq:baro-hess-eta} holds. First we observe that $f$ does not depend on $\vm$. Indeed for all $i=1,...,n$ we can find\footnote{Here it is crucial that we are considering the multi-dimensional case, i.e. $n\geq 2$.} $j\neq i$ and hence
	$$
	\partial_{m_i} f = \partial_{m_i} \partial_{m_j} \partial_{m_j} \eta = \partial_{m_j} \partial_{m_i} \partial_{m_j} \eta = 0 \ed
	$$
	Thus we have 
	$$
	\partial_{m_i}\partial_{m_j} \eta = f(\rho)\delta_{ij} \ed
	$$
	This yields that there exists a function $\vb\in C^\infty(\R^+; \R^n)$ such that
	$$
	\partial_{m_i} \eta = m_i f(\rho) + b_i (\rho)
	$$
	for all $i=1,...,n$. With the fact that 
	$$
	\partial_\rho\partial_{m_i} \eta = -\frac{m_i}{\rho} f(\rho)
	$$
	we deduce that 
	$$
	b_i'(\rho) = -m_i \left(\frac{f(\rho)}{\rho} + f'(\rho)\right)
	$$
	for all $i=1,...,n$. Since the left-hand side does not depend on $\vm$, we obtain the ordinary differential equation
	$$
	\frac{f(\rho)}{\rho} + f'(\rho) = 0\ec
	$$
	whose solution is 
	$$
	f(\rho)=\frac{a}{\rho}
	$$
	with an $a\in\R$. Furthermore we obtain that $b_i'(\rho)=0$, i.e. $\vb=\const$.
	
	Hence there exists $g\in C^\infty (\R^+)$ such that 
	$$
	\eta = a \frac{|\vm|^2}{2\rho} + \vm\cdot \vb + g(\rho) \ed
	$$
	With
	$$
	\partial_\rho\partial_\rho \eta = a\frac{|\vm|^2}{\rho^3} + a\frac{p'(\rho)}{\rho}
	$$
	we deduce 
	$$
	g''(\rho) = a\frac{p'(\rho)}{\rho}
	$$
	and thus
	$$
	g(\rho) = a P(\rho) + c\rho + d
	$$
	with $c,d\in \R$.
	
	Hence 
	$$
	\eta =  a \left(\frac{|\vm|^2}{2\rho} + P(\rho)\right) + \vm\cdot \vb + c\rho + d \ec
	$$
	with suitable $a,c,d\in \R$ and $\vb\in \R^n$. Since $\eta$ is an entropy and therefore convex, we deduce that the inequality \eqref{eq:1-temp-euler} holds. Hence $a\geq 0$, which finishes the proof.
\end{proof}

\begin{rem} 
	The entropy given by $a=1$, $\vb=\vz$ and $c=d=0$, i.e. 
	\begin{equation} \label{eq:energy}
		\eta_{\rm energy}(\rho,\vm) = \frac{|\vm|^2}{2\rho} + P(\rho) \ec
	\end{equation}
	represents the \emph{energy}, more precisely the \emph{energy density}, of the fluid. It consists of two parts, the first of which, $\frac{|\vm|^2}{2\rho}$, stands for the kinetic energy, whereas $P(\rho)$ constitutes the internal energy. The corresponding entropy flux reads 
	\begin{equation} \label{eq:energy-flux}
		\vq_{\rm energy}(\rho,\vm) = \left( \frac{|\vm|^2}{2\rho} + P(\rho) + p(\rho)\right) \frac{\vm}{\rho} \ec
	\end{equation}
	and is called \emph{energy flux}. 
\end{rem}

Finally we prove that the consideration of the energy is enough to select admissible weak solutions.

\begin{prop} \label{prop:energy-is-enough} 
	A weak solution $\vU\in L^\infty\big((0,T)\times \Omega;\sO\big)$ of the initial (boundary) value problem \eqref{eq:cons-law-strong}, \eqref{eq:conslaws-ic}, \eqref{eq:conslaws-bc} for the barotropic Euler system, i.e. $\sO$, $\vU$ and $\mF$ are given by \eqref{eq:baro-O} - \eqref{eq:baro-F}, is admissible if and only if \eqref{eq:admissibility-strong} holds for $(\eta_{\rm energy},\vq_{\rm energy})$ in the sense of distributions.
\end{prop}

\begin{proof}
In the sequel when writing that \eqref{eq:admissibility-strong} is satisfied, we always mean that this inequality is satisfied in the sense of distributions.

If a weak solution $\vU\in L^\infty\big((0,T)\times \Omega;\sO\big)$ is admissible, then \eqref{eq:admissibility-strong} holds for all entropy pairs. In particular \eqref{eq:admissibility-strong} is fulfilled for $(\eta_{\rm energy},\vq_{\rm energy})$.

The converse is much more interesting. Let $(\eta,\vq)$ be an arbitrary entropy pair. Our goal is to show that \eqref{eq:admissibility-strong} is true for $(\eta,\vq)$. Due to Proposition~\ref{prop:baro-entropies} $\eta$ and $\vq$ are given by \eqref{eq:baro-entropies} and \eqref{eq:baro-entr-fluxes} with suitable $a\geq 0$, $\vb\in \R^n$ and $c,d\in \R$. Furthermore by assumption, \eqref{eq:admissibility-strong} holds for $(\eta_{\rm energy},\vq_{\rm energy})$. Multiplying by $a\geq 0$, we obtain \eqref{eq:admissibility-strong} for $(a\eta_{\rm energy},a\vq_{\rm energy})$. In addition to that, Proposition~\ref{prop:weak-linentropies} implies that \eqref{eq:admissibility-strong} holds for 
$$
(\til{\eta},\til{\vq}) = \left(\vm\cdot \vb + c\rho+d,(\vm\cdot \vb) \frac{\vm}{\rho} + p(\rho) \vb + c\vm\right)\ed
$$
Adding those two entropy inequalities we obtain \eqref{eq:admissibility-strong} for 
$$
(a\eta_{\rm energy},a\vq_{\rm energy}) + (\til{\eta},\til{\vq}) = (\eta,\vq) \ed
$$ 
This finishes the proof.
\end{proof}

\subsection{Admissible Weak Solutions} \label{subsec:euler-baro-aws}

As pointed out earlier in this book, we consider only the following two cases: 
\begin{itemize}
	\item $\Omega \subsetneq \R^n$ bounded with impermeability boundary condition \eqref{eq:impermeability}; $T\in \R^+$.
	\item $\Omega=\R^2$; $T=\infty$.
\end{itemize}

Let us recall the definition of an admissible weak solution to the corresponding initial (boundary) value problem as introduced in Definitions \ref{defn:weaksol} and \ref{defn:admissibility}. Here we switch back from conserved variables $(\rho,\vm)$ to primitive variables $(\rho,\vu)$. Note that as we exclude vacuum in our considerations, i.e. $\rho>0$, it is simply possible to jump back and forth between those two formulations. Furthermore we keep in mind Proposition~\ref{prop:energy-is-enough}, which says that it is enough to consider the energy \eqref{eq:energy} to select admissible weak solutions.

\begin{defn} \label{defn:aws-baro-bdd}
	A pair $(\rho,\vu) \in L^\infty\big((0,T) \times \Omega; \R^+ \times \R^n\big)$ is a weak solution of the initial boundary value problem \eqref{eq:baro-euler-pv-dens}, \eqref{eq:baro-euler-pv-mom}, \eqref{eq:baro-initial}, \eqref{eq:impermeability} if the following equations are satisfied for all test functions $(\phi,\vphi) \in \Cc\big([0,T) \times \closure{\Omega}; \R\times \R^n\big)$ with $\vphi\cdot \vn\big|_{\partial \Omega}=0$:
	\begin{align}
		\int_0^T \int_{\Omega} \Big[\rho \partial_t \phi + \rho\vu\cdot\Grad \phi\Big]\dx\dt + \int_{\Omega} \rho_\init\phi(0,\cdot) \dx &= 0 \es \label{eq:baro-euler-weak-bdd-dens} \\
		\int_0^T \int_{\Omega} \Big[\rho\vu \cdot\partial_t \vphi + \rho\vu\otimes\vu:\Grad \vphi + p(\rho)\Div \vphi\Big]\dx\dt + \int_{\Omega} \rho_\init\vu_\init\cdot\vphi(0,\cdot) \dx &= 0\ed \label{eq:baro-euler-weak-bdd-mom}
	\end{align}
	A weak solution is admissible if in addition
	\begin{align}
		\int_0^T \int_{\Omega} \left[\bigg(\half\rho|\vu|^2 + P(\rho)\bigg) \partial_t \varphi + \bigg(\half\rho|\vu|^2 + P(\rho) + p(\rho)\bigg)\vu\cdot\Grad \varphi \right]\dx\dt \qquad & \notag \\ 
		+ \int_{\Omega} \bigg(\half\rho_\init|\vu_\init|^2 + P(\rho_\init)\bigg)\varphi(0,\cdot) \dx &\geq 0 \label{eq:baro-euler-weak-bdd-admissibility}
	\end{align}
	for all $\varphi \in \Cc\big([0,T) \times \closure{\Omega};\R^+_0\big)$.
\end{defn}

\begin{rem}
	When plugging $\vU$ and $\mF$ as given by \eqref{eq:baro-U} and \eqref{eq:baro-F} into \eqref{eq:conslaws-ivp-weak}, we obtain one equation rather than the two equations \eqref{eq:baro-euler-weak-bdd-dens}, \eqref{eq:baro-euler-weak-bdd-mom}. However the two equations \eqref{eq:baro-euler-weak-bdd-dens}, \eqref{eq:baro-euler-weak-bdd-mom} are equivalent to \eqref{eq:conslaws-ivp-weak}. Indeed if \eqref{eq:conslaws-ivp-weak} holds for all test functions $\vpsi$, then setting $\vpsi=(\phi,\vz)$ and $\vpsi=(0,\vphi)$ yields \eqref{eq:baro-euler-weak-bdd-dens} and \eqref{eq:baro-euler-weak-bdd-mom} respectively. For the converse we add \eqref{eq:baro-euler-weak-bdd-dens} with $\phi=\psi_t$ and \eqref{eq:baro-euler-weak-bdd-mom} with $\vphi=\vpsi_\vx$ to obtain \eqref{eq:conslaws-ivp-weak}.
\end{rem}	

\begin{rem}
	The boundary term \eqref{eq:boundary-term} in \eqref{eq:conslaws-ivp-weak} actually yields the terms
	\begin{align*}
		&\int_0^T \int_{\partial\Omega} \rho\vu\cdot \vn \phi \dS_\vx\dt \qquad \text{ and }\\
		&\int_0^T \int_{\partial\Omega} \Big[\rho (\vu\cdot \vphi)(\vu\cdot \vn) + p(\rho) \vphi\cdot \vn \Big]\dS_\vx \dt
	\end{align*}
	in \eqref{eq:baro-euler-weak-bdd-dens} and \eqref{eq:baro-euler-weak-bdd-mom} respectively. However due to the impermeability boundary condition \eqref{eq:impermeability} the boundary term in \eqref{eq:baro-euler-weak-bdd-dens} vanishes whereas in \eqref{eq:baro-euler-weak-bdd-mom} we remain with
	$$
		\int_0^T \int_{\partial\Omega} p(\rho) \vphi\cdot \vn \dS_\vx \dt \ed
	$$
	As mentioned in Definition~\ref{defn:weaksol} we adjust the support of the test function $\vphi$ to make the remaining term vanish since we do not prescribe $\rho$ or $p(\rho)$ on $\partial\Omega$. This is the reason why we want $\vphi$ to satisfy $\vphi\cdot \vn\big|_{\partial \Omega}=0$. Similarly the boundary term in \eqref{eq:admissibility} yields a term
	$$
	\int_0^T \int_{\partial\Omega}  \bigg(\half\rho|\vu|^2 + P(\rho) + p(\rho)\bigg)\vu\cdot \vn \varphi \dS_\vx\dt
	$$
	in \eqref{eq:baro-euler-weak-bdd-admissibility}. This term again vanishes as we consider the impermeability boundary condition \eqref{eq:impermeability}.
\end{rem}

Analogously we obtain the following.

\begin{defn} \label{defn:aws-baro-whole}
	A pair $(\rho,\vu) \in L^\infty\big((0,\infty) \times \R^2; \R^+ \times \R^2\big)$ is a weak solution of the initial value problem \eqref{eq:baro-euler-pv-dens}, \eqref{eq:baro-euler-pv-mom}, \eqref{eq:baro-initial} if the following equations are satisfied for all test functions $(\phi,\vphi) \in \Cc\big([0,\infty) \times \R^2; \R\times \R^2\big)$:
	\begin{align}
		\int_0^\infty \int_{\R^2} \Big[\rho \partial_t \phi + \rho\vu\cdot\Grad \phi\Big]\dx\dt + \int_{\R^2} \rho_\init\phi(0,\cdot) \dx &= 0 \es \label{eq:baro-euler-weak-whole-dens} \\
		\int_0^\infty \int_{\R^2} \Big[\rho\vu \cdot\partial_t \vphi + \rho\vu\otimes\vu:\Grad \vphi + p(\rho)\Div \vphi\Big]\dx\dt + \int_{\R^2} \rho_\init\vu_\init\cdot\vphi(0,\cdot) \dx &= 0\ed \label{eq:baro-euler-weak-whole-mom}
	\end{align}
	A weak solution is admissible if in addition
	\begin{align}
		\int_0^\infty \int_{\R^2} \left[\bigg(\half\rho|\vu|^2 + P(\rho)\bigg) \partial_t \varphi + \bigg(\half\rho|\vu|^2 + P(\rho) + p(\rho)\bigg)\vu\cdot\Grad \varphi \right]\dx\dt \qquad & \notag \\ 
		+ \int_{\R^2} \bigg(\half\rho_\init|\vu_\init|^2 + P(\rho_\init)\bigg)\varphi(0,\cdot) \dx &\geq 0 \label{eq:baro-euler-weak-whole-admissibility}
	\end{align}
	for all $\varphi \in \Cc\big([0,\infty) \times \R^2;\R^+_0\big)$.
\end{defn}

\section{Full Euler System} \label{sec:euler-full}

The full Euler equations \eqref{eq:full-euler-cv-dens} - \eqref{eq:full-euler-cv-en} is a system of conservation laws as well, where 
\begin{align}
	\sO &= \R^+ \times \R^n \times \R^+ \ec \label{eq:full-O} \\
	\vU &= \left(\begin{array}{c}
		\rho \\ \vm \\ E
	\end{array}\right) \ \in\ \R^{n+2}\ec \label{eq:full-U} \\
	\mF(\vU) &= \left(\begin{array}{c}
		\vm^\trans \\ \frac{\vm\otimes\vm}{\rho} + p(\rho,\vm,E) \, \id \\
		\big(E + p(\rho,\vm,E)\big)\frac{\vm^\trans}{\rho} 
	\end{array}\right) \ \in\ \R^{(n+2)\times n} \ed \label{eq:full-F}
\end{align}
The number of unknown is $m=n+2$.

One can simply check that 
\begin{equation} \label{eq:full-flux-jac}
	\Grad_\vU \vF_k = \left(\begin{array}{ccc}
		0 & \ve_k^\trans & 0 \\
		-\frac{m_k\,\vm}{\rho^2} + \partial_{\rho} p \, \ve_k & \frac{\vm}{\rho} \ve_k^\trans + \ve_k \Grad_\vm p + \frac{m_k}{\rho} \id & \partial_E p \, \ve_k \\
		-(E+p)\frac{m_k}{\rho^2} + \frac{m_k}{\rho}\partial_\rho p & (E+p) \frac{\ve_k^\trans}{\rho} + \frac{m_k}{\rho} \Grad_\vm p & \frac{m_k}{\rho} + \frac{m_k}{\rho} \partial_E p
	\end{array}\right)
\end{equation}
for $k=1,...,n$.

From the incomplete equation of state \eqref{eq:incomp-EOS-cv} we deduce
\begin{align}
	\partial_\rho p &= \half (\gamma-1) \frac{|\vm|^2}{\rho^2} \ec \label{eq:full-der-p1} \\
	\Grad_{\vm} p &= -(\gamma-1) \frac{\vm^\trans}{\rho} \ec \label{eq:full-der-p2}\\
	\partial_E p &= (\gamma-1) \ed \label{eq:full-der-p3}
\end{align}

\subsection{Hyperbolicity}

As in the barotropic case, we only sketch the proof of the hyperbolicity. 

For an arbitrary $\vnu\in \sphere^{n-1}$ we obtain from \eqref{eq:full-flux-jac} that 
\begin{equation} \label{eq:full-flux-jac-sum}
	\sum_{k=1}^n \nu_k \Grad_{\vU}\vF_k = \left(\begin{array}{ccc}
	0 & \vnu^\trans & 0 \\
	-\frac{(\vm\cdot\vnu)\vm}{\rho^2} + \partial_{\rho} p \, \vnu & \frac{\vm}{\rho} \vnu^\trans + \vnu \Grad_\vm p + \frac{\vm\cdot\vnu}{\rho} \id & \partial_E p \, \vnu \\
	-(E+p)\frac{\vm\cdot\vnu}{\rho^2} + \frac{\vm\cdot\vnu}{\rho}\partial_\rho p & (E+p) \frac{\vnu^\trans}{\rho} + \frac{\vm\cdot\vnu}{\rho} \Grad_\vm p & \frac{\vm\cdot\vnu}{\rho} + \frac{\vm\cdot\vnu}{\rho} \partial_E p
	\end{array}\right) \ed
\end{equation}
Let $\va_1,...,\va_{n-1}\in\sphere^{n-1}$ be $n-1$ linearly independent vectors which are perpendicular to $\vnu$. Together with \eqref{eq:incomp-EOS-cv} and \eqref{eq:full-der-p1} - \eqref{eq:full-der-p3} it is straightforward to check that the $n+2$ vectors
$$
\left(\begin{array}{c}
	1 \\ \frac{\vm}{\rho} - \sqrt{\gamma\frac{p}{\rho}} \vnu \\ \frac{E+p}{\rho} - \sqrt{\gamma\frac{p}{\rho}} \frac{\vm\cdot \vnu}{\rho}
\end{array}\right) \ec  \left(\begin{array}{c}
	0 \\ \va_1 \\ \frac{\vm\cdot \va_1}{\rho}
\end{array}\right) \ec  ... \ec \left(\begin{array}{c}
	0 \\ \va_{n-1} \\ \frac{\vm\cdot \va_{n-1}}{\rho}
\end{array}\right)\ec \left(\begin{array}{c}
	1 \\ \frac{\vm}{\rho} \\ \frac{|\vm|^2}{2\rho^2}
\end{array}\right) \ec \left(\begin{array}{c}
	1 \\ \frac{\vm}{\rho} + \sqrt{\gamma\frac{p}{\rho}} \vnu \\ \frac{E+p}{\rho} + \sqrt{\gamma\frac{p}{\rho}} \frac{\vm\cdot \vnu}{\rho}
\end{array}\right)
$$
are eigenvectors of the matrix \eqref{eq:full-flux-jac-sum}, where the corresponding eigenvalues read 
$$
\frac{\vm\cdot \vnu}{\rho} - \sqrt{\gamma\frac{p}{\rho}} \ \ec\  \frac{\vm\cdot \vnu}{\rho} \ec\  ... \ \ec\  \frac{\vm\cdot \vnu}{\rho} \ \ec\  \frac{\vm\cdot \vnu}{\rho} \ \ec\  \frac{\vm\cdot \vnu}{\rho} + \sqrt{\gamma\frac{p}{\rho}}
$$
respectively. Furthermore one can show that the eigenvectors written above are linearly independent. Hence the full Euler system \eqref{eq:full-euler-cv-dens} - \eqref{eq:full-euler-cv-en} is hyperbolic in the sense of Definition~\ref{defn:hyperbolicity}.

\subsection{Entropies} \label{subsec:euler-full-entropies}

We want to find the entropies for the full Euler system \eqref{eq:full-euler-cv-dens} - \eqref{eq:full-euler-cv-en} using Proposition~\ref{prop:companion-int-bed}. Note that $\sO$, see \eqref{eq:full-O}, is simply connected and $\mF\in C^2(\sO;\R^{(n+2)\times n} )$, see \eqref{eq:full-F} and \eqref{eq:incomp-EOS-cv}. It turns out that it is much more convenient to work with the variables $(\rho,\vm,p)=:\til{\vU}$ instead of the conserved variables $(\rho,\vm,E)$. We denote the space, in which $\til{\vU}$ lies, by $\til{\sO}$. Note that in our case $\til{\sO}= \R^+\times \R^n\times \R^+$. Since $\rho>0$, it is easily possible to jump back and forth between the variables $(\rho,\vm,E)$ and $(\rho,\vm,p)$ (and even $(\rho,\vu,p)$) via the equation of state \eqref{eq:incomp-EOS-cv}. We denote the entropy as a function of $\til{\vU}$ by $\til{\eta}$, i.e. 
$$
\eta(\vU) = \til{\eta}\big(\til{\vU}(\vU)\big)\ec
$$
where 
$$
\til{\vU}(\vU) = \left(\begin{array}{c}
\rho \\ \vm \\ p(\rho,\vm,E)
\end{array}\right)  \ed
$$

\begin{prop} \label{prop:companion-full}
	The function $\til{\eta}\in C^\infty(\til{\sO})$ is a companion for full Euler \eqref{eq:full-euler-cv-dens} - \eqref{eq:full-euler-cv-en} if and only if its Hessian takes the form 
	\begin{equation} \label{eq:full-hess-eta}
		\Hess_{\til{\vU}} \til{\eta}(\rho,\vm,p) = \left(\begin{array}{ccc}
			\frac{|\vm|^2}{\rho^2} f(\rho,\vm,g) -\gamma \frac{p}{\rho} g(\rho,\vm,p) & -\frac{\vm^\trans}{\rho} f(\rho,\vm,p) & g(\rho,\vm,p) \\
			-\frac{\vm}{\rho} f(\rho,\vm,p) & \id_n f(\rho,\vm,p) & \vz \\
			g(\rho,\vm,p) & \vz^\trans & h(\rho,\vm,p)
		\end{array}\right)
	\end{equation}
	with functions $f,g,h\in C^\infty(\til{\sO})$ where 
	\begin{equation} \label{eq:full-hess-zusatz}
		f(\rho,\vm,p) = g(\rho,\vm,p) + \gamma \frac{p}{\rho} h(\rho,\vm,p) + \frac{\gamma-1}{\rho} \partial_p \eta(\rho,\vm,p) \ed
	\end{equation}
\end{prop}

\begin{proof}
	It can be simply checked that 
	\begin{equation} \label{eq:1-temp-full-entropy}
		\Hess_\vU \eta = \partial_p \til{\eta} \ \Hess_\vU p + \big( \Grad_{\vU} \til{\vU} \big)^\trans \cdot \Hess_{\til{\vU}} \til{\eta} \cdot \Grad_{\vU} \til{\vU} \ec
	\end{equation}
	and 
	\begin{equation} \label{eq:2-temp-full-entropy}
		\Grad_{\vU} \til{\vU} =\left(\begin{array}{ccc}
		1 & \vz^\trans & 0 \\
		\vz & \id_n & \vz \\
		\partial_\rho p & \Grad_{\vm} p & \partial_E p
		\end{array} \right) = \left(\begin{array}{ccc}
			1 & \vz^\trans & 0 \\
			\vz & \id_n & \vz \\
			\half (\gamma-1) \frac{|\vm|^2}{\rho^2} & -(\gamma-1)\frac{\vm^\trans}{\rho} & \gamma-1 
		\end{array} \right)
	\end{equation}
	(see also \eqref{eq:full-der-p1} - \eqref{eq:full-der-p3}) as well as 
	\begin{equation} \label{eq:3-temp-full-entropy}
		\Hess_\vU p = -\frac{\gamma-1}{\rho}\left(\begin{array}{ccc}
			\frac{|\vm|^2}{\rho^2} & -\frac{\vm^\trans}{\rho} & 0 \\
			-\frac{\vm}{\rho} & \id_n & \vz \\
			0 & \vz^\trans & 0
		\end{array} \right) \ed
	\end{equation}
	
	Due to Proposition~\ref{prop:companion-int-bed}, $\eta$ is a companion if and only if $\Hess_\vU\eta \cdot \Grad_{\vU} \vF_k$ is symmetric for all $k=1,...,n$ and all $\vU\in \sO$. Together with \eqref{eq:1-temp-full-entropy}, \eqref{eq:2-temp-full-entropy} and \eqref{eq:3-temp-full-entropy} a lengthy but straightforward calculation shows that this condition can be translated into the following four equations:
	\begin{equation} \label{eq:full-entropy-cond1}
		\left( \partial_\rho \partial_p \til{\eta}  + \partial_p \partial_p \til{\eta} \ \gamma \frac{p}{\rho} + \sum_{\ell=1}^n \partial_p\partial_{m_\ell} \til{\eta} \ \frac{m_\ell}{\rho}  + \partial_p \til{\eta} \  \frac{\gamma-1}{\rho} \right)\delta_{ik} - \partial_{m_i} \partial_{m_k} \til{\eta} \ =\  0 
	\end{equation}
	for all $i,k=1,...,n$;
	\begin{equation} \label{eq:full-entropy-cond2}
		\left(\partial_\rho \partial_p \til{\eta} + \partial_p\partial_p \til{\eta} \ \gamma\frac{p}{\rho} +  \sum_{\ell=1}^n \partial_p\partial_{m_\ell} \til{\eta} \ \frac{m_\ell}{\rho} + \partial_p \til{\eta} \ \frac{\gamma-1}{\rho}\right)\frac{m_k}{\rho} + \partial_\rho\partial_{m_k} \til{\eta} \ =\ 0 
	\end{equation}
	for all $k=1,...,n$;
	\begin{align} 
		&\partial_\rho\partial_\rho \til{\eta} \  \delta_{ik} + \partial_\rho\partial_p \til{\eta} \ \gamma\frac{p}{\rho} \delta_{ik} + \partial_\rho \partial_{m_i} \til{\eta} \ \frac{m_k}{\rho} - \partial_\rho\partial_{m_k} \til{\eta} \ (\gamma-1) \frac{m_i}{\rho} + \sum_{\ell=1}^n \partial_\rho\partial_{m_\ell} \til{\eta} \ \frac{m_\ell}{\rho} \delta_{ik} +  \notag \\
		&+ \partial_p\partial_{m_i}\til{\eta}\ \gamma\frac{p \,m_k}{\rho^2}  + \sum_{\ell=1}^n \partial_{m_i}\partial_{m_\ell} \til{\eta}\ \frac{m_k m_\ell}{\rho^2} - \partial_{m_i}\partial_{m_k} \til{\eta}\ \half(\gamma-1)\frac{|\vm|^2}{\rho^2} \notag \\
		& + (\gamma-1) \left(\half\frac{|\vm|^2}{\rho^2} \delta_{ik} - \frac{m_k m_i}{\rho^2}\right) \left(\partial_\rho\partial_p \til{\eta} + \partial_p\partial_p\til{\eta}\ \gamma\frac{p}{\rho} + \sum_{\ell=1}^n \partial_p\partial_{m_\ell}\til{\eta} \ \frac{m_\ell}{\rho} + \partial_p \til{\eta}\ \frac{\gamma-1}{\rho} \right) \  =\  0 \label{eq:full-entropy-cond3}
	\end{align} 	
	for all $i,k=1,...,n$;
	\begin{align}
		& - \left( \partial_\rho\partial_p \til{\eta} + \partial_p\partial_p\til{\eta}\ \gamma\frac{p}{\rho} + \sum_{\ell=1}^n \partial_p\partial_{m_\ell} \til{\eta}\ \frac{m_\ell}{\rho} + \partial_p \til{\eta} \ \frac{\gamma-1}{\rho} \right) \frac{\gamma-1}{\rho}\left(m_i\delta_{jk} - m_j \delta_{ik}\right) \notag \\
		& + \partial_\rho\partial_{m_i} \til{\eta} \ \delta_{jk} - \partial_\rho\partial_{m_j} \til{\eta} \ \delta_{ik} + \gamma\frac{p}{\rho} \left(\partial_p\partial_{m_i} \til{\eta} \ \delta_{jk} - \partial_p\partial_{m_j} \til{\eta} \ \delta_{ik}\right) \notag \\
		& + \sum_{\ell=1}^n \partial_{m_i}\partial_{m_\ell} \til{\eta} \ \frac{m_\ell}{\rho}\delta_{jk} - \sum_{\ell=1}^n \partial_{m_j}\partial_{m_\ell} \til{\eta} \ \frac{m_\ell}{\rho}\delta_{ik} \notag \\ 
		& - (\gamma-1) \left( \partial_{m_i}\partial_{m_k}\til{\eta}\ \frac{m_j}{\rho} - \partial_{m_j}\partial_{m_k}\til{\eta}\ \frac{m_i}{\rho} \right) \ = \ 0 \label{eq:full-entropy-cond4}
	\end{align}
	for all $i,j,k=1,...,n$.
	
	On the one hand, one easily shows that \eqref{eq:full-entropy-cond1} - \eqref{eq:full-entropy-cond4} hold if $\Hess_{\til{\vU}}\til{\eta}$ is of the form \eqref{eq:full-hess-eta} with \eqref{eq:full-hess-zusatz}. 
	
	For the converse let the conditions \eqref{eq:full-entropy-cond1} - \eqref{eq:full-entropy-cond4} be satisfied. From \eqref{eq:full-entropy-cond1} we deduce that 
	\begin{equation} \label{eq:mi-mk-full-entropy}
		\partial_{m_i}\partial_{m_k} \til{\eta} = 0 \qquad \text{ for }i\neq k \ec
	\end{equation}
	and 
	\begin{equation} \label{eq:mi-mi-full-entropy}
		\partial_{m_i}\partial_{m_i} \til{\eta} = \partial_\rho \partial_p \til{\eta}  + \partial_p \partial_p \til{\eta} \ \gamma \frac{p}{\rho} + \sum_{\ell=1}^n \partial_p\partial_{m_\ell} \til{\eta} \ \frac{m_\ell}{\rho}  + \partial_p \til{\eta} \  \frac{\gamma-1}{\rho} \ =:\ f(\rho,\vm,p)
	\end{equation}
	for all $i=1,...,n$. In particular $\partial_{m_i}\partial_{m_i} \til{\eta}$ does not depend on $i$. Furthermore we obtain from \eqref{eq:full-entropy-cond2} that
	\begin{equation} \label{eq:rho-mi-full-entropy}
		\partial_\rho\partial_{m_i} \til{\eta} = -\frac{m_i}{\rho} f(\rho,\vm,p) 
	\end{equation}
	for all $i=1,...,n$. Plugging \eqref{eq:mi-mi-full-entropy}, \eqref{eq:mi-mk-full-entropy} and \eqref{eq:rho-mi-full-entropy} into \eqref{eq:full-entropy-cond4} we end up with
	$$
	\gamma\frac{p}{\rho}\left( \partial_p \partial_{m_i} \til{\eta} \ \delta_{jk} - \partial_p \partial_{m_j} \til{\eta} \ \delta_{ik} \right) =0 \ec
	$$
	which must hold for all $i,j,k=1,...,n$. For any $i=1,...,n$ choose $j=k\neq i$ to obtain 
	\begin{equation} \label{eq:p-mi-full-entropy}
		\partial_p \partial_{m_i} \til{\eta} = 0 \ed
	\end{equation}
	Let us set
	\begin{align*}
		g(\rho,\vm,p) &:= \partial_\rho\partial_p \til{\eta} \ec\\
		h(\rho,\vm,p) &:= \partial_p\partial_p \til{\eta} \ed
	\end{align*}
	Plugging this and \eqref{eq:mi-mk-full-entropy}, \eqref{eq:mi-mi-full-entropy}, \eqref{eq:rho-mi-full-entropy} and \eqref{eq:p-mi-full-entropy} into \eqref{eq:full-entropy-cond3}, we obtain 
	$$
		\partial_\rho\partial_\rho \til{\eta} \  \delta_{ik} + \gamma\frac{p}{\rho} g \delta_{ik} - \frac{|\vm|^2}{\rho^2} f \delta_{ik} = 0 
	$$ 	
	for all $i,k=1,...,n$, which implies 
	$$
		\partial_\rho\partial_\rho \til{\eta} = \frac{|\vm|^2}{\rho^2} f - \gamma\frac{p}{\rho} g \ed
	$$
	Hence we have shown that $\Hess_{\til{\vU}} \til{\eta}$ is of the form \eqref{eq:full-hess-eta}. Finally using \eqref{eq:p-mi-full-entropy} in \eqref{eq:mi-mi-full-entropy} yields \eqref{eq:full-hess-zusatz}.
\end{proof}

As in the barotropic case, we use Proposition~\ref{prop:companion-full} to find the entropies for the full Euler system, i.e. we prove a ``full'' version of Proposition~\ref{prop:baro-entropies}. However we split this statement into two propositions since we will not give an equivalent condition for the companions to be convex. Such a condition is difficult to derive and would go beyond the issues of this book. Let us first solve \eqref{eq:full-hess-eta}, \eqref{eq:full-hess-zusatz} in order to find all companions. 

\begin{prop} \label{prop:full-companions} 
	The companions for the full Euler system \eqref{eq:full-euler-cv-dens} - \eqref{eq:full-euler-cv-en} are all functions of the form 
	\begin{equation} \label{eq:full-entropies}
		\eta(\rho,\vm,E) = - \rho Z \left( \log \left(\frac{p(\rho,\vm,E}{\rho^\gamma}\right)\right) + a E + \vm \cdot \vb + c \ec 
	\end{equation}
	where $a,c\in \R$, $\vb\in \R^n$ and $Z\in C^\infty(\R)$. The corresponding fluxes are given by 
	\begin{equation} \label{eq:full-entr-fluxes}
		\vq(\rho,\vm,E) = - \vm Z \left( \log \left(\frac{p(\rho,\vm,E}{\rho^\gamma}\right)\right) + a\big(E+p(\rho,\vm,E)\big)\frac{\vm}{\rho} + (\vm\cdot \vb) \frac{\vm}{\rho} + p(\rho,\vm,E) \vb\ed 
	\end{equation}
\end{prop}

\begin{proof}
	A long but straightforward computation proves that $(\eta,\vq)$ of the form given by \eqref{eq:full-entropies} and \eqref{eq:full-entr-fluxes} fulfill \eqref{eq:companion-cond}. Thus $\eta$ is a companion and $\vq$ the corresponding flux. 
	
	For the converse let $\eta$ be an entropy. Our goal is to show that this implies that $\eta$ is of the form \eqref{eq:full-entropies}. As above we consider $\eta$ as a function of $\til{\vU}$ instead of $\vU$ and write $\til{\eta}$. From Proposition~\ref{prop:companion-full} we know that there exist functions $f,g,h\in C^\infty(\til{\sO})$ such that \eqref{eq:full-hess-eta} and \eqref{eq:full-hess-zusatz} hold. We proceed as in the proof of Proposition~\ref{prop:baro-entropies}. For all $i=1,...,n$ we find\footnote{Here again it is essential that $n\geq 2$.} $j\neq i$ and hence
	$$
	\partial_{m_i} f = \partial_{m_i} \partial_{m_j} \partial_{m_j} \til{\eta} = \partial_{m_j} \partial_{m_i} \partial_{m_j} \til{\eta} = 0 \ed
	$$
	Furthermore 
	$$
	\partial_p f = \partial_p \partial_{m_j} \partial_{m_j} \til{\eta} = \partial_{m_j} \partial_p \partial_{m_j} \til{\eta} = 0 \ed
	$$
	Thus 
	$$
	\partial_{m_i} \partial_{m_j} \til{\eta} = f(\rho) \delta_{ij} \ed
	$$
	As in the proof of Proposition~\ref{prop:baro-entropies} this implies together with the fact that 
	$$
	\partial_\rho \partial_{m_i} \til{\eta}= -\frac{m_i}{\rho} f(\rho) \ec
	$$
	that
	$$
	f(\rho) = \frac{a}{\rho}
	$$
	with an $a\in\R$, and there exists a constant $\vb\in \R^n$ such that 
	$$
	\partial_{m_i} \til{\eta} = m_i f(\rho) + b_i \ed
	$$
	Hence there exists $G\in C^\infty(\R^+\times \R^+)$ such that 
	$$
	\til{\eta} = a \frac{|\vm|^2}{2\rho} + \vm\cdot \vb + G(\rho,p) \ed
	$$
	From 
	$$
	\partial_\rho \partial_\rho \til{\eta} = \frac{|\vm|^2}{\rho^2} f(\rho) - \gamma\frac{p}{\rho} g(\rho,\vm,p) =  a\frac{|\vm|^2}{\rho^3} - \gamma\frac{p}{\rho} \, \partial_\rho \partial_p \til{\eta} 
	$$ 
	and \eqref{eq:full-hess-zusatz} we obtain
	\begin{equation} \label{eq:4-temp-full-entropy}
		\partial_\rho \partial_\rho G(\rho,p) = -\gamma \frac{p}{\rho} \, \partial_\rho \partial_p G(\rho,p)
	\end{equation} 
	and 
	\begin{equation} \label{eq:5-temp-full-entropy}
		\frac{a}{\rho} = \partial_\rho \partial_p G(\rho,p) + \gamma\frac{p}{\rho} \, \partial_p \partial_p G(\rho,p) + \frac{\gamma-1}{\rho} \, \partial_p G(\rho,p) \ec
	\end{equation}
	respectively. Next we define $H\in C^\infty(\R^+\times \R^+)$ by
	$$
	H(\rho,p) := \frac{a p}{(\gamma-1) \rho} - \frac{G(\rho,p)}{\rho}\ed
	$$
	A straightforward calculation shows that \eqref{eq:4-temp-full-entropy} and \eqref{eq:5-temp-full-entropy} imply 
	$$
	\partial_\rho \Big(\rho^2\, \partial_\rho H(\rho,p) + \gamma\rho p \,\partial_p H(\rho,p)\Big) = 0
	$$
	and 
	$$
	\partial_p \Big(\rho^2\, \partial_\rho H(\rho,p) + \gamma\rho p \,\partial_p H(\rho,p)\Big) = 0 \ec
	$$
	respectively. Hence there exists $c\in \R$ such that 
	\begin{equation} \label{eq:6-temp-full-entropy}
		\rho^2\, \partial_\rho H(\rho,p) + \gamma\rho p \,\partial_p H(\rho,p) = c \ed
	\end{equation}
	Define $Y\in C^\infty(\R^+\times\R^+)$ via
	$$
	Y(\rho,p) := H(\rho,p) + \frac{c}{\rho} \ed
	$$
	Then \eqref{eq:6-temp-full-entropy} implies 
	\begin{equation} \label{eq:7-temp-full-entropy}
		\partial_\rho Y(\rho,p) + \gamma\frac{p}{\rho}\,\partial_p Y(\rho,p) = 0 \ed 
	\end{equation}
	From \eqref{eq:7-temp-full-entropy} we deduce that $Y$ is constant along lines $(\rho,p(\rho)):=(\rho,\alpha \rho^\gamma)$ with some $\alpha>0$. Indeed
	$$
	\frac{\rm d}{{\rm d}\rho} Y(\rho,\alpha\rho^\gamma) = \partial_\rho Y(\rho,\alpha\rho^\gamma) + \alpha \gamma \rho^{\gamma-1} \partial_p Y(\rho,\alpha\rho^\gamma) = 0\ed
	$$
	Thus 
	$$
	Y(\rho,p) = Y(1,\alpha) = Y\left(1,\frac{p}{\rho^\gamma}\right) \ec
	$$
	and with the definition of $\til{Z}\in C^\infty(\R^+)$ by
	$$
	\til{Z}(s) := Y(1,s) \ec
	$$
	we end up with 
	$$
	\til{\eta} = - \rho \til{Z}\left(\frac{p}{\rho^\gamma}\right) + a \left( \frac{|\vm|^2}{2\rho} +\frac{p}{\gamma-1}\right) + \vm\cdot \vb + c \ed
	$$
	Finally we introduce the function $Z\in C^\infty(\R)$ by 
	$$
	Z(s) := \til{Z}\big(\exp(s)\big) \ec
	$$
	which yields that 
	$$
	\til{\eta} = - \rho Z\left(\log\left(\frac{p}{\rho^\gamma}\right)\right) + a \left( \frac{|\vm|^2}{2\rho} +\frac{p}{\gamma-1}\right) + \vm\cdot \vb + c \ed
	$$
	Writing $\til{\eta}$ with respect to the conserved variables $\vU=(\rho,\vm,E)$ again, we obtain \eqref{eq:full-entropies}. 
\end{proof}

\begin{rem}
	At this point the usage of the logarithm in the argument of the function $Z$ seems to be unnecessary. Indeed the statement is still correct if one replaces $Z\big(\log(p / \rho^\gamma)\big)$ in \eqref{eq:full-entropies} and \eqref{eq:full-entr-fluxes} by $Z(p/\rho^\gamma)$. The use of the logarithm is inspired by the function $s$ which is introduced in Chapter~\ref{chap:intro}, see \eqref{eq:s}, and will be plugged in later, see Proposition~\ref{prop:entropy-is-enough}. Furthermore using the logarithm allows for a simple formulation of the following fact. 
\end{rem}

\begin{prop} \label{prop:full-entropies} 
	If $Z\in C^\infty(\R)$ is such that $Z'(s)\geq 0$ and $Z''(s)\leq 0$ for all $s\in \R$, then $\eta$ given by Proposition~\ref{prop:full-companions} is convex, i.e. an entropy. Conversely if $\eta$ is an entropy of the form given by Proposition~\ref{prop:full-companions}, then $Z'(s)\geq 0$ for all $s\in \R$.
\end{prop}

\begin{proof} 
	A lengthy calculation shows that the Hessian of $\eta$ with respect to $\vU$ is given by
	\begin{equation} \label{eq:8-temp-full-entropy}
	\Hess_\vU \eta = \frac{Z' \left(\log\left(\frac{p}{\rho^\gamma}\right)\right)}{p^2} \mA(\rho,\vm,p) + \frac{Z'' \left(\log\left(\frac{p}{\rho^\gamma}\right)\right)}{p^2} \mB(\rho,\vm,p) \ec
	\end{equation}
	where\footnote{Keeping in mind that $\mA$ and $\mB$ are symmetric, we write $\ast$ for the entries in the upper triangle for convenience.}
	$$
	\mA(\rho,\vm,p) := \left(\begin{array}{ccc}
	\gamma\frac{p^2}{\rho} + \frac{1}{4} (\gamma-1)^2 \frac{|\vm|^4}{\rho^3} & \ast & \ast \\
	-\half (\gamma-1)^2 \frac{|\vm|^2}{\rho^2} \vm & (\gamma-1)\left(p \id_n + (\gamma-1) \frac{\vm\otimes \vm}{\rho} \right) & \ast \\
	\half (\gamma-1)^2\frac{|\vm|^2}{\rho} - (\gamma-1) p & -(\gamma-1)^2 \vm^\trans & (\gamma-1)^2\rho 
	\end{array} \right)
	$$
	and 
	$$
	\mB(\rho,\vm,p) := \left(\begin{array}{ccc}
	-\frac{1}{\rho} \left(\half (\gamma-1) \frac{|\vm|^2}{\rho} - \gamma p \right)^2 & \ast & \ast \\
	(\gamma-1)\frac{\vm}{\rho} \left(\half (\gamma-1) \frac{|\vm|^2}{\rho} - \gamma p \right) & -(\gamma-1)^2 \frac{\vm\otimes \vm}{\rho} & \ast \\
	-(\gamma-1) \left(\half (\gamma-1) \frac{|\vm|^2}{\rho} - \gamma p \right) & (\gamma-1)^2 \vm^\trans & -(\gamma-1)^2 \rho
	\end{array} \right) \ed
	$$
	
	To show that $\eta$ is convex for all $Z\in C^\infty(\R)$ with the desired properties, i.e. $Z'\geq 0$ and $Z''\leq 0$, it suffices to show that $\mA(\rho,\vm,p)$ is positive semi-definite, whereas $\mB(\rho,\vm,p)$ is negative semi-definite. To this end let $\vw\in \R^{n+2}$ be arbitrary whose components are denoted by $\vw=(w_t,\vw_\vx,w_s)^\trans$. On can show that 
	\begin{align}
	\vw^\trans \cdot \mA(\rho,\vm,p)\cdot \vw &= \frac{(\gamma-1)^2}{\rho} \left(\left(\half \frac{|\vm|^2}{\rho} - \frac{p}{\gamma-1}\right) w_t - \vm\cdot\vw_\vx + \rho w_s\right)^2 \label{eq:wAw} \\
	&\qquad + \frac{(\gamma-1) p^2 w_t^2}{\rho} + (\gamma-1) p \left|\frac{\vm}{\rho} w_t - \vw_\vx\right|^2 \ec \notag
	\end{align}
	which is obviously non-negative. Thus $\mA(\rho,\vm,p)$ is positive semi-definite for all $(\rho,\vm,p)\in \R^+\times \R^n\times \R^+$. Furthermore
	\begin{equation} \label{eq:wBw}
	\vw^\trans \cdot \mB(\rho,\vm,p)\cdot \vw = -\frac{(\gamma-1)^2}{\rho} \left(\left(\half \frac{|\vm|^2}{\rho} - \frac{\gamma p}{\gamma-1}\right) w_t - \vm\cdot\vw_\vx + \rho w_s\right)^2 \ec
	\end{equation}
	which is non-positive. Hence $\mB(\rho,\vm,p)$ is negative semi-definite for any arbitrary $(\rho,\vm,p)\in \R^+\times \R^n\times \R^+$. 
	
	For the converse we assume that $\eta$ is convex, i.e. the Hessian \eqref{eq:8-temp-full-entropy} is positive semi-definite. Thus
	\begin{equation} \label{eq:9-temp-full-entropy}
		\frac{Z' \left(\log\left(\frac{p}{\rho^\gamma}\right)\right)}{p^2} \vw^\trans\cdot \mA(\rho,\vm,p) \cdot \vw+ \frac{Z'' \left(\log\left(\frac{p}{\rho^\gamma}\right)\right)}{p^2} \vw^\trans\cdot \mB(\rho,\vm,p)\cdot \vw \geq 0
	\end{equation}
	for all $\vw\in \R^{n+2}$ and all $(\rho,\vm,p)\in \R^+\times \R^n\times \R^+$. Assume that there was $s\in \R$ such that $Z'(s)<0$. Fix any $(\rho,p)\in\R^+\times \R^+$ such that $s=\log(p / \rho^\gamma)$. Moreover, choosing $\vm=\vz$ and $\vw=(0,\va,0)^\trans$ with $\va\in \R^n\setminus\{\vz\}$ we deduce from \eqref{eq:wAw} and \eqref{eq:wBw} that 
	\begin{align*}
		\vw^\trans \cdot \mA(\rho,\vm,p)\cdot \vw &= (\gamma-1) p |\va|^2 > 0 \qquad \text{ and } \\
		\vw^\trans \cdot \mB(\rho,\vm,p)\cdot \vw &= 0 \ec
	\end{align*}
	respectively. Plugging this into \eqref{eq:9-temp-full-entropy} we obtain the contradiction 
	$$
	0>\frac{Z' \big(\log(s)\big)}{p^2} \vw^\trans\cdot \mA(\rho,\vm,p) \cdot \vw+ \frac{Z'' \big(\log(s)\big)}{p^2} \vw^\trans\cdot \mB(\rho,\vm,p)\cdot \vw \geq 0 \ed
	$$
	Hence $Z'(s) \geq 0$ for all $s\in \R$.
\end{proof}

\begin{rem}
	Note that Propostion \ref{prop:full-entropies} does not provide an equivalent condition for a companion to be an entropy. More precisely it does not say which property of the second derivative $Z''$ is necessary and sufficient for $\eta$ being convex. In the sequel we call a weak solution of the full Euler system \eqref{eq:full-euler-cv-dens} - \eqref{eq:full-euler-cv-en} \emph{admissible} or \emph{entropy solution} if the entropy inequality \eqref{eq:admissibility-strong} holds for all $\eta$ as in Proposition~\ref{prop:full-companions} with $Z'(s)\geq 0$ for all $s\in \R$ in the sense of distributions. Due to the issue described a few lines above, this definition does not coincide with the actual admissibility criterion given by Definition~\ref{defn:admissibility}. More precisely the notion of admissibility for the full Euler system introduced in the current remark (i.e. with $Z'\geq 0$) is a bit stricter. Indeed every weak solution that is admissible in our sense (i.e. with $Z'\geq 0$) is also admissible in the sense that the entropy inequality hold for all entropies, because according to Propositions \ref{prop:full-companions} and \ref{prop:full-entropies} every entropy must be of the form \eqref{eq:full-entropies} with $Z'\geq 0$ plus a requirement on $Z''$. We show in Chapter~\ref{chap:appl-riemann} that there are initial data for the full Euler system which admit infinitely many admissible (in our sense, i.e. with $Z'\geq 0$) weak solutions. With the arguments above, these solutions are also admissible in the actual sense. Hence our work in Chapter~\ref{chap:appl-riemann} even proves that there are initial data for the full Euler system which lead to infinitely many admissible (in the actual sense) weak solutions.
\end{rem}

As in the barotropic case, the linear terms in \eqref{eq:full-entropies} are irrelevant. Additionally we can make use of the function $s$ introduced in Chapter~\ref{chap:intro}, see \eqref{eq:s}. More precisely, the following holds:
\begin{prop} \label{prop:entropy-is-enough} 
	A weak solution $\vU\in L^\infty\big((0,T)\times \Omega;\sO\big)$ of the initial (boundary) value problem \eqref{eq:cons-law-strong}, \eqref{eq:conslaws-ic}, \eqref{eq:conslaws-bc} for the full Euler system, i.e. $\sO$, $\vU$ and $\mF$ are given by \eqref{eq:full-O} - \eqref{eq:full-F}, is admissible\footnote{Again note that in the context of the full Euler system, we use a modified notion of \emph{admissibility}, see the remark above.} if and only if the entropy inequality \eqref{eq:admissibility-strong} holds for all $(\eta,\vq)$ of the form 
	$$
	(\eta,\vq) = - ( \rho,\vm) Z \Big(s\big(\rho,p(\rho,\vm,E)\big)\Big) \ec
	$$
	where $Z\in C^\infty(\R)$ with $Z'\geq 0$, in the sense of distributions.
\end{prop}

\begin{proof}
The fact that the linear terms in \eqref{eq:full-entropies} are irrelevant, can be shown in the same way as in Proposition~\ref{prop:energy-is-enough}, where one makes use of Proposition~\ref{prop:weak-linentropies}. Note furthermore that 
$$
s(\rho,p)= \frac{1}{\gamma-1} \log \left(\frac{p}{\rho^\gamma}\right) \ed
$$ 
It can be simply proven that the positive factor $\frac{1}{\gamma-1}$ can be omitted by redefining $Z$ which does not change the sign of $Z'$. The details are left to the reader.
\end{proof}

\subsection{Admissible Weak Solutions} \label{subsec:euler-full-aws} 

For the full Euler system we consider only the case $\Omega=\R^2$, $T=\infty$ in this book. As in the barotropic case we recap the definition of an admissible weak solution to the corresponding initial value problem. Again notice that we can jump back and forth between primitive variables $(\rho,\vu,p)$ and conserved variables $(\rho,\vm,E)$ since we exclude vacuum. Moreover we bear in mind Proposition~\ref{prop:entropy-is-enough}. Finally we refer to Subsection~\ref{subsec:euler-baro-aws} for some remarks which are valid in the case of full Euler, too.

\begin{defn} \label{defn:aws-full-whole}
	A triple $(\rho,\vu,p) \in L^\infty\big((0,\infty) \times \R^2; \R^+ \times \R^2 \times \R^+\big)$ is a weak solution of the initial value problem \eqref{eq:full-euler-pv-dens} - \eqref{eq:full-euler-pv-en}, \eqref{eq:full-initial} if the following equations are satisfied for all test functions $(\phi,\vphi,\psi)\in \Cc\big([0,\infty) \times \R^2;\R\times \R^2\times \R\big)$:
	\begin{align}
		\int_0^\infty \int_{\R^2} \Big[\rho \partial_t \phi + \rho\vu\cdot\Grad \phi\Big]\dx\dt + \int_{\R^2} \rho_\init\phi(0,\cdot) \dx &= 0 \es \label{eq:full-euler-weak-dens} \\
		\int_0^\infty \int_{\R^2} \Big[\rho\vu \cdot\partial_t \vphi + \rho\vu\otimes\vu:\Grad \vphi + p\Div \vphi\Big]\dx\dt + \int_{\R^2} \rho_\init\vu_\init\cdot\vphi(0,\cdot) \dx &= 0 \es \label{eq:full-euler-weak-mom} \\
		\int_0^\infty \int_{\R^2} \left(\bigg(\half\rho|\vu|^2 + \rho e(\rho,p)\bigg) \partial_t \psi + \bigg(\half\rho|\vu|^2 + \rho e(\rho,p) + p\bigg)\vu\cdot\Grad \psi \right)\dx\dt \quad & \notag \\ 
		+ \int_{\R^2} \bigg(\half\rho_\init|\vu_\init|^2 + \rho_\init e(\rho_\init,p_\init)\bigg)\psi(0,\cdot) \dx &= 0 \ed \label{eq:full-euler-weak-en}
	\end{align}
	A weak solution is admissible if in addition
	\begin{equation} \label{eq:full-euler-weak-entropy}
		\int_0^\infty \int_{\R^2} \Big[\rho Z\big(s(\rho,p)\big) \partial_t \varphi + \rho Z\big(s(\rho,p)\big) \vu\cdot\Grad \varphi\Big]\dx\dt + \int_{\R^2} \rho_\init Z\big(s(\rho_\init,p_\init)\big)\varphi(0,\cdot) \dx \leq 0
	\end{equation}
	for all $Z\in C^\infty(\R)$ with $Z'\geq 0$, and all $\varphi \in \Cc\big([0,\infty) \times \R^2;\R^+_0\big)$. 
\end{defn}

%% file: DissertationCh4.tex
\chapter{Convex Integration} \label{chap:convint} 

This chapter is the main part of this book. Here convex integration is applied to the barotropic Euler system. Our goal is to prove Theorem~\ref{thm:convint}, which can be seen as a ``compressible analogue'' of a result by \name{De~Lellis} and \name{Sz{\'e}kelyhidi}, see \cite[Proposition 2]{DelSze10} or \cite[Proposition 2.4]{DelSze12}. We mimick the procedure performed by \name{De~Lellis} and \name{Sz{\'e}kelyhidi} in \cite{DelSze09} and \cite{DelSze10}. Almost every definition and claim in this chapter has an ``incompressible variant'' in those papers. Note that Theorem~\ref{thm:convint} does not immediately provide admissible weak solutions to an initial boundary value problem for the barotropic Euler system. However this theorem is later in Chapters \ref{chap:appl-ibvp} and \ref{chap:appl-riemann} used to find solutions to particular initial (boundary) value problems, even for the full Euler equations.

This chapter is organized as follows. In Section~\ref{sec:convint-prel} we adjust the problem in such a way that we can apply convex integration. Furthermore we outline how the convex integration technique works. In Section~\ref{sec:convint-convex-hulls} we study $\Lambda$-convex hulls in general since the $\Lambda$-convex hull of a particular set $K$ is needed in order to implement convex integration. This particular $\Lambda$-convex hull is computed in Section~\ref{sec:convint-U}. In order to apply convex integration, one needs operators yielding solutions to a particular system of linear PDEs. Those operators are constructed in Section~\ref{sec:convint-operators}. Sections \ref{sec:convint-thm} and \ref{sec:convint-pp} represent the main part of this chapter. In Section~\ref{sec:convint-thm} we formulate Theorem~\ref{thm:convint} and prove it using the so-called \emph{Perturbation Property}. The latter is proven in Section~\ref{sec:convint-pp}. In fact this is the proof using convex integration. Finally in Section~\ref{sec:convint-nodens} we prove a fixed-density-version of Theorem~\ref{thm:convint} which represents in some sense the ``incompressible'' convex integration by \name{De~Lellis} and \name{Sz{\'e}kelyhidi}.

\section{Outline and Preliminaries} \label{sec:convint-prel} 

\subsection{Adjusting the Problem} \label{subsec:convint-prel-adjusting}

Mimicking the procedure by \name{De~Lellis} and \name{Sz{\'e}kelyhidi}, see e.g. \cite{DelSze09} and \cite{DelSze10}, where the kinetic energy $|\vv|^2=\tr(\vv\otimes\vv)$ was prescribed for the incompressible Euler system, we want to prescribe
\begin{equation} \label{eq:convint-temp-trace}
\tr\left(\frac{\vm\otimes\vm}{\rho} + p(\rho) \id\right)=\frac{|\vm|^2}{\rho} + np(\rho) \mathop{=}\limits^! 2\ov{e}\ec
\end{equation}
where $\ov{e}=\ov{e}(t,\vx)$. Since we want to work with traceless matrices, we reformulate \eqref{eq:baro-euler-cv-mom} as
$$
\partial_t \vm + \Div\left(\frac{\vm\otimes\vm}{\rho} + p(\rho) \id - \frac{2\ov{e}}{n} \id\right) + \frac{2}{n}\Grad \ov{e} =\vz\ed
$$

In this book we look for solutions with constant $\ov{e}=:c$.

\begin{rem}
	Note that this is a difference to the work by \name{De~Lellis} and \name{Sz{\'e}kelyhidi} for the incompressible Euler system, where the kinetic energy is prescribed by a not necessarily constant function $\ov{e}$. The generalization towards a non-constant $\ov{e}$ in the compressible case is future work.
\end{rem}

Now we can rewrite \eqref{eq:baro-euler-cv-dens}, \eqref{eq:baro-euler-cv-mom} as
\begin{align} 
\partial_t \rho + \Div \vm &=0 \ec \label{eq:euler-traceless-dens}\\
\partial_t \vm + \Div\left(\frac{\vm\otimes\vm}{\rho} + p(\rho) \id - \frac{2c}{n}\id\right) &=\vz\ed \label{eq:euler-traceless-mom}
\end{align}

\begin{rem} 
	Note that if a \emph{monoatomic gas} is considered, i.e. $p(\rho)=\rho^{\frac{2}{n}+1}$, then $P(\rho)=\frac{n}{2} p(\rho)$, see \eqref{eq:isentropic-P}. Hence in this case we have 
	$$
		\ov{e} = \half\frac{|\vm|^2}{\rho} + P(\rho)\ed
	$$
	In other words we prescribe the energy. 
\end{rem}

\subsection{Tartar's Framework} \label{subsec:convint-prel-tartar}

The following procedure is known as \emph{Tartar's Framework} since is has been introduced by \name{Tartar} \cite{Tartar79}. We replace $\frac{\vm\otimes\vm}{\rho} + p(\rho) \id - \frac{2c}{n}\id$ in \eqref{eq:euler-traceless-mom} by a new unknown $\mU=\mU(t,\vx)$, which takes values in $\szn$. According to the arguments in the preceding subsection, we obtain a linear system with the following property.

\smallskip 

If $(\rho,\vm,\mU)$ solves 
\begin{align} 
	\partial_t \rho + \Div \vm &=0 \ec \label{eq:euler-lin-dens}\\
	\partial_t \vm + \Div\mU &=\vz \ec \label{eq:euler-lin-mom}
\end{align}
and takes values in 
\begin{equation} \label{eq:K}
	K:=\left\{(\rho,\vm,\mU)\in\R^+\times\R^n\times\szn\,\Big|\,\mU=\frac{\vm\otimes\vm}{\rho} + \left(p(\rho) - \frac{2c}{n}\right)\id\right\}\ec
\end{equation}
then $(\rho,\vm)$ solves \eqref{eq:euler-traceless-dens}, \eqref{eq:euler-traceless-mom} and hence \eqref{eq:baro-euler-cv-dens}, \eqref{eq:baro-euler-cv-mom} together with 
\begin{equation} \label{eq:prescribed-trace}
\half\frac{|\vm|^2}{\rho} + \frac{n}{2}p(\rho)  = c\ed
\end{equation}

\smallskip

Hence in order to solve the barotropic Euler system \eqref{eq:baro-euler-cv-dens}, \eqref{eq:baro-euler-cv-mom} together with property \eqref{eq:prescribed-trace}, we may find $(\rho,\vm,\mU)$ solving \eqref{eq:euler-lin-dens}, \eqref{eq:euler-lin-mom} and taking values in $K$ as defined in \eqref{eq:K}.

\begin{rem}
	We donote the space $\R\times \R^n\times \szn$ as \emph{extended phase space}.
\end{rem}

\subsection{Plane Waves and the Wave Cone} \label{subsec:convint-prel-wavecone} 

Let us relax the set $K$ to a larger set $\sU$ which will be explained more detailed below. Justified by Tartar's framework above, we call a triple $(\rho,\vm,\mU)$ a \emph{solution} if it solves \eqref{eq:euler-lin-dens}, \eqref{eq:euler-lin-mom} and takes values in $K$. A triple $(\rho,\vm,\mU)$ solving \eqref{eq:euler-lin-dens}, \eqref{eq:euler-lin-mom} and taking values in $\sU$ is called a \emph{subsolution}. Let us assume that we have a subsolution $(\rho,\vm,\mU)$. The basic idea of convex integration is to construct another subsolution $(\rho_\pert,\vm_\pert,\mU_\pert)$ by adding oscillations to $(\rho,\vm,\mU)$, such that the new subsolution $(\rho_\pert,\vm_\pert,\mU_\pert)$ is 
\begin{itemize}
	\item in some sense closer to $K$ than $(\rho,\vm,\mU)$, 
	\item but still in another sense close to $(\rho,\vm,\mU)$, 
\end{itemize}
see Proposition~\ref{prop:pert-prop} for details. When \name{De~Lellis} and \name{Sz{\'e}kelyhidi} began to apply convex integration to the incompressible Euler equations \cite{DelSze09}, \cite{DelSze10}, they used planar waves as such oscillations. The relaxed set $\sU$ has to be chosen such that it is compatible with the oscillations that are added, which we choose to be planar waves as well. In order to specify $\sU$ we have to study planar waves and the wave cone. 

A plane wave solution to \eqref{eq:euler-lin-dens}, \eqref{eq:euler-lin-mom} is a solution $(\rho,\vm,\mU)$ of the form
\begin{equation} \label{eq:planewave}
(\rho,\vm,\mU)(t,\vx)= (\ov{\rho},\ov{\vm},\ov{\mU})\,h\big((t,\vx)\cdot \veta\big)
\end{equation}
with a constant triple $(\ov{\rho},\ov{\vm},\ov{\mU})\in \R\times\R^n\times\szn$, a function $h:\R\rightarrow\R$ and a direction in the space-time $\veta\in\R^{1+n}\setminus\{\vz\}$. 

\begin{defn} \label{defn:wavecone}
	Define the \emph{wave cone} as\footnote{See Subsection~\ref{subsec:not-spacetime} for the notation of the components of a vector in the space-time $\R^{1+n}$.}
	\begin{equation*}
	\begin{split}
	\Lambda:=\bigg\{(\ov{\rho},\ov{\vm},\ov{\mU})&\in\R\times\R^n\times\szn\,\Big|\,\\
	&\exists\veta\in\R^{1+n}\setminus\{\vz\}\text{ such that }\ \ov{\rho}\,\eta_t + \ov{\vm} \cdot \veta_\vx=0\ \text{ and }\ \ov{\vm}\,\eta_t + \ov{\mU} \,\veta_\vx=\vz\ \bigg\}\ed
	\end{split}
	\end{equation*}
\end{defn}

The following proposition shows the meaning of Definition~\ref{defn:wavecone}. 
\begin{prop}\label{prop:planarwave}
	Let $(\ov{\rho},\ov{\vm},\ov{\mU})\in \Lambda$ and $h\in C^1(\R)$ arbitrary. Then $(\rho,\vm,\mU)$ defined as in \eqref{eq:planewave} with the corresponding $\veta$, which exists according to Definition~\ref{defn:wavecone}, is a plane wave solution to \eqref{eq:euler-lin-dens}, \eqref{eq:euler-lin-mom}. 
\end{prop}

\begin{proof}
	The proof is a simple calculation. Indeed we have
	\begin{align*}
	\partial_t \rho + \Div \vm &= \ov{\rho} \, \partial_t \Big(h\big((t,\vx)\cdot \veta\big)\Big) + \ov{\vm} \cdot \Grad \Big(h\big((t,\vx)\cdot \veta\big)\Big) \\
	&=h'\big((t,\vx)\cdot \veta\big) \Big(\ov{\rho}\,\eta_t + \ov{\vm} \cdot \veta_\vx\Big) \ =\ 0
	\end{align*}
	and 
	\begin{align*}
	\partial_t \vm + \Div \mU &= \ov{\vm} \, \partial_t \Big(h\big((t,\vx)\cdot \veta\big)\Big) + \ov{\mU} \cdot \Grad \Big(h\big((t,\vx)\cdot \veta\big)\Big) \\
	&=h'\big((t,\vx)\cdot \veta\big) \Big(\ov{\vm}\,\eta_t + \ov{\mU} \, \veta_\vx\Big) \ =\ \vz\ec 
	\end{align*}
	where we have used the fact that $\ov{\mU}$ is symmetric.
\end{proof}

Note that we can shorten the definition of $\Lambda$ to 
\begin{equation} \label{eq:wavecone-kernel}
\Lambda=\left\{(\ov{\rho},\ov{\vm},\ov{\mU})\in\R\times\R^n\times\szn\,\Big|\,\exists\veta\in\R^{1+n}\smallsetminus\{\vz\}\text{ such that }\left(\begin{array}{cc} \ov{\rho} & \ov{\vm}^\trans \\ \ov{\vm} & \ov{\mU} \end{array}\right)\cdot \veta=\vz\ \right\}\ed
\end{equation}
In other words $(\ov{\rho},\ov{\vm},\ov{\mU})\in\Lambda$ if and only if the kernel of the matrix 
$$
\left(\begin{array}{cc} \ov{\rho} & \ov{\vm}^\trans \\ \ov{\vm} & \ov{\mU} \end{array}\right)
$$ 
contains $\veta\neq \vz$. This holds if and only if its determinant is zero. Hence we can write
\begin{equation} \label{eq:wavecone-det}
\Lambda = \left\{(\ov{\rho},\ov{\vm},\ov{\mU})\in\R\times\R^n\times\szn\,\Big|\,\det\left(\begin{array}{cc}
\ov{\rho} & \ov{\vm}^\trans \\
\ov{\vm} & \ov{\mU}
\end{array}\right)=0\right\}\ed
\end{equation}

\subsection{Sketch of the Convex Integration Technique} \label{subsec:convint-prel-U}

In this subsection we shall sketch how convex integration works and give a rough idea how the relaxed set $\sU$ must be defined.

\begin{figure}[tb] 
	\centering
	\subfloat[$H_2$-condition.\label{fig:HN-2}]{
		\centering
		\includegraphics[width=0.43\textwidth]{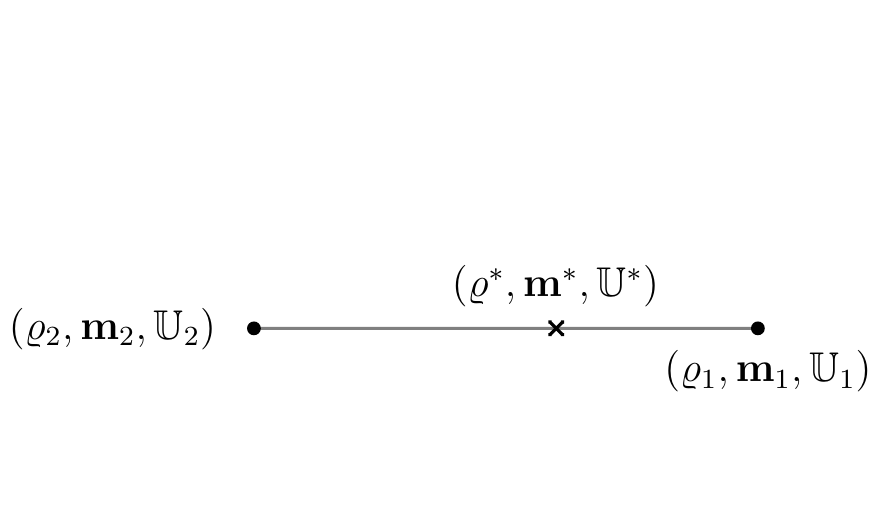}
	}
	\hspace{1.5cm}
	\subfloat{
		\centering
		\includegraphics[width=0.43\textwidth]{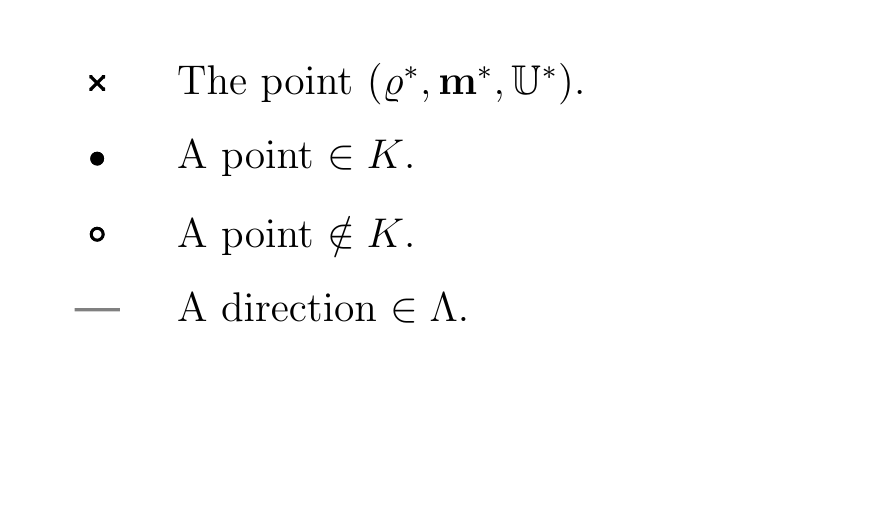}
	} \setcounter{subfigure}{1} \\[4mm]
	\subfloat[$H_3$-condition.\label{fig:HN-3}]{
		\centering
		\includegraphics[width=0.43\textwidth]{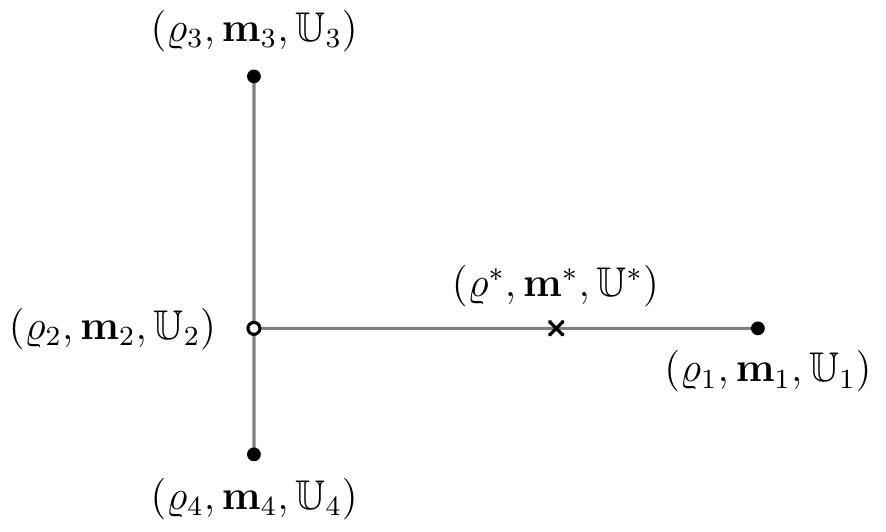}
	}
	\hspace{1.5cm}
	\subfloat[$H_4$-condition, version 1.\label{fig:HN-4_1}]{
		\centering 
		\includegraphics[width=0.43\textwidth]{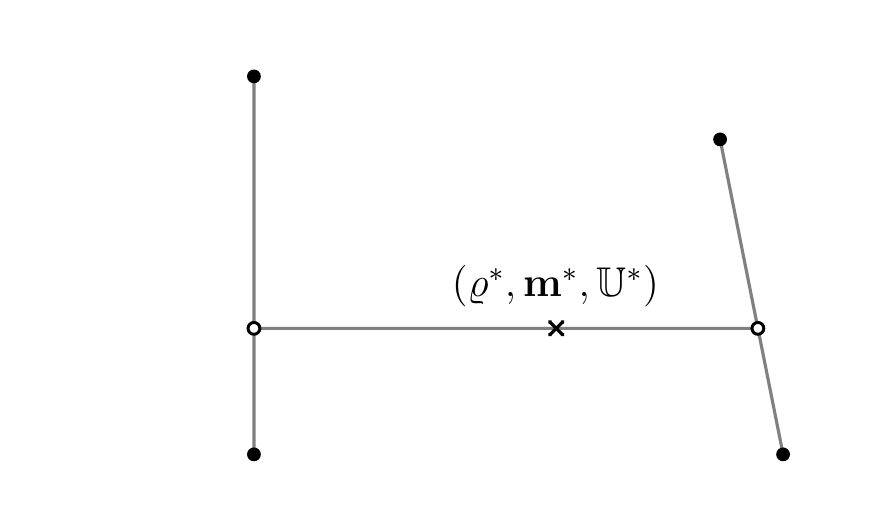}
	} \\[4mm]
	\subfloat[$H_4$-condition, version 2.\label{fig:HN-4_2}]{
		\centering
		\includegraphics[width=0.43\textwidth]{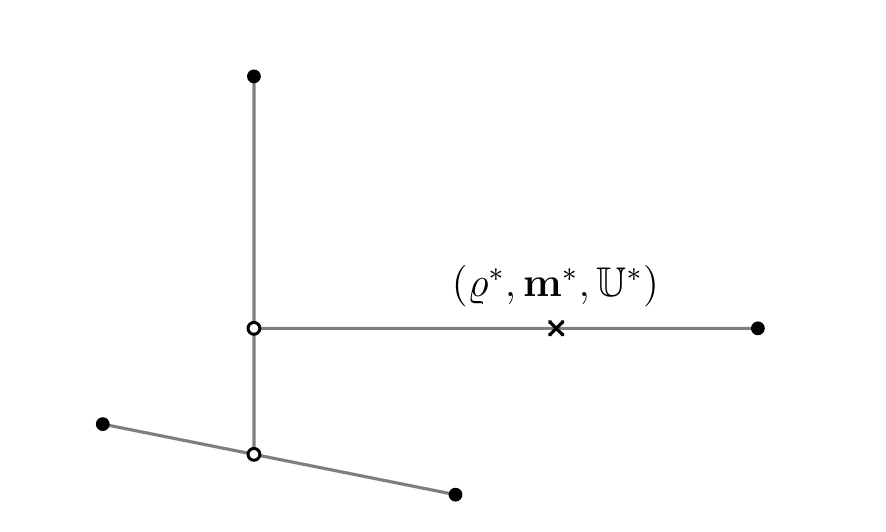}
	}
	\hspace{1.5cm}
	\subfloat[$H_5$-condition.\label{fig:HN-5}]{
		\centering
		\includegraphics[width=0.43\textwidth]{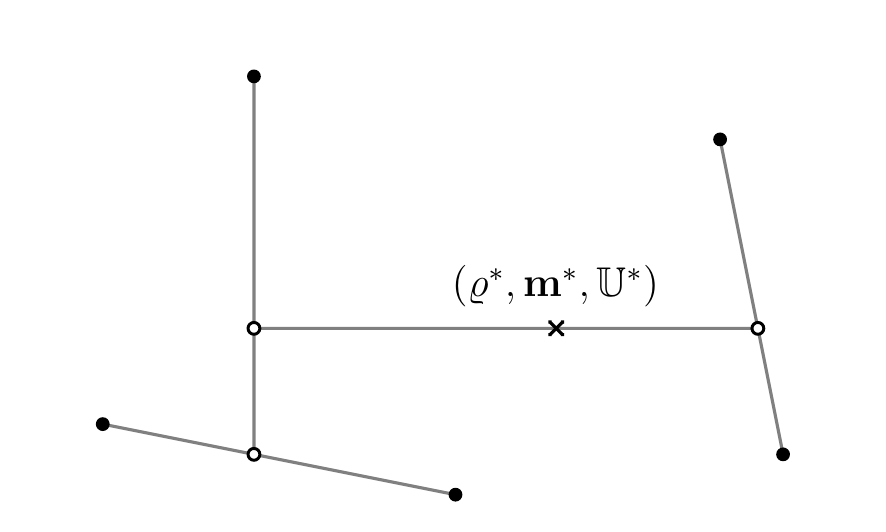}
	}
	\caption{Sample configurations in the extended phase space, which illustrate the $H_N$-condition introduced in Section~\ref{sec:convint-convex-hulls}.} 
	\label{fig:HN}
\end{figure}

The convex integration technique as used in this book in Sections \ref{sec:convint-thm} and \ref{sec:convint-pp} can be explained as follows. We start with one subsolution $(\rho,\vm,\mU)$ whose existence must be guaranteed somehow. In Theorem~\ref{thm:convint} we \emph{assume} existence of such a subsolution and shift the issue of existence to Chapters \ref{chap:appl-ibvp} and \ref{chap:appl-riemann}, where Theorem~\ref{thm:convint} is applied. As mentioned above, we want to construct another subsolution $(\rho_\pert,\vm_\pert,\mU_\pert)$ by adding plane waves to $(\rho,\vm,\mU)$ such that $(\rho_\pert,\vm_\pert,\mU_\pert)$ is in some sense closer to $K$ than $(\rho,\vm,\mU)$, but still in another sense close to $(\rho,\vm,\mU)$. To this end we divide the space-time domain, on which we want to solve the Euler system, into small cubes, on each of which we approximate $(\rho,\vm,\mU)$ by a constant. These constants lie in $\sU$ since $(\rho,\vm,\mU)$ is a subsolution and therefore takes values in $\sU$. In order to sketch how $\sU$ must be defined, let us first assume for simplicity that there was only one cube on which $(\rho,\vm,\mU)$ is approximated by a constant $(\rho^\ast,\vm^\ast,\mU^\ast)\in \sU$.

If $(\rho^\ast,\vm^\ast,\mU^\ast)\in K$, then the subsolution $(\rho,\vm,\mU)$ is already close to $K$. Hence $\sU$ shall contain $K$, i.e. $K\subset \sU$. Now assume that there exist $(\rho_1,\vm_1,\mU_1),(\rho_2,\vm_2,\mU_2)\in K$ with 
\begin{align}
	& \bullet \ (\rho_2,\vm_2,\mU_2)-(\rho_1,\vm_1,\mU_1)\ \in\ \Lambda \qquad \text{ and } \hspace{6.4cm} \label{eq:h2a} \\[2mm]
	& \bullet \ (\rho^\ast,\vm^\ast,\mU^\ast) \ \in\ \big[ (\rho_1,\vm_1,\mU_1) , (\rho_2,\vm_2,\mU_2) \big]\ec \label{eq:h2b}
\end{align}
see Figure~\ref{fig:HN-2}. According to Proposition~\ref{prop:planarwave}, we can thus find a plane wave solution to \eqref{eq:euler-lin-dens}, \eqref{eq:euler-lin-mom} oscillating between $(\rho_1,\vm_1,\mU_1)$ and $(\rho_2,\vm_2,\mU_2)$. Since the profile $h$ can be chosen arbitrarily, let $h$ jump between two values such that 
$$
(\rho^\ast,\vm^\ast,\mU^\ast) + \Big((\rho_2,\vm_2,\mU_2)-(\rho_1,\vm_1,\mU_1)\Big) h\big((t,\vx)\cdot \veta\big)
$$
is either equal to $(\rho_1,\vm_1,\mU_1)$ or $(\rho_2,\vm_2,\mU_2)$, both of which lie in $K$. Hence adding this plane wave to the constant $(\rho^\ast,\vm^\ast,\mU^\ast)$ gives something which takes values in $K$. Since the constant $(\rho^\ast,\vm^\ast,\mU^\ast)$ is just an approximation of the subsolution $(\rho,\vm,\mU)$, the sum  
$$
(\rho_\pert,\vm_\pert,\mU_\pert):=(\rho,\vm,\mU) + \Big((\rho_2,\vm_2,\mU_2)-(\rho_1,\vm_1,\mU_1)\Big) h\big((t,\vx)\cdot \veta\big)
$$
only takes values close to $K$. In addition to that, the property of taking values close to $K$ does not hold for all $(t,\vx)$ in the space-time domain, but for ``most'' of the pairs $(t,\vx)$, due to the fact that we have to mollify the $h$ considered above to obtain $h\in C^1(\R)$. Furthermore taking $h$ such that it changes its values rapidly, we achieve that $(\rho_\pert,\vm_\pert,\mU_\pert)$ is close to $(\rho,\vm,\mU)$ with respect to the $L^\infty$ weak-$\ast$ topology, cf. Lemma~\ref{lemma:not-periodic-weak-convergence}. The whole procedure is described in datail in Section~\ref{sec:convint-pp}. As a consequence, the relaxed set $\sU$ shall also contain points in $\R\times \R^n\times \szn$ with the property of $(\rho^\ast,\vm^\ast,\mU^\ast)$ above, i.e. existence of $(\rho_1,\vm_1,\mU_1),(\rho_2,\vm_2,\mU_2)\in K$ with \eqref{eq:h2a}, \eqref{eq:h2b}. 

In addition to that, more complex configurations are conceivable. For instance consider existence of $(\rho_1,\vm_1,\mU_1),(\rho_2,\vm_2,\mU_2)$ with \eqref{eq:h2a}, \eqref{eq:h2b}, where $(\rho_1,\vm_1,\mU_1)\in K$ but $(\rho_2,\vm_2,\mU_2)\notin K$. Instead there exist $(\rho_3,\vm_3,\mU_3),(\rho_4,\vm_4,\mU_4)\in K$ with 
\begin{align}
	& \bullet \ (\rho_4,\vm_4,\mU_4)-(\rho_3,\vm_3,\mU_3)\ \in\ \Lambda \qquad \text{ and } \hspace{6.4cm} \label{eq:h3a} \\[2mm]
	& \bullet \ (\rho_2,\vm_2,\mU_2) \ \in\ \big[ (\rho_3,\vm_3,\mU_3) , (\rho_4,\vm_4,\mU_4) \big]\ec \label{eq:h3b}
\end{align}
see Figure~\ref{fig:HN-3}. Then iterating the process outlined above, i.e. adding two plane wave oscillations, leads to the desired result, where $(\rho_\pert,\vm_\pert,\mU_\pert)$ takes values close to $K$ for ``most'' of the pairs $(t,\vx)$, but on the other hand $(\rho_\pert,\vm_\pert,\mU_\pert)$ is still close to $(\rho,\vm,\mU)$ with respect to the $L^\infty$ weak-$\ast$ topology. Hence points in $\R\times \R^n\times \szn$ such that there exist $(\rho_1,\vm_1,\mU_1),(\rho_3,\vm_3,\mU_3),(\rho_4,\vm_4,\mU_4)\in K$ and $(\rho_2,\vm_2,\mU_2)\notin K$ with \eqref{eq:h2a}, \eqref{eq:h2b}, \eqref{eq:h3a} and \eqref{eq:h3b} shall lie in $\sU$ as well.

Some even more complex configurations are sketched in Figures \ref{fig:HN-4_1} - \ref{fig:HN-5}. The $H_N$-condition, which is introduced in Section~\ref{sec:convint-convex-hulls} below, serves as an elegant way to form the ideas above into a rigorous definition of the relaxed set $\sU$. Later in Section~\ref{sec:convint-U} we give such a definition and compute $\sU$.

Above we have described convex integration performed on just one cube. However in general we have to deal with several cubes in each of which we want to add a different plane wave. To this end we have to cut off the plane waves in order to get oscillations that are compactly supported in one cube. One could also say that we use \emph{localized} plane waves. This localization or cut off requires some effort, see Section~\ref{sec:convint-operators}. 

The procedure described above allows to construct a sequence of subsolutions getting closer and closer to $K$ and hence converging to a solution, see for example \name{De~Lellis} and \name{Sz{\'e}kelyhidi} \cite[Section 5]{DelSze09} where this has been carried out for the incompressible Euler equations. However this \emph{constructive approach} only yields \emph{one} solution. Since one has much freedom when designing the localized plane waves, one could construct another sequence that converges to another solution. Hence one obtains \emph{two} solutions. Iterating this argument one could prove existence of \emph{infinitely many} solutions. However we will use a much more elegant way to prove existence of infinitely many solutions instead of the constructive approach by using Baire category arguments. Here we consider the set of solutions $\Xi$ and show that this set contains infinitely many elements using the Baire category theorem. Convex integration as described above enters this \emph{Baire category approach} as a contradiction argument, see Sections \ref{sec:convint-thm} and \ref{sec:convint-pp} for details.

\section{On $\Lambda$-Convex Hulls} \label{sec:convint-convex-hulls} 

The statements presented in this section hold for every cone\footnote{See Definition~\ref{defn:cone} for a proper definition of a \emph{cone}.} $\Lambda\subseteq\R^M$ and every subset $K\subseteq\R^M$ with $M>1$. In Section~\ref{sec:convint-U} we will come back to $\Lambda$ and $K$ as given by \eqref{eq:wavecone-det} and \eqref{eq:K} respectively.

\subsection{Definitions and Basic Facts} \label{subsec:convint-ch-defns}

We discuss some general facts about $\Lambda$-convex hulls in this section. Let us begin with some definitions.

\begin{defn} \label{defn:cone}
	A set $\Lambda\subseteq\R^M$ is called \emph{cone} if $\Lambda\neq\emptyset$ and $\forall \,\vp\in \Lambda, \forall \alpha\in \R: \alpha \vp \in \Lambda$.
\end{defn}

\begin{defn} \label{defn:convex}
	\begin{itemize}
	\item A set $S\subseteq\R^M$ is called 
	\begin{itemize}
		\item \emph{convex} if $\forall \,\vp,\vq \in S : [\vp,\vq]\subseteq S$, 
		\item \emph{$\Lambda$-convex} if $\forall \,\vp,\vq \in S$ with $\vp-\vq\in\Lambda : [\vp,\vq]\subseteq S$.
	\end{itemize}	
	\item The \emph{$\Lambda$-convex hull} $K^\Lambda$ is the smallest $\Lambda$-convex set which contains $K$.
	\end{itemize}
\end{defn}

\begin{rem} 
	The $\Lambda$-convex hull $K^\Lambda$ of a set $K\subset \R^M$ as defined above is well-defined. The reason for this is the same as for the convex hull, see Proposition~\ref{prop:convhull=convcombis}~\ref{item:convex-hull.a}. 
\end{rem}

Let us continue with some basic facts about the notions defined above.
\begin{prop} \label{prop:generalfacts} 
	\begin{enumerate}
		\item \label{item:generalfacts.a} Every convex set is $\Lambda$-convex.
		\item \label{item:generalfacts.b} $K^\Lambda\subseteq K^\co$.
	\end{enumerate}
\end{prop}

\begin{proof}
	Part \ref{item:generalfacts.a} is an immediate consequence of Definition~\ref{defn:convex} and \ref{item:generalfacts.b} simply follows from \ref{item:generalfacts.a}.
\end{proof}

\subsection{The $H_N$-Condition and a Way to Define $\sU$} \label{subsec:convint-ch-hn}

We introduce the $H_N$-condition, cf. \name{Pedregal} \cite[Definition 9.1]{Pedregal} or \name{Crippa et al.} \cite[Definition 2.3 (a)]{CGSW17}.
\begin{defn} \label{defn:hn}
	Let $N\in \N$ and $(\tau_i,\vp_i)\in \R^+ \times \R^M$ for $i=1,...,N$. We say that the family of pairs $\big\{(\tau_i,\vp_i)\big\}_{i=1,...,N}$ satisfies the \emph{$H_N$-condition} if the following holds.
	\begin{itemize}
		\item If $N=1$, then $\tau_1=1$.
		\item If $N\geq 2$, then (after relabeling if necessary) $\vp_2-\vp_1 \in\Lambda$ and the family 
		\begin{equation} \label{eq:defn-hn-iteration}
			\left\{ \left( \tau_1 + \tau_2 , \frac{\tau_1}{\tau_1 + \tau_2} \vp_1 + \frac{\tau_2}{\tau_1 + \tau_2} \vp_2\right) \right\} \cup \big\{(\tau_i,\vp_i)\big\}_{i=3,...,N}
		\end{equation}
		satisfies the $H_{N-1}$-condition.
	\end{itemize}
\end{defn} 

\begin{defn}
	Let $N\in \N$ and $(\tau_i,\vp_i)\in \R^+ \times \R^M$ for $i=1,...,N$ with $\sum_{i=1}^N \tau_i =1 $. Then 
	$$
	\vp := \sum_{i=1}^N \tau_i \vp_i
	$$
	is the \emph{barycenter} of the family of pairs $\big\{(\tau_i,\vp_i)\big\}_{i=1,...,N}$.
\end{defn}

Before we define $\sU$, we state some simple facts concerning the $H_N$-condition.

\begin{lemma} \label{lemma:hn1}
	Let $N\in \N$ and $(\tau_i,\vp_i)\in\R^+\times \R^M$ for $i=1,...,N$ such that the family of pairs $\big\{(\tau_i,\vp_i)\big\}_{i=1,...,N}$ satisfies the $H_N$-condition. Then the following two claims are true.
	\begin{enumerate}
		\item \label{item:hn1.a} $\sum_{i=1}^N \tau_i = 1$;
		\item \label{item:hn1.b} If $N\geq 2$, then the family given by \eqref{eq:defn-hn-iteration} and the original family $\big\{(\tau_i,\vp_i)\big\}_{i=1,...,N}$ have the same barycenter.
	\end{enumerate}
\end{lemma}

\begin{proof}
	\begin{enumerate}
		\item We proceed by induction over $N\in\N$. If $N=1$, the claim immediately follows from Definition~\ref{defn:hn}. If $N\geq 2$, then 
		$$
			\left\{ \left( \tau_1 + \tau_2 , \frac{\tau_1}{\tau_1 + \tau_2} \vp_1 + \frac{\tau_2}{\tau_1 + \tau_2} \vp_2\right) \right\} \cup \big\{(\tau_i,\vp_i)\big\}_{i=3,...,N}
		$$
		satisfy the $H_{N-1}$-condition (after relabeling if necassary) and hence by induction hypothesis
		$$
			\sum_{i=1}^N \tau_i = (\tau_1 + \tau_2) + \sum_{i=3}^N \tau_i = 1 \ed
		$$
		\item The proof is finished as soon as we have shown that 
		$$
		(\tau_1 + \tau_2) \left(\frac{\tau_1}{\tau_1 + \tau_2} \vp_1 + \frac{\tau_2}{\tau_1 + \tau_2} \vp_2\right) + \sum_{i=3}^N \tau_i \vp_i = \sum_{i=1}^{N} \tau_i \vp_i \ed
		$$
		But this is obvious.
	\end{enumerate}
\end{proof}

Let us consider an example.

\begin{ex}
	Figure~\ref{fig:HN} illustrates Definition~\ref{defn:hn} in the case $\R^M \cong\R\times \R^n\times \szn$. For instance let us look at Figure~\ref{fig:HN-3}, where
	$$
	\Big\{ \big(\tau_1,(\rho_1,\vm_1,\mU_1)\big),\big(\tau_3,(\rho_3,\vm_3,\mU_3)\big),\big(\tau_4,(\rho_4,\vm_4,\mU_4)\big) \Big\}
	$$
	satisfy the $H_3$-condition with some suitable $\tau_1,\tau_3,\tau_4\in (0,1)$. Note that the barycenter of the family above is given by $(\rho^\ast,\vm^\ast,\mU^\ast)$. Furthermore in this example the family \eqref{eq:defn-hn-iteration} simply reads 
	$$
	\Big\{ \big(\tau_1,(\rho_1,\vm_1,\mU_1)\big),\big(\tau_2,(\rho_2,\vm_2,\mU_2)\big) \Big\}
	$$
	where $\tau_2=\tau_3+\tau_4$ and 
	$$
	(\rho_2,\vm_2,\mU_2) = \frac{\tau_3}{\tau_3 + \tau_4} (\rho_3,\vm_3,\mU_3) + \frac{\tau_4}{\tau_3 + \tau_4} (\rho_4,\vm_4,\mU_4) \ed
	$$
\end{ex}

\begin{lemma} \label{lemma:hn2}
	Let $N_1,N_2\in \N$ and $(\tau_i,\vp_i)\in\R^+\times \R^M$ for $i=1,...,N_1$, $(\mu_i,\vq_i)\in\R^+\times \R^M$ for $i=1,...,N_1$ such that the families $\big\{(\tau_i,\vp_i)\big\}_{i=1,...,N_1}$ and $\big\{(\mu_i,\vq_i)\big\}_{i=1,...,N_2}$ satisfy the $H_{N_1}$- and $H_{N_2}$-condition, respectively. Assume furthermore that $\vp - \vq \in \Lambda$, where $\vp$ and $\vq$ denote the barycenters of the families $\big\{(\tau_i,\vp_i)\big\}_{i=1,...,N_1}$ and $\big\{(\mu_i,\vq_i)\big\}_{i=1,...,N_2}$ respectively. Then for every $\sigma \in(0,1)$ the pairs
	$$
		\big\{\big(\sigma \tau_i,\vp_i\big)\big\}_{i=1,...,N_1} \cup \big\{\big( (1-\sigma)\mu_i,\vq_i \big)\big\}_{i=1,...,N_2}
	$$
	satisfy the $H_{N_1+N_2}$-condition.
\end{lemma}

\begin{proof}
	We proceed by induction. Let us start with the case $N_1=N_2=1$. Then $\tau_1=\mu_1=1$ and hence
	$$
		\big\{ \big(\sigma \tau_1 , \vp_1\big), \big((1-\sigma)\mu_1 , \vq_1\big) \big\} = \big\{ \big(\sigma , \vp_1\big), \big((1-\sigma) , \vq_1\big) \big\}
	$$
	satisfy the $H_2$-condition since $\vp_1 - \vq_1 \in \Lambda$ and 
	$$
		\big\{ \big(1 , \sigma \vp_1 + (1-\sigma) \vq_1\big) \big\}
	$$
	obviously satisfies the $H_1$-condition.
	
	Now let $N_1\geq 2$ and $N_2\in \N$. Then $\vp_2 - \vp_1 \in\Lambda$ and 
	\begin{equation} \label{eq:temp-defn-hn}
		\left\{ \left( \tau_1 + \tau_2 , \frac{\tau_1}{\tau_1 + \tau_2} \vp_1 + \frac{\tau_2}{\tau_1 + \tau_2} \vp_2\right) \right\} \cup \big\{(\tau_i,\vp_i)\big\}_{i=3,...,N_1}
	\end{equation}
	satisfy the $H_{N_1 - 1}$-condition by assumption (after relabeling if neccessary). Note that the barycenter of the family given in \eqref{eq:temp-defn-hn} is still $\vp$ according to Lemma~\ref{lemma:hn1}. Since $\vp-\vq\in \Lambda$ by assumption, the induction hypothesis yields, that 
	$$
		\left\{ \left( \sigma(\tau_1 + \tau_2) , \frac{\tau_1}{\tau_1 + \tau_2} \vp_1 + \frac{\tau_2}{\tau_1 + \tau_2} \vp_2\right) \right\} \cup \big\{(\sigma \tau_i,\vp_i)\big\}_{i=3,...,N_1} \cup \big\{\big( (1-\sigma)\mu_i,\vq_i \big)\big\}_{i=1,...,N_2}
	$$
	satisfy the $H_{N_1+N_2 - 1}$-condition. This immediately shows the claim due to the fact that the roles of $N_1$ and $N_2$ can be switched.
\end{proof}

Now we are ready to define the set $\sU$ as follows.
\begin{align}
\sU := \Bigg\{ \vp \in \R^M\,\Big|\,\exists &N\in\N, \exists(\tau_i,\vp_i)\in\R^+\times \R^M\text{ for all }i=1,...,N\text{ such that } \label{eq:U} \\
& \bullet \ \text{the family } \big\{(\tau_i,\vp_i)\big\}_{i=1,...,N}\text{ satisfies the }H_N\text{-condition\footnotemark} \ec \notag \\
& \bullet \ \vp_i\in K \text{ for all }i=1,...,N \text{ and } \notag \\
& \bullet \ \vp\text{ is the barycenter of the family }\big\{(\tau_i,\vp_i)\big\}_{i=1,...,N}\text{, i.e. }\vp=\sum_{i=1}^N \tau_i \vp_i \Bigg\} \ed\notag 
\end{align} 
\footnotetext{As pointed out in Lemma~\ref{lemma:hn1}, this implies that $\sum_{i=1}^N\tau_i = 1$.} 

The following proposition relates $\sU$ to the $\Lambda$-convex hull of $K$.

\begin{prop} \label{prop:KLambda=U}
	It holds that $K^\Lambda = \sU$.
\end{prop}	

\begin{proof} 
	We start with the inclusion $K^\Lambda \subset \sU$. To this end we show that $K\subset \sU$ and $\sU$ is $\Lambda$-convex. The desired inclusion $K^\Lambda \subset \sU$ then follows from the definition of $K^\Lambda$ as the smallest $\Lambda$-convex set which contains $K$. The fact that $K\subset\sU$ is quite easy: Let $\vp\in K$ then $(1,\vp)$ satisfies the $H_1$-condition, and hence $\vp\in\sU$. To prove that $\sU$ is $\Lambda$-convex, let $\vp,\vq\in \sU$ with $\vp-\vq\in\Lambda$, and $\vs\in [\vp,\vq]$. We have to show that $\vs\in\sU$. There exists $\sigma\in [0,1]$ with $\vs = \sigma\vp + (1-\sigma)\vq$. W.l.o.g. we may assume that $\sigma\in(0,1)$ (otherwise $\vs=\vp$ or $\vs=\vq$ and hence $\vs\in\sU$). Since $\vp,\vq\in \sU$ there exist two families of pairs $\big\{(\tau_i,\vp_i)\big\}_{i=1,...,N_1}$, $\big\{(\mu_i,\vq_i)\big\}_{i=1,...,N_2}$ which fulfill the $H_{N_1}$- and $H_{N_2}$-condition, respecitively, with $\vp_i\in K$ for all $i=1,...,N_1$ and $\vq_i\in K$ for all $i=1,...,N_2$, and such that $\vp$ and $\vq$ are the barycenters of the families $\big\{(\tau_i,\vp_i)\big\}_{i=1,...,N_1}$ and $\big\{(\mu_i,\vq_i)\big\}_{i=1,...,N_2}$, respectively. Due to Lemma~\ref{lemma:hn2} the family of pairs
	$$
		\big\{\big(\sigma \tau_i,\vp_i\big)\big\}_{i=1,...,N_1} \cup \big\{\big( (1-\sigma)\mu_i,\vq_i \big)\big\}_{i=1,...,N_2}
	$$
	satisfy the $H_{N_1+N_2}$-condition. It is simple to varify that its barycenter is $\vs$. Indeed,
	$$
		\sum_{i=1}^{N_1} \sigma \tau_i \vp_i + \sum_{i=1}^{N_2} (1-\sigma) \mu_i \vq_i = \sigma \vp + (1-\sigma) \vq = \vs\ed
	$$
	Hence $\vs\in\sU$.
	
	In order to prove the remaining inclusion $\sU \subset K^\Lambda$, let $\vp\in\sU$ arbitrary. Then by definition of the set $\sU$ there exist $N\in \N$, $(\tau_i,\vp_i)\in \R^+\times \R^M$ for $i=1,...,N$ such that the family $\{(\tau_i,\vp_i)\big\}_{i=1,...,N}$ satisfies the $H_N$-condition, $\vp_i\in K$ for all $i=1,...,N$ and $\vp$ is the barycenter of the family $\{(\tau_i,\vp_i)\big\}_{i=1,...,N}$. Note, that $\vp_i\in K^\Lambda$ since $K\subset K^\Lambda$. In order to show that $\vp\in K^\Lambda$ we prove the following by induction over $N\in \N$. 
	
	\smallskip
	
	\textbf{Claim:} If the family $\big\{(\tau_i,\vp_i)\big\}_{i=1,...,N}$ satisfies the $H_N$-condition and $\vp_i\in K^\Lambda$, then the barycenter $\vp$ of $\big\{(\tau_i,\vp_i)\big\}_{i=1,...,N}$ lies in $K^\Lambda$.
	
	\smallskip
	
	For $N=1$ the claim is trivial because $\vp=\vp_1\in K^\Lambda$. If $N\geq 2$, we have (after relabeling if necessary) $\vp_2-\vp_1\in\Lambda$ and the family 
	\begin{equation} \label{eq:temp-defn-hn-2}
		\left\{ \left( \tau_1 + \tau_2 , \frac{\tau_1}{\tau_1 + \tau_2} \vp_1 + \frac{\tau_2}{\tau_1 + \tau_2} \vp_2\right) \right\} \cup \big\{(\tau_i,\vp_i)\big\}_{i=3,...,N}
	\end{equation}
	satisfies the $H_{N-1}$-condition. Moreover $\frac{\tau_1}{\tau_1 + \tau_2} \vp_1 + \frac{\tau_2}{\tau_1 + \tau_2} \vp_2 \in K^\Lambda$ since $\vp_1,\vp_2\in K^\Lambda$, $\vp_2-\vp_1\in\Lambda$ and $\frac{\tau_1}{\tau_1 + \tau_2} + \frac{\tau_2}{\tau_1 + \tau_2} = 1$. Hence the induction hypothesis says that the barycenter of the family in \eqref{eq:temp-defn-hn-2} lies in $K^\Lambda$. According to Lemma~\ref{lemma:hn1} this barycenter coincides with $\vp$. This finishes the proof.
\end{proof}

\subsection{The $\Lambda$-Convex Hull of Slices} \label{subsec:convint-ch-slices}

The following proposition is a simple observation. However it will be an important ingredient when finding the $\Lambda$-convex hull for the Euler equations, see the proof of Proposition~\ref{prop:compKLambda}.

For conveniece we write 
$$
\{p_1=\alpha\} := \left\{\vp\in \R^M \,\big|\, p_1 = \alpha\right\} \ec
$$
for a given $\alpha\in\R$.

\begin{prop} \label{prop:Kstar} 
	Define $K^\ast:=\bigcup_{\alpha\in \R} \left(K\cap \{p_1 = \alpha\}\right)^\Lambda$. Then the following statements hold:
	\begin{enumerate}
		\item \label{item:Kstar.a} $K^\ast \subset K^\Lambda$; 
		\item \label{item:Kstar.b} If $K^\ast = K^\Lambda$, then $\left(K\cap \{p_1 = \alpha\}\right)^\Lambda = K^\Lambda \cap \{p_1 = \alpha\}$ for any fixed $\alpha\in\R$.
	\end{enumerate}
\end{prop}

\begin{proof} 
	Let $\alpha\in\R$ be fixed. We first show that 
	\begin{equation} \label{eq:temp-Kstar}
		\left(K\cap \{p_1 = \alpha\}\right)^\Lambda \subset K^\Lambda \cap \{p_1 = \alpha\} \ed
	\end{equation}
	According to Proposition~\ref{prop:KLambda=U} and \eqref{eq:U} we can write an arbitrary $\vp\in \left(K\cap \{p_1 = \alpha\}\right)^\Lambda$ as $\vp = \sum_{i=1}^N \tau_i \vp_i$ where $\big\{(\tau_i,\vp_i)\big\}_{i=1,...,N}$ satisfy the $H_N$-condition and $\vp_i\in K\cap \{p_1 = \alpha\}$ for all $i=1,...,N$. In particular these $\vp_i$ lie in $K$ and hence the same argumentation yields $\vp\in K^\Lambda$. Moreover $p_1=\sum_{i=1}^N \tau_i [\vp_i]_1 = \sum_{i=1}^N \tau_i \alpha = \alpha$, i.e. $\vp\in \{p_1 = \alpha\}$. Hence \eqref{eq:temp-Kstar} is proven. 
	
	\begin{enumerate}
		\item From \eqref{eq:temp-Kstar} we immediately obtain claim \ref{item:Kstar.a}. 
		\item Because of \eqref{eq:temp-Kstar} it suffices to show 
		$$
			\left(K\cap \{p_1 = \alpha\}\right)^\Lambda \supset K^\Lambda \cap \{p_1 = \alpha\} \ed
		$$
		Let $\vp\in K^\Lambda \cap \{p_1 = \alpha\}$. Together with the assumption and another application of \eqref{eq:temp-Kstar}, we deduce
		$$
			\vp\in K^\ast \cap \{p_1 = \alpha\} = \bigg(\bigcup_{\beta\in \R} \left(K\cap \{p_1 = \beta\}\right)^\Lambda\bigg) \cap \{p_1 = \alpha\} =  \left(K\cap \{p_1 = \alpha\}\right)^\Lambda \ed
		$$
	\end{enumerate}
\end{proof} 

\begin{figure}[tb] 
	\centering
	\subfloat[$K$ (black line).\label{fig:Kstar-K}]{ 
		\centering
		\includegraphics[width=0.43\textwidth]{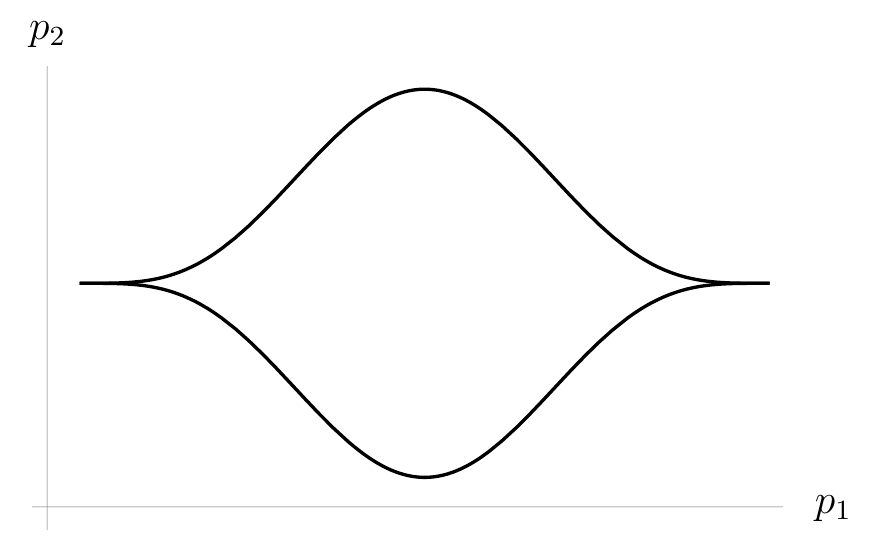}
	}
	\hspace{1.5cm}
	\subfloat[$K\cap \{p_1 = \alpha\}$ (blue points) for a particular $\alpha\in \R$, and $\left(K\cap \{p_1 = \alpha\}\right)^\Lambda$ (blue line segment).\label{fig:Kstar-alpha}]{
		\centering
		\includegraphics[width=0.43\textwidth]{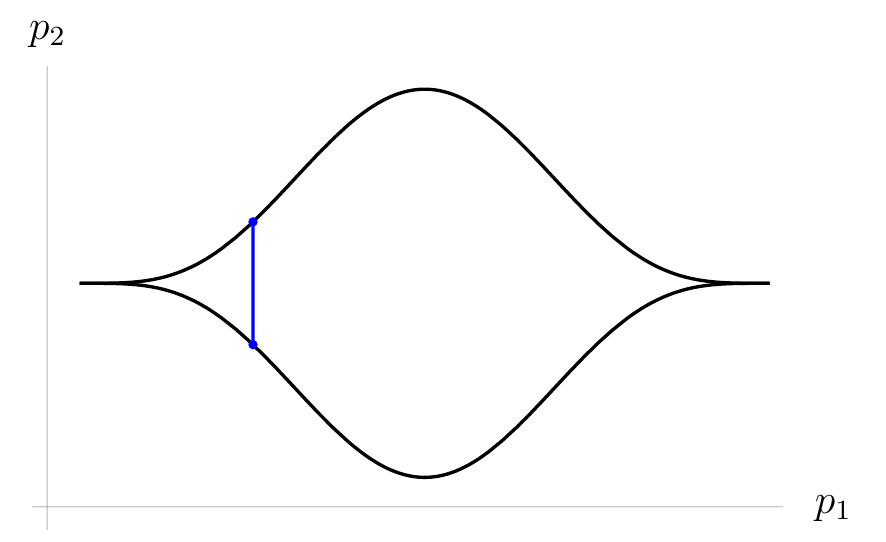}
	} \\[4mm]
	\subfloat[$K^\ast$ (blue region).\label{fig:Kstar-Kstar}]{
		\centering
		\includegraphics[width=0.43\textwidth]{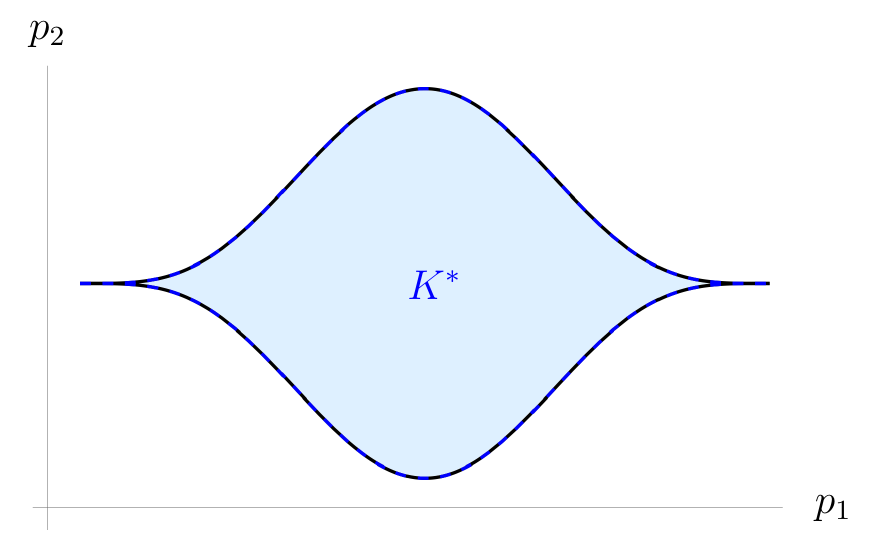}
	}
	\hspace{1.5cm}
	\subfloat[$K^\Lambda$ (red region).\label{fig:Kstar-KLambda}]{
		\centering 
		\includegraphics[width=0.43\textwidth]{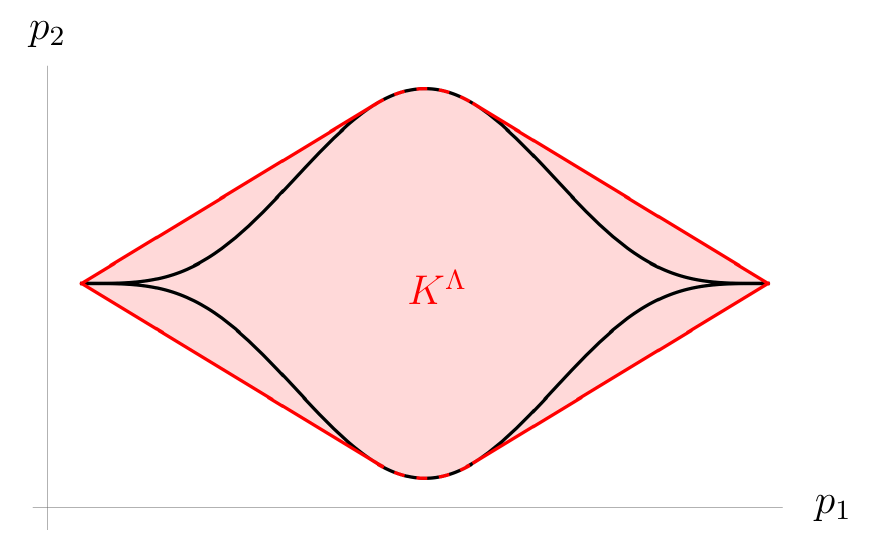}
	} 
	\caption{Example of a set $K$ and the corresponding $K^\ast$ and $K^\Lambda$.} 
	\label{fig:Kstar}
\end{figure}

\begin{ex} 
	To illustrate statement \ref{item:Kstar.a} of Proposition~\ref{prop:Kstar} let us consider $M=2$, $K$ given as in Figure~\ref{fig:Kstar-K} and $\Lambda=\R^2$. The latter condition implies that the notions \emph{convex} and \emph{$\Lambda$-convex} are equivalent. Hence $K^\Lambda = K^\co$ as well as $\left(K\cap \{p_1 = \alpha\}\right)^\Lambda = \left(K\cap \{p_1 = \alpha\}\right)^\co$ for all $\alpha\in\R$. The set $K\cap \{p_1 = \alpha\}$ consists of a most two points for each fixed $\alpha\in \R$. Hence $\left(K\cap \{p_1 = \alpha\}\right)^\Lambda$ is a straight line segment or empty for each $\alpha\in \R$. The set $\left(K\cap \{p_1 = \alpha\}\right)^\Lambda$ is depicted in Figure~\ref{fig:Kstar-alpha} for a particular $\alpha\in \R$ by a blue line. As a consequence $K^\ast$ is given by the blue colored region in Figure~\ref{fig:Kstar-Kstar}. On the other hand $K^\Lambda$ is equal to the red colored set in Figure~\ref{fig:Kstar-KLambda}. Observe that $K^\ast\subset K^\Lambda$ as stated by Proposition~\ref{prop:Kstar}~\ref{item:Kstar.a}.
\end{ex}

\subsection{The $\Lambda$-Convex  Hull if the Wave Cone is Complete} \label{subsec:convint-ch-complete}
 
Let us now study the important case where the wave cone is complete.

\begin{defn}
	The wave cone $\Lambda\subset \R^M$ is called \emph{complete with respect to $K\subset \R^M$} if $\vp-\vq\in \Lambda$ for all $\vp,\vq\in K$.
\end{defn}

The final goal of this subsection is to prove the following fact, which will follow from Proposition~\ref{prop:complete-wc}.

\begin{cor} \label{cor:complete-wc}
	If $\Lambda$ is complete with respect to $K$ then $K^\co=K^\Lambda$.
\end{cor}

The inclusion $K^\Lambda\subset K^\co$ is always true (i.e. not only when the wave cone is complete) and has already been shown in Proposition~\ref{prop:generalfacts}. For the other inclusion  $K^\co\subset K^\Lambda$ we would like to proceed as follows. Let $\vp\in K^\co$ be arbitrary. According to Proposition~\ref{prop:convhull=convcombis}~\ref{item:convex-hull.b} there exist $N\in \N$ and $(\mu_j,\vq_j)\in \R^+ \times K$ for all $j=1,...,N$ such that $\sum_{j=1}^N \mu_j = 1$ and $\vp$ is the barycenter of the family $\big\{(\mu_j,\vq_j)\big\}_{j=1,...,N}$. If we could show that this family satisfies the $H_N$-condition, we were done because this would mean that $\vp\in \sU = K^\Lambda$, cf. \eqref{eq:U} and Proposition~\ref{prop:KLambda=U}. However, the family of pairs $\big\{(\mu_j,\vq_j)\big\}_{j=1,...,N}$ does not satisfy the $H_N$-condition in general, see Example~\ref{ex:complete-wc}. But we can construct a new family $\big\{(\tau_i,\vp_i)\big\}_{i=1,...,2^{N-1}}$ depending on the family $\big\{(\mu_j,\vq_j)\big\}_{j=1,...,N}$ and having the desired properties.

\begin{prop} \label{prop:complete-wc}
	Let $N\in \N$, $(\mu_j,\vq_j)\in \R^+\times \R^M$ for $j=1,...,N$ with $\sum_{j=1}^N \mu_j = 1$. Let furthermore the wave cone $\Lambda$ be such that $\vq_{j_1} - \vq_{j_2} \in \Lambda$ for all $j_1,j_2\in\{1,...,N\}$. Then there exist $(\tau_i,\vp_i)\in \R^+ \times \R^M$ for $i=1,...,2^{N-1}$ such that 
	\begin{enumerate}
		\item \label{item:comp-wc.a} the family $\big\{(\tau_i,\vp_i)\big\}_{i=1,...,2^{N-1}}$ satisfies the $H_{2^{N-1}}$-condition,
		\item \label{item:comp-wc.b} for all $i\in\{1,...,2^{N-1}\}$ there exists $j\in\{1,...,N\}$ such that $\vp_i = \vq_j$, and
		\item \label{item:comp-wc.c} the families $\big\{(\tau_i,\vp_i)\big\}_{i=1,...,2^{N-1}}$ and $\big\{(\mu_j,\vq_j)\big\}_{j=1,...,N}$ have the same barycenter, i.e. 
		$$
			\sum_{i=1}^{2^{N-1}} \tau_i \vp_i = \sum_{j=1}^N \mu_j \vq_j \ed
		$$
	\end{enumerate}
\end{prop}

\begin{proof} 
	We start by setting $(\tau_1,\vp_1) := (\mu_1,\vq_1)$. If $N\geq 2$, we additionally define $(\tau_i,\vp_i)$ for $i\geq 2$ inductively: Assume that $(\tau_i,\vp_i)$ have been set for $i\leq 2^{k-2}$ with $k\in\{2,...,N\}$. Then we define $(\tau_i,\vp_i)$ for $i\in\{2^{k-2}+1,...,  2^{k-1}\}$ by 
	\begin{align*}
		\tau_i &:= \frac{\mu_k}{\sum_{j=1}^{k-1}\mu_j} \tau_{i- 2^{k-2}} \ec \\
		\vp_i &:= \vq_k \ed
	\end{align*} 
		
	Hence \ref{item:comp-wc.b} is obvious. In order to show \ref{item:comp-wc.a}, we prove the following claim by induction over $k=1,...,N$. Then \ref{item:comp-wc.a} coincides with the case $k=N$.
	
	\medskip
	
	\textbf{Claim:} The family\footnote{Note that the sum, where $j=k+1,...,N$, might be an \emph{empty} sum. More precisely, this is the case if $k=N$. As common, we define the empty sum to be equal to $0$.}
	\begin{equation} \label{eq:complete-wc-family}
		\left\{\left( \frac{\tau_i}{\sum_{j=1}^k \mu_j} , \Big(\sum_{j = 1}^k \mu_j\Big) \vp_i + \sum_{j = k+1}^N \mu_j \vq_j \right) \right\}_{i=1,...,2^{k-1}}
	\end{equation}
	satisfies the $H_{2^{k-1}}$-condition for all $k=1,...,N$. 
	
	\medskip
	
	For $k=1$ we obtain 
	$$
		\left\{\left( \frac{\tau_1}{\mu_1} , \mu_1 \vp_1 + \sum_{j = 2}^N \mu_j \vq_j \right) \right\}
	$$
	which satisfies the $H_1$-condition due to the fact that $\tau_1=\mu_1$. For $k\geq 2$ we can write 
	\begin{align*}
		& \left\{\left( \frac{\tau_i}{\sum_{j=1}^k \mu_j} , \Big(\sum_{j = 1}^k \mu_j\Big) \vp_i + \sum_{j = k+1}^N \mu_j \vq_j \right) \right\}_{i=1,...,2^{k-1}} \\ 
		&= \left\{\left( \frac{\tau_i}{\sum_{j=1}^k \mu_j} , \Big(\sum_{j = 1}^k \mu_j\Big) \vp_i + \sum_{j = k+1}^N \mu_j \vq_j \right) \right\}_{i=1,...,2^{k-2}} \\
		&\qquad \cup \left\{\left( \frac{\tau_i}{\sum_{j=1}^k \mu_j} , \Big(\sum_{j = 1}^k \mu_j\Big) \vp_i + \sum_{j = k+1}^N \mu_j \vq_j \right) \right\}_{i=2^{k-2} + 1,...,2^{k-1}} \\
		&= \left\{\left( \frac{\tau_i}{\sum_{j=1}^k \mu_j} , \Big(\sum_{j = 1}^k \mu_j\Big) \vp_i + \sum_{j = k+1}^N \mu_j \vq_j \right) \right\}_{i=1,...,2^{k-2}} \\
		&\qquad \cup \left\{\left( \frac{\tau_i \mu_k}{\left(\sum_{j=1}^k \mu_j \right) \left(\sum_{j=1}^{k-1} \mu_j \right)} , \Big(\sum_{j = 1}^k \mu_j\Big) \vq_k + \sum_{j = k+1}^N \mu_j \vq_j \right) \right\}_{i=1,...,2^{k-2}} \ed
	\end{align*}
	Moreover we have 
	$$
		\left(\sum_{j = 1}^k \mu_j\right) \vp_i + \sum_{j = k+1}^N \mu_j \vq_j - \left(\sum_{j = 1}^k \mu_j\right) \vq_k - \sum_{j = k+1}^N \mu_j \vq_j = \left(\sum_{j = 1}^k \mu_j\right) (\vp_i - \vq_k) \in \Lambda
	$$
	by \ref{item:comp-wc.b} and assumption. Therefore we are done as soon as we have shown that $\big\{(\sigma_i,\vs_i)\big\}_{i=1,...,2^{k-2}}$ satisfies the $H_{2^{k-2}}$-condition, where 
	\begin{align*} 
		 \sigma_i &:= \frac{\tau_i}{\sum_{j=1}^k \mu_j} + \frac{\tau_i \mu_k}{\left(\sum_{j=1}^k \mu_j \right) \left(\sum_{j=1}^{k-1} \mu_j \right)} \\
		 &= \frac{\tau_i}{\sum_{j=1}^k \mu_j} \left( 1 + \frac{\mu_k}{\sum_{j=1}^{k-1} \mu_j } \right) \\
		 &= \frac{\tau_i}{\sum_{j=1}^{k-1} \mu_j} \ec \\
		 \vs_i &:= \frac{\tau_i}{\sum_{j=1}^k \mu_j} \frac{\sum_{j=1}^{k-1} \mu_j}{\tau_i}  \left(\left(\sum_{j = 1}^k \mu_j\right) \vp_i + \sum_{j = k+1}^N \mu_j \vq_j \right) \\
		 &\qquad + \frac{\tau_i \mu_k}{\left(\sum_{j=1}^k \mu_j \right) \left(\sum_{j=1}^{k-1} \mu_j \right)}  \frac{\sum_{j=1}^{k-1} \mu_j}{\tau_i} \left( \left(\sum_{j = 1}^k \mu_j\right) \vq_k + \sum_{j = k+1}^N \mu_j \vq_j \right)  \\
		 &= \left(\sum_{j = 1}^{k-1} \mu_j\right) \vp_i + \mu_k \vq_k + \left(\frac{\sum_{j=1}^{k-1} \mu_j}{\sum_{j=1}^k \mu_j} + \frac{\mu_k}{\sum_{j=1}^k \mu_j }  \right) \sum_{j = k+1}^N \mu_j \vq_j \\
		 &= \left(\sum_{j = 1}^{k-1} \mu_j\right) \vp_i + \sum_{j = k}^N \mu_j \vq_j \ed
	\end{align*}
	But $\big\{(\sigma_i,\vs_i)\big\}_{i=1,...,2^{k-2}}$ satisfies the $H_{2^{k-2}}$-condition by induction hypothesis. 
	
	It remains to prove \ref{item:comp-wc.c}. A simple computation shows 
	\begin{align*}
		\sum_{i=1}^{2^{N-1}} \tau_i \vp_i &= \tau_1 \vp_1 + \sum_{k=2}^N \sum_{i=2^{k-2}+1}^{2^{k-1}} \tau_i \vp_i \\
		&= \mu_1 \vq_1 + \sum_{k=2}^N \left( \sum_{i=2^{k-2}+1}^{2^{k-1}} \tau_i \right) \vq_k \ed 
	\end{align*}
	If $N=1$, we are done. To show \ref{item:comp-wc.c} for $N\geq 2$, we show the following claim by induction over $k=2,...,N$.
	
	\medskip
	
	\textbf{Claim:} For all $k=2,...,N$ it holds that $\mu_k = \sum_{i=2^{k-2} + 1}^{2^{k-1}} \tau_i$.
	
	\medskip
	
	If $k=2$, then $\sum_{i=2}^2 \tau_i = \tau_2 = \frac{\mu_2}{\mu_1} \tau_1 = \mu_2$ since $\tau_1=\mu_1$. For $k> 2$ we obtain
	\begin{align*}
		\sum_{i=2^{k-2} + 1}^{2^{k-1}} \tau_i &= \frac{\mu_k}{\sum_{j=1}^{k-1}\mu_j} \sum_{i=2^{k-2} + 1}^{2^{k-1}} \tau_{i- 2^{k-2}} \\
		&= \frac{\mu_k}{\sum_{j=1}^{k-1}\mu_j} \sum_{i=1}^{2^{k-2}} \tau_i \\
		&= \frac{\mu_k}{\sum_{j=1}^{k-1}\mu_j} \left( \tau_1 + \sum_{i=2}^{2^{k-2}} \tau_i\right) \ed
	\end{align*}
	By induction hypothesis, this is equal to 
	\begin{align*}
	\frac{\mu_k}{\sum_{j=1}^{k-1}\mu_j} \left( \mu_1 + \sum_{j=2}^{k-1} \mu_j\right) &= \frac{\mu_k}{\sum_{j=1}^{k-1}\mu_j} \sum_{j=1}^{k-1} \mu_j \\
	&= \mu_k \ed
	\end{align*}
\end{proof} 

\begin{proof}[Proof of Corollary~\ref{cor:complete-wc}]
	The proof is sketched above, where the conclusion immediately follows from Proposition~\ref{prop:complete-wc}.
\end{proof} 

The following example shows that if $\vp$ is the barycenter of a family $\big\{(\mu_i,\vq_i)\big\}_{i=1,2,3}$ with $\mu_1+\mu_2+\mu_3=1$, then this family does not need to satisfy the $H_3$-condition. However according to Proposition\ref{prop:complete-wc} there is another family $\big\{(\tau_i,\vp_i)\big\}_{i=1,2,3,4}$ whose barycenter is again $\vp$ and which satisfies the $H_4$-condition ($4=2^{3-1}$).

\begin{figure}[tbp] 
	\centering
	\subfloat[The points $\vp$ and $\vq_1,\vq_2,\vq_3$.\label{fig:triangular-plain}]{ 
		\centering
		\includegraphics[width=0.43\textwidth]{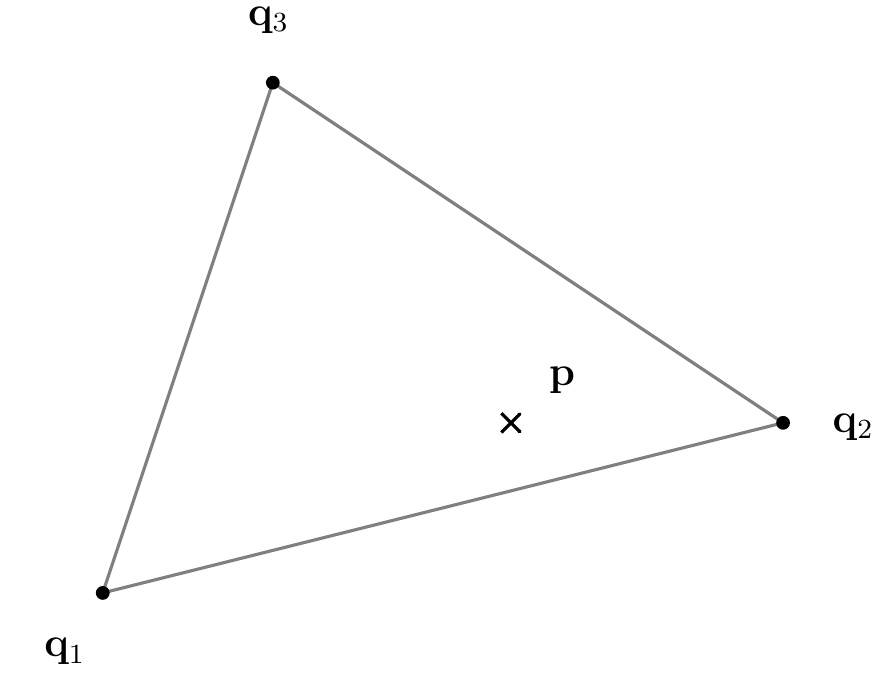}
	}
	\hspace{1.5cm}
	\subfloat{
		\centering
		\includegraphics[width=0.43\textwidth]{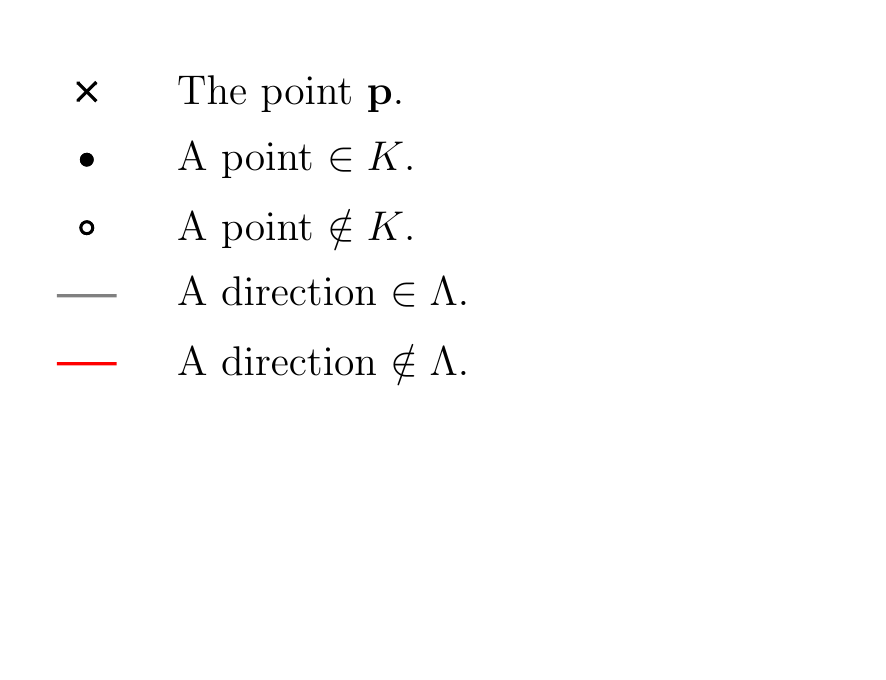} 
	} \setcounter{subfigure}{1} \\[5mm]
	\subfloat[The direction $\til{\vq}_1 - \vq_1 \notin \Lambda$.\label{fig:triangular-false1}]{
		\centering
		\includegraphics[width=0.43\textwidth]{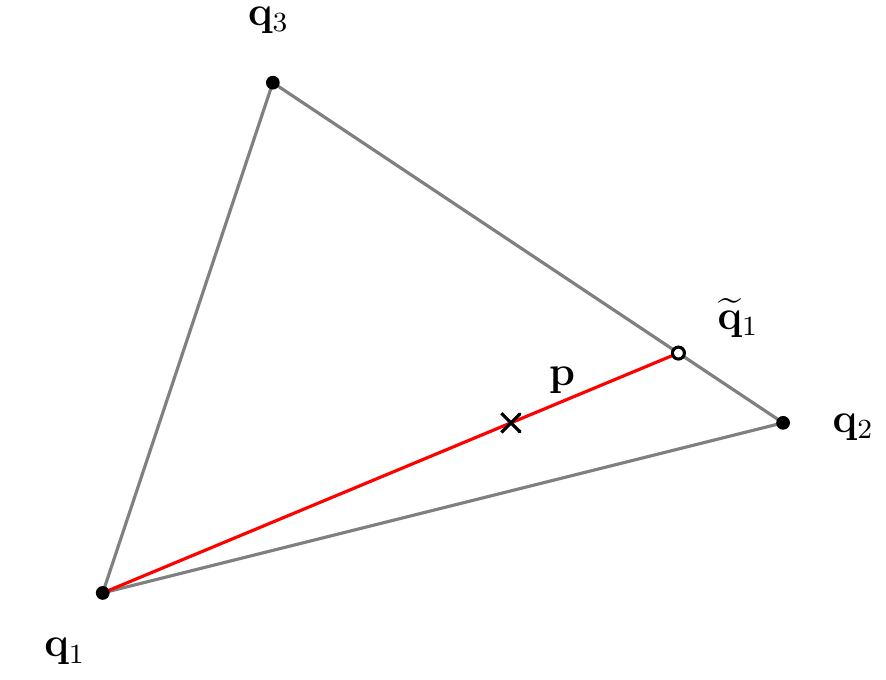}
	}
	\hspace{1.5cm}
	\subfloat[The direction $\til{\vq}_2 - \vq_2 \notin \Lambda$.\label{fig:triangular-false2}]{
		\centering
		\includegraphics[width=0.43\textwidth]{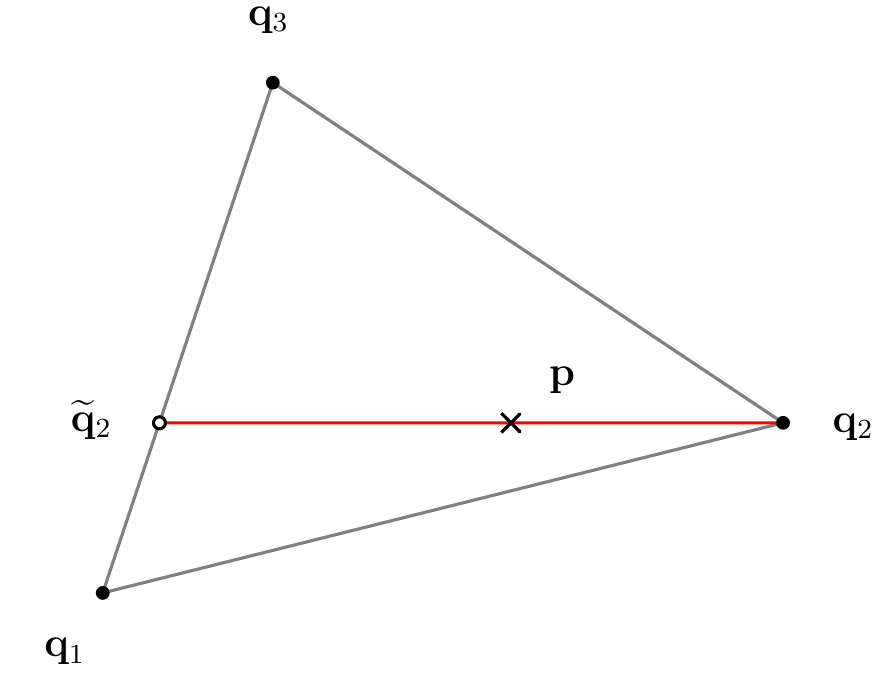}
	} \\[5mm]
	\subfloat[The direction $\til{\vq}_3 - \vq_3 \notin \Lambda$.\label{fig:triangular-false3}]{
		\centering
		\includegraphics[width=0.43\textwidth]{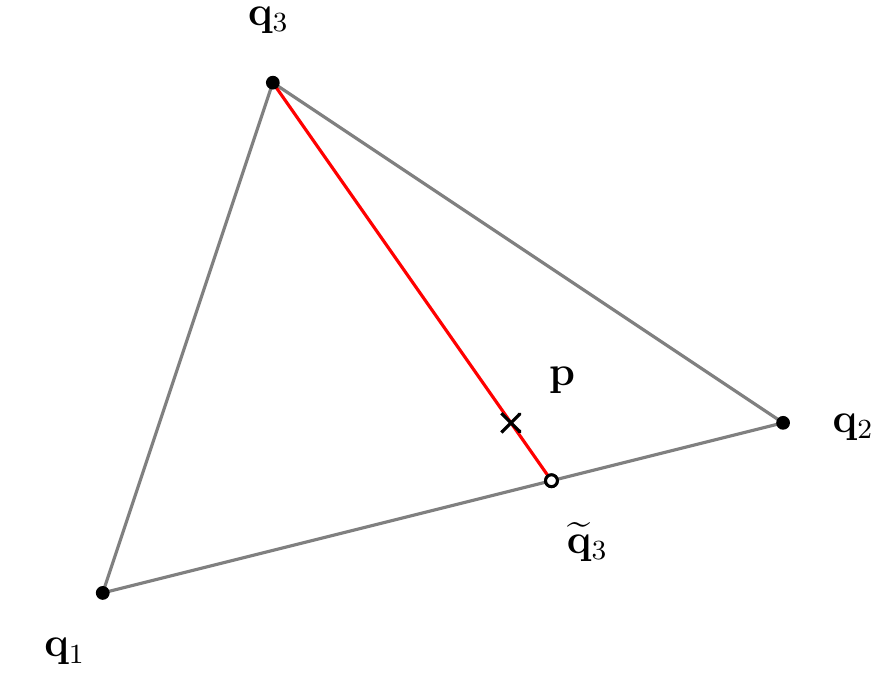}
	}
	\hspace{1.5cm}
	\subfloat[However the direction $\hat{\vq}_1 - \hat{\vq}_2 \in \Lambda$.\label{fig:triangular-true}]{
		\centering
		\includegraphics[width=0.43\textwidth]{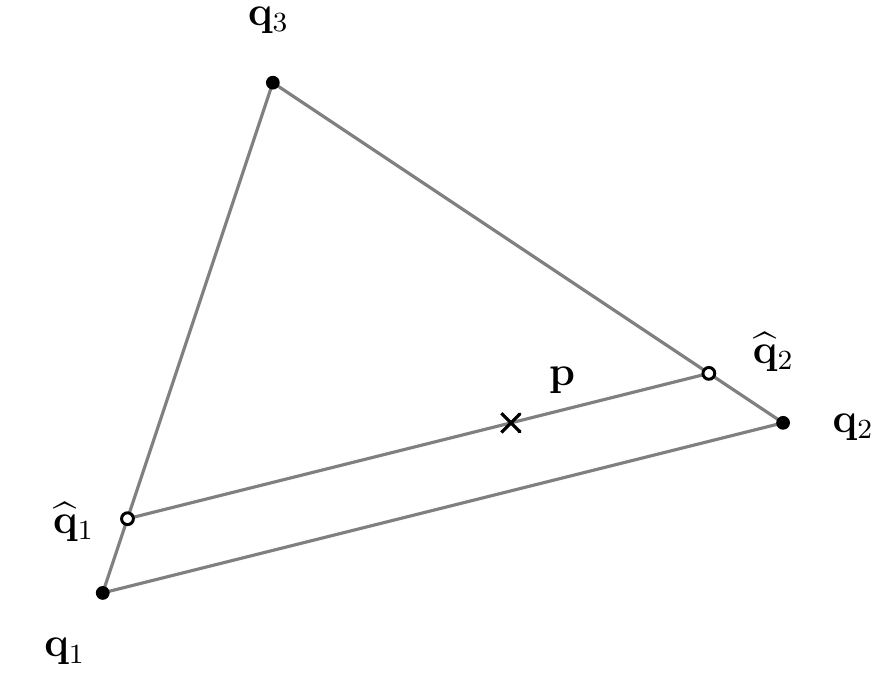}
	}
	\caption{Illustration of Example~\ref{ex:complete-wc}.}
	\label{fig:triangular} 
\end{figure}

\begin{ex} \label{ex:complete-wc} 
	Let $M=2$ and $K$ consist of three points $K:=\{\vq_1,\vq_2,\vq_3\}$. Let furthermore 
	\begin{align*}
		\Lambda:=\Big\{\vp\in\R^M\,\Big|\,\exists &\alpha\in\R\text{ such that one of the following cases holds: } \\
		& \bullet \ \vp=\alpha(\vq_1-\vq_2)\ec \\
		& \bullet \ \vp=\alpha(\vq_2-\vq_3)\ec \\
		& \bullet \ \vp=\alpha(\vq_1-\vq_3) \qquad \Big\} \ed
	\end{align*}
	Then $\Lambda$ is complete with respect to $K$ by construction. Furthermore $K^\co$ is the simplex (here a triangle) spanned by $\vq_1,\vq_2,\vq_3$. Let $\vp\in \interior{(K^\co)}$, then there exist $\mu_1,\mu_2,\mu_3\in \R^+$ such that $\vp=\mu_1 \vq_1 + \mu_2 \vq_2 + \mu_3 \vq_3$, see Figure~\ref{fig:triangular-plain}. It holds that 
	\begin{itemize}
		\item $\til{\vq}_3 - \vq_3 := \left(\frac{\mu_1}{\mu_1 + \mu_2}\vq_1 + \frac{\mu_2}{\mu_1 + \mu_2}\vq_2\right) - \vq_3 \notin \Lambda$, 
		\item $\til{\vq}_1 - \vq_1 := \left(\frac{\mu_2}{\mu_2 + \mu_3}\vq_2 + \frac{\mu_3}{\mu_2 + \mu_3}\vq_3\right) - \vq_1 \notin \Lambda$,
		\item $\til{\vq}_2 - \vq_2 := \left(\frac{\mu_1}{\mu_1 + \mu_3}\vq_1 + \frac{\mu_3}{\mu_1 + \mu_3}\vq_3\right) - \vq_2 \notin \Lambda$,
	\end{itemize}
	see Figures \ref{fig:triangular-false1} - \ref{fig:triangular-false3}.
	Hence none of the families
	\begin{itemize}
		\item $\left\{\left(\mu_1 + \mu_2 , \frac{\mu_1}{\mu_1 + \mu_2}\vq_1 + \frac{\mu_2}{\mu_1 + \mu_2}\vq_2\right),(\mu_3,\vq_3)\right\}$,
		\item $\left\{\left(\mu_2 + \mu_3 , \frac{\mu_2}{\mu_2 + \mu_3}\vq_2 + \frac{\mu_3}{\mu_2 + \mu_3}\vq_3\right),(\mu_1,\vq_1)\right\}$,
		\item $\left\{\left(\mu_1 + \mu_3 , \frac{\mu_1}{\mu_1 + \mu_3}\vq_1 + \frac{\mu_3}{\mu_1 + \mu_3}\vq_3\right),(\mu_2,\vq_2)\right\}$
	\end{itemize}
	satisfies the $H_2$-condition. This implies that the family $\big\{(\mu_1,\vq_1),(\mu_2,\vq_2),(\mu_3,\vq_3)\big\}$ does not satisfy the $H_3$-condition. 
	
	But the family 
	\begin{equation} \label{eq:temp-defn-hn-3}
		\left\{\left(\mu_1,\vq_1\right),\left(\mu_2,\vq_2\right),\left(\frac{\mu_1 \mu_3}{\mu_1+\mu_2},\vq_3\right),\left(\frac{\mu_2 \mu_3}{\mu_1+\mu_2},\vq_3\right)\right\} \ec
	\end{equation}
	which is constructed as in the proof of Proposition~\ref{prop:complete-wc}, satisfies the $H_4$-condition. This can be shown as follows. Because 
	$$
		\hat{\vq}_1 - \hat{\vq}_2 = (\mu_1+\mu_2) \vq_1 + \mu_3 \vq_3 - (\mu_1+\mu_2) \vq_2 - \mu_3 \vq_3 = (\mu_1+\mu_2) (\vq_1 - \vq_2)\in \Lambda \ec
	$$
	the family
	$$
		\left\{\left(\frac{\mu_1}{\mu_1+\mu_2} , (\mu_1+\mu_2) \vq_1 + \mu_3 \vq_3\right),\left( \frac{\mu_2}{\mu_1+\mu_2} , (\mu_1+\mu_2) \vq_2 + \mu_3 \vq_3\right)\right\}
	$$
	satisfies the $H_2$-condition. Due to $\vq_1-\vq_3 \in \Lambda$ this implies that 
	$$
		\left\{\left(\mu_1,\vq_1\right),\left(\frac{\mu_1 \mu_3}{\mu_1+\mu_2},\vq_3\right),\left( \frac{\mu_2}{\mu_1+\mu_2} , (\mu_1+\mu_2) \vq_2 + \mu_3 \vq_3\right)\right\}
	$$
	fulfills the $H_3$-condition and because of $\vq_1-\vq_2 \in \Lambda$ we deduce that \eqref{eq:temp-defn-hn-3} satisfies the $H_4$-condition. See also Figure~\ref{fig:triangular-true}. Note that Figure~\ref{fig:triangular-true} can be viewed as a degenerated version of Figure~\ref{fig:HN-4_1}, where two endpoints coincide.
\end{ex}

\section{The Relaxed Set $\sU$ Revisited} \label{sec:convint-U}

Now we turn our attention back to the Euler equations, i.e. we consider $K$ as in \eqref{eq:K} and $\Lambda$ as in \eqref{eq:wavecone-det}. 

\subsection{Definition of $\sU$} \label{subsec:convint-U-defn}

We define $\sU$ as in \eqref{eq:U}. Note, that this definition agrees with the heuristical description of $\sU$ given in Subsection~\ref{subsec:convint-prel-U}. Due to Proposition~\ref{prop:KLambda=U} we have $\sU=K^\Lambda$. In the following subsection we are going to find a better desription of $K^\Lambda$, which allows us to check quickly if a given point $(\rho,\vm,\mU)\in \R\times \R^n \times \szn$ lies in $K^\Lambda$ or not.

\subsection{Computation of $\sU$} 
\label{subsec:convint-U-comp}

First of all note that it is much easier to deduce facts for convex hulls rather than for $\Lambda$-convex hulls because for the former one can use e.g. Minkowski's Theorem (Proposition~\ref{prop:minkowski}) or describe them via a convex mapping. So at first we should check whether $\Lambda$ is complete with respect to $K$ since if this was the case, we could deduce that $K^\Lambda = K^\co$ by Corollary~\ref{cor:complete-wc}. Note, that as long as the incompressible Euler equations are considered, the corresponding wave cone is complete. This was observed and used by \name{De~Lellis} and \name{Sz{\'e}kelyhidi} in \cite{DelSze09} and \cite{DelSze10}\footnote{The completeness of the wave cone is hidden in the proof of \cite[Lemma 4.3]{DelSze09} and \cite[Proposition 4]{DelSze10}.}. However for the compressible Euler equations the wave cone is not complete, a fact which has been observed originally by \name{Chiodaroli et al.} \cite{CFKW17} (see also \name{Gallenm\"uller}-\name{Wiedemann} \cite{GalWie21}). To be precise, the settings, i.e. $K$ and $\Lambda$, in the cited papers are slightly different from ours, but the circumstance that compressible Euler is considered is a common feature of \cite{CFKW17}, \cite{GalWie21} and this book. The following example, which is highly inspired by the example given by \name{Chiodaroli et al.} \cite[Proof of Theorem 4]{CFKW17}, shows that the wave cone in our setting is generally not complete.

\begin{ex} 
	Let $n=2$, $p(\rho)=\rho^2$, $c=\frac{3}{2}$ and furthermore 
	\begin{align*}
		\rho_1 &= 1 \ec & \rho_2 &= \frac{\sqrt{3}-1}{2} \ec \\
		\vm_1 &= \left(\begin{array}{c} 1 \\ 0 \end{array}\right) \ec & \vm_2 &= \left(\begin{array}{c} 1 \\ 0 \end{array}\right) \ec \\
		\mU_1 &= \half \left(\begin{array}{rr} 1 & 0 \\ 0 & -1 \end{array}\right) \ec & \mU_2 &= \frac{\sqrt{3}+1}{2}\left(\begin{array}{rr} 1 & 0 \\ 0 & -1 \end{array}\right) \ed
	\end{align*}
	It is easy to check that $(\rho_1,\vm_1,\mU_1),(\rho_2,\vm_2,\mU_2)\in K$. Furthermore 
	$$
	\det \left(\begin{array}{cc} \rho_2 - \rho_1 & \vm_2^\trans - \vm_1^\trans \\ \vm_2-\vm_1 & \mU_2 - \mU_1 \end{array}\right) = \det\left(\begin{array}{ccc} \frac{\sqrt{3}-3}{2} & 0 & 0 \\ 0 & \frac{\sqrt{3}}{2} & 0 \\
	0 & 0 & -\frac{\sqrt{3}}{2} \end{array}\right) \neq 0
	$$
	and hence $(\rho_2,\vm_2,\mU_2) - (\rho_1,\vm_1,\mU_1) \notin \Lambda$.
\end{ex}

However we will show that we still have $K^\Lambda = K^\co$. To this end we define
\begin{equation} \label{eq:Kstar}
K^\ast=\bigcup\limits_{\ov{\rho}\in\R^+} \big(K\cap\{\rho=\ov{\rho}\}\big)^\Lambda
\end{equation} 
as in Proposition~\ref{prop:Kstar}.

We will finally prove the following.

\begin{prop} \label{prop:compKLambda} 
	It holds that $K^\ast = K^\Lambda = K^\co$.
\end{prop} 

To prove this, we need three lemmas, the first of which (Lemma~\ref{lemma:e-convex}) is a ``compressible variant'' of \cite[Lemma 3 (i)]{DelSze10}. Lemma~\ref{lemma:U-bdd} is similar to \cite[Lemma 3 (ii) and (iii)]{DelSze10}. The third lemma (Lemma~\ref{lemma:Kast}) can be deduced form \cite[Lemma 3 (iv)]{DelSze10} since here the density $\ov{\rho}$ is constant. For completeness we redo this proof. The key to prove Lemma~\ref{lemma:Kast} is the observation that $\Lambda$ is complete with respect to $K\cap\{\rho=\ov{\rho}\}$ for each $\ov{\rho}\in \R^+$, see below for details. This is in accordance with the completeness of the wave cone in the incompressible case.

\begin{defn} 
	Define the mapping $e : \R^+ \times \R^n \times \szn \to \R$, by
	$$
	(\rho,\vm,\mU)\mapsto e(\rho,\vm,\mU):= \frac{n}{2}\lambda_{\max}\left(\frac{\vm\otimes\vm}{\rho} + p(\rho)\id - \mU\right) \ed
	$$ 
\end{defn}

\begin{lemma} \label{lemma:e-convex}
	The mapping $e : \R^+ \times \R^n \times \szn \to \R$ is continuous and convex. 
\end{lemma}

\begin{proof} 
	The continuity of $e$ simply follows from the facts, that $(\rho,\vm,\mU)\mapsto \frac{\vm\otimes\vm}{\rho}+p(\rho)\id-\mU$ as well as $\mA\mapsto \lambda_{\max}(\mA)$ are continuous. 
	
	To show the convexity, we mimick the proof of \cite[Lemma 3]{DelSze10}. The first steps are exactly the same as in \cite{DelSze10}. For completeness we redo them. First we show that 
	\begin{equation} \label{eq:temp-convint-1}
	\lambda_{\max}\left(\frac{\vm\otimes\vm}{\rho} + p(\rho)\id - \mU\right)=\max\limits_{\vy\in\R^n,|\vy|=1} \vy\cdot\left(\left(\frac{\vm\otimes \vm}{\rho} + p(\rho)\id - \mU\right)\vy\right)\ed
	\end{equation}
	Since $\frac{\vm\otimes \vm}{\rho} + p(\rho)\id - \mU$ is symmetric, it is diagonalizable with orthogonal eigenvectors. Let $\lambda_1\leq\ldots\leq\lambda_n$ be the eigenvalues and $\vb_1,\ldots,\vb_n\in\R^n$ the corresponding normed eigenvectors. Then 
	\begin{align*}
	\lambda_{\max}\left(\frac{\vm\otimes\vm}{\rho} + p(\rho)\id - \mU\right)&=\lambda_n =\lambda_n\left(\vb_n\cdot\vb_n\right)=\vb_n\cdot\left(\lambda_n\vb_n\right) \\
	&=\vb_n\cdot\left(\left(\frac{\vm\otimes \vm}{\rho} + p(\rho)\id - \mU\right)\vb_n\right) \\
	&\leq\max\limits_{\vy\in\R^n,|\vy|=1} \vy\cdot\left(\left(\frac{\vm\otimes \vm}{\rho} + p(\rho)\id - \mU\right)\vy\right)\ed
	\end{align*} 
	
	Let $\ov{\vy}\in\R^n$, $|\ov{\vy}|=1$ such that 
	$$
	\max\limits_{\vy\in\R^n,|\vy|=1} \vy\cdot\left(\left(\frac{\vm\otimes \vm}{\rho} + p(\rho)\id - \mU\right)\vy\right)=\ov{\vy}\cdot\left(\left(\frac{\vm\otimes \vm}{\rho} + p(\rho)\id - \mU\right)\ov{\vy}\right)\ed
	$$ 
	Because $\vb_1,...,\vb_n$ form a basis of $\R^n$, there exist unique coefficients $\alpha_1,...,\alpha_n\in\R$ such that $\ov{\vy}=\sum\limits_{i=1}^n \alpha_i \vb_i$, and from $|\ov{\vy}|=1$ and the fact that $\vb_1,...,\vb_n$ form an orthonormal basis we deduce $\sum\limits_{i=1}^n\alpha_i^2= 1$. So we obtain
	\begin{align*}
	\max\limits_{\vy\in\R^n,|\vy|=1} \vy\cdot\left(\left(\frac{\vm\otimes \vm}{\rho} + p(\rho)\id - \mU\right)\vy\right) &=\ov{\vy}\cdot\left(\left(\frac{\vm\otimes \vm}{r} + p(\rho)\id - \mU\right)\ov{\vy}\right) \\
	&= \sum_{i,j=1}^n \alpha_i\alpha_j\,\vb_i\cdot\left(\left(\frac{\vm\otimes \vm}{\rho} + p(\rho)\id - \mU\right)\vb_j\right)\\
	&= \sum_{i=1}^n \alpha_i^2\lambda_i \leq \lambda_n\ed
	\end{align*} 
	Hence \eqref{eq:temp-convint-1} is proven.
	
	An easy calculation yields 
	\begin{align*}
	\vy\cdot\left(\left(\frac{\vm\otimes \vm}{\rho} + p(\rho)\id - \mU\right)\vy\right) &=\frac{1}{\rho}\vy\cdot\left(\vm\vm^\trans\vy\right) + p(\rho) - \vy\cdot\left(\mU\vy\right) \\
	&= \frac{\left(\vy\cdot\vm\right)^2}{\rho} +p(\rho) - \vy\cdot\left(\mU\vy\right)
	\end{align*} 
	for all $\vy\in \R^n$ with $|\vy|=1$.
	
	Hence it suffices to show that 
	$$
	(\rho,\vm,\mU)\mapsto \max\limits_{\vy\in\R^n,|\vy|=1} \left(\frac{\left(\vy\cdot\vm\right)^2}{\rho}+p(\rho) - \vy\cdot\left(\mU\vy\right)\right)
	$$ 
	is convex. From here on the proof slightly differs from the one of \cite[Lemma 3]{DelSze10}.
	
	Let $(\rho_1,\vm_1,\mU_1),(\rho_2,\vm_2,\mU_2)\in \R^+\times \R^n\times \szn$ and $\tau\in [0,1]$. Furthermore let $\ov{\vy}\in\R^n$, $|\ov{\vy}|=1$ such that 
	\begin{equation} \label{eq:p31}
	\begin{split}
		&\max\limits_{\vy\in\R^n,|\vy|=1} \left(\frac{\big(\vy\cdot\left(\tau\vm_1+(1-\tau)\vm_2\right)\big)^2}{\tau\rho_1+(1-\tau)\rho_2} +p(\tau\rho_1+(1-\tau)\rho_2) - \vy\cdot\big(\left(\tau\mU_1+(1-\tau)\mU_2\right)\vy\big)\right) \\
		&= \frac{\big(\ov{\vy}\cdot\left(\tau\vm_1+(1-\tau)\vm_2\right)\big)^2}{\tau\rho_1+(1-\tau)\rho_2} +p(\tau\rho_1+(1-\tau)\rho_2) - \ov{\vy}\cdot\big(\left(\tau\mU_1+(1-\tau)\mU_2\right)\ov{\vy}\big)\ed
	\end{split}
	\end{equation}
	We consider each summand in \eqref{eq:p31} seperately and obtain
	$$
	p(\tau\rho_1+(1-\tau)\rho_2)  \leq \tau p(\rho_1) + (1-\tau) p(\rho_2)
	$$
	since $\rho\mapsto p(\rho)$ is convex. Furthermore 
	$$
	- \ov{\vy}\cdot\big(\left(\tau\mU_1+(1-\tau)\mU_2\right)\ov{\vy}\big) = - \tau\ov{\vy}\cdot\left(\mU_1\ov{\vy}\right) - (1-\tau)\ov{\vy}\cdot\left(\mU_2\ov{\vy}\right)\ed
	$$
	What remains is to look at the first summand in \eqref{eq:p31}. To handle this summand, we apply Lemma~\ref{lemma:convexity-kin-en} leading to
	$$
	\frac{\big(\tau \ov{\vy}\cdot\vm_1 + (1-\tau) \ov{\vy}\cdot\vm_2\big)^2}{\tau \rho_1 + (1-\tau) \rho_2}  \leq \tau \frac{\left(\ov{\vy}\cdot\vm_1\right)^2}{\rho_1} + (1-\tau) \frac{\left(\ov{\vy}\cdot\vm_2\right)^2}{\rho_2}\ed
	$$
	
	All in all we have
	\begin{equation*}
	\begin{split}
	&\max\limits_{\vy\in\R^n,|\vy|=1} \left(\frac{\big(\vy\cdot\left(\tau\vm_1+(1-\tau)\vm_2\right)\big)^2}{\tau\rho_1+(1-\tau)\rho_2} +p(\tau\rho_1+(1-\tau)\rho_2) - \vy\cdot\big(\left(\tau\mU_1+(1-\tau)\mU_2\right)\vy\big)\right) \\
	&\leq \tau\left(\frac{\left(\ov{\vy}\cdot\vm_1\right)^2}{\rho_1} +p(\rho_1) - \ov{\vy}\cdot\left(\mU_1\ov{\vy}\right)\right) + (1-\tau) \left(\frac{\left(\ov{\vy}\cdot\vm_2\right)^2}{\rho_2} +p(\rho_2) - \ov{\vy}\cdot\left(\mU_2\ov{\vy}\right)\right) \\
	&\leq \tau\max\limits_{\vy\in\R^n,|\vy|=1}\left(\frac{\left(\vy\cdot\vm_1\right)^2}{\rho_1} +p(\rho_1) - \vy\cdot\left(\mU_1\vy\right)\right) \\
	& \qquad + (1-\tau) \max\limits_{\vy\in\R^n,|\vy|=1}\left(\frac{\left(\vy\cdot\vm_2\right)^2}{\rho_2} +p(\rho_2) - \vy\cdot\left(\mU_2\vy\right)\right)\ed
	\end{split}
	\end{equation*} 
	Multiplying by $\frac{n}{2}$ yields the claim
	$$
	e \big( \tau (\rho_1,\vm_1,\mU_1) + (1-\tau) (\rho_2,\vm_2,\mU_2) \big) \leq \tau e(\rho_1,\vm_1,\mU_1) + (1-\tau) e(\rho_2,\vm_2,\mU_2)\ed
	$$
\end{proof}

\begin{lemma} \label{lemma:U-bdd} 
	There exists $M>0$ with the following property: If $(\rho,\vm,\mU)\in \R^+\times \R^n \times \szn$ with $e(\rho,\vm,\mU)\leq c$ then
	$$
		\max\left\{ |\rho|,|\vm|,\|\mU\|\right\} \leq M \ed
	$$
\end{lemma}

\begin{proof}
	Let $(\rho,\vm,\mU)\in \R^+\times \R^n \times \szn$ with $e(\rho,\vm,\mU)\leq c$ be arbitrary. First of all, notice that 
	$$
		\frac{|\vm|^2}{\rho} + np(\rho) = \tr\left(\frac{\vm\otimes\vm}{\rho} + p(\rho)\id - \mU\right) \leq n\lambda_{\max}\left(\frac{\vm\otimes\vm}{\rho} + p(\rho)\id - \mU\right)\leq 2c\ed
	$$
	Since
	$$
		n p(\rho) \leq \frac{|\vm|^2}{\rho} + np(\rho) \qquad \text{ and }\qquad 	\frac{|\vm|^2}{\rho} \leq \frac{|\vm|^2}{\rho} + np(\rho) \ec 
	$$
	we deduce $p(\rho)\leq \frac{2c}{n}$ and $|\vm|\leq \sqrt{2\rho c}$. These inequalities give the desired bounds on $\rho$ and $\vm$ respectively. 
		
	To show the bound on $\mU$, we proceed as in the proof of \cite[Lemma 3 (iii)]{DelSze10}. Note that $\|\mU\|=\max\{|\lambda_{\min}(\mU)|,|\lambda_{\max}(\mU)|\}$. First we show that $|\lambda_{\min}(\mU)|$ is bounded. Since $\tr(\mU)=0$, we have $\lambda_{\min}(\mU)\leq 0$ and therefore
	\begin{align*}
	|\lambda_{\min}(\mU)| &= -\lambda_{\min}(\mU) = - \min\limits_{\vy\in\R^n,|\vy|=1}\vy\cdot\left(\mU\vy\right) \leq \max\limits_{\vy\in\R^n,|\vy|=1} \left(\frac{1}{\rho}\left(\vy\cdot\vm\right)^2 +p(\ov{\rho}) - \vy\cdot\left(\mU\vy\right)\right) \\
	& = \lambda_{\max}\left(\frac{\vm\otimes\vm}{\ov{\rho}} + p(\rho)\id - \mU\right)\leq \frac{2c}{n}\ed
	\end{align*} 
	The fact that $|\lambda_{\max}(\mU)|$ is bounded, too, follows because $\mU$ is traceless. Indeed
	\begin{align*} 
	|\lambda_{\max}(\mU)| &\leq \sum\limits_{\text{pos. eigenvalues}} |\lambda_i| = \sum\limits_{\text{neg. eigenvalues}} |\lambda_i| \leq (n-1) |\lambda_{\min}(\mU)| \leq \frac{2c (n-1)}{n}\ed
	\end{align*}
\end{proof}

\begin{lemma} \label{lemma:Kast}
	It holds that 	
	$$
	K^\ast = \left\{(\rho,\vm,\mU)\in\R^+\times\R^n\times\szn\,\Big|\,e(\rho,\vm,\mU)\leq c\right\}\ed
	$$
\end{lemma} 

\begin{proof}
	Let us fix $\ov{\rho}\in \R^+$. We want to show that 
	\begin{equation} \label{eq:p33}
		\big(K\cap\{\rho=\ov{\rho}\}\big)^\Lambda = \left\{(\rho,\vm,\mU)\in\R^+\times\R^n\times\szn\,\Big|\,e(\rho,\vm,\mU)\leq c\text{ and }\rho=\ov{\rho}\right\}\ec
	\end{equation}
	which proves the lemma.
	
	The proof of \eqref{eq:p33} is the same as the proof of \cite[Lemma 3]{DelSze10}. For completeness we recall this proof and furthermore we present more details. 
	
	First of all we show that $\Lambda$ is complete with respect to $K\cap\{\rho=\ov{\rho}\}$. To this end let $(\rho_1,\vm_1,\mU_1),(\rho_2,\vm_2,\mU_2)\in K\cap\{\rho=\ov{\rho}\}$ be arbitrary. To prove that 
	$$
		(\rho_1,\vm_1,\mU_1)-(\rho_2,\vm_2,\mU_2)\in \Lambda \ec
	$$
	we must look at 
	\begin{align*}
	\left(\begin{array}{cc}
	\rho_1-\rho_2 & \vm_1^\trans - \vm_2^\trans \\
	\vm_1 - \vm_2 & \mU_1-\mU_2
	\end{array}\right) &= \left(\begin{array}{cc}
	0 & \vm_1^\trans - \vm_2^\trans \\
	\vm_1 - \vm_2 & \frac{1}{\ov{\rho}}\big(\vm_1\otimes \vm_1-\vm_2\otimes \vm_2\big)
	\end{array}\right) \\ &= \left(\begin{array}{cc}
	0 & \vm_1^\trans - \vm_2^\trans \\
	\vm_1 - \vm_2 & \frac{1}{\ov{\rho}}\vm_1 \vm_1^\trans-\frac{1}{\ov{\rho}}\vm_2 \vm_2^\trans
	\end{array}\right)\ec
	\end{align*}
	since $\rho_1=\ov{\rho} = \rho_2$ and $(\rho_1,\vm_1,\mU_1),(\rho_2,\vm_2,\mU_2)\in K$. To compute the determinant, observe that 
	\begin{align*}
	&\left(\begin{array}{cc}
	1 & \vz^\trans \\
	-\frac{1}{\ov{\rho}}\vm_1& \id
	\end{array}\right)\cdot \left(\begin{array}{cc}
	0 & \vm_1^\trans - \vm_2^\trans \\
	\vm_1 - \vm_2 & \frac{1}{\ov{\rho}}\vm_1 \vm_1^\trans-\frac{1}{\ov{\rho}}\vm_2 \vm_2^\trans
	\end{array}\right) \cdot \left(\begin{array}{cc}
	1 & -\frac{1}{\ov{\rho}}\vm_2^\trans \\
	\vz & \id
	\end{array}\right) \\
	&= \left(\begin{array}{cc}
	0 & \vm_1^\trans - \vm_2^\trans \\
	\vm_1 - \vm_2 & \mathbb{O}
	\end{array}\right)\ed
	\end{align*}	
	Obviously
	\begin{align*}
		\det\left(\begin{array}{cc}
			1 & \vz^\trans \\
			-\frac{1}{\ov{\rho}}\vm_1& \id
		\end{array}\right) &=1 \ec\\
		\det\left(\begin{array}{cc}
			1 & -\frac{1}{\ov{\rho}}\vm_2^\trans \\
			\vz & \id
		\end{array}\right) &= 1\ec\\
		\det\left(\begin{array}{cc}
			0 & \vm_1^\trans - \vm_2^\trans \\
			\vm_1 - \vm_2 & \mathbb{O}
		\end{array}\right) &=0\ec 
	\end{align*}	
	and hence
	$$
	\det \left(\begin{array}{cc}
		0 & \vm_1^\trans - \vm_2^\trans \\
		\vm_1 - \vm_2 & \frac{1}{\ov{\rho}}\vm_1 \vm_1^\trans-\frac{1}{\ov{\rho}}\vm_2 \vm_2^\trans
	\end{array}\right) = 0\ed
	$$
	Therefore $(\rho_1,\vm_1,\mU_1)-(\rho_2,\vm_2,\mU_2)\in \Lambda$.
	
	Now Corollary~\ref{cor:complete-wc} yields that $\big(K\cap\{\rho=\ov{\rho}\}\big)^\Lambda=\big(K\cap\{\rho=\ov{\rho}\}\big)^\co$, which means that we can use Minkowski's Theorem (Proposition~\ref{prop:minkowski}) in order to find $\big(K\cap\{\rho=\ov{\rho}\}\big)^\Lambda$.
	
	Let us now check the assumptions of Proposition~\ref{prop:minkowski}. We first show that the set 
	$$
	S:= \left\{(\rho,\vm,\mU)\in\R^+\times\R^n\times\szn\,\Big|\,e(\rho,\vm,\mU)\leq c\text{ and }\rho=\ov{\rho}\right\} 
	$$ 
	is compact and convex. To prove compactness, it is enough to show that $S$ is closed and bounded. As $S$ can be written as intersection of the pre-image of the closed set $(-\infty,c]$ under $e$, where $e$ is continuous according to Lemma~\ref{lemma:e-convex}, and the closed set $\{\rho=\ov{\rho}\}$, $S$ is closed\footnote{To be precise the pre-image of $(-\infty,c]$ under $e$ is only closed with respect to the subspace topology on $\R^+\times\R^n\times \szn$. It may (and will) not be closed in $\R\times\R^n\times \szn$. However $S$ is closed even in $\R\times\R^n\times \szn$ since $S$ is the intersection of this pre-image and the set $\{\rho=\ov{\rho}\}$ with $\ov{\rho}>0$.}. Moreover, the convexity and boundedness of $S$ simply follow from Lemmas \ref{lemma:e-convex} and \ref{lemma:U-bdd}, respectively. 
	
	To show that $K\cap\{\rho=\ov{\rho}\}\subset S$, let $(\rho,\vm,\mU)\in K\cap\{\rho=\ov{\rho}\}$. Then
	$$
		e(\rho,\vm,\mU) = \frac{n}{2} \lambda_{\max}\left(\frac{\vm\otimes\vm}{\rho} + p(\rho)\id - \mU\right) = \frac{n}{2} \lambda_{\max}\left(\frac{2c}{n} \id \right) = c
	$$
	and $\rho=\ov{\rho}$, which proves that $(\rho,\vm,\mU)\in S$.
	
	Hence the assumptions of Minkowski's Theorem (Propostion \ref{prop:minkowski}) hold. What remains is to prove that the extreme points of $S$ lie in $K\cap\{\rho=\ov{\rho}\}$. In order to do this, let $(\rho,\vm,\mU)\in S$ but assume that $(\rho,\vm,\mU)\notin K\cap\{\rho=\ov{\rho}\}$. It suffices to show that this implies that $(\rho,\vm,\mU)$ is not an extreme point of $S$. 
	
	From $(\rho,\vm,\mU)\in S$ we obtain $\rho=\ov{\rho}$ and 
	\begin{equation} \label{eq:temp1}
		\lambda_{\max}\left(\frac{\vm\otimes\vm}{\ov{\rho}} + p(\ov{\rho})\id - \mU\right)\leq \frac{2c}{n}\ed
	\end{equation}
	Since the matrix 
	$$
	\frac{\vm\otimes\vm}{\ov{\rho}} + p(\ov{\rho})\id - \mU
	$$
	is symmetric, there exists an orthogonal matrix $\mT\in \orth{n}$ such that 
	\begin{equation} \label{eq:temp2}
		\frac{\vm\otimes\vm}{\ov{\rho}} + p(\ov{\rho})\id - \mU = \mT \left(\begin{array}{ccc}
		\lambda_1 & & \\ & \ddots & \\ & & \lambda_n
		\end{array}\right)	\mT^{-1}\ed
	\end{equation}
	We may assume without loss of generality that the eigenvalues are ordered $\lambda_1\leq...\leq \lambda_n$. We denote the normed eigenvector, which corresponds to the $i$-th eigenvalue $\lambda_i$, by $\vb_i$. Then the vectors $\vb_1,...,\vb_n$ form an orthonormal basis of $\R^n$ and $\mT=(\vb_1 \,\cdots\,\vb_n)$.
	
	From \eqref{eq:temp1} we deduce that $\lambda_i\leq \frac{2c}{n}$ for all $i=1,...,n$. Assume that $\lambda_1=\frac{2c}{n}$. Then we have $\frac{2c}{n}\leq\lambda_1\leq...\leq \lambda_n\leq \frac{2c}{n}$, i.e. $\lambda_i= \frac{2c}{n}$ for all $i=1,...,n$. Hence with \eqref{eq:temp2} we get 
	$$
		\frac{\vm\otimes\vm}{\ov{\rho}} + p(\ov{\rho})\id - \mU = \mT \frac{2c}{n} \id \mT^{-1} = \frac{2c}{n}\id\ec
	$$
	which means that $(\rho,\vm,\mU)\in K\cap\{\rho=\ov{\rho}\}$, a contradiction. Therefore $\lambda_1<\frac{2c}{n}$. 
	
	Because $\vb_1,...,\vb_n$ form a basis of $\R^n$, there exist unique coefficients $\alpha_1,...,\alpha_n\in\R$ such that $\vm = \sum\limits_{i=1}^n \alpha_i\vb_i$. Let us define
	$$
		(\widehat{\rho},\widehat{\vm},\widehat{\mU}) := \left(0, \vb_1 , \frac{\vm\otimes\vb_1 + \vb_1\otimes\vm - 2\alpha_1 \vb_1\otimes\vb_1}{\rho}\right)\ed
	$$
	Obviously $\widehat{\mU}$ is symmetric and furthermore
	$$
		\tr \widehat{\mU} = \frac{2}{\rho}\left(\vm\cdot \vb_1 - \alpha_1 |\vb_1|^2\right) = 0
	$$
	due to the facts that $\vm\cdot \vb_1 = \alpha_1$ and $|\vb_1|=1$. In other words $\widehat{\mU}\in\szn$.
	
	For $\tau\in\R$ we compute 
	\begin{align*}
	&\mT^{-1} \left( \frac{(\vm + \tau\widehat{\vm})\otimes(\vm + \tau\widehat{\vm})}{\rho + \tau\widehat{\rho}} + p(\rho + \tau\widehat{\rho})\id  - (\mU + \tau\widehat{\mU})\right) \mT \\
	&=\mT^{-1} \left( \frac{\vm\otimes\vm}{\rho} + p(\rho)\id  - \mU\right) \mT + \tau \mT^{-1} \left(\frac{\vm\otimes\vb_1 + \vb_1\otimes\vm}{\rho} - \widehat{\mU}\right) \mT + \tau^2 \mT^{-1} \frac{\vb_1\otimes\vb_1}{\rho}\mT \\
	&=\left(\begin{array}{ccc}
	\lambda_1 & & \\ & \ddots & \\ & & \lambda_n
	\end{array}\right) + (2\alpha_1\tau + \tau^2) \mT^{-1} \frac{\vb_1\otimes\vb_1}{\rho}\mT\ed
	\end{align*}
	Note that 
	$$
	\mT^{-1}\, \vb_1\otimes\vb_1 \,\mT = \left(\begin{array}{cccc}
	1 & & & \\ & 0 & & \\ & & \ddots & \\ & & & 0 
	\end{array}\right) \ec
	$$
	due to the facts that the columns of $\mT$ are the vectors $\vb_1,...,\vb_n$ and these vectors form an orthonormal basis. Hence 
	\begin{align*}
		\lambda_{\max}\left( \frac{(\vm + \tau\widehat{\vm})\otimes(\vm + \tau\widehat{\vm})}{\rho + \tau\widehat{\rho}} + p(\rho + \tau\widehat{\rho})\id  - (\mU + \tau\widehat{\mU})\right) &= \max\left\{\lambda_1 + \frac{2\alpha_1\tau + \tau^2}{\rho},\lambda_n\right\} \\
		&\leq \frac{2c}{n}\ec
	\end{align*}
	as long as $|\tau|$ is sufficiently small, because $\lambda_1<\frac{2c}{n}$ and $\lambda_n\leq \frac{2c}{n}$. Furthermore we have $\rho + \tau\widehat{\rho} = \rho = \ov{\rho}$ for all $\tau\in\R$ since $\widehat{\rho}=0$. In other words $(\rho + \tau\widehat{\rho},\vm + \tau\widehat{\vm},\mU + \tau\widehat{\mU})\in S$ for sufficiently small $|\tau|$. Since $(\rho,\vm,\mU)$ is a convex combination of $(\rho \pm \tau\widehat{\rho},\vm \pm \tau\widehat{\vm},\mU \pm \tau\widehat{\mU})$, $(\rho,\vm,\mU)$ is not an extreme point of $S$. This finishes the proof.
\end{proof}

\begin{proof}[Proof of Proposition~\ref{prop:compKLambda}]
	Due to Propositions \ref{prop:Kstar}~\ref{item:Kstar.a} and \ref{prop:generalfacts}~\ref{item:generalfacts.b} we have $K^\ast \subset K^\Lambda \subset K^\co$. Lemmas \ref{lemma:e-convex} and \ref{lemma:Kast} show that $K^\ast$ is convex. Since $K\subset K^\ast$, we have $K^\co\subset K^\ast$. This yields the claim. 
\end{proof}

\section{Operators} \label{sec:convint-operators} 

As mentioned in Subsection~\ref{subsec:convint-prel-U} we have to find suitable \emph{localized} plane wave oscillations. The operators considered in this section help to localize or cut-off arbitrary plane waves.

\subsection{Statement of the Operators} 
\label{subsec:convint-op-statement}

The following proposition claims that desired operators exist. Propostion \ref{prop:operators} is proven in Subsection~\ref{subsec:convint-op-proof}. Part \ref{item:operators.c} is only needed in Section~\ref{sec:convint-nodens}, where solutions with fixed density are constructed.

\begin{prop} \label{prop:operators}
	Let $n\in\{2,3\}$, $(\ov{\rho},\ov{\vm},\ov{\mU})\in\R\times\R^n\times\szn$ and assume that there exists $\veta\in\R^{1+n}\smallsetminus\{\vz\}$ such that
	\begin{equation} \label{eq:op-kernel}
		\left(\begin{array}{cc}
			\ov{\rho} & \ov{\vm}^\trans \\
			\ov{\vm} & \ov{\mU}  
		\end{array}\right) \veta = \vz \ed
	\end{equation}
	Then there exist third order homogeneous\footnote{A differential operator is called \emph{third order homogeneous} if it can be written as a sum of derivatives, \emph{all} of which are of third order.} differential operators 
	\begin{align*}
		\opL_\rho &: C^\infty(\R^{1+n}) \to C^\infty(\R^{1+n}) \ec\\
		\opL_\vm &: C^\infty(\R^{1+n}) \to C^\infty(\R^{1+n};\R^n) \ec\\
		\opL_\mU &: C^\infty(\R^{1+n}) \to C^\infty(\R^{1+n};\szn)\ec
	\end{align*}
	with the following properties:
	\begin{enumerate}
		\item \label{item:operators.a} For any function $g\in C^\infty(\R^{1+n})$ 
		the PDEs
		\begin{align}
			\partial_t \opL_\rho[g] + \Div \opL_\vm[g] &= 0 \ec \label{eq:op-p1a}\\
			\partial_t \opL_\vm[g] + \Div \opL_\mU[g] &= \vz \label{eq:op-p1b}
		\end{align}
		hold. 
		\item \label{item:operators.b} If we set $g(t,\vx):= h\big((t,\vx)\cdot \veta\big)$ with an arbitrary function $h\in C^\infty(\R)$, we obtain
		\begin{align}
			\opL_\rho [g](t,\vx) &= \ov{\rho} \,h'''\big((t,\vx)\cdot \veta\big) \ec \label{eq:op-p2a}\\
			\opL_\vm [g](t,\vx) &= \ov{\vm} \,h'''\big((t,\vx)\cdot \veta\big) \ec \label{eq:op-p2b}\\
			\opL_\mU [g](t,\vx) &= \ov{\mU} \,h'''\big((t,\vx)\cdot \veta\big) \ed \label{eq:op-p2c}
		\end{align}
		\item \label{item:operators.c} If $\ov{\rho}=0$ and $\veta$ is not parallel to $\ve_t$, then $\opL_\rho \equiv 0$.
	\end{enumerate}
\end{prop}

\subsection{Lemmas for the Proof of Proposition~\ref{prop:operators}} \label{subsec:convint-op-lemmas}

Let us first prove Proposition~\ref{prop:operators} in some special cases. In Subsection~\ref{subsec:convint-op-proof} we reduce the general case to one of the special cases studied in the following lemmas.

\begin{lemma} \label{lemma:op-eta=e1} 
	Let the assumptions of Proposition~\ref{prop:operators} hold with $\veta= (a,\vz)^\trans \in \R^{1+n}$, where $a\neq 0$. Then the operators $\opL_\rho$, $\opL_\vm$ and $\opL_\mU$ defined by 
	\begin{align*}
		\opL_\rho[g] &:= \frac{1}{a^3} \sum_{i,j=1}^n \ov{U}_{ij} \,\parthree{t}{i}{j} g \ec\\
		\opL_{m_k}[g] &:= - \frac{1}{a^3} \sum_{i=1}^n \ov{U}_{ik} \,\parthree{t}{t}{i} g \qquad \text{ for }k\in\{1,...,n\} \ec\\
		\opL_\mU[g] &:= \frac{\ov{\mU}}{a^3}\parthree{t}{t}{t} g
	\end{align*}
	for all $g\in C^\infty(\R^{1+n})$, satisfy the properties \ref{item:operators.a}, \ref{item:operators.b} of Proposition~\ref{prop:operators}.
\end{lemma}

\begin{proof} 
\begin{enumerate}
	\item It is straightforward to check that \eqref{eq:op-p1a} and \eqref{eq:op-p1b} hold:
	\begin{align*}
		\partial_t \opL_\rho[g] + \Div \opL_\vm[g] &= \frac{1}{a^3} \sum_{i,j=1}^n \ov{U}_{ij} \parfour{t}{t}{i}{j} g  - \frac{1}{a^3} \sum_{i,j=1}^n \ov{U}_{ij} \parfour{t}{t}{i}{j} g = 0 \es \\
		\partial_t \opL_{m_k}[g] + \sum_{i=1}^n \partial_i \opL_{U_{ki}}[g] &= - \frac{1}{a^3} \sum_{i=1}^n \ov{U}_{ik} \parfour{t}{t}{t}{i} + \frac{1}{a^3} \sum_{i=1}^n \ov{U}_{ki} \parfour{t}{t}{t}{i} = 0\ed
	\end{align*}
	
	\item Now set $g(t,\vx):= h\big( (t,\vx)\cdot \veta \big)= h(at)$ with an arbitrary function $h\in C^\infty(\R)$. In particular we have $\partial_i g = 0$ for all $i=1,...,n$. This yields
	\begin{align}
		\opL_\rho[g](t,\vx) &= 0 \ec \label{eq:op-1temp}\\
		\opL_\vm[g](t,\vx) &= \vz \ec \label{eq:op-2temp}\\
		\opL_\mU[g](t,\vx) &= \ov{\mU}\,h'''\big( (t,\vx)\cdot \veta \big) \ed \label{eq:op-3temp}
	\end{align}
	Note that we obtain $\ov{\rho} = 0$ and $\ov{\vm} = \vz$ from \eqref{eq:op-kernel} and the fact that $\veta= (a,\vz)^\trans$. Thus \eqref{eq:op-1temp} - \eqref{eq:op-3temp} imply \eqref{eq:op-p2a} - \eqref{eq:op-p2c}.
\end{enumerate}
\end{proof}

\begin{rem}
	Notice that property \ref{item:operators.c} of Proposition~\ref{prop:operators} is also true, since $\veta$ is parallel to $\ve_t$ in Lemma~\ref{lemma:op-eta=e1} and hence property \ref{item:operators.c} of Proposition~\ref{prop:operators} is void. 
\end{rem}

\begin{lemma} \label{lemma:op-2D}
	Let $n=2$ and suppose that the assumptions of Proposition~\ref{prop:operators} hold with $\veta= (a,b,0)^\trans$, where $b\neq 0$. Then the operators $\opL_\rho$, $\opL_\vm$ and $\opL_\mU$ defined by 
	\begin{align}
		\opL_\rho[g] &:= \frac{\ov{\rho}}{b^3} \Big( \parthree{1}{1}{1} g + \parthree{1}{2}{2} g \Big) \ec \label{eq:op-defn-2D-rho} \\
		\opL_{m_1}[g] &:= \frac{\ov{\rho}}{b^3} \Big( -\parthree{t}{1}{1} g - \parthree{t}{2}{2} g \Big) + \frac{\ov{m}_2}{b^3} \Big( -\parthree{1}{1}{2} g - \parthree{2}{2}{2} g \Big) \ec\\
		\opL_{m_2}[g] &:= \frac{\ov{m}_2}{b^3} \Big( \parthree{1}{1}{1} g + \parthree{1}{2}{2} g \Big) \ec\\
		\opL_{U_{11}}[g] &:= \frac{\ov{\rho}}{b^3} \Big( \parthree{t}{t}{1} g \Big) + \frac{\ov{m}_2}{b^3} \Big( 2 \parthree{t}{1}{2} g \Big) \ec\\
		\opL_{U_{12}}[g] &:= \frac{\ov{\rho}}{b^3} \Big( \parthree{t}{t}{2} g \Big) + \frac{\ov{m}_2}{b^3} \Big( -\parthree{t}{1}{1} g + \parthree{t}{2}{2} g \Big) 
	\end{align}
	for all $g\in C^\infty(\R^{1+2})$, satisfy the properties \ref{item:operators.a} - \ref{item:operators.c} of Proposition~\ref{prop:operators}.
\end{lemma}

\begin{proof} 
\begin{enumerate}	
	\item It is straightforward to check that \eqref{eq:op-p1a} and \eqref{eq:op-p1b} hold:
	\begin{align*}
		\partial_t \opL_\rho[g] + \partial_1 \opL_{m_1}[g] + \partial_2 \opL_{m_2}[g] &= \frac{\ov{\rho}}{b^3} \Big( \parfour{t}{1}{1}{1} g + \parfour{t}{1}{2}{2} g  -\parfour{t}{1}{1}{1} g - \parfour{t}{1}{2}{2} g \Big) \\
		&\qquad + \frac{\ov{m}_2}{b^3} \Big( -\parfour{1}{1}{1}{2} g - \parfour{1}{2}{2}{2} g + \parfour{1}{1}{1}{2} g + \parfour{1}{2}{2}{2} g\Big) \\
		& = 0 \es \\
		\partial_t \opL_{m_1}[g] + \partial_1 \opL_{U_{11}}[g] + \partial_2 \opL_{U_{12}}[g] &= \frac{\ov{\rho}}{b^3} \Big( -\parfour{t}{t}{1}{1} g - \parfour{t}{t}{2}{2} g + \parfour{t}{t}{1}{1} g + \parfour{t}{t}{2}{2} g \Big) \\
		&\qquad + \frac{\ov{m}_2}{b^3} \Big( -\parfour{t}{1}{1}{2} g - \parfour{t}{2}{2}{2} g + 2 \parfour{t}{1}{1}{2} g -\parfour{t}{1}{1}{2} g + \parfour{t}{2}{2}{2} g\Big) \\
		& = 0\es \\
		\partial_t \opL_{m_2}[g] + \partial_1 \opL_{U_{12}}[g] - \partial_2 \opL_{U_{11}}[g] &= \frac{\ov{\rho}}{b^3} \Big( \parfour{t}{t}{1}{2} g - \parfour{t}{t}{1}{2} g \Big) \\
		&\qquad + \frac{\ov{m}_2}{b^3} \Big( \parfour{t}{1}{1}{1} g + \parfour{t}{1}{2}{2} g -\parfour{t}{1}{1}{1} g + \parfour{t}{1}{2}{2} g - 2 \parfour{t}{1}{2}{2} g \Big) \\
		& = 0 \ed
	\end{align*}
	
	\item Now set $g(t,\vx):= h\big( (t,\vx)\cdot \veta \big)= h(at+bx)$ with an arbitrary function $h\in C^\infty(\R)$. In particular we have $\partial_2 g = 0$. This yields
	\begin{align}
		\opL_\rho[g](t,\vx) &= \frac{\ov{\rho}}{b^3}  b^3 \,h'''\big( (t,\vx)\cdot \veta \big) = \ov{\rho}\,h'''\big( (t,\vx)\cdot \veta \big) \ec \label{eq:op-1temp2} \\
		\opL_{m_1}[g](t,\vx) &= - \frac{\ov{\rho}}{b^3} ab^2 \,h'''\big( (t,\vx)\cdot \veta \big) = - \frac{a\ov{\rho}}{b} \,h'''\big( (t,\vx)\cdot \veta \big) \ec \label{eq:op-2temp2} \\
		\opL_{m_2}[g](t,\vx) &= \frac{\ov{m}_2}{b^3} b^3 \,h'''\big( (t,\vx)\cdot \veta \big) = \ov{m}_2 \,h'''\big( (t,\vx)\cdot \veta \big) \ec \label{eq:op-3temp2}\\
		\opL_{U_{11}}[g](t,\vx) &= \frac{\ov{\rho}}{b^3} a^2 b \,h'''\big( (t,\vx)\cdot \veta \big) = \frac{a^2 \ov{\rho}}{b^2} \,h'''\big( (t,\vx)\cdot \veta \big) \ec \label{eq:op-4temp2}\\
		\opL_{U_{12}}[g](t,\vx) &= -\frac{\ov{m}_2}{b^3} ab^2 \,h'''\big( (t,\vx)\cdot \veta \big) = -\frac{a\ov{m}_2}{b}\,h'''\big( (t,\vx)\cdot \veta \big) \ed \label{eq:op-5temp2}
	\end{align}
	Note that we obtain 
	\begin{align*}
		a\ov{\rho} + b \ov{m}_1 &= 0 \ec \\
		a\ov{m}_1 + b \ov{U}_{11} &= 0 \ec \\
		a\ov{m}_2 + b \ov{U}_{12} &= 0 
	\end{align*}
	from \eqref{eq:op-kernel} and the fact that $\veta=(a,b,0)^\trans$. Thus \eqref{eq:op-1temp2} - \eqref{eq:op-5temp2} imply \eqref{eq:op-p2a} - \eqref{eq:op-p2c}.
	
	\item If $\ov{\rho}=0$ we immediately get $\opL_\rho \equiv 0 $ from \eqref{eq:op-defn-2D-rho}.
\end{enumerate}
\end{proof}

\begin{lemma} \label{lemma:op-3D}
	Let $n=3$ and suppose that the assumptions of Proposition~\ref{prop:operators} hold with $\veta= (a,b,0,0)^\trans$, where $b\neq 0$. Then the operators $\opL_\rho$, $\opL_\vm$ and $\opL_\mU$ defined by 
	\begin{align}
		\opL_\rho[g] &:= \frac{\ov{\rho}}{b^3} \Big( \parthree{1}{1}{1} g + \parthree{1}{2}{2} g \Big) \ec \label{eq:op-defn-3D-rho} \\
		\opL_{m_1}[g] &:= \frac{\ov{\rho}}{b^3} \Big( -\parthree{t}{1}{1} g - \parthree{t}{2}{2} g \Big) + \frac{\ov{m}_2}{b^3} \Big( -\parthree{1}{1}{2} g - \parthree{2}{2}{2} g \Big) + \frac{\ov{m}_3}{b^3} \Big( -\parthree{1}{1}{3} g - \parthree{3}{3}{3} g \Big) \ec\\
		\opL_{m_2}[g] &:= \frac{\ov{m}_2}{b^3} \Big( \parthree{1}{1}{1} g + \parthree{1}{2}{2} g \Big) \ec\\
		\opL_{m_3}[g] &:= \frac{\ov{m}_3}{b^3} \Big( \parthree{1}{1}{1} g + \parthree{1}{3}{3} g \Big) \ec \\
		\opL_{U_{11}}[g] &:= \frac{\ov{\rho}}{b^3} \Big( \parthree{t}{t}{1} g \Big) + \frac{\ov{m}_2}{b^3} \Big( 2 \parthree{t}{1}{2} g \Big) + \frac{\ov{m}_3}{b^3} \Big( 2 \parthree{t}{1}{3} g \Big) + \frac{\ov{U}_{11}+\ov{U}_{22}}{b^3} \Big( \parthree{1}{2}{2} g - \parthree{1}{3}{3} g \Big) \\
		&\qquad + \frac{\ov{U}_{23}}{b^3} \Big( 2 \parthree{1}{2}{3} g - \parthree{2}{3}{3} g - \parthree{2}{2}{2} g \Big) \ec \notag \\
		\opL_{U_{12}}[g] &:= \frac{\ov{\rho}}{b^3} \Big( \parthree{t}{t}{2} g \Big) + \frac{\ov{m}_2}{b^3} \Big( - \parthree{t}{1}{1} g + \parthree{t}{2}{2} g \Big) + \frac{\ov{U}_{11}+\ov{U}_{22}}{b^3} \Big( - \parthree{1}{1}{2} g - \parthree{2}{3}{3} g \Big) \\
		&\qquad + \frac{\ov{U}_{23}}{b^3} \Big( - \parthree{1}{1}{3} g - \parthree{3}{3}{3} g + \parthree{1}{3}{3} g + \parthree{1}{2}{2} g \Big) \ec \notag\\
		\opL_{U_{13}}[g] &:= \frac{\ov{m}_3}{b^3} \Big( - \parthree{t}{1}{1} g + \parthree{t}{3}{3} g \Big) + \frac{\ov{U}_{11}+\ov{U}_{22}}{b^3} \Big( \parthree{1}{1}{3} g + \parthree{2}{2}{3} g \Big) \\
		&\qquad + \frac{\ov{U}_{23}}{b^3} \Big( \parthree{2}{3}{3} g - \parthree{1}{1}{2} g \Big) \ec \notag \\
		\opL_{U_{22}}[g] &:= \frac{\ov{\rho}}{b^3} \Big( - \parthree{t}{t}{1} g \Big) + \frac{\ov{m}_2}{b^3} \Big( -2 \parthree{t}{1}{2} g \Big) + \frac{\ov{U}_{11}+\ov{U}_{22}}{b^3} \Big( \parthree{1}{1}{1} g + \parthree{1}{3}{3} g \Big) \\
		&\qquad + \frac{\ov{U}_{23}}{b^3} \Big( \parthree{2}{3}{3} g - \parthree{1}{1}{2} g \Big) \ec \notag\\
		\opL_{U_{23}}[g] &:= \frac{\ov{U}_{23}}{b^3} \Big( \parthree{1}{1}{1} g + \parthree{1}{3}{3} g - \parthree{2}{2}{3} g - \parthree{1}{1}{3} g \Big)
	\end{align}
	for all $g\in C^\infty(\R^{1+3})$, satisfy the properties \ref{item:operators.a} - \ref{item:operators.c} of Proposition~\ref{prop:operators}.
\end{lemma}

\begin{proof} 
\begin{enumerate}
	\item It is straightforward to check that \eqref{eq:op-p1a} and \eqref{eq:op-p1b} hold:
	\begin{align*}
		&\partial_t \opL_{\rho}[g] + \partial_1 \opL_{m_1}[g] + \partial_2 \opL_{m_2}[g] + \partial_3 \opL_{m_3}[g] \\
		&= \frac{\ov{\rho}}{b^3} \Big( \parfour{t}{1}{1}{1} g + \parfour{t}{1}{2}{2} g - \parfour{t}{1}{1}{1} g - \parfour{t}{1}{2}{2} g \Big) \\
		&\qquad + \frac{\ov{m}_2}{b^3} \Big( \parfour{1}{1}{1}{2} g + \parfour{1}{2}{2}{2} g - \parfour{1}{1}{1}{2} g - \parfour{1}{2}{2}{2} g \Big) \\
		&\qquad + \frac{\ov{m}_3}{b^3} \Big( \parfour{1}{1}{1}{3} g + \parfour{1}{3}{3}{3} g - \parfour{1}{1}{1}{3} g - \parfour{1}{3}{3}{3} g \Big)\\
		& = 0 \es
	\end{align*}
	\begin{align*}
		&\partial_t \opL_{m_1}[g] + \partial_1 \opL_{U_{11}}[g] + \partial_2 \opL_{U_{12}}[g] + \partial_3 \opL_{U_{13}}[g] \\
		&= \frac{\ov{\rho}}{b^3} \Big( -\parfour{t}{t}{1}{1} g - \parfour{t}{t}{2}{2} g + \parfour{t}{t}{1}{1} g + \parfour{t}{t}{2}{2} g \Big) \\
		&\qquad + \frac{\ov{m}_2}{b^3} \Big( -\parfour{t}{1}{1}{2} g - \parfour{t}{2}{2}{2} g + 2 \parfour{t}{1}{1}{2} g - \parfour{t}{1}{1}{2} g + \parfour{t}{2}{2}{2} g \Big) \\
		&\qquad + \frac{\ov{m}_3}{b^3} \Big( -\parfour{t}{1}{1}{3} g - \parfour{t}{3}{3}{3} g + 2 \parfour{t}{1}{1}{3} g - \parfour{t}{1}{1}{3} g + \parfour{t}{3}{3}{3} g \Big) \\
		&\qquad + \frac{\ov{U}_{11}+\ov{U}_{22}}{b^3} \Big( \parfour{1}{1}{2}{2} g - \parfour{1}{1}{3}{3} g - \parfour{1}{1}{2}{2} g - \parfour{2}{2}{3}{3} g + \parfour{1}{1}{3}{3} g + \parfour{2}{2}{3}{3} g \Big) \\
		&\qquad + \frac{\ov{U}_{23}}{b^3} \Big( 2 \parfour{1}{1}{2}{3} g - \parfour{1}{2}{3}{3} g - \parfour{1}{2}{2}{2} g - \parfour{1}{1}{2}{3} g - \parfour{2}{3}{3}{3} g + \parfour{1}{2}{3}{3} g \\
		&\qquad \qquad \qquad \qquad + \parfour{1}{2}{2}{2} g + \parfour{2}{3}{3}{3} g - \parfour{1}{1}{2}{3} g \Big) \\
		& = 0\es 
	\end{align*}
	\begin{align*} 
		&\partial_t \opL_{m_2}[g] + \partial_1 \opL_{U_{12}}[g] + \partial_2 \opL_{U_{22}}[g] + \partial_3 \opL_{U_{23}}[g] \\
		&= \frac{\ov{\rho}}{b^3} \Big( \parfour{t}{t}{1}{2} g - \parfour{t}{t}{1}{2} g \Big) \\
		&\qquad + \frac{\ov{m}_2}{b^3} \Big( \parfour{t}{1}{1}{1} g + \parfour{t}{1}{2}{2} g - \parfour{t}{1}{1}{1} g + \parfour{t}{1}{2}{2} g - 2 \parfour{t}{1}{2}{2} g \Big) \\
		&\qquad + \frac{\ov{U}_{11}+\ov{U}_{22}}{b^3} \Big( - \parfour{1}{1}{1}{2} g - \parfour{1}{2}{3}{3} g + \parfour{1}{1}{1}{2} g + \parfour{1}{2}{3}{3} g \Big) \\
		&\qquad + \frac{\ov{U}_{23}}{b^3} \Big( - \parfour{1}{1}{1}{3} g - \parfour{1}{3}{3}{3} g + \parfour{1}{1}{3}{3} g + \parfour{1}{1}{2}{2} g + \parfour{2}{2}{3}{3} g \\
		&\qquad \qquad \qquad \qquad - \parfour{1}{1}{2}{2} g + \parfour{1}{1}{1}{3} g + \parfour{1}{3}{3}{3} g - \parfour{2}{2}{3}{3} g - \parfour{1}{1}{3}{3} g \Big) \\
		& = 0 \es
	\end{align*}
	\begin{align*} 
		&\partial_t \opL_{m_3}[g] + \partial_1 \opL_{U_{13}}[g] + \partial_2 \opL_{U_{23}}[g] - \partial_3 \opL_{U_{11}}[g] - \partial_3 \opL_{U_{22}}[g] \\
		&= \frac{\ov{\rho}}{b^3} \Big( - \parfour{t}{t}{1}{3} g + \parfour{t}{t}{1}{3} g\Big) \\
		&\qquad + \frac{\ov{m}_2}{b^3} \Big( - 2 \parfour{t}{1}{2}{3} g + 2 \parfour{t}{1}{2}{3} g \Big) \\
		&\qquad + \frac{\ov{m}_3}{b^3} \Big( \parfour{t}{1}{1}{1} g + \parfour{t}{1}{3}{3} g - \parfour{t}{1}{1}{1} g + \parfour{t}{1}{3}{3} g - 2 \parfour{t}{1}{3}{3} g \Big) \\
		&\qquad + \frac{\ov{U}_{11}+\ov{U}_{22}}{b^3} \Big( \parfour{1}{1}{1}{3} g + \parfour{1}{2}{2}{3} g - \parfour{1}{2}{2}{3} g + \parfour{1}{3}{3}{3} g - \parfour{1}{1}{1}{3} g - \parfour{1}{3}{3}{3} g \Big) \\
		&\qquad + \frac{\ov{U}_{23}}{b^3} \Big( \parfour{1}{2}{3}{3} g - \parfour{1}{1}{1}{2} g + \parfour{1}{1}{1}{2} g + \parfour{1}{2}{3}{3} g - \parfour{2}{2}{2}{3} g - \parfour{1}{1}{2}{3} g \\
		&\qquad \qquad \qquad \qquad - 2 \parfour{1}{2}{3}{3} g + \parfour{2}{3}{3}{3} g + \parfour{2}{2}{2}{3} g - \parfour{2}{3}{3}{3} g + \parfour{1}{1}{2}{3} g \Big) \\
		& = 0 \ed
	\end{align*}
	
	\item Now set $g(t,\vx):= h\big( (t,\vx)\cdot \eta \big)= h(at+bx)$ with an arbitrary function $h\in C^\infty(\R)$. In particular we have $\partial_2 g = \partial_3 g= 0$. This yields
	\begin{align}
		\opL_{\rho}[g](t,\vx) &= \ov{\rho} \,h'''\big( (t,\vx)\cdot \veta \big) \ec \label{eq:op-1temp3}\\
		\opL_{m_1}[g](t,\vx) &= -\frac{a}{b}\ov{\rho} \,h'''\big( (t,\vx)\cdot \veta \big) \ec \label{eq:op-2temp3}\\
		\opL_{m_2}[g](t,\vx) &= \ov{m}_2 \,h'''\big( (t,\vx)\cdot \veta \big) \ec \label{eq:op-3temp3}\\
		\opL_{m_3}[g](t,\vx) &= \ov{m}_3 \,h'''\big( (t,\vx)\cdot \veta \big) \ec \label{eq:op-4temp3}\\
		\opL_{U_{11}}[g](t,\vx) &= \frac{a^2}{b^2} \ov{\rho} \,h'''\big( (t,\vx)\cdot \veta \big) \ec \label{eq:op-5temp3}\\
		\opL_{U_{12}}[g](t,\vx) &= -\frac{a}{b} \ov{m}_2 \,h'''\big( (t,\vx)\cdot \veta \big) \ec \label{eq:op-6temp3}\\
		\opL_{U_{13}}[g](t,\vx) &= -\frac{a}{b} \ov{m}_3 \,h'''\big( (t,\vx)\cdot \veta \big) \ec \label{eq:op-7temp3}\\
		\opL_{U_{22}}[g](t,\vx) &= \Big( -\frac{a^2}{b^2} \ov{\rho} + \ov{U}_{11} + \ov{U}_{22} \Big) \,h'''\big( (t,\vx)\cdot \veta \big) \ec \label{eq:op-8temp3}\\
		\opL_{U_{23}}[g](t,\vx) &= \ov{U}_{23} \,h'''\big( (t,\vx)\cdot \veta \big) \ed \label{eq:op-9temp3}
	\end{align}
	Note that we obtain 
	\begin{align*}
		a\ov{\rho} + b \ov{m}_1 &= 0 \ec \\
		a\ov{m}_1 + b \ov{U}_{11} &= 0 \ec \\
		a\ov{m}_2 + b \ov{U}_{12} &= 0 \ec \\
		a\ov{m}_3 + b \ov{U}_{13} &= 0  
	\end{align*}
	from \eqref{eq:op-kernel} and the fact that $\veta=(a,b,0,0)^\trans$. Thus \eqref{eq:op-1temp3} - \eqref{eq:op-9temp3} imply \eqref{eq:op-p2a} - \eqref{eq:op-p2c}.
	
	\item If $\ov{\rho}=0$ we immediately get $\opL_\rho \equiv 0 $ from \eqref{eq:op-defn-3D-rho}.
\end{enumerate}
\end{proof}

In order to reduce the general case in the proof of Proposition~\ref{prop:operators} (see Subsection~\ref{subsec:convint-op-proof}) to one of the cases studied in Lemmas \ref{lemma:op-2D} and \ref{lemma:op-3D}, we need the following lemma:

\begin{lemma} \label{lemma:prop-A}
	Let $\mA\in\R^{(1+n)\times(1+n)}$ of the form
	$$
	\mA = \left( \begin{array}{cc}
		1 & \vz^\trans \\
		\vz & \mB 
	\end{array} \right)\ec
	$$
	where $\mB\in\R^{n\times n}$ is orthogonal. The the following statements hold:
	\begin{itemize}
		 \item $\mA$ is orthogonal;
		 \item For any symmetric matrix $\mM\in\R^{(1+n)\times(1+n)}$ 
		\begin{itemize}
			\item $\mA^\trans \mM \mA$ is symmetric, 
			\item $\tr (\mA^\trans \mM \mA) = \tr \mM$ and 
			\item $[\mA^\trans \mM \mA]_{tt}  = M_{tt}$.
		\end{itemize}
	\end{itemize}
\end{lemma}

\begin{proof} 
	A simple computation shows that 
	$$
		\mA \cdot \mA^\trans = \left( \begin{array}{cc}
			1 & \vz^\trans \\
			\vz & \mB 
		\end{array} \right) \cdot \left( \begin{array}{cc}
			1 & \vz^\trans \\
			\vz & \mB^\trans
		\end{array} \right) = \left( \begin{array}{cc}
			1 & \vz^\trans \\
			\vz & \mB \cdot \mB^\trans
		\end{array} \right) = \id_{1+n}
	$$
	and hence $\mA$ is orthogonal.
	
	The symmetry of $\mA^\trans \mM \mA$ is clear. Furthermore we have
	\begin{align*}
		\tr (\mA^\trans \mM \mA) &= \sum_{i,j,k \in\{t,1,...,n\}} [\mA^\trans]_{ij} M_{jk} [\mA]_{ki} = \sum_{j,k\in\{t,1,...,n\}} M_{jk} \sum_{i\in\{t,1,...,n\}} [\mA]_{ki} [\mA^\trans]_{ij} \\
	&= \sum_{j,k\in\{t,1,...,n\}} M_{jk} \delta_{kj} = \sum_{j\in\{t,1,...,n\}} M_{jj} = \tr\mM
	\end{align*}
	and
	\begin{align*}
		[\mA^\trans \mM \mA]_{tt} &= \sum_{j,k\in\{t,1,...,n\}} [\mA^\trans]_{tj} M_{jk} A_{kt} = \sum_{j,k\in\{t,1,...,n\}} A_{jt} M_{jk} A_{kt} = \sum_{j,k\in\{t,1,...,n\}} \delta_{jt} M_{jk} \delta_{kt} = M_{tt} \ed 
	\end{align*}
\end{proof}

\subsection{Proof of Proposition~\ref{prop:operators}} \label{subsec:convint-op-proof}

\begin{proof}[Proof of Proposition~\ref{prop:operators}]
	First of all note that the cases
	\begin{itemize}
		\item $\veta=(\eta_t,\vz)^\trans\in\R^{1+n}$ with $\eta_t\neq 0$,
		\item $n=2$ and $\veta=(\eta_t,\eta_1,0)^\trans$ with $\eta_1\neq 0$ and 
		\item $n=3$ and $\veta=(\eta_t,\eta_1,0,0)^\trans$ with $\eta_1\neq 0$
	\end{itemize}
	were already covered in Lemmas \ref{lemma:op-eta=e1}, \ref{lemma:op-2D} and \ref{lemma:op-3D}, respectively.
	
	It remains to consider the case $\veta=(\eta_t,\veta_\vx)^\trans$ with $\veta_\vx\neq \vz$. We reduce this case to the cases treated in Lemmas \ref{lemma:op-2D} and \ref{lemma:op-3D} via a rotation.
	
	Since $\veta_\vx\neq\vz$, we can find vectors $\vb_2,...,\vb_n\in\R^n$ such that $\left\{\frac{\veta_\vx}{|\veta_\vx|},\vb_2,...,\vb_n\right\}$ form an orthonormal basis of $\R^n$. Define a matrix $\mA\in\R^{(1+n)\times(1+n)}$ by
	\begin{equation}
	\mA:= \left(\begin{array}{c|c|c|c|c}
	1 & \multicolumn{1}{c}{0} & \multicolumn{1}{c}{\cdots} & \multicolumn{1}{c}{\cdots} & 0 \\ \hline
	0  & &  &  &  \\
	\vdots & \frac{\veta_\vx}{|\veta_\vx|} & \vb_2 & \cdots & \vb_n \\
	0 & &  &  & 
	\end{array}\right)\ed
	\end{equation}
	Note that the matrix $\mA$ is orthogonal due to Lemma~\ref{lemma:prop-A}.
	
	Now we define $(\ov{\rho}^\rot,\ov{\vm}^\rot,\ov{\mU}^\rot)\in\R\times\R^n\times\szn$ by 
	\begin{equation} \label{eq:op-forwardrot}
	\left(\begin{array}{cc}
		\ov{\rho}^\rot & (\ov{\vm}^\rot)^\trans \\
		\ov{\vm}^\rot & \ov{\mU}^\rot 
	\end{array}\right) = \mA^\trans \left(\begin{array}{cc}
		\ov{\rho} & \ov{\vm}^\trans \\
		\ov{\vm} & \ov{\mU} 
	\end{array}\right) \mA\ed
	\end{equation}
	According to Lemma~\ref{lemma:prop-A} this is a proper definition, i.e. the right-hand side of \eqref{eq:op-forwardrot} is symmetric and $\ov{\mU}^\rot $ is trace-less.
	
	In addtion to that we define $\veta^\rot\in\R^{1+n}$ by 
	\begin{equation} \label{eq:eta-rot}
		\veta^\rot:=\mA^\trans \veta \ed
	\end{equation} 
	This yields 
	\begin{equation*}
		\veta^\rot=\mA^\trans \veta = \left(\begin{array}{c|ccc}
			1 & 0 & \cdots & 0 \\ \hline
			0 & & \frac{\veta_\vx^\trans}{|\veta_\vx|} \vphantom{\vdots} & \\ \cline{2-4}
			\vdots & & \vb_2^\trans & \\ \cline{2-4}
			\vdots & & \vdots & \\ \cline{2-4}
			0 & & \vb_n^\trans \vphantom{\vdots} & 
		\end{array}\right) \left(\begin{array}{c}
			\eta_t \\
			\veta_\vx
		\end{array}\right) = \left(\begin{array}{c}
			\eta_t \\
			|\veta_\vx| \\
			0 \\
			\vdots \\
			0
		\end{array} \right)
	\end{equation*}
	and 
	\begin{equation*}
	\left(\begin{array}{cc}
		\ov{\rho}^\rot & (\ov{\vm}^\rot)^\trans \\
		\ov{\vm}^\rot & \ov{\mU}^\rot 
	\end{array}\right) \veta^\rot = \mA^\trans \left(\begin{array}{cc}
		\ov{\rho} & \ov{\vm}^\trans \\
		\ov{\vm} & \ov{\mU} 
	\end{array}\right) \mA \mA^\trans \veta = \mA^\trans \vz = \vz\ed
	\end{equation*} 
	Hence we can apply Lemmas \ref{lemma:op-2D} and \ref{lemma:op-3D} to $(\ov{\rho}^\rot,\ov{\vm}^\rot,\ov{\mU}^\rot)$ and $\veta^\rot$ to obtain operators $\opL_{\rho^\rot}$, $\opL_{\vm^\rot}$ and $\opL_{\mU^\rot}$. Next define  
	\begin{equation} \label{eq:tx-rot}
	\left(\begin{array}{c}
		t^\rot \\
		\vx^\rot 
	\end{array}\right) := \mA^\trans \left(\begin{array}{c}
		t \\
		\vx 
	\end{array}\right) 
	\end{equation}
	and 
	\begin{equation} \label{eq:op-backwardrot}
	\left(\begin{array}{cc}
		\opL_{\rho}[g](t,\vx) & \big(\opL_\vm[g](t,\vx)\big)^\trans \\
		\opL_\vm[g](t,\vx) & \opL_\mU[g](t,\vx)
	\end{array}\right) := \mA \left(\begin{array}{cc}
		\opL_{\rho^\rot}[g](t^\rot,\vx^\rot) & \big(\opL_{\vm^\rot}[g](t^\rot,\vx^\rot)\big)^\trans \\
		\opL_{\vm^\rot}[g](t^\rot,\vx^\rot) & \opL_{\mU^\rot}[g](t^\rot,\vx^\rot)
	\end{array}\right) \mA^\trans \ed
	\end{equation}
	This indeed yields homogeneous differential operators of order three 
	\begin{align*}
		\opL_\rho &: C^\infty(\R^{1+n}) \to C^\infty(\R^{1+n}) \ec\\
		\opL_\vm &: C^\infty(\R^{1+n}) \to C^\infty(\R^{1+n};\R^n) \ec\\
		\opL_\mU &: C^\infty(\R^{1+n}) \to C^\infty(\R^{1+n};\szn)\ec
	\end{align*}
	due to Lemma~\ref{lemma:prop-A}. It remains to show that the operators $\opL_{\rho}$, $\opL_{\vm}$ and $\opL_{\mU}$ satisfy the properties \ref{item:operators.a} - \ref{item:operators.c} of Proposition~\ref{prop:operators}.
	\begin{enumerate}
		\item Let $g\in C^\infty(\R^{1+n})$ be arbitrary. Note that we can write \eqref{eq:op-p1a} and \eqref{eq:op-p1b} as 
		\begin{equation} \label{eq:divtx} 
		\Divtx \left(\begin{array}{cc}
			\opL_{\rho}[g](t,\vx) & \big(\opL_\vm[g](t,\vx)\big)^\trans \\
			\opL_\vm[g](t,\vx) & \opL_\mU[g](t,\vx)
		\end{array}\right) = \left(\begin{array}{c}
			0 \\
			\vz
		\end{array}\right)\ec
		\end{equation}
		which is what we have to show. 
		
		From Lemmas \ref{lemma:op-2D} and \ref{lemma:op-3D} we obtain
		\begin{equation} \label{eq:divtxrot} 
		\Divtxrot \left(\begin{array}{cc}
			\opL_{\rho^\rot}[g](t^\rot,\vx^\rot) & \big(\opL_{\vm^\rot}[g](t^\rot,\vx^\rot)\big)^\trans \\
			\opL_{\vm^\rot}[g](t^\rot,\vx^\rot) & \opL_{\mU^\rot}[g](t^\rot,\vx^\rot)
		\end{array}\right) = \left(\begin{array}{c}
			0 \\
			\vz
		\end{array}\right)\ed
		\end{equation}
		
		Let us use the abbreviations 
		\begin{align*}
			\mM &:= \left(\begin{array}{cc}
				\opL_{\rho}[g](t,\vx) & \big(\opL_\vm[g](t,\vx)\big)^\trans \\
				\opL_\vm[g](t,\vx) & \opL_\mU[g](t,\vx)
			\end{array}\right) \ec \\
			\mM^\rot &:= \left(\begin{array}{cc}
				\opL_{\rho^\rot}[g](t^\rot,\vx^\rot) & \big(\opL_{\vm^\rot}[g](t^\rot,\vx^\rot)\big)^\trans \\
				\opL_{\vm^\rot}[g](t^\rot,\vx^\rot) & \opL_{\mU^\rot}[g](t^\rot,\vx^\rot)
			\end{array}\right) \ed
		\end{align*}
		For all $i = t,1,...,n$ we have 
		\begin{align*}
			\big[\Divtx \mM \big]_i &= \sum_{j\in\{t,1,...,n\}} \partial_j \mM_{ij} \\
			&= \sum_{j,k,\ell\in\{t,1,...,n\}} [\mA]_{ik} \, \partial_j [\mM^\rot]_{k\ell} \,[\mA^\trans]_{\ell j} \\
			&= \sum_{j,k,\ell,\mu\in\{t,1,...,n\}} [\mA]_{ik} \, \partial_{\mu^\rot} [\mM^\rot]_{k\ell}\,\partial_j \mu^\rot \,[\mA^\trans]_{\ell j} \ed
		\end{align*}
		From \eqref{eq:tx-rot} we obtain 
		$$
		\partial_j \mu^\rot = [\mA^\trans]_{\mu j} 
		$$
		for all $j,\mu\in\{t,1,...,n\}$ and hence
		\begin{align*}
			\big[\Divtx \mM \big]_i &= \sum_{j,k,\ell,\mu\in\{t,1,...,n\}} [\mA]_{ik} \, \partial_{\mu^\rot} [\mM^\rot]_{k\ell}\,\partial_j \mu^\rot \,[\mA^\trans]_{\ell j} \\
			&= \sum_{j,k,\ell,\mu\in\{t,1,...,n\}} [\mA]_{ik} \, \partial_{\mu^\rot} [\mM^\rot]_{k\ell}\,[\mA^\trans]_{\mu j} \,[\mA^\trans]_{\ell j} \\
			&= \sum_{k,\ell,\mu\in\{t,1,...,n\}} [\mA]_{ik} \, \partial_{\mu^\rot} [\mM^\rot]_{k\ell}\,\delta_{\mu \ell} \\
			&= \sum_{k,\ell\in\{t,1,...,n\}} [\mA]_{ik} \, \partial_{\ell^\rot} [\mM^\rot]_{k\ell} \quad = \quad 0
		\end{align*}
		due to \eqref{eq:divtxrot}. This shows \eqref{eq:divtx}.
		
		\item First of all note that 
		\begin{equation} \label{eq:rot-scalarproduct}
			(t^\rot,\vx^\rot) \cdot \veta^\rot = (t,\vx) \mA \mA^\trans \veta = (t,\vx) \cdot \veta
		\end{equation}
		according to \eqref{eq:eta-rot} and \eqref{eq:tx-rot}. Hence we have $g(t,\vx) = h\big( (t,x)\cdot \veta\big) = h\big( (t^\rot,\vx^\rot) \cdot \veta^\rot\big)$ and thus Lemmas \ref{lemma:op-2D} and \ref{lemma:op-3D} yield 
		\begin{align*} 
			&\left(\begin{array}{cc}
				\opL_{\rho^\rot}[g](t^\rot,\vx^\rot) & \big(\opL_{\vm^\rot}[g](t^\rot,\vx^\rot)\big)^\trans \\
				\opL_{\vm^\rot}[g](t^\rot,\vx^\rot) & \opL_{\mU^\rot}[g](t^\rot,\vx^\rot)
			\end{array}\right) \\
			&= \left(\begin{array}{cc}
				\ov{\rho}^\rot & (\ov{\vm}^\rot)^\trans \\
				\ov{\vm}^\rot & \ov{\mU}^\rot 
			\end{array}\right) h'''\big( (t^\rot,\vx^\rot) \cdot \veta^\rot\big) \ed
		\end{align*}
		With \eqref{eq:op-forwardrot}, \eqref{eq:op-backwardrot} and \eqref{eq:rot-scalarproduct} this shows 
		\begin{equation*} 
		\left(\begin{array}{cc}
			\opL_{\rho}[g](t,\vx) & \big(\opL_{\vm}[g](t,\vx)\big)^\trans \\
			\opL_{\vm}[g](t,\vx) & \opL_{\mU}[g](t,\vx)
		\end{array}\right) = \left(\begin{array}{cc}
			\ov{\rho} & \ov{\vm}^\trans \\
			\ov{\vm} & \ov{\mU} 
		\end{array}\right) h'''\big( (t,\vx) \cdot \veta\big) \ec
		\end{equation*}
		i.e. \eqref{eq:op-p2a} - \eqref{eq:op-p2c}.
		
		\item Let $\ov{\rho} = 0$. Then $\ov{\rho}^\rot = 0$ due to Lemma~\ref{lemma:prop-A} and \eqref{eq:op-forwardrot}. Hence Lemmas \ref{lemma:op-2D} and \ref{lemma:op-3D} yield $\opL_{\rho^\rot} \equiv 0$. Now again by Lemma~\ref{lemma:prop-A} we deduce form \eqref{eq:op-backwardrot} that $\opL_\rho \equiv 0$.
	\end{enumerate}
\end{proof}

\section{The Convex-Integration-Theorem} \label{sec:convint-thm}

\subsection{Statement of the Theorem} \label{subsec:convint-thm-statement}

Let us recall that 
\begin{align*}
	K &=  \left\{(\rho,\vm,\mU)\in\R^+\times\R^n\times\szn\,\Big|\,\mU=\frac{\vm\otimes\vm}{\rho} + \left(p(\rho) - \frac{2c}{n}\right)\id\right\}\ec \\
	\sU = K^\Lambda &= \left\{(\rho,\vm,\mU)\in\R^+\times\R^n\times\szn\,\Big|\,e(\rho,\vm,\mU) \leq c\right\}\ec 
\end{align*}
see \eqref{eq:K}, Propositions \ref{prop:KLambda=U} and \ref{prop:compKLambda}, and Lemma~\ref{lemma:Kast}. 

\begin{lemma} 
	The following statements hold.
	\begin{align}
		\interior{\sU} &= \left\{(\rho,\vm,\mU)\in\R^+\times\R^n\times\szn\,\Big|\,e(\rho,\vm,\mU) < c\right\} \ec \label{eq:interiorU}\\
		K &\subsetneq  \left\{(\rho,\vm,\mU)\in\R^+\times\R^n\times\szn\,\Big|\,e(\rho,\vm,\mU) = c\right\} \ed \label{eq:boundaryU}
	\end{align}
\end{lemma}


\begin{proof}
	Let us start with the prove of \eqref{eq:interiorU}. Let $S$ denote the right-hand side of \eqref{eq:interiorU}. Due to Lemma~\ref{lemma:e-convex}, the set $S$ is open as it is the pre-image of the open set $(-\infty,c)$ under the continuous mapping $e$. Moreover it is obvious that $S\subset\sU$. Hence $S\subset \interior{\sU}$. To prove the converse we proceed as in the proof of Lemma~\ref{lemma:Kast}. Let $(\rho,\vm,\mU)\in \interior{\sU}$. We want to show that $e(\rho,\vm,\mU)< c$. Since 
	$$
		\frac{\vm\otimes \vm}{\rho} + p(\rho)\id - \mU \ \in\  \sym{n} \ec
	$$
	there exists $T\in\orth{n}$ such that 
	$$
		\frac{\vm\otimes \vm}{\rho} + p(\rho)\id - \mU = \mT \left(\begin{array}{ccc}
			\lambda_1 & & \\ & \ddots & \\ & & \lambda_n \end{array} \right) \mT^{-1} \ec
	$$
	where we may assume that the eigenvalues are ordered $\lambda_1\leq ...\leq\lambda_n$. Since $(\rho,\vm,\mU)\in \sU$, we duduce that $\frac{n}{2}\lambda_n = e(\rho,\vm,\mU) \leq c$. Again we denote the normed eigenvector, which corresponds to the eigenvalue $\lambda_i$, with $\vb_i$. Then $\mT=(\vb_1\cdots \vb_n)$. Since $\vb_1,..., \vb_n$ form a basis of $\R^n$, there are unique coefficients $\alpha_1,...,\alpha_n\in \R$ such that $\vm=\sum_{i=1}^n \alpha_i\vb_i$. Define\footnote{The fact that $\hat{\mU}\in\szn$ can be shown as in the proof of Lemma~\ref{lemma:Kast}.}
	$$
		(\hat{\rho},\hat{\vm},\hat{\mU}) := \left(0,\vb_n,\frac{\vm\otimes\vb_n + \vb_n\otimes \vm - 2\alpha_n \vb_n\otimes \vb_n}{\rho}\right) \ \in\  \R\times \R^n\times \szn\ed
	$$ 
	For $\tau\in \R$ a straightforward computation, whose details can be found in the proof of Lemma~\ref{lemma:Kast}, yields
	\begin{align*}
		&\mT^{-1} \left(\frac{(\vm+\tau\hat{\vm})\otimes (\vm+\tau \hat{\vm})}{\rho+\tau \hat{\rho}} + p(\rho+\tau \hat{\rho})\id - (\mU+\tau \hat{\mU})\right) \mT \\
		&= \left(\begin{array}{ccc}
		\lambda_1 & & \\ & \ddots & \\ & & \lambda_n \end{array} \right) + \frac{2\alpha_n \tau + \tau^2}{\rho} \left(\begin{array}{cccc}
		0 & & & \\ & \ddots & & \\ & & 0 & \\ & & & 1 \end{array} \right) \ed
	\end{align*}
	From $(\rho,\vm,\mU)\in \interior{\sU}$ we deduce that $(\rho+\tau\hat{\rho},\vm+\tau\hat{\vm},\mU+\tau\hat{\mU})\in \sU$ if $|\tau|$ is sufficiently small. Hence 
	$$
		\frac{n}{2} \left(\lambda_n + \frac{2\alpha_n \tau + \tau^2}{\rho} \right) \leq e(\rho+\tau\hat{\rho},\vm+\tau\hat{\vm},\mU+\tau\hat{\mU}) \leq c \ed
	$$
	This implies that $e(\rho,\vm,\mU)= \frac{n}{2} \lambda_n < c$ and therefore $\interior{\sU}\subset S$. Hence we have proven \eqref{eq:interiorU}.
	
	Let us turn our attention towards \eqref{eq:boundaryU}. Let $(\rho,\vm,\mU)\in K$. Then all eigenvalues of the matrix
	$$
		\frac{\vm\otimes \vm}{\rho} + p(\rho)\id -\mU
	$$
	are equal to $\frac{2c}{n}$. Hence $e(\rho,\vm,\mU)=c$. The interesting issue in \eqref{eq:boundaryU} is the fact that $K$ is a \emph{strict} subset. This is shown by the following example. Set $\rho$ such that $p(\rho)=\frac{2c}{n^2}$, $\vm:=\vz$, and 
	$$
	\mU := \left(\begin{array}{cccc}
		(1-n)\frac{2c}{n^2} & & & \\
		& \frac{2c}{n^2} & & \\
		& & \ddots & \\
		& & & \frac{2c}{n^2} 
	\end{array} \right)\ed
	$$
	We claim that $e(\rho,\vm,\mU)=c$ but $(\rho,\vm,\mU)\notin K$.

	Indeed it holds, that
	$$
		\frac{\vm\otimes\vm}{\rho} + p(\rho)\id - \mU = \left(\begin{array}{cccc}
			\frac{2c}{n} & & & \\
			& 0& & \\
			& & \ddots & \\
			& & & 0 
		\end{array} \right) \neq \frac{2c}{n} \id
	$$
	and hence $(\rho,\vm,\mU)\notin K$. But on the other hand
	$$
		e(\rho,\vm,\mU) = \frac{n}{2} \lambda_{\max} \left(\frac{\vm\otimes\vm}{\rho} + p(\rho)\id - \mU\right)	= \frac{n}{2} \frac{2c}{n} = c \ed
	$$
\end{proof}

We are ready to state the most important theorem in this book, see \name{De~Lellis} and \name{Sz{\'e}kelyhidi} \cite[Proposition 2]{DelSze10} or \cite[Proposition 2.4]{DelSze12} for the corresponding ``incompressible version''. 

\begin{thm} \label{thm:convint} 
	Let $n\in \{2,3\}$, $\Gamma\subset\R^{1+n}$ be a Lipschitz space-time domain (not necessarily bounded) and let $r>0$ and $c>0$.
		
	Assume there exist $(\rho_0,\vm_0,\mU_0)\in C^1\big(\closure{\Gamma};\R^+ \times \R^n\times \szn\big)$ with the following properties:
	\begin{itemize}
		\item The PDEs 
			\begin{align}
				\partial_t \rho_0 + \Div \vm_0 &= 0 \ec \label{eq:p0-pde1}\\
				\partial_t \vm_0 + \Div \mU_0 &= \vz \label{eq:p0-pde2}
			\end{align}
			hold pointwise for all $(t,\vx)\in \Gamma$;
		\item They take values in $\interior{\sU}$, i.e. 
			\begin{equation}  \label{eq:p0-subs} 
				(\rho_0,\vm_0,\mU_0)(t,\vx) \ \in\  \interior{\sU}\qquad \text{ for all }(t,\vx)\in \Gamma \es 
			\end{equation} 
		\item The density is bounded from below
			\begin{equation} \label{eq:p0-dens-bdd}
				\rho_0(t,\vx) > r \qquad \text{ for all }(t,\vx)\in \closure{\Gamma} \ed 
			\end{equation}
	\end{itemize}
	
	Then there exist infinitely many $(\rho,\vm)\in L^\infty(\Gamma; \R^+ \times \R^n)$ which solve the barotropic Euler equations in the following sense: 
	\begin{enumerate} 
		\item \label{item:convint.a} The PDEs hold weakly, i.e. 
		\begin{align}
			\iint_\Gamma \big[\rho \partial_t\phi + \vm\cdot \Grad\phi \big] \dx\dt - \int_{\partial \Gamma} \big[\rho_0 \,n_t + \vm_0\cdot \vn_\vx \big] \phi \dS_{t,\vx} &= 0 \ec \label{eq:sol-pde1}\\ 
			\iint_\Gamma \bigg[\vm\cdot\partial_t\vphi + \frac{\vm\otimes\vm}{\rho} : \Grad\vphi + p(\rho) \Div\vphi\bigg] \dx\dt \ \;\quad \qquad & \notag \\
			- \int_{\partial\Gamma} \left[ \vm_0\cdot \vphi\,n_t + (\mU_0\cdot \vphi)\cdot \vn_\vx + \frac{2c}{n} \vphi \cdot \vn_\vx \right] \dS_{t,\vx} &= 0 \label{eq:sol-pde2} 
		\end{align} 
		for all test functions\footnote{Once more we would like to remark that this means that the test functions $(\phi,\vphi)$ might not vanish on the boundary $\partial\Gamma$ and hence one could also write $(\phi,\vphi)\in C_c^\infty(\R^{1+n};\R\times\R^n)$ instead.} $(\phi,\vphi)\in C_c^\infty(\closure{\Gamma};\R\times\R^n)$;
		
		\item \label{item:convint.b} The density is bounded from below, i.e.  
		\begin{equation} \label{eq:sol-dens-bdd}
			\rho(t,\vx) \geq r \qquad \text{ for a.e. }(t,\vx)\in \Gamma \es 
		\end{equation} 
		
		\item \label{item:convint.c} The ``trace condition'' holds, i.e.
		\begin{equation} \label{eq:sol-trace}
			\frac{|\vm(t,\vx)|^2}{2 \rho(t,\vx)} + \frac{n}{2} p\big(\rho(t,\vx)\big) = c \qquad \text{ for a.e. }(t,\vx)\in \Gamma \ed 
		\end{equation}
	\end{enumerate}
\end{thm}

\begin{rem}
	Note that in this book initial data are denoted by $(\rho_\init,\vm_\init)$ whereas in the literature it is common to write $(\rho_0,\vm_0)$ instead. Consequently the triple $(\rho_0,\vm_0,\mU_0)$ in Theorem~\ref{thm:convint} does not represent initial data but a subsolution. In particular the functions $\rho_0,\vm_0,\mU_0$ do depend on time $t$. As already mentioned in the beginning of this chapter, Theorem~\ref{thm:convint} does not immediately yield solutions to an initial (boundary) value problem. Hence initial data do not appear in Theorem~\ref{thm:convint}. We will consider special initial (boundary) value problems in Chapters \ref{chap:appl-ibvp} and \ref{chap:appl-riemann}. Here we will have to find a suitable triple $(\rho_0,\vm_0,\mU_0)$ such that Theorem~\ref{thm:convint} can be applied and the boundary terms in \eqref{eq:sol-pde1}, \eqref{eq:sol-pde2} coincide with the boundary terms determined by the problem.
\end{rem}

\begin{rem}
	Almost every argument used to prove Theorem~\ref{thm:convint} is valid for higher dimensions $n\geq 4$ as well. The only ingredient which is not applicable for $n\geq 4$ is Proposition~\ref{prop:operators}. Note that this does not mean that operators like in Proposition~\ref{prop:operators} do not exist in the case $n\geq 4$. However we didn't try to find such operators because $n\geq 4$ is physically not relevant.
\end{rem}

The remaining part of the current section is dedicated to prove Theorem~\ref{thm:convint}.

\subsection{Functional Setup} \label{subsec:convint-thm-fs} 

Let $n\in\{2,3\}$, $\Gamma\subset \R^{1+n}$, $r,c>0$ and $(\rho_0,\vm_0,\mU_0)\in C^1(\closure{\Gamma};\R^+ \times \R^n\times \szn)$ given such that the properties in the assumption of Theorem~\ref{thm:convint}, i.e. \eqref{eq:p0-pde1} - \eqref{eq:p0-dens-bdd}, hold. 

\begin{defn} \label{defn:X0}
Define the set $X_0$ by
\begin{equation*}
	X_0 := \Big\{ (\rho,\vm) \in C^1\big(\closure{\Gamma};\R^+\times \R^n\big) \,\Big|\,\exists\, \mU\in C^1\big(\closure{\Gamma};\szn\big) \text{ such that \ref{item:pfs-pde} - \ref{item:pfs-dens-bdd} hold} \Big\} \ec
\end{equation*}
where the properties \ref{item:pfs-pde} - \ref{item:pfs-dens-bdd} read as follows:
\begin{enumerate}
	\item \label{item:pfs-pde} The PDEs 
	\begin{align}
		\partial_t \rho + \Div \vm &= 0 \ec \label{eq:pfs-pde1}\\
		\partial_t \vm + \Div \mU &= \vz \label{eq:pfs-pde2}
	\end{align}
	hold pointwise for all $(t,\vx)\in \Gamma$; 
		
	\item \label{item:pfs-subs} The $(\rho,\vm,\mU)$ take values in $\interior{\sU}$, i.e. 
	\begin{equation}  \label{eq:pfs-subs}
		(\rho,\vm,\mU)(t,\vx) \ \in\ \interior{\sU}\qquad \text{ for all }(t,\vx)\in \Gamma \es 
	\end{equation} 
	
	\item \label{item:pfs-bc} On the boundary, $(\rho,\vm,\mU)$ coincide with $(\rho_0,\vm_0,\mU_0)$, i.e.
	\begin{equation} \label{eq:pfs-bc}
		(\rho,\vm,\mU)(t,\vx) = (\rho_0,\vm_0,\mU_0)(t,\vx) \qquad \text{ for all }(t,\vx)\in \partial \Gamma \es 
	\end{equation} 
	
	\item \label{item:pfs-dens-bdd} The density is bounded from below, i.e.
	\begin{equation} \label{eq:pfs-dens-bdd} 
		\rho(t,\vx) > r \qquad \text{ for all }(t,\vx)\in \closure{\Gamma} \ed 
	\end{equation} 
\end{enumerate}
\end{defn}

First of all we observe that $(\rho,\vm,\mU)$ take values on a bounded set if $(\rho,\vm)\in X_0$:

\begin{lemma} \label{lemma:valuesX0bounded}
	There exists $M>r$ with the following properties. If $(\rho,\vm)\in X_0$, then
	\begin{equation} \label{eq:31-temp-convint} 
		(\rho,\vm,\mU)(t,\vx) \ \in\  [r,M]\times \closure{B_n(\vz,M)} \times \closure{B_{\szn}(\mZ,M)} 
	\end{equation}
	for all $(t,\vx)\in\closure{\Gamma}$, where $\mU$ is the matrix-valued function which corresponds to $(\rho,\vm)$.
\end{lemma}

\begin{proof} 
		From \eqref{eq:pfs-dens-bdd} we obtain the bound of $\rho$ from below. Due to Lemma~\ref{lemma:U-bdd} (which is applicable according to \eqref{eq:pfs-subs}), \eqref{eq:31-temp-convint} holds for all $(t,\vx)\in \Gamma$. Since every $(\rho,\vm)\in X_0$ is continuous and so is the corresponding $\mU$, we deduce that \eqref{eq:31-temp-convint} even holds on the closure, i.e. for all $(t,\vx)\in \closure{\Gamma}$.
\end{proof}

\begin{defn} \label{defn:X}
	Let $X$ be the closure of $X_0$ with respect to the $L^\infty$ weak-$\ast$ topology. 
\end{defn}

\begin{rem}
	With the notation used in Section~\ref{sec:convint-prel}, $X$ is the set of subsolutions. In this manner one could call $X_0$ the set of \emph{strict} subsolutions since they take values in the interior $\interior{\sU}$, see \eqref{eq:pfs-subs}.
\end{rem}

\begin{prop} \label{prop:Xcomplete}
	There exists a metric $d$ on $X$ which induces the weak-$\ast$ topology, and furthermore the metric space $(X,d)$ is compact and complete. 
\end{prop}

\begin{proof}
	From Lemma~\ref{lemma:valuesX0bounded} we obtain that the set $X_0$ is bounded with respect to the $\|\cdot \|_{L^\infty}$-norm. Hence the properties of the weak-$\ast$ topology given in Section~\ref{sec:not-weakconv} yield existence of $d$, which induces the weak-$\ast$ topology, and compactness of $(X,d)$. It is a well-known fact that compact metric spaces are complete, i.e. $(X,d)$ is complete.
\end{proof}

\subsection{The Functionals $I_{\Gamma_0}$ and the Perturbation Property} \label{subsec:convint-thm-pp} 

In this subsection we prepare the two main ingredients of the proof of Theorem~\ref{thm:convint}. In order to state them, we introduce functionals $I_{\Gamma_0}$. The first ingredient is Lemma~\ref{lemma:prop-I}~\ref{item:prop-I.e} which yields a sufficient condition for a pair $(\rho,\vm)\in X$ to be a solution in the sense of Theorem~\ref{thm:convint}. The second main ingredient is the so-called \emph{Perturbation Property} (Proposition~\ref{prop:pert-prop}). This property will yield a wanted contradiction in order to prove Theorem~\ref{thm:convint}.

\begin{defn} \label{defn:E}
	Define the mapping $E:\R^+\times \R^n \to \R$ by 
	$$
		(\rho,\vm)\mapsto E(\rho,\vm) := \frac{|\vm|^2}{2\rho} + \frac{n}{2} p(\rho) -c \ed
	$$
\end{defn}

Notice, that $E(\rho,\vm)>-c$ for all $(\rho,\vm)\in\R^+\times \R^n$. This simple fact will be used several times later. 
With the following lemma we summarize some properties of $E$.

\begin{lemma} \label{lemma:prop-E}
	The following claims hold.
	\begin{enumerate}
		\item \label{item:prop-E.a} The restriction of $E$ to the set $[r,M]\times \closure{B_n(\vz,M)}$ is uniformly continuous\footnote{This statement is valid for all $M>r$. What we need here in particular, is the $M$ from Lemma~\ref{lemma:valuesX0bounded}.}. Similarly the restriction of $e$ to the set $[r,M]\times \closure{B_n(\vz,M)} \times \closure{B_{\szn}(\mZ,M)}$ is uniformly continuous.
		\item \label{item:prop-E.b} The mapping $E$ is convex.
		\item \label{item:prop-E.c} The following inequality is valid for all $ (\rho,\vm,\mU)\in \R^+\times \R^n\times \szn$:
		\begin{equation} \label{eq:40-temp-convint}
			E(\rho,\vm) \leq e(\rho,\vm,\mU) - c  \ed
		\end{equation}
		\item \label{item:prop-E.d} If $(\rho,\vm)\in X_0$, then $E\big(\rho(t,\vx),\vm(t,\vx)\big)<0$ for all $(t,\vx)\in \Gamma$. 
	\end{enumerate}
\end{lemma}

\begin{proof} 
	\begin{enumerate}
		\item Obviously $E$ is continuous on $\R^+\times \R^n$. Hence $E$ is uniformly continuous on the compact set $[r,M]\times \closure{B_n(\vz,M)}$. In order to prove the analogous claim for $e$, we argue in the same way. 
		
		\item The convexity of $E$ simply follows from Lemma~\ref{lemma:convexity-kin-en} and the convexity of $p$. 
		
		\item This follows from Lemma~\ref{lemma:not-lmax-tr}. Indeed  
		\begin{align*}
			E(\rho,\vm) &= \frac{|\vm|^2}{2\rho} + \frac{n}{2} p(\rho) - c \\
			&= \half \tr \left(\frac{\vm\otimes\vm}{\rho} + p(\rho)\id - \mU \right) -c \\
			&\leq \frac{n}{2} \lambda_{\max} \left(\frac{\vm\otimes\vm}{\rho} + p(\rho)\id - \mU \right) -c  \\
			&=  e(\rho,\vm,\mU) - c \ed
		\end{align*} 
		
		\item Let $(\rho,\vm)\in X_0$. According to the definition of the set $X_0$, there exists a matrix-valued function $\mU\in C^1(\closure{\Gamma};\szn)$ such that in particular \eqref{eq:pfs-subs} holds for all $(t,\vx)\in\Gamma$. Using \eqref{eq:interiorU} and \eqref{eq:40-temp-convint} we obtain
		\begin{align*}
			E\big((\rho,\vm)(t,\vx)\big) \leq  e\big((\rho,\vm,\mU)(t,\vx)\big) - c  <0 
		\end{align*}
		for all $(t,\vx)\in \Gamma$.
	\end{enumerate}
\end{proof}

Now we are ready to define the functionals $I_{\Gamma_0}$ via the mapping $E$.

\begin{defn} \label{defn:I}
	For any open and bounded subset $\Gamma_0 \subset \R^{1+n}$ we define the functional $I_{\Gamma_0}: L^\infty (\Gamma_0;\R^+\times \R^n) \to \R$ by
	\begin{align*} 
	(\rho,\vm)\mapsto I_{\Gamma_0} (\rho,\vm) := &\iint_{\Gamma_0} E\big(\rho(t,\vx),\vm(t,\vx)\big)\dx\dt \\
	= &\iint_{\Gamma_0} \left(\frac{\big|\vm(t,\vx)\big|^2}{2\rho(t,\vx)} + \frac{n}{2} p\big(\rho(t,\vx)\big) -c \right)\dx\dt \ed
	\end{align*}
\end{defn}

Let us summarize some properties of $I_{\Gamma_0}$.

\begin{lemma}  \label{lemma:prop-I}
	The following claims hold.
	\begin{enumerate}
		\item \label{item:prop-I.a} For all open and bounded sets $\Gamma_0\subset \Gamma$ the map $I_{\Gamma_0}:X\to \R$ is lower semi-continuous with respect to $d$.
		\item \label{item:prop-I.b} For all non-empty, open and bounded sets $\Gamma_0\subset \Gamma$ and all $(\rho,\vm) \in X_0$ we have $I_{\Gamma_0} (\rho,\vm) < 0$. 
		\item \label{item:prop-I.c} For all open and bounded sets $\Gamma_0\subset \Gamma$ and all $(\rho,\vm) \in X$ we have $I_{\Gamma_0} (\rho,\vm) \leq 0$. 
		\item \label{item:prop-I.d} If $(\rho,\vm)\in X$, then $E\big(\rho(t,\vx),\vm(t,\vx)\big)\leq 0$ for a.e. $(t,\vx)\in \Gamma$. 
		\item \label{item:prop-I.e} If $(\rho,\vm) \in X$ with $I_{\Gamma_0} (\rho,\vm) = 0$ for all open and bounded sets $\Gamma_0\subset \Gamma$, then the pair $(\rho,\vm)$ is a solution in the sense of Theorem~\ref{thm:convint}, i.e. $(\rho,\vm)$ satisfies properties \ref{item:convint.a} - \ref{item:convint.c} of Theorem~\ref{thm:convint}.
	\end{enumerate}
\end{lemma}

\begin{rem}
	Statement \ref{item:prop-I.d} of Lemma~\ref{lemma:prop-I} is an analogue of Lemma~\ref{lemma:prop-E}~\ref{item:prop-E.d} for $(\rho,\vm)\in X_0$. Whereas the latter (i.e. Lemma~\ref{lemma:prop-E}~\ref{item:prop-E.d}) was quite simple to prove, the former (i.e. Lemma~\ref{lemma:prop-I}~\ref{item:prop-I.d}) is more difficult. In particular we will use Lemma~\ref{lemma:prop-I}~\ref{item:prop-I.c} together with Lemma~\ref{lemma:not-local-property} to prove it. 
\end{rem}

\begin{proof}
	\begin{enumerate}
		\item According to Lemma~\ref{lemma:prop-E}~\ref{item:prop-E.b}, $E$ is convex. Keeping in mind that the metric $d$ induces the weak-$\ast$ topology on $X$, the claim immediately follows from Lemma~\ref{lemma:not-wls-convex}.
		
		\item This follows immediately from Lemma~\ref{lemma:prop-E}~\ref{item:prop-E.d}.
		
		\item Let $(\rho,\vm)\in X$. Since $X$ is the closure of $X_0$ with respect to $d$, there exists a sequence $(\rho_k,\vm_k)_{k\in\N} \subset X_0$ with $(\rho_k,\vm_k)\mathop{\to}\limits^d (\rho,\vm)$ as $k\to\infty$. Hence we obtain using \ref{item:prop-I.a} and \ref{item:prop-I.b}
		$$
			I_{\Gamma_0} (\rho,\vm)\leq \liminf\limits_{k\to\infty} I_{\Gamma_0} (\rho_k,\vm_k) \leq 0 \ed
		$$
				
		\item This is a direct consequence of \ref{item:prop-I.c} and Lemma~\ref{lemma:not-local-property}, where we set $f:= - E(\rho,\vm)$.
				
		\item As in the proof of \ref{item:prop-I.d}, we can apply Lemma~\ref{lemma:not-local-property} to $f:= - E(\rho,\vm)$ and $f:= E(\rho,\vm)$ to obtain $E\big(\rho(t,\vx),\vm(t,\vx)\big)=0$ for a.e. $(t,\vx)\in\Gamma$. This shows \eqref{eq:sol-trace}, i.e. property \ref{item:convint.c} of Theorem~\ref{thm:convint}.
		
		Next, note again that there exists a sequence $(\rho_k,\vm_k)_{k\in\N} \subset X_0$ with $(\rho_k,\vm_k)\mathop{\to}\limits^d (\rho,\vm)$ as $k\to\infty$, see the proof of \ref{item:prop-I.c}. Because the $\rho_k$ satisfy $\rho_k(t,\vx)>r$ for all $(t,\vx)\in \Gamma$ (see \eqref{eq:pfs-dens-bdd}, property \ref{item:pfs-dens-bdd} of Definition~\ref{defn:X0}), we obtain $\rho(t,\vx)\geq r $ for a.e. $(t,\vx)\in \Gamma$. Indeed if this was not true, we would have $|S|>0$ where 
		\begin{align*}
			S:= \Big\{ (t,\vx) \in \Gamma \,\Big|\, \rho(t,\vx) < r \Big\} \ed
		\end{align*}
		Similar arguments as in the proof of Lemma~\ref{lemma:not-local-property} allow us to assume without loss of generality that $S$ is bounded.  Since 
		$$ 
			\iint_\Gamma \rho_k \charf_S \dx\dt > r|S|\qquad \text{ for all }k\in \N
		$$
		and $\rho_k \mathop{\rightharpoonup}\limits^\ast \rho$, we find\footnote{Use $\varphi:=\charf_S\in L^1(\Gamma)$ in Definition~\ref{defn:not-weak-ast}.}
		$$
			\iint_\Gamma \rho \charf_S \dx\dt \geq r|S| \ed
		$$
		On the other hand from the assumption above we get 
		$$
			\iint_\Gamma \rho \charf_S \dx\dt < r|S| \ec
		$$
		a contradiction. Hence we have proven \eqref{eq:sol-dens-bdd}, i.e. property \ref{item:convint.b} of Theorem~\ref{thm:convint}.
		
		By definition of the set $X_0$ (see Definition~\ref{defn:X0}) there exists $\mU_k \in C^1\big(\closure{\Gamma};\szn\big)$ for each $k\in\N$ with the properties given in Definition~\ref{defn:X0}. From \eqref{eq:pfs-subs} and Lemma~\ref{lemma:valuesX0bounded} we obtain that the sequence $(\mU_k)_{k\in \N}$ is uniformly bounded in $L^\infty$. Due to the facts about the weak-$\ast$ topology given in Section~\ref{sec:not-weakconv}, there exists $\mU\in L^\infty(\Gamma;\szn)$ such that\footnote{To be precise, we might have to consider a subsequence of $(\mU_k)_{k\in\N}$.} $\mU_k \mathop{\rightharpoonup}\limits^\ast \mU$ in $L^\infty$. It remains to prove \eqref{eq:sol-pde1} and \eqref{eq:sol-pde2}, i.e. property \ref{item:convint.a} of Theorem~\ref{thm:convint}.
		
		For each $k\in\N$ we multiply the PDEs \eqref{eq:pfs-pde1} and \eqref{eq:pfs-pde2} with arbitrary test functions $\phi\in \Cc(\closure{\Gamma})$ and $\vphi\in \Cc(\closure{\Gamma};\R^n)$, respectively. Afterwards we integrate over $\Gamma$ to obtain
		\begin{align}
			\iint_\Gamma \big[ \phi\, \partial_t \rho_k + \phi\, \Div\vm_k \big] \dx\dt &= 0 \ec \label{eq:temp18-convint}\\
			\iint_\Gamma \big[ \vphi \cdot \partial_t \vm_k + \vphi \cdot \Div\mU_k \big] \dx\dt &= \vz \ed \label{eq:temp19-convint}
		\end{align}
		Next, we are going to apply the Divergence Theorem (Proposition~\ref{prop:not-divergence}) to the integrals
		\begin{align}
			&\iint_\Gamma \big[ \partial_t (\rho_k\,\phi) + \Div(\vm_k\,\phi) \big] \dx\dt \ec \label{eq:temp20-convint} \\
			&\iint_\Gamma \big[ \partial_t (\vm_k\cdot \vphi) + \Div(\mU_k\cdot \vphi) \big] \dx\dt \ed \label{eq:temp21-convint}
		\end{align}
		
		At a first glance the fact that $\Gamma$ might be unbounded seems to cause problems. However since $\phi$ and $\vphi$ have compact support, we can integrate over a bounded (and still Lipschitz) domain instead of $\Gamma$ in \eqref{eq:temp20-convint} and \eqref{eq:temp21-convint}. Note that the integral over the additional boundary (i.e. the part which appears due to restriction to a bounded domain) vanishes because of the compact support of the test functions $\phi$ and $\vphi$. Hence we obtain
		\begin{align}
			\iint_\Gamma \big[ \partial_t (\rho_k\,\phi) + \Div(\vm_k\,\phi) \big] \dx\dt &= \int_{\partial\Gamma} \big[\rho_k\,n_t + \vm_k\cdot \vn_\vx\big]\phi \dS_{t,\vx} \ec \label{eq:temp22-convint} \\
			\iint_\Gamma \big[ \partial_t (\vm_k\cdot \vphi) + \Div(\mU_k\cdot \vphi) \big] \dx\dt &= \int_{\partial\Gamma} \big[ \vm_k\cdot \vphi\,n_t + (\mU_k\cdot \vphi)\cdot \vn_\vx\big] \dS_{t,\vx} \ed \label{eq:temp23-convint}
		\end{align}
		Together with \eqref{eq:temp18-convint}, \eqref{eq:temp19-convint} and the chain rule, we deduce from \eqref{eq:temp22-convint} and \eqref{eq:temp23-convint} that
		\begin{align} 
		 	\iint_\Gamma \big[ \rho_k\,\partial_t\phi + \vm_k \cdot \Grad\phi) \big] \dx\dt &= \int_{\partial\Gamma} \big[\rho_k\,n_t + \vm_k\cdot \vn_\vx\big]\phi \dS_{t,\vx} \ec \label{eq:temp24-convint} \\
		 	\iint_\Gamma \big[ \vm_k\cdot \partial_t\vphi + \mU_k : \Grad\vphi) \big] \dx\dt &= \int_{\partial\Gamma} \big[ \vm_k\cdot \vphi\,n_t + (\mU_k\cdot \vphi)\cdot \vn_\vx\big] \dS_{t,\vx} \ec \label{eq:temp25-convint}
		 \end{align}
		 where in addition we made use of the fact that $\mU_k$ is symmetric. Due to \eqref{eq:pfs-bc}, the right-hand sides of \eqref{eq:temp24-convint} and \eqref{eq:temp25-convint} can be simplified, which yields
		 \begin{align} 
		 	\iint_\Gamma \big[ \rho_k\,\partial_t\phi + \vm_k \cdot \Grad\phi) \big] \dx\dt &= \int_{\partial\Gamma} \big[\rho_0\,n_t + \vm_0\cdot \vn_\vx\big]\phi \dS_{t,\vx} \ec \label{eq:temp26-convint} \\
			\iint_\Gamma \big[ \vm_k\cdot \partial_t\vphi + \mU_k : \Grad\vphi) \big] \dx\dt &= \int_{\partial\Gamma} \big[ \vm_0\cdot \vphi\,n_t + (\mU_0\cdot \vphi)\cdot \vn_\vx\big] \dS_{t,\vx} \ed \label{eq:temp27-convint}
		\end{align}
		Since $(\rho_k,\vm_k,\mU_k)\mathop{\rightharpoonup}\limits^\ast (\rho,\vm,\mU)$ in $L^\infty$ as $k\to\infty$, we deduce from \eqref{eq:temp26-convint} and \eqref{eq:temp27-convint} that
		\begin{align} 
			\iint_\Gamma \big[ \rho\,\partial_t\phi + \vm \cdot \Grad\phi) \big] \dx\dt &= \int_{\partial\Gamma} \big[\rho_0\,n_t + \vm_0\cdot \vn_\vx\big]\phi \dS_{t,\vx} \ec \label{eq:temp28-convint} \\
			\iint_\Gamma \big[ \vm\cdot \partial_t\vphi + \mU : \Grad\vphi) \big] \dx\dt &= \int_{\partial\Gamma} \big[ \vm_0\cdot \vphi\,n_t + (\mU_0\cdot \vphi)\cdot \vn_\vx\big] \dS_{t,\vx} \ed \label{eq:temp29-convint}
		\end{align} 
		Note, that \eqref{eq:temp28-convint} is equivalent to \eqref{eq:sol-pde1}.
		
		Next, we're going to show that 
		\begin{equation} \label{eq:temp40-convint}
			(\rho,\vm,\mU)(t,\vx) \ \in \ \sU \qquad \text{ for a.e. }(t,\vx)\in \Gamma \ed
		\end{equation}
		This is done once we have proven that $e\big((\rho,\vm,\mU)(t,\vx)\big) \leq c$ for a.e. $(t,\vx)\in\Gamma$. We proceed analogously to \ref{item:prop-I.d}: Since $(\rho_k,\vm_k\mU_k) \mathop{\rightharpoonup}\limits^\ast (\rho,\vm,\mU)$ in $L^\infty$ and $e$ is convex and continuous (see Lemma~\ref{lemma:e-convex}), we obtain from Lemma~\ref{lemma:not-wls-convex} and the fact that $e\big((\rho_k,\vm_k,\mU_k)(t,\vx)\big) < c$ for all $(t,\vx)\in\Gamma$ and all $k\in \N$ (see \eqref{eq:pfs-subs}), that 
		\begin{align*}
			\iint_{\Gamma_0} \Big[e\big((\rho,\vm,\mU)(t,\vx)\big) - c\Big] \dx\dt &\leq \liminf\limits_{k\to \infty} \iint_{\Gamma_0} \Big[e\big((\rho_k,\vm_k,\mU_k)(t,\vx)\big) - c\Big] \dx\dt  \\
			& \leq 0
		\end{align*}
		for all open and bounded subsets $\Gamma_0\subset \Gamma$. Using Lemma~\ref{lemma:not-local-property} we conclude with $e\big((\rho,\vm,\mU)(t,\vx)\big) \leq c$ for a.e. $(t,\vx)\in\Gamma$ and hence \eqref{eq:temp40-convint}.
		 
		From \eqref{eq:temp40-convint} and \eqref{eq:sol-trace}, which we already proved above, we obtain 
		\begin{align*}
			\frac{n}{2} \lambda_{\max}\left(\frac{\vm\otimes \vm}{\rho} + p(\rho) \id - \mU\right) &= e(\rho,\vm,\mU) \\
			&\leq c \\
			&= \frac{|\vm|^2}{2\rho} + \frac{n}{2} p(\rho) \\
			&= \half \tr\left(\frac{\vm\otimes \vm}{\rho} + p(\rho) \id - \mU\right)
		\end{align*}
		a.e. on $\Gamma$. According to Lemma~\ref{lemma:not-lmax-tr} this yields 
		$$
			\frac{\vm\otimes \vm}{\rho} + p(\rho) \id - \mU = \frac{1}{n} \tr\left(\frac{\vm\otimes \vm}{\rho} + p(\rho) \id - \mU\right) \id = \frac{2c}{n} \id 
		$$
		a.e. on $\Gamma$. In other words $(\rho,\vm,\mU)$ takes values in $K$ for a.e. $(t,\vx)\in \Gamma$. Hence we can replace $\mU$ in \eqref{eq:temp29-convint} to obtain
		\begin{align} \label{eq:temp30-convint}
			&\iint_\Gamma \left[ \vm\cdot \partial_t\vphi + \left(\frac{\vm\otimes\vm}{\rho} + p(\rho) \id - \frac{2c}{n} \id\right) : \Grad\vphi) \right] \dx\dt \notag \\
			&= \int_{\partial\Gamma} \big[ \vm_0\cdot \vphi\,n_t + (\mU_0\cdot \vphi)\cdot \vn_\vx\big] \dS_{t,\vx} \ed
		\end{align}
		Another application of the Divergence Theorem (Proposition~\ref{prop:not-divergence}) shows
		$$
			\iint_\Gamma \frac{2c}{n} \id : \Grad\vphi \dx\dt = \iint_\Gamma \frac{2c}{n} \Div\vphi \dx\dt = \int_{\partial\Gamma} \frac{2c}{n} \vphi\cdot \vn_\vx \dS_{t,\vx} \ed
		$$
		Plugging this into \eqref{eq:temp30-convint}, we find
		\begin{align*}
			&\iint_\Gamma \left[ \vm\cdot \partial_t\vphi + \left(\frac{\vm\otimes\vm}{\rho} + p(\rho) \id \right) : \Grad\vphi) \right] \dx\dt \\
			&= \int_{\partial\Gamma} \left[ \vm_0\cdot \vphi\,n_t + (\mU_0\cdot \vphi)\cdot \vn_\vx + \frac{2c}{n} \vphi \cdot \vn_\vx \right] \dS_{t,\vx} 
		\end{align*}
		and hence \eqref{eq:sol-pde2}. 
	\end{enumerate}
\end{proof}

Finally we state the Perturbation Property, which we prove in Section~\ref{sec:convint-pp}.

\begin{prop}[The Perturbation Property] \label{prop:pert-prop} 
	Let $\ep,\ov{\ep}>0$ and $\Gamma_0\subset \Gamma$ open and bounded. For all $(\rho,\vm)\in X_0$,
	there exists a perturbation $(\rho_\pert,\vm_\pert)\in X_0$ with the following properties:
	\begin{align}
		d\big((\rho_\pert,\vm_\pert),(\rho,\vm)\big) &\leq \ov{\ep} \es \label{eq:pert-prop-1}\\
		I_{\Gamma_0}(\rho_\pert,\vm_\pert) &\geq - \ep \ed \label{eq:pert-prop-2} 
	\end{align} 
\end{prop} 

Roughly speaking, for each $(\rho,\vm)\in X_0$ there exists a perturbation $(\rho_\pert,\vm_\pert)\in X_0$ which is arbitrary close to $(\rho,\vm)$ with respect to $d$ (see \eqref{eq:pert-prop-1}) and, in addition to that, the value of $I_{\Gamma_0}(\rho_\pert,\vm_\pert)$ is almost zero (see \eqref{eq:pert-prop-2}). 

Comparing Proposition~\ref{prop:pert-prop} with the corresponding statement by \name{De~Lellis} and \name{Sz{\'e}kelyhidi} \cite[Proposition 3]{DelSze10}, one observes three differences. First of all our Perturbation Property can be viewed as a compressible version of the one by \name{De~Lellis} and \name{Sz{\'e}kelyhidi}. Second, the functionals $I_{\Gamma_0}$ are integrals over space and time, whereas in \cite{DelSze10} they are only spatial integrals and infima with respect to $t$. This helps \name{De~Lellis} and \name{Sz{\'e}kelyhidi} to achieve an analogue of property \ref{item:convint.c} in Theorem~\ref{thm:convint}, where \eqref{eq:sol-trace} holds for \emph{all} times $t$ and a.e. $\vx$. In this sense our Proposition~\ref{prop:pert-prop} is merely like \cite[Lemma 4.6]{DelSze09}. We will come back to this issue in Chapter~\ref{chap:appl-ibvp}, more precisely in Section~\ref{sec:ibvp-other}. The third difference is that we can achieve that \eqref{eq:pert-prop-2} holds for arbitrary small $\ep>0$ whereas in \cite[Proposition 3]{DelSze10} the right-hand side of the inequality which corresponds to \eqref{eq:pert-prop-2} cannot be arbitrary close to 0 but depends on $I_{\Gamma_0} (\rho,\vm)$.

\subsection{Proof of the Convex-Integration-Theorem} \label{subsec:convint-thm-proof} 

Before we prove Theorem~\ref{thm:convint}, we state and prove a corollary of the Perturbation Property.

\begin{lemma} \label{lemma:X0infinite} 
	The set $X_0$ contains infinitely many elements.
\end{lemma}

\begin{proof}
	First of all, note that $(\rho_0,\vm_0,\mU_0)$ satisfy the properties \ref{item:pfs-pde} - \ref{item:pfs-dens-bdd} of Definition~\ref{defn:X0} and hence $(\rho_0,\vm_0)\in X_0$. This means that $X_0\neq \emptyset$. In order to prove the claim, we construct a sequence $(\rho_k,\vm_k)_{k\in \N_0}\subset X_0$ whose members are pairwise different. Assume that we already have $(\rho_0,\vm_0),...,(\rho_{k-1},\vm_{k-1})\in X_0$ pairwise different for some $k\in \N_0$. Let us construct $(\rho_k,\vm_k)$. Choose an arbitrary non-empty, open and bounded subset $\Gamma_0\subset\Gamma$. Moreover, set 
	$$
		\ep:=\half \min_{j=0,...,k-1} \left( -I_{\Gamma_0}(\rho_j,\vm_j)\right)
	$$ 
	and choose $\ov{\ep}>0$ arbitrary. Lemma~\ref{lemma:prop-I}~\ref{item:prop-I.b} ensures that $\ep>0$. Hence the Perturbation Property (Proposition~\ref{prop:pert-prop}), guarantees existence of $(\rho_k,\vm_k):=(\rho_\pert,\vm_\pert)\in X_0$ satisfying \eqref{eq:pert-prop-1} and \eqref{eq:pert-prop-2}. Assume that there was $i\in\{0,...,k-1\}$ such that $(\rho_k,\vm_k)=(\rho_i,\vm_i)$. Then we obtain from \eqref{eq:pert-prop-2}
	$$
	- I_{\Gamma_0}(\rho_i,\vm_i) = - I_{\Gamma_0}(\rho_k,\vm_k) \leq \ep = \half \min_{j=0,...,k-1} \left( -I_{\Gamma_0}(\rho_j,\vm_j)\right) \leq - \half I_{\Gamma_0}(\rho_i,\vm_i) \ec
	$$
	which is a contradiction since $I_{\Gamma_0}(\rho_i,\vm_i)<0$, see Lemma~\ref{lemma:prop-I}~\ref{item:prop-I.b}. Therefore we conclude that $(\rho_0,\vm_0),...,(\rho_{k},\vm_{k})\in X_0$ are pairwise different.
\end{proof}

Now we are ready to prove Theorem~\ref{thm:convint}. In order to do this, we define a subset $\Xi\subset X$ for which we show two claims. First, on $\Xi$ the assumptions of Lemma~\ref{lemma:prop-I}~\ref{item:prop-I.e} hold. This is shown via contradiction using the Perturbation Property. Second, we prove that $\Xi$ contains infinitely many elements using the facts from Section~\ref{sec:not-baire} together with Lemmas~\ref{lemma:prop-I}~\ref{item:prop-I.a} and \ref{lemma:X0infinite}.

\begin{proof}[Proof of Theorem~\ref{thm:convint}] 
	Let $(\Gamma_j)_{j\in \N}$ be an exhausting sequence of open and bounded subsets of $\Gamma$. Define
	\begin{align*}
		\Xi_j &:= \left\{(\rho,\vm)\in X\,\Big|\, I_{\Gamma_j}\text{ is continuous at }(\rho,\vm)\text{ with respect to }d\right\} \ec \\
		\Xi \hphantom{_j}&:=\bigcap_{j\in\N} \Xi_j \ed 
	\end{align*}
	The proof of Theorem~\ref{thm:convint} is finished as soon as we have shown the following two claims.
	\begin{itemize}
		\item[ ] \textbf{Claim 1:} If $(\rho,\vm)\in \Xi$, then $I_{\Gamma_0}(\rho,\vm)=0$ for all open and bounded $\Gamma_0\subset \Gamma$.  
		\item[ ] \textbf{Claim 2:} The set $\Xi$ contains infinitely many elements.
	\end{itemize}
	Indeed Lemma~\ref{lemma:prop-I}~\ref{item:prop-I.e} together with Claim 1 implies that every $(\rho,\vm)\in \Xi$ is a solution in the sense of Theorem~\ref{thm:convint}, i.e. $(\rho,\vm)$ satisfies properties \ref{item:convint.a} - \ref{item:convint.c} of Theorem~\ref{thm:convint}. Claim~2 says that there are infinitely many such solutions, which proves Theorem~\ref{thm:convint}.
	
	Hence it remains to show that the two claims above are valid.
	
	\medskip
	
	\textbf{Claim 1:} Let us fix $j\in\N$ for a moment. Let us first prove that if $(\rho,\vm)\in \Xi_j$, then $I_{\Gamma_j} (\rho,\vm) = 0$. Assume that this was not true, i.e. there exists $(\rho,\vm) \in \Xi_j$ but $I_{\Gamma_j} (\rho,\vm) < 0$. In particular $(\rho,\vm)\in X$ and since $X$ is the closure of $X_0$ with respect to $d$, there exists a sequence $(\rho_k,\vm_k)_{k\in\N} \subset X_0$ with $(\rho_k,\vm_k) \mathop{\to}\limits^d (\rho,\vm)$ as $k \to \infty$. We may assume that $d\big((\rho_k,\vm_k),(\rho,\vm)\big)\leq \frac{1}{k}$ for all $k\in\N$ by considering a subsequence of $(\rho_k,\vm_k)_{k\in\N}$ if necessary. Now we apply the Perturbation Property (Proposition~\ref{prop:pert-prop}), to each $(\rho_k,\vm_k)$ with $\ep:=-\half I_{\Gamma_j} (\rho,\vm) > 0$, $\ov{\ep}:=\frac{1}{k}$ and $\Gamma_0 := \Gamma_j$. For each $k\in\N$ this proposition yields a pair $(\rho_{k,\pert},\vm_{k,\pert})\in X_0$ with 
	\begin{align}
		d\big((\rho_{k,\pert},\vm_{k,\pert}),(\rho_k,\vm_k)\big) &\leq \frac{1}{k} \ec \label{eq:temp-pp1} \\
		I_{\Gamma_j}(\rho_{k,\pert},\vm_{k,\pert}) &\geq \half I_{\Gamma_j} (\rho,\vm) \ed \label{eq:temp-pp2}
	\end{align}
	Because of  
	$$
		d\big((\rho_{k,\pert},\vm_{k,\pert}),(\rho,\vm)\big) \leq d\big((\rho_{k,\pert},\vm_{k,\pert}),(\rho_k,\vm_k)\big) + d\big((\rho_k,\vm_k),(\rho,\vm)\big) \leq \frac{2}{k} \ec
	$$
	we deduce that $(\rho_{k,\pert},\vm_{k,\pert}) \mathop{\to}\limits^d (\rho,\vm)$ as $k\to \infty$. Since $I_{\Gamma_j}$ is continuous at $(\rho,\vm)$ with respect to $d$ according to the definition of the set $\Xi_j$, we have 
	$$
		\lim_{k\to \infty} I_{\Gamma_j} (\rho_{k,\pert},\vm_{k,\pert}) = I_{\Gamma_j} (\rho,\vm) \ed
	$$
	Together with \eqref{eq:temp-pp2} this implies
	$$
		\half I_{\Gamma_j} (\rho,\vm) \leq \lim_{k\to \infty} I_{\Gamma_j} (\rho_{k,\pert},\vm_{k,\pert}) = I_{\Gamma_j} (\rho,\vm) \ec
	$$
	which contradicts $I_{\Gamma_j} (\rho,\vm) < 0$. Thus 
	\begin{equation} \label{eq:18-temp-convint}
		I_{\Gamma_j} (\rho,\vm) = 0\ed
	\end{equation}
	
	Let now $(\rho,\vm)\in \Xi$ and $\Gamma_0\subset \Gamma$ be any open and bounded subset. From Lemma~\ref{lemma:prop-I}~\ref{item:prop-I.c} we obtain $I_{\Gamma_0}(\rho,\vm)\leq 0$. Assume that $I_{\Gamma_0}(\rho,\vm)< 0$. Then there exists $j\in \N$ such that $\Gamma_0\subset \Gamma_j$ because $(\Gamma_j)_{j\in \N}$ is a exhausting sequence of $\Gamma$. With Lemma~\ref{lemma:prop-I}~\ref{item:prop-I.d} we obtain
	$$
		I_{\Gamma_j} (\rho,\vm) = I_{\Gamma_0} (\rho,\vm) + \iint_{\Gamma_j \setminus\Gamma_0} E(\rho,\vm)\dx\dt <0 \ed
	$$
	which contradicts \eqref{eq:18-temp-convint}. Hence $I_{\Gamma_0}(\rho,\vm)= 0$ for all open and bounded $\Gamma_0\subset \Gamma$. 
	
	\medskip
	
	\textbf{Claim 2:} According to Lemma~\ref{lemma:prop-I}~\ref{item:prop-I.a} the functionals $I_{\Gamma_j}$ are lower semi-continuous with respect to $d$ for each $j\in \N$.
	Furthermore the $I_{\Gamma_j}$ take values in a bounded interval of $\R$. Indeed the upper bound is given by 0 according to Lemma~\ref{lemma:prop-I}~\ref{item:prop-I.c} whereas the lower bound is trivially given by $-c |\Gamma_j|$. In other words $I_{\Gamma_j}$ takes values in the bounded interval $[-c|\Gamma_j|, 0]$. Hence Proposition~\ref{prop:lsc-residual} says, that the sets $\Xi_j$ are residual in $X$. As $\Xi$ is an intersection of countably many residual sets, it is residual itself, see Proposition~\ref{prop:intersection-residual}. 
	Because $\Xi$ is residual in $X$ and $(X,d)$ is a complete metric space, the Baire Category Theorem (Proposition~\ref{prop:Baire}) says that $\Xi$ is dense in $X$. Assume $\Xi$ was finite. Since finite sets are closed (with respect to any metric), this and the density of $\Xi$ imply $\Xi=X$, and hence $X$ is finite. Since $X_0\subset X$, we deduce from Lemma~\ref{lemma:X0infinite} that $X$ contains infinitely many elements, a contradiction. Therefore $\Xi$ contains infinitely many elements. 
\end{proof}

\section{Proof of the Perturbation Property} \label{sec:convint-pp} 

We prove the Perturbation Property in Subsection~\ref{subsec:convint-pp-proof}. The key ingredient of this proof is Lemma~\ref{lemma:K-oscillatory-lemma} below. For the proof of Lemma~\ref{lemma:K-oscillatory-lemma} we need three further lemmas, see Subsection~\ref{subsec:convint-pp-lemmas}.

\subsection{Lemmas for the Proof} \label{subsec:convint-pp-lemmas} 

Let in this subsection $R>r$. See the remark below Lemma~\ref{lemma:K-oscillatory-lemma} for details why we consider $R$ instead of $r$ here. 

To start with the proof of the Perturbation Property we divide the space-time domain $\Gamma_0$ into small pieces, in each of which $(\rho,\vm,\mU)$ is approximated by a constant. We next state and prove several lemmas which look at one of those pieces denoted by $\Gamma^\ast$. The constant which approximates $(\rho,\vm,\mU)$ in $\Gamma^\ast$ is called $(\rho^\ast,\vm^\ast,\mU^\ast)$. We would like to prove the following lemma, which will be the building block in the proof of the Perturbation Property, see Subsection~\ref{subsec:convint-pp-proof}. Later on in the proof of the Perturbation Property, the pieces $\Gamma^\ast$ will be cubes. However for the time being, we assume $\Gamma^\ast$ to be a general open and bounded space-time domain (not necessarily a cube).

\begin{lemma} \label{lemma:K-oscillatory-lemma} 
	For all $\gamma,C>0$ there exists\footnote{Note that $\beta$ may (and will) also depend on $r$, $c$ and $M$.} 
	$\beta=\beta(\gamma,C)>0$ such that the following holds: Let $\Gamma^\ast \subset \R^{1+n}$ be an open and bounded space-time domain and $\ep>0$ such that $\frac{\ep}{|\Gamma^\ast|}= C$. Let furthermore $(\rho^\ast,\vm^\ast,\mU^\ast)\in \interior{\sU}$ with $\rho^\ast>R$ and $e(\rho^\ast,\vm^\ast,\mU^\ast)\leq c-\gamma$. Then there exists a sequence of oscillations 
	\begin{equation} \label{eq:K-oscillations}
		(\til{\rho}_k,\til{\vm}_k,\til{\mU}_k)_{k\in\N} \subset \Cc \big(\Gamma^\ast;\R\times \R^n\times \szn\big)
	\end{equation}
	with the following properties: 
	\begin{enumerate}
		\item \label{item:K-pp.a} The sequence $(\til{\rho}_k,\til{\vm}_k)_{k\in\N}$ converges to $(0,\vz)$ with respect to $d$, i.e. 
		\begin{equation} \label{eq:K-pp-conv}
			(\til{\rho}_k,\til{\vm}_k) \mathop{\to}\limits^d (0,\vz)\qquad \text{ as }k\to \infty\ed
		\end{equation}
	\end{enumerate}
	For each fixed $k\in \N$ the following statements hold:
	\begin{enumerate} \setcounter{enumi}{1}		
		\item \label{item:K-pp.b} The PDEs 
		\begin{align}
			\partial_t \til{\rho}_k + \Div \til{\vm}_k &= 0 \ec \label{eq:K-pp-pde1}\\
			\partial_t \til{\vm}_k + \Div \til{\mU}_k &= \vz \label{eq:K-pp-pde2}
		\end{align}
		hold pointwise for all $(t,\vx)\in \Gamma^\ast$;
		\item \label{item:K-pp.c} The sum of $(\rho^\ast,\vm^\ast,\mU^\ast)$ and the oscillations still takes values in $\interior{\sU}$, more precisely
		\begin{equation}  \label{eq:K-pp-subs}
			e\Big((\rho^\ast,\vm^\ast,\mU^\ast) + (\til{\rho}_k,\til{\vm}_k,\til{\mU}_k)(t,\vx)\Big) \leq c - \beta \qquad \text{ for all }(t,\vx)\in \Gamma^\ast \es 
		\end{equation} 
		\item \label{item:K-pp.d} The density is bounded from below, i.e. 
		\begin{equation} \label{eq:K-pp-dens-bdd} 
			\rho^\ast + \til{\rho}_k(t,\vx) > R \qquad \text{ for all }(t,\vx)\in \Gamma^\ast\es
		\end{equation}
		\item \label{item:K-pp.e} The value of the functional $I_{\Gamma^\ast}$ is close to 0, i.e. 
		\begin{equation} \label{eq:K-pp-impr}
			I_{\Gamma^\ast}(\rho^\ast + \til{\rho}_k,\vm^\ast + \til{\vm}_k) > -\ep \ed 
		\end{equation}
	\end{enumerate} 
\end{lemma}

\begin{rem}
	As already mentioned, in the proof of Propostion \ref{prop:pert-prop} we divide the space-time domain $\Gamma_0$ into small pieces and approximate $(\rho,\vm,\mU)$ in each of this pieces by a constant. In Lemma~\ref{lemma:K-oscillatory-lemma} $\Gamma^\ast$ is such a piece and $(\rho^\ast,\vm^\ast,\mU^\ast)$ is the corresponding constant. Since $(\rho^\ast,\vm^\ast,\mU^\ast)$ is merely an approximation of $(\rho,\vm,\mU)$, we need not only that the left-hand side of \eqref{eq:K-pp-subs} is $<c$, but it must be bounded away from $c$ where this bound must not depend on the piece under consideration. This the reason why we have to deal with $\beta$ in \eqref{eq:K-pp-subs}, where $\beta$ does not depend on $\Gamma^\ast$. For the same reason we work with $R$ instead of $r$ in \eqref{eq:K-pp-dens-bdd}. We refer to Subsection~\ref{subsec:convint-pp-proof} for more details.
\end{rem}

We prove Lemma~\ref{lemma:K-oscillatory-lemma} after we have considered the following three lemmas.

As a first step in the proof of Lemma~\ref{lemma:K-oscillatory-lemma}, we have to find a family 
$$
\big\{\big(\tau_i,(\rho_i,\vm_i,\mU_i)\big)\big\}_{i=1,...,N}
$$ 
with some properties:

\begin{lemma} \label{lemma:existence-family}
	Let $(\rho^\ast,\vm^\ast,\mU^\ast)\in \interior{\sU}$ with $\rho^\ast>R$. Then there exist $N\in\N$ with $N\geq 2$ and $\big(\tau_i,(\rho_i,\vm_i,\mU_i)\big)\in \R^+ \times K$ for $i=1,...,N$ such that 
	\begin{itemize}
		\item $\rho_i>R$ for all $i=1,...,N$,
		\item the family $\big\{\big(\tau_i,(\rho_i,\vm_i,\mU_i)\big)\big\}_{i=1,...,N}$ satisfies the $H_N$-condition and
		\item $(\rho^\ast,\vm^\ast,\mU^\ast) = \sum_{i=1}^N \tau_i (\rho_i,\vm_i,\mU_i) $.
	\end{itemize}
\end{lemma} 

Keeping the definition of $\sU$ \eqref{eq:U} in mind, the proof is quite easy.

\begin{proof}
	Since $(\rho^\ast,\vm^\ast,\mU^\ast)\in \interior{\sU}\subset \sU$ there exist $N\in\N$ and $\big(\tau_i,(\rho_i,\vm_i,\mU_i)\big)\in \R^+ \times K$ for $i=1,...,N$ such that 
	\begin{itemize}
		\item the family $\big\{\big(\tau_i,(\rho_i,\vm_i,\mU_i)\big)\big\}_{i=1,...,N}$ satisfies the $H_N$-condition and
		\item $(\rho^\ast,\vm^\ast,\mU^\ast) = \sum_{i=1}^N \tau_i (\rho_i,\vm_i,\mU_i) $,
	\end{itemize}
	according to the definition of $\sU$, see \eqref{eq:U}.	Assume that $N=1$, i.e. 
	$$
		(\rho^\ast,\vm^\ast,\mU^\ast) = (\rho_1,\vm_1,\mU_1)\ \in\ K\ed
	$$ 
	But according to \eqref{eq:boundaryU} this implies $e(\rho^\ast,\vm^\ast,\mU^\ast)=c$ which contradicts $(\rho^\ast,\vm^\ast,\mU^\ast) \in \interior{\sU}$. Hence $N\geq 2$. 
	
	It remains to show that we can achieve that $\rho_i>R$ for all $i=1,...,N$. From Propositions~\ref{prop:compKLambda} and \ref{prop:Kstar}~\ref{item:Kstar.b} we deduce that $(\rho^\ast,\vm^\ast,\mU^\ast)\in (K\cap\{\rho=\rho^\ast\})^\Lambda$. This implies together with Proposition~\ref{prop:KLambda=U} that the $(\rho_i,\vm_i,\mU_i)$ can be chosen to lie in $K\cap\{\rho=\rho^\ast\}$. Hence $\rho_i=\rho^\ast>R$ for all $i=1,...,N$.
\end{proof} 

The fact that the $(\rho_i,\vm_i,\mU_i)$ given by Lemma~\ref{lemma:existence-family} lie in $K$ rather than $\interior{\sU}$ is an obstacle. However the $(\rho_i,\vm_i,\mU_i)$ can be slightly perturbed to obtain $(\hat{\rho}_i,\hat{\vm}_i,\hat{\mU}_i)\in \interior{\sU}$ with similar properties. This is the content of the following lemma.

\begin{lemma} \label{lemma:UstattK} 
	Let $(\rho^\ast,\vm^\ast,\mU^\ast)\in \interior{\sU}$ with $\rho^\ast>R$ and furthermore $N\in\N$ with $N\geq 2$, $\big(\tau_i,(\rho_i,\vm_i,\mU_i)\big)\in \R^+ \times K$ for $i=1,...,N$ such that 
	\begin{itemize}
		\item $\rho_i>R$ for all $i=1,...,N$,
		\item the family $\big\{\big(\tau_i,(\rho_i,\vm_i,\mU_i)\big)\big\}_{i=1,...,N}$ satisfies the $H_N$-condition and
		\item $(\rho^\ast,\vm^\ast,\mU^\ast) = \sum_{i=1}^N \tau_i (\rho_i,\vm_i,\mU_i) $.
	\end{itemize}
	Define a family $\big\{(\hat{\rho}_i,\hat{\vm}_i,\hat{\mU}_i)\big\}_ {i=1,...,N}$ by 
	\begin{equation} \label{eq:defn-hat}
		(\hat{\rho}_i,\hat{\vm}_i,\hat{\mU}_i) := \tau (\rho^\ast,\vm^\ast,\mU^\ast) + (1-\tau) (\rho_i,\vm_i,\mU_i) 
	\end{equation}
	with fixed $\tau\in (0,1)$. Then the following statements hold:
	\begin{enumerate}
		\item $\hat{\rho}_i>R$ for all $i=1,...,N$;
		\item $(\hat{\rho}_i,\hat{\vm}_i,\hat{\mU}_i)\in \interior{\sU}$ for all $i=1,...,N$;
		\item The family $\big\{\big(\tau_i,(\hat{\rho}_i,\hat{\vm}_i,\hat{\mU}_i)\big)\big\}_{i=1,...,N}$ satisfies the $H_N$-condition;
		\item $(\rho^\ast,\vm^\ast,\mU^\ast) = \sum_{i=1}^N \tau_i (\hat{\rho}_i,\hat{\vm}_i,\hat{\mU}_i)$.
	\end{enumerate}
\end{lemma}

The proof is straightforward.

\begin{proof}
	\begin{enumerate}
		\item For $i=1,...,N$ we easily see, that 
		$$
			\hat{\rho}_i = \tau \rho^\ast + (1-\tau) \rho_i > \tau R + (1-\tau) R = R\ed
		$$
		\item The fact that $(\hat{\rho}_i,\hat{\vm}_i,\hat{\mU}_i)\in \interior{\sU}$ for all $i=1,...,N$ follows from the convexity of $e$, see Lemma~\ref{lemma:e-convex}. Indeed we find
		$$
			e(\hat{\rho}_i,\hat{\vm}_i,\hat{\mU}_i) \leq \tau e(\rho^\ast,\vm^\ast,\mU^\ast) + (1-\tau) e(\rho_i,\vm_i,\mU_i) < c \ec
		$$
		because $(\rho^\ast,\vm^\ast,\mU^\ast)\in\interior{\sU}$ and $(\rho_i,\vm_i,\mU_i)\in K$, which mean that $e(\rho^\ast,\vm^\ast,\mU^\ast)<c$ and $e(\rho_i,\vm_i,\mU_i)=c$ respectively.
		
		\item We proceed by induction over $N\geq 2$. To begin with, let $N=2$. Since the family 
		$$
			\Big\{\big(\tau_1,(\rho_1,\vm_1,\mU_1)\big),\big(\tau_2,(\rho_2,\vm_2,\mU_2)\big)\Big\}
		$$	
		satisfies the $H_2$-condition, we have $(\rho_2,\vm_2,\mU_2)-(\rho_1,\vm_1,\mU_1)\in \Lambda$ and $\tau_1+\tau_2=1$. The former implies
		\begin{align}
			&(\hat{\rho}_2,\hat{\vm}_2,\hat{\mU}_2)-(\hat{\rho}_1,\hat{\vm}_1,\hat{\mU}_1) \notag \\
			&= \tau (\rho^\ast,\vm^\ast,\mU^\ast) + (1-\tau) (\rho_2,\vm_2,\mU_2) - \tau (\rho^\ast,\vm^\ast,\mU^\ast) - (1-\tau) (\rho_1,\vm_1,\mU_1) \notag \\
			&= (1-\tau) \Big( (\rho_2,\vm_2,\mU_2)-(\rho_1,\vm_1,\mU_1) \Big) \ \in\  \Lambda\ed \label{eq:20-temp-convint}
		\end{align}
		Together with $\tau_1+\tau_2=1$ this means that $\big\{\big(\tau_1,(\hat{\rho}_1,\hat{\vm}_1,\hat{\mU}_1)\big),\big(\tau_2,(\hat{\rho}_2,\hat{\vm}_2,\hat{\mU}_2)\big)\big\}$ satisfy the $H_2$-condition as well.
		
		If $N>2$, the fact that $\big\{\big(\tau_i,(\rho_i,\vm_i,\mU_i)\big)\big\}_{i=1,...,N}$ satisfies the $H_N$-condition, implies (after relabeling if necessary) that $(\rho_2,\vm_2,\mU_2)-(\rho_1,\vm_1,\mU_1)\in \Lambda$ and
		$$
			\left\{ \left( \tau_1 + \tau_2 , \frac{\tau_1}{\tau_1 + \tau_2} (\rho_1,\vm_1,\mU_1) + \frac{\tau_2}{\tau_1 + \tau_2} (\rho_2,\vm_2,\mU_2)\right) \right\} \cup \big\{\big(\tau_i,(\rho_i,\vm_i,\mU_i)\big)\big\}_{i=3,...,N}
		$$
		satisfies the $H_{N-1}$-condition. Note first that $(\hat{\rho}_2,\hat{\vm}_2,\hat{\mU}_2)-(\hat{\rho}_1,\hat{\vm}_1,\hat{\mU}_1)\in \Lambda$ which can be shown exactly as in \eqref{eq:20-temp-convint}. Second, the induction hypothesis says that 
		\begin{align}
			&\left\{ \left( \tau_1 + \tau_2 , \tau (\rho^\ast,\vm^\ast,\mU^\ast) + (1-\tau) \frac{\tau_1}{\tau_1 + \tau_2} (\rho_1,\vm_1,\mU_1) + (1-\tau) \frac{\tau_2}{\tau_1 + \tau_2} (\rho_2,\vm_2,\mU_2) \right) \right\} \notag \\
			&\qquad\qquad \cup \big\{\big(\tau_i,(\hat{\rho}_i,\hat{\vm}_i,\hat{\mU}_i)\big)\big\}_{i=3,...,N} \notag \\
			&=\left\{ \left( \tau_1 + \tau_2 , \frac{\tau_1}{\tau_1 + \tau_2} (\hat{\rho}_1,\hat{\vm}_1,\hat{\mU}_1) + \frac{\tau_2}{\tau_1 + \tau_2} (\hat{\rho}_2,\hat{\vm}_2,\hat{\mU}_2)\right) \right\} \cup \big\{\big(\tau_i,(\hat{\rho}_i,\hat{\vm}_i,\hat{\mU}_i)\big)\big\}_{i=3,...,N} \label{eq:19-temp-convint}
		\end{align}
		satisfies the $H_{N-1}$-condition. Hence $\big\{\big(\tau_i,(\hat{\rho}_i,\hat{\vm}_i,\hat{\mU}_i)\big)\big\}_{i=1,...,N}$ satisfies the $H_N$-condi\-tion. 
		\item A simple calculation yields
		\begin{align*}
			\sum_{i=1}^N \tau_i (\hat{\rho}_i,\hat{\vm}_i,\hat{\mU}_i) &= \sum_{i=1}^N \tau_i \Big(\tau (\rho^\ast,\vm^\ast,\mU^\ast) + (1-\tau) (\rho_i,\vm_i,\mU_i)\Big) \\
			&= \tau (\rho^\ast,\vm^\ast,\mU^\ast) \sum_{i=1}^N \tau_i + (1-\tau) \sum_{i=1}^N \tau_i (\rho_i,\vm_i,\mU_i)\\
			&= \tau (\rho^\ast,\vm^\ast,\mU^\ast) + (1-\tau ) (\rho^\ast,\vm^\ast,\mU^\ast) \\
			&= (\rho^\ast,\vm^\ast,\mU^\ast) \ed
		\end{align*}
	\end{enumerate}
\end{proof}

The key to prove Lemma~\ref{lemma:K-oscillatory-lemma} is the following lemma. Here convex integration is implemented. The suitable oscillations are constructed using the differential operators studied in Section~\ref{sec:convint-operators}. 

\begin{lemma} \label{lemma:U-oscillatory-lemma} 
	Let $\Gamma^\ast \subset \R^{1+n}$ be an open and bounded set (not necessarily connected) and $\ep>0$. Let furthermore $(\rho^\ast,\vm^\ast,\mU^\ast)\in \interior{\sU}$ with $\rho^\ast>R$, $N\in\N$ with $N\geq 2$, and $\big(\tau_i,(\rho_i,\vm_i,\mU_i)\big)\in \R^+ \times \interior{\sU}$ for $i=1,...,N$ such that 
	\begin{itemize}
		\item $\rho_i>R$ for all $i=1,...,N$,
		\item the family $\big\{\big(\tau_i,(\rho_i,\vm_i,\mU_i)\big)\big\}_{i=1,...,N}$ satisfies the $H_N$-condition and
		\item $(\rho^\ast,\vm^\ast,\mU^\ast) = \sum_{i=1}^N \tau_i (\rho_i,\vm_i,\mU_i) $.
	\end{itemize}
	Then there exists a sequence of oscillations 
	\begin{equation} \label{eq:U-oscillations}
		(\til{\rho}_k,\til{\vm}_k,\til{\mU}_k)_{k\in\N}\subset \Cc \big(\Gamma^\ast;\R\times \R^n\times \szn\big)
	\end{equation} 
	with the following properties:
	\begin{enumerate}
		\item \label{item:U-pp.a} The sequence $(\til{\rho}_k,\til{\vm}_k)_{k\in \N}$ converges to $(0,\vz)$ with respect to $d$, i.e. 
		\begin{equation} \label{eq:U-pp-conv}
			(\til{\rho}_k,\til{\vm}_k) \mathop{\to}\limits^d (0,\vz) \qquad \text{ as }k\to \infty \ed
		\end{equation}
	\end{enumerate}
	For each fixed $k\in \N$ the following statements hold:
	\begin{enumerate} \setcounter{enumi}{1}
		\item \label{item:U-pp.b} The PDEs 
		\begin{align}
			\partial_t \til{\rho}_k + \Div \til{\vm}_k &= 0 \ec \label{eq:U-pp-pde1}\\
			\partial_t \til{\vm}_k + \Div \til{\mU}_k &= \vz \label{eq:U-pp-pde2}
		\end{align}
		hold pointwise for all $(t,\vx)\in \Gamma^\ast$;
		
		\item \label{item:U-pp.c} There exist open sets $\Gamma_i \subsetcomp \Gamma^\ast$ for all $i = 1,...,N$, such that 
		\begin{align}
			&\bullet \ \text{the }\ov{\Gamma_i}\text{ are pairwise disjoint and } \label{eq:U-pp-gdis} \\
			&\bullet \ (\rho^\ast,\vm^\ast,\mU^\ast) + (\til{\rho}_k,\til{\vm}_k,\til{\mU}_k)(t,\vx) = (\rho_i,\vm_i,\mU_i)\qquad \text{ for all } (t,\vx)\in \Gamma_i \es  \label{eq:U-pp-gconst}
		\end{align}
		
		\item \label{item:U-pp.d} The sum of $(\rho^\ast,\vm^\ast,\mU^\ast)$ and the oscillations still takes values in $\interior{\sU}$, more precisely we have\footnote{Note that $\max_{i=1,...,N} e(\rho_i,\vm_i,\mU_i)<c$ since $(\rho_i,\vm_i,\mU_i)\in\interior{\sU}$ for all $i=1,...,N$. Hence the right-hand side of \eqref{eq:U-pp-subs} is $<c$.} 
		\begin{equation}  \label{eq:U-pp-subs}
			e\Big((\rho^\ast,\vm^\ast,\mU^\ast) + (\til{\rho}_k,\til{\vm}_k,\til{\mU}_k)(t,\vx)\Big) \leq \half\Big(\max_{i=1,...,N} e(\rho_i,\vm_i,\mU_i) + c \Big)
		\end{equation} 
		for all $(t,\vx)\in \Gamma^\ast$;
		
		\item \label{item:U-pp.e} The density is bounded from below, i.e. 
		\begin{equation} \label{eq:U-pp-dens-bdd} 
			\rho^\ast + \til{\rho}_k(t,\vx) > R \qquad \text{ for all }(t,\vx)\in \Gamma^\ast\es
		\end{equation}
		
		\item \label{item:U-pp.f} The value of the functional $I_{\Gamma^\ast}$ can be estimated as follows:
		\begin{equation}  \label{eq:U-pp-impr}
			I_{\Gamma^\ast}(\rho^\ast + \til{\rho}_k,\vm^\ast + \til{\vm}_k) > -\ep + \sum_{i = 1}^N I_{\Gamma_i}(\rho_i,\vm_i) \ed 
		\end{equation}
	\end{enumerate}
\end{lemma} 

\begin{rem} 
	The reader might have noticed that $\Gamma^\ast$ in Lemma~\ref{lemma:K-oscillatory-lemma} is a domain, i.e. connected, whereas $\Gamma^\ast$ in Lemma~\ref{lemma:U-oscillatory-lemma} is not necessarily connected. Note that we actually need Lemma~\ref{lemma:U-oscillatory-lemma} only for connected $\Gamma^\ast$. However in order to prove Lemma~\ref{lemma:U-oscillatory-lemma}, we proceed by induction over $N\in \N$, $N\geq 2$, and in the induction step we have to deal with a not necessarily connected $\Gamma^\ast$, see also Example~\ref{ex:CI}.
\end{rem} 
\begin{proof} 
	We prove the proposition by induction over $N\in \N$, $N\geq 2$. 
	
	\medskip
	
	\textbf{Induction basis:} Let first $N=2$. Then by assumption we have 
	$$
		(\rho^\ast,\vm^\ast,\mU^\ast) = \tau_1 (\rho_1,\vm_1,\mU_1) + \tau_2 (\rho_2,\vm_2,\mU_2)
	$$ 
	where $ (\rho_2,\vm_2,\mU_2) - (\rho_1,\vm_1,\mU_1) \in \Lambda$, $\tau_1 + \tau_2 = 1$, and $\rho_1,\rho_2>R$. Set
	$$
		(\ov{\rho},\ov{\vm},\ov{\mU}) := (\rho_2,\vm_2,\mU_2) - (\rho_1,\vm_1,\mU_1) \ed
	$$
	According to the definition of the wave cone $\Lambda$ (see \eqref{eq:wavecone-kernel}) there exists $\veta\in \R^{1+n}\setminus\{\vz\}$ such that 
	\begin{equation} \label{eq:6-temp-convint}
	\left(\begin{array}{cc}
		\ov{\rho} & \ov{\vm}^\trans \\
		\ov{\vm} & \ov{\mU}  
	\end{array}\right) \veta = \vz \ed
	\end{equation}
	Now we fix an open set $\til{\Gamma}\subsetcomp \Gamma^\ast$ in such a way that for the ``frame'' $\Gamma_\tf := \Gamma^\ast \setminus \til{\Gamma}$ 
	\begin{equation} \label{eq:smallness-frame} 
		| \Gamma_\tf | < \frac{\ep}{2c} \ec
	\end{equation}
	see Figure~\ref{fig:CI-frame}.
	
	\begin{figure}[tb] 
		\centering
		\includegraphics[width=0.7\textwidth]{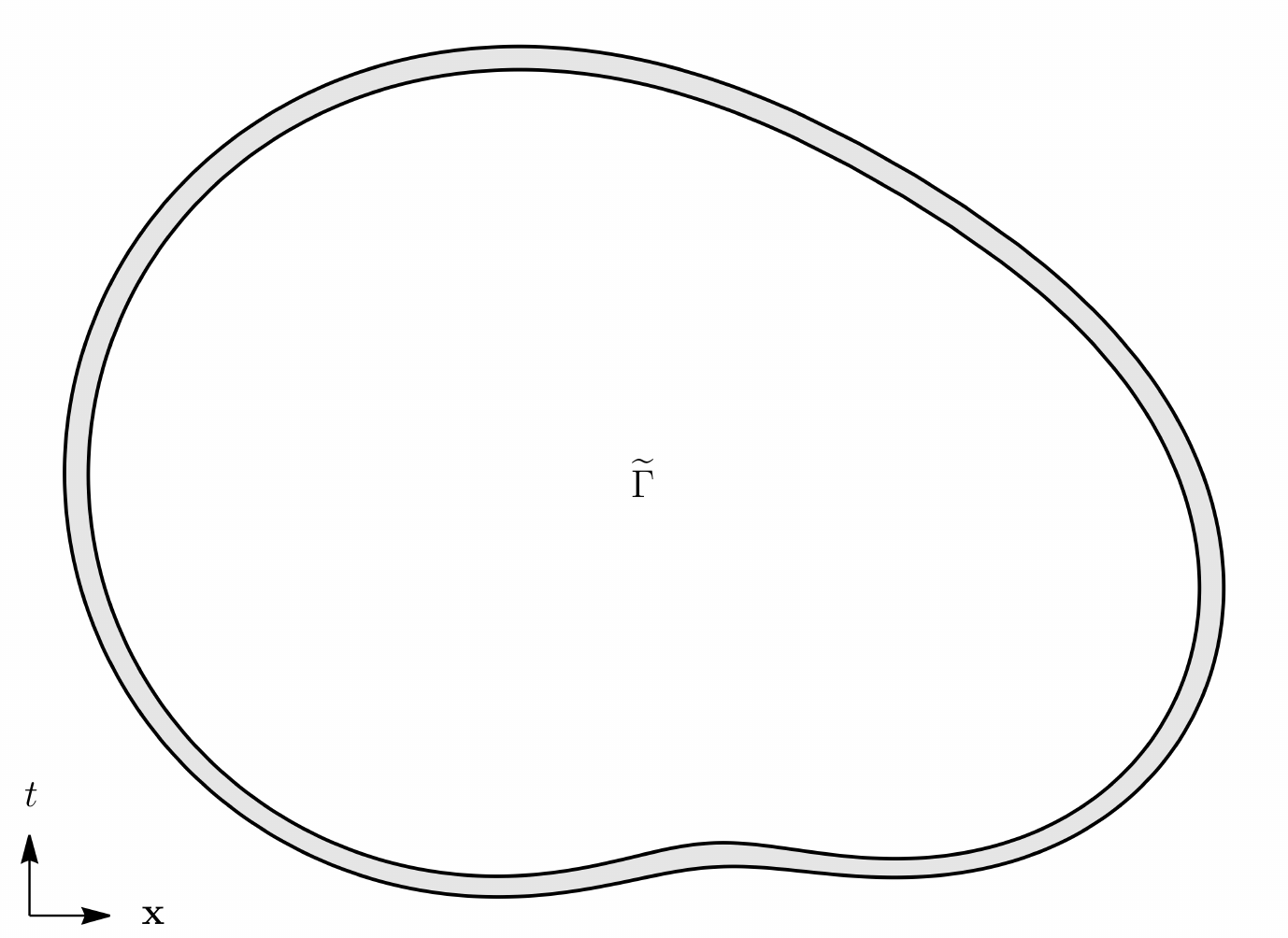} 
		\caption{An example of $\Gamma^\ast$ and $\til{\Gamma} \subsetcomp\Gamma^\ast$ with the corresponding frame $\Gamma_\tf = \Gamma^\ast \setminus \til{\Gamma}$ (gray).} 
		\label{fig:CI-frame}
	\end{figure}
	
	Let $Q\in \R^{1+n}$ be an open cube such that $\til{\Gamma} \subsetcomp Q$ and one edge of $Q$ is parallel to $\veta$. Furthermore we require the edge length of $Q$ to be a natural multiple of $\frac{1}{|\veta|}$. Fix 
	\begin{equation} \label{eq:smallness-delta}
		0<\delta< \min\left\{\frac{\ep}{8c |Q| },\frac{\tau_1}{2} , \frac{\tau_2}{2} \right\} \ed
	\end{equation} 
	The reason for choosing $\delta$ in this way will become clear later, see \eqref{eq:50-temp-convint}. For an example of $\til{\Gamma}$ and $Q$ see also Figure~\ref{fig:CI-cube-around}.
	
	Let $\Phi\in \Cc(\Gamma^\ast ; [0,1])$ be a smooth cutoff function with $\Phi\equiv 1 $ on $\til{\Gamma}$. Let furthermore $f:\R \to \R$ be defined as
	$$
	f(t) := \left\{ \begin{array}{rl}
		-\tau_2 & \text{ if } t\in[0,\tau_1) + \Z \ec \\
		\tau_1 & \text{ if } t\in[\tau_1,1) + \Z  \ed
	\end{array} \right. 
	$$
	Note that $-\tau_2 = \tau_1 -1 $ and hence $f$ is of the form \eqref{eq:periodic-function} in Subsection~\ref{subsec:not-periodic-functions}. 
	
	Next we mollify $f$ to obtain $f_\delta \in C^\infty (\R)$, which is periodic with zero mean, takes values in the interval $[-\tau_2,\tau_1 ]$ and satisfies 
	\begin{align}
		f_\delta (t) &= -\tau_2 & &\hspace{-2.5cm}\text{ for all }t\in [\delta, \tau_1- \delta] + \Z \ec \label{eq:fdelta1}\\
		f_\delta (t) &= \tau_1 & &\hspace{-2.5cm}\text{ for all }t\in [\tau_1+\delta , 1-\delta] + \Z \ec \label{eq:fdelta2}
	\end{align}
	according to Lemma~\ref{lemma:periodic-mollification}. Because of Lemma~\ref{lemma:not-periodic-primitive} there exists $h\in C^\infty(\R)$ such that $h'''=f_\delta$ and $h$, as well as all derivatives of $h$ up to third order, are bounded. Next we define $g_k\in  C^\infty(\R^{1+n})$ by 
	$$
		g_k(t,\vx) := \frac{1}{k^3} h\big( k (t,\vx)\cdot \veta\big)
	$$
	for $k\in\N$. Then the product $g_k\Phi\in \Cc(\Gamma^\ast)$. Furthermore we set 
	\begin{equation} \label{eq:defn-perturbation}
		(\til{\rho}_k,\til{\vm}_k,\til{\mU}_k) := \big(\opL_\rho[g_k\Phi], \opL_\vm[g_k\Phi],\opL_\mU[g_k\Phi]\big) \ec
	\end{equation}
	where the operators $\opL_\rho$, $\opL_\vm$ and $\opL_\mU$ are given by Proposition~\ref{prop:operators} (note that the assumption \eqref{eq:op-kernel} of Proposition~\ref{prop:operators} holds according to \eqref{eq:6-temp-convint}). 
	
	In order to obtain the desired properties of the oscillations $(\til{\rho}_k,\til{\vm}_k,\til{\mU}_k)$ defined in \eqref{eq:defn-perturbation}, we have to choose $k$ sufficiently large. Let us now specify how large $k$ must be, i.e. let us find $k_{\min}\in \N$ such that $(\til{\rho}_k,\til{\vm}_k,\til{\mU}_k)$ satisfies the desired properties for each $k\geq k_{\min}$. 
	
	Since $\opL_j$ for $j\in\{\rho,\vm,\mU\}$ are third order homogeneous differential operators according to Proposition~\ref{prop:operators}, the expression
	$$
	\opL_j [g_k\Phi] - \opL_j [g_k] \Phi
	$$ 
	can be written as a sum of pruducts of second derivatives of $g_k$ and first derivatives of $\Phi$, of first derivatives of $g_k$ and second derivatives of $\Phi$ and of $g_k$ and third derivatives of $\Phi$. Thus we deduce that 
	\begin{equation} \label{eq:2-temp-convint}
	\max_{j\in\{\rho,\vm,\mU\}}\Big\| \opL_j [g_k\Phi] - \opL_j [g_k] \Phi\Big\|_\infty  \leq \frac{C}{k} 
	\end{equation}
	for a suitable constant $C$ since $h$ and the occuring derivatives of $h$ are bounded. 
	
	According to Proposition~\ref{prop:operators}, we have
	\begin{align}
		\big(\opL_\rho[g_k], \opL_\vm[g_k],\opL_\mU[g_k]\big)(t,\vx) &= (\ov{\rho},\ov{\vm},\ov{\mU}) h'''\big(k (t,\vx)\cdot \veta\big) \notag\\
		&= (\ov{\rho},\ov{\vm},\ov{\mU}) f_\delta\big(k (t,\vx)\cdot \veta\big) \ed \label{eq:5-temp-convint}
	\end{align}
	The fact that $f_\delta$ takes values in $[-\tau_2,\tau_1]$ and $\Phi$ in $[0,1]$, implies that 
	$$
		f_\delta\big(k (t,\vx)\cdot \veta\big) \Phi(t,\vx) \ \in\  [-\tau_2,\tau_1] \qquad \text{ for all }(t,\vx)\in \Gamma^\ast \ed
	$$ 
	In other words for each $(t,\vx)\in \Gamma^\ast$ there exists $\tau\in[-\tau_2,\tau_1]$ such that $f_\delta\big(k (t,\vx)\cdot \veta\big) \Phi(t,\vx) = \tau$. A simple computation shows that 
	\begin{align}
		&(\rho^\ast,\vm^\ast,\mU^\ast) + \tau (\ov{\rho},\ov{\vm},\ov{\mU}) \notag \\
		&= \tau_1 (\rho_1,\vm_1,\mU_1) + \tau_2 (\rho_2,\vm_2,\mU_2) + \tau (\rho_2,\vm_2,\mU_2) - \tau (\rho_1,\vm_1,\mU_1) \notag \\
		&= (\tau_1 - \tau) (\rho_1,\vm_1,\mU_1) + (\tau_2 + \tau) (\rho_2,\vm_2,\mU_2) \ed \label{eq:22-temp-convint}
	\end{align}
	Hence we obtain for each $(t,\vx)\in \Gamma^\ast$
	$$
		\rho^\ast + \opL_\rho[g_k](t,\vx) \, \Phi(t,\vx) = (\tau_1 - \tau) \rho_1 + (\tau_2 + \tau) \rho_2 > R
	$$
	where we used $\rho_1, \rho_2>R$ and the fact that $\tau_1+\tau_2 = 1$. Using \eqref{eq:2-temp-convint} we deduce that there exists $k_1\in \N$ such that
	\begin{equation} \label{eq:8-temp-convint}
		\rho^\ast + \opL_\rho[g_k\Phi](t,\vx) >R \qquad\text{ for all }(t,\vx)\in \Gamma^\ast
	\end{equation}
	if $k\geq k_1$.
	
	Furthermore \eqref{eq:5-temp-convint}, \eqref{eq:22-temp-convint}, the convexity of $e$ and the fact that $\tau_1+\tau_2 = 1$ imply
	\begin{align}
		&e\Big((\rho^\ast,\vm^\ast,\mU^\ast) + \big(\opL_\rho[g_k]\Phi, \opL_\vm[g_k]\Phi,\opL_\mU[g_k]\Phi\big)(t,\vx)\Big) \notag \\
		&= e\Big((\tau_1 - \tau) (\rho_1,\vm_1,\mU_1) + (\tau_2 + \tau) (\rho_2,\vm_2,\mU_2)\Big) \notag \\
		& \leq (\tau_1 - \tau) e(\rho_1,\vm_1,\mU_1) + (\tau_2 + \tau) e(\rho_2,\vm_2,\mU_2) \notag \\
		& \leq \max_{i=1,2} e(\rho_i,\vm_i,\mU_i) \label{eq:7-temp-convint} \\
		& < c \notag \ec
	\end{align}
	which shows that 
	$$
		(\rho^\ast,\vm^\ast,\mU^\ast) + \big(\opL_\rho[g_k]\Phi, \opL_\vm[g_k]\Phi,\opL_\mU[g_k]\Phi\big)(t,\vx) \ \in\ \interior{\sU} \qquad \text{ for all }(t,\vx)\in \Gamma^\ast \ed
	$$
	This implies together with \eqref{eq:2-temp-convint} that
	$$
		(\rho^\ast,\vm^\ast,\mU^\ast) + \big(\opL_\rho[g_k\Phi], \opL_\vm[g_k\Phi],\opL_\mU[g_k\Phi]\big)(t,\vx) \ \in\ \interior{\sU} \qquad \text{ for all }(t,\vx)\in \Gamma^\ast \ec
	$$
	too, as long as $k$ is sufficently large, say $k\geq k_2$ with a suitable $k_2\in\N$. Due to Lemma~\ref{lemma:U-bdd} and the fact that $R>r$, we deduce that both
	\begin{align*}
		(\rho^\ast,\vm^\ast,\mU^\ast) + \big(\opL_\rho[g_k]\Phi, \opL_\vm[g_k]\Phi,\opL_\mU[g_k]\Phi\big)(t,\vx) &\ \in\  [r,M]\times \closure{B_n(\vz,M)}\times \closure{B_{\szn}(\mZ,M)} \ec \\
		(\rho^\ast,\vm^\ast,\mU^\ast) + \big(\opL_\rho[g_k\Phi], \opL_\vm[g_k\Phi],\opL_\mU[g_k\Phi]\big)(t,\vx) &\ \in\  [r,M]\times \closure{B_n(\vz,M)}\times \closure{B_{\szn}(\mZ,M)}  
	\end{align*}
	for all $(t,\vx)\in \Gamma^\ast$ as long as $k\geq \max\{k_1,k_2\}$. Now \eqref{eq:2-temp-convint} and the uniform continuity of $e$ (see Lemma~\ref{lemma:prop-E}~\ref{item:prop-E.a}) yield existence of $k_{\min}\geq \max\{k_1,k_2\}$ such that 
	\begin{align}
		&\bigg|e\Big((\rho^\ast,\vm^\ast,\mU^\ast) + \big(\opL_\rho[g_k\Phi], \opL_\vm[g_k\Phi],\opL_\mU[g_k\Phi]\big)(t,\vx)\Big) \notag \\
		&\qquad - e\Big((\rho^\ast,\vm^\ast,\mU^\ast) + \big(\opL_\rho[g_k]\Phi, \opL_\vm[g_k]\Phi,\opL_\mU[g_k]\Phi\big)(t,\vx)\Big)\bigg| \notag \\
		&\leq \half \Big( c - \max_{i=1,2} e(\rho_i,\vm_i,\mU_i) \Big) \label{eq:23-temp-convint}
	\end{align}	
	for all $(t,\vx)\in\Gamma^\ast$ as long as $k\geq k_{\min}$. 
	
	We claim that the sequence\footnote{Note, that we can simply redefine $k$ to obtain a sequence $(\til{\rho}_k,\til{\vm}_k,\til{\mU}_k)_{k\in\N}$.} $(\til{\rho}_k,\til{\vm}_k,\til{\mU}_k)_{k\geq k_{\min}}$ has the desired properties, which shall be proven in the sequel.
	
	Keeping \eqref{eq:defn-perturbation} in mind, we have $(\til{\rho}_k,\til{\vm}_k,\til{\mU}_k)\in C^\infty (\R^{1+n};\R\times \R^n \times \szn)$ according to Proposition~\ref{prop:operators} and, since $\Phi\in \Cc (\Gamma^\ast)$, we even have 
	$$
		(\til{\rho}_k,\til{\vm}_k,\til{\mU}_k) \ \in\  \Cc \big(\Gamma^\ast;\R\times \R^n \times \szn\big)
	$$ 
	for all $k\geq k_{\min}$.
	
	Let us turn our attention towards the properties \ref{item:U-pp.a} - \ref{item:U-pp.f}:
	\begin{enumerate}
		\item For any fixed $\varphi\in L^1(\R^{1+n})$ we obtain 
		\begin{align*}
			&\left| \iint_{\R^{1+n}} \big(\til{\rho}_k,\til{\vm}_k\big) \varphi \dt\dx \right| \\
			&= \left| \iint_{\R^{1+n}} \Big(\opL_{\rho}[g_k\Phi],\opL_{\vm}[g_k\Phi]\Big) \varphi \dt\dx \right| \\
			&= \bigg| \iint_{\R^{1+n}} \Big(\opL_{\rho}[g_k\Phi] - \opL_{\rho}[g_k] \Phi,\opL_{\vm}[g_k\Phi] - \opL_{\vm}[g_k] \Phi\Big) \varphi \dt\dx \\
			&\qquad + \iint_{\R^{1+n}} \Big(\opL_{\rho}[g_k] \Phi,\opL_{\vm}[g_k] \Phi\Big) \varphi \dt\dx \bigg|\\
			&\leq \Big\| \Big(\opL_{\rho}[g_k\Phi] - \opL_{\rho}[g_k] \Phi,\opL_{\vm}[g_k\Phi] - \opL_{\vm}[g_k] \Phi\Big) \Big\|_{L^\infty} \big\|\varphi\big\|_{L^1} \\
			&\qquad + \left| \iint_{\R^{1+n}} \Big(\ov{\rho}\Phi f_\delta\big(k(t,\vx)\cdot\veta\big),\ov{\vm}\Phi f_\delta\big(k(t,\vx)\cdot\veta\big) \Big) \varphi(t,\vx) \dt\dx \right| \ \to \ 0 
		\end{align*}
		as $k\to \infty$ due to \eqref{eq:2-temp-convint} and Lemma~\ref{lemma:not-periodic-weak-convergence}. Hence $(\til{\rho}_k,\til{\vm}_k)\mathop{\rightharpoonup}\limits^\ast (0,\vz)$ as $k\to\infty$, i.e. \eqref{eq:U-pp-conv}.
	\end{enumerate}
	Let $k\geq k_{\min}$ be fixed.
	\begin{enumerate} \setcounter{enumi}{1}
		\item The validity of the PDEs \eqref{eq:U-pp-pde1} and \eqref{eq:U-pp-pde2} follows immediately from \eqref{eq:defn-perturbation} and Proposition~\ref{prop:operators}.
		
		\item On $\til{\Gamma}$ it holds that $\Phi\equiv 1$ and hence \eqref{eq:5-temp-convint}, \eqref{eq:fdelta1}, \eqref{eq:fdelta2} and \eqref{eq:22-temp-convint} yield that $(\rho^\ast,\vm^\ast,\mU^\ast) + (\til{\rho}_k,\til{\vm}_k,\til{\mU}_k)(t,\vx) = (\rho_1,\vm_1,\mU_1)$ for 
		$$
			 (t,\vx) \ \in\ \Gamma_1 := \til{\Gamma} \cap \left\{ (t,\vx)\in\R^{1+n}\,\Big|\, k(t,\vx)\cdot \veta \in (\delta, \tau_1- \delta) + \Z\right\}
		$$
		and $(\rho^\ast,\vm^\ast,\mU^\ast) + (\til{\rho}_k,\til{\vm}_k,\til{\mU}_k)(t,\vx) = (\rho_2,\vm_2,\mU_2)$ for 
		$$
			(t,\vx) \ \in\ \Gamma_2 := \til{\Gamma} \cap \left\{ (t,\vx)\in\R^{1+n}\,\Big|\, k(t,\vx)\cdot \veta \in (\tau_1 + \delta, 1 - \delta) + \Z\right\} \ed
		$$
		Note that $\Gamma_1,\Gamma_2$ are open, and $\Gamma_i\subsetcomp \Gamma^\ast$ for $i=1,2$ due to the fact that $\til{\Gamma} \subsetcomp\Gamma^\ast$ and $\Gamma_1,\Gamma_2 \subset \til{\Gamma}$. Furthermore $\closure{\Gamma_1}\cap \closure{\Gamma_2} = \emptyset$ because of the choice of $\delta$, \eqref{eq:smallness-delta}. See Figure~\ref{fig:CI-2} for a sketch of $\Gamma_1$ and $\Gamma_2$. 
		
		\begin{figure}[b]
			\centering
			\includegraphics[width=0.7\textwidth]{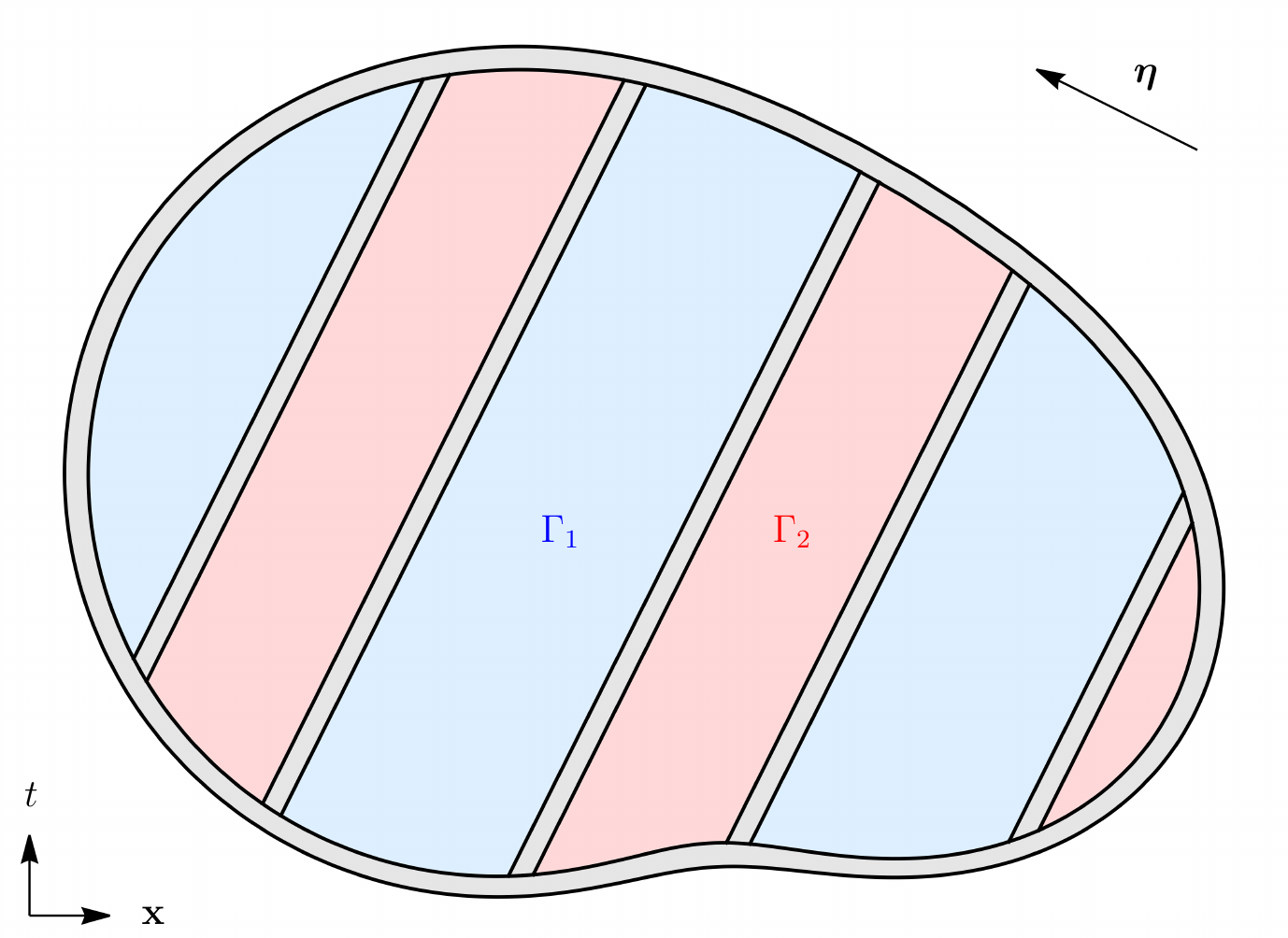} 
			\caption{The set $\til{\Gamma}$ is divided into sets $\Gamma_1$ (blue) and $\Gamma_2$ (red). The slices $\Gamma_\ts$ and the frame $\Gamma_\tf$ are colored in gray.} 
			\label{fig:CI-2}
		\end{figure}

		\item Equations \eqref{eq:defn-perturbation}, \eqref{eq:7-temp-convint} and \eqref{eq:23-temp-convint} imply
		\begin{align*}
			&e\Big((\rho^\ast,\vm^\ast,\mU^\ast) + \big(\til{\rho}_k,\til{\vm}_k,\til{\mU}_k\big)(t,\vx)\Big) \\
			&\leq e\Big((\rho^\ast,\vm^\ast,\mU^\ast) + \big(\opL_\rho[g_k]\Phi, \opL_\vm[g_k]\Phi,\opL_\mU[g_k]\Phi\big)(t,\vx)\Big) + \half \Big( c - \max_{i=1,2} e(\rho_i,\vm_i,\mU_i) \Big) \\
			&\leq  \max_{i=1,2} e(\rho_i,\vm_i,\mU_i) + \half \Big( c - \max_{i=1,2} e(\rho_i,\vm_i,\mU_i) \Big) \\
			&= \half \Big( \max_{i=1,2} e(\rho_i,\vm_i,\mU_i) + c\Big)
		\end{align*} 
		for all $(t,\vx)\in \Gamma^\ast$.
		
		\item Equation \eqref{eq:U-pp-dens-bdd} has already been shown, see \eqref{eq:8-temp-convint}.
		
		\item Let us define the ``slices''
		$$
			\Gamma_\ts := \til{\Gamma} \cap \left\{ (t,\vx)\in\R^{1+n}\,\Big|\, k(t,\vx)\cdot \veta \in \Big([0,\delta] \cup [\tau_1 - \delta , \tau_1 + \delta] \cup [1-\delta , 1)\Big) + \Z\right\} \ec
		$$
		see also Figure~\ref{fig:CI-2}.	Then we have 
		\begin{align}
			&I_{\Gamma^\ast}(\rho^\ast+ \til{\rho}_k,\vm^\ast+ \til{\vm}_k) \notag \\
			&= \iint_{\Gamma_\tf} E(\rho^\ast+ \til{\rho}_k,\vm^\ast+ \til{\vm}_k) \dx\dt + \iint_{\Gamma_\ts} E(\rho^\ast+ \til{\rho}_k,\vm^\ast+ \til{\vm}_k) \dx\dt \notag \\
			&\qquad + \sum_{i=1}^2I_{\Gamma_i}(\rho^\ast+ \til{\rho}_k,\vm^\ast+ \til{\vm}_k) \notag \\ 
			& = \iint_{\Gamma_\tf} E(\rho^\ast+ \til{\rho}_k,\vm^\ast+ \til{\vm}_k) \dx\dt + \iint_{\Gamma_\ts} E(\rho^\ast+ \til{\rho}_k,\vm^\ast+ \til{\vm}_k) \dx\dt + \sum_{i = 1}^2 I_{\Gamma_i}(\rho_i,\vm_i) \label{eq:10-temp-convint}
		\end{align} 
		according to \eqref{eq:U-pp-gconst}.
	
		Furthermore we obtain using \eqref{eq:smallness-frame}, that
		\begin{equation} \label{eq:int-frame-temp-convint}
			\iint_{\Gamma_\tf} E(\rho^\ast+ \til{\rho}_k,\vm^\ast+ \til{\vm}_k) \dx\dt > -c |\Gamma_\tf| > -\frac{\ep}{2} \ed
		\end{equation}
		
		\begin{figure}[tb]
			\centering
			\includegraphics[width=0.9\textwidth]{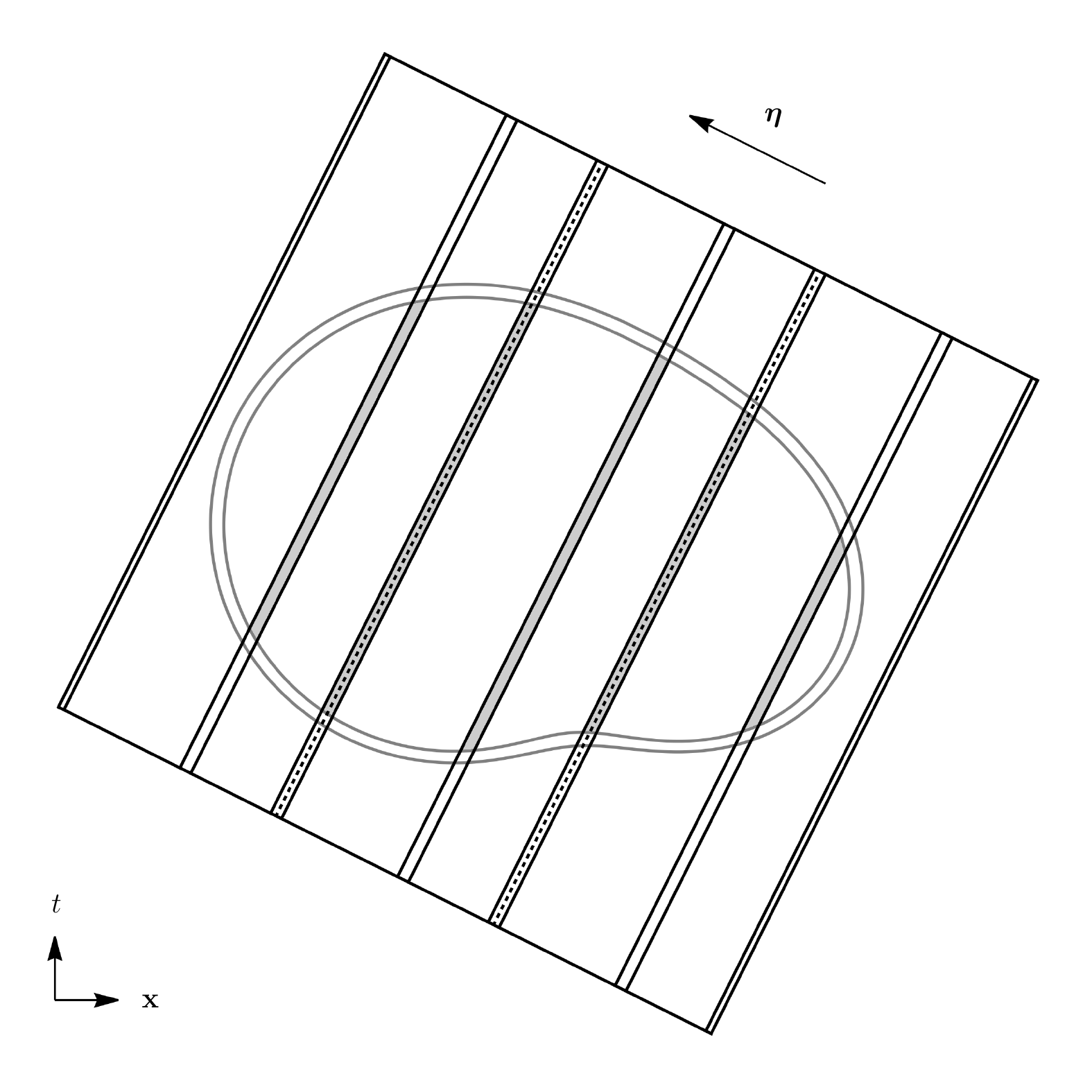} 
			\caption{Sketch of the slices $\Gamma_\ts$ (gray) and the cube $Q$ which is used to estimate $|\Gamma_\ts|$. Here $\ell=3$ and $k=1$.} 
			\label{fig:CI-cube-around}
		\end{figure}
	
		For the remaining integral, note that 
		\begin{equation} \label{eq:50-temp-convint}
			\Gamma_\ts \subset Q \cap \left\{ (t,\vx)\in\R^{1+n}\,\Big|\, k(t,\vx)\cdot \veta \in \Big([0,\delta] \cup [\tau_1 - \delta , \tau_1 + \delta] \cup [1-\delta , 1)\Big) + \Z\right\} \ec
		\end{equation}
		see Figure~\ref{fig:CI-cube-around}. Let us now compute 
		\begin{equation*}
			\left| Q \cap \left\{ (t,\vx)\in\R^{1+n}\,\Big|\, k(t,\vx)\cdot \veta \in \Big([0,\delta] \cup [\tau_1 - \delta , \tau_1 + \delta] \cup [1-\delta , 1)\Big) + \Z\right\} \right| \ed
		\end{equation*}
		Since the edge length of $Q$ is a natural multiple of $\frac{1}{|\veta|}$, we can write it as $\frac{\ell}{|\veta|}$ with $\ell\in\N$. Via rotation and translation we see that
		\begin{align}
			&\left| Q \cap \left\{ (t,\vx)\in\R^{1+n}\,\Big|\, k(t,\vx)\cdot \veta \in \Big([0,\delta] \cup [\tau_1 - \delta , \tau_1 + \delta] \cup [1-\delta , 1)\Big) + \Z\right\} \right| \notag \\
			&= \left| \left(0,\frac{\ell}{|\veta|}\right)^{1+n} \cap \left\{ (t,\vx)\in\R^{1+n}\,\Big|\, k|\veta| t \in \Big([0,\delta] \cup [\tau_1 - \delta , \tau_1 + \delta] \cup [1-\delta , 1)\Big) + \Z\right\} \right| \notag \\
			&= \left(\frac{\ell}{|\veta|}\right)^{n} \int_0^{\frac{\ell}{|\veta|}} \chi(k|\veta| t) \dt \ec \label{eq:51-temp-convint}
		\end{align}
		where $\chi:\R\to \R$, 
		$$
		\chi(t) := \left\{\begin{array}{ll}
			1 & \text{ if } t \in \Big([0,\delta] \cup [\tau_1 - \delta , \tau_1 + \delta] \cup [1-\delta , 1)\Big) + \Z \ec \\
			0 & \text{ else}\ed
		\end{array} \right.
		$$
		A simple substitution yields together with the fact that $\chi$ is periodic and $\ell k\in\N$, that 
		\begin{equation} \label{eq:52-temp-convint}
			\int_0^{\frac{\ell}{|\veta|}} \chi(k|\veta| t) \dt = \frac{1}{k|\veta|}\int_0^{\ell k} \chi(t) \dt = \frac{\ell k}{k|\veta|}\int_0^{1} \chi(t) \dt \ed
		\end{equation}
		Combining \eqref{eq:50-temp-convint}, \eqref{eq:51-temp-convint} and \eqref{eq:52-temp-convint} with the obvious fact, that $\int_0^{1} \chi(t) \dt = 4\delta$, we obtain 
		$$
		|\Gamma_\ts| \leq 4\delta\left(\frac{\ell}{|\veta|}\right)^{1+n} = 4\delta |Q| \ed
		$$
	
		Hence 
		\begin{equation} \label{eq:int-slices-temp-convint}
			\iint_{\Gamma_\ts} E(\rho^\ast+ \til{\rho}_k,\vm^\ast+ \til{\vm}_k) \dx\dt > -c |\Gamma_\ts| \geq -4\delta c |Q| > -\frac{4c |Q| \ep}{8c |Q|} = -\frac{\ep}{2}
		\end{equation}
		due to \eqref{eq:smallness-delta}.
	
		Combining \eqref{eq:int-frame-temp-convint} and \eqref{eq:int-slices-temp-convint} with \eqref{eq:10-temp-convint}, we conclude with \eqref{eq:U-pp-impr}.
	\end{enumerate}

	This finishes the inducion basis. The induction step is deduced from the induction hypothesis in combination with the (already proven) statement of Lemma~\ref{lemma:U-oscillatory-lemma} (the lemma we are proving at the moment) for $N=2$.
	
	\textbf{Induction step:} Let $N\geq 3$. Then according to Definition~\ref{defn:hn} we have (after relabeling if necessary) 
	\begin{equation} \label{eq:11-temp-convint}
		(\rho_2,\vm_2,\mU_2) - (\rho_1,\vm_1,\mU_1) \in \Lambda
	\end{equation} 
	and the family 
	\begin{equation} \label{eq:13-temp-convint}
		\left\{ \left( \tau_1 + \tau_2 , \frac{\tau_1}{\tau_1 + \tau_2} (\rho_1,\vm_1,\mU_1) + \frac{\tau_2}{\tau_1 + \tau_2} (\rho_2,\vm_2,\mU_2)\right) \right\} \cup \big\{\big(\tau_i,(\rho_i,\vm_i,\mU_i)\big)\big\}_{i=3,...,N}
	\end{equation} 
	satisfies the $H_{N-1}$-condition.
	
	Note, that due to Lemma~\ref{lemma:hn1}, the family \eqref{eq:13-temp-convint} has the same barycenter $(\rho^\ast,\vm^\ast,\mU^\ast)$ as the original family $\big\{\big(\tau_i,(\rho_i,\vm_i,\mU_i)\big)\big\}_{i=1,...,N}$. Furthermore 
	\begin{equation} \label{eq:14-temp-convint}
		(\rho_a,\vm_a,\mU_a) := \frac{\tau_1}{\tau_1 + \tau_2} (\rho_1,\vm_1,\mU_1) + \frac{\tau_2}{\tau_1 + \tau_2} (\rho_2,\vm_2,\mU_2) 
	\end{equation}
	satisfies
	\begin{align}
		e(\rho_a,\vm_a,\mU_a) &\leq \frac{\tau_1}{\tau_1 + \tau_2} e(\rho_1,\vm_1,\mU_1) + \frac{\tau_2}{\tau_1 + \tau_2} e(\rho_2,\vm_2,\mU_2) \notag\\
		&\leq \max_{i=1,2} e(\rho_i,\vm_i,\mU_i) \label{eq:24b-temp-convint}\\
		& < c \label{eq:24-temp-convint}
	\end{align}
	due to the convexity of $e$ and the fact that both $(\rho_1,\vm_1,\mU_1),(\rho_2,\vm_2,\mU_2)\in\interior{\sU}$. In particular $(\rho_a,\vm_a,\mU_a)\in \interior{\sU}$. Moreover we obtain from \eqref{eq:14-temp-convint} that $\rho_a > R$ because $\rho_1,\rho_2>R$ by assumption. Hence the family \eqref{eq:13-temp-convint} fulfills the assumptions of Lemma~\ref{lemma:U-oscillatory-lemma} (the lemma we are proving at the moment) with $N-1$ and we can apply the induction hypothesis\footnote{More precisely we set here ``$\ep:=\half \ep$'', i.e. we write $\frac{\ep}{2}$ instead of $\ep$ in \eqref{eq:U-pp-impr-is-ih}.}. This yields a sequence of oscillations 
	$$
		(\til{\rho}_{\oscA,k},\til{\vm}_{\oscA,k},\til{\mU}_{\oscA,k})_{k\in\N}\subset \Cc \big(\Gamma^\ast;\R\times \R^n\times \szn\big)
	$$ 
	with the following properties:
	\begin{enumerate}
		\item The sequence $(\til{\rho}_{\oscA,k},\til{\vm}_{\oscA,k})_{k\in\N}$ converges to $(0,\vz)$ with respect to $d$, i.e. 
			\begin{equation} \label{eq:U-pp-conv-is-ih}
				(\til{\rho}_{\oscA,k},\til{\vm}_{\oscA,k})\mathop{\to}\limits^d (0,\vz)\qquad\text{ as }k\to\infty \ed
			\end{equation}
	\end{enumerate}
	For each fixed $k\in \N$ the following statements hold:
	\begin{enumerate} \setcounter{enumi}{1}
		\item The PDEs 
			\begin{align}
				\partial_t \til{\rho}_{\oscA,k} + \Div \til{\vm}_{\oscA,k} &= 0 \ec \label{eq:U-pp-pde1-is-ih}\\
				\partial_t \til{\vm}_{\oscA,k} + \Div \til{\mU}_{\oscA,k} &= \vz \label{eq:U-pp-pde2-is-ih}
			\end{align}
			hold pointwise for all $(t,\vx)\in \Gamma^\ast$;
		\item There exist open sets $\Gamma_i \subsetcomp \Gamma^\ast$ for all $i = a,3,...,N$, such that 
			\begin{align}
				&\bullet \ \text{the }\ov{\Gamma_i}\text{ are pairwise disjoint and } \label{eq:U-pp-gdis-is-ih} \\
				&\bullet \ (\rho^\ast,\vm^\ast,\mU^\ast) + (\til{\rho}_{\oscA,k},\til{\vm}_{\oscA,k},\til{\mU}_{\oscA,k})(t,\vx) = (\rho_i,\vm_i,\mU_i)\qquad \text{ for all } (t,\vx)\in \Gamma_i \es \label{eq:U-pp-gconst-is-ih}
			\end{align}
		\item The sum of $(\rho^\ast,\vm^\ast,\mU^\ast)$ and the oscillations still takes values in $\interior{\sU}$, more precisely we have 
			\begin{equation} \label{eq:U-pp-subs-is-ih}
				e\Big((\rho^\ast,\vm^\ast,\mU^\ast) + (\til{\rho}_{\oscA,k},\til{\vm}_{\oscA,k},\til{\mU}_{\oscA,k})(t,\vx) \Big) \leq \half \Big(\max_{i=a,3,...,N} e(\rho_i,\vm_i,\mU_i) + c \Big)
			\end{equation} 
			for all $(t,\vx)\in \Gamma^\ast$;
		\item The density is bounded from below, i.e. 
			\begin{equation} \label{eq:U-pp-dens-bdd-is-ih}
				\rho^\ast + \til{\rho}_{\oscA,k}(t,\vx) > R \qquad \text{ for all } (t,\vx)\in\Gamma^\ast\es
			\end{equation}
		\item The value of the functional $I_{\Gamma^\ast}$ can be estimated as follows:
			\begin{equation} \label{eq:U-pp-impr-is-ih}
				I_{\Gamma^\ast}(\rho^\ast + \til{\rho}_{\oscA,k},\vm^\ast + \til{\vm}_{\oscA,k}) > -\frac{\ep}{2} + I_{\Gamma_a}(\rho_a,\vm_a) + \sum_{i = 3}^N I_{\Gamma_i}(\rho_i,\vm_i) \ed 
			\end{equation}
	\end{enumerate} 

	Note that the sets $\Gamma_i$ for $i=a,3,...,N$, in particular the set $\Gamma_a$, depend on $k\in \N$. Now for each $k\in \N$ we apply Lemma~\ref{lemma:U-oscillatory-lemma} (the lemma we are proving at the moment)\footnote{Since we already proved the lemma in the case $N=2$, this procedure is valid.} with $N=2$, where $\Gamma_a$, $\frac{\ep}{2}$, $(\rho_a,\vm_a,\mU_a)$ and $\big\{\big(\frac{\tau_1}{\tau_1 + \tau_2},(\rho_1,\vm_1,\mU_1)\big),\big(\frac{\tau_2}{\tau_1 + \tau_2},(\rho_2,\vm_2,\mU_2)\big)\big\}$ play the role of $\Gamma^\ast$, $\ep$, $(\rho^\ast,\vm^\ast,\mU^\ast)$ and $\big\{\big(\tau_i,(\rho_i,\vm_i,\mU_i)\big)\big\}_{i=1,2}$, respectively\footnote{Note that $\Gamma_a$ is in general not connected, which is the reason why we treat $\Gamma^\ast$ as not necessarily connected in Lemma~\ref{lemma:U-oscillatory-lemma}, see also Example~\ref{ex:CI} and Figure~\ref{fig:CI-example-2}.}. 
	
	The assumptions can easily be checked: 
	Note first, that $\Gamma_a \subsetcomp \Gamma^\ast$ and hence $\Gamma_a$ is bounded. Furthermore $(\rho_a,\vm_a,\mU_a)\in \interior{\sU}$, see \eqref{eq:24-temp-convint}, and $\rho_a>R$. In addition to that, it is clear that $\big(\frac{\tau_1}{\tau_1 + \tau_2},(\rho_1,\vm_1,\mU_1)\big),\big(\frac{\tau_2}{\tau_1 + \tau_2},(\rho_2,\vm_2,\mU_2)\big)\in \R^+\times \interior{\sU}$. Finally we have that 
	\begin{itemize}
		\item $\rho_1,\rho_2>R$ by assumption, 
		\item the family 
		\begin{equation} \label{eq:80-temp-convint}
			\left\{\left(\frac{\tau_1}{\tau_1 + \tau_2},(\rho_1,\vm_1,\mU_1)\right),\left(\frac{\tau_2}{\tau_1 + \tau_2},(\rho_2,\vm_2,\mU_2)\right)\right\}
		\end{equation}
		satisfies the $H_2$-condition because $(\rho_2,\vm_2,\mU_2)-(\rho_1,\vm_1,\mU_1)\in \Lambda$, see \eqref{eq:11-temp-convint}, and $\frac{\tau_1}{\tau_1 + \tau_2} + \frac{\tau_2}{\tau_1 + \tau_2} = 1$, and 
		\item the barycenter of the family in \eqref{eq:80-temp-convint} is equal to $(\rho_a,\vm_a,\mU_a)$ by \eqref{eq:14-temp-convint}.
	\end{itemize}

	Hence for each fixed $k\in\N$ we find a sequence 
	$$
	(\til{\rho}_{\oscB,\ell},\til{\vm}_{\oscB,\ell},\til{\mU}_{\oscB,\ell})_{\ell\in\N}\subset \Cc \big(\Gamma_a;\R\times \R^n\times \szn\big)
	$$ 
	with the following properties:	
	\begin{enumerate}
		\item The sequence $(\til{\rho}_{\oscB,\ell},\til{\vm}_{\oscB,\ell})_{\ell\in\N}$ converges to $(0,\vz)$ with respect to $d$, i.e.
			\begin{equation} \label{eq:U-pp-conv-is-2-temp}
				(\til{\rho}_{\oscB,\ell},\til{\vm}_{\oscB,\ell}) \mathop{\to}\limits^d (0,\vz) \qquad \text{ as }\ell\to \infty \ed 
			\end{equation}
	\end{enumerate}
	For each fixed $\ell\in \N$ the following statements hold:
	\begin{enumerate} \setcounter{enumi}{1}
		\item The PDEs 
			\begin{align}
				\partial_t \til{\rho}_{\oscB,\ell} + \Div \til{\vm}_{\oscB,\ell} &= 0 \ec \label{eq:U-pp-pde1-is-2} \\
				\partial_t \til{\vm}_{\oscB,\ell} + \Div \til{\mU}_{\oscB,\ell} &= \vz \label{eq:U-pp-pde2-is-2}
			\end{align}
			hold pointwise for all $(t,\vx)\in \Gamma_a$;
		
		\item There exist open sets $\Gamma_1,\Gamma_2 \subsetcomp \Gamma_a$ such that 
			\begin{align} 
				&\bullet \ \closure{\Gamma_1}\cap \closure{\Gamma_2} = \emptyset \quad \text{ and} \label{eq:U-pp-gdis-is-2}\\
				&\bullet \ (\rho_a,\vm_a,\mU_a) + (\til{\rho}_{\oscB,\ell},\til{\vm}_{\oscB,\ell},\til{\mU}_{\oscB,\ell})(t,\vx) = (\rho_i,\vm_i,\mU_i)\qquad \text{ for all } (t,\vx)\in \Gamma_i \es \label{eq:U-pp-gconst-is-2}
			\end{align}
		\item The sum of $(\rho_a,\vm_a,\mU_a)$ and the oscillations still takes values in $\interior{\sU}$, more precisely we have
			\begin{equation} \label{eq:U-pp-subs-is-2}
				e\Big((\rho_a,\vm_a,\mU_a) + (\til{\rho}_{\oscB,\ell},\til{\vm}_{\oscB,\ell},\til{\mU}_{\oscB,\ell})(t,\vx) \Big) \leq \half \Big(\max_{i=1,2} e(\rho_i,\vm_i,\mU_i) + c\Big)
			\end{equation} 
			for all $(t,\vx)\in \Gamma_a$;
		\item The density if bounded from below, i.e. 
			\begin{equation} \label{eq:U-pp-dens-bdd-is-2}
				\rho_a + \til{\rho}_{\oscB,\ell}(t,\vx) > R \qquad\text{ for all }(t,\vx)\in \Gamma_a \es
			\end{equation}
		\item The value of the functional $I_{\Gamma_a}$ can be estimated as follows:
			\begin{equation} \label{eq:U-pp-impr-is-2}
				I_{\Gamma_a}(\rho_a + \til{\rho}_{\oscB,\ell},\vm_a + \til{\vm}_{\oscB,\ell}) > -\frac{\ep}{2} + \sum_{i = 1}^2 I_{\Gamma_i}(\rho_i,\vm_i) \ed 
			\end{equation}
	\end{enumerate} 
	
	Note again that the sequence $(\til{\rho}_{\oscB,\ell},\til{\vm}_{\oscB,\ell})_{\ell\in\N}$ still depends on $k\in \N$. In other words we deal with a sequence of sequences. Since \eqref{eq:U-pp-conv-is-2-temp} holds for each $k\in\N$, we find $\ell(k)\in \N$ for each $k\in \N$ such that 
	$$
	d\big((\til{\rho}_{\oscB,\ell(k)},\til{\vm}_{\oscB,\ell(k)}),(0,\vz)\big) \leq \frac{1}{k} \ed
	$$
	This means that the sequence $(\til{\rho}_{\oscB,\ell(k)},\til{\vm}_{\oscB,\ell(k)})_{k\in\N}$ converges to $(0,\vz)$, i.e. 
	\begin{equation} \label{eq:U-pp-conv-is-2}
		(\til{\rho}_{\oscB,\ell(k)},\til{\vm}_{\oscB,\ell(k)}) \mathop{\to}\limits^d (0,\vz) \qquad \text{ as }k\to \infty \ec
	\end{equation} 
	where the properties \eqref{eq:U-pp-pde1-is-2} - \eqref{eq:U-pp-impr-is-2} still hold for each fixed $k\in \N$.

	Finally define the sequence of oscillations
	\begin{equation*}
		(\til{\rho}_k,\til{\vm}_k,\til{\mU}_k)_{k\in\N} := (\til{\rho}_{\oscA,k},\til{\vm}_{\oscA,k},\til{\mU}_{\oscA,k})_{k\in\N} + (\til{\rho}_{\oscB,\ell(k)},\til{\vm}_{\oscB,\ell(k)},\til{\mU}_{\oscB,\ell(k)})_{k\in\N} \ed
	\end{equation*}
	
	Since 
	$$
		(\til{\rho}_{\oscA,k},\til{\vm}_{\oscA,k},\til{\mU}_{\oscA,k})\in \Cc \big(\Gamma^\ast;\R\times \R^n\times \szn\big)
	$$ 
	and 
	$$
		(\til{\rho}_{\oscB,\ell(k)},\til{\vm}_{\oscB,\ell(k)},\til{\mU}_{\oscB,\ell(k)})\in \Cc \big(\Gamma_a;\R\times \R^n\times \szn\big) \ec
	$$ 
	we deduce 
	$$
		(\til{\rho}_k,\til{\vm}_k,\til{\mU}_k)\in \Cc \big(\Gamma^\ast;\R\times \R^n\times \szn\big) \ed
	$$ 
	It remains to prove the desired properties \ref{item:U-pp.a} - \ref{item:U-pp.f}: 
	\begin{enumerate} 
		\item We simply derive \eqref{eq:U-pp-conv} from \eqref{eq:U-pp-conv-is-ih} and \eqref{eq:U-pp-conv-is-2}.
	\end{enumerate}
	Let now $k\in\N$ be fixed. For convenience we write $(\til{\rho}_{\oscA},\til{\vm}_{\oscA},\til{\mU}_{\oscA})$ and $(\til{\rho}_{\oscB},\til{\vm}_{\oscB},\til{\mU}_{\oscB})$ instead of $(\til{\rho}_{\oscA,k},\til{\vm}_{\oscA,k},\til{\mU}_{\oscA,k})$ and $(\til{\rho}_{\oscB,\ell(k)},\til{\vm}_{\oscB,\ell(k)},\til{\mU}_{\oscB,\ell(k)})$, respectively. 
	\begin{enumerate} \setcounter{enumi}{1}
		\item The validity of the PDEs \eqref{eq:U-pp-pde1} and \eqref{eq:U-pp-pde2} immediately follow from \eqref{eq:U-pp-pde1-is-ih}, \eqref{eq:U-pp-pde1-is-2} and \eqref{eq:U-pp-pde2-is-ih}, \eqref{eq:U-pp-pde2-is-2}, respectively.
		
		\item Due to \eqref{eq:U-pp-gdis-is-ih} and \eqref{eq:U-pp-gdis-is-2}, we obtain \eqref{eq:U-pp-gdis}. To show \eqref{eq:U-pp-gconst} we consider two cases.
		\begin{itemize}
			\item Let first $i=3,...,N$. On $\Gamma_i$ we have $(\til{\rho}_{\oscB},\til{\vm}_{\oscB},\til{\mU}_{\oscB}) \equiv (0,\vz,\mZ)$ due to the fact that $(\til{\rho}_{\oscB},\til{\vm}_{\oscB},\til{\mU}_{\oscB})\in \Cc (\Gamma_a;\R\times \R^n\times \szn)$ and $\closure{\Gamma_i}\cap \closure{\Gamma_a} = \emptyset$, see \eqref{eq:U-pp-gdis-is-ih}. Hence \eqref{eq:U-pp-gconst} follows from \eqref{eq:U-pp-gconst-is-ih}.
			\item Let now $i=1,2$. Then \eqref{eq:U-pp-gconst} is a consequence of \eqref{eq:U-pp-gconst-is-ih} and \eqref{eq:U-pp-gconst-is-2}. Indeed for all $(t,\vx)\in \Gamma_i \subset\Gamma_a$ we obtain
			\begin{align*}
				&(\rho^\ast,\vm^\ast,\mU^\ast) + (\til{\rho}_k,\til{\vm}_k,\til{\mU}_k)(t,\vx) \\
				&= (\rho^\ast,\vm^\ast,\mU^\ast) + (\til{\rho}_{\oscA},\til{\vm}_{\oscA},\til{\mU}_{\oscA})(t,\vx) + (\til{\rho}_{\oscB},\til{\vm}_{\oscB},\til{\mU}_{\oscB})(t,\vx) \\
				&= (\rho_a,\vm_a,\mU_a) + (\til{\rho}_{\oscB},\til{\vm}_{\oscB},\til{\mU}_{\oscB})(t,\vx) \\
				&= (\rho_i,\vm_i,\mU_i) \ed
			\end{align*}
		\end{itemize}
		
		\item We again distinguish between two cases. 
		\begin{itemize}
			\item Let first $(t,\vx)\in \Gamma^\ast \setminus \Gamma_a$. As above this implies that $(\til{\rho}_{\oscB},\til{\vm}_{\oscB},\til{\mU}_{\oscB})(t,\vx) = (0,\vz,\mZ)$. Hence
			\begin{align*} 
				&e\Big((\rho^\ast,\vm^\ast,\mU^\ast) + (\til{\rho}_k,\til{\vm}_k,\til{\mU}_k)(t,\vx)\Big) \\
				&= e\Big((\rho^\ast,\vm^\ast,\mU^\ast) + (\til{\rho}_{\oscA},\til{\vm}_{\oscA},\til{\mU}_{\oscA})(t,\vx)\Big) \\
				&\leq \half \Big(\max_{i=a,3,...,N} e(\rho_i,\vm_i,\mU_i) + c \Big) \\
				&\leq \half \Big(\max_{i=1,...,N} e(\rho_i,\vm_i,\mU_i) + c \Big)
			\end{align*}
			according to \eqref{eq:U-pp-subs-is-ih} and \eqref{eq:24b-temp-convint}.
			
			\item Let now $(t,\vx)\in \Gamma_a$. Then by \eqref{eq:U-pp-gconst-is-ih} and \eqref{eq:U-pp-subs-is-2}
			\begin{align*}
				&e\Big((\rho^\ast,\vm^\ast,\mU^\ast) + (\til{\rho}_k,\til{\vm}_k,\til{\mU}_k)(t,\vx)\Big) \\
				&= e\Big((\rho^\ast,\vm^\ast,\mU^\ast) + (\til{\rho}_{\oscA},\til{\vm}_{\oscA},\til{\mU}_{\oscA})(t,\vx) + (\til{\rho}_{\oscB},\til{\vm}_{\oscB},\til{\mU}_{\oscB})(t,\vx)\Big) \\
				&= e\Big((\rho_a,\vm_a,\mU_a) + (\til{\rho}_{\oscB},\til{\vm}_{\oscB},\til{\mU}_{\oscB})(t,\vx)\Big) \\
				&\leq \half \Big(\max_{i=1,2} e(\rho_i,\vm_i,\mU_i) + c\Big) \\
				&\leq \half \Big(\max_{i=1,...,N} e(\rho_i,\vm_i,\mU_i) + c\Big) \ed
			\end{align*}
		\end{itemize}
		So we have shown that \eqref{eq:U-pp-subs} holds in both cases. 
		
		\item Let us again consider two cases. 
		\begin{itemize}
			\item Let $(t,\vx)\in\Gamma^\ast\setminus \Gamma_a$. Then $\rho^\ast+\til{\rho}_k(t,\vx) = \rho^\ast + \til{\rho}_{\oscA}(t,\vx)$, and hence \eqref{eq:U-pp-dens-bdd-is-ih} yields \eqref{eq:U-pp-dens-bdd}.
			\item Let now $(t,\vx)\in \Gamma_a$. Then $\rho^\ast+\til{\rho}_k(t,\vx) = \rho_a + \til{\rho}_{\oscB}(t,\vx)$ due to \eqref{eq:U-pp-gconst-is-ih}. We deduce \eqref{eq:U-pp-dens-bdd} from \eqref{eq:U-pp-dens-bdd-is-2}.
		\end{itemize}
		
		\item Using once again, that $(\til{\rho}_{\oscB},\til{\vm}_{\oscB},\til{\mU}_{\oscB}) \equiv (0,\vz,\mZ)$ on $\Gamma^\ast \setminus \Gamma_a$, we obtain
		\begin{align} 
			&I_{\Gamma^\ast} (\rho^\ast + \til{\rho}_k, \vm^\ast + \til{\vm}_k ) \notag\\
			&= I_{\Gamma^\ast} (\rho^\ast + \til{\rho}_{\oscA} + \til{\rho}_{\oscB}, \vm^\ast + \til{\vm}_{\oscA} + \til{\vm}_{\oscB} ) \notag\\
			&= \iint_{\Gamma^\ast\setminus\Gamma_a} E(\rho^\ast + \til{\rho}_{\oscA}, \vm^\ast + \til{\vm}_{\oscA}) \dx\dt + I_{\Gamma_a} (\rho^\ast + \til{\rho}_{\oscA} + \til{\rho}_{\oscB}, \vm^\ast + \til{\vm}_{\oscA} + \til{\vm}_{\oscB} ) \ed \label{eq:15-temp-convint} 
		\end{align}
		We treat both integrals in \eqref{eq:15-temp-convint} separately.
		
		First, as a consequence of \eqref{eq:U-pp-impr-is-ih} and \eqref{eq:U-pp-gconst-is-ih}, we get 
		\begin{align}
			&\iint_{\Gamma^\ast\setminus\Gamma_a} E(\rho^\ast + \til{\rho}_{\oscA}, \vm^\ast + \til{\vm}_{\oscA}) \dx\dt \notag \\
			&= I_{\Gamma^\ast} (\rho^\ast + \til{\rho}_{\oscA}, \vm^\ast + \til{\vm}_{\oscA}) - I_{\Gamma_a} (\rho^\ast + \til{\rho}_{\oscA}, \vm^\ast + \til{\vm}_{\oscA}) \notag \\
			&> -\frac{\ep}{2} + I_{\Gamma_a} (\rho_a,\vm_a) + \sum_{i=3}^N I_{\Gamma_i} (\rho_i,\vm_i) - I_{\Gamma_a} (\rho_a, \vm_a) \notag \\
			&= -\frac{\ep}{2} + \sum_{i=3}^N I_{\Gamma_i} (\rho_i,\vm_i) \ed \label{eq:16-temp-convint}
		\end{align}
		
		Second, \eqref{eq:U-pp-gconst-is-ih} and \eqref{eq:U-pp-impr-is-2} yield
		\begin{align}
			&I_{\Gamma_a} (\rho^\ast + \til{\rho}_{\oscA} + \til{\rho}_{\oscB}, \vm^\ast + \til{\vm}_{\oscA} + \til{\vm}_{\oscB} ) \notag \\
			&= I_{\Gamma_a} (\rho_a + \til{\rho}_{\oscB}, \vm_a + \til{\vm}_{\oscB} ) \notag \\
			&> -\frac{\ep}{2} + \sum_{i=1}^2 I_{\Gamma_i} (\rho_i,\vm_i) \ed \label{eq:17-temp-convint}
		\end{align}
		
		Plugging \eqref{eq:16-temp-convint} and \eqref{eq:17-temp-convint} into \eqref{eq:15-temp-convint} yields \eqref{eq:U-pp-impr}.
	\end{enumerate}
	This finishes the proof. 
\end{proof} 

In order to illustrate the induction step, let us consider the following example where $N=3$.

\begin{figure}[tb] 
	\centering
	\subfloat[Constellation in the extended phase space.\label{fig:CI-example-HN}]{ 
		\centering
		\includegraphics[width=0.43\textwidth]{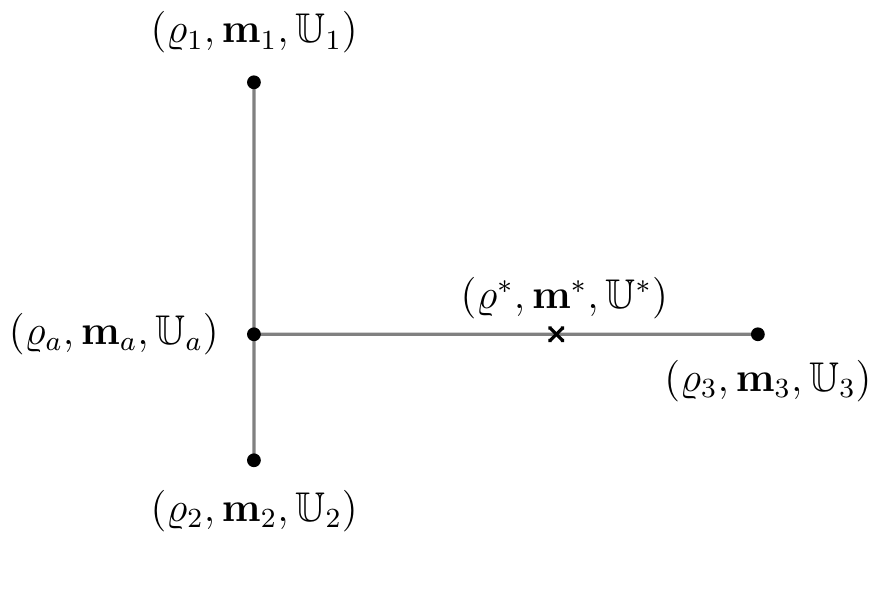} 
	}
	\hspace{1.5cm}
	\subfloat[Illustration of the induction hypothesis.\label{fig:CI-example-2}]{
		\centering
		\includegraphics[width=0.43\textwidth]{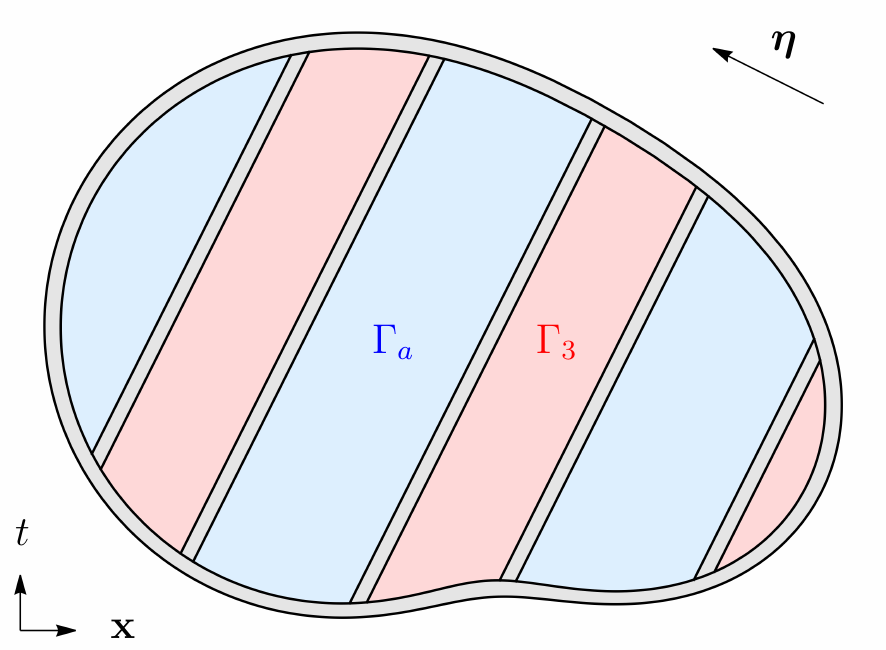} 
	} \\
	\subfloat[For the induction step we apply Lemma~\ref{lemma:U-oscillatory-lemma} for $N=2$ on $\Gamma_a$.\label{fig:CI-example-induction}]{
		\centering
		\includegraphics[width=0.43\textwidth]{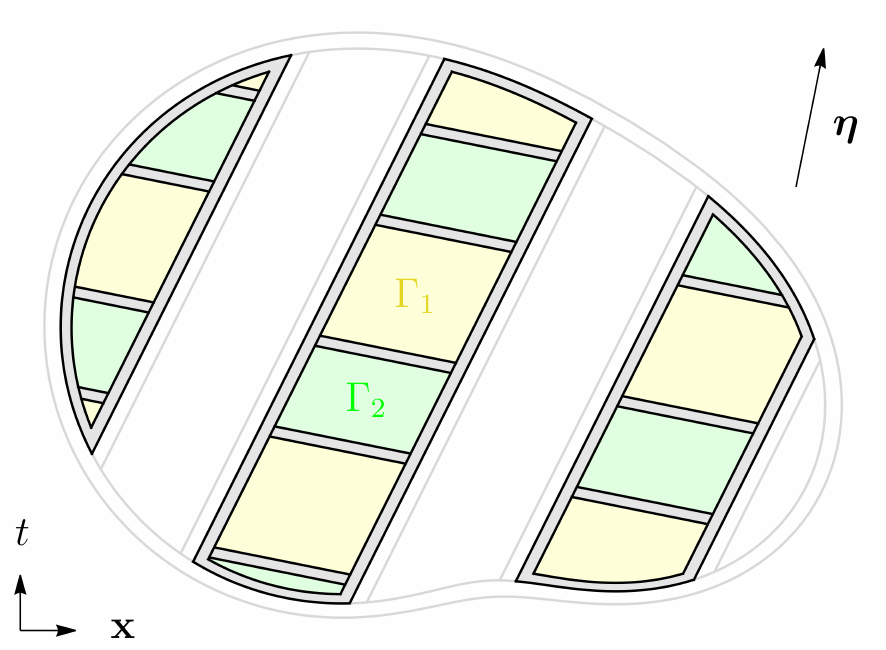}
	}
	\hspace{1.5cm}
	\subfloat[Combination of Figures (b) and (c).\label{fig:CI-example-3}]{
		\centering
		\includegraphics[width=0.43\textwidth]{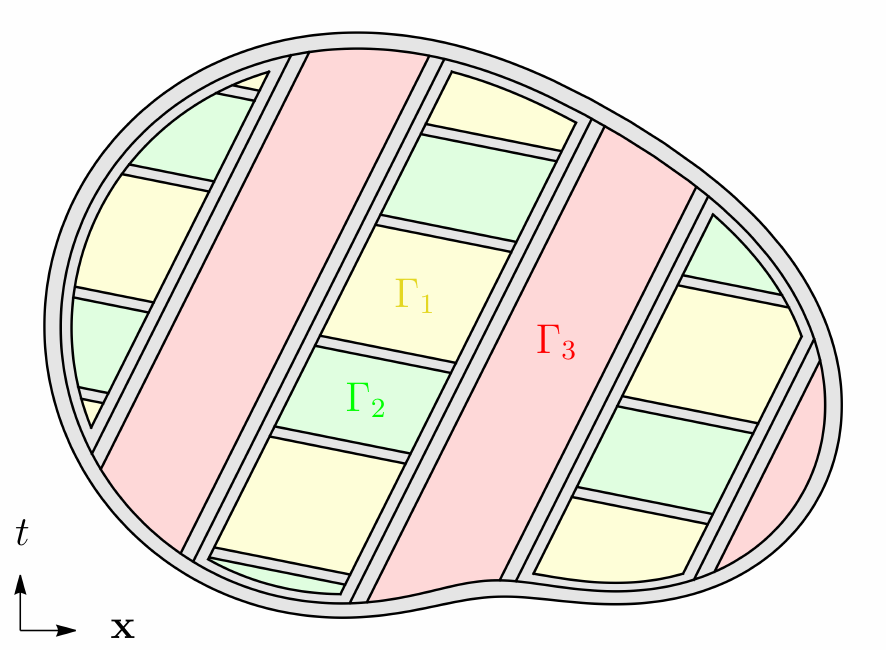} 
	} 
	\caption{Illustration of Example~\ref{ex:CI}. Frames and slices are marked in gray, whereas the sets $\Gamma_1$, $\Gamma_2$, $\Gamma_3$ and $\Gamma_a$ are colored yellow, green, red and blue, respectively.}
	\label{fig:CI-example} 
\end{figure}

\begin{ex} \label{ex:CI}
	Let the assumptions of Lemma~\ref{lemma:U-oscillatory-lemma} hold with $N=3$, i.e. the family 
	$$
		\Big\{ \big( \tau_1,(\rho_1,\vm_1,\mU_1)\big) , \big( \tau_2,(\rho_2,\vm_2,\mU_2)\big) , \big( \tau_3,(\rho_3,\vm_3,\mU_3)\big) \Big\}
	$$
	satisfies the $H_3$-condition. The constellation in the extended phase space is depicted in Figure~\ref{fig:CI-example-HN}. Let $(\rho_a,\vm_a,\mU_a)$ defined as in \eqref{eq:14-temp-convint}. Then the family in \eqref{eq:13-temp-convint} reads 
	$$
		\big\{\big(\tau_1+\tau_2,(\rho_a,\vm_a,\mU_a)\big),\big(\tau_3,(\rho_3,\vm_3,\mU_3)\big)\big\}
	$$ 
	and satisfies the $H_2$-condition. Applying Lemma~\ref{lemma:U-oscillatory-lemma} to the latter family, we obtain a sequence of oscillations 
	$$
		(\til{\rho}_{\oscA,k},\til{\vm}_{\oscA,k},\til{\mU}_{\oscA,k})_{k\in\N}\subset \Cc \big(\Gamma^\ast;\R\times \R^n\times \szn\big)
	$$
	fulfilling \eqref{eq:U-pp-conv-is-ih} - \eqref{eq:U-pp-impr-is-ih}, see Figure~\ref{fig:CI-example-2}. In particular 
	$$
		 (\rho^\ast,\vm^\ast,\mU^\ast) + (\til{\rho}_{\oscA,k},\til{\vm}_{\oscA,k},\til{\mU}_{\oscA,k})(t,\vx) = (\rho_a,\vm_a,\mU_a)\qquad \text{ for all } (t,\vx)\in \Gamma_a \ec
	$$
	see \eqref{eq:U-pp-gconst-is-ih}. Since the family
	\begin{equation} \label{eq:9-temp-convint}
		\left\{\left(\frac{\tau_1}{\tau_1 + \tau_2},(\rho_1,\vm_1,\mU_1)\right),\left(\frac{\tau_2}{\tau_1 + \tau_2},(\rho_2,\vm_2,\mU_2)\right)\right\}
	\end{equation}
	satisfies the $H_2$-condition with barycenter $(\rho_a,\vm_a,\mU_a)$, we can apply Lemma~\ref{lemma:U-oscillatory-lemma} again, this time to the family in \eqref{eq:9-temp-convint} and on $\Gamma_a$. Figure~\ref{fig:CI-example-2} shows that $\Gamma_a$ is not connected. This is the reason why we treat $\Gamma^\ast$ as not necessarily connected in Lemma~\ref{lemma:U-oscillatory-lemma}. Applying Lemma~\ref{lemma:U-oscillatory-lemma} yields 
	$$
		(\til{\rho}_{\oscB,\ell},\til{\vm}_{\oscB,\ell},\til{\mU}_{\oscB,\ell})_{\ell\in\N}\subset \Cc \big(\Gamma_a;\R\times \R^n\times \szn\big)
	$$
	such that \eqref{eq:U-pp-conv-is-2-temp} - \eqref{eq:U-pp-impr-is-2} hold, see Figure~\ref{fig:CI-example-induction}. 
	Finally the sum 
	$$
		(\til{\rho}_k,\til{\vm}_k,\til{\mU}_k)_{k\in\N} := (\til{\rho}_{\oscA,k},\til{\vm}_{\oscA,k},\til{\mU}_{\oscA,k})_{k\in\N} + (\til{\rho}_{\oscB,\ell(k)},\til{\vm}_{\oscB,\ell(k)},\til{\mU}_{\oscB,\ell(k)})_{k\in\N} 
	$$ 
	satisfies \eqref{eq:U-pp-conv} - \eqref{eq:U-pp-impr}, see Figure~\ref{fig:CI-example-3}.
\end{ex}

Now we are ready to prove Lemma~\ref{lemma:K-oscillatory-lemma} using Lemmas \ref{lemma:existence-family}, \ref{lemma:UstattK} and \ref{lemma:U-oscillatory-lemma}.

\begin{proof}[Proof of Lemma~\ref{lemma:K-oscillatory-lemma}] 
	From Lemma~\ref{lemma:existence-family} we obtain $N\in\N$ with $N\geq 2$ and furthermore $\big(\tau_i,(\rho_i,\vm_i,\mU_i)\big)\in \R^+\times K$ for $i=1,...,N$ such that 
	\begin{itemize}
		\item $\rho_i>R$ for all $i=1,...,N$,
		\item the family $\big\{\big(\tau_i,(\rho_i,\vm_i,\mU_i)\big)\big\}_{i=1,...,N}$ satisfies the $H_N$-condition and
		\item $(\rho^\ast,\vm^\ast,\mU^\ast) = \sum_{i=1}^N \tau_i (\rho_i,\vm_i,\mU_i) $.
	\end{itemize}

	Set 
	\begin{equation} \label{eq:definition-tau}
		\tau := \min\left\{ \frac{\frac{C}{2} }{c + 2\frac{M^4}{r^3} + n\big(p(M) + p'(M) M\big)} , \half\right\} \ec 
	\end{equation}
	with the $M$ of Lemma~\ref{lemma:U-bdd}. Note that $\tau\in(0,1)$. Then define 
	\begin{equation} \label{eq:26-temp-convint}
		(\hat{\rho}_i,\hat{\vm}_i,\hat{\mU}_i) := \tau (\rho^\ast,\vm^\ast,\mU^\ast) + (1-\tau) (\rho_i,\vm_i,\mU_i)
	\end{equation}
	for all $i=1,...,N$. 
	
	
	\begin{figure}[b] 
		\centering
		\hspace{1cm}
		\subfloat{ 
			\centering
			\includegraphics[width=0.43\textwidth]{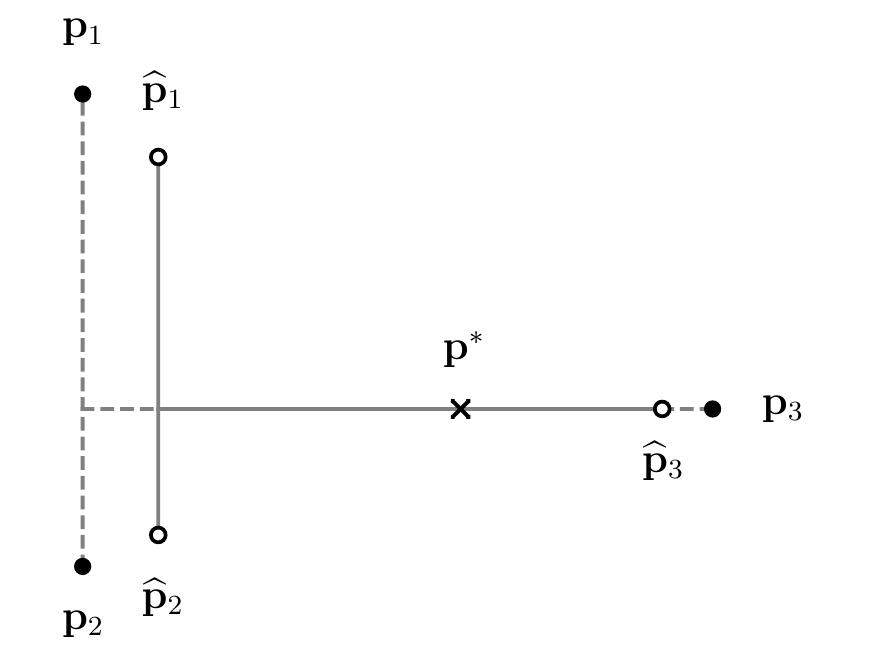} 
		}
		\hspace{0.5cm}
		\subfloat{
			\centering
			\includegraphics[width=0.43\textwidth]{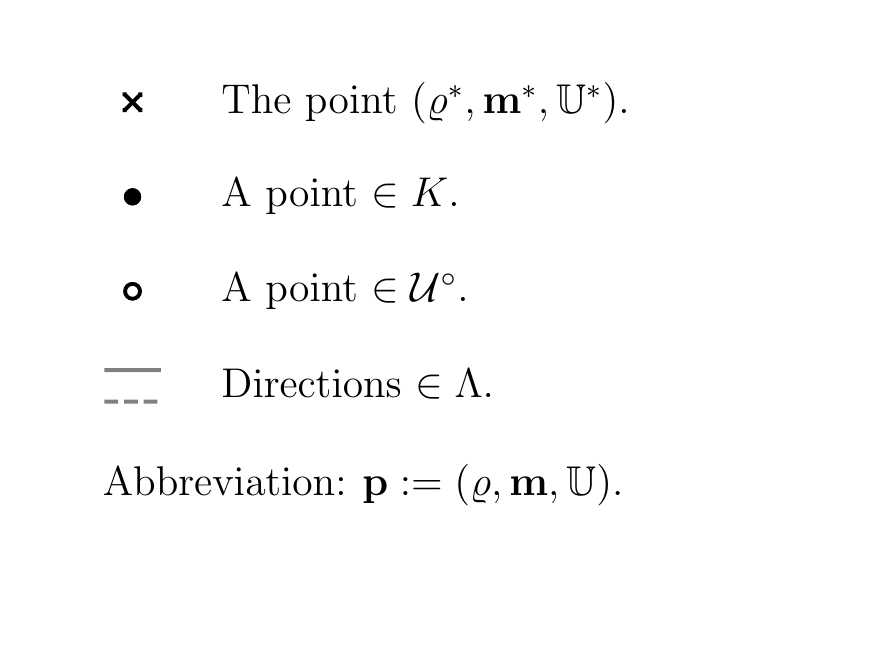} 
		} 
		\caption{An exemplary constellation in the extended phase space where $N=3$.} 
		\label{fig:CI-UstattK}
	\end{figure}
	
	According to Lemma~\ref{lemma:UstattK} we have
	\begin{itemize}
		\item $\hat{\rho}_i>R$ for all $i=1,...,N$,
		\item $(\hat{\rho}_i,\hat{\vm}_i,\hat{\mU}_i)\in \interior{\sU}$ for all $i=1,...,N$,
		\item The family $\big\{\big(\tau_i,(\hat{\rho}_i,\hat{\vm}_i,\hat{\mU}_i)\big)\big\}_{i=1,...,N}$ satisfies the $H_N$-condition and
		\item $(\rho^\ast,\vm^\ast,\mU^\ast) = \sum_{i=1}^N \tau_i (\hat{\rho}_i,\hat{\vm}_i,\hat{\mU}_i)$,
	\end{itemize}
	see Figure~\ref{fig:CI-UstattK}. In other words the assumptions of Lemma~\ref{lemma:U-oscillatory-lemma} hold for the family 
	$$
	\Big\{\big(\tau_i,(\hat{\rho}_i,\hat{\vm}_i,\hat{\mU}_i)\big)\Big\}_{i=1,...,N} \ed
	$$
	Hence Lemma~\ref{lemma:U-oscillatory-lemma} yields a sequence 
	$$
		(\til{\rho}_k,\til{\vm}_k,\til{\mU}_k)_{k\in \N}\subset \Cc \big(\Gamma^\ast;\R\times \R^n\times \szn\big)
	$$ 
	fulfilling properties \ref{item:U-pp.a} - \ref{item:U-pp.f} of Lemma~\ref{lemma:U-oscillatory-lemma}, where we set ``$\ep:=\half\ep$'', i.e. we obtain \eqref{eq:U-pp-impr} with $\half\ep$ instead of $\ep$. Let us show that this sequence satisfies the desired properties \ref{item:K-pp.a} - \ref{item:K-pp.e} of Lemma~\ref{lemma:K-oscillatory-lemma}.
	
	\begin{enumerate}
		\item The convergence \eqref{eq:K-pp-conv} holds according to \eqref{eq:U-pp-conv}.
	\end{enumerate}
	Let now $k\in\N$ be fixed. 
	\begin{enumerate} \setcounter{enumi}{1}
		\item The PDEs \eqref{eq:K-pp-pde1}, \eqref{eq:K-pp-pde2} hold due to \eqref{eq:U-pp-pde1}, \eqref{eq:U-pp-pde2}.
		
		\item For all $(t,\vx)\in \Gamma^\ast$ we obtain from \eqref{eq:U-pp-subs}, \eqref{eq:26-temp-convint} and the convexity of $e$
		\begin{align}
			&e\Big((\rho^\ast,\vm^\ast,\mU^\ast) + (\til{\rho}_k,\til{\vm}_k,\til{\mU}_k)(t,\vx)\Big) \notag \\
			&\leq \half \Big(\max_{i=1,...,N} e(\hat{\rho}_i,\hat{\vm}_i,\hat{\mU}_i) + c\Big) \notag\\
			&\leq \half \tau e(\rho^\ast,\vm^\ast,\mU^\ast) + \half (1-\tau) \max_{i=1,...,N} e(\rho_i,\vm_i,\mU_i) + \frac{c}{2} \ed \label{eq:25-temp-convint}
		\end{align}
		Since $(\rho_i,\vm_i,\mU_i)\in K$, we have $e(\rho_i,\vm_i,\mU_i)=c$ for all $i=1,...,N$, see \eqref{eq:boundaryU}. Hence \eqref{eq:25-temp-convint} implies together with the assumption, that $e(\rho^\ast,\vm^\ast,\mU^\ast) \leq c-\gamma$, 
		\begin{align*}
			e\Big((\rho^\ast,\vm^\ast,\mU^\ast) + (\til{\rho}_k,\til{\vm}_k,\til{\mU}_k)(t,\vx)\Big) &\leq \half \tau \Big(e(\rho^\ast,\vm^\ast,\mU^\ast) -c \Big) +c \\
			&\leq c - \frac{\tau \gamma}{2}\ed
		\end{align*}
		Setting $\beta:=\frac{\tau \gamma}{2}$, this yields \eqref{eq:K-pp-subs}. Note that $\beta$ only depends on $\gamma$ and $C$, see \eqref{eq:definition-tau}.
		
		\item We obtain \eqref{eq:K-pp-dens-bdd} from \eqref{eq:U-pp-dens-bdd}.
		
		\item In order to show \eqref{eq:K-pp-impr} we start by proving 
		\begin{equation} \label{eq:29-temp-convint}
			E(\hat{\rho}_i,\hat{\vm}_i) \geq -\frac{C}{2} \qquad \text{ for all }i=1,...,N \ed
		\end{equation} 
		To this end we define the mappings
		\begin{align*}
			\tau\mapsto E_i(\tau) &:= E\big(\tau(\rho^\ast,\vm^\ast)+(1-\tau)(\rho_i,\vm_i)\big) \\
			&= \frac{\big|\tau\vm^\ast + (1-\tau) \vm_i\big|^2}{2\big(\tau\rho^\ast + (1-\tau) \rho_i\big)} + \frac{n}{2} p\big(\tau\rho^\ast + (1-\tau) \rho_i\big) - c \ec
		\end{align*}
		for $i=1,...,N$. It simply follows from the convexity of $E$, see Lemma~\ref{lemma:prop-E}~\ref{item:prop-E.b}, that the mappings $E_i$ are convex as well. This implies 
		\begin{equation} \label{eq:30-temp-convint}
			E_i(\tau) \geq \tau E_i'(0) + E_i(0)
		\end{equation}
		for all $\tau\in [0,1]$, in particular for $\tau$ as in \eqref{eq:definition-tau}. Note that
		\begin{align}
			E_i(0) = E(\rho_i,\vm_i) &= \frac{|\vm_i|^2}{2\rho_i} + \frac{n}{2} p(\rho_i) - c \notag\\
			&= \half \tr\left( \frac{\vm_i\otimes \vm_i}{\rho_i} + p(\rho_i) \id - \mU_i \right) - c \notag\\
			&= 0 \qquad \qquad \text{ for all }i=1,...,N \label{eq:27-temp-convint}
		\end{align}
		because $ (\rho_i,\vm_i,\mU_i)\in K$. Furthermore a straighforward computation yields
		\begin{align*}
			E_i'(\tau) &= \frac{\big(\tau\vm^\ast + (1-\tau) \vm_i\big)\cdot (\vm^\ast-\vm_i)}{\tau \rho^\ast + (1-\tau)\rho_i} -  \frac{\big|\tau\vm^\ast + (1-\tau) \vm_i\big|^2}{2\big(\tau\rho^\ast + (1-\tau) \rho_i\big)^2} (\rho^\ast - \rho_i) \\
			&\qquad + \frac{n}{2} p'\big(\tau\rho^\ast + (1-\tau) \rho_i\big) (\rho^\ast - \rho_i)
		\end{align*}
		and in particular
		\begin{align*}
			E_i'(0) &= \frac{\vm_i\cdot (\vm^\ast-\vm_i)}{\rho_i} -  \frac{|\vm_i|^2}{2\rho_i^2} (\rho^\ast - \rho_i) + \frac{n}{2} p'(\rho_i) (\rho^\ast - \rho_i) \ed
		\end{align*}
		Using \eqref{eq:27-temp-convint} the latter can be written as
		\begin{equation*}
			E_i'(0) = E(\rho^\ast,\vm^\ast) - \half \frac{|\rho_i\vm^\ast - \rho^\ast \vm_i|^2}{\rho_i^2 \rho^\ast} + \frac{n}{2} \Big( p(\rho_i) - p(\rho^\ast) + p'(\rho_i) (\rho^\ast - \rho_i) \Big) \ec
		\end{equation*}
		which we estimate using Lemma~\ref{lemma:U-bdd} to 
		\begin{equation} \label{eq:28-temp-convint}
			E_i'(0) \geq -c - \frac{2 M^4}{r^3} - n \Big( p(M) + p'(M) M \Big) \ed 
		\end{equation}
		Putting together \eqref{eq:30-temp-convint}, \eqref{eq:27-temp-convint}, \eqref{eq:28-temp-convint} and \eqref{eq:definition-tau}, we end up with \eqref{eq:29-temp-convint}.
		
		Now with \eqref{eq:U-pp-impr} and \eqref{eq:29-temp-convint} we obtain
		\begin{align*}
			I_{\Gamma^\ast}(\rho^\ast + \til{\rho}_k,\vm^\ast + \til{\vm}_k) &> -\frac{\ep}{2} + \sum_{i = 1}^N I_{\Gamma_i} (\hat{\rho}_i,\hat{\vm}_i)\\
			&= -\frac{\ep}{2} + \sum_{i = 1}^N |\Gamma_i| E(\hat{\rho}_i,\hat{\vm}_i) \\
			&\geq -\frac{\ep}{2} - \frac{C}{2}\sum_{i = 1}^N |\Gamma_i| \\
			&\geq -\frac{\ep}{2} - \frac{\ep}{2} \\
			& = -\ep \ec
		\end{align*}
		where we made use of $\sum_{i = 1}^N |\Gamma_i| \leq |\Gamma^\ast| = \frac{\ep}{C}$, which follows from the fact that the $\closure{\Gamma_i}\subset\Gamma^\ast$ are pairwise disjoint, see \eqref{eq:U-pp-gdis}. This shows \eqref{eq:K-pp-impr} and hence finishes the proof.
	\end{enumerate}
\end{proof}

\subsection{Proof} \label{subsec:convint-pp-proof} 
With Lemma~\ref{lemma:K-oscillatory-lemma} at hand we are ready to prove the Perturbation Property.

\begin{proof}[Proof of Proposition~\ref{prop:pert-prop}]
Note first of all, that we may assume without loss of generality that $\Gamma_0\neq \emptyset$ because the claim is trivial in the case $\Gamma_0=\emptyset$ by setting $(\rho_\pert,\vm_\pert)=(\rho,\vm)$.

Let us introduce a grid in the space-time $\R^{1+n}$ with size $h>0$: We consider open cubes $Q_{(i,\valpha),h}\subset \R^{1+n}$ for $(i,\valpha)\in \Z\times \Z^n$ with mid points $(ih,\valpha h)$ and edge length $h$, i.e. 
$$
	Q_{(i,\valpha),h} := (ih,\valpha h) + \left(-\frac{h}{2},\frac{h}{2}\right)^{1+n} \ed
$$

We fix $\ov{h}>0$ so small such that 
\begin{equation} \label{eq:smallness-frame-pp} 
	|\Gamma_\tf| < \frac{\ep}{3c} \ec 
\end{equation}
where the ``frame'' is given by $\Gamma_\tf := \Gamma_0\setminus \til{\Gamma}$ and 
\begin{equation*}
	\til{\Gamma} := \interior{\left(\bigcup_{(i,\valpha)\in \left\{ (i,\valpha)\in \Z\times \Z^n\,|\, Q_{(i,\valpha),\ov{h}}\subsetcomp \Gamma_0\right\}} \closure{Q_{(i,\valpha),\ov{h}}}\right)} \ed 
\end{equation*}
Note that the set 
$$
	\left\{ (i,\valpha)\in \Z\times \Z^n\,\Big|\, Q_{(i,\valpha),\ov{h}}\subsetcomp \Gamma_0\right\}\ed
$$
contains only finitely many pairs $(i,\valpha)$ since $\Gamma_0$ is bounded. See Figure~\ref{fig:Gamma-large-cubes} for an example of $\Gamma_0$. 

\begin{figure}[b]
	\centering
	\includegraphics[width=0.7\textwidth]{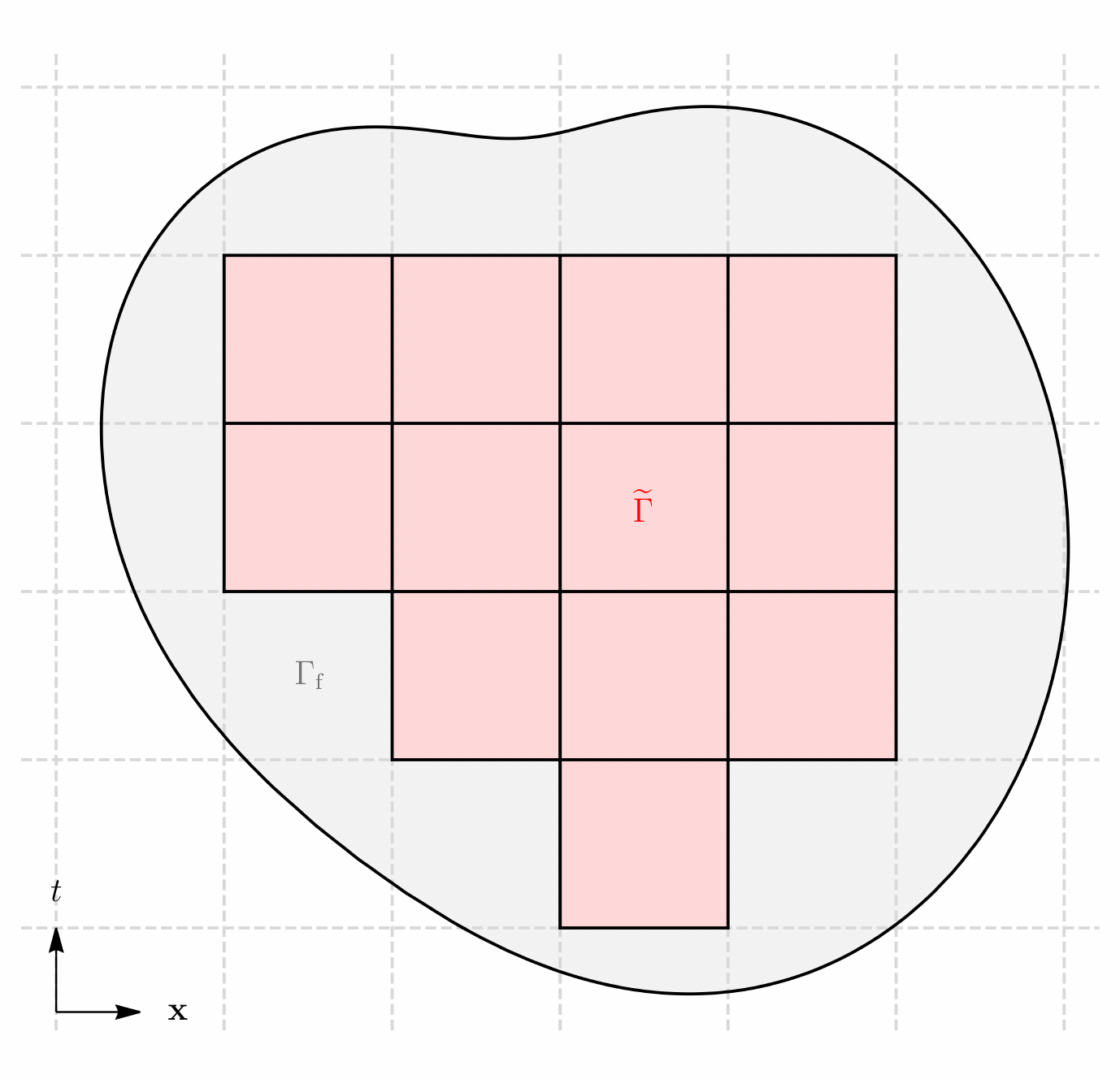} 
	\caption{An example of the domain $\Gamma_0$, the grid with grid size $\ov{h}$, the set $\til{\Gamma}$ (red) and the frame $\Gamma_\tf$ (gray).} 
	\label{fig:Gamma-large-cubes}
\end{figure}

By continuity of $e$ and due to the fact that $\til{\Gamma} \subsetcomp \Gamma_0$ and $e(\rho,\vm,\mU)(t,\vx)<c$ for all $(t,\vx)\in \Gamma_0$, there exists $\gamma>0$ such that 
\begin{equation} \label{eq:32-temp-convint}
	e\big((\rho,\vm,\mU)(t,\vx)\big)\leq c-\gamma \qquad \text{ for all } (t,\vx)\in \til{\Gamma} \ed
\end{equation}	
Similarly, since $\rho(t,\vx)>r$ for all $(t,\vx)\in \Gamma_0$, there exists $R>r$ such that 
\begin{equation} \label{eq:38-temp-convint}
	\rho(t,\vx)>R \qquad \text{ for all } (t,\vx)\in \til{\Gamma} \ed
\end{equation}

For $\gamma$ as above and $C:=\frac{\ep}{3|\til{\Gamma}|}$, Lemma~\ref{lemma:K-oscillatory-lemma} yields $\beta=\beta(\gamma,C)>0$ with the properties stated in Lemma~\ref{lemma:K-oscillatory-lemma}. 

Accoring to Lemma~\ref{lemma:prop-E}~\ref{item:prop-E.a}, $E$ and $e$ are uniformly continuous on $[r,M]\times \closure{B_n(\vz,M)}$ and $[r,M]\times \closure{B_n(\vz,M)} \times \closure{B_{\szn}(\mZ,M)}$, respectively. Hence there exists $\delta>0$ such that 
\begin{align}
	\Big| E(\rho_1,\vm_1)  - E(\rho_2,\vm_2) \Big| &< \frac{\ep}{3 |\til{\Gamma}|} \qquad \text{ and } \label{eq:33-temp-convint} \\ 
	\Big| e(\rho_1,\vm_1,\mU_1)  - e(\rho_2,\vm_2,\mU_2) \Big| &< \beta \ec \label{eq:34-temp-convint}
\end{align}
whenever $\big|(\rho_1,\vm_1,\mU_1)-(\rho_2,\vm_2,\mU_2)\big|<\delta$. 

Due to Lemma~\ref{lemma:valuesX0bounded}, $(\rho,\vm,\mU)$ takes values in $[r,M]\times \closure{B_n(\vz,M)} \times \closure{B_{\szn}(\mZ,M)}$. Since $\Gamma_0$ is a bounded subset of $\Gamma$ and $(\rho,\vm,\mU)\in C^1(\closure{\Gamma};\R^+ \times \R^n\times \szn)$, the restriction of $(\rho,\vm,\mU)$ onto $\closure{\Gamma_0}$ is uniformly continuous. Therefore we find $q\in\N$ such that for $h:=\frac{\ov{h}}{q}$ the following holds: If $(t_1,\vx_1)$ and $(t_2,\vx_2)$ lie in the same cube $Q_{(i,\valpha),h}$, then
\begin{equation} \label{eq:35-temp-convint} 
	\Big|\big(\rho(t_1,\vx_1),\vm(t_1,\vx_1),\mU(t_1,\vx_1)\big)-\big(\rho(t_2,\vx_2),\vm(t_2,\vx_2),\mU(t_2,\vx_2)\big)\Big|<\min\{\delta,R-r\} \ed 
\end{equation} 

Let 
$$
	\mathcal{J} := \left\{ (i,\valpha)\in \Z\times \Z^n\,\Big|\, Q_{(i,\valpha),h}\subset \til{\Gamma}\right\}\ed
$$
Since $\til{\Gamma}$ is bounded, $\mathcal{J}$ contains only finitely many pairs $(i,\valpha)$. Note that 
$$
	\til{\Gamma}=\interior{\left(\bigcup_{(i,\valpha)\in \mathcal{J}} \closure{Q_{(i,\valpha),h}}\right)}
$$ 
by construction of $h$ and hence 
\begin{equation} \label{eq:60-temp-convint}
	|\til{\Gamma}|=|\mathcal{J}| \cdot |Q_{(i,\valpha),h}| \ec 
\end{equation}
see also Figure~\ref{fig:Gamma-small-cubes}.

\begin{figure}[tb] 
	\centering
	\includegraphics[width=0.7\textwidth]{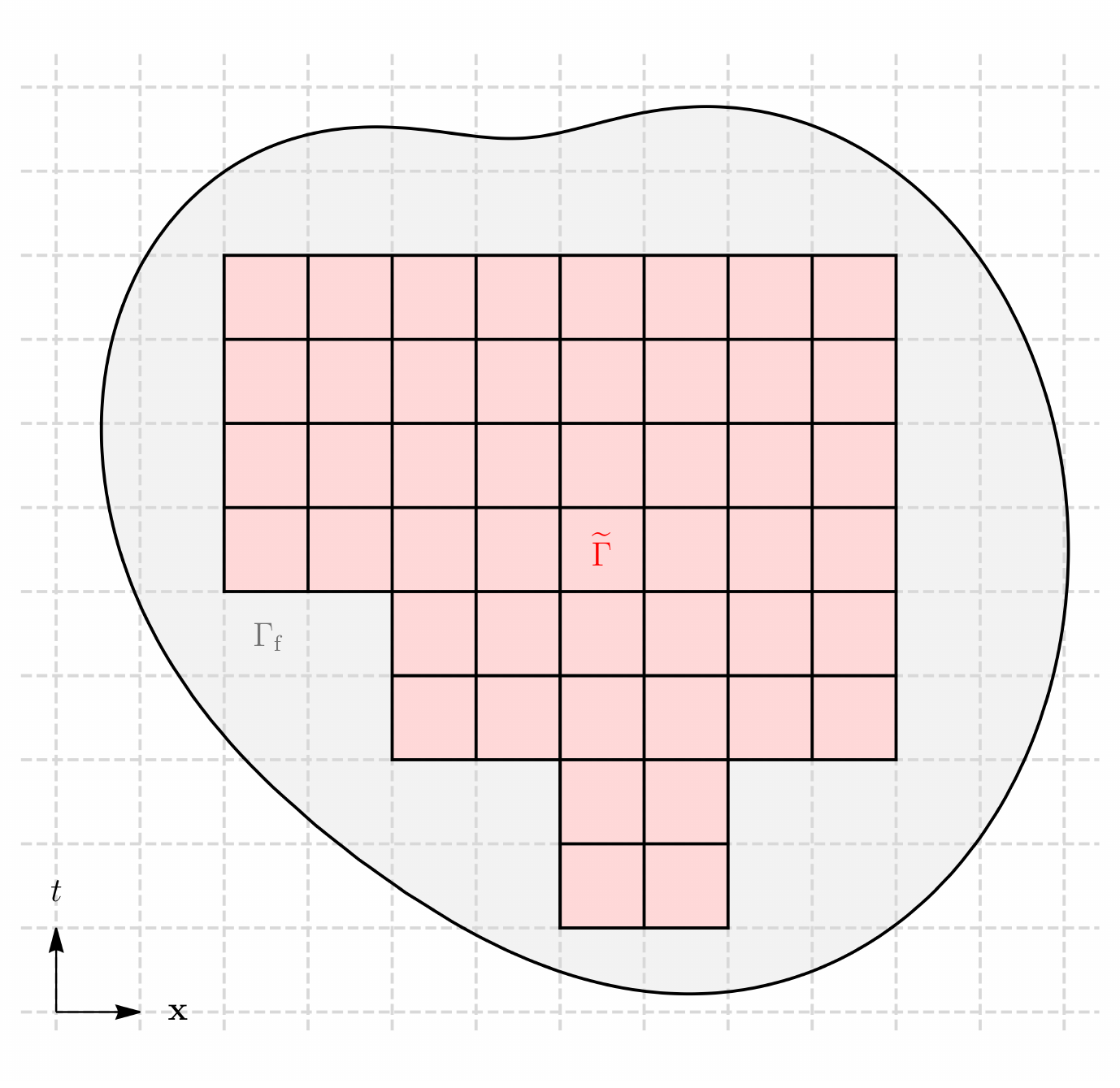} 
	\caption{The cubes $Q_{(i,\valpha),h}$, where here $h=\frac{\ov{h}}{2}$, exhaust the set $\til{\Gamma}$ (red).} 
	\label{fig:Gamma-small-cubes}
\end{figure}
 
Next, for each $(i,\valpha)\in\mathcal{J}$ we apply Lemma~\ref{lemma:K-oscillatory-lemma} where $Q_{(i,\valpha),h}$, $(\rho,\vm,\mU)(ih,\valpha h)$ and $\frac{\ep}{3 |\mathcal{J}|}$ play the role of $\Gamma^\ast$, $(\rho^\ast,\vm^\ast,\mU^\ast)$ and $\ep$, respectively. Note that $\rho(ih,\valpha h)>R$ for all $(i,\valpha)\in \mathcal{J}$ because of \eqref{eq:38-temp-convint}. In addition to that, we have $e\big((\rho,\vm,\mU)(ih,\valpha h)\big)\leq c-\gamma$ for all $(i,\valpha)\in \mathcal{J}$, see \eqref{eq:32-temp-convint}. Moreover 
$$
	\frac{\frac{\ep}{3 |\mathcal{J}|}}{|Q_{(i,\valpha),h}|} = \frac{\ep}{3 |\til{\Gamma}|} = C 
$$
by \eqref{eq:60-temp-convint} and definition of $C$. Hence the assumptions of Lemma~\ref{lemma:K-oscillatory-lemma} hold. 

We obtain for each $(i,\valpha)\in\mathcal{J}$ a sequence of oscillations 
$$
	(\til{\rho}_{i,\valpha,k},\til{\vm}_{i,\valpha,k},\til{\mU}_{i,\valpha,k})_{k\in \N}\subset \Cc\big(Q_{(i,\valpha),h};\R\times \R^n\times \szn\big)
$$ 
with properties \ref{item:K-pp.a} - \ref{item:K-pp.e} stated in Lemma~\ref{lemma:K-oscillatory-lemma}. 
 
Now define a sequence
\begin{equation} \label{eq:defn-pert-sequence}
	(\rho_{\pert,k},\vm_{\pert,k},\mU_{\pert,k}) := (\rho,\vm,\mU) + \sum_{(i,\valpha)\in\mathcal{J}} (\til{\rho}_{i,\valpha,k},\til{\vm}_{i,\valpha,k},\til{\mU}_{i,\valpha,k}) \ed
\end{equation}
Let us first show that $(\rho_{\pert,k},\vm_{\pert,k})\in X_0$ for each fixed $k\in \N$. It is obvious that $(\rho_{\pert,k},\vm_{\pert,k},\mU_{\pert,k})\in C^1\big(\closure{\Gamma};\R^+\times \R^n\times \szn\big)$. Furthermore properties \ref{item:pfs-pde} - \ref{item:pfs-dens-bdd} of Definition~\ref{defn:X0} hold, which we show in the sequel. 

\begin{enumerate}
	\item The PDEs \eqref{eq:pfs-pde1} and \eqref{eq:pfs-pde2} for $(\rho_{\pert,k},\vm_{\pert,k},\mU_{\pert,k})$ hold pointwise for all $(t,\vx)\in \Gamma$ due to the fact they hold for $(\rho,\vm,\mU)$, since $(\rho,\vm)\in X_0$, and because the PDEs \eqref{eq:K-pp-pde1} and \eqref{eq:K-pp-pde2} hold for each $(\til{\rho}_{i,\valpha,k},\til{\vm}_{i,\valpha,k},\til{\mU}_{i,\valpha,k})$. 
	
	\item For each $(t,\vx)\in \Gamma$ there exists at most one pair $(i,\valpha)\in \mathcal{J}$ such that $(t,\vx)\in Q_{(i,\valpha),h}$. If there is no such pair, then 
	$$
		(\rho_{\pert,k},\vm_{\pert,k},\mU_{\pert,k})(t,\vx) = (\rho,\vm,\mU)(t,\vx) 
	$$ 
	since the oscillations $(\til{\rho}_{i,\valpha,k},\til{\vm}_{i,\valpha,k},\til{\mU}_{i,\valpha,k})$ are compactly supported in $Q_{(i,\valpha),h}$. Hence \eqref{eq:pfs-subs} for $(\rho_{\pert,k},\vm_{\pert,k},\mU_{\pert,k})$ follows from the fact that $(\rho,\vm)\in X_0$. 
	
	Let now $(t,\vx)\in Q_{(i,\valpha),h}$ with a particular pair $(i,\valpha)\in \mathcal{J}$. Then 
	$$
		(\rho_{\pert,k},\vm_{\pert,k},\mU_{\pert,k})(t,\vx) = (\rho,\vm,\mU)(t,\vx) + (\til{\rho}_{i,\valpha,k},\til{\vm}_{i,\valpha,k},\til{\mU}_{i,\valpha,k})(t,\vx) \ed
	$$
	Because of our choice of $h$ we have  
	\begin{align}
		&\Big| (\rho,\vm,\mU)(t,\vx) + (\til{\rho}_{i,\valpha,k},\til{\vm}_{i,\valpha,k},\til{\mU}_{i,\valpha,k})(t,\vx) \notag\\
		& \qquad - (\rho,\vm,\mU)(ih,\valpha h) - (\til{\rho}_{i,\valpha,k},\til{\vm}_{i,\valpha,k},\til{\mU}_{i,\valpha,k})(t,\vx) \Big| \notag\\
		&= \Big| (\rho,\vm,\mU)(t,\vx) - (\rho,\vm,\mU)(ih,\valpha h) \Big| \ <\ \delta \ec \label{eq:36-temp-convint}
	\end{align}
	see \eqref{eq:35-temp-convint}. Hence \eqref{eq:34-temp-convint} together with \eqref{eq:K-pp-subs} yields
	\begin{align*}
		&e\big( (\rho_{\pert,k},\vm_{\pert,k},\mU_{\pert,k})(t,\vx) \big) \\
		&= e\Big((\rho,\vm,\mU)(t,\vx) + (\til{\rho}_{i,\valpha,k},\til{\vm}_{i,\valpha,k},\til{\mU}_{i,\valpha,k})(t,\vx) \Big) \\
		&< e\Big((\rho,\vm,\mU)(ih,\valpha h) + (\til{\rho}_{i,\valpha,k},\til{\vm}_{i,\valpha,k},\til{\mU}_{i,\valpha,k})(t,\vx) \Big) + \beta \\
		&\leq c \ed 
	\end{align*}
	This shows 
	$$
		(\rho_{\pert,k},\vm_{\pert,k},\mU_{\pert,k})(t,\vx)\ \in \ \interior{\sU} \qquad \text{ for all }(t,\vx)\in Q_{(i,\valpha),h} \ed
	$$
	
	\item Since $(\til{\rho}_{i,\valpha,k},\til{\vm}_{i,\valpha,k},\til{\mU}_{i,\valpha,k})$ have compact support in $Q_{(i,\valpha),h}\subset \Gamma_0 \subset \Gamma$ and the cubes $Q_{(i,\valpha),h}$ are open, we have
	$$
		\sum_{(i,\valpha)\in\mathcal{J}} (\til{\rho}_{i,\valpha,k},\til{\vm}_{i,\valpha,k},\til{\mU}_{i,\valpha,k})(t,\vx) = (0,\vz,\mZ)\qquad \text{ for all }(t,\vx)\in \partial\Gamma\ed
	$$
	Because $(\rho,\vm)\in X_0$ and hence \eqref{eq:pfs-bc} holds for $(\rho,\vm,\mU)$, we deduce
	$$
		(\rho_{\pert,k},\vm_{\pert,k},\mU_{\pert,k})(t,\vx) = (\rho,\vm,\mU)(t,\vx) + (0,\vz,\mZ) = (\rho_0,\vm_0,\mU_0)(t,\vx) 
	$$
	for all $(t,\vx)\in \partial\Gamma$, i.e. \eqref{eq:pfs-bc} holds for $(\rho_{\pert,k},\vm_{\pert,k},\mU_{\pert,k})$ as well.
	
	\item Let $(t,\vx)\in \Gamma$. Again we consider two cases. Assume first that there is no $(i,\valpha)\in\mathcal{J}$ such that $(t,\vx)\in Q_{(i,\valpha),h}$. Then $\rho_{\pert,k}(t,\vx)=\rho(t,\vx)$ and \eqref{eq:pfs-dens-bdd} for $\rho_{\pert,k}$ follows from \eqref{eq:pfs-dens-bdd} for $\rho$. 
	
	Let now $(t,\vx)\in Q_{(i,\valpha),h}$ with a particular pair $(i,\valpha)\in\mathcal{J}$. Hence 
	\begin{align*}
		\rho_{\pert,k}(t,\vx) &= \rho(t,\vx) + \til{\rho}_{i,\valpha,k}(t,\vx) \\
		&= \rho(ih,\valpha h) + \til{\rho}_{i,\valpha,k}(t,\vx) + \big(\rho(t,\vx) - \rho(ih,\valpha h)\big) \\
		&\geq \rho(ih,\valpha h) + \til{\rho}_{i,\valpha,k}(t,\vx) - \big|\rho(t,\vx) - \rho(ih,\valpha h)\big| \\
		&> R - (R-r) \\ 
		&= r
	\end{align*} 
	due to \eqref{eq:K-pp-dens-bdd} and \eqref{eq:35-temp-convint}. Thus $\rho_{\pert,k}$ satisfies \eqref{eq:pfs-dens-bdd}.
\end{enumerate} 

We have proven that $(\rho_{\pert,k},\vm_{\pert,k})\in X_0$ for each $k\in \N$.

Due to the convergence of $(\til{\rho}_{i,\valpha,k},\til{\vm}_{i,\valpha,k})_{k\in \N}$ to $(0,\vz)$ with respect to $d$ for all $(i,\valpha)\in\mathcal{J}$ (see \eqref{eq:K-pp-conv}), we easily see from \eqref{eq:defn-pert-sequence}, that
$$
	(\rho_{\pert,k},\vm_{\pert,k}) \mathop{\rightharpoonup}\limits^\ast (\rho,\vm) \qquad \text{ as }k\to \infty\ec
$$ 
in other words
$$
	(\rho_{\pert,k},\vm_{\pert,k}) \mathop{\to}\limits^d (\rho,\vm) \qquad \text{ as }k\to \infty \ed
$$
Hence we may fix $k\in\N$ large enough such that 
$$
	d\big( (\rho_{\pert,k},\vm_{\pert,k}), (\rho,\vm)\big) \leq \ov{\ep}
$$
and set 
\begin{equation*} 
	(\rho_\pert,\vm_\pert) := (\rho_{\pert,k},\vm_{\pert,k})\ed 
\end{equation*}
Then \eqref{eq:pert-prop-1} holds.

It remains to show \eqref{eq:pert-prop-2}. Let us write
\begin{align} 
	&I_{\Gamma_0} (\rho_\pert,\vm_\pert) \notag \\ 
	&= \iint_{\Gamma_\tf} E(\rho_\pert,\vm_\pert) \dx\dt + I_{\til{\Gamma}} (\rho_\pert,\vm_\pert) \notag\\
	&= \iint_{\Gamma_\tf} E(\rho_\pert,\vm_\pert) \dx\dt + \sum_{(i,\valpha)\in\mathcal{J}} I_{Q_{(i,\valpha),h}} (\rho_\pert,\vm_\pert) \notag\\
	&= \iint_{\Gamma_\tf} E(\rho_\pert,\vm_\pert) \dx\dt  + \sum_{(i,\valpha)\in\mathcal{J}} I_{Q_{(i,\valpha),h}} \Big((\rho,\vm)(ih,\valpha h) + (\til{\rho}_{i,\valpha,k},\til{\vm}_{i,\valpha,k})\Big)\label{eq:39-temp-convint} \\
	& \qquad + \sum_{(i,\valpha)\in\mathcal{J}} \bigg[ I_{Q_{(i,\valpha),h}} \Big((\rho,\vm) + (\til{\rho}_{i,\valpha,k},\til{\vm}_{i,\valpha,k})\Big) \notag\\
	&\qquad \qquad \qquad \qquad - I_{Q_{(i,\valpha),h}} \Big((\rho,\vm)(ih,\valpha h) + (\til{\rho}_{i,\valpha,k},\til{\vm}_{i,\valpha,k})\Big) \bigg] \notag
\end{align} 
and consider each summand on the right-hand side of \eqref{eq:39-temp-convint} separately.

First, with the help of \eqref{eq:smallness-frame-pp} we have 
\begin{equation}  \label{eq:i1-temp-convint}
	\iint_{\Gamma_\tf} E(\rho_\pert,\vm_\pert) \dx\dt > -c |\Gamma_\tf| > -\frac{\ep}{3} \ed
\end{equation} 

Second, \eqref{eq:K-pp-impr} implies
\begin{equation} \label{eq:i2-temp-convint}
\sum_{(i,\valpha)\in\mathcal{J}} I_{Q_{(i,\valpha),h}}\Big((\rho,\vm)(ih,\valpha h) + (\til{\rho}_{i,\valpha,k},\til{\vm}_{i,\valpha,k})\Big) > - \frac{\ep}{3|\mathcal{J}|} |\mathcal{J}| = -\frac{\ep}{3} \ed 
\end{equation}

Third, due to \eqref{eq:36-temp-convint} we obtain from \eqref{eq:33-temp-convint}
\begin{align}
	&\sum_{(i,\valpha)\in\mathcal{J}} \bigg[ I_{Q_{(i,\valpha),h}} \Big((\rho,\vm) + (\til{\rho}_{i,\valpha,k},\til{\vm}_{i,\valpha,k})\Big) \notag\\
	&\qquad \qquad \quad - I_{Q_{(i,\valpha),h}} \Big((\rho,\vm)(ih,\valpha h) + (\til{\rho}_{i,\valpha,k},\til{\vm}_{i,\valpha,k})\Big) \bigg] \notag \\
	&= \sum_{(i,\valpha)\in\mathcal{J}} \iint_{Q_{(i,\valpha),h}} \bigg[ E\Big((\rho,\vm)(t,\vx) + (\til{\rho}_{i,\valpha,k},\til{\vm}_{i,\valpha,k})(t,\vx)\Big) \notag\\
	&\qquad \qquad \qquad\qquad - E\Big((\rho,\vm)(ih,\valpha h) + (\til{\rho}_{i,\valpha,k},\til{\vm}_{i,\valpha,k})(t,\vx)\Big) \bigg] \dx\dt \notag\\
	&\geq -\frac{\ep}{3|\til{\Gamma}|}\sum_{(i,\valpha)\in\mathcal{J}} |Q_{(i,\valpha),h}| \notag\\
	&= -\frac{\ep}{3} \ec \label{eq:i3-temp-convint}
\end{align}
where we have used that $\sum_{(i,\valpha)\in\mathcal{J}} |Q_{(i,\valpha),h}|= |\til{\Gamma}|$. 

Plugging \eqref{eq:i1-temp-convint}, \eqref{eq:i2-temp-convint} and \eqref{eq:i3-temp-convint} into \eqref{eq:39-temp-convint}, we end up with \eqref{eq:pert-prop-2}.

\end{proof}

\section{Convex Integration with Fixed Density} 
\label{sec:convint-nodens}

Theorem~\ref{thm:convint} generally allows for solutions $(\rho,\vm)$ where $\rho\not\equiv \rho_0$. 

\begin{rem}
	At this point we should be more precise. The reader might have noticed that in the proof of Lemma~\ref{lemma:existence-family} we proved existence of $(\rho_i,\vm_i,\mU_i)\in K$ where in particular $\rho_i=\rho^\ast$ for all $i=1,...,N$. As a consequence the $\hat{\rho}_i$ obtained in Lemma~\ref{lemma:UstattK} have the same property, i.e. $\hat{\rho}_i=\rho^\ast$ for all $i=1,...,N$. Hence if $\veta$ is not parallel to $\ve_t$, then the oscillations $\til{\rho}_k$ constructed in Lemma~\ref{lemma:U-oscillatory-lemma} are equal to 0, since $\opL_\rho\equiv 0$ according to Proposition~\ref{prop:operators}~\ref{item:operators.c}. Thus $\rho\equiv \rho_0$, which seems to contradict that in general $\rho\not\equiv \rho_0$. However the $\rho_i$ in Lemma~\ref{lemma:existence-family} can be found in a different way as well. Indeed consider $(\rho^\ast \pm a , \vm^\ast,\mU^\ast)$ where $|a|$ is sufficiently small such that $\rho^\ast \pm a >R$ and $(\rho^\ast \pm a , \vm^\ast,\mU^\ast)\in \interior{\sU}$. An application of the proved version of Lemma~\ref{lemma:existence-family} to each $(\rho^\ast \pm a , \vm^\ast,\mU^\ast)$ yields then two families which can be combined to one family with barycenter $(\rho^\ast , \vm^\ast,\mU^\ast)$ according to Lemma~\ref{lemma:hn2}. Consequently the resulting family has the property that either $\rho_i=\rho^\ast - a$ or $\rho_i=\rho^\ast +a$. In particular the $\rho_i$ are not all equal, which in fact leads to an oscillation $\til{\rho}_k\not\equiv 0$ in Lemma~\ref{lemma:U-oscillatory-lemma} and finally $\rho\not\equiv \rho_0$.
\end{rem}

As pointed out in Section~\ref{sec:intro-well-posedness}, solutions to the compressible Euler system have been constructed by using clever ansatzes to reduce the system to some kind of ``incompressible Euler system'' and apply convex integration to the latter. This way the density $\rho$ does not join in the convex integration which means (using our notation) $\rho\equiv \rho_0$.

In this section we briefly explain how the steps we made in Sections \ref{sec:convint-thm} and \ref{sec:convint-pp} can be modified to obtain solutions with $\rho\equiv\rho_0$. We call this a \emph{fixed-density-version} of Theorem~\ref{thm:convint}. The main ingredients are Corollary~\ref{cor:complete-wc} and Proposition~\ref{prop:operators}~\ref{item:operators.c}.

\subsection{A Modified Version of the Convex-Integration-Theorem} \label{subsec:convint-nodens-thm}

We want to prove the following version of Theorem~\ref{thm:convint}.

\begin{thm} \label{thm:convint-nodens}
	Let the assumptions of Theorem~\ref{thm:convint} be true. Then infinitely many among the solutions $(\rho,\vm)\in L^\infty(\Gamma; \R^+ \times \R^n)$ have the additional property that $\rho\equiv \rho_0$. 
\end{thm}

In order to prove this, we need a variant of the Perturbation Property (Proposition~\ref{prop:pert-prop}).

\begin{prop} \label{prop:pert-prop-nodens}
	Let the assumptions of Proposition~\ref{prop:pert-prop} be true. Then one can achieve that $\rho_\pert\equiv \rho$ on $\closure{\Gamma}$ in addition to the properties given in Proposition~\ref{prop:pert-prop}.
\end{prop}

\begin{proof}[Proof of Theorem~\ref{thm:convint-nodens}] 
In order to prove Theorem~\ref{thm:convint-nodens}, one proceeds as in the proof of Theorem~\ref{thm:convint}. To start with, define a set $\til{X_0}$, which is used instead of $X_0$, by 
$$
\til{X_0} := \left\{ \vm \in C^1\big(\closure{\Gamma};\R^n\big) \, \Big|\, (\rho_0,\vm)\in X_0\right\} \ed
$$
For the remaining steps one has to slightly modify the steps in the proof of Theorem~\ref{thm:convint}. In particular one has to use Proposition~\ref{prop:pert-prop-nodens} rather than Proposition~\ref{prop:pert-prop}. We leave the details to the reader.
\end{proof}

\subsection{Proof the Modified Perturbation Property} \label{subsec:convint-nodens-pp}

To prove Propostion \ref{prop:pert-prop-nodens} we proceed in the same way as in the proof of Proposition~\ref{prop:pert-prop}. The priciple ingredient is the following lemma, a version of Lemma~\ref{lemma:K-oscillatory-lemma}. The final proof of  Propostion \ref{prop:pert-prop-nodens} is postponed to the end of this section.

\begin{lemma} \label{lemma:K-oscillatory-lemma-nodens}
	Let the assumptions of Lemma~\ref{lemma:K-oscillatory-lemma} be true. We can achieve that the sequence of oscillations \eqref{eq:K-oscillations} has the additional property that $\til{\rho}_k\equiv 0$ for all $k\in\N$. 
\end{lemma}

We prove Lemma~\ref{lemma:K-oscillatory-lemma-nodens} analogously to Lemma~\ref{lemma:K-oscillatory-lemma}. To this end we have to show suitable versions of Lemmas \ref{lemma:existence-family}, \ref{lemma:UstattK} and \ref{lemma:U-oscillatory-lemma}. The main idea is to make use of part \ref{item:operators.c} of Proposition~\ref{prop:operators}, which says that the operator $\opL_{\rho}$, which yields the oscillation in the density in the proof of Lemma~\ref{lemma:U-oscillatory-lemma}, can be chosen to be equal zero. However this only works as long as $\ov{\rho}=0$ and $\veta$ is not parallel to $\ve_t$. Hence the variant of Lemma~\ref{lemma:existence-family} which is needed here, must yield a family whose densities $\rho_i$ are equal to $\rho^\ast$ in order to guarantee $\ov{\rho}=0$. This will be quite easy. However the fact that we must achieve that the occuring $\veta$ is not parallel to $\ve_t$, is an obstacle. To get along with that, we need the following definition.

\begin{defn}
	Let $N\in\N$ with $N\geq 2$ and $\big(\tau_i,(\rho_i,\vm_i,\mU_i)\big)\in \R^+ \times \interior{\sU}$ for $i=1,...,N$. We say that the family $\big\{\big(\tau_i,(\rho_i,\vm_i,\mU_i)\big)\big\}_{i=1,...,N}$ satisfies the \emph{$m_N$-condition} if it satisfies both the $H_N$-condition and the following:
	\begin{itemize}
		\item If $N=2$, then $\vm_2-\vm_1\neq \vz$. 
		\item If $N\geq 3$, then $\vm_2-\vm_1\neq \vz$ and the family
		\begin{equation} \label{eq:defn-hn-iteration-b}
			\left\{ \left( \tau_1 + \tau_2 , \frac{\tau_1}{\tau_1 + \tau_2} (\rho_1,\vm_1,\mU_1) + \frac{\tau_2}{\tau_1 + \tau_2} (\rho_2,\vm_2,\mU_2)\right) \right\} \cup \big\{\big(\tau_i,(\rho_i,\vm_i,\mU_i)\big)\big\}_{i=3,...,N}
		\end{equation}
		satisfies the $m_{N-1}$-condition, where the family $\big\{\big(\tau_i,(\rho_i,\vm_i,\mU_i)\big)\big\}_{i=1,...,N}$ is relabeled\footnote{Note that this is possible according to the definition of the $H_N$-condition (Definition~\ref{defn:hn}).} such that $(\rho_2,\vm_2,\mU_2)-(\rho_1,\vm_1,\mU_1)\in\Lambda$ and the family \eqref{eq:defn-hn-iteration-b}	satisfies the $H_{N-1}$-condition. 
	\end{itemize}
\end{defn}

Now we are able to state Lemma~\ref{lemma:existence-family-nodens}, a modified version of Lemma~\ref{lemma:existence-family}. 

\begin{lemma} \label{lemma:existence-family-nodens} 
	Let the assumptions of Lemma~\ref{lemma:existence-family} be true. We can achieve that the family $\big\{\big(\tau_i,(\rho_i,\vm_i,\mU_i)\big)\big\}_{i=1,...,N}$ additionally satisfies
	\begin{itemize}
		\item the $m_N$-condition and
		\item $\rho_i=\rho^\ast$ for all $i=1,...,N$.
	\end{itemize}
\end{lemma}

\begin{proof} 
	Due to Propositions \ref{prop:KLambda=U}, \ref{prop:Kstar}~\ref{item:Kstar.b} and \ref{prop:compKLambda} we have 
	$$
		\sU \cap \{\rho=\rho^\ast\} = (K\cap \{\rho=\rho^\ast\})^\Lambda \ed
	$$
	From Corollary~\ref{cor:complete-wc} we obtain $(K\cap \{\rho=\rho^\ast\})^\Lambda = (K\cap \{\rho=\rho^\ast\})^\co$, since $\Lambda$ is complete with respect to $K\cap \{\rho=\rho^\ast\}$, see the proof of Lemma~\ref{lemma:Kast} for details. Hence	
	\begin{equation} \label{eq:71-temp-convint}
		\sU \cap \{\rho=\rho^\ast\} = (K\cap \{\rho=\rho^\ast\})^\co \ed
	\end{equation}
	
	Let us now define
	$$
		K_{\rho^\ast} := \Big\{(\vm,\mU)\in \R^n\times \szn\,\Big|\, (\rho^\ast,\vm,\mU)\in K\Big\} \ed
	$$
	Next, we show\footnote{Note that the interior in \eqref{eq:70-temp-convint} is an interior in the space $\R^n\times \szn$ and not in $\R\times\R^n\times \szn$.}
	\begin{equation} \label{eq:70-temp-convint}
		(\vm^\ast,\mU^\ast) \in \interior{\big((K_{\rho^\ast})^\co\big)} \ed
	\end{equation}
	
	Since $(\rho^\ast,\vm^\ast,\mU^\ast)\in \interior{\sU}$, there exists $\delta>0$ such that $(\rho,\vm,\mU)\in \sU$ for all $(\rho,\vm,\mU)\in \R\times\R^n\times\szn$ with $|(\rho,\vm,\mU) - (\rho^\ast,\vm^\ast,\mU^\ast)| <\delta$. Hence for all $(\vm,\mU)\in \R^n\times\szn$ with $|(\vm,\mU) - (\vm^\ast,\mU^\ast)| <\delta$ we duduce that $(\rho^\ast,\vm,\mU)\in \sU$. Together with \eqref{eq:71-temp-convint} this implies that $(\rho^\ast,\vm,\mU)\in (K\cap\{\rho=\rho^\ast\})^\co$. Proposition~\ref{prop:convhull=convcombis} yields $\un{N}\in \N$ and 
	$$
		\big(\mu_i,(\rho^\ast,\un{\vm}_i,\un{\mU}_i)\big)\ \in\  \R^+\times \big(K\cap\{\rho=\rho^\ast\}\big)
	$$ 
	for $i=1,...,\un{N}$ such that $\sum_{i=1}^{\un{N}} \mu_i=1 $ and 
	$$
		(\rho^\ast,\vm,\mU)=\sum_{i = 1}^{\un{N}} \mu_i (\rho^\ast,\un{\vm}_i,\un{\mU}_i) \ed
	$$ 
	In particular $(\vm,\mU)=\sum_{i = 1}^{\un{N}} \mu_i (\un{\vm}_i,\un{\mU}_i)$ and $(\un{\vm}_i,\un{\mU}_i)\in K_{\rho^\ast}$ for all $i=1,...,\un{N}$. Therefore $(\vm,\mU)\in (K_{\rho^\ast})^\co$ which shows \eqref{eq:70-temp-convint}.
	
	Let us now apply Proposition~\ref{prop:caratheodory-inner-points} to obtain $\hat{N}\in\N$ and $(\hat{\vm}_j,\hat{\mU}_j) \in K_{\rho^\ast}$ for $j=1,...,\hat{N}$ such that 
	\begin{equation} \label{eq:72-temp-convint}
		(\vm^\ast,\mU^\ast) \in \interior{\Big(\big\{(\hat{\vm}_1,\hat{\mU}_1),...,(\hat{\vm}_{\hat{N}},\hat{\mU}_{\hat{N}})\big\}^\co\Big)} \ed
	\end{equation}
	
	Note that the $(\hat{\vm}_j,\hat{\mU}_j)\in K_{\rho^\ast}$ can be perturbed\footnote{This idea was orginally used by \name{De~Lellis} and \name{Sz{\'e}kelyhidi} \cite[Proof of Lemma 6]{DelSze10}.} such that $\hat{\vm}_1,...,\hat{\vm}_{\hat{N}}$ are pairwise disjoint and $(\hat{\vm}_j,\hat{\mU}_j)$ still lie in $K_{\rho^\ast}$. Indeed for any $\vm\in \R^n$ and $\alpha>0$ we obtain by definition of $K$, see \eqref{eq:K},
	$$
		\left(\hat{\vm}_j + \alpha \vm , \frac{(\hat{\vm}_j + \alpha\vm)\otimes(\hat{\vm}_j + \alpha\vm)}{\rho^\ast} + \left(p(\rho^\ast) - \frac{2c}{n}\right)\id \right) \in K_{\rho^\ast} \ed
	$$
	Furthermore
	$$
	\big| (\hat{\vm}_j + \alpha \vm) - \hat{\vm}_j \big| = \alpha |\vm| \ec 
	$$
	and 
	\begin{align*}
		&\left\|\frac{(\hat{\vm}_j + \alpha\vm)\otimes(\hat{\vm}_j + \alpha\vm)}{\rho^\ast} + \left(p(\rho^\ast) - \frac{2c}{n}\right)\id - \hat{\mU}_j \right\| \\
		&= \left\|\frac{(\hat{\vm}_j + \alpha\vm)\otimes(\hat{\vm}_j + \alpha\vm)}{\rho^\ast} - \frac{\hat{\vm}_j\otimes \hat{\vm}_i}{\rho^\ast} \right\| \\
		&\leq \alpha \left\|\frac{\vm\otimes\hat{\vm}_j + \hat{\vm}_j\otimes\vm}{\rho^\ast} \right\| + \alpha^2 \left\|\frac{\vm\otimes\vm}{\rho^\ast}\right\| \ec
	\end{align*}
	where we again used the definition of $K$. Thus \eqref{eq:72-temp-convint} implies that if $\alpha$ is sufficiently small, then \eqref{eq:72-temp-convint} still holds for the perturbed points $(\hat{\vm}_j,\hat{\mU}_j)$.
	
	Define $\hat{\rho}_j:=\rho^\ast$ for all $j=1,...,\hat{N}$. Moreover note, that \eqref{eq:72-temp-convint} yields $\hat{N}\geq 2$ and existence of $\mu_1,...,\mu_{\hat{N}}\in\R^+$ such that $(\rho^\ast,\vm^\ast,\mU^\ast) = \sum_{j=1}^{\hat{N}} \mu_j (\hat{\rho}_j,\hat{\vm}_j,\hat{\mU}_j)$.
	
	Next we construct a new family $\big\{\big(\tau_i,(\rho_i,\vm_i,\mU_i)\big)\big\}_{i=1,...,N}$ as in Proposition~\ref{prop:complete-wc}, where $N:=2^{\hat{N}-1}\geq 2$. Proposition~\ref{prop:complete-wc} is applicable since $(\hat{\rho}_j,\hat{\vm}_j,\hat{\mU}_j)\in K\cap \{\rho=\rho^\ast\}$ for all $j=1,...,\hat{N}$ and $\Lambda$ is complete with respect to $K\cap \{\rho=\rho^\ast\}$. Hence we obtain that  
	\begin{itemize}
		\item the family $\big\{\big(\tau_i,(\rho_i,\vm_i,\mU_i)\big)\big\}_{i=1,...,N}$ satisfies the $H_N$-condition,
		\item for all $i\in\{1,...,N\}$ there exists $j\in\{1,...,\hat{N}\}$ such that $\rho_i=\hat{\rho}_j = \rho^\ast >R$ and 
		\item $\sum_{i = 1}^N \tau_i (\rho_i,\vm_i,\mU_i) = \sum_{j = 1}^{\hat{N}} \mu_j (\hat{\rho}_j,\hat{\vm}_j,\hat{\mU}_j) = (\rho^\ast,\vm^\ast,\mU^\ast)$.
	\end{itemize}

	It remains to show that the family $\big\{\big(\tau_i,(\rho_i,\vm_i,\mU_i)\big)\big\}_{i=1,...,N}$ satisfies the $m_N$-condition. To do so, we have to look into the proof of Proposition~\ref{prop:complete-wc} again. We are done once we have shown that the family \eqref{eq:complete-wc-family} satisfies the $m_{2^{k-1}}$-condition for $k=2,...,\hat{N}$, where 
	\begin{align*}
		\vp_i &= (\rho_i,\vm_i,\mU_i) \ec \\
		\vq_j &= (\hat{\rho}_j,\hat{\vm}_j,\hat{\mU}_j) \ed
	\end{align*} 
	
	For the induction basis, let $k=2$. Then the family \eqref{eq:complete-wc-family} reads
	$$
		\left\{\left( \frac{\tau_1}{\mu_1 + \mu_2} , \Big(\mu_1+\mu_2\Big) \vp_1 + \sum_{j = k+1}^{\hat{N}} \mu_j \vq_j \right) ,\left( \frac{\tau_2}{\mu_1 + \mu_2} , \Big(\mu_1+\mu_2\Big) \vp_2 + \sum_{j = k+1}^{\hat{N}} \mu_j \vq_j \right) \right\}
	$$
	where $\vp_1=\vq_1$ and $\vp_2=\vq_2$. Hence
	$$
		(\mu_1+\mu_2) (\vm_2-\vm_1) = (\mu_1+\mu_2) (\hat{\vm}_2-\hat{\vm}_1) \neq \vz
	$$
	since the $\hat{\vm}_j$ are pairwise disjoint, i.e. the $m_2$-condition holds. 
	
	Let now $3\leq k\leq \hat{N}$. It remains to prove that 
	$$
		\left(\sum_{j = 1}^k \mu_j\right) (\vm_i - \hat{\vm}_k) \neq \vz
	$$
	for all $i=1,...,2^{k-2}$. By construction there exists $\ell<k$ for all $i=1,...,2^{k-2}$ such that $\vp_i=\vq_\ell$, i.e. $\vm_i=\hat{\vm}_\ell$. Together with the fact that the $\hat{\vm}_j$ are pairwise disjoint, this yields the claim.
\end{proof}

The fact that the family in Lemma~\ref{lemma:existence-family-nodens} satisfies the $m_N$-condition was not easy to prove and made a lot of effort. However the variants of Lemmas \ref{lemma:UstattK} and \ref{lemma:U-oscillatory-lemma} are quite simple.

\begin{lemma} \label{lemma:UstattK-nodens}
	Let the assumptions of Lemma~\ref{lemma:UstattK} be true. Assume in addition to that, that the family $\big\{\big(\tau_i,(\rho_i,\vm_i,\mU_i)\big)\big\}_{i=1,...,N}$ satisfies
	\begin{itemize}
		\item the $m_N$-condition and
		\item $\rho_i=\rho^\ast$ for all $i=1,...,N$.
	\end{itemize}
	Then the family $\big\{\big(\tau_i,(\hat{\rho}_i,\hat{\vm}_i,\hat{\mU}_i)\big)\big\}_{i=1,...,N}$ given by \eqref{eq:defn-hat} additionlly satisfies
	\begin{enumerate} \setcounter{enumi}{4}
		\item the $m_N$-condition and
		\item $\hat{\rho}_i=\rho^\ast$ for all $i=1,...,N$.
	\end{enumerate}
\end{lemma}

\begin{proof}
	\begin{enumerate} \setcounter{enumi}{4}
		\item Let first $N=2$. Since the family $\big\{\big(\tau_1,(\rho_1,\vm_1,\mU_1)\big),\big(\tau_2,(\rho_2,\vm_2,\mU_2)\big)\big\}$ satisfies the $m_2$-condition, it holds that $\vm_ 2-\vm_1\neq \vz$. Hence 
		$$
			\hat{\vm}_2 - \hat{\vm}_1 = \tau \vm^\ast + (1-\tau) \vm_2 - \tau \vm^\ast - (1-\tau) \vm_1 = (1-\tau) (\vm_2 - \vm_1) \neq \vz \ed
		$$
		Let now $N>2$. Because the family $\big\{\big(\tau_i,(\rho_i,\vm_i,\mU_i))\big)\big\}_{i=1,...,N}$ satisfies the $m_N$-condition, it holds (after relabeling if necessary) that $\vm_2-\vm_1\neq \vz$, 
		$$
			(\rho_2,\vm_2,\mU_2)-(\rho_1,\vm_1,\mU_1) \ \in\ \Lambda 
		$$ 
		and the family \eqref{eq:defn-hn-iteration-b} satisfies both the $m_{N-1}$- and the $H_{N-1}$-condition. 
		
		We have already shown in the proof of Lemma~\ref{lemma:UstattK} that $(\hat{\rho}_2,\hat{\vm}_2,\hat{\mU}_2)-(\hat{\rho}_1,\hat{\vm}_1,\hat{\mU}_1)\in \Lambda$ and the family in \eqref{eq:19-temp-convint} satisfies the $H_{N-1}$-condition. The same computation as in the case $N=2$ yields that $\hat{\vm}_2 - \hat{\vm}_1\neq \vz$. Finally by induction hypothesis the family in \eqref{eq:19-temp-convint} satisfies the $m_{N-1}$-condition. This proves that the family $\big\{\big(\tau_i,(\hat{\rho}_i,\hat{\vm}_i,\hat{\mU}_i)\big)\big\}_{i=1,...,N}$ satisfies the $m_N$-condition.
		
		\item From $\rho_i=\rho^\ast$ we simply deduce that 
		$$
			\hat{\rho}_i = \tau\rho^\ast + (1-\tau) \rho_i = \tau\rho^\ast + (1-\tau) \rho^\ast = \rho^\ast\ed
		$$
	\end{enumerate}
\end{proof}

Finally the version of Lemma~\ref{lemma:U-oscillatory-lemma} which we need here, is proven by using Proposition~\ref{prop:operators}~\ref{item:operators.c}. 

\begin{lemma} \label{lemma:U-oscillatory-lemma-nodens}
	Let the assumptions of Lemma~\ref{lemma:U-oscillatory-lemma} be true. Assume in addition to that, that the family $\big\{\big(\tau_i,(\rho_i,\vm_i,\mU_i)\big)\big\}_{i=1,...,N}$ satisfies
	\begin{itemize}
		\item the $m_N$-condition and
		\item $\rho_i=\rho^\ast$ for all $i=1,...,N$.
	\end{itemize}
	Then we can achieve that the sequence of oscillations \eqref{eq:U-oscillations} has the additional property that $\til{\rho}_k\equiv 0$ for all $k\in\N$. 
\end{lemma} 

\begin{proof}
	Let first $N=2$. By assumtion $\rho_1=\rho_2=\rho^\ast$ and $\vm_2 - \vm_1 \neq \vz$ (since the $m_2$-condition holds). Hence if we proceed as in the proof of Lemma~\ref{lemma:U-oscillatory-lemma}, we find $\ov{\rho}=0$ and $\ov{\vm}\neq \vz$. If $\veta$ was parallel to $\ve_t$, then \eqref{eq:6-temp-convint} would yield the contradiction $\ov{\vm}=\vz$. Therefore $\veta$ is not parallel to $\ve_t$. Thus Proposition~\ref{prop:operators}~\ref{item:operators.c} tells us that $\opL_{\rho}\equiv 0$, i.e. $\til{\rho}_k\equiv 0$ for all $k\in\N$, see \eqref{eq:defn-perturbation}. 
	
	Let $N\geq 3$. Since the family $\big\{\big(\tau_i,(\rho_i,\vm_i,\mU_i)\big)\big\}_{i=1,...,N}$ satisfies the $m_N$-condition by assumption, we obtain that $\vm_2-\vm_1 \neq \vz$ for $\vm_1,\vm_2$ in \eqref{eq:11-temp-convint}, and furthermore the family \eqref{eq:13-temp-convint} satisfies the $m_{N-1}$-condition. Moreover $\rho_a=\rho^\ast$ due to the fact that $\rho_1=\rho_2=\rho^\ast$, see \eqref{eq:14-temp-convint}. Therefore the induction hypothesis yields $\til{\rho}_{\oscA,k}\equiv 0$ for all $k\in \N$. In addition to that the family in \eqref{eq:80-temp-convint} satisfies the $m_2$-condition because $\vm_2-\vm_1 \neq \vz$, see above. Hence $\til{\rho}_{\oscB,\ell}\equiv 0$ for all $\ell\in \N$, due to the induction basis. Together we obtain for all $k\in\N$
	$$ 
		\til{\rho}_k = \til{\rho}_{\oscA,k} + \til{\rho}_{\oscB,\ell(k)} \equiv 0 \ed
	$$
\end{proof}

We are ready to prove Lemma~\ref{lemma:K-oscillatory-lemma-nodens} and Proposition~\ref{prop:pert-prop-nodens}.

\begin{proof}[Proof of Lemma~\ref{lemma:K-oscillatory-lemma-nodens}]
	First, the family $\big\{\big(\tau_i,(\rho_i,\vm_i,\mU_i)\big)\big\}_{i=1,...,N}$ which is used in the proof of Lemma~\ref{lemma:K-oscillatory-lemma} can be chosen in such a way that in addition the $m_N$-condition is fulfilled and $\rho_i=\rho^\ast$ for all $i=1,...,N$, due to Lemma~\ref{lemma:existence-family-nodens}. 
	
	Second, the family $\big\{\big(\tau_i,(\hat{\rho}_i,\hat{\vm}_i,\hat{\mU}_i)\big)\big\}_{i=1,...,N}$ defined in \eqref{eq:26-temp-convint} still satisfies additionally the $m_N$-condition and $\hat{\rho}_i=\rho^\ast$ for all $i=1,...,N$ according to Lemma~\ref{lemma:UstattK-nodens}.
	
	Finally Lemma~\ref{lemma:U-oscillatory-lemma-nodens} says that we can achieve $\til{\rho}_k\equiv 0$ for all $k\in \N$ in addition to the properties given in Lemma~\ref{lemma:K-oscillatory-lemma}. This proves Lemma~\ref{lemma:K-oscillatory-lemma-nodens}.
\end{proof}

\begin{proof}[Proof of Proposition~\ref{prop:pert-prop-nodens}] 
	We proceed as in the proof of Proposition~\ref{prop:pert-prop} where now we make use of Lemma~\ref{lemma:K-oscillatory-lemma-nodens}. In view of \eqref{eq:defn-pert-sequence} this immediately yields $\rho_{\pert,k} \equiv \rho$ for all $k\in\N$ because $\til{\rho}_{i,\valpha,k} \equiv 0$ for all $(i,\valpha)\in\mathcal{J}$ according to Lemma~\ref{lemma:K-oscillatory-lemma-nodens}. Thus $\rho_{\pert} \equiv \rho$.
\end{proof}

%% file: DissertationCh5.tex
\chapter{Application: Infinitely Many Solutions of the Initial Boundary Value Problem for Barotropic Euler} \label{chap:appl-ibvp}
\chaptermark{Solutions of the Initial Boundary Value Problem}

In this chapter we consider the initial boundary value problem for the barotropic Euler system \eqref{eq:baro-euler-pv-dens}, \eqref{eq:baro-euler-pv-mom} with any given initial data \eqref{eq:baro-initial} and impermeability boundary condition \eqref{eq:impermeability} on a bounded domain $\Omega\subset\R^n$. What is meant by an (admissible) weak solution to this problem is defined in Definition~\ref{defn:aws-baro-bdd}. 

We are going to show how convex integration is used to produce solutions to this initial boundary value problem. With the help of Theorem~\ref{thm:convint} we only present in detail a less impressive result in the sense that the solutions, which are obtained here, are only weak solutions and not admissible, see Section~\ref{sec:ibvp-weak}. Furthermore this result only works for a narrow class of initial data. The proof of this result boils down to finding a subsolution as required by Theorem~\ref{thm:convint} which additionally complies with the initial and boundary condition. 

In view of much more general results in the literature, we indicate in Sections \ref{sec:ibvp-adm} and \ref{sec:ibvp-other} how Theorem~\ref{thm:convint} has to be modified in order to get better results. Note that in this book we actually focus on the application of Theorem~\ref{thm:convint} to the so-called Riemann problem, which is considered in Chapter~\ref{chap:appl-riemann}. This is the reason why Theorem~\ref{thm:convint} is formulated in such a way that it can easily be applied there. However this formulation may be viewed as ``not ideal'' as far as the application in the current chapter is concerned, in the sense that it does not yield most general results. 

In order to apply Theorem~\ref{thm:convint} to the initial boundary value problem under consideration, we set $\Gamma:= (0,T)\times \Omega$. In this chapter we allow $T=\infty$. However for the closure of $(0,T)$ we will write for simplicity $[0,T]$ which actually means $[0,T]$ if $T\in \R$ and $[0,\infty)$ if $T=\infty$.

\section{A Simple Result on Weak Solutions} \label{sec:ibvp-weak}

The following statement can be easily derived from Theorem~\ref{thm:convint}.

\begin{thm} \label{thm:ibvp-many-weak}
	Assume there exist $r,c>0$ and $(\rho_0,\vm_0,\mU_0)\in C^1\big(\closure{\Gamma};\R^+ \times \R^n\times \szn\big)$ with the following properties:
	\begin{itemize}
		\item The assumptions of Theorem~\ref{thm:convint}, i.e. \eqref{eq:p0-pde1} - \eqref{eq:p0-dens-bdd}, hold;
		\item The initial condition is fulfilled, i.e. 
		\begin{equation} \label{eq:ibvp-ic}
			(\rho_0,\vm_0)(0,\cdot) = (\rho_\init,\rho_\init \vu_\init)\es
		\end{equation}
		\item The boundary condition is satisfied in the following sense:
		\begin{align}
			\vm_0 \cdot \vn \big|_{\partial\Omega} &= 0 \qquad \text{ and} \label{eq:ibvp-bc1} \\
			(\mU_0\cdot \vphi)\cdot \vn\big|_{\partial \Omega} &= 0 \qquad \text{ for all }\vphi \in \Cc\big([0,T) \times \closure{\Omega};\R^n\big) \text{ with }\vphi\cdot \vn\big|_{\partial \Omega}=0\ed \label{eq:ibvp-bc2}
		\end{align}
	\end{itemize}
	Then there exist infinitely many weak solutions (not necessarily admissible\footnote{In fact these solutions \emph{are} not admissible, see Section~\ref{sec:ibvp-adm} for details.}) of the initial boundary value problem \eqref{eq:baro-euler-pv-dens}, \eqref{eq:baro-euler-pv-mom}, \eqref{eq:baro-initial}, \eqref{eq:impermeability}.
\end{thm}

\begin{proof}
	An application of Theorem~\ref{thm:convint} yields infinitely many bounded functions $(\rho,\vm)\in L^\infty\big((0,T)\times \Omega;\R^+\times \R^n\big)$ such that in particular property \ref{item:convint.a} of Theorem~\ref{thm:convint} holds. In other words \eqref{eq:sol-pde1} and \eqref{eq:sol-pde2} hold for all test functions $(\phi,\vphi) \in \Cc\big([0,T] \times \closure{\Omega}; \R\times \R^n\big)$. For each such $\vm$, define $\vu:= \frac{\vm}{\rho}$. Note that $\rho>0$ a.e. on $(0,T)\times \Omega$ due to \eqref{eq:sol-dens-bdd}. In order to show that each pair $(\rho,\vu)$ is a weak solution in the sense of Definition~\ref{defn:aws-baro-bdd}, let $(\phi,\vphi) \in \Cc\big([0,T) \times \closure{\Omega}; \R\times \R^n\big)$ be arbitrary test functions with $\vphi\cdot \vn\big|_{\partial \Omega}=0$. From \eqref{eq:ibvp-ic}, \eqref{eq:ibvp-bc1} and \eqref{eq:sol-pde1} we obtain
	\begin{align*}
		&\int_0^T \int_{\Omega} \Big[\rho \partial_t \phi + \rho\vu\cdot\Grad \phi\Big]\dx\dt + \int_{\Omega} \rho_\init\phi(0,\cdot) \dx \\
		&= \int_0^T \int_{\Omega} \Big[\rho \partial_t \phi + \vm\cdot\Grad \phi\Big]\dx\dt + \int_{\Omega} \rho_0\phi(0,\cdot) \dx - \int_0^T \int_{\partial\Omega} \vm_0 \cdot \vn \, \phi \dS_\vx \dt \\
		&= 0 \ec
	\end{align*}	
	whereas \eqref{eq:ibvp-ic}, \eqref{eq:ibvp-bc2}, the fact that $\vphi\cdot \vn\big|_{\partial \Omega}=0$, and \eqref{eq:sol-pde2} yield
	\begin{align*}
		&\int_0^T \int_{\Omega} \Big[\rho\vu \cdot\partial_t \vphi + \rho\vu\otimes\vu:\Grad \vphi + p(\rho)\Div \vphi\Big]\dx\dt + \int_{\Omega} \rho_\init\vu_\init\cdot\vphi(0,\cdot) \dx \\
		&= \int_0^T \int_{\Omega} \Big[\vm \cdot\partial_t \vphi + \frac{\vm\otimes\vm}{\rho}:\Grad \vphi + p(\rho)\Div \vphi\Big]\dx\dt + \int_{\Omega} \vm_0\cdot\vphi(0,\cdot) \dx \\
		&\qquad\qquad - \int_0^T \int_{\partial\Omega} \left[ (\mU_0\cdot \vphi)\cdot \vn + \frac{2c}{n} \vphi\cdot \vn\right] \dS_\vx \dt \\
		&= 0 \ed
	\end{align*}		
	In other words \eqref{eq:baro-euler-weak-bdd-dens} and \eqref{eq:baro-euler-weak-bdd-mom} hold, i.e. each pair $(\rho,\vu)$ is in fact a weak solution.
\end{proof}

With Theorem \ref{thm:ibvp-many-weak} at hand, one finds simple examples of initial data for which there exist infinitely many solutions, see e.g. the following lemma. 

\begin{cor} \label{cor:ibvp-many-weak-divcond}
	Let $(\rho_\init,\vu_\init)\in C^1(\closure{\Omega},\R^+\times \R^n)$ where $\vu_\init\cdot \vn\big|_{\partial\Omega} = 0$. Moreover we assume that $\Div(\rho_\init\vu_\init) = 0$. Then the initial boundary value problem \eqref{eq:baro-euler-pv-dens}, \eqref{eq:baro-euler-pv-mom}, \eqref{eq:baro-initial}, \eqref{eq:impermeability} has infinitely many weak solutions. 
\end{cor} 

\begin{proof} 
	Define
	\begin{align*}
		\rho_0(t,\cdot) &:= \rho_\init \ec \\
		\vm_0(t,\cdot) &:= \rho_\init \vu_\init \ec \\
		\mU_0(t,\cdot) &:= \mZ
	\end{align*}
	for all $t\in [0,T]$. Together with the assumption $\Div(\rho_\init\vu_\init) = 0$ it is easy to show that \eqref{eq:p0-pde1}, \eqref{eq:p0-pde2} hold. Since $\closure{\Omega}$ is compact and $\rho_0,\vm_0,\mU_0$ do not depend on $t$, there exists $c>0$ such that $e(\rho_0,\vm_0,\mU_0)(t,\vx) < c$ for all $(t,\vx)\in \closure{\Gamma}$, i.e. \eqref{eq:p0-subs} is fulfilled. Analogously there exists $r>0$ such that $\rho_0(t,\vx)>r$ for all $(t,\vx)\in \closure{\Gamma}$, which shows \eqref{eq:p0-dens-bdd}.
	
	Moreover \eqref{eq:ibvp-ic} and \eqref{eq:ibvp-bc2} are satisfied by construction, whereas \eqref{eq:ibvp-bc1} holds by assumption. Thus Theorem~\ref{thm:ibvp-many-weak} yields the claim.
\end{proof}

In order to show that for \emph{all} initial data there exist infinitely many weak solutions (not necessarily admissible), one needs a more refined version convex integration rather than Theorem \ref{thm:convint}. In particular one has to replace the constant $c$ by a function $\ov{e}$ which depends on $t$ and $\vx$ as indicated in Subsection~\ref{subsec:convint-prel-adjusting}, see also Section \ref{sec:ibvp-other} and the references cited there.

\section[Possible Improvements to Obtain Admissible Weak Solutions]{Possible Improvements to Obtain Admissible Weak Solutions%
	\sectionmark{Admissible Weak Solutions}}
\label{sec:ibvp-adm}
\sectionmark{Admissible Weak Solutions}

Note that Theorem~\ref{thm:convint} does not allow to produce admissible solutions when using a cylindrical space-time domain $\Gamma=(0,T)\times \Omega$. The reason for this fact is that the requirement that the subsolution $(\rho_0,\vm_0,\mU_0)$ lies in $C^1\big([0,T]\times \closure{\Omega};\R^+ \times \R^n\times \szn\big)$, see Theorem~\ref{thm:convint}, is an obstacle in order to achieve admissible solutions. Indeed this implies that $(\rho_0,\vm_0,\mU_0)(0,\cdot) \in C^1\big( \closure{\Omega};\R^+ \times \R^n\times \szn\big)$ and due to \eqref{eq:sol-pde1} and \eqref{eq:sol-pde2} we must require $(\rho_0,\vm_0)(0,\cdot) = (\rho_\init,\rho_\init \vu_\init)$ in order to satisfy the initial condition. Hence $\rho_\init,\vu_\init$ are $C^1$ which means that there exists a unique strong solution at least on a short time interval, which is even unique in the class of admissible weak solutions due to the \emph{weak-strong-uniqueness} principle. 

This problem can be overcome by requiring the subsolution $(\rho_0,\vm_0,\mU_0)$ in Theorem~\ref{thm:convint} to lie in $C^1\big((0,T)\times \closure{\Omega};\R^+ \times \R^n\times \szn\big)$. Mind the small difference: Now the time interval $(0,T)$ is an open interval. Then one has to prescribe the initial values in \eqref{eq:sol-pde1}  and \eqref{eq:sol-pde2} that are included in the boundary integrals over $\partial\Gamma$ since $\rho_0,\vm_0,\mU_0$ are not defined for $t=0$. Similarly convex integration has been carried out in the literature, see e.g. \name{De~Lellis}-\name{Sz{\'e}kelyhidi} \cite[Proposition 2]{DelSze10} for the incompressible case, \name{Chiodaroli} \cite[Proposition 4.1]{Chiodaroli14} or \name{Feireisl} \cite{Feireisl14}.

Another problem is that we must guarantee that the solutions additionaly satisfy the energy inquality \eqref{eq:baro-euler-weak-bdd-admissibility}. Here the ``trace-condition'' \eqref{eq:sol-trace} is helpful. As already pointed out in a remark in Subsection~\ref{subsec:convint-prel-adjusting}, in the case of a monoatomic gas, i.e. $p(\rho)=\rho^{\frac{2}{n}+1}$, we have $P(\rho)=\frac{n}{2} p(\rho)$ and hence \eqref{eq:sol-trace} turns into
$$
\frac{|\vm|^2}{2\rho} + P(\rho) = c \qquad \text{ for a.e. }(t,\vx)\in (0,T)\times \Omega \ed
$$
Note that the left-hand side is the energy, i.e. \eqref{eq:sol-trace} says that the energy is constant for a.e. $(t,\vx)\in(0,T)\times \Omega$.

However this is not enough to make the energy inquality valid as we don't know anything about the behaviour of the energy flux. To solve this issue, one may use the fixed-density-version Theorem~\ref{thm:convint-nodens} rather than Theorem~\ref{thm:convint}. Then it is not even necessary to study a monoatomic gas. We find using \eqref{eq:sol-trace} 
\begin{equation} \label{eq:5-temp-1} 
	\frac{|\vm|^2}{2\rho} + P(\rho) = \frac{|\vm|^2}{2\rho} +\frac{n}{2}p(\rho) + P(\rho) - \frac{n}{2}p(\rho) = c + P(\rho_0) - \frac{n}{2}p(\rho_0)
\end{equation}
for a.e. $(t,\vx)\in (0,T)\times \Omega$. For simplicity we choose $\rho_\init\equiv \ov{\rho} = \const$ and also $\rho_0\equiv \ov{\rho}$. This way we obtain from \eqref{eq:sol-pde1} together with the Divergence Theorem (Proposition~\ref{prop:not-divergence}) and the impermeability boundary condition, that 
\begin{equation} \label{eq:5-temp-2}
\int_0^T \int_\Omega \vm\cdot \Grad\phi \dx\dt = 0 
\end{equation}
for all test functions $\phi\in \Cc\big([0,T]\times \closure{\Omega};\R\big)$. With \eqref{eq:5-temp-2} we are able to handle the energy flux. Indeed plugging \eqref{eq:5-temp-1} and \eqref{eq:5-temp-2} into the left-hand side of \eqref{eq:baro-euler-weak-bdd-admissibility}, we obtain
\begin{align*}
	&\int_0^T \int_{\Omega} \left[\bigg(c + P(\rho_0) - \frac{n}{2}p(\rho_0)\bigg) \partial_t \varphi + \frac{1}{\rho_0}\bigg(c + P(\rho_0) +  \left(1- \frac{n}{2}\right) p(\rho_0)\bigg)\vm\cdot\Grad \varphi \right]\dx\dt \\ 
	&\qquad + \int_{\Omega} \bigg(\half\rho_\init|\vu_\init|^2 + P(\rho_\init)\bigg) \varphi(0,\cdot) \dx \\
	&= -\int_{\Omega} \bigg(c + P(\ov{\rho}) - \frac{n}{2}p(\ov{\rho})\bigg) \varphi(0,\cdot) \dx + \int_{\Omega} \bigg(\half\ov{\rho} |\vu_\init|^2 + P(\ov{\rho})\bigg) \varphi(0,\cdot) \dx \\
	&= \int_{\Omega} \bigg(-c  - \frac{n}{2}p(\ov{\rho}) + \half\ov{\rho} |\vu_\init|^2\bigg) \varphi(0,\cdot) \dx
\end{align*} 
for all test functions $\varphi \in \Cc\big([0,T) \times \closure{\Omega};\R^+_0\big)$. Note that this would be equal to zero if 
\begin{equation} \label{eq:5-temp-3}
	\half\ov{\rho} |\vu_\init|^2 + \frac{n}{2} p(\ov{\rho}) = c \qquad \text{ for a.e. }\vx\in \Omega \ed
\end{equation}
In other words we must require the energy to be continuous at $t=0$.

However for simple choices of the subsolution like the one in Corollary~\ref{cor:ibvp-many-weak-divcond} this is generally not true. The reason for this is the fact that one first fixes the subsolution $(\rho_0,\vm_0,\mU_0)$ and then chooses $c>0$ sufficiently large to achieve \eqref{eq:p0-subs}. This typically leads to a jump of the energy at $t=0$. If $c$ is already fixed by \eqref{eq:5-temp-3}, then there is no such jump, but on the other it is more difficult to guarantee that \eqref{eq:p0-subs} holds. In fact in the literature subsolutions which satisfy \eqref{eq:5-temp-3} are constructed using convex integration once more, see e.g. \name{De~Lellis} and \name{Sz{\'e}kelyhidi} \cite[Section~5]{DelSze10} for the incompressible case, \name{Chiodaroli} \cite[Section 7]{Chiodaroli14} or \name{Feireisl} \cite[Theorem~1.4]{Feireisl14}. Note that it is however not possible to find a subsolution fulfilling \eqref{eq:5-temp-3} for all initial data, even if one replaces the constant $c$ by a function $\ov{e}$. Instead one constructs suitable initial data and a corresponding subsolution $(\rho_0,\vm_0,\mU_0)$ which fulfills \eqref{eq:5-temp-3} simultaneously.

\section{Further Possible Improvements} \label{sec:ibvp-other}

Let us finish this chapter with mentioning how Theorem~\ref{thm:convint} can be further improved.

As indicated in Section~\ref{sec:ibvp-weak} and in Subsection~\ref{subsec:convint-prel-adjusting} one could replace $c$ by a function $\ov{e}$ which depends on $t$ and $\vx$. For example if one wants to find a possibly large class of initial data that admit infinitely many admissible weak solutions, i.e. \eqref{eq:5-temp-3} must hold, then the requirement that $\ov{e}\equiv c=\const$ is quite restrictive. Indeed there are not many initial data for which the left-hand side of \eqref{eq:5-temp-3} is constant. Note that in fact in many papers on convex integration for compressible Euler that are available in the literature, e.g. \name{De~Lellis}-\name{Sz{\'e}kelyhidi} \cite{DelSze10}, \name{Chiodaroli} \cite{Chiodaroli14} and \name{Feireisl} \cite{Feireisl14}, the trace which corresponds to \eqref{eq:convint-temp-trace} in our framework, is considered as not necessarily constant.

Another issue is the following. It is natural to require weak solutions to be continuous maps from $[0,T)$ to $L^\infty(\Omega)$ where the latter is endowed with the weak-$\ast$ topology. The corresponding function space is denoted by $C_{\rm weak\text{-}\ast}\big([0,T);L^\infty(\Omega)\big)$. In fact one can prove that weak solutions as defined in Section~\ref{sec:conslaws-ibvp} can be modified (if necessary) on a set of zero measure such that they lie in the space $C_{\rm weak\text{-}\ast}\big([0,T);L^\infty(\Omega)\big)$, see \name{Dafermos} \cite[Lemma~1.3.3]{Dafermos}. In other words the \emph{instantaneous values} $\vU(t,\cdot)$ are well-defined for all times $t\in[0,T)$ and the equation
\begin{align*} 
	\int_{t_0}^{t_1} \int_{\Omega} \Big(\vU \cdot \partial_t\vpsi + \mF(\vU) : \Grad \vpsi \Big)\dx \dt - \bigg[\int_\Omega \vU(t,\cdot) \cdot \vpsi(t,\cdot) \dx\bigg]_{t=t_0}^{t=t_1} \qquad &  \\
	- \int_{t_0}^{t_1} \int_{\partial\Omega} \vpsi \cdot \vF_{\partial \Omega} \dS_\vx \dt &= 0 
\end{align*}
holds for all test functions $\vpsi \in \Cc\big([0,T)\times\closure{\Omega};\R^m\big)$ and all $0\leq t_0  \leq t_1 < T$, rather than just \eqref{eq:conslaws-ivp-weak}, see also \cite[Lemma 1.3.3]{Dafermos}. In the context of the barotropic Euler equations, this means that every weak solution $(\rho,\vu)$ in the sense of Definition~\ref{defn:aws-baro-bdd} is also a weak solution in the sense described above. However one could ask for solutions which fulfill also the energy inequality in the sense above instead of \eqref{eq:baro-euler-weak-bdd-admissibility}. In other words one requires 
\begin{align*}
	\int_{t_0}^{t_1} \int_{\Omega} \left[\bigg(\half\rho|\vu|^2 + P(\rho)\bigg) \partial_t \varphi + \bigg(\half\rho|\vu|^2 + P(\rho) + p(\rho)\bigg)\vu\cdot\Grad \varphi \right]\dx\dt \qquad &  \\ 
	- \bigg[\int_{\Omega} \bigg(\half\rho|\vu|^2 + P(\rho)\bigg)(t,\cdot)\ \varphi(t,\cdot) \dx \bigg]_{t=t_0}^{t=t_1} &\geq 0 
\end{align*}
for all $\varphi \in \Cc\big([0,T) \times \closure{\Omega};\R^+_0\big)$ and all $0\leq t_0  \leq t_1 < T$. To include this in the convex integration method, the solutions we are looking for must satisfy \eqref{eq:sol-trace} not only for a.e. $(t,\vx)\in (0,T)\times \Omega$ but for all $t\in (0,T)$ and a.e. $\vx\in \Omega$. For the incompressible Euler system this has been done by \name{De~Lellis} and \name{Sz{\'e}kelyhidi} \cite{DelSze10}, see the beginning of Section 4 therein for a more extensive discussion of this issue. In order to achieve a similar result in the framework presented in Chapter~\ref{chap:convint} one needs to implement the ideas of \cite{DelSze10}.

%% file: DissertationCh6.tex
\chapter{Application: Riemann Initial Data in Two Space Dimensions} \label{chap:appl-riemann} 
\chaptermark{Riemann Problem}

In this chapter we consider the Euler equations -- both isentropic, i.e. barotropic with the particular pressure law \eqref{eq:isentropic-EOS}, and full -- on the whole two-dimensional space, i.e. $\Omega=\R^2$. Keep in mind the definitions of admissible weak solutions to the corresponding initial value problems: Definition~\ref{defn:aws-baro-whole} for isentropic, and Definition~\ref{defn:aws-full-whole} for the full Euler system. As we are considering the two-dimensional Euler equations in this chapter, we slightly differ from our notation used in the previous chapters: The components of the velocity and the spatial variable are from now on denoted by $\vu=(u,v)^\trans$ and $\vx=(x,y)^\trans$ respectively. 

Moreover we look at a special type of initial data, for a subtype of which we prove existence of infinitely many \emph{admissible} weak solutions. In particular these data are constant in each of the two half spaces, separated by a discontinuity along a straight line. Such data can be viewed as a one-dimensional Riemann problem which is extended to two space dimensions in a trivial way, i.e. constant with respect to the second dimension. More precisely the initial data for the isentropic Euler equations read
\begin{equation} \label{eq:riemann-init-isen}
	(\rho_\init,\vu_\init)(\vx) :=\left\{
	\begin{array}{ll}
		(\rho_-,\vu_-) & \text{ if }y<0\ec \\
		(\rho_+,\vu_+) & \text{ if }y>0\ec
	\end{array}
	\right. 
\end{equation}
with constants $\rho_\pm\in\R^+$ and $\vu_\pm\in\R^2$, and for the full Euler system 
\begin{equation} \label{eq:riemann-init-full}
	(\rho_\init,\vu_\init,p_\init)(\vx) :=\left\{
	\begin{array}{ll}
		(\rho_-,\vu_-,p_-) & \text{ if }y<0\ec \\
		(\rho_+,\vu_+,p_+) & \text{ if }y>0\ec
	\end{array}
	\right.
\end{equation}
with constants $\rho_\pm\in\R^+$, $\vu_\pm\in\R^2$ and $p_\pm\in\R^+$. Such type of initial data is illustrated in Figure~\ref{fig:riemann-init}. As indicated above, we denote the components of the velocities as $\vu_-=(u_-,v_-)^\trans$ and $\vu_+=(u_+,v_+)^\trans$. 

\begin{figure}[tb] 
	\centering
	\includegraphics[width=0.7\textwidth]{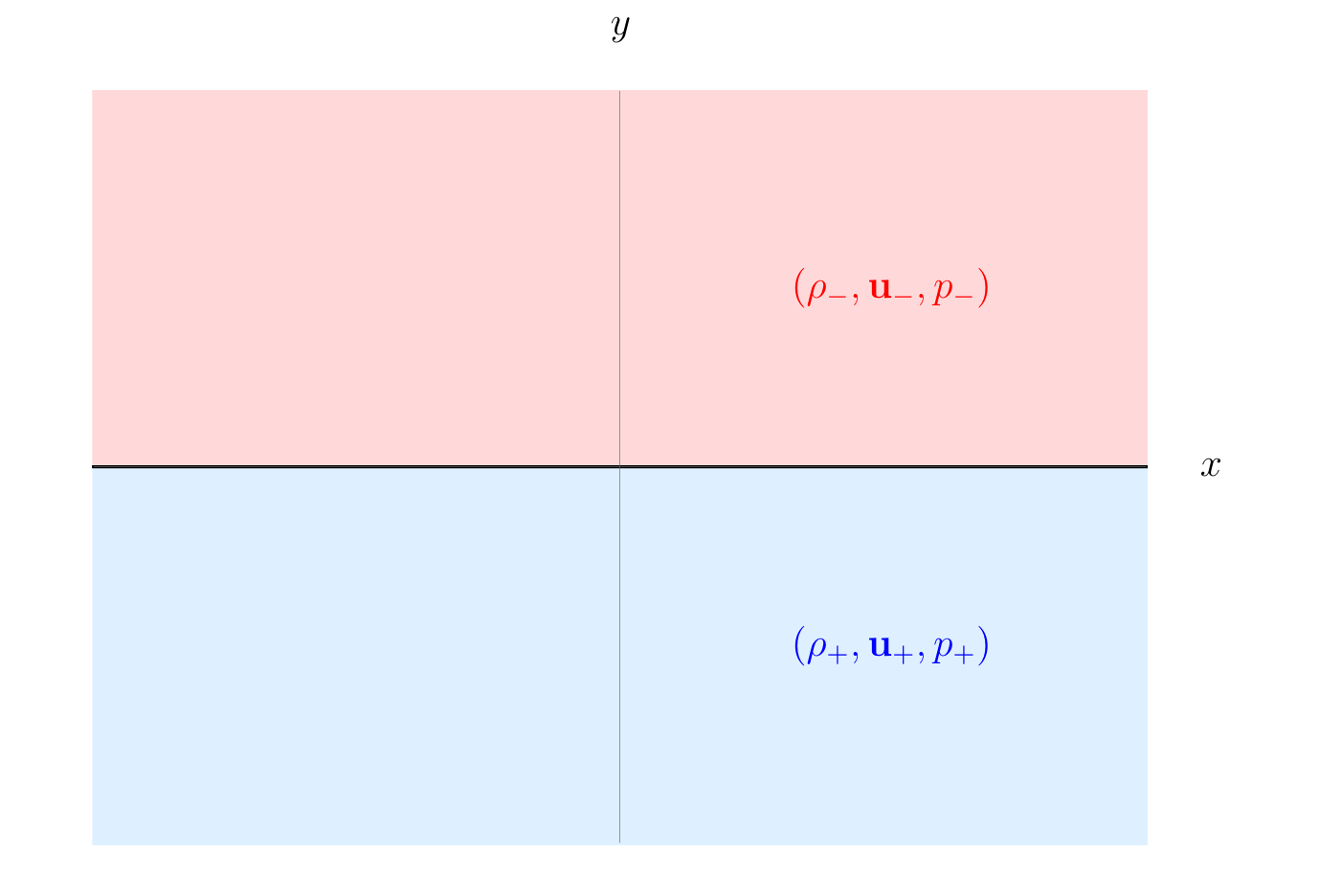} 
	\caption{Initial data considered in Chapter~\ref{chap:appl-riemann}, see \eqref{eq:riemann-init-full} and also \eqref{eq:riemann-init-isen}.} 
	\label{fig:riemann-init}
\end{figure}

One is able to solve the initial value problem for such data by considering a corresponding one-dimensional Riemann problem, which can be solved by classical, well-known methods, see e.g. the textbooks by \name{Dafermos} \cite[Chapters 7 - 9]{Dafermos} or \name{Smoller} \cite[Chapters 17 and 18]{Smoller}. This solution is admissible, self-similar (i.e. a function of a single variable $\frac{y}{t}$) and consists of constant states that are separated by shocks, rarefactions and contact discontinuities. Interpreting this one-dimensional solution -- i.e. a function of $(t,y)\in[0,\infty)\times \R$ -- as a function of $(t,\vx)\in[0,\infty)\times \R^2$ which is constant with respect to $x\in\R$, yields an admissible weak solution to the original two-dimensional problem as well. Hence the initial value problem under consideration (isentropic Euler \eqref{eq:baro-euler-pv-dens}, \eqref{eq:baro-euler-pv-mom}, \eqref{eq:isentropic-EOS} with initial data \eqref{eq:riemann-init-isen} or full Euler \eqref{eq:full-euler-pv-dens} - \eqref{eq:full-euler-pv-en} with initial data \eqref{eq:riemann-init-full}) has at least one admissible weak solution. This solution is referred to as the \emph{one-dimensional self-similar} or simply \emph{self-similar} solution. For more details on the self-similar solution see Subsections \ref{subsec:riemann-isen-1D} and \ref{subsec:riemann-full-1D}.

Initial data of the form \eqref{eq:riemann-init-isen} for the isentropic compressible Euler system have been first studied in the context convex integration by \name{Chiodaroli}, \name{De~Lellis} and \name{Kreml} \cite{ChiDelKre15}, who developed most of the ingredients of the non-uniqueness proof presented in this chapter. They proved that there exist data of the form \eqref{eq:riemann-init-isen} which lead to infinitely many admissible weak solutions using convex integration. In contrast to the self-similar solution, those solutions are genuinely two-dimensional. The work in \cite{ChiDelKre15} was inspired by a result by \name{Sz{\'e}kelyhidi} \cite{Szekelyhidi11} for the incompressible Euler system \eqref{eq:incomp-euler-divfree}, \eqref{eq:incomp-euler-velocity}. Here the shear flow was considered, i.e. 
$$ 
\vv_\init(\vx) :=\left\{
\begin{array}{rl}
	(-1,0)^\trans & \text{ if }y<0\ec \\
	(1,0)^\trans & \text{ if }y>0\ec
\end{array}
\right. 
$$ 
which can be viewed as a particular example of incompressible data of the type \eqref{eq:riemann-init-isen} or \eqref{eq:riemann-init-full}. It is simple to show that there is a stationary solution to this problem. However \name{Sz{\'e}kelyhidi} \cite{Szekelyhidi11} proved existence of infinitely many other solutions. Note that using a vanishing viscosity criterion one is able to rule out all the non-stationary solutions in this case, which was shown by \name{Bardos}, \name{Titi} and \name{Wiedemann} \cite{BarTitWie12}. 

Let us come back to the compressible Euler equations. The principal idea of the non-uniqueness proof originally developed by \name{Chiodaroli}, \name{De~Lellis} and \name{Kreml} \cite{ChiDelKre15}, is to work with a suitable notion of a subsolution, the so-called \emph{fan subsolution}. Such an object consists of piecewise constant functions which are constant in three sets: In the two exterior sets the fan subsolution takes the initial states $(\rho_-,\vu_-)$ and $(\rho_+,\vu_+)$ respectively, whereas in the set in the middle\footnote{Because of its shape, we later call this middle set a \emph{wedge}.} the fan subsolutions is in fact a subsolution in the sense that \eqref{eq:p0-subs} holds. Then convex integration, i.e. Theorem~\ref{thm:convint} or Theorem~\ref{thm:convint-nodens}, can be applied on the middle set. Hence again the proof of non-uniqueness of solutions reduces to showing existence of a fan subsolution. 

This way it was shown in \cite{ChiDelKre15} that there exist particular data of the form \eqref{eq:riemann-init-isen} for the isentropic Euler equations which admit infinitely many admissible weak solution. On the other hand it was proven by \name{Chen} and \name{Chen} \cite{CheChe07} and independently also by \name{Feireisl} and \name{Kreml} \cite{FeiKre15} that there exists other data, still of the form \eqref{eq:riemann-init-isen} but for which the self-similar solution is unique. So it has been worked on a classification of Riemann data \eqref{eq:riemann-init-isen} with respect to uniqueness or non-uniqueness of the corresponding solutions. It turned out that the structure of the corresponding self-similar solution determines whether there is a unique admissible weak solution or infinitely many. We refer to Section~\ref{sec:riemann-isen} for more details on the Riemann problem for the isentropic Euler system. 

In the context of the full Euler system \eqref{eq:full-euler-pv-dens} - \eqref{eq:full-euler-pv-en} one established a similar classification with analogue techniques, see Section~\ref{sec:riemann-full}.

\begin{rem}
	The results presented in this chapter can be simply transferred to a corresponding three-dimensional setup. 
\end{rem}

This chapter is organized as follows. In two sections we consider the isentropic Euler system and the full Euler system respectively. In each section we begin with a summary of some facts concerning the self-similar solution, see Subsections \ref{subsec:riemann-isen-1D} and \ref{subsec:riemann-full-1D}. Next we summarize the results concerning non-/uniqueness of admissible weak solutions in Subsections \ref{subsec:riemann-isen-results} and \ref{subsec:riemann-full-results}. In Subsection~\ref{subsec:riemann-isen-SR} we explain the non-uniqueness proof by taking the example where the self-similar solution consists of one shock and one rarefaction. Note that the above mentioned principal idea of considering fan subsolutions, which is explained in detail in Subsection~\ref{subsec:riemann-isen-SR}, is also used to prove non-uniqueness in the other cases. For the latter we just mention what has to be adapted, see Subsection~\ref{subsec:riemann-isen-other-cases}. Finally in Subsection~\ref{subsec:riemann-isen-other-results} we collect some further results related to the problem under consideration, i.e. \eqref{eq:baro-euler-pv-dens}, \eqref{eq:baro-euler-pv-mom}, \eqref{eq:isentropic-EOS}, \eqref{eq:riemann-init-isen}. 

For the full Euler system we take the example where the self-similar solution consists of two shocks , see Subsection~\ref{subsec:riemann-full-SS}, and mention in Subsection~\ref{subsec:riemann-full-other-cases} how the other cases can be handled. As in the isentropic case, we finish in Subsection~\ref{subsec:riemann-full-other-results} with mentioning other results which are related to the initial value problem \eqref{eq:full-euler-pv-dens} - \eqref{eq:full-euler-pv-en}, \eqref{eq:riemann-init-full}.

\section[Riemann Problem for the Isentropic Euler System]{Riemann Problem for the Isentropic Euler System%
	\sectionmark{Isentropic Euler}} \label{sec:riemann-isen}
\sectionmark{Isentropic Euler}

\subsection{One-Dimensional Self-Similar Solution} \label{subsec:riemann-isen-1D} 

One solution to the initial value problem for the isentropic Euler equations \eqref{eq:baro-euler-pv-dens}, \eqref{eq:baro-euler-pv-mom}, \eqref{eq:isentropic-EOS} with initial data \eqref{eq:riemann-init-isen} can be found by solving the corresponding one-dimensional Riemann problem. More precisely this one-dimensional problem reads
\begin{align}
	\partial_t \rho + \partial_2 (\rho v) &= 0 \ec \label{eq:isen-euler-1d-dens} \\
	\partial_t (\rho u) + \partial_2 (\rho u v) &=0 \ec \label{eq:isen-euler-1d-mom1} \\
	\partial_t (\rho v) + \partial_2 \big(\rho v^2 + p(\rho)\big) &= 0 \label{eq:isen-euler-1d-mom2} 
\end{align}
with initial data \eqref{eq:riemann-init-isen} and results from the ansatz that we are looking for solutions $(\rho,\vu)$ which do not depend on $x$. In other words the unknowns $\rho$ and $\vu$ of system \eqref{eq:isen-euler-1d-dens} - \eqref{eq:isen-euler-1d-mom2} are functions of $t\in[0,\infty)$ and $y\in\R$ and still take values in $\R^+$ and $\R^2$ respectively. Note that system \eqref{eq:isen-euler-1d-dens}, \eqref{eq:isen-euler-1d-mom2} with isentropic equation of state \eqref{eq:isentropic-EOS} can be viewed as the one-dimensional isentropic Euler equations for $\rho$ and $v$, whereas \eqref{eq:isen-euler-1d-mom1} represents an additional transport equation for $u$.

One can show that the solution of \eqref{eq:isen-euler-1d-dens} - \eqref{eq:isen-euler-1d-mom2}, \eqref{eq:isentropic-EOS}, \eqref{eq:riemann-init-isen} consists of three waves, the slowest and the fastest of which is either a shock or a rarefaction, whereas the ``middle'' wave is a contact discontinuity where $u$ jumps from $u_-$ to $u_+$. This solution of \eqref{eq:isen-euler-1d-dens} - \eqref{eq:isen-euler-1d-mom2}, \eqref{eq:riemann-init-isen} is an admissible weak solution of the two-dimensional problem \eqref{eq:baro-euler-pv-dens}, \eqref{eq:baro-euler-pv-mom}, \eqref{eq:isentropic-EOS}, \eqref{eq:riemann-init-isen} as well, called \emph{self-similar solution}.

Let us summarize this in the following proposition.

\begin{prop} \label{prop:isen-selfsimilar} 
	Let $\rho_\pm\in\R^+$, $\vu_\pm\in\R^2$. The self-similar solution to the problem \eqref{eq:baro-euler-pv-dens}, \eqref{eq:baro-euler-pv-mom}, \eqref{eq:isentropic-EOS}, \eqref{eq:riemann-init-isen} is constant in four regions which are separated by three waves. The leftmost and rightmost states are given by the initial states $(\rho_-,\vu_-)$ and $(\rho_+,\vu_+)$ respectively. The left intermediate state, i.e. the state in the left intermediate region, is equal to $\big(\rho_M,(u_-,v_M)^\trans\big)$, whereas the right intermediate state equals $\big(\rho_M,(u_+,v_M)^\trans\big)$, see Figure~\ref{fig:isen-SCR} for an example. In particular, the density and the velocity component, that is perpendicular to the initial discontinuity, coincide in both intermediate regions, whereas the velocity component, that is parallel to the initial discontinuity, only takes two values $u_-$ and $u_+$ with a jump at the 2-wave. For the 1- and 3-wave there are the following possibilities, whereas the 2-wave is a contact discontinuity.
	\begin{enumerate}
		\item \label{item:isen-cases.RRvac} If
		\begin{equation*}
			v_+ - v_- \geq \int_0^{\rho_-}\frac{\sqrt{p'(r)}}{r} \dr + \int_0^{\rho_+}\frac{\sqrt{p'(r)}}{r} \dr \ec
		\end{equation*}
		then the self-similar solution consists of a 1-rarefaction and a 3-rarefaction. The intermediate states are vacuum states, i.e. $\rho_M=0$.
		\item \label{item:isen-cases.RR} If 
		\begin{equation*}
			\bigg|\int_{\rho_-}^{\rho_+}\frac{\sqrt{p'(r)}}{r} \dr\bigg| < v_+ - v_- < \int_0^{\rho_-}\frac{\sqrt{p'(r)}}{r} \dr + \int_0^{\rho_+}\frac{\sqrt{p'(r)}}{r} \dr \ec
		\end{equation*}
		then the self-similar solution consists of a 1-rarefaction and a 3-rarefaction, where $\rho_M,v_M$ are uniquely determined by 
		\begin{align*}
			\rho_M&<\min\{\rho_-,\rho_+\} \ec \\
			v_+ - v_- &= \int_{\rho_M}^{\rho_-}\frac{\sqrt{p'(r)}}{r} \dr + \int_{\rho_M}^{\rho_+}\frac{\sqrt{p'(r)}}{r} \dr \ec \\
			v_M &= v_- + \int_{\rho_M}^{\rho_-}\frac{\sqrt{p'(r)}}{r} \dr \ed
		\end{align*}
		\item \label{item:isen-cases.R} If
		\begin{equation*}
			\bigg|\int_{\rho_-}^{\rho_+}\frac{\sqrt{p'(r)}}{r} \dr\bigg| = v_+ - v_- \ec
		\end{equation*}
		then the self-similar solution consists of one rarefaction. More precisely this rarefaction is a 1-rarefaction if $\rho_->\rho_+$ and a 3-rarefaction if $\rho_-<\rho_+$.
		\item \label{item:isen-cases.RS} If $\rho_->\rho_+$ and
		\begin{equation*}
			-\sqrt{\frac{\big(\rho_- - \rho_+\big) \big(p(\rho_-)-p(\rho_+)\big)}{\rho_- \rho_+}} < v_+ - v_- < \int_{\rho_+}^{\rho_-}\frac{\sqrt{p'(r)}}{r} \dr \ec
		\end{equation*}
		then the self-similar solution consists of a 1-rarefaction and a 3-shock, where $\rho_M,v_M$ are uniquely determined by 
		\begin{align*}
		\rho_+&<\rho_M<\rho_- \ec \\
			v_+ - v_- &= \int_{\rho_M}^{\rho_-}\frac{\sqrt{p'(r)}}{r} \dr - \sqrt{\frac{\big(\rho_M - \rho_+\big) \big(p(\rho_M) - 	p(\rho_+)\big)}{\rho_M \rho_+}} \ec \\
			v_M &= v_- + \int_{\rho_M}^{\rho_-}\frac{\sqrt{p'(r)}}{r} \dr \ed
		\end{align*}
		\item \label{item:isen-cases.SR} If $\rho_-<\rho_+$ and
		\begin{equation*}
			-\sqrt{\frac{\big(\rho_- - \rho_+\big) \big(p(\rho_-)-p(\rho_+)\big)}{\rho_- \rho_+}} < v_+ - v_- < \int_{\rho_-}^{\rho_+}\frac{\sqrt{p'(r)}}{r} \dr \ec
		\end{equation*}
		then the self-similar solution consists of a 1-shock and a 3-rarefaction, where $\rho_M,v_M$ are uniquely determined by 
		\begin{align*}
			\rho_-&<\rho_M<\rho_+ \ec \\
			v_+ - v_- &= 	\int_{\rho_M}^{\rho_+}\frac{\sqrt{p'(r)}}{r}\dr - \sqrt{\frac{\big(\rho_M - \rho_-\big) \big(p(\rho_M) - p(\rho_-)\big)}{\rho_M\rho_-}} \ec \\
			v_M &= v_- - \sqrt{\frac{\big(\rho_M - \rho_-\big)\big(p(\rho_M) - p(\rho_-)\big)}{\rho_M\rho_-}} \ed
		\end{align*}
		\item \label{item:isen-cases.S} If 
		\begin{equation*}
			v_+ - v_- = -\sqrt{\frac{\big(\rho_- - \rho_+\big)\big(p(\rho_-)-p(\rho_+)\big)}{\rho_- \rho_+}} \ec
		\end{equation*}
		then the self-similar solution consists of one shock. More precisely this shock is a 1-shock if $\rho_-<\rho_+$ and a 3-shock if $\rho_->\rho_+$.
		\item \label{item:isen-cases.SS} If 
		\begin{equation*}
			v_+ - v_- < -\sqrt{\frac{\big(\rho_- - \rho_+\big)\big(p(\rho_-)-p(\rho_+)\big)}{\rho_-\rho_+}} \ec
		\end{equation*}
		then the self-similar solution consists of a 1-shock and a 3-shock, where $\rho_M,v_M$ are uniquely determined by 
		\begin{align*}
			\rho_M&>\max\{\rho_-,\rho_+\} \ec \\
			v_+ - v_- &= - \sqrt{\frac{\big(\rho_M - \rho_+\big) \big(p(\rho_M) - p(\rho_+)\big)}{\rho_M\rho_+}} - \sqrt{\frac{\big(\rho_M - \rho_-\big)\big(p(\rho_M) - p(\rho_-)\big)}{\rho_M\rho_-}} \ec \\
			v_M &= v_- - \sqrt{\frac{\big(\rho_M - \rho_-\big)\big(p(\rho_M) - p(\rho_-)\big)}{\rho_M\rho_-}} \ed
		\end{align*}
	\end{enumerate}
\end{prop}

\begin{proof} 
	We refer to \name{Chiodaroli} and \name{Kreml} \cite[Section 2]{ChiKre14} or the textbook by \name{Dafermos} \cite[Chapters 7 - 9]{Dafermos}. At this point it is important that the system \eqref{eq:isen-euler-1d-dens}, \eqref{eq:isen-euler-1d-mom2} is hyperbolic, see Definition~\ref{defn:hyperbolicity}. Note that in the references above, the system \eqref{eq:isen-euler-1d-dens}, \eqref{eq:isen-euler-1d-mom2} -- i.e. without \eqref{eq:isen-euler-1d-mom1} -- is considered. However the extension to the system with equation \eqref{eq:isen-euler-1d-mom1} is not difficult and simply leads to the 2-contact discontinuity, at which $u$ jumps. 
\end{proof}

\begin{figure}[htb] 
	\centering
	\includegraphics[width=0.7\textwidth]{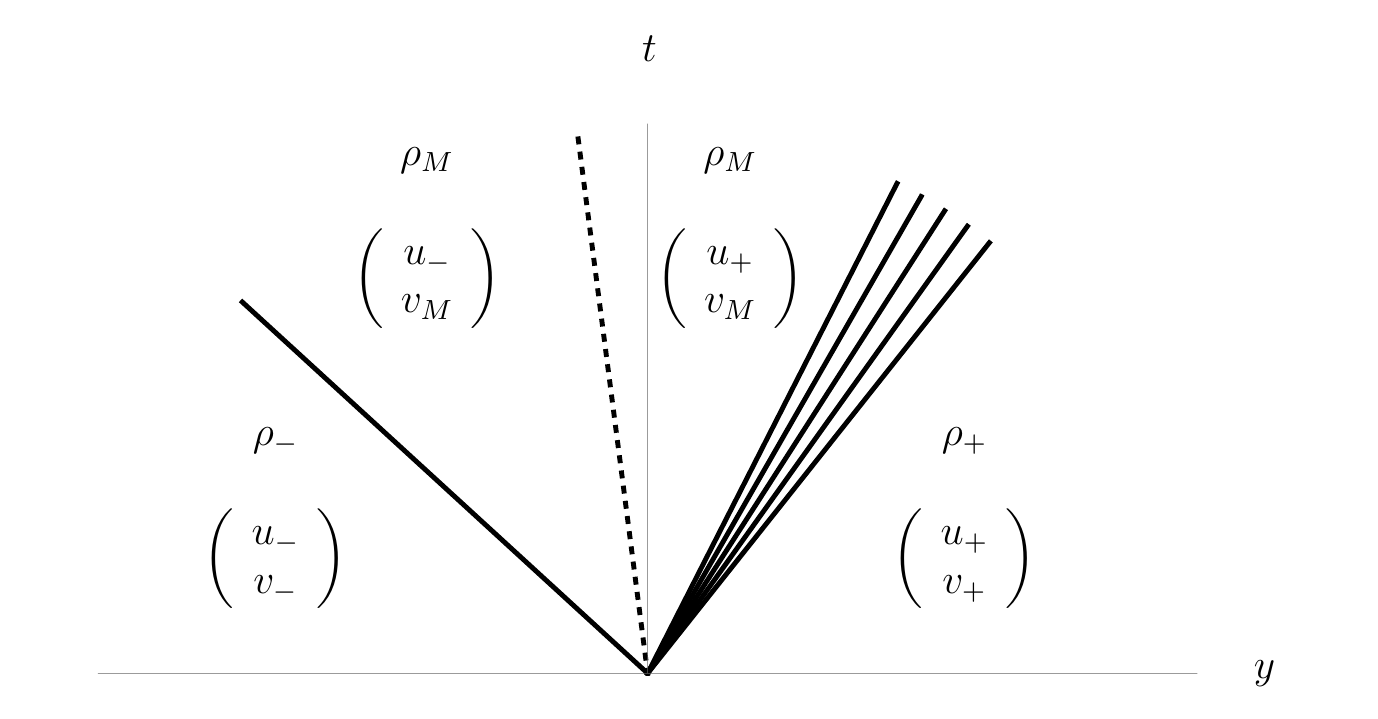} 
	\caption{An example of the self-similar solution for isentropic Euler where the 1-wave is a shock, the 2-wave a contact discontinuity and the 3-wave a rarefaction.} 
	\label{fig:isen-SCR}
\end{figure} 

\begin{rem}
	In case \ref{item:isen-cases.R} and case \ref{item:isen-cases.S}, only one wave among the 1- and 2- wave is ``visible''. This means that for example if there is only a 1-shock, then the right intermediate state and the rightmost state (i.e. the initial state $(\rho_+,\vu_+)$) coincide. The same happens for the 2-contact discontinuity, which ``disappears'' if $u_-=u_+$.
\end{rem}

\subsection{Summary of the Results on Non-/Uniqueness} \label{subsec:riemann-isen-results}

\begin{thm} \label{thm:isen-summary-riemann}
	It depends on the shape of the self-similar solution whether or not there is a unique admissible weak solution of the initial value problem for the isentropic Euler system \eqref{eq:baro-euler-pv-dens}, \eqref{eq:baro-euler-pv-mom}, \eqref{eq:isentropic-EOS} with initial data \eqref{eq:riemann-init-isen}. Table~\ref{tab:isen-results} summarizes the results. In the cases where the solution is not unique, there are even infinitely many admissible weak solutions.
\end{thm}

\renewcommand{\arraystretch}{1.2}
\begin{table}[h] 
	\centering 
	\begin{tabular}{|c|c|c|c|c|c|} \cline{2-6} 
		\multicolumn{1}{c|}{}& \multicolumn{3}{c|}{Structure of the self-similar solution} & \centering Solution & \centering Reference \tabularnewline 
		\multicolumn{1}{c|}{}& \centering 1-wave & \centering 2-wave & \centering 3-wave & \centering unique? & \tabularnewline \cline{1-6}
		1 & \centering - & \centering - & \centering - & \textcolor{green}{yes} & e.g. \cite{CheChe07} or \cite{FeiKre15} \tabularnewline \cline{1-6}
		2 & \centering - & \centering - & \centering shock & \textcolor{red}{no} & \cite{KliMar18_1} \tabularnewline \cline{1-6}
		3 & \centering - & \centering - & \centering rarefaction & \textcolor{green}{yes} & \cite{CheChe07}, \cite{FeiKre15} \tabularnewline \cline{1-6}
		4 & \centering shock & \centering - & \centering - & \textcolor{red}{no} & \cite{KliMar18_1} \tabularnewline \cline{1-6}
		5 & \centering shock & \centering - & \centering shock & \textcolor{red}{no} & \cite{ChiKre14} \tabularnewline \cline{1-6}
		6 & \centering shock & \centering - & \centering rarefaction & \textcolor{red}{no} & \cite{KliMar18_1}, partially in \cite{ChiKre18} \tabularnewline \cline{1-6}
		7 & \centering rarefaction & \centering - & \centering - & \textcolor{green}{yes} & \cite{CheChe07}, \cite{FeiKre15} \tabularnewline \cline{1-6}
		8 & \centering rarefaction & \centering - & \centering shock & \textcolor{red}{no} & \cite{KliMar18_1}, partially in \cite{ChiKre18} \tabularnewline \cline{1-6}
		9 &\centering rarefaction  & \centering - & \centering rarefaction & \textcolor{green}{yes} & \cite{CheChe07}, \cite{FeiKre15} \tabularnewline \cline{1-6}
		10 & \centering - & \centering contact & \centering - & open & \tabularnewline \cline{1-6}
		11 & \centering - & \centering contact & \centering shock & \textcolor{red}{no} & combine ideas of \cite{BreChiKre18} and \cite{KliMar18_1} \tabularnewline \cline{1-6}
		12 & \centering - & \centering contact & \centering rarefaction & open &  \tabularnewline \cline{1-6}
		13 & \centering shock & \centering contact & \centering - & \textcolor{red}{no} & combine ideas of \cite{BreChiKre18} and \cite{KliMar18_1} \tabularnewline \cline{1-6}
		14 & \centering shock & \centering contact & \centering shock & \textcolor{red}{no} & \cite{BreChiKre18} \tabularnewline \cline{1-6}
		15 & \centering shock & \centering contact & \centering rarefaction & \textcolor{red}{no} & combine ideas of \cite{BreChiKre18} and \cite{KliMar18_1}, \tabularnewline 
		 &  &  &  &  & partially in \cite{BreChiKre18} \tabularnewline \cline{1-6}
		16 & \centering rarefaction & \centering contact & \centering - & open &  \tabularnewline \cline{1-6}
		17 & \centering rarefaction & \centering contact & \centering shock & \textcolor{red}{no} & combine ideas of \cite{BreChiKre18} and \cite{KliMar18_1}, \tabularnewline 
		 &  &  &  &  & partially in \cite{BreChiKre18} \tabularnewline \cline{1-6}
		18 &\centering rarefaction & \centering contact & \centering rarefaction & open & \tabularnewline \cline{1-6}
	\end{tabular}
	\caption{All the 18 possibilities of the structure of the self-similar solution to the initial value problem for the isentropic Euler system \eqref{eq:baro-euler-pv-dens}, \eqref{eq:baro-euler-pv-mom}, \eqref{eq:isentropic-EOS} with initial data \eqref{eq:riemann-init-isen}. Furthermore, if known, the answer to the question on uniqueness.} \label{tab:isen-results} 
\end{table}
\renewcommand{\arraystretch}{1}

The uniqueness result by \name{Chen} and \name{Chen} \cite{CheChe07} as well as the one by \name{Feireisl} and \name{Kreml} \cite{FeiKre15}, both of which address the cases (1,)\footnote{Case 1 follows immediately from the weak-strong-uniqueness principle without any adaptions.} 3, 7  and 9 in Table~\ref{tab:isen-results}, are achieved by adapting the well-known \emph{weak-strong-uniqueness} principle: If there exists a strong solution of the initial value problem for the barotropic Euler system \eqref{eq:baro-euler-pv-dens}, \eqref{eq:baro-euler-pv-mom}, then this solution is unique in the class of admissible\footnote{In fact it is unique in the even larger class of \emph{dissipative} weak solutions which fulfill the energy inequality in an integral sense, see \cite[Section 2.1]{FeiKre15}.} weak solutions. As shown in the cited papers, a solution containing only rarefaction waves is ``strong enough'' (here piecewise $C^1$) such that this principle still holds, even if such a solution is not strong in sense of Definition~\ref{defn:strongsol} as it is not $C^1$. Note furthermore that the result of \name{Chen} and \name{Chen} \cite{CheChe07} also holds if the intermediate state is a vacuum state. Moreover the result by \name{Feireisl} and \name{Kreml} \cite{FeiKre15} is true for any convex and strictly increasing pressure law $p\in C^1$, i.e. not only for the isentropic equation of state \eqref{eq:isentropic-EOS}.

As already mentioned, the initial value problem considered in this chapter was first studied in the context of non-uniqueness of solutions by \name{Chiodaroli}, \name{De~Lellis} and \name{Kreml} \cite{ChiDelKre15}. In this paper it was shown that there exists a particular example of initial data of the form \eqref{eq:riemann-init-isen} which admit infinitely many solutions. Those data belong to class 6 in Table~\ref{tab:isen-results}, i.e. the corresponding self-similar solution consists of one shock and one rarefaction. 

In \cite{ChiKre14} it was shown by \name{Chiodaroli} and \name{Kreml} that if the initial states $(\rho_\pm,\vu_\pm)$ are such that the self-similar solution contains two shocks (i.e. case 5 in Table~\ref{tab:isen-results}), then there are infinitely many solutions. For the remaining cases where $u_-=u_+$ (2, 4, 6, 8 in Table~\ref{tab:isen-results}) the non-uniqueness proof was achieved by the \name{Klingenberg} and the author, see \cite{KliMar18_1}. Subcases of 6 and 8 in Table~\ref{tab:isen-results} have been also solved independently by \name{Chiodaroli} and \name{Kreml}~\cite{ChiKre18}. Note that both results \cite{ChiKre14} and \cite{KliMar18_1} are still valid if $\gamma=1$. 

The case 14 in Table~\ref{tab:isen-results} (and subcases of 15 and 17) has been covered by \name{B{\v r}ezina}, \name{Chiodaroli} and \name{Kreml} \cite{BreChiKre18}. Here one uses similar ideas together with an additional wedge in the fan, cf. Definition~\ref{defn:isen-fanpart}. The cases 11, 13, 15, 17 in Table~\ref{tab:isen-results} can be treated by combining the results of \cite{BreChiKre18} and \cite{KliMar18_1}.

Finally if the self-similar solution consists of a contact discontinuity and possible rarefactions (cases 10, 12, 16, 18 in Table~\ref{tab:isen-results}) the question on uniqueness remains open. Note that a generic example for initial data for which the self-similar solution consists only of a contact discontinuity (case 10) is the shear flow
$$
(\rho_\init,\vu_\init)(\vx) :=\left\{
\begin{array}{ll}
	\Big(1,(-1,0)^\trans\Big) & \text{ if }y<0\ec \\
	\Big(1,(1,0)^\trans\Big) & \text{ if }y>0\ed
\end{array}
\right. 
$$
In view of the above mentioned result by \name{Sz{\'e}kelyhidi} \cite{Szekelyhidi11}, where it has been shown that analogue initial data for the incompressible Euler system lead to infinitely many solutions, it is surprising that the question on uniqueness remains open in case 10.

\subsection{Non-Uniqueness Proof if the Self-Similar Solution Consists of One Shock and One Rarefaction} \label{subsec:riemann-isen-SR}

In this subsection we prove existence of infinitely many admissible weak solutions if the self-similar solution consists of one shock and one rarefaction (6 and 8 in Table~\ref{tab:isen-results}). To this end we exhibit the principal strategy developed in \cite{ChiDelKre15}: We introduce \emph{fan subsolutions} and apply Theorem~\ref{thm:convint}, more precisely its fixed-density-version, i.e. Theorem~\ref{thm:convint-nodens}. We need to use the latter theorem in order to obtain \emph{admissible} solutions. Hence it suffices to show existence of a fan subsolution, which is equivalent to finding a solution of a suitable system of algebraic equations and inequalities. In particular for the case where the self-similar solution consists of one shock and one rarefaction, the construction of a fan subsolution requires to introduce an auxiliary state, see below. The content of this subsection has already been published in \cite{KliMar18_1}. Parts can be also found in \name{Chiodaroli}-\name{Kreml} \cite{ChiKre18}.

Since the isentropic Euler equations \eqref{eq:baro-euler-pv-dens}, \eqref{eq:baro-euler-pv-mom} are invariant under rotations, we may restrict ourselves to the case where the self-similar solution consists of a 1-shock and 3-rarefaction. Indeed if the self-similar solution consists of a 1-rarefaction and a 3-shock, we rotate the coordinate system 180 degrees to obtain new initial data for which the self-similar solution consists of a 1-shock and a 3-rarefaction. 

In a nutshell, the objective of the current subsection is to prove the following theorem.

\begin{thm} \label{thm:isen-SR} 
	Let $\rho_\pm\in \R^+$, $\vu_\pm\in \R^2$ be such that the self-similar solution consists of a 1-shock and a 3-rarefaction. Then there exist infinitely many admissible weak solutions to the initial value problem \eqref{eq:baro-euler-pv-dens}, \eqref{eq:baro-euler-pv-mom}, \eqref{eq:isentropic-EOS}, \eqref{eq:riemann-init-isen}.
\end{thm}

According to Proposition~\ref{prop:isen-selfsimilar} \ref{item:isen-cases.SR} the fact that the self-similar solution consists of a 1-shock and a 3-rarefaction means that  
\begin{align}
 	\rho_-&<\rho_+\ec \label{eq:1S3R.a} \\
	u_-&=u_+ \qquad\text{ and} \label{eq:1S3R.b}\\
	-\sqrt{\frac{(\rho_- - \rho_+) \big(p(\rho_-)-p(\rho_+)\big)}{\rho_- \rho_+}} &< v_+ - v_- < \int_{\rho_-}^{\rho_+}\frac{\sqrt{p'(r)}}{r} \dr \ed \label{eq:1S3R.c}
\end{align}

\subsubsection{Condition for Non-Uniqueness}

In order to state a sufficient condition for the existence of infinitely many solutions, let us define \emph{admissible fan subsolutions} as first introduced by \name{Chiodaroli}, \name{De~Lellis} and \name{Kreml}~\cite[Definitions 3.4 and 3.5]{ChiDelKre15}, where our notation is tailored to the usage of Theorems~\ref{thm:convint} and \ref{thm:convint-nodens}. 

\begin{defn} \label{defn:isen-fanpart}
	Let $\mu_0<\mu_1$ be real numbers. A \emph{fan partition} of the space-time domain $(0,\infty)\times\R^2$ consists of three open sets $\Gamma_-,\Gamma_1,\Gamma_+$ of the form
	\begin{align*}
		\Gamma_-&=\{(t,\vx):t>0\text{ and }y<\mu_0 t\}\ec \\
		\Gamma_1\: &=\{(t,\vx):t>0\text{ and }\mu_0 t<y<\mu_1 t\} \ec \\
		\Gamma_+&=\{(t,\vx):t>0\text{ and }y>\mu_1 t\} \ec
	\end{align*}
	see Figure~\ref{fig:isen-fanpart}.
\end{defn}

\begin{figure}[htb] 
	\centering
	\includegraphics[width=0.7\textwidth]{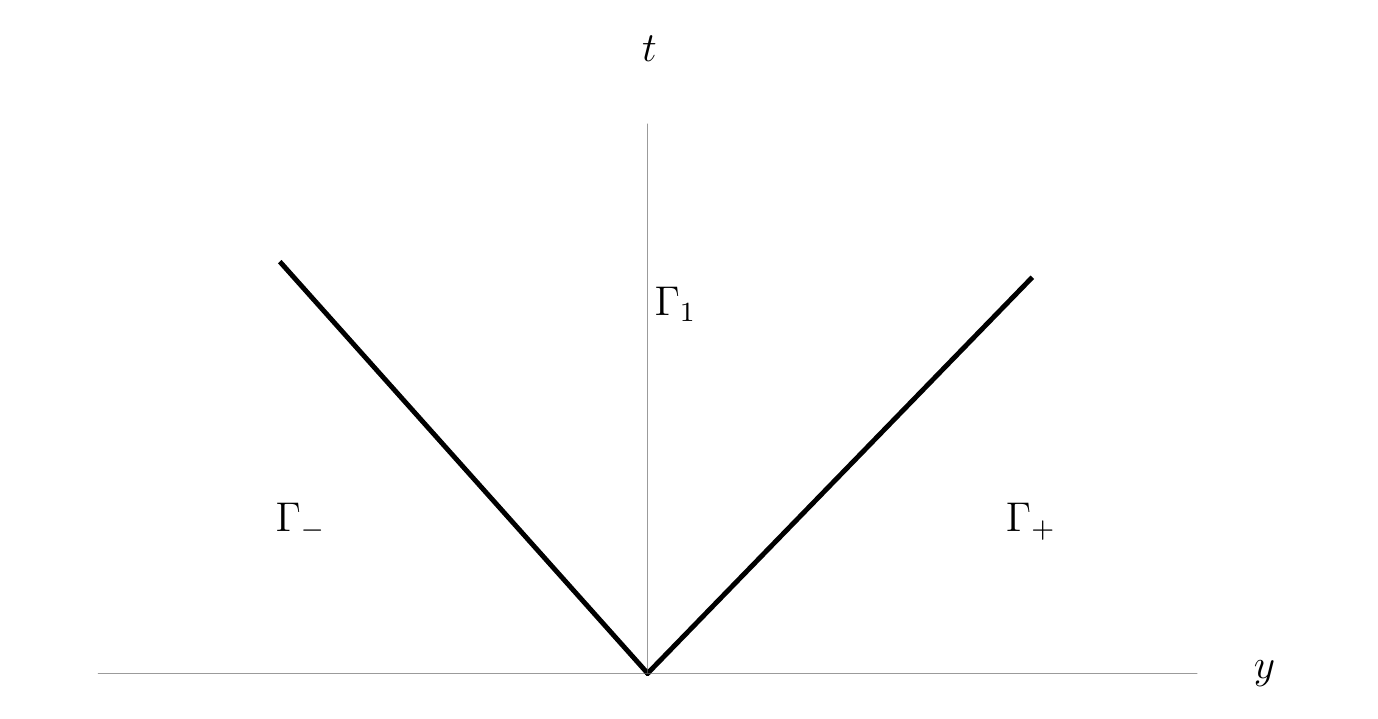} 
	\caption{A fan partition, see Definition~\ref{defn:isen-fanpart}.} 
	\label{fig:isen-fanpart}
\end{figure} 

\begin{defn} \label{defn:isen-fansubs}
	An \emph{admissible fan subsolution} to the initial value problem for the isentropic Euler system \eqref{eq:baro-euler-pv-dens}, \eqref{eq:baro-euler-pv-mom}, \eqref{eq:isentropic-EOS} with initial data \eqref{eq:riemann-init-isen} is a quadruple
	$$
	(\ov{\rho},\ov{\vm},\ov{\mU},\ov{c}) \ \in\  L^\infty\big( (0,\infty)\times\R^2;\R^+\times\R^2\times\sztwo\times \R^+\big)
	$$ 
	of piecewise constant functions, which satisfies the following properties:
	\begin{enumerate}
		\item There exists a fan partition $\Gamma_-,\Gamma_1,\Gamma_+$ of $(0,\infty)\times\R^2$ and constants $\rho_1\in\R^+$, $\vm_1\in\R^2$, $\mU_1\in\sztwo$ and $c_1 \in \R^+$, such that
		\begin{align*}
			(\ov{\rho},\ov{\vm},\ov{\mU},\ov{c})&= \left\{ 
			\begin{array}[c]{ll}
				\big(\rho_-,\vm_-,\mU_-, c_- \big) & \text{ if } (t,\vx)\in \Gamma_- \ec\\
				\big(\rho_1\: ,\vm_1 \:,\mU_1 \:,c_1\: \big) & \text{ if } (t,\vx)\in \Gamma_1 \:\ec\\
				\big(\rho_+,\vm_+,\mU_+, c_+\big) & \text{ if } (t,\vx)\in \Gamma_+ \ec 
			\end{array}
			\right.
		\end{align*}
		where 
		\begin{align*}
			\vm_\pm &:= \rho_\pm \vu_\pm \ec \\
			\mU_\pm &:= \rho_\pm \vu_\pm \otimes \vu_\pm - \frac{\rho_\pm |\vu_\pm|^2}{2}\id \ec\\
			c_\pm &:=  \frac{\rho_\pm |\vu_\pm|^2}{2} + p(\rho_\pm) 
		\end{align*}
		with the given initial states $(\rho_\pm,\vu_\pm) \in \R^+\times \R^2$;
		
		\item The following inequality holds: 
		\begin{equation} \label{eq:isen-fansubs-subs}
			e(\rho_1,\vm_1,\mU_1) < c_1 \es
		\end{equation} 
		\item For all test functions $(\phi,\vphi)\in \Cc \big([0,\infty)\times\R^2;\R\times\R^2\big)$ the following identities hold:
		\begin{align}
			\int_0^\infty \int_{\R^2} \Big[\ov{\rho} \partial_t \phi + \ov{\vm}\cdot\Grad \phi\Big]\dx\dt + \int_{\R^2} \rho_\init\phi(0,\cdot) \dx &= 0 \es \label{eq:isen-fansubs-dens}\\
			\int_0^\infty \int_{\R^2} \Big[\ov{\vm} \cdot\partial_t \vphi + \ov{\mU}:\Grad \vphi + \ov{c} \Div\vphi\Big]\dx\dt + \int_{\R^2} \rho_\init\vu_\init\cdot\vphi(0,\cdot) \dx &= 0\es \label{eq:isen-fansubs-mom}
		\end{align}
		\item For every non-negative test function $\varphi\in \Cc\big([0,\infty)\times\R^2;\R_0^+\big)$, the inequality
		\begin{align}
			\int_0^\infty \int_{\R^2} \left[\Big(\ov{c} -p(\ov{\rho}) + P(\ov{\rho}) \Big) \partial_t \varphi + \Big(\ov{c} + P(\ov{\rho})\Big)\frac{\ov{\vm}}{\ov{\rho}}\cdot\Grad \varphi \right]\dx\dt \qquad & \notag\\ 
			+ \int_{\R^2} \bigg(\half\rho_\init|\vu_\init|^2 + P(\rho_\init)\bigg)\varphi(0,\cdot) \dx &\geq 0 \label{eq:isen-fansubs-adm}
		\end{align}
		is fulfilled.
	\end{enumerate}
\end{defn}

For the existence of infinitely many admissible weak solutions it suffices that there exists an admissible fan subsolution. This is the content of the following theorem. To prove it we apply Theorem~\ref{thm:convint}. More precisely we have to apply the version with fixed density (i.e. Theorem~\ref{thm:convint-nodens}) to obtain \emph{admissible} weak solutions.

\begin{thm} \label{thm:isen-condition}
	Let the initial states $(\rho_\pm,\vu_\pm)$ be such that there exists an admissible fan subsolution $(\ov{\rho},\ov{\vm},\ov{\mU},\ov{c})$ to the initial value problem \eqref{eq:baro-euler-pv-dens}, \eqref{eq:baro-euler-pv-mom}, \eqref{eq:isentropic-EOS}, \eqref{eq:riemann-init-isen}. Then this initial value problem admits infinitely many admissible weak solutions $(\rho,\vu)$ with the following properties:
	\begin{enumerate}
		\item \label{item:isen-riemann.a} $\rho=\ov{\rho}$,
		\item \label{item:isen-riemann.b} $\vu(t,\vx)=\vu_-$ for all $(t,\vx)\in \Gamma_-$ and $\vu(t,\vx)=\vu_+$ for all $(t,\vx)\in \Gamma_+$,
		\item \label{item:isen-riemann.c} $|\vu(t,\vx)|^2=\frac{2}{\rho_1} \big(c_1 - p(\rho_1)\big)$ for\footnote{Note that $c_1-p(\rho_1)>0$ due to \eqref{eq:isen-fansubs-subs}. Indeed we obtain using Lemma~\ref{lemma:not-lmax-tr}
		\begin{align*}
			p(\rho_1) &\leq \frac{|\vm_1|^2}{2\rho_1} + p(\rho_1) = \half\tr \left(\frac{\vm_1\otimes\vm_1}{\rho_1} + p(\rho_1)\id - \mU_1\right) \\ 
			&\leq \lambda_{\max} \left(\frac{\vm_1\otimes\vm_1}{\rho_1} + p(\rho_1)\id - \mU_1\right) = e(\rho_1,\vm_1,\mU_1) < c_1\ed
		\end{align*}} a.e. $(t,\vx)\in \Gamma_1$. 
	\end{enumerate}
\end{thm}

\begin{proof}
	In order to apply Theorem~\ref{thm:convint} we set $\Gamma:= \Gamma_1$, $r:= \half \rho_1$, $c:=c_1$ and the triple of functions $(\rho_0,\vm_0,\mU_0)\in C^1\big(\closure{\Gamma};\R^+\times \R^n\times \sztwo\big)$ constant, more precisely 
	$$
		(\rho_0,\vm_0,\mU_0)(t,\vx) := (\rho_1,\vm_1,\mU_1) \qquad \text{ for all }(t,\vx)\in \closure{\Gamma} \ed
	$$
	Let us check the assumptions of Theorem~\ref{thm:convint}:
	\begin{itemize}
		\item The PDEs \eqref{eq:p0-pde1}, \eqref{eq:p0-pde2} hold obviously because the functions $\rho_0,\vm_0,\mU_0$ are constant.
		\item From \eqref{eq:isen-fansubs-subs} and \eqref{eq:interiorU} we duduce that \eqref{eq:p0-subs} is satisfied.
		\item The density is bounded from below (i.e. \eqref{eq:p0-dens-bdd} holds) by construction. Indeed for all $(t,\vx)\in \closure{\Gamma}$ we have 
		$$
		\rho_0(t,\vx) = \rho_1 > \frac{\rho_1}{2} = r \ed
		$$
	\end{itemize} 
	Hence Theorem~\ref{thm:convint} yields infinitely many bounded functions $(\til{\rho},\til{\vm})\in L^\infty(\Gamma;\R^+\times \R^2)$ with the properties \ref{item:convint.a} - \ref{item:convint.c} stated in Theorem~\ref{thm:convint}. For each such pair $(\til{\rho},\til{\vm})$ we define $(\rho,\vu)\in L^\infty \big((0,\infty)\times \R^2;\R^+\times \R^2\big)$ by 
	\begin{equation} \label{eq:i1-temp-riemann}
		(\rho,\vu)= \left\{ 
		\begin{array}[c]{ll}
			\big(\rho_-,\vu_-\big) & \text{ if } (t,\vx)\in \Gamma_- \ec\\
			\big(\til{\rho},\til{\vm} / \til{\rho}\big) & \text{ if } (t,\vx)\in \Gamma_1 \:\ec\\
			\big(\rho_+,\vu_+\big) & \text{ if } (t,\vx)\in \Gamma_+ \ed
		\end{array}
		\right.
	\end{equation} 
	We claim that each $(\rho,\vu)$ is indeed an admissible weak solution to the initial value problem \eqref{eq:baro-euler-pv-dens}, \eqref{eq:baro-euler-pv-mom}, \eqref{eq:isentropic-EOS}, \eqref{eq:riemann-init-isen} with the desired properties. To show this, we choose arbitrary test functions $(\phi,\vphi,\varphi) \in \Cc\big([0,\infty) \times \R^2; \R\times \R^2\times\R^+_0\big)$. With \eqref{eq:isen-fansubs-dens}, the Divergence Theorem (Proposition~\ref{prop:not-divergence}) and \eqref{eq:sol-pde1}, we find
	\begin{align*} 
		&\int_0^\infty \int_{\R^2} \Big[\rho \partial_t \phi + \rho\vu\cdot\Grad \phi\Big]\dx\dt + \int_{\R^2} \rho_\init\phi(0,\cdot) \dx \\
		&= \int_0^\infty \int_{\R^2} \Big[\ov{\rho} \partial_t \phi + \ov{\vm}\cdot\Grad \phi\Big]\dx\dt + \int_{\R^2} \rho_\init\phi(0,\cdot) \dx \\
		&\qquad - \iint_{\Gamma_1} \Big[\rho_1 \partial_t \phi + \vm_1\cdot\Grad \phi\Big]\dx\dt + \iint_{\Gamma_1} \Big[\til{\rho} \partial_t \phi + \til{\vm}\cdot\Grad \phi\Big]\dx\dt \\
		&= -\int_{\partial \Gamma_1} \big[\rho_1 \,n_t + \vm_1\cdot \vn_\vx \big] \phi \dS_{t,\vx} + \int_{\partial \Gamma} \big[\rho_1 \,n_t + \vm_1\cdot \vn_\vx \big] \phi \dS_{t,\vx} \ = \ 0\ec
	\end{align*}
	i.e. \eqref{eq:baro-euler-weak-whole-dens} holds. Similarly \eqref{eq:isen-fansubs-mom}, the Divergence Theorem and \eqref{eq:sol-pde2} imply
	\begin{align*}
		&\int_0^\infty \int_{\R^2} \Big[\rho\vu \cdot\partial_t \vphi + \rho\vu\otimes\vu:\Grad \vphi + p(\rho)\Div \vphi\Big]\dx\dt + \int_{\R^2} \rho_\init\vu_\init\cdot\vphi(0,\cdot) \dx \\
		&= \int_0^\infty \int_{\R^2} \Big[\ov{\vm} \cdot\partial_t \vphi + \ov{\mU}:\Grad \vphi + \ov{c} \Div\vphi\Big]\dx\dt + \int_{\R^2} \rho_\init\vu_\init\cdot\vphi(0,\cdot) \dx \\
		&\qquad - \iint_{\Gamma_1} \Big[\vm_1 \cdot\partial_t \vphi + \mU_1:\Grad \vphi + c_1 \Div\vphi\Big]\dx\dt \\
		&\qquad + \iint_{\Gamma_1} \bigg[\til{\vm}\cdot\partial_t\vphi + \frac{\til{\vm}\otimes\til{\vm}}{\til{\rho}} : \Grad\vphi + p(\til{\rho}) \Div\vphi\bigg] \dx\dt \\
		&= - \int_{\partial\Gamma_1} \left[ \vm_1\cdot \vphi\,n_t + (\mU_1\cdot \vphi)\cdot \vn_\vx + c_1 \vphi \cdot \vn_\vx \right] \dS_{t,\vx} \\
		&\qquad + \int_{\partial\Gamma_1} \left[ \vm_1\cdot \vphi\,n_t + (\mU_1\cdot \vphi)\cdot \vn_\vx + c_1 \vphi \cdot \vn_\vx \right] \dS_{t,\vx}\ =\ 0 \ec
	\end{align*}
	and thus \eqref{eq:baro-euler-weak-whole-mom} is fulfilled. 
	
	It remains to show the properties \ref{item:isen-riemann.a} - \ref{item:isen-riemann.c} and the energy inequality \eqref{eq:baro-euler-weak-whole-admissibility}. To achieve that, one needs to apply Theorem~\ref{thm:convint-nodens} rather than Theorem~\ref{thm:convint}. Hence we may assume that $\til{\rho}\equiv \rho_1$. Then we obtain from \eqref{eq:i1-temp-riemann} that the properties \ref{item:isen-riemann.a} and \ref{item:isen-riemann.b} hold. Moreover \eqref{eq:sol-trace} yields 
	\begin{equation} \label{eq:i3-temp-riemann}
		\frac{|\til{\vm}(t,\vx)|^2}{2\rho_1} + p(\rho_1) = \frac{|\til{\vm}(t,\vx)|^2}{2\til{\rho}} + p\big(\til{\rho}(t,\vx)\big) =  c_1
	\end{equation} 
	for a.e. $(t,\vx)\in \Gamma_1$ and thus \ref{item:isen-riemann.c}. Using the fact that $\til{\rho}\equiv \rho_1$ again, we deduce from \eqref{eq:sol-pde1} and the Divergence Theorem that 
	\begin{equation} \label{eq:i2-temp-riemann}
		\iint_\Gamma \til{\vm}\cdot \Grad\varphi \dx\dt - \int_{\partial \Gamma} \vm_1\cdot \vn_\vx \varphi \dS_{t,\vx} = 0\ed
	\end{equation}
	
	Finally with the help of \eqref{eq:i2-temp-riemann} we are able to prove the energy inequality \eqref{eq:baro-euler-weak-whole-admissibility}. Indeed, \eqref{eq:isen-fansubs-adm}, the Divergence Theorem, \eqref{eq:i3-temp-riemann} and \eqref{eq:i2-temp-riemann} imply
	\begin{align*} 
		&\int_0^\infty \int_{\R^2} \left[\bigg(\half\rho|\vu|^2 + P(\rho)\bigg) \partial_t \varphi + \bigg(\half\rho|\vu|^2 + P(\rho) + p(\rho)\bigg)\vu\cdot\Grad \varphi \right]\dx\dt  \\ 
		&\qquad + \int_{\R^2} \bigg(\half\rho_\init|\vu_\init|^2 + P(\rho_\init)\bigg)\varphi(0,\cdot) \dx \\
		&= \int_0^\infty \int_{\R^2} \left[\Big(\ov{c} -p(\ov{\rho}) + P(\ov{\rho}) \Big) \partial_t \varphi + \Big(\ov{c} + P(\ov{\rho})\Big)\frac{\ov{\vm}}{\ov{\rho}}\cdot\Grad \varphi \right]\dx\dt  \\ 
		&\qquad + \int_{\R^2} \bigg(\half\rho_\init|\vu_\init|^2 + P(\rho_\init)\bigg)\varphi(0,\cdot) \dx \\
		&\qquad - \iint_{\Gamma_1} \left[\Big(c_1 -p(\rho_1) + P(\rho_1) \Big) \partial_t \varphi + \Big(c_1 + P(\rho_1)\Big)\frac{\vm_1}{\rho_1}\cdot\Grad \varphi \right]\dx\dt \\
		&\qquad + \iint_{\Gamma_1} \left[\bigg(\half\frac{|\til{\vm}|^2}{\til{\rho}} + P(\til{\rho})\bigg) \partial_t \varphi + \bigg(\half\frac{|\til{\vm}|^2}{\til{\rho}} + P(\til{\rho}) + p(\til{\rho})\bigg)\frac{\til{\vm}}{\til{\rho}}\cdot\Grad \varphi \right]\dx\dt \\
		&\geq - \iint_{\Gamma_1} \left[\Big(c_1 -p(\rho_1) + P(\rho_1) \Big) \partial_t \varphi + \Big(c_1 + P(\rho_1)\Big)\frac{\vm_1}{\rho_1}\cdot\Grad \varphi \right]\dx\dt \\ 
		&\qquad + \iint_{\Gamma_1} \left[\Big(c_1 -p(\rho_1) + P(\rho_1) \Big) \partial_t \varphi + \Big(c_1 + P(\rho_1) \Big)\frac{\til{\vm}}{\rho_1}\cdot\Grad \varphi \right]\dx\dt \\
		&= - \frac{c_1 + P(\rho_1)}{\rho_1} \int_{\partial \Gamma} \vm_1\cdot \vn_\vx \varphi \dS_{t,\vx} + \frac{c_1 + P(\rho_1)}{\rho_1} \int_{\partial \Gamma} \vm_1\cdot \vn_\vx \varphi \dS_{t,\vx} \ = \ 0 \ed
	\end{align*}
	Hence each $(\rho,\vu)$ defined in \eqref{eq:i1-temp-riemann} is an admissible weak solution to the initial value problem \eqref{eq:baro-euler-pv-dens}, \eqref{eq:baro-euler-pv-mom}, \eqref{eq:isentropic-EOS}, \eqref{eq:riemann-init-isen}. This finishes the proof.
\end{proof}

\begin{rem}
	The density of the solutions obtained from Theorem~\ref{thm:isen-condition} is piecewise constant, whereas the velocity is constant in $\Gamma_-$ and $\Gamma_+$ and ``wild'' in the wedge $\Gamma_1$. All solutions which originate from the same fan subsolution only differ in the velocity in wedge $\Gamma_1$. Note that in most cases there will be more than one admissible fan subsolutions with different $\ov{\rho}$ and different underlying fan partitions. Hence there are also admissible weak solutions which do differ in the density as well. 
\end{rem}

\begin{rem}
	Theorem~\ref{thm:isen-condition} corresponds to \cite[Proposition 3.6]{ChiDelKre15}. Note that our notation slightly differs from the one used in \cite{ChiDelKre15}, especially the notion of a fan subsolution is formulated differently. The reason for this is that our notation is adjusted to the notation used in Theorems \ref{thm:convint} and \ref{thm:convint-nodens}, whereas the notation used in \cite{ChiDelKre15} is suitable to apply a version of \name{De~Lellis}'s and \name{Sz{\'e}kelyhidi}'s ``incompressible'' convex integration.
\end{rem}

\subsubsection{The Corresponding System of Algebraic Equations and Inequalities}

Due to Theorem~\ref{thm:isen-condition} we search for an admissible fan subsolution in order to prove existence of infinitely many admissible weak solutions. Since a fan subsolution consists of piecewise constant functions, the conditions \eqref{eq:isen-fansubs-dens} - \eqref{eq:isen-fansubs-adm} can be translated into a set of algebraic equations and inequalities. This is the content of the following proposition, which was orginally established by \name{Chiodaroli}, \name{De~Lellis} and \name{Kreml} \cite[Propostion 5.1]{ChiDelKre15}.

\begin{prop} \label{prop:isen-alg-eq} 
	Let $\rho_\pm\in\R^+$, $\vu_\pm\in\R^2$. Assume that there exist numbers $\mu_0,\mu_1\in\R$, $\rho_1\in\R^+$, $\alpha_1,\beta_1,\gamma_1,\delta_1\in \R$ and\footnote{Mind the difference between capital $C_1$ and small $c_1$.} $C_1 \in \R^+$ which fulfill the following algebraic equations and inequalities: 
	\begin{itemize}
		\item Order of the speeds:
		\begin{equation} \label{eq:isen-order}
			\mu_0<\mu_1 \es
		\end{equation}
		\item Rankine Hugoniot conditions on the left interface:
		\begin{align}
			\mu_0 (\rho_- - \rho_1) &= \rho_- v_- - \rho_1 \beta_1 \es \label{eq:isen-rhl1}\\
			\mu_0 (\rho_- u_- - \rho_1 \alpha_1) &= \rho_- u_- v_- - \rho_1 \delta_1 \es \label{eq:isen-rhl2}\\
			\mu_0 (\rho_- v_- - \rho_1 \beta_1) &= \rho_- \big(v_-\big)^2 - \rho_1 \left(\frac{C_1}{2} - \gamma_1\right) + p(\rho_-) - p(\rho_1) \es \label{eq:isen-rhl3}
		\end{align}
		\item Rankine Hugoniot conditions on the right interface:
		\begin{align}
			\mu_1 (\rho_1 - \rho_+) &= \rho_1 \beta_1 - \rho_+ v_+ \es \label{eq:isen-rhr1}\\
			\mu_1 (\rho_1 \alpha_1 - \rho_+ u_+) &= \rho_1 \delta_1 - \rho_+ u_+ v_+ \es \label{eq:isen-rhr2}\\
			\mu_1 (\rho_1 \beta_1 - \rho_+ v_+) &= \rho_1 \left(\frac{C_1}{2} - \gamma_1\right) - \rho_+ \big(v_+\big)^2 + p(\rho_1) - p(\rho_+)  \es \label{eq:isen-rhr3}
		\end{align}
		\item Subsolution condition:
		\begin{align}
			C_1 - \big(\alpha_1\big)^2 - \big(\beta_1\big)^2 &>0 \es \label{eq:isen-sc1}\\
			\bigg(\frac{C_1}{2} - \big(\alpha_1\big)^2 + \gamma_1\bigg) \bigg(\frac{C_1}{2} - \big(\beta_1\big)^2 - \gamma_1\bigg) - (\delta_1-\alpha_1 \beta_1)^2 &> 0 \es \label{eq:isen-sc2} 
		\end{align}
		\item Admissibility condition on the left interface: 
		\begin{align}
			&\mu_0 \left(\rho_- \frac{|\vu_-|^2}{2} + P(\rho_-) -\rho_1 \frac{C_1}{2}-P(\rho_1)\right) \notag\\
			&\leq\left(\rho_- \frac{|\vu_-|^2}{2} + P(\rho_-)+p(\rho_-)\right) v_- - \left(\rho_1 \frac{C_1}{2} + P(\rho_1)+p(\rho_1)\right) \beta_1 \es	\label{eq:isen-adml} 
		\end{align}
		\item Admissibility condition on the right interface: 
		\begin{align}
			&\mu_1 \bigg(\rho_1 \frac{C_1}{2} + P(\rho_1) - \rho_+ \frac{|\vu_+|^2}{2} -P(\rho_+) \bigg) \notag\\
			&\leq\left(\rho_1 \frac{C_1}{2} + P(\rho_1) + p(\rho_1)\right) \beta_1 - \left(\rho_+ \frac{|\vu_+|^2}{2} + P(\rho_+)+p(\rho_+)\right) v_+ \es \label{eq:isen-admr} 
		\end{align}
	\end{itemize}
	Then 
	\begin{align}
		\rho_1 &\ec \label{eq:isen-fansubs-defn-rho}\\
		\vm_1&:=\rho_1 \left(
		\begin{array}{c}
			\alpha_1 \\
			\beta_1
		\end{array}\right) \ec \label{eq:isen-fansubs-defn-m}\\
		\mU_1&:=\rho_1\left( 
		\begin{array}{cc}
			\gamma_1 & \delta_1 \\
			\delta_1 & -\gamma_1
		\end{array} \right) \ec \text{ and } \label{eq:isen-fansubs-defn-U} \\
		c_1 &:= \rho_1 \frac{C_1}{2} + p(\rho_1) \label{eq:isen-fansubs-defn-c}
	\end{align} 
	define an admissible fan subsolution to the initial value problem \eqref{eq:baro-euler-pv-dens}, \eqref{eq:baro-euler-pv-mom}, \eqref{eq:isentropic-EOS}, \eqref{eq:riemann-init-isen}, where the corresponding fan partition is determined by the speeds $\mu_0,\mu_1$.
\end{prop}

\begin{proof} 
	Because of \eqref{eq:isen-order}, the numbers $\mu_0,\mu_1$ determine a fan partition. 
	
	Since an admissible fan subsolution $(\ov{\rho},\ov{\vm},\ov{\mU},\ov{c})$ consists of piecewise constant functions, equations \eqref{eq:isen-fansubs-dens} and \eqref{eq:isen-fansubs-mom} are equivalent to their corresponding Rankine-Hugoniot equations\footnote{For more details on \emph{Rankine-Hugoniot conditions} we refer to standard textbooks on hyperbolic conservation laws, e.g. \name{Dafermos} \cite{Dafermos}. The principal ideal in order to derive the algebraic equations \eqref{eq:isen-rhl1-temp} - \eqref{eq:isen-rhr3-temp} is to apply the Divergence Theorem in \eqref{eq:isen-fansubs-dens} and \eqref{eq:isen-fansubs-mom} on each $\Gamma_i$ with $i=-,1,+$. This procedure yields boundary terms which hold if and only if \eqref{eq:isen-rhl1-temp} - \eqref{eq:isen-rhr3-temp} are satisfied.}, i.e.
	\begin{align}
		\mu_0 (\rho_- - \rho_1) &= \big[ \vm_- - \vm_1\big]_2 \ec \label{eq:isen-rhl1-temp} \\
		\mu_0 \big[ \vm_- - \vm_1\big]_1 &= \big[ \mU_- - \mU_1 \big]_{12} \label{eq:isen-rhl2-temp}\ec\\
		\mu_0 \big[ \vm_- - \vm_1\big]_2 &= \big[ \mU_- - \mU_1 \big]_{22} + c_- - c_1 \label{eq:isen-rhl3-temp}
	\end{align} 
	on the left interface, and 
	\begin{align}
		\mu_1 (\rho_1 - \rho_+) &= \big[ \vm_1 - \vm_+\big]_2 \ec \label{eq:isen-rhr1-temp}\\
		\mu_1 \big[ \vm_1 - \vm_+\big]_1 &= \big[ \mU_1 - \mU_+ \big]_{12} \ec \label{eq:isen-rhr2-temp}\\
		\mu_1 \big[ \vm_1 - \vm_+\big]_2 &= \big[ \mU_1 - \mU_+ \big]_{22} + c_1 - c_+ \label{eq:isen-rhr3-temp}
	\end{align}
	on the right interface. Analogously we obtain that \eqref{eq:isen-fansubs-adm} is equivalent to the admissibility conditions
	\begin{align} 
		&\mu_0\Big( c_- - p(\rho_-) + P(\rho_-) - c_1 + p(\rho_1) - P(\rho_1) \Big) \notag \\
		&\leq \Big(c_- + P(\rho_-) \Big) \frac{[\vm_-]_2}{\rho_-} - \Big(c_1 + P(\rho_1) \Big) \frac{[\vm_1]_2}{\rho_1} \label{eq:isen-adml-temp}
	\end{align}
	on the left interface, and 
	\begin{align} 
		&\mu_1\Big( c_1 - p(\rho_1) + P(\rho_1) - c_+ + p(\rho_+) - P(\rho_+) \Big) \notag \\
		&\leq \Big(c_1 + P(\rho_1) \Big) \frac{[\vm_1]_2}{\rho_1} - \Big(c_+ + P(\rho_+) \Big) \frac{[\vm_+]_2}{\rho_+} \label{eq:isen-admr-temp}
	\end{align}
	on the right interface. 
	
	In order to show \eqref{eq:isen-fansubs-dens} - \eqref{eq:isen-fansubs-adm}, it suffices to prove \eqref{eq:isen-rhl1-temp} - \eqref{eq:isen-admr-temp}. Indeed using \eqref{eq:isen-fansubs-defn-rho} - \eqref{eq:isen-fansubs-defn-c}, the equations \eqref{eq:isen-rhl1-temp} - \eqref{eq:isen-admr-temp} turn into \eqref{eq:isen-rhl1} - \eqref{eq:isen-rhr3} and \eqref{eq:isen-adml}, \eqref{eq:isen-admr} which shows that \eqref{eq:isen-rhl1-temp} - \eqref{eq:isen-admr-temp} hold.
	
	It remains to show \eqref{eq:isen-fansubs-subs}. To this end we must prove that both eigenvalues $\lambda_1,\lambda_2$ of the matrix
	\begin{equation} \label{eq:i4-temp-riemann}
		\frac{\vm_1\otimes\vm_1}{\rho_1} + p(\rho_1)\id -\mU_1 = \rho_1 \left( 
		\begin{array}{cc}
			\big(\alpha_1\big)^2 -\gamma_1 & \alpha_1 \beta_1 - \delta_1 \\
			\alpha_1 \beta_1 - \delta_1 & \big(\beta_1\big)^2 + \gamma_1
		\end{array} \right) + p(\rho_1)\id 
	\end{equation}
	are smaller than $c_1$. In other words we need to show that 
	\begin{align*}
		c_1 - \lambda_1 &> 0 \qquad \text{ and }\\
		c_1 - \lambda_2 &> 0 \ec
	\end{align*}
	which is equivalent to 
	\begin{align}
		2c_1 - (\lambda_1 + \lambda_2) = (c_1 - \lambda_1) + (c_1 - \lambda_2) &> 0 \qquad \text{ and } \label{eq:i5-temp-riemann} \\
		c_1^2 - c_1 (\lambda_1 + \lambda_2) + \lambda_1\lambda_2 = (c_1 - \lambda_1) (c_1 - \lambda_2) &> 0 \ed \label{eq:i6-temp-riemann}
	\end{align}
	With \eqref{eq:i4-temp-riemann} we obtain
	\begin{align*}
		\lambda_1+\lambda_2 &= \tr \left(\frac{\vm_1\otimes\vm_1}{\rho_1} + p(\rho_1)\id -\mU_1\right) = \rho_1\big((\alpha_1)^2 + (\beta_1)^2 \big) + 2p(\rho_1) \qquad \text{ and } \\
		\lambda_1 \lambda_2 &= \det \left(\frac{\vm_1\otimes\vm_1}{\rho_1} + p(\rho_1)\id -\mU_1\right) \\
		&= \Big( \rho_1 \big((\alpha_1)^2 - \gamma_1\big) +p(\rho_1) \Big) \Big( \rho_1 \big((\beta_1)^2 + \gamma_1\big) +p(\rho_1) \Big) - \rho_1^2 (\alpha_1\beta_1 - \delta_1)^2 \ed
	\end{align*}
	Plugging this and \eqref{eq:isen-fansubs-defn-c} into \eqref{eq:i5-temp-riemann} and \eqref{eq:i6-temp-riemann}, we observe that \eqref{eq:isen-fansubs-subs} is true as soon as \eqref{eq:isen-sc1} and \eqref{eq:isen-sc2} are satisfied.

	Thus $(\rho_1,\vm_1,\mU_1,c_1)$ indeed define an admissible fan subsolution.
\end{proof}

\begin{rem} 
	The converse of Proposition~\ref{prop:isen-alg-eq} is also true: If there exists an admissible fan subsolution to \eqref{eq:baro-euler-pv-dens}, \eqref{eq:baro-euler-pv-mom}, \eqref{eq:isentropic-EOS}, \eqref{eq:riemann-init-isen} then $\alpha_1,\beta_1,\gamma_1,\delta_1\in \R$ and $C_1\in\R^+$, which are uniquely determined by \eqref{eq:isen-fansubs-defn-m} - \eqref{eq:isen-fansubs-defn-c}, satisfy \eqref{eq:isen-rhl1} - \eqref{eq:isen-admr}. In other words: If there are no numbers $\mu_0,\mu_1\in \R$, $\rho_1\in \R^+$, $\alpha_1,\beta_1,\gamma_1,\delta_1\in \R$ and $C_1\in\R^+$ that fulfill \eqref{eq:isen-order} - \eqref{eq:isen-admr}, then there is no admissible fan subsolution. 
\end{rem}

\begin{rem}
	In contrast to our definition of an admissible fan subsolution (Definition~\ref{defn:isen-fansubs}) and Theorem~\ref{thm:isen-condition}, where we use a notation that is suitable for the application of Theorems \ref{thm:convint} and \ref{thm:convint-nodens}, we now switch to the notation used in literature, e.g. in \cite{ChiDelKre15}, \cite{ChiKre14} and \cite{KliMar18_1}. This is also the reason for replacing small $c_1$ by capital $C_1$ via \eqref{eq:isen-fansubs-defn-c}.
\end{rem}

\begin{rem} 
	The index $1$ of $\alpha,\beta,$ etc. indicates to which set $\Gamma_i$ of the fan partition the quantity corresponds. Since there is only $\Gamma_1$, see Definition~\ref{defn:isen-fanpart}, only the index $1$ appears. Hence one could drop it for convenience. However we decided to keep it for consistency. Later when we consider the full Euler system, we need slightly different fan partitions, which contain two wedges $\Gamma_1$ and $\Gamma_2$, see Definition~\ref{defn:full-fanpart}. Then the index becomes essential.
\end{rem}

\begin{rem}
	Once more we would like to emphasize that 
	\begin{itemize}
		\item the consideration of fan subsolutions (Definition~\ref{defn:isen-fansubs}), 
		\item the fact that existence of infinitely many solutions reduces to finding a fan subsolution (Theorem~\ref{thm:isen-condition}) and 
		\item the reformulation of Definition~\ref{defn:isen-fansubs} as a set of algebraic equations and inequalities (Proposition~\ref{prop:isen-alg-eq}), 
	\end{itemize}
	are not only used in the case where the self-similar solution consists of one shock and one rarefaction, see also Subsection~\ref{subsec:riemann-isen-other-cases}. 
\end{rem}

\subsubsection{Simplification of the Algebraic System}

Now we start with the actual proof of Theorem~\ref{thm:isen-SR} as we have discussed the preliminaries. So let the assumptions of Theorem~\ref{thm:isen-SR}, i.e. equations \eqref{eq:1S3R.a} - \eqref{eq:1S3R.c}, be satisfied. The following can be also found in \name{Klingenberg}-\name{Markfelder} \cite{KliMar18_1}.

The proof of Theorem~\ref{thm:isen-SR} requires some inequalities which we summarize in the following Lemma. 

\begin{lemma} \label{lemma:isen-inequalities} 
	\begin{enumerate}
		\item \label{item:isen-inequ.a} For all $\rho_-\neq\rho_+$, $\rho_\pm>0$ it holds that 
		\begin{equation*}
		p(\rho_-) + p(\rho_+) - 2 \frac{\rho_- P(\rho_+) - \rho_+ P(\rho_-)}{\rho_+ - \rho_-} > 0 \ed
		\end{equation*}
		
		\item \label{item:isen-inequ.b} For all $\rho_- < \rho_+$ it holds that 
		\begin{equation*}
		\int_{\rho_-}^{\rho_+}\frac{\sqrt{p'(r)}}{r} \dr < \sqrt{\frac{(\rho_- - \rho_+) \big(p(\rho_-)-p(\rho_+)\big)}{\rho_- \rho_+}} \ed
		\end{equation*} 
		
		\item \label{item:isen-inequ.c} For all $\rho_- < \rho_M < \rho_+$ it holds that 
		\begin{equation*}
		\sqrt{\frac{\big(\rho_M - \rho_-\big) \big(p(\rho_M)-p(\rho_-)\big)}{\rho_- \rho_M}} < \sqrt{\frac{\big(\rho_+ - \rho_-\big) \big(p(\rho_+)-p(\rho_-)\big)}{\rho_- \rho_+}} \ed
		\end{equation*} 
	\end{enumerate}
\end{lemma}

\begin{proof}
	For the proof of Lemma~\ref{lemma:isen-inequalities} we refer to the literature: \ref{item:isen-inequ.a} can be found in \name{Chiodaroli}-\name{Kreml} \cite[Lemma 2.1]{ChiKre14}, whereas \ref{item:isen-inequ.b} and \ref{item:isen-inequ.c} are proven by \name{Klingenberg} and the author \cite[Lemmas 4.2 and 4.3]{KliMar18_1}. 
\end{proof}

Next we rearrange the equations in inequalities in Proposition~\ref{prop:isen-alg-eq}. Since there are 6 equations (\eqref{eq:isen-rhl1} - \eqref{eq:isen-rhr3}) for 8 unknowns, we choose two unknowns as parameters and express the other 6 quantities as functions of those parameters via equations \eqref{eq:isen-rhl1} - \eqref{eq:isen-rhr3}. 

Define the functions $\beta_1,\ep_1:(\rho_-,\rho_+)\to \R$ by 
\begin{align}
	\beta_1(\rho_1) &:= \frac{1}{\rho_1 (\rho_- - \rho_+)} \Bigg(-\rho_- v_- (\rho_+ - \rho_1) - \rho_+ v_+ (\rho_1 - \rho_-) \label{eq:SRdefnv12}\\
	&\quad + \sqrt{\Big[(\rho_- - \rho_+) \big(p(\rho_-) - p(\rho_+)\big) - \rho_+ \rho_- (v_- - v_+)^2\Big] (\rho_1 - \rho_-) (\rho_+ - \rho_1)}\Bigg) \ec \notag \\
	\ep_1(\rho_1) &:= -\frac{p(\rho_1) - p(\rho_-)}{\rho_1} + \frac{\rho_- (\rho_1 - \rho_-)}{\rho_1^2 (\rho_- - \rho_+)^2} \Bigg(\rho_+ (v_- - v_+) \label{eq:SRdefndelta1} \\
	&\quad + \sqrt{\Big[(\rho_- - \rho_+) \big(p(\rho_-) - p(\rho_+)\big) - \rho_+ \rho_- (v_- - v_+)^2\Big] \frac{\rho_+ - \rho_1}{\rho_1 - \rho_-}}\Bigg)^2 \ed \notag
\end{align}

Note that these functions are well-defined for initial states $(\rho_\pm,\vu_\pm)$ fulfilling \eqref{eq:1S3R.a} - \eqref{eq:1S3R.c} and for $\rho_- < \rho_1 < \rho_+$. Indeed \eqref{eq:1S3R.a}, \eqref{eq:1S3R.c} and Lemma~\ref{lemma:isen-inequalities} \ref{item:isen-inequ.b} imply 
$$
	\frac{(\rho_- - \rho_+)\big(p(\rho_-) - p(\rho_+)\big)}{\rho_- \rho_+} > (v_- - v_+)^2
$$
which is equivalent to 
$$
	(\rho_- - \rho_+) \big(p(\rho_-) - p(\rho_+)\big) - \rho_+ \rho_- (v_- - v_+)^2 > 0\ed
$$
Hence the square roots in \eqref{eq:SRdefnv12} and \eqref{eq:SRdefndelta1} are well-defined.

With the help of the functions $\beta_1$ and $\ep_1$ we are able to define numbers $\mu_0,\mu_1\in\R$, $\rho_1\in\R^+$, $\alpha_1,\beta_1,\gamma_1,\delta_1\in \R$ and $C_1 \in \R^+$ as required by Propostion \ref{prop:isen-alg-eq}:

\begin{prop} \label{prop:SRsubs} 
	Assume that there exist numbers $\rho_1,\til{\ep}_1\in\R^+$ that fulfill the following inequalities:
	\begin{align} 
		\rho_- &< \rho_1 < \rho_+\es \label{eq:SRcond1} \\
		0 &<\ep_1(\rho_1) \es \label{eq:SRcond2} 
	\end{align} 
	\begin{align}
		&(\beta_1(\rho_1)-v_-) \bigg(p(\rho_-)+p(\rho_1)-2  \frac{\rho_1 P(\rho_-)-\rho_- P(\rho_1)}{\rho_- -\rho_1}\bigg) \notag \\
		&\leq\ep_1(\rho_1) \rho_1 (\beta_1(\rho_1)+v_-)-(\ep_1(\rho_1)+\til{\ep}_1) \frac{\rho_- \rho_1 (\beta_1(\rho_1)-v_-)}{\rho_- -\rho_1} \es \label{eq:SRcond3} \\
		&(v_+-\beta_1(\rho_1)) \bigg(p(\rho_1)+p(\rho_+)-2 \frac{\rho_+ P(\rho_1)-\rho_1 P(\rho_+)}{\rho_1 -\rho_+}\bigg) \notag \\
		&\leq-\ep_1(\rho_1) \rho_1 (v_++\beta_1(\rho_1))+(\ep_1(\rho_1)+\til{\ep}_1) \frac{\rho_1 \rho_+ (v_+-\beta_1(\rho_1))}{\rho_1 -\rho_+} \ed \label{eq:SRcond4}
	\end{align}
	Then 
	\begin{align}
		\mu_0 &:= \frac{\rho_- v_- - \rho_+ v_+}{\rho_- - \rho_+} \notag \\
		&\qquad + \frac{1}{\rho_- - \rho_+} \sqrt{\Big[(\rho_- - \rho_+) \big(p(\rho_-) - p(\rho_+)\big) - \rho_+ \rho_- (v_- - v_+)^2\Big] \frac{\rho_+ - \rho_1}{\rho_1 - \rho_-}} \ec \notag \\
		\mu_1 &:= \frac{\rho_- v_- - \rho_+ v_+}{\rho_- - \rho_+} \notag \\
		&\qquad - \frac{1}{\rho_- - \rho_+} \sqrt{\Big[(\rho_- - \rho_+) \big(p(\rho_-) - p(\rho_+)\big) - \rho_+ \rho_- (v_- - v_+)^2\Big] \frac{\rho_1 - \rho_-}{\rho_+ - \rho_1}} \ec \notag \\
		\rho_1 & \ec \notag \\
		\alpha_1 &:= u_- \ec \notag \\ 
		\beta_1 &:= \beta_1(\rho_1) \ec \notag \\
		\gamma_1 &:= \half \Big(\til{\ep}_1 - \ep_1(\rho_1) + \big(u_-\big)^2- \beta_1(\rho_1)^2\Big)\ec \label{eq:defn-gamma-temp} \\ 
		\delta_1 &:= u_- \beta_1(\rho_1) \ec \notag \\ 
		C_1 &:=\til{\ep}_1 + \ep_1(\rho_1) + \big(u_-\big)^2 + \beta_1(\rho_1)^2 \label{eq:defn-C-temp}
	\end{align} 
	satisfy \eqref{eq:isen-order} - \eqref{eq:isen-admr}.
\end{prop}

\begin{proof}
	The proof of Proposition~\ref{prop:SRsubs} is a matter of straightforward calculation, where one has to recall that by \eqref{eq:1S3R.b} $u_-=u_+$. See also \cite[Theorem 5.2]{KliMar18_1} for more details.
\end{proof}

\begin{rem} 
	The converse of Proposition~\ref{prop:SRsubs} is also true, see \cite[Theorem 5.2]{KliMar18_1}: If there are no numbers $\rho_1,\til{\ep}_1\in\R^+$ that fulfill \eqref{eq:SRcond1} - \eqref{eq:SRcond4}, then there are no numbers $\mu_0,\mu_1\in\R$, $\rho_1\in\R^+$, $\alpha_1,\beta_1,\gamma_1,\delta_1\in \R$ and $C_1 \in \R^+$ which satisfy \eqref{eq:isen-order} - \eqref{eq:isen-admr}. Hence in this case there is no admissible fan subsolution, see the remark below Proposition~\ref{prop:isen-alg-eq}. In fact there are initial states for which the self-similar solution consists of a 1-shock and a 3-rarefaction, such that there are no numbers as required by Proposition~\ref{prop:SRsubs}, see the master thesis of the author \cite[Example 5.9]{MarkfelderMaster} for an example. This means that there is no admissible fan subsolution. However this does not mean that the self-similar solution is the only solution. Indeed we can still prove existence of infinitely many solutions. The idea is to work with an auxiliary state, see below.
\end{rem}

\subsubsection{Solution of the Algebraic System if the Rarefaction is ``Small''}

Let us now prove that there exist numbers $\mu_0,\mu_1\in\R$, $\rho_1\in\R^+$, $\alpha_1,\beta_1,\gamma_1,\delta_1\in \R$ and $C_1 \in \R^+$ as required by Proposition~\ref{prop:SRsubs}, i.e. fulfilling \eqref{eq:SRcond1} - \eqref{eq:SRcond4}, as long as the rarefaction is ``small''. This fact has also been shown independently by \name{Chiodaroli} and \name{Kreml} \cite{ChiKre18}.

For convenience we use from now on the notation $\sO:=\R^+\times\R$ for the phase space and $\vU:=(\rho,v)\in\sO$ for a state. Note that there is a small difference between the notation used here and in Section~\ref{sec:euler-barotropic}. Here $\vU$ is a state in primitive variables where as in \eqref{eq:baro-U} it is a state in conserved variables. Furthermore only the relevant component of the velocity, i.e. $v=[\vu]_2$, is contained in $\vU$ here. Analogously there is a difference between $\sO$ used here and in \eqref{eq:baro-O}.

The following ``smallness'' result will help to prove Theorem~\ref{thm:isen-SR}. We will forget about the given initial states $\vU_-=(\rho_-,v_-)$, $\vU_+=(\rho_+,v_+)$ for a moment.

\begin{prop} \label{prop:isen-SR-smallness}
	Let $\til{\vU}_-=(\til{\rho}_-,\til{v}_-)\in\sO$ be any given state and $\til{\vU}_M=(\til{\rho}_M,\til{v}_M)\in\sO$ a state that can be connected to $\til{\vU}_-$ by a 1-shock. Then the following statement is valid: If $\til{\vU}_+=(\til{\rho}_+,\til{v}_+)\in\sO$ is a state such that  
	\begin{itemize}
		\item $\til{\rho}_+>\til{\rho}_M$,
		\item $\til{\vU}_+$ is sufficiently close to $\til{\vU}_M$, and
		\item the self-similar solution to the initial value problem \eqref{eq:baro-euler-pv-dens}, \eqref{eq:baro-euler-pv-mom}, \eqref{eq:isentropic-EOS}, \eqref{eq:riemann-init-isen} with 
		$$
		\big(\til{\rho}_-,(u_-,\til{v}_-)^\trans\big)\quad \text{ and }\quad \big(\til{\rho}_+,(u_-,\til{v}_+)^\trans\big)
		$$ 
		as initial states consists of a 1-shock and a 3-rarefaction,
	\end{itemize}
	then there exists an admissible fan subsolution to the initial value problem \eqref{eq:baro-euler-pv-dens}, \eqref{eq:baro-euler-pv-mom}, \eqref{eq:isentropic-EOS}, \eqref{eq:riemann-init-isen} with $\big(\til{\rho}_-,(u_-,\til{v}_-)^\trans\big)$ and $\big(\til{\rho}_+,(u_-,\til{v}_+)^\trans\big)$ as initial states. In addition to that the density $\rho_1$ that appears in the admissible fan subsolution satisfies $\rho_1<\til{\rho}_M$.
\end{prop}

\begin{proof}
	We are going to use Propositions \ref{prop:SRsubs} and \ref{prop:isen-alg-eq}. Hence in order to prove that there exists an admissible fan subsolution, it suffices to show that there exist $\rho_1,\til{\ep}_1\in\R^+$ which satisfy \eqref{eq:SRcond1} - \eqref{eq:SRcond4}. 
	
	In view of the functions $\beta_1$ and $\ep_1$ (see \eqref{eq:SRdefnv12}, \eqref{eq:SRdefndelta1}), we define $\beta_1^\ast,\ep_1^\ast:\R^+\times\sO\rightarrow\R$ as 
	\begin{align*} 
		\beta_1^\ast(\rho_1,\til{\vU}_+) &:= \frac{1}{\rho_1 (\til{\rho}_- - \til{\rho}_+)} \Bigg(-\til{\rho}_- \til{v}_- (\til{\rho}_+ - \rho_1) - \til{\rho}_+ \til{v}_+ (\rho_1 - \til{\rho}_-) \\
		&\quad + \sqrt{\Big[(\til{\rho}_- - \til{\rho}_+) \big(p(\til{\rho}_-) - p(\til{\rho}_+)\big) - \til{\rho}_+ \til{\rho}_- (\til{v}_- - \til{v}_+)^2\Big] (\rho_1 - \til{\rho}_-) (\til{\rho}_+ - \rho_1)}\Bigg) \ec\\
		\ep_1^\ast(\rho_1,\til{\vU}_+) &:= -\frac{p(\rho_1) - p(\til{\rho}_-)}{\rho_1} + \frac{\til{\rho}_- (\rho_1 - \til{\rho}_-)}{\rho_1^2 (\til{\rho}_- - \til{\rho}_+)^2} \Bigg(\til{\rho}_+ (\til{v}_- - \til{v}_+) \\
		&\quad+ \sqrt{\Big[(\til{\rho}_- - \til{\rho}_+) \big(p(\til{\rho}_-) - p(\til{\rho}_+)\big) - \til{\rho}_+ \til{\rho}_- (\til{v}_- - \til{v}_+)^2\Big] \frac{\til{\rho}_+ - \rho_1}{\rho_1 - \til{\rho}_-}}\Bigg)^2 \ed
	\end{align*}
	In addition to that we define functions $A,B:\R^+\times\R^+\times\sO\rightarrow\R$ as\footnote{Note that the functions $\beta_1^\ast,\ep_1^\ast$ and $A,B$ are not well-defined on the whole sets $\R^+\times \sO$ and $\R^+\times \R^+\times \sO$ respectively. However we will only look at points where they are well-defined or consider limits to the boundary of those domains.}
	\begin{align*}
		A(\rho_1,\til{\ep}_1,\til{\vU}_+) &:= \ep_1^\ast(\rho_1,\til{\vU}_+) \rho_1 \big(\beta_1^\ast(\rho_1,\til{\vU}_+)+\til{v}_-\big) \\
		&\qquad -\big(\ep_1^\ast(\rho_1,\til{\vU}_+)+\til{\ep}_1\big) \frac{\til{\rho}_- \rho_1 \big(\beta_1^\ast(\rho_1,\til{\vU}_+)-\til{v}_-\big)}{\til{\rho}_- -\rho_1} \\
		&\qquad - \big(\beta_1^\ast(\rho_1,\til{\vU}_+)-\til{v}_-\big) \bigg(p(\til{\rho}_-)+p(\rho_1)- 2  \frac{\rho_1 P(\til{\rho}_-) - \til{\rho}_- P(\rho_1)}{\til{\rho}_- -\rho_1}\bigg) \ec \\
		B(\rho_1,\til{\ep}_1,\til{\vU}_+) &:= -\ep_1^\ast(\rho_1,\til{\vU}_+) \rho_1 \big(\til{v}_++\beta_1^\ast(\rho_1,\til{\vU}_+)\big) \\
		&\qquad +\big(\ep_1^\ast(\rho_1,\til{\vU}_+)+\til{\ep}_1\big) \frac{\rho_1 \til{\rho}_+ \big(\til{v}_+-\beta_1^\ast(\rho_1,\til{\vU}_+)\big)}{\rho_1 -\til{\rho}_+} \\
		&\qquad - \big(\til{v}_+-\beta_1^\ast(\rho_1,\til{\vU}_+)\big) \bigg(p(\rho_1)+p(\til{\rho}_+)- 2 \frac{\til{\rho}_+ P(\rho_1) - \rho_1  P(\til{\rho}_+)}{\rho_1 -\til{\rho}_+}\bigg) \ed
	\end{align*}
	
	Since $\til{\vU}_-$ and $\til{\vU}_M$ can be connected by a 1-shock we obtain according to Proposition~\ref{prop:isen-selfsimilar}~\ref{item:isen-cases.S} that $\til{\rho}_-<\til{\rho}_M$ and 
	\begin{equation} \label{eq:temp42}
		\til{v}_- - \til{v}_M = \sqrt{\frac{(\til{\rho}_M - \til{\rho}_-) \big(p(\til{\rho}_M) - p(\til{\rho}_-)\big)}{\til{\rho}_M \til{\rho}_-}} \ed
	\end{equation} 
	
	Next we show that there exists $\rho_1\in(\til{\rho}_-,\til{\rho}_M)$ such that 
	\begin{align}
		\ep_1^\ast(\rho_1,\til{\vU}_+=\til{\vU}_M) &>0 \ec \label{eq:SRconddeltaABd}\\
		A(\rho_1,\til{\ep}_1=0,\til{\vU}_+=\til{\vU}_M) &>0 \ec \label{eq:SRconddeltaABA}\\
		B(\rho_1,\til{\ep}_1=0,\til{\vU}_+=\til{\vU}_M) &>0 \ed \label{eq:SRconddeltaABB}
	\end{align}
	
	First we prove that \eqref{eq:SRconddeltaABd} is true for all $\rho_1\in(\til{\rho}_-,\til{\rho}_M)$. Using \eqref{eq:temp42} we obtain
	\begin{equation*}
		\ep_1^\ast(\rho_1,\til{\vU}_+=\til{\vU}_M) = -\frac{p(\rho_1) - p(\til{\rho}_-)}{\rho_1} + \frac{\til{\rho}_M}{\rho_1^2}\ \frac{p(\til{\rho}_M) - p(\til{\rho}_-)}{\til{\rho}_M - \til{\rho}_-}\ (\rho_1 - \til{\rho}_-) \ed
	\end{equation*}
	Each $\rho_1\in(\til{\rho}_-,\til{\rho}_M)$ can be written as a convex combination of $\til{\rho}_-$ and $\til{\rho}_M$. In other words there exists $\tau\in(0,1)$ such that 
	\begin{equation*}
		\rho_1=\tau \til{\rho}_- + (1-\tau) \til{\rho}_M \ed
	\end{equation*} 
	Since $p$ is a convex function of $\rho$ we have 
	\begin{equation*}
		p(\rho_1)=p\big(\tau \til{\rho}_- + (1-\tau) \til{\rho}_M\big) \leq \tau p(\til{\rho}_-) + (1-\tau) p(\til{\rho}_M)
	\end{equation*}
	and hence
	\begin{align*}
		\ep_1^\ast(\rho_1,\til{\vU}_+=\til{\vU}_M) &= \frac{1}{\rho_1} \bigg( -p(\rho_1) + p(\til{\rho}_-) + \frac{\til{\rho}_M}{\rho_1}\ \frac{p(\til{\rho}_M) - p(\til{\rho}_-)}{\til{\rho}_M - \til{\rho}_-}\ (\rho_1 - \til{\rho}_-)\bigg) \\
		&\geq \frac{1}{\rho_1^2} \tau (1-\tau) \big(p(\til{\rho}_M) - p(\til{\rho}_-)\big) (\til{\rho}_M - \til{\rho}_-)\ >\ 0 \ed
	\end{align*}
	Therefore \eqref{eq:SRconddeltaABd} is true for all $\rho_1\in(\til{\rho}_-,\til{\rho}_M)$. 
	
	For convenience we define
	\begin{equation*}
		R := \sqrt{\frac{(\til{\rho}_M - \til{\rho}_-) \big(p(\til{\rho}_M) - p(\til{\rho}_-)\big)}{\til{\rho}_M \til{\rho}_-}} \ed
	\end{equation*}
	
	To show the existence of $\rho_1\in(\til{\rho}_-,\til{\rho}_M)$ that satisfies \eqref{eq:SRconddeltaABA} and \eqref{eq:SRconddeltaABB} we consider two cases: Let first 
	\begin{equation*}
		\til{v}_->\frac{\til{\rho}_M}{2 (\til{\rho}_M-\til{\rho}_-)} R \ed
	\end{equation*}
	An easy computation leads to
	\begin{equation*}
	\lim\limits_{\rho_1\rightarrow\til{\rho}_-} A(\rho_1,\til{\ep}_1=0,\til{\vU}_+=\til{\vU}_M)=0 \ec
	\end{equation*}
	and also
	\begin{align*} 
	\lim\limits_{\rho_1\rightarrow\til{\rho}_-}&\bigg(\frac{\partial}{\partial\rho_1}A(\rho_1,\til{\ep}_1=0,\til{\vU}_+=\til{\vU}_M)\bigg)\\
	&=\bigg(-\frac{\til{\rho}_M}{\til{\rho}_M -\til{\rho}_-} R + 2 \til{v}_-\bigg) \bigg(-p'(\til{\rho}_-) + \frac{\til{\rho}_M}{\til{\rho}_-}\ \ \frac{p(\til{\rho}_M) - p(\til{\rho}_-)}{\til{\rho}_M - \til{\rho}_-}\bigg) \ed
	\end{align*}
	
	In the case under consideration it holds that 
	\begin{equation*}
		-\frac{\til{\rho}_M}{\til{\rho}_M -\til{\rho}_-} R + 2 v_- > -\frac{\til{\rho}_M}{\til{\rho}_M -\til{\rho}_-} R + \frac{\til{\rho}_M}{\til{\rho}_M-\til{\rho}_-} R = 0 \ed
	\end{equation*}
	In addition to that, the fact that $\til{\rho}_-<\til{\rho}_M$ and the convexity of $p$ lead to 
	\begin{equation*}
		-p'(\til{\rho}_-) + \frac{\til{\rho}_M}{\til{\rho}_-}\ \ \frac{p(\til{\rho}_M) - p(\til{\rho}_-)}{\til{\rho}_M - \til{\rho}_-} > -p'(\til{\rho}_-) + \frac{p(\til{\rho}_M) - p(\til{\rho}_-)}{\til{\rho}_M - \til{\rho}_-} \geq 0 \ed
	\end{equation*}
	Hence
	\begin{equation*} 
		\lim\limits_{\rho_1\rightarrow\til{\rho}_-}\bigg(\frac{\partial}{\partial\rho_1}A(\rho_1,\til{\ep}_1=0,\til{\vU}_+=\til{\vU}_M)\bigg) > 0 \ed
	\end{equation*}
	
	By continuity of the function $\rho_1 \mapsto A(\rho_1,\til{\ep}_1=0,\til{\vU}_+=\til{\vU}_M)$ there exists $\rho_1\in(\til{\rho}_-,\til{\rho}_M)$ where $\rho_1\approx\til{\rho}_-$ such that \eqref{eq:SRconddeltaABA} holds. 
	
	Another computation shows that 
	\begin{align*} 
		\lim\limits_{\rho_1\rightarrow\til{\rho}_-}&B(\rho_1,\til{\ep}_1=0,\til{\vU}_+=\til{\vU}_M) \\
		&=R \bigg(p(\til{\rho}_-)+p(\til{\rho}_M) - 2 \frac{\til{\rho}_- P(\til{\rho}_M) - \til{\rho}_M P(\til{\rho}_-)}{\til{\rho}_M - \til{\rho}_-}\bigg)>0 \ec 
	\end{align*} 
	according to Lemma~\ref{lemma:isen-inequalities} \ref{item:isen-inequ.a}. Hence by continuity of $\rho_1 \mapsto B(\rho_1,\til{\ep}_1=0,\til{\vU}_+=\til{\vU}_M)$ we can choose $\rho_1\in(\til{\rho}_-,\til{\rho}_M)$ such that \eqref{eq:SRconddeltaABB} is fulfilled in addition to \eqref{eq:SRconddeltaABA}. 
	
	Suppose now the second case
	\begin{equation*}
		\til{v}_- \leq\frac{\til{\rho}_M}{2 (\til{\rho}_M-\til{\rho}_-)} R \ed
	\end{equation*}
	Similar computations yield
	\begin{align*}
		\lim\limits_{\rho_1\rightarrow\til{\rho}_-}&A(\rho_1,\til{\ep}_1=0,\til{\vU}_+=\til{\vU}_M) \\ 
		&=R \bigg(p(\til{\rho}_-)+p(\til{\rho}_M) - 2 \frac{\til{\rho}_- P(\til{\rho}_M) - \til{\rho}_M P(\til{\rho}_-)}{\til{\rho}_M - \til{\rho}_-}\bigg)>0 \ec 
	\end{align*}
	and furthermore
	\begin{equation*} 
		\lim\limits_{\rho_1\rightarrow\til{\rho}_M} B(\rho_1,\til{\ep}_1=0,\til{\vU}_+=\til{\vU}_M)=0 \ec
	\end{equation*}
	together with
	\begin{align*}
		\lim\limits_{\rho_1\rightarrow\til{\rho}_M}&\bigg(\frac{\partial}{\partial\rho_1}B(\rho_1,\til{\ep}_1=0,\til{\vU}_+=\til{\vU}_M)\bigg)\\
		&=\bigg(-\frac{2 \til{\rho}_M-\til{\rho}_-}{\til{\rho}_M -\til{\rho}_-} R + 2 \til{v}_-\bigg) \bigg(p'(\til{\rho}_M) - \frac{\til{\rho}_-}{\til{\rho}_M}\ \ \frac{p(\til{\rho}_M) - p(\til{\rho}_-)}{\til{\rho}_M - \til{\rho}_-}\bigg) \ed
	\end{align*}
	
	In the considered case we have
	\begin{equation*}
		-\frac{2 \til{\rho}_M-\til{\rho}_-}{\til{\rho}_M -\til{\rho}_-} R + 2 \til{v}_- \leq -\frac{2 \til{\rho}_M-\til{\rho}_-}{\til{\rho}_M -\til{\rho}_-} R + \frac{\til{\rho}_M}{\til{\rho}_M-\til{\rho}_-} R = -R < 0 \ec
	\end{equation*}
	Additionally the convexity of $p$ and $\til{\rho}_-<\til{\rho}_M$ lead to
	\begin{equation*}
		p'(\til{\rho}_M) - \frac{\til{\rho}_-}{\til{\rho}_M}\ \ \frac{p(\til{\rho}_M) - p(\til{\rho}_-)}{\til{\rho}_M - \til{\rho}_-} > p'(\til{\rho}_M) - \frac{p(\til{\rho}_M) - p(\til{\rho}_-)}{\til{\rho}_M - \til{\rho}_-} \geq 0 \ed
	\end{equation*}
	Hence
	\begin{equation*} 
		\lim\limits_{\rho_1\rightarrow\til{\rho}_M}\bigg(\frac{\partial}{\partial\rho_1}B(\rho_1,\til{\ep}_1=0,\til{\vU}_+=\til{\vU}_M)\bigg) < 0 \ed
	\end{equation*}
	Thus by continuity of $\rho_1 \mapsto A(\rho_1,\til{\ep}_1=0,\til{\vU}_+=\til{\vU}_M)$ and $\rho_1 \mapsto B(\rho_1,\til{\ep}_1=0,\til{\vU}_+=\til{\vU}_M)$ there exists $\rho_1\in(\til{\rho}_-,\til{\rho}_M)$ such that \eqref{eq:SRconddeltaABA} and \eqref{eq:SRconddeltaABB} hold, where $\rho_1\approx\til{\rho}_M$. 
	
	By continuity of the functions $\til{\ep}_1 \mapsto A(\rho_1,\til{\ep}_1,\til{\vU}_+=\til{\vU}_M)$ and $\til{\ep}_1 \mapsto B(\rho_1,\til{\ep}_1,\til{\vU}_+=\til{\vU}_M)$, we can find $\til{\ep}_1>0$ in addition to $\rho_1$ found above, such that 
	\begin{align*}
		\ep_1^\ast(\rho_1,\til{\vU}_+=\til{\vU}_M) &>0 \ec \\
		A(\rho_1,\til{\ep}_1,\til{\vU}_+=\til{\vU}_M) &>0 \ec \\
		B(\rho_1,\til{\ep}_1,\til{\vU}_+=\til{\vU}_M) &>0 \ed
	\end{align*}
	
	Again by continuity of $\til{\vU}_+ \mapsto \ep_1^\ast(\rho_1,\til{\vU}_+)$, $\til{\vU}_+ \mapsto A(\rho_1,\til{\ep}_1,\til{\vU}_+)$ and $\til{\vU}_+ \mapsto B(\rho_1,\til{\ep}_1,\til{\vU}_+)$ we have 
	\begin{align}
		\ep_1^\ast(\rho_1,\til{\vU}_+) &>0 \ec \label{eq:dg0} \\
		A(\rho_1,\til{\ep}_1,\til{\vU}_+) &>0 \ec \label{eq:Ag0}\\
		B(\rho_1,\til{\ep}_1,\til{\vU}_+) &>0 \label{eq:Bg0}
	\end{align} 
	as long as $\til{\vU}_+$ is sufficiently close to $\til{\vU}_M$.
	
	In other words if $\til{\vU}_+$ is sufficiently close to $\til{\vU}_M$, we can find $\rho_1,\til{\ep}_1\in\R^+$ such that $\til{\rho}_-<\rho_1<\til{\rho}_M$ and \eqref{eq:dg0} - \eqref{eq:Bg0} are true. By assumption we have $\til{\rho}_M<\til{\rho}_+$ and hence \eqref{eq:SRcond1} is true. Additionally \eqref{eq:SRcond2} holds because of \eqref{eq:dg0} and finally \eqref{eq:Ag0} and \eqref{eq:Bg0} imply \eqref{eq:SRcond3} and \eqref{eq:SRcond4} respectively. 
\end{proof}

\subsubsection{Proof of Theorem~\ref{thm:isen-SR} via an Auxiliary State}

With Proposition~\ref{prop:isen-SR-smallness} at hand, we are ready to prove Theorem~\ref{thm:isen-SR}.

\begin{proof}[Proof of Theorem~\ref{thm:isen-SR}] 
	Let $\vU_M$ be the intermediate state of the self-similar solution. In other words $\vU_M$ is connected to $\vU_-$ by a 1-shock. Thus we can apply Proposition~\ref{prop:isen-SR-smallness}. For $\rho_a\in (\rho_M,\rho_+)$ define 
	$$
	v_a=v_M + \int_{\rho_M}^{\rho_a} \frac{\sqrt{p'(r)}}{r} \dr \ed
	$$
	Note that $v_a \to v_M$ as $\rho_a\to \rho_M$. Hence we can fix $\rho_a\in (\rho_M,\rho_+)$ such that the state $\vU_a=(\rho_a,v_a)$ is as close to $\vU_M$ as required by Proposition~\ref{prop:isen-SR-smallness}.
	
	Then consider the two Riemann initial value problems for isentropic Euler with initial states 
	\begin{align*}
		\big(\til{\rho}_-,\til{\vu}_-\big) &= \big(\rho_-,\vu_-\big) \ec\\
		\big(\til{\rho}_+,\til{\vu}_+\big) &= \big(\rho_a,(u_-,v_a)^\trans\big) \ec
	\end{align*}
	called problem $\sim$, and 
	\begin{align*}
		\big(\hat{\rho}_-,\hat{\vu}_-\big) &= \big(\rho_a,(u_-,v_a)^\trans\big) \ec \\
		\big(\hat{\rho}_+,\hat{\vu}_-\big) &= \big(\rho_-,\vu_+\big) \ec
	\end{align*}
	which we call problem $ \wedge$. 
	
	Let us first discuss problem $\sim$. Using Propostion \ref{prop:isen-selfsimilar}, it is easy to check that the self-similar solution of problem $\sim$ consists of a 1-shock and a 3-rarefaction: We have $\rho_-<\rho_M$ and $\rho_M<\rho_a$ and hence $\rho_-<\rho_a$. In addition to that it holds that
	\begin{align*}
		v_a - v_- &= v_M - v_- + \int_{\rho_M}^{\rho_a} \frac{\sqrt{p'(r)}}{r} \dr \\
		&= -\sqrt{\frac{\big(\rho_M - \rho_-\big) \big(p(\rho_M) - p(\rho_-)\big)}{\rho_M \rho_-}} + \int_{\rho_M}^{\rho_a} \frac{\sqrt{p'(r)}}{r} \dr \\
		&< \int_{\rho_M}^{\rho_a} \frac{\sqrt{p'(r)}}{r} \dr \\
		&< \int_{\rho_-}^{\rho_a} \frac{\sqrt{p'(r)}}{r} \dr
	\end{align*}
	and 
	\begin{align*}
		v_a - v_- &= -\sqrt{\frac{\big(\rho_M - \rho_-\big) \big(p(\rho_M) - p(\rho_-)\big)}{\rho_M \rho_-}} + \int_{\rho_M}^{\rho_a} \frac{\sqrt{p'(r)}}{r} \dr \\
		&> -\sqrt{\frac{\big(\rho_M - \rho_-\big) \big(p(\rho_M) - p(\rho_-)\big)}{\rho_M \rho_-}} \\
		&> -\sqrt{\frac{\big(\rho_a - \rho_-\big) \big(p(\rho_a) - p(\rho_-)\big)}{\rho_a \rho_-}} \ec
	\end{align*}
	where the last inequality comes from Lemma~\ref{lemma:isen-inequalities} \ref{item:isen-inequ.c}. This shows that the self-similar solution to problem $\sim$ consists of a 1-shock and a 3-rarefaction. 
	
	Hence there exists an admissible fan subsolution to problem $\sim$ according to Propostion~\ref{prop:isen-SR-smallness}. Thus there are infinitely many admissible weak solutions to problem $\sim$. Proposition~\ref{prop:isen-SR-smallness} yields furthermore that $\rho_1<\rho_M$. 
	
	Now consider problem $\wedge$. We are going to prove that the self-similar solution to problem $\wedge$ consists only of a 3-rarefaction using Proposition~\ref{prop:isen-selfsimilar}. By definition of $\vU_a$ we have $\rho_a<\rho_+$ and additionally
	\begin{equation*}
	v_+ - v_a= v_+ - v_M - \int_{\rho_M}^{\rho_a} \frac{\sqrt{p'(r)}}{r} \dr = \int_{\rho_a}^{\rho_+} \frac{\sqrt{p'(r)}}{r} \dr \ed
	\end{equation*}	
	This shows that the self-similar solution of problem $\wedge$ consists of a just a 3-rarefaction wave. 
	
	To conclude we patch together the ``wild'' solutions to problem $\sim$ and the self-similar solution to problem $\wedge$. This patching procedure is possible if $\mu_1<\mu_2$, where $\mu_1$ is the speed of the right interface of the ``wild'' solutions of problem $\sim$ (see Definition~\ref{defn:isen-fanpart}), and $\mu_2=\lambda_3\big(\rho_a,(u_-,v_a)^\trans\big)$ is the left border of the rarefaction wave of the self-similar solution to problem $\wedge$. Here $\lambda_3(\rho,\vu)=v + \sqrt{p'(\rho)}$ denotes the 3rd eigenvalue of the one-dimensional system \eqref{eq:isen-euler-1d-dens} - \eqref{eq:isen-euler-1d-mom2}, see \cite[Equation (2.3)]{ChiKre14}. Hence it remains to show $\mu_1<\mu_2$.
	
	Replacing $\gamma_1$ and $C_1$ in \eqref{eq:isen-rhr3} via \eqref{eq:defn-gamma-temp} and \eqref{eq:defn-C-temp} and solving the resulting equation for $\ep_1$, we obtain 
	\begin{equation} \label{eq:temp43}
		\ep_1=\frac{\mu_1}{\rho_1} (\rho_1 \beta_1 - \rho_a v_a) + \frac{\rho_a}{\rho_1} \big(v_a\big)^2 - \frac{p(\rho_1) - p(\rho_a)}{\rho_1} - \big(\beta_1\big)^2
	\end{equation}
 	Furthermore, solving \eqref{eq:isen-rhr1} for $\beta_1$ yields
	\begin{equation*}
		\beta_1=\frac{1}{\rho_1}\big(\mu_1 (\rho_1-\rho_a)+\rho_a v_a\big) \ed
	\end{equation*}
	We use the latter to eliminate $\beta_1$ in \eqref{eq:temp43} and obtain after some calculation
	\begin{equation*}
		\ep_1 =\frac{\rho_1 - \rho_a}{\rho_1^2} \rho_a (\mu_1 - v_a)^2 - \frac{p(\rho_1) - p(\rho_a)}{\rho_1} \ed
	\end{equation*}
	Since $\ep_1>0$, see \eqref{eq:SRcond2}, it follows that 
	\begin{equation*}
		\frac{\rho_1 - \rho_a}{\rho_1^2} \rho_a (\mu_1 - v_a)^2 - \frac{p(\rho_1) - p(\rho_a)}{\rho_1} > 0 \ed
	\end{equation*}
	Because $\rho_1<\rho_M$ and $\rho_M<\rho_a$, we have $\rho_1 - \rho_a < 0$. Therefore the inequality above is equivalent to 
	\begin{equation*}
		(\mu_1 - v_a)^2 < \frac{\rho_1}{\rho_a}\ \ \frac{p(\rho_1) - p(\rho_a)}{\rho_1 - \rho_a} \ed
	\end{equation*}
	
	Hence
	\begin{equation*}
		\mu_1 < v_a + \sqrt{\frac{\rho_1}{\rho_a}\ \ \frac{p(\rho_1) - p(\rho_a)}{\rho_1 - \rho_a}} < v_a + \sqrt{\frac{p(\rho_1) - p(\rho_a)}{\rho_1 - \rho_a}} \leq v_a + \sqrt{p'(\rho_a)} 
	\end{equation*}
	where the last inequality follows from the convexity of $p$. Since
	\begin{equation*}
		\mu_2 = \lambda_3\big(\rho_a,(u_-,v_a)^\trans\big) = v_a + \sqrt{p'(\rho_a)}
	\end{equation*}
	we found the desired inequality $\mu_1 < \mu_2$. 
\end{proof}

For an illustration of the solutions provided by Theorem~\ref{thm:isen-SR} together with a comparison with the self-similar solution, we refer to \cite[Figure 3]{KliMar18_1}.

\subsection{Sketches of the Non-Uniqueness Proofs for the Other Cases} \label{subsec:riemann-isen-other-cases}

In this subsection we briefly explain how non-uniqueness is proven in the other cases (2, 4, 5, 11, 13, 14, 15 and 17 in Table~\ref{tab:isen-results}). 

\subsubsection{Two Shocks} 

If the self-similar solution consists of two shocks (case 5 in Table~\ref{tab:isen-results}), the non-uniqueness proof has been established by \name{Chiodaroli} and \name{Kreml} \cite{ChiKre14}. The strategy is the same as in the case of one shock and one rarefaction where in the case of two shocks it is not necessary to work with an auxiliary state. In fact for each initial states $(\rho_\pm,\vu_\pm)$ which lead to a self-similar solution consisting of two shocks, one can find numbers $\mu_0,\mu_1\in\R$, $\rho_1\in\R^+$, $\alpha_1,\beta_1,\gamma_1,\delta_1\in\R$ and $C_1\in\R^+$ as required by Proposition~\ref{prop:isen-alg-eq}. In fact these numbers are constructed by slightly perturbing the corresponding values of the self-similar solution, e.g. $\beta_1\approx v_M$. Note that the two-shock-case in the context of the full Euler system, which is explained in detail in Subsection~\ref{subsec:riemann-full-SS}, works similarly.

\subsubsection{One Shock}

The fact that there exist infinitely many admissible weak solutions if the self-similar solution consists of a single shock (cases 2 and 4 in Table~\ref{tab:isen-results}), has been shown by \name{Klingenberg} and the author \cite[Section 6]{KliMar18_1}. Their proof is again built upon the smallness result of Proposition~\ref{prop:isen-SR-smallness} and the use of an auxiliary state.  In contrast to the case of one shock and one rarefaction, where the auxiliary state is connect to $(\rho_+,\vu_+)$ by a rarefaction, here the auxiliary state and $(\rho_+,\vu_+)$ are connected by an admissible shock.

\subsubsection{A Contact Discontinuity and at Least One Shock}

If the self-similar solution contains a contact discontinuity together with at least one shock (cases 11, 13, 14, 15 and 17 in Table~\ref{tab:isen-results}), non-uniqueness of admissible weak solutions has been proven by \name{B{\v r}ezina}, \name{Chiodaroli} and \name{Kreml} \cite{BreChiKre18}. More precisely they only consider two shocks (case 14) and the subcases of 15 and 17 that correspond to the subcases of 6 and 8 which were studied by \name{Chiodaroli} and \name{Kreml} \cite{ChiKre18}. However the other cases can be handled by combining the ideas of \cite{BreChiKre18} and \cite{KliMar18_1}. The approach in \cite{BreChiKre18} is to introduce another wedge such that the fan partition for isentropic Euler, see Definition~\ref{defn:isen-fanpart}, turns into the one we will consider for the full Euler equations, cf. Definition~\ref{defn:full-fanpart}. Inspired by the self-similar solution, they use the ansatz to set $\rho_1=\rho_2$, $\alpha_1=u_-$, $\alpha_2=u_+$ and $\beta_1=\beta_2$. This ansatz allows them to show the following
\begin{itemize}
	\item If the self-similar solution contains two shocks together with the contact discontinuity, then there exists an admissible fan subsolution, which is now constant in four regions $\Gamma_-,\Gamma_1,\Gamma_2,\Gamma_+$.
	\item In the case of one shock and one rarefaction together with the contact discontinuity, a smallness result which corresponds to Proposition~\ref{prop:isen-SR-smallness} or the one achieved in \cite{ChiKre18}, holds.
\end{itemize}
To conclude with the remaining cases, i.e. one shock (11 and 13 in Table~\ref{tab:isen-results}) or one shock and one ``large'' rarefaction  (cases 15 and 17), one can use the patching procedure as in the proof of Theorem~\ref{thm:isen-SR}.

\subsection{Other Results in the Context of the Riemann Problem} \label{subsec:riemann-isen-other-results}

In the context of the initial value problem for the isentropic Euler equations \eqref{eq:baro-euler-pv-dens}, \eqref{eq:baro-euler-pv-mom}, \eqref{eq:isentropic-EOS} with Riemann initial data \eqref{eq:riemann-init-isen}, several other results are worth mentioning. 

\begin{itemize}
	\item The first example of initial data of the form \eqref{eq:riemann-init-isen} that allow for infinitly many admissible weak solutions, which was achieved by \name{Chiodaroli}, \name{De~Lellis} and \name{Kreml} \cite{ChiDelKre15}, is also interesting from another point of view. It can be shown that this data is obtained by a compression wave. More precisely, there exist Lipschitz continuous initial data such that the one-dimensional solution of the corresponding initial value problem coincides at a positive time with this Riemann data. This proves that there exist Lipschitz continuous initial data which lead to infinitely many solutions. These solutions coincide on some non-empty time interval, but differ thereafter.

	\item The result we just mentioned, has been further improved by \name{Chiodaroli et al.} \cite{CKMS19} who show that there are even $C^1$ initial data which lead to infinitely many admissible weak solutions. As above these solutions coincide for small positive times and differ afterwards. 
	
	\item In the literature another condition to select physically relevant solutions has been discussed, namely the \emph{entropy rate admissibility criterion}, or -- in the case of isentropic Euler -- the \emph{energy rate admissibility criterion}, proposed by \name{Dafermos} \cite{Dafermos73}. Roughly speaking in the context of the isentropic Euler system, this criterion says that the weak solution which dissipates the largest amount of energy should be selected as the physically relevant solution, see \cite[Definition 1]{ChiKre14} for a detailed definition. Apart from the fact that no one knows if such a maximum exists, \name{Chiodaroli} and \name{Kreml} \cite{ChiKre14} show that there exist Riemann initial data of the form \eqref{eq:riemann-init-isen} such that the self-similar solution is not entropy rate admissible as there exist ``wild'' solutions that dissipate more energy than the self-similar solution. 
	
	\item Concerning the entropy rate admissibility criterion there is another interesting result by \name{Feireisl} \cite{Feireisl14}, which is not restricted to Riemann initial data. \name{Feireisl} applies convex integration to obtain infinitely many admissible weak solutions to the initial boundary value problem for barotropic Euler with periodic boundary conditions. Furthermore he shows that none of these infinitely many solutions is entropy rate admissible as for each ``wild'' solution there is another one which dissipates more energy.
	
	\item As mentioned in Subsection~\ref{subsec:riemann-isen-results}, the question on uniqueness is open if the self-similar solution consists only of a contact discontinuity (case 10 in Table~\ref{tab:isen-results}). However there exists a result on non-uniqueness for barotropic Euler system with \emph{Chaplygin pressure law} $p(\rho)=-\frac{1}{\rho}$. Initial data of the form \eqref{eq:riemann-init-isen} have been studied together with this particular pressure law by \name{B{\v r}ezina}, \name{Kreml} and \name{M{\'a}cha} \cite{BreKreMac18}. For some of those initial data the self-similar solution consists only of two or three contact discontinuities. In these cases there exist infinitely many other admissible weak solutions as shown in  \cite{BreKreMac18}.
	
	\item Finally as shown by \name{Klingenberg} and the author \cite{KliMar18_2}, for a particular example of Riemann initial data there exist infinitely many admissible weak solutions all of which fulfill the energy inequality \eqref{eq:baro-euler-weak-whole-admissibility} as an equation. In other words there exist initial data of the form \eqref{eq:riemann-init-isen} such that the corresponding initial value problem admits infinitely many \emph{energy-conservative} solutions. Therefore these solutions can be interpreted as solutions of the \emph{full} Euler system. This result builds the bridge to the next section where the initial value problem with Riemann data for the \emph{full} Euler equations is considered.
\end{itemize}

\section[Riemann Problem for the Full Euler System]{Riemann Problem for the Full Euler System%
	\sectionmark{Full Euler}} \label{sec:riemann-full}
\sectionmark{Full Euler}

\subsection{One-Dimensional Self-Similar Solution} \label{subsec:riemann-full-1D}

As in the isentropic case, one considers the one-dimensional Riemann problem that corresponds to the initial value problem for the full Euler system \eqref{eq:full-euler-pv-dens} - \eqref{eq:full-euler-pv-en} with initial data \eqref{eq:riemann-init-full}. This one-dimensional Riemann problem reads
\begin{align}
	\partial_t \rho + \partial_2 (\rho v) &= 0 \ec \label{eq:full-euler-1D-dens} \\
	\partial_t (\rho u) + \partial_2 (\rho uv) &=0 \ec \label{eq:full-euler-1D-mom1} \\
	\partial_t (\rho v) + \partial_2 (\rho v^2 + p )  &= 0 \ec \label{eq:full-euler-1D-mom2} \\
	\partial_t \bigg(\half\rho(u^2 + v^2) + \rho e(\rho,p)\bigg) + \partial_2 \bigg[\bigg(\half\rho (u^2 + v^2) + \rho e(\rho,p) + p\bigg)v\bigg] &= 0 \label{eq:full-euler-1D-en}
\end{align} 
with initial data \eqref{eq:riemann-init-full}. This system is solved with classical methods to obtain the \emph{self-similar solution}. Similarly as in Proposition~\ref{prop:isen-selfsimilar} we summarize some facts on the self-similar solution, where we do not exhibit as much details as in Proposition~\ref{prop:isen-selfsimilar}. Instead we refer to the literatur cited below. 

\begin{prop} \label{prop:full-selfsimilar}
	Let $\rho_\pm\in \R^+$, $\vu_\pm\in \R^2$ and $p_\pm\in \R^+$. The self-similar solution to the problem \eqref{eq:full-euler-pv-dens} - \eqref{eq:full-euler-pv-en}, \eqref{eq:riemann-init-full} is constant in four regions which are separated by three waves. The leftmost and rightmost states are given by the initial states $(\rho_-,\vu_-,p_-)$ and $(\rho_+,\vu_+,p_+)$ respectively. The left intermediate state is equal to $\big(\rho_{M-},(u_-,v_M)^\trans,p_M\big)$, whereas the right intermediate state equals $\big(\rho_{M+},(u_+,v_M)^\trans,p_M\big)$, see Figure~\ref{fig:full-SCR} for an example. In particular, the pressure and the velocity component, that is perpendicular to the initial discontinuity, coincide in both intermediate regions, whereas the density and the velocity component, that is parallel to the initial discontinuity, jump at the 2-wave. Each of the 1- and the 3-wave is either a shock or a rarefaction, whereas the 2-wave is a contact discontinuity. 
\end{prop}

\begin{proof}
	For the proof we refer to the textbooks by \name{Smoller} \cite[Chapter 18 \S B]{Smoller} or \name{Dafermos}~\cite[Chapters 7 - 9]{Dafermos}. 
\end{proof}

\begin{figure}[htb] 
	\centering
	\includegraphics[width=0.7\textwidth]{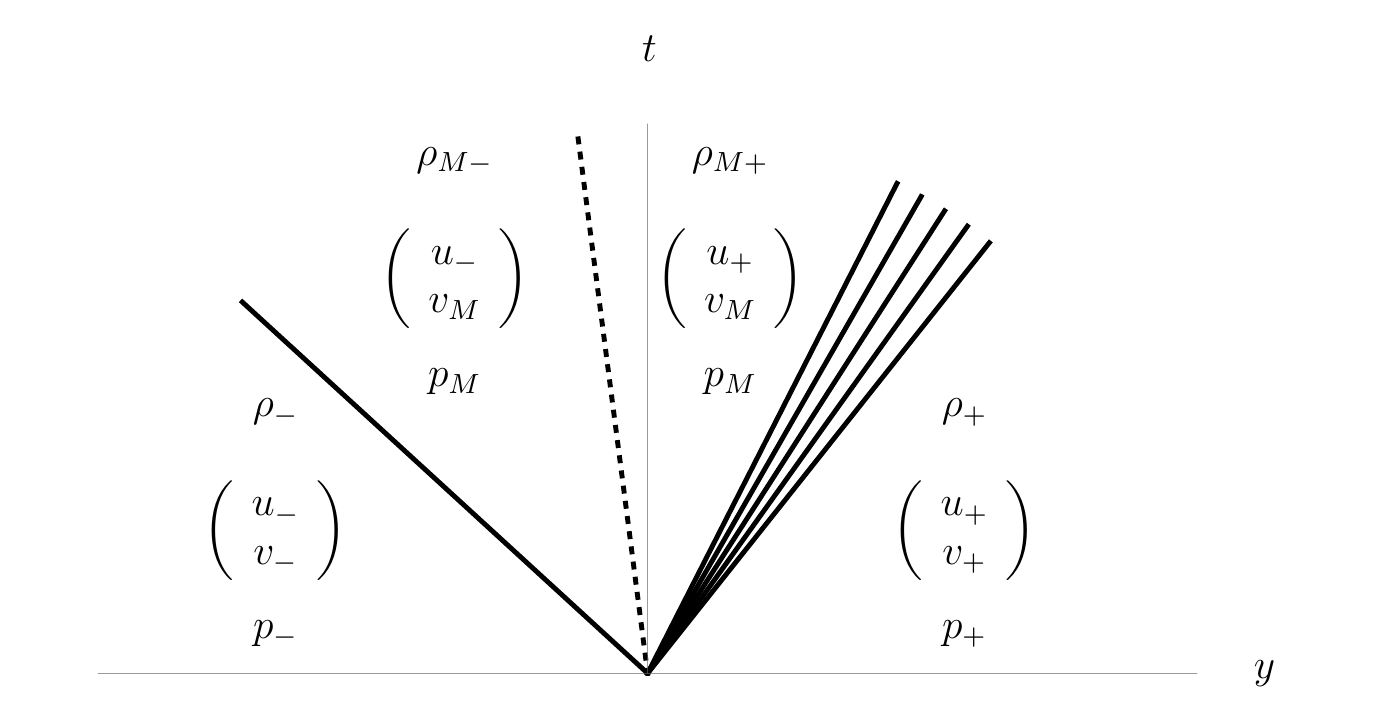} 
	\caption{An example of the self-similar solution for full Euler where the 1-wave is a shock, the 2-wave a contact discontinuity and the 3-wave a rarefaction.} 
	\label{fig:full-SCR}
\end{figure} 

\begin{rem}
	Note that \eqref{eq:full-euler-1D-dens} - \eqref{eq:full-euler-1D-en} is a system of four equations and therefore has four eigenvalues, i.e. there should be four wave families. However the eigenvalue that corresponds to the 2-contact discontinuity, has multiplicity two. Indeed there are two issues that lead to a 2-contact discontinuity: First, even the Riemann problem for the one-dimensional full Euler system \eqref{eq:full-euler-1D-dens}, \eqref{eq:full-euler-1D-mom2}, \eqref{eq:full-euler-1D-en}\footnote{To be precise, one has to erase ``$u^2$'' in \eqref{eq:full-euler-1D-en} to obtain the one-dimensional Euler system.} exhibits a contact discontinuity where the density jumps from $\rho_{M-}$ to $\rho_{M+}$. Second, the transport equation for $u$ \eqref{eq:full-euler-1D-mom1} yields a contact discontinuity where $u$ jumps from $u_-$ to $u_+$. To be more precise, one could also speak of a 2-contact discontinuity and a 3-contact discontinuity which lie on each other, and call what we called 3-wave, a 4-wave. 
\end{rem}

As in Proposition~\ref{prop:isen-selfsimilar}, there are necessary and sufficient conditions on the initial states for each possibility of the shape of the self-similar solution. Since they are quite lengthy we decided not to present them in this book with one exception, see below. Instead we refer to \name{Smoller} \cite[Corollary 18.7]{Smoller}. 

In Subsection~\ref{subsec:riemann-full-SS} we consider the case where the self-similar solution consists of a 1-shock, a possible 2-contact discontinuity and a 3-shock, where we additionally assume that $u_-=u_+$. Hence in this case the 2-contact discontinuity is only due to one of the two mechanisms that lead to a contact discontinuity, see the remark above. Let us quote the condition on the initial states that lead to a self-similar solution containing two shocks. 

\begin{prop} \label{prop:full-selfsimilar-SS}
	Let $\rho_\pm\in \R^+$, $\vu_\pm\in \R^2$ and $p_\pm\in \R^+$ such that $u_-=u_+$. Then the self-similar solution to the problem \eqref{eq:full-euler-pv-dens} - \eqref{eq:full-euler-pv-en}, \eqref{eq:riemann-init-full} consists of a 1-shock, a possible\footnote{The 2-contact discontinuity is not apparent when $\rho_{M-}=\rho_{M+}$.} 2-contact discontinuity and a 3-shock if and only if one of the conditions
	\begin{align*}
		\bullet \ \ &p_-\leq p_+ \quad\text{ and }\quad v_+ - v_- < - \sqrt{2} \frac{p_+ - p_-}{\sqrt{\rho_-\big( (\gamma-1)p_- + (\gamma+1) p_+\big)}} \ec \qquad\qquad\qquad \\
		\bullet \ \ &p_+\leq p_- \quad\text{ and }\quad v_+ - v_- < - \sqrt{2} \frac{p_- - p_+}{\sqrt{\rho_+\big((\gamma-1)p_+ + (\gamma+1) p_-\big)}}
	\end{align*}
	holds. Furthermore $\rho_{M\pm},v_M,p_M$ are uniquely determined by 
	\begin{align*}
		v_+ - v_- &= -\sqrt{2}\left(\frac{p_M - p_-}{\sqrt{\rho_- \big((\gamma-1)p_- + (\gamma+1) p_M\big)}} + \frac{p_M - p_+}{\sqrt{\rho_+ \big((\gamma-1)p_+ + (\gamma-1) p_M\big)}}\right) \ec \\
		v_M &= v_- - \sqrt{2}\frac{p_M - p_-}{\sqrt{\rho_- \big((\gamma-1)p_- + (\gamma+1) p_M\big)}} \ec \\
		\rho_{M-} &= \rho_- \frac{(\gamma-1) p_- + (\gamma+1) p_M}{(\gamma-1) p_M + (\gamma+1) p_-} \ec \\
		\rho_{M+} &= \rho_+ \frac{(\gamma-1) p_+ + (\gamma+1) p_M}{(\gamma-1) p_M + (\gamma+1) p_+} \ec
	\end{align*}
	and the following properties hold: 
	\begin{itemize}
		\item $p_M>\max\{p_-,p_+\}$, $\rho_{M-}>\rho_-$ and $\rho_{M+}>\rho_+$;
		\item The shock speeds, denoted by $\sigma_-,\sigma_+$, satisfy the Rankine-Hugoniot conditions
		 \begin{align}
			\sigma_\pm(\rho_\pm - \rho_{M\pm}) &= \rho_\pm v_{\pm} - \rho_{M\pm}v_M \es \label{eq:full-selfsimilar-SS-dens}\\ 
			\sigma_\pm(\rho_\pm v_{\pm} - \rho_{M\pm}v_M) &= \rho_\pm \big(v_{\pm}\big)^2 - \rho_{M\pm} \big(v_M\big)^2 + p_\pm - p_M \es \label{eq:full-selfsimilar-SS-mom} 
		\end{align}
		\begin{align}
			&\sigma_\pm \left(\rho_\pm \frac{\big(v_\pm\big)^2}{2} + \rho_\pm e(\rho_\pm,p_\pm) - \rho_{M\pm} \frac{\big(v_M\big)^2}{2} - \rho_{M\pm} e(\rho_{M\pm},p_M) \right) \notag\\
			&= \left(\rho_\pm \frac{\big(v_\pm\big)^2}{2} + \rho_\pm e(\rho_\pm,p_\pm) + p_\pm \right) v_\pm - \left(\rho_{M\pm} \frac{\big(v_M\big)^2}{2} + \rho_{M\pm} e(\rho_{M\pm},p_M) + p_M \right) v_M \es \label{eq:full-selfsimilar-SS-en}
		\end{align}
		\item The admissibility conditions at the shocks are satisfied as strict inequalities, i.e. 
		\begin{align}
			\sigma_- \big(\rho_{M-} s(\rho_{M-},p_M) - \rho_- s(\rho_-,p_-)\big) &< \rho_{M-} s(\rho_{M-},p_M) v_M - \rho_- s(\rho_-,p_-) v_- \es \label{eq:full-selfsimilar-SS-adm-} \\
			\sigma_+ \big(\rho_+ s(\rho_+,p_+) - \rho_{M+} s(\rho_{M+},p_M)\big) &< \rho_+ s(\rho_+,p_+) v_+ - \rho_{M+} s(\rho_{M+},p_M) v_M \es \label{eq:full-selfsimilar-SS-adm+}
		\end{align}
		\item The speed of the contact discontinuity is given by $v_{M}$.
	\end{itemize}
\end{prop}

\begin{proof}
	Again we refer to the literature, e.g. \name{Smoller} \cite[Chapter 18]{Smoller}. 
\end{proof}

\subsection{Summary of the Results on Non-/Uniqueness} \label{subsec:riemann-full-results}

\begin{thm} \label{thm:full-summary-riemann}
	It depends on the shape of the self-similar solution whether or not there is a unique admissible weak solution of the initial value problem for the full Euler system \eqref{eq:full-euler-pv-dens} - \eqref{eq:full-euler-pv-en} with initial data \eqref{eq:riemann-init-full}. Table~\ref{tab:full-results} summarizes the results. In the cases where the solution is not unique, there are even infinitely many admissible weak solutions.
\end{thm}

\renewcommand{\arraystretch}{1.2}
\begin{table}[h]
	\centering 
	\begin{tabular}{|c|c|c|c|c|c|} \cline{2-6} 
		\multicolumn{1}{c|}{}& \multicolumn{3}{c|}{Structure of the self-similar solution} & \centering Solution & \centering Reference \tabularnewline 
		\multicolumn{1}{c|}{}& \centering 1-wave & \centering 2-wave & \centering 3-wave & \centering unique? & \tabularnewline \cline{1-6}
		1 & \centering - & \centering - & \centering - & \textcolor{green}{\textcolor{green}{yes}} & e.g. \cite{CheChe07} or \cite{FeiKreVas15} \tabularnewline \cline{1-6}
		2 & \centering - & \centering - & \centering shock & \textcolor{red}{\textcolor{red}{no}} & \cite{KKMM20} \tabularnewline \cline{1-6}
		3 & \centering - & \centering - & \centering rarefaction & \textcolor{green}{yes} & \cite{CheChe07}, \cite{FeiKreVas15} \tabularnewline \cline{1-6}
		4 & \centering shock & \centering - & \centering - & \textcolor{red}{no} & \cite{KKMM20} \tabularnewline \cline{1-6}
		5 & \centering shock & \centering - & \centering shock & \textcolor{red}{no} & \cite{AKKMM20} \tabularnewline \cline{1-6}
		6 & \centering shock & \centering - & \centering rarefaction & \textcolor{red}{no} & \cite{KKMM20} \tabularnewline \cline{1-6}
		7 & \centering rarefaction & \centering - & \centering - & \textcolor{green}{yes} & \cite{CheChe07}, \cite{FeiKreVas15} \tabularnewline \cline{1-6}
		8 & \centering rarefaction & \centering - & \centering shock & \textcolor{red}{no} & \cite{KKMM20} \tabularnewline \cline{1-6}
		9 &\centering rarefaction  & \centering - & \centering rarefaction & \textcolor{green}{yes} & \cite{CheChe07}, \cite{FeiKreVas15} \tabularnewline \cline{1-6}
		10 & \centering - & \centering contact & \centering - & open & \tabularnewline \cline{1-6}
		11 & \centering - & \centering contact & \centering shock & \textcolor{red}{no} & \cite{KKMM20} \tabularnewline \cline{1-6}
		12 & \centering - & \centering contact & \centering rarefaction & open &  \tabularnewline \cline{1-6}
		13 & \centering shock & \centering contact & \centering - & \textcolor{red}{no} & \cite{KKMM20} \tabularnewline \cline{1-6}
		14 & \centering shock & \centering contact & \centering shock & \textcolor{red}{no}\footnotemark & \cite{AKKMM20} \tabularnewline \cline{1-6}
		15 & \centering shock & \centering contact & \centering rarefaction & \textcolor{red}{no} & \cite{KKMM20} \tabularnewline \cline{1-6}
		16 & \centering rarefaction & \centering contact & \centering - & open &  \tabularnewline \cline{1-6}
		17 & \centering rarefaction & \centering contact & \centering shock & \textcolor{red}{no} & \cite{KKMM20} \tabularnewline \cline{1-6}
		18 &\centering rarefaction & \centering contact & \centering rarefaction & open & \tabularnewline \cline{1-6}
	\end{tabular}
	\caption{All the 18 possibilities of the structure of the self-similar solution to the initial value problem for the full Euler system \eqref{eq:full-euler-pv-dens} - \eqref{eq:full-euler-pv-en} with initial data \eqref{eq:riemann-init-full}. Furthermore, if known, the answer to the question on uniqueness.} \label{tab:full-results} 
\end{table}
\renewcommand{\arraystretch}{1}
\footnotetext{The paper by \name{Al Baba et al.} \cite{AKKMM20} only considers $u_-=u_+$. Keeping in mind that there are two mechanisms which lead to the contact discontinuity, one of which is $u_-\neq u_+$, this means that \cite{AKKMM20} does not cover the whole class of initial states $(\rho_\pm,\vu_\pm)$ for which the self-similar solution consists of two shocks and a contact discontinuity.} 

As shown by \name{Chen} and \name{Chen} \cite{CheChe07}, and also by \name{Feireisl}, \name{Kreml} and \name{Vasseur} \cite{FeiKreVas15}, the self-similar solution is unique if it contains only rarefaction waves (cases 1, 3, 7 and 9 in Table~\ref{tab:full-results}).
As in the isentropic case, these uniqueness results are achieved with the help of the relative energy inequality, which is also used to prove weak-strong-uniqueness. Again \name{Chen} and \name{Chen} \cite{CheChe07} allow for a vacuum in the intermediate state.

As mentioned in Subsection~\ref{subsec:riemann-isen-other-results}, the paper \cite{KliMar18_2} by \name{Klingenberg} and the author can be viewed as a first non-uniqueness result for the initial value problem for the full Euler system \eqref{eq:full-euler-pv-dens} - \eqref{eq:full-euler-pv-en} with initial data \eqref{eq:riemann-init-full}. In view of that, one began to work on a classification of the initial data of the form \eqref{eq:riemann-init-full} regarding non-/uniqueness of solutions as in the isentropic case. In \cite{AKKMM20} \name{Al Baba}, the author and others show that there are infinitely many admissible weak solutions if the self-similar solutions contain two shocks (cases 5 and 14 in Table~\ref{tab:full-results}). This non-uniqueness proof works similarly as in the isentropic case. Note that in \cite{AKKMM20} only the case $u_-=u_+$ is considered. Furthermore \name{Al Baba et al.} propose two approaches: One of them works for all initial data that belong to the cases 5 and 14 in Table~\ref{tab:full-results} but requires the additional assumption that $\gamma<3$. The other approach does not need this additional assumption but only works for subcases of 5 and 14. 

\begin{rem}
	Note that the restriction $\gamma<3$ covers the physically relevant cases. The theory of thermodynamics requires $\gamma = \frac{f+2}{f}$ where $f$ is the number of degrees of freedom of the gas under consideration. From $f>1$, which is justified as a gas in two space dimensions should have more than one degree of freedom, we deduce $\gamma<3$. 
\end{rem}

If the self-similar solution contains only one shock (i.e. cases 2, 4, 6, 8, 11, 13, 15 and 17 in Table~\ref{tab:full-results}), then there are infinitely many solutions as well, which was shown by \name{Klingenberg}, the author and others \cite{KKMM20}. Here again one assumes that $\gamma<3$. In contrast to the result for two shocks mentioned above, \name{Klingenberg et al.} do not assume that $u_-=u_+$ in \cite{KKMM20}.

Similarly to the isentropic case, if the self-similar solution consists of a contact discontinuity and possible rarefaction waves (cases 10, 12, 16 and 18 in Table~\ref{tab:full-results}), it remains an open question whether or not the self-similar solution is the unique admissible weak solution in the sense of Definition~\ref{defn:aws-full-whole}. For a recent uniqueness result regarding case 10 in the context of a different notion of a solution, we refer to Subsection~\ref{subsec:riemann-full-other-results}.

\subsection{Non-Uniqueness Proof if the Self-Similar Solution Contains Two Shocks} \label{subsec:riemann-full-SS}

In this subsection we show how the non-uniqueness proof for the full Euler system works. To this end we take the example where the self-similar solution contains two shocks (cases 5 and 14 in Table~\ref{tab:full-results}). We additionally assume that $u_-=u_+$ and $\gamma<3$. We present one of the two approaches performed by \name{Al Baba et al.} \cite{AKKMM20}. Parts of this subsection have been copied verbatim from \cite{AKKMM20}.

The strategy of the non-uniqueness proof is similar to the isentropic case, where the notion of a fan subsolution has to by adapted. In particular another wedge has to be considered, see Definition~\ref{defn:full-fanpart}.

Since the full Euler equations \eqref{eq:full-euler-pv-dens} - \eqref{eq:full-euler-pv-en} are invariant under Galilean transformations we may assume without loss of generality that $u_-=u_+=0$ as well as $v_M=0$. Indeed if this was not the case, then we subtract the constant velocity $(u_-,v_M)^\trans$ to achieve what we want. For more details we refer to \cite[Section 5.1]{AKKMM20} or \cite[Proposition 2.8]{KKMM20}.

In short, the aim of the current subsection is to prove the following theorem.

\begin{thm} \label{thm:full-SS} 
	Let $\gamma < 3$ and $\rho_\pm\in\R^+$, $\vu_\pm\in\R^2$, $p_\pm\in\R^+$ with $u_-=u_+=0$ be such that the self-similar solution consists either 
	\begin{itemize}
		\item of a 1-shock and a 3-shock or 
		\item of a 1-shock, a 2-contact discontinuity and a 3-shock.
	\end{itemize}
	Assume furthermore that $v_M=0$. Then there exist infinitely many admissible weak solutions to the initial value problem \eqref{eq:full-euler-pv-dens} - \eqref{eq:full-euler-pv-en}, \eqref{eq:riemann-init-full}.
\end{thm}

Recall that Proposition~\ref{prop:full-selfsimilar-SS} contains conditions for the self-similar solution to be of the form required by Theorem~\ref{thm:full-SS}.

\subsubsection{Condition for Non-Uniqueness}

We proceed as in the isentropic case. However we must adjust our strategy and work with another wedge in the part partition. This leads to the following adapted notion of a fan partition. 

\begin{defn} \label{defn:full-fanpart} 
	Let $\mu_0<\mu_1<\mu_2$ be real numbers. A \emph{fan partition} of the space-time domain $(0,\infty)\times\R^2$ consists of four open sets $\Gamma_-,\Gamma_1,\Gamma_2,\Gamma_+$ of the form
	\begin{align*}
		\Gamma_-&=\{(t,\vx):t>0\text{ and }y<\mu_0 t\} \ec \\
		\Gamma_1\:&=\{(t,\vx):t>0\text{ and }\mu_0 t<y<\mu_1 t\} \ec \\
		\Gamma_2\:&=\{(t,\vx):t>0\text{ and }\mu_1 t<y<\mu_2 t\} \ec \\
		\Gamma_+&=\{(t,\vx):t>0\text{ and }y>\mu_2 t\} \ec
	\end{align*}
	see Figure~\ref{fig:full-fanpart}.
\end{defn}

\begin{figure}[htb] 
	\centering
	\includegraphics[width=0.7\textwidth]{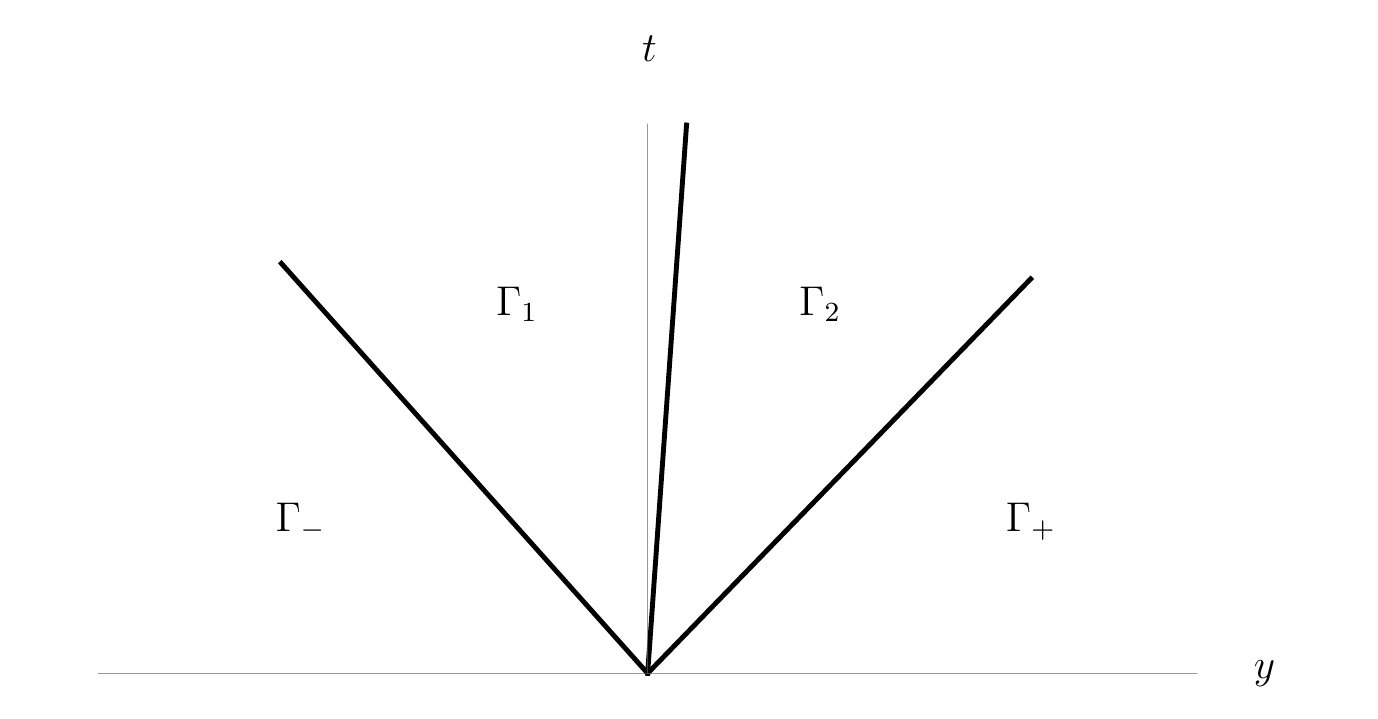} 
	\caption{A fan partition, see Definition~\ref{defn:full-fanpart}.} 
	\label{fig:full-fanpart}
\end{figure} 

A fan subsolution is defined analogously to the isentropic case, cf. Definition~\ref{defn:isen-fansubs}, where now the underlying fan partition is of the type given by Definition~\ref{defn:full-fanpart}.

\begin{defn} \label{defn:full-fansubs} 
	An \emph{admissible fan subsolution} to the initial value problem for the full Euler system \eqref{eq:full-euler-pv-dens} - \eqref{eq:full-euler-pv-en} with initial data \eqref{eq:riemann-init-full} is a quintuple
	$$
	(\ov{\rho},\ov{\vm},\ov{\mU},\ov{c},\ov{p})\ \in \ L^\infty\big((0,\infty)\times\R^2;\R^+\times\R^2\times\sztwo\times\R^+\times\R^+\big)
	$$ 
	of piecewise constant functions, which satisfies the following properties:
	\begin{enumerate}
		\item There exists a fan partition  $\Gamma_-,\Gamma_1,\Gamma_2,\Gamma_+$ of $(0,\infty)\times\R^2$ and for $i\in\{1,2\}$ there exist constants $\rho_i\in\R^+$, $\vm_i\in\R^2$, $\mU_i\in\sztwo$, $c_i\in\R^+$ and $p_i\in\R^+$, such that
		\begin{align*}
			(\ov{\rho},\ov{\vm},\ov{\mU},\ov{c},\ov{p})&= \left\{ 
			\begin{array}[c]{ll}
				\big(\rho_-,\vm_-,\mU_-, c_- , p_-\big) & \text{ if } (t,\vx)\in \Gamma_- \ec\\
				\big(\rho_1\: ,\vm_1 \:,\mU_1 \:,c_1\: ,p_1 \:\big) & \text{ if } (t,\vx)\in \Gamma_1 \:\ec\\
				\big(\rho_2\: ,\vm_2 \:,\mU_2 \:,c_2\: ,p_2 \:\big) & \text{ if } (t,\vx)\in \Gamma_2 \:\ec\\
				\big(\rho_+,\vm_+,\mU_+, c_+,p_+\big) & \text{ if } (t,\vx)\in \Gamma_+ \ec 
			\end{array}
			\right.
		\end{align*}
		where 
		\begin{align*}
			\vm_\pm &:= \rho_\pm \vu_\pm \ec \\
			\mU_\pm &:= \rho_\pm \vu_\pm \otimes \vu_\pm - \frac{\rho_\pm |\vu_\pm|^2}{2}\id \ec\\
			c_\pm &:=  \frac{\rho_\pm |\vu_\pm|^2}{2} + p_\pm 
		\end{align*}
		with the given initial states $(\rho_\pm,\vu_\pm,p_\pm) \in \R^+\times \R^2\times \R^+$;
		
		\item For $i\in\{1,2\}$ the following inequality holds:
		\begin{equation} \label{eq:full-fansubs-subs}
			\lambda_{\max}\left(\frac{\vm_i\otimes\vm_i}{\rho_i} + p_i \id - \mU_i\right) < c_i \es
		\end{equation}
		\item For all test functions $(\phi,\vphi,\psi)\in \Cc\big([0,\infty)\times\R^2;\R\times\R^2\times\R\big)$ the following identities hold:
		\begin{align} 
			\int_0^\infty \int_{\R^2} \Big[\ov{\rho} \partial_t\phi + \ov{\vm}\cdot\Grad\phi\Big] \dx\dt + \int_{\R^2} \rho_\init \phi(0,\cdot)\dx &= 0 \es \label{eq:full-fansubs-dens}\\
			\int_0^\infty \int_{\R^2}\Big[\ov{\vm}\cdot\partial_t\vphi + \ov{\mU}:\Grad\vphi + \ov{c}\Div\vphi\Big]\dx\dt  + \int_{\R^2} \rho_\init \vu_\init\cdot\vphi(0,\cdot)\dx &= 0 \es \label{eq:full-fansubs-mom}\\
			\int_0^\infty \int_{\R^2} \left[\Big(\ov{c} -\ov{p} + \ov{\rho}e(\ov{\rho},\ov{p}) \Big) \partial_t \psi + \Big(\ov{c} + \ov{\rho}e(\ov{\rho},\ov{p})\Big)\frac{\ov{\vm}}{\ov{\rho}}\cdot\Grad \psi \right]\dx\dt \qquad & \notag\\ 
			+ \int_{\R^2} \bigg(\half\rho_\init|\vu_\init|^2 + \rho_\init e(\rho_\init,p_\init)\bigg)\psi(0,\cdot) \dx &= 0 \es \label{eq:full-fansubs-en}
		\end{align}
		\item For every non-negative test function $\varphi\in \Cc\big([0,\infty)\times\R^2;\R_0^+\big)$ and all $Z\in C^\infty(\R)$ with $Z'\geq 0$, the inequality
		\begin{align}
			\int_0^\infty \int_{\R^2} \Big[\ov{\rho} Z\big(s(\ov{\rho},\ov{p})\big) \partial_t \varphi + \ov{\vm} Z\big(s(\ov{\rho},\ov{p})\big) \cdot\Grad \varphi\Big]\dx\dt \qquad & \notag\\  
			+ \int_{\R^2} \rho_\init Z\big(s(\rho_\init,p_\init)\big)\varphi(0,\cdot) \dx &\leq 0 \label{eq:full-fansubs-adm}
		\end{align}
		is fulfilled.
	\end{enumerate}
\end{defn}

Again the existence of an admissible fan subsolution implies existence of infinitely many admissible weak solutions, which is proven using Theorem~\ref{thm:convint-nodens}.

\begin{thm} \label{thm:full-condition} 
	Let $(\rho_\pm,\vu_\pm,p_\pm)$ be such that there exists an admissible fan subsolution $(\ov{\rho},\ov{\vm},\ov{\mU},\ov{c},\ov{p})$ to the initial value problem \eqref{eq:full-euler-pv-dens} - \eqref{eq:full-euler-pv-en}, \eqref{eq:riemann-init-full}. Then this initial value problem admits infinitely many admissible weak solutions $(\rho,\vu,p)$ with the following properties:
	\begin{enumerate}
		\item \label{item:full-riemann.a} $(\rho,p)=(\ov{\rho},\ov{p})$,
		\item \label{item:full-riemann.b} $\vu(t,\vx)= \vu_-$ for all $(t,\vx)\in\Gamma_-$ and $\vu(t,\vx)= \vu_+$ for all $(t,\vx)\in\Gamma_+$, 
		\item \label{item:full-riemann.c} $|\vu(t,\vx)|^2= \frac{\rho_i}{2} (c_i - p_i)$ for\footnote{As in Theorem~\ref{thm:isen-condition}, \eqref{eq:full-fansubs-subs} guarantees that $c_i-p_i>0$.} a.e. $(t,\vx) \in \Gamma_i$, $i=1,2$. 
	\end{enumerate}
\end{thm}

\begin{proof} 
	As in the isentropic case, we apply Theorem~\ref{thm:convint}, more precisely Theorem~\ref{thm:convint-nodens}, in order to prove Theorem~\ref{thm:full-condition}. In contrast to the isentropic case, the fan partition possesses now two middle wedges $\Gamma_1$ and $\Gamma_2$, see Definition~\ref{defn:full-fanpart}. We apply Theorem~\ref{thm:convint-nodens} in each of the two wedges $\Gamma_1$ and $\Gamma_2$, more precisely we set $\Gamma:= \Gamma_i$, $r:=\half \rho_i$, $c:=c_i$ and 
	$$
		(\rho_0,\vm_0,\mU_0)(t,\vx) := (\rho_i,\vm_i,\mU_i) \qquad \text{ for all } (t,\vx)\in \closure{\Gamma} \ec
	$$
	for $i=1,2$. Note that for each $i=1,2$ we must furthermore determine functions $\rho\mapsto p_i(\rho)$. We choose 
	\begin{equation} \label{eq:f2-temp-riemann}
		p_i(\rho) := \frac{p_i}{\rho_i^\gamma}\rho^\gamma \ed
	\end{equation}	
	Hence both ``pressure laws'' satisfy $p_i(\rho_i)=p_i$ for $i=1,2$, and are of the form \eqref{eq:isentropic-EOS}, i.e. the corresponding ``pressure potentials'' read
	$$
	P_i(\rho) = \frac{p_i(\rho)}{\gamma -1}
	$$
	according to \eqref{eq:isentropic-P}. As in the isentropic case, see the proof of Theorem~\ref{thm:isen-condition}, one simply checks that the assumptions of Theorem~\ref{thm:convint-nodens}, which are the same as the assumptions of Theorem~\ref{thm:convint}, hold. Note that with the choice \eqref{eq:f2-temp-riemann}, the left-hand side of \eqref{eq:full-fansubs-subs} turns into $e(\rho_i,\vm_i,\mU_i)$ and hence \eqref{eq:p0-subs} is satisfied. Theorem~\ref{thm:convint-nodens} yields infinitely many $(\til{\rho}_1,\til{\vm}_1)\in L^\infty(\Gamma_1;\R^+\times \R^2)$ and infinitely many $(\til{\rho}_2,\til{\vm}_2)\in L^\infty(\Gamma_2;\R^+\times \R^2)$ with the properties \ref{item:convint.a} - \ref{item:convint.c} stated in Theorem~\ref{thm:convint}, and additionally $\til{\rho}_1\equiv \rho_1$ and $\til{\rho}_2\equiv \rho_2$. For any two pairs $(\til{\rho}_1,\til{\vm}_1)$ and $(\til{\rho}_2,\til{\vm}_2)$ we define $(\rho,\vu,p)\in L^\infty \big((0,\infty)\times \R^2;\R^+\times \R^2\times \R^+\big)$ by 
	\begin{equation} \label{eq:f1-temp-riemann}
		(\rho,\vu,p)= \left\{ 
		\begin{array}{ll}
			\big(\rho_-,\vu_-,p_-\big) & \text{ if } (t,\vx)\in \Gamma_- \ec\\
			\big(\til{\rho}_1,\til{\vm}_1 / \til{\rho}_1 ,p_1\big) & \text{ if } (t,\vx)\in \Gamma_1 \:\ec\\
			\big(\til{\rho}_2,\til{\vm}_2 / \til{\rho}_2 ,p_2\big) & \text{ if } (t,\vx)\in \Gamma_2 \:\ec\\
			\big(\rho_+,\vu_+,p_+\big) & \text{ if } (t,\vx)\in \Gamma_+ \ed
		\end{array}
		\right.
	\end{equation} 
	We claim that each $(\rho,\vu,p)$ is indeed an admissible weak solution to the initial value problem \eqref{eq:full-euler-pv-dens} - \eqref{eq:full-euler-pv-en}, \eqref{eq:riemann-init-full} with the desired properties. This can be shown similarly as in the isentropic case. Indeed equations \eqref{eq:full-euler-weak-dens} and \eqref{eq:full-euler-weak-mom} can be handled as \eqref{eq:baro-euler-weak-whole-dens} and \eqref{eq:baro-euler-weak-whole-mom} in the proof of Theorem~\ref{thm:isen-condition} with obvious modifications. Note that for the validity of \eqref{eq:full-euler-weak-mom} it is essential that $p_i(\rho_i)=p_i$ for $i=1,2$. In order to show the energy equation \eqref{eq:full-euler-weak-en} and the entropy inequality \eqref{eq:full-euler-weak-entropy} one has to make use of the fact that $\til{\rho}_i\equiv \rho_i$ for $i=1,2$. As in the isentropic case, this implies that \eqref{eq:sol-pde1} turns into 
	\begin{equation} \label{eq:f4-temp-riemann}
		\iint_{\Gamma_i} \til{\vm}_i\cdot \Grad\varphi \dx\dt - \int_{\partial \Gamma_i} \vm_i\cdot \vn_\vx \varphi \dS_{t,\vx} = 0
	\end{equation}
	for $i=1,2$ and all $\phi\in \Cc\big([0,\infty) \times \R^2\big)$. 
	
	The properties \ref{item:full-riemann.a} - \ref{item:full-riemann.c} can again be proved as in Theorem~\ref{thm:isen-condition}. With property \ref{item:full-riemann.c}, \eqref{eq:full-fansubs-en}, \eqref{eq:f4-temp-riemann} and the fact that 
	$$
	P_i(\rho_i) = \frac{p_i(\rho_i)}{\gamma-1} = \frac{p_i}{\gamma-1} = \rho_i e(\rho_i, p_i) \ec
	$$
	one shows \eqref{eq:full-euler-weak-en}.

	Analogously \eqref{eq:full-euler-weak-entropy} follows from \eqref{eq:full-fansubs-adm} and \eqref{eq:f4-temp-riemann}.
	
	Hence each $(\rho,\vu,p)$ defined in \eqref{eq:f1-temp-riemann} is an admissible weak solution of the initial value problem \eqref{eq:full-euler-pv-dens} - \eqref{eq:full-euler-pv-en}, \eqref{eq:riemann-init-full}. This finishes the proof. 
\end{proof}

\subsubsection{The Corresponding System of Algebraic Equations and Inequalities}

In order to prove existence of an admissible fan subsolution, we again make use of the fact that \eqref{eq:full-fansubs-dens} - \eqref{eq:full-fansubs-adm} can be translated into a set of algebraic equations and inequalities as a fan subsolution is a piecewise constant object. The following Proposition is taken from \name{Al~Baba et al.} \cite[Proposition 4.4]{AKKMM20}.

\begin{prop} \label{prop:full-alg-eq}
	Let $\rho_\pm\in\R^+$, $\vu_\pm\in\R^2$, $p_\pm\in\R^+$ be given. Assume that there exist numbers $\mu_0,\mu_1,\mu_2\in\R$, $\rho_i,p_i\in\R^+$, $\alpha_i,\beta_i,\gamma_i,\delta_i\in \R$ and $C_i\in\R^+$ (for $i=1,2$) which fulfill the following algebraic equations and inequalities: 
	\begin{itemize}
		\item Order of the speeds:
		\begin{align} \label{eq:full-order}
			\mu_0&<\mu_1<\mu_2 \es
		\end{align}
		
		\item Rankine Hugoniot conditions on the left interface:
		\begin{align}
			\mu_0 (\rho_- - \rho_1) &= \rho_- v_- - \rho_1 \beta_1 \es \label{eq:full-rhl1}\\
			\mu_0 (\rho_- u_- - \rho_1 \alpha_1) &= \rho_- u_- v_- - \rho_1 \delta_1 \es\label{eq:full-rhl2}\\
			\mu_0 (\rho_- v_- - \rho_1 \beta_1) &= \rho_- \big(v_-\big)^2 - \rho_1 \bigg(\frac{C_1}{2}-\gamma_1\bigg) + p_- - p_1 \es\label{eq:full-rhl3} 
		\end{align}
		\begin{align} 
			&\mu_0 \bigg(\rho_- \frac{|\vu_-|^2}{2} + \frac{1}{\gamma-1} p_- - \rho_1 \frac{C_1}{2} - \frac{1}{\gamma-1} p_1\bigg) \notag \\
			&= \bigg(\rho_- \frac{|\vu_-|^2}{2} + \frac{\gamma}{\gamma-1} p_-\bigg) v_- - \bigg(\rho_1 \frac{C_1}{2} + \frac{\gamma}{\gamma-1} p_1\bigg) \beta_1 \es \label{eq:full-rhl4}
		\end{align}
		
		\item Rankine Hugoniot conditions on the middle interface: 
		\begin{align}
			\mu_1 (\rho_1 - \rho_2) &= \rho_1 \beta_1 - \rho_2 \beta_2 \es \label{eq:full-rhm1}\\
			\mu_1 (\rho_1 \alpha_1 - \rho_2 \alpha_2) &= \rho_1 \delta_1 - \rho_2 \delta_2 \es \label{eq:full-rhm2}\\
			\mu_1 (\rho_1 \beta_1 - \rho_2 \beta_2) &= \rho_1 \bigg(\frac{C_1}{2}-\gamma_1\bigg) - \rho_2 \bigg(\frac{C_2}{2}-\gamma_2\bigg) + p_1 - p_2 \es \label{eq:full-rhm3}
		\end{align}
		\begin{align}
			&\mu_1 \bigg(\rho_1 \frac{C_1}{2} + \frac{1}{\gamma-1} p_1 - \rho_2 \frac{C_2}{2} - \frac{1}{\gamma-1} p_2\bigg) \notag \\
			&= \bigg(\rho_1 \frac{C_1}{2} + \frac{\gamma}{\gamma-1} p_1\bigg) \beta_1 - \bigg(\rho_2 \frac{C_2}{2} + \frac{\gamma}{\gamma-1} p_2\bigg) \beta_2 \es \label{eq:full-rhm4}
		\end{align}
		
		\item Rankine Hugoniot conditions on the right interface:
		\begin{align}
			\mu_2 (\rho_2 - \rho_+) &= \rho_2 \beta_2 - \rho_+ v_+ \es \label{eq:full-rhr1}\\
			\mu_2 (\rho_2 \alpha_2 - \rho_+ u_+) &= \rho_2 \delta_2 - \rho_+ u_+ v_+ \es \label{eq:full-rhr2}\\
			\mu_2 (\rho_2 \beta_2 - \rho_+ v_+) &= \rho_2 \bigg(\frac{C_2}{2}-\gamma_2\bigg) - \rho_+ \big(v_+\big)^2 + p_2 - p_+ \es \label{eq:full-rhr3}
		\end{align} 
		\begin{align}
			&\mu_2 \bigg(\rho_2 \frac{C_2}{2} + \frac{1}{\gamma-1} p_2 - \rho_+ \frac{|\vu_+|^2}{2} - \frac{1}{\gamma-1} p_+\bigg) \notag \\ 
			&= \bigg(\rho_2 \frac{C_2}{2} + \frac{\gamma}{\gamma-1} p_2\bigg) \beta_2 - \bigg(\rho_+ \frac{|\vu_+|^2}{2} + \frac{\gamma}{\gamma-1} p_+\bigg) v_+ \es \label{eq:full-rhr4}
		\end{align}
		
		\item Subsolution conditions for $i=1,2$:
		\begin{align}
			C_i - \big(\alpha_i\big)^2 - \big(\beta_i\big)^2 &> 0 \es \label{eq:full-sc1}\\ 
			\bigg(\frac{C_i}{2} - \big(\alpha_i\big)^2 + \gamma_i\bigg) \bigg(\frac{C_i}{2} - \big(\beta_i\big)^2 - \gamma_i\bigg) - (\delta_i-\alpha_i \beta_i)^2 &> 0 \es \label{eq:full-sc2} 
		\end{align}
		
		\item Admissibility condition on the left interface: 
		\begin{equation} \label{eq:full-adml}
			\mu_0 \Big(\rho_1 s(\rho_1,p_1) - \rho_- s(\rho_-,p_-)\Big)\leq \rho_1 s(\rho_1,p_1) \beta_1 - \rho_- s(\rho_-,p_-) v_- \es 
		\end{equation}
		
		\item Admissibility condition on the middle interface: 
		\begin{equation} \label{eq:full-admm}
			\mu_1 \Big(\rho_2 s(\rho_2,p_2) - \rho_1 s(\rho_1,p_1)\Big) \leq \rho_2 s(\rho_2,p_2) \beta_2 - \rho_1 s(\rho_1,p_1) \beta_1 \es 
		\end{equation}
		
		\item Admissibility condition on the right interface: 
		\begin{equation} \label{eq:full-admr}
			\mu_2 \Big(\rho_+ s(\rho_+,p_+) - \rho_2 s(\rho_2,p_2)\Big) \leq \rho_+ s(\rho_+,p_+) v_+ - \rho_2 s(\rho_2,p_2) \beta_2 \ed
		\end{equation}
	\end{itemize}
	
	Then
	\begin{align}
		\rho_i &\ec \label{eq:full-fansubs-defn-rho} \\ 
		\vm_i&:= \rho_i \left(\begin{array}{c}
			\alpha_i \\ 
			\beta_i
		\end{array}\right)\ec \label{eq:full-fansubs-defn-m} \\
		\mU_i&:=\rho_i\left(\begin{array}{rr} 
			\gamma_i & \delta_i \\
			\delta_i & -\gamma_i
		\end{array}\right) \ec \label{eq:full-fansubs-defn-U} \\
		c_i &:= \rho_i \frac{C_i}{2} + p_i \ec \quad\text{ and } \label{eq:full-fansubs-defn-c} \\
		p_i & \label{eq:full-fansubs-defn-p} 
	\end{align}
	define an admissible fan subsolution to the initial value problem \eqref{eq:full-euler-pv-dens} - \eqref{eq:full-euler-pv-en}, \eqref{eq:riemann-init-full}, where the corresponding fan partition is determined by the speeds $\mu_0,\mu_1,\mu_2$.
\end{prop} 

\begin{proof}
	The first part of the proof works similarly to the isentropic case, see proof of Proposition~\ref{prop:isen-alg-eq}. Indeed plugging \eqref{eq:full-fansubs-defn-rho} - \eqref{eq:full-fansubs-defn-p} into \eqref{eq:full-rhl1} -  \eqref{eq:full-rhr4} yields the corresponding Rankine-Hugoniot conditions of the PDEs \eqref{eq:full-fansubs-dens} - \eqref{eq:full-fansubs-en}, where for the energy equation \eqref{eq:full-fansubs-en} one has to take into account the equation of state \eqref{eq:incomp-EOS-pv}. 
	
	As in the proof of Proposition~\ref{prop:isen-alg-eq}, one shows that \eqref{eq:full-sc1}, \eqref{eq:full-sc2} imply \eqref{eq:full-fansubs-subs}.
	
	It remains to prove that \eqref{eq:full-adml} - \eqref{eq:full-admr} indeed imply the entropy inequality \eqref{eq:full-fansubs-adm}. The latter, i.e. \eqref{eq:full-fansubs-adm}, is equivalent to the fact that 
	\begin{equation}  \label{eq:full-adm-temp}
		\mu_i \big(\rho_R Z(s_R) - \rho_L Z(s_L)\big) \leq \rho_R Z(s_R) v_R - \rho_L Z(s_L) v_L 
	\end{equation}
	holds for all $(i,L,R)\in \big\{(0,-,1),(1,1,2),(2,2,+)\}$ and all $Z\in C^\infty(\R)$ with $Z'\geq 0$, where we have set $v_i:=\beta_i$ for $i=1,2$, and $s_i:=s(\rho_i,p_i)$ for $i=-,1,2,+$.
	
	So let $Z\in C^\infty(\R)$ with $Z'\geq 0$ and $(i,L,R)\in \big\{(0,-,1),(1,1,2),(2,2,+)\}$ be arbitrary. Due to \eqref{eq:full-rhl1}, \eqref{eq:full-rhm1}, \eqref{eq:full-rhr1} and \eqref{eq:full-adml} - \eqref{eq:full-admr}, the statements
	\begin{align}
		\mu_i (\rho_L - \rho_R) &= \rho_L v_L - \rho_R v_R \qquad \text{ and } \label{eq:f5-temp-riemann} \\
		\mu_i (\rho_R s_R - \rho_L s_L) &\leq \rho_R s_R v_R - \rho_L s_L v_L \label{eq:f6-temp-riemann}
	\end{align}
	are satisfied. From \eqref{eq:f5-temp-riemann} and \eqref{eq:f6-temp-riemann} we deduce that 
	\begin{equation} \label{eq:f7-temp-riemann}
		\mu_i (s_R - s_L) \leq v_L (s_R - s_L) \ed
	\end{equation}
	Since $Z'\geq 0$, i.e. $Z$ is non-decreasing, \eqref{eq:f7-temp-riemann} implies
	$$
		\mu_i \big(Z(s_R) - Z(s_L)\big) \leq v_L \big(Z(s_R) - Z(s_L)\big) \ed
	$$
	Together with \eqref{eq:f5-temp-riemann} this yields \eqref{eq:full-adm-temp}.
	
	Thus $(\rho_i,\vm_i,\mU_i,c_i,p_i)$ (for $i=1,2$) indeed define an admissible fan subsolution.
\end{proof}

\begin{rem}
	The converse of Proposition~\ref{prop:full-alg-eq} holds as well. 
\end{rem}

\subsubsection{Solution of the Algebraic System}

After we have discussed the preliminaries, we now turn our attention towards the proof of Theorem~\ref{thm:full-SS}. So let the assumptions of Theorem~\ref{thm:full-SS} be true. We recall that the proof of Theorem~\ref{thm:full-SS} is finished as soon as we find numbers $\mu_0,\mu_1,\mu_2\in\R$, $\rho_i,p_i\in\R^+$, $\alpha_i,\beta_i,\gamma_i,\delta_i\in \R$ and $C_i\in\R^+$ (for $i=1,2$) which satisfy \eqref{eq:full-order} - \eqref{eq:full-admr} due to Proposition~\ref{prop:full-alg-eq} and Theorem~\ref{thm:full-condition}. Inspired by the corresponding case for the isentropie Euler equations, see Subsection~\ref{subsec:riemann-isen-other-cases}, we look for those numbers as suitable perturbations of the corresponding numbers in the self-similar solution. This perturbation will be quantified by a small parameter $\ep > 0$. Note that in contrast to the case where the self-similar solution contains only one shock, it is not necessary to introduce an auxiliary state. The following can be also found in \name{Al Baba et al.} \cite[Subsection 5.2]{AKKMM20}.

For convenience we define functions $A,B,D:\R\rightarrow\R$ by 
\begin{align*} 
	A(\ep)&:= \rho_-(\rho_{M-}+\ep)(\rho_{M+}-\ep - \rho_+) - \rho_+(\rho_{M+}-\ep)(\rho_{M-}+\ep - \rho_-) \es \\
	B(\ep)&:= \rho_-\rho_+(\rho_{M-}+\ep)(\rho_{M+}-\ep)\big(v_{-} - v_{+}\big)^2 - (p_- - p_+)\ A(\ep) \es \\
	D(\ep) &:= v_{-} \rho_- (\rho_{M-} + \ep) (\rho_{M+} - \ep - \rho_+) - v_{+} \rho_+ (\rho_{M+}-\ep) (\rho_{M-} + \ep - \rho_-)\ed
\end{align*}

First of all we want to show some properties of the functions $A$ and $B$. It is easy to deduce from Proposition~\ref{prop:full-selfsimilar-SS} and the assumption $v_M=0$, that 
\begin{align}
	v_- &= \sqrt{2}\frac{p_M - p_-}{\sqrt{\rho_- \big((\gamma-1)p_- + (\gamma+1) p_M\big)}} \ec\label{eq:1} \\
	v_+ &= 	-\sqrt{2}\frac{p_M - p_+}{\sqrt{\rho_+ \big((\gamma-1)p_+ + (\gamma+1) p_M\big)}} \ed \label{eq:2}
\end{align}
Since $p_M > \max\{p_-,p_+\}$, we have $v_+<0<v_-$. Furthermore we obtain from Proposition~\ref{prop:full-selfsimilar-SS} and \eqref{eq:1}, \eqref{eq:2} that 
\begin{align*}
	\frac{\rho_{M-} - \rho_-}{\rho_-\rho_{M-}} = \frac{2 (p_M - p_-)}{\rho_- \big((\gamma-1)p_- + (\gamma+1) p_M\big)} = \frac{\big(v_-\big)^2}{p_M - p_-} \ec \\
	\frac{\rho_{M+} - \rho_+}{\rho_+\rho_{M+}} = \frac{2 (p_M - p_+)}{\rho_+ \big((\gamma-1)p_+ + (\gamma+1) p_M\big)} = \frac{\big(v_+\big)^2}{p_M - p_+} \ed 
\end{align*}
This leads to 
\begin{align*}
	\big(v_-\big)^2 + (p_- - p_+)\frac{\rho_{M-} - \rho_-}{\rho_-\rho_{M-}} &= \big(v_-\big)^2 \frac{p_M - p_+}{p_M - p_-}\ec \\ 
	\big(v_+\big)^2 - (p_- - p_+)\frac{\rho_{M+} - \rho_+}{\rho_+\rho_{M+}} &= \big(v_+\big)^2 \frac{p_M - p_-}{p_M - p_+}\ec
\end{align*}
and finally
\begin{align*}
	B(0) &= \rho_-\rho_+\rho_{M-}\rho_{M+}\big(v_- - v_+\big)^2 \\
	&\qquad - (p_- - p_+)\Big(\rho_-\rho_{M-}\big(\rho_{M+} - \rho_+\big) - \rho_+\rho_{M+}\big(\rho_{M-} - \rho_-\big)\Big) \\
	& = \rho_-\rho_+\rho_{M-}\rho_{M+}\bigg[\big(v_-\big)^2 + (p_- - p_+)\frac{\rho_{M-} - \rho_-}{\rho_-\rho_{M-}} + \big(v_+\big)^2 \\
	&\qquad - (p_- - p_+)\frac{\rho_{M+} - \rho_+}{\rho_+\rho_{M+}} -2v_-v_+ \bigg] \\
	&= \rho_-\rho_+\rho_{M-}\rho_{M+}\bigg[\big(v_-\big)^2 \frac{p_M - p_+}{p_M - p_-} + \big(v_+\big)^2 \frac{p_M - p_-}{p_M - p_+} -2v_-v_+ \bigg] \\
	&= \rho_-\rho_+\rho_{M-}\rho_{M+}\bigg[v_- \sqrt{\frac{p_M - p_+}{p_M - p_-}} - v_+ \sqrt{\frac{p_M - p_-}{p_M - p_+}} \ \bigg]^2 \\
	&\geq  0 \ed
\end{align*} 
Due to $p_M > \max\{p_-,p_+\}$ and $v_+<0<v_-$, we even have $B(0)>0$. Now, by continuity of the function $B$, there exists an $\ep_{\max,1}>0$ such that $B(\ep)>0$ for all $\ep\in(0,\ep_{\max,1}]$. Because $\rho_{M+}>\rho_+$, there exists $\ep_{\max,2}>0$ such that $\rho_{M+}-\ep-\rho_+ > 0$ for all $\ep\in(0,\ep_{\max,2}]$.

Next we want to show that there is an $\ep_{\max,3}>0$ such that $A(\ep)\neq 0$ for all $\ep\in(0,\ep_{\max,3}]$. To this end, let us first assume, that $A(0)\neq 0$. Then, by continuity of the function $A$, there exists such an $\ep_{\max,3}>0$. Now consider the case where $A(0) = 0$. Then we obtain 
\begin{align*}
	A(\ep) &=  \rho_-(\rho_{M-}+\ep)(\rho_{M+}-\ep - \rho_+) - \rho_+(\rho_{M+}-\ep)(\rho_{M-}+\ep - \rho_-) \\
	&= \ep^2 (\rho_+ - \rho_-) - \ep \big(2\rho_-\rho_+ + (\rho_- - \rho_+)(\rho_{M-} - \rho_{M+})\big) + \underbrace{A(0)}_{=0} \\
	&= \ep \big((\rho_+ - \rho_-)\ep -2\rho_-\rho_+ - (\rho_- - \rho_+)(\rho_{M-} - \rho_{M+})\big) \ec
\end{align*}
which has at most two zeros: If $\rho_- = \rho_+$ then 
\begin{align*}
	A(\ep)=0 \qquad \Longleftrightarrow \qquad \ep=0 \ec
\end{align*}
and if $\rho_- \neq \rho_+$ then 
\begin{align*}
	A(\ep)=0 \qquad \Longleftrightarrow \qquad \ep=0\ \text{ or }\ \ep=\frac{2\rho_-\rho_+ + (\rho_- - \rho_+)(\rho_{M-} - \rho_{M+})}{\rho_+ - \rho_-} \ed
\end{align*}
Hence there exists $\ep_{\max,3}>0$ such that $A(\ep)\neq 0$ for all $\ep\in(0,\ep_{\max,3}]$.

We set $\ep_{\max}:=\min\{\ep_{\max,1},\ep_{\max,2},\ep_{\max,3}\}$ and then we have
\begin{align} 
	A(\ep) &\neq 0\ec \label{eq:28a} \\
	B(\ep) &> 0\ec \label{eq:28b}\\
	\rho_{M+}-\ep-\rho_+ &> 0\ec \label{eq:28c}\\
	\rho_{M-} + \ep-\rho_- &> 0 \label{eq:28d}\ec
\end{align}
for all $\ep\in(0,\ep_{\max}]$.

Next we define the functions $\mu_0,\mu_1,\mu_2:(0,\ep_{\max}]\rightarrow\R$ by 
\begin{align*} 
	\mu_0(\ep)&:=\frac{1}{A(\ep)} \bigg[D(\ep) +\rho_-\rho_+(\rho_{M+} - \ep) (v_- - v_+)  \\
	&\qquad\qquad\qquad - \sqrt{\big(\rho_{M-} + \ep\big)^2\ \frac{\rho_{M+} - \ep - \rho_+}{\rho_{M-} + \ep-\rho_-}\ B(\ep)} \ \bigg]\es \\
	\mu_1(\ep)&:=\frac{1}{A(\ep)} \bigg[D(\ep)- \sqrt{(\rho_{M-} + \ep-\rho_-)(\rho_{M+} - \ep - \rho_+)\ B(\ep)}\ \bigg]\es \\
	\mu_2(\ep)&:= \frac{1}{A(\ep)} \bigg[D(\ep) + \rho_-\rho_+(\rho_{M-} + \ep) (v_- - v_+) \\
	&\qquad\qquad\qquad - \sqrt{(\rho_{M+} - \ep)^2\ \frac{\rho_{M-} + \ep - \rho_-}{\rho_{M+} - \ep - \rho_+}\ B(\ep)}\ \bigg]\ed
\end{align*}
Note first that the functions $\mu_0,\mu_1,\mu_2$ are well-defined because of \eqref{eq:28a} - \eqref{eq:28d}. We claim that these functions define perturbations of the shock speeds $\sigma_-$, $v_M$ and $\sigma_+$ of the self-similar solution. More precisely the following is true.

\begin{prop} \label{prop:limittostandard1}
	It holds that 
	\begin{align*}
		\lim\limits_{\ep\to 0}\mu_0(\ep) &= \sigma_- \ec &
		\lim\limits_{\ep\to 0}\mu_1(\ep) &= v_M \ec &
		\lim\limits_{\ep\to 0}\mu_2(\ep) &= \sigma_+ \ed 
	\end{align*}
\end{prop}

\begin{proof}
	We start with the Rankine-Hugoniot conditions for the self-similar solution \eqref{eq:full-selfsimilar-SS-dens}, \eqref{eq:full-selfsimilar-SS-mom} and we obtain by eliminating $\sigma_-$, $\sigma_+$ and $p_M$ that
	\begin{align*}
		&(\rho_{M+}-\rho_+) \big(\rho_- v_- - \rho_{M-} v_M\big)^2 + (\rho_- - \rho_{M-}) \big(\rho_{M+} v_M - \rho_+ v_+\big)^2 \\
		&= \Big(\rho_- \big(v_-\big)^2 - \rho_+ \big(v_+\big)^2 - \big(v_M\big)^2 (\rho_{M-} - \rho_{M+}) + p_- - p_+\Big) (\rho_{M+} - \rho_+) (\rho_- - \rho_{M-}) \ed
	\end{align*}
	This is equivalent to
	\begin{equation} \label{eq:quadratic_vm2}
		A(0)\ \big(v_M\big)^2 - 2 D(0)\ v_M + E = 0 \ec
	\end{equation}
	where the constant
	\begin{align*} 
		E &:= (p_- - p_+)(\rho_{M-} - \rho_-)(\rho_{M+} - \rho_+) \\
		&\qquad+ \big(v_-\big)^2 \rho_- \rho_{M-} (\rho_{M+} - \rho_+) - \big(v_+\big)^2 \rho_+ \rho_{M+} (\rho_{M-} - \rho_-) 
	\end{align*}
	only depends on the initial states.	Now we have to consider two cases, namely $A(0)=0$ and $A(0)\neq 0$.
	
	Let us start with $A(0)=0$. Then we easily deduce that 
	\begin{equation*}
	D(0) = \rho_- \rho_{M-} (\rho_{M+} - \rho_+) (v_- - v_+) \ec
	\end{equation*}
	which does not vanish because $v_+<0<v_-$ and $\rho_{M+} > \rho_+$. Hence we get from \eqref{eq:quadratic_vm2}, that 
	\begin{equation} \label{eq:4}
		v_M = \frac{E}{2D(0)} \ed
	\end{equation}
	Next we want to compute $\lim\limits_{\ep \to 0}\mu_1(\ep)$ and compare it with \eqref{eq:4}. Keeping in mind that we are considering the case $A(0)=0$, we get 
	\begin{align*}
		&D(0)- \sqrt{(\rho_{M-} - \rho_-)(\rho_{M+} - \rho_+)\ B(0)} \\
		&=\big(v_- - v_+\big) \rho_- \rho_{M-} (\rho_{M+} - \rho_+) \\
		&\qquad- \sqrt{(\rho_{M-} - \rho_-)(\rho_{M+} - \rho_+)  \rho_- \rho_+ \rho_{M-}\rho_{M+}(v_- - v_+)^2} \\
		&=\big(v_- - v_+\big) \rho_- \rho_{M-} (\rho_{M+} - \rho_+) - \sqrt{(\rho_{M+} - \rho_+)^2  (\rho_+)^2 (\rho_{M+})^2(v_- - v_+)^2} \\  &= 0\ed
	\end{align*}
	Hence we can apply L'Hospital's rule. We obtain 
	\begin{align*}
		\lim\limits_{\ep\to 0 } A'(\ep) &= \lim\limits_{\ep\to 0 }\Big(\rho_-(\rho_{M+}-\ep - \rho_+) - \rho_-(\rho_{M-}+\ep) - \rho_+(\rho_{M+}-\ep) \\
		&\qquad\qquad\qquad+ \rho_+(\rho_{M-}+\ep - \rho_-)\Big) \\
		&= -2\rho_-\rho_+ - (\rho_- - \rho_+) (\rho_{M-} - \rho_{M+}) \ed
	\end{align*} 
	A short calculation shows that this is nonzero: From $A(0)=0$ we can deduce that 
	$$
		\frac{\rho_- - \rho_+}{\rho_-\rho_+} = \frac{\rho_{M-} - \rho_{M+}}{\rho_{M-}\rho_{M+}} \ed
	$$
	This means that $\rho_- - \rho_+$ and $\rho_{M-} - \rho_{M+}$ have the same sign, which implies 
	$$
		(\rho_- - \rho_+)(\rho_{M-} - \rho_{M+})\geq 0 \ed
	$$ 
	Since $\rho_-\rho_+ >0$, we have $A'(0)<0$, in particular $A'(0)\neq 0$.
	
	Furthermore a long but straightforward computation yields
	\begin{equation*}
		\lim\limits_{\ep\to 0 }\Big[D(\ep)- \sqrt{(\rho_{M-} + \ep-\rho_-)(\rho_{M+} - \ep - \rho_+)\ B(\ep)}\ \Big]' = \frac{A'(0)\ E}{2D(0)} \ed
	\end{equation*} 
	Hence by L'Hospital's rule we obtain $\lim\limits_{\ep\to 0 }\mu_1(\ep)= \frac{E}{2D(0)}$ and recalling \eqref{eq:4}, we deduce $\lim\limits_{\ep\to 0 }\mu_1(\ep) = v_M$.
	
	Let us now consider the case $A(0)\neq 0$. Then we obtain from \eqref{eq:quadratic_vm2} that 
	\begin{align*}
		v_M &= \frac{1}{A(0)}\Big[D(0) \pm \sqrt{D(0)^2 - A(0)\ E\ }\ \Big] \ed
	\end{align*}
	The correct sign in the equation above is ``$-$'' because\footnote{Alternatively, this can be verified by considering the admissibility criterion.} $v_M=0$ and $D(0)>0$, which easily follows from $v_+<0<v_-$.
	Furthermore it is simple to check that 
	$$
		D(0)^2 - A(0)\ E=(\rho_{M-} - \rho_-)(\rho_{M+} - \rho_+)\ B(0) \ed
	$$ 
	Then it is easy to conclude $\mu_1(0) = v_M$. 
	
	To finish the proof of Proposition~\ref{prop:limittostandard1} we have to show that $\lim\limits_{\ep\to 0 } \mu_1(\ep)=v_M$ implies that $\lim\limits_{\ep\to 0 } \mu_0(\ep)=\sigma_-$ and $\lim\limits_{\ep\to 0 } \mu_2(\ep)=\sigma_+$. It is straightforward to deduce that 
	\begin{align*}
		\mu_0(\ep)&=v_- + \frac{\rho_{M-} + \ep}{\rho_{M-} + \ep - \rho_-} \big(\mu_1(\ep) - v_-\big) \ec \\
		\mu_2(\ep)&=v_+ + \frac{\rho_{M+} - \ep}{\rho_{M+} - \ep - \rho_+} \big(\mu_1(\ep) - v_+\big) \ed
	\end{align*}
	On the other hand we get from \eqref{eq:full-selfsimilar-SS-dens}, that
	\begin{align*}
		\sigma_-&=v_- + \frac{\rho_{M-}}{\rho_{M-} - \rho_-} \big(v_M - v_-\big) \ec \\
		\sigma_+&=v_+ + \frac{\rho_{M+}}{\rho_{M+} - \rho_+} \big(v_M - v_+\big) \ed
	\end{align*}
	Hence we easily deduce $\lim\limits_{\ep\to 0 } \mu_0(\ep)=\sigma_-$ and $\lim\limits_{\ep\to 0 } \mu_2(\ep)=\sigma_+$.
\end{proof} 

Because of $\sigma_- < v_M < \sigma_+$ and the continuity of the functions $\mu_0,\mu_1,\mu_2$, we may assume that $\mu_0(\ep) < \mu_1(\ep) < \mu_2(\ep)$ for all $\ep \in(0,\ep_{\max}]$. If this is not the case we redefine $\ep_{\max}$ to be a bit smaller than the smallest positive value of $\ep$ for which $\mu_0(\ep)<\mu_1(\ep)<\mu_2(\ep)$ is violated. 

In order to proceed further we need to introduce a second positive parameter $\ov{\ep} > 0$. We define the functions $C_1,C_2,\gamma_1,\gamma_2:(0,\ep_{\max}]\times(0,p_M)$ by 
\begin{align*} 
	C_1(\ep,\ov{\ep})&:= \frac{2}{(\rho_{M-} + \ep)\big(\mu_0(\ep) - \mu_1(\ep)\big)} \bigg[-\mu_0(\ep)\bigg(\frac{1}{\gamma-1}(p_M - \ov{\ep} - p_-)-\rho_- \frac{\big(v_-\big)^2}{2}\bigg) \\ 
	&\qquad\qquad\quad\qquad +\mu_1(\ep)\frac{\gamma}{\gamma-1}(p_M - \ov{\ep}) - \bigg(\rho_- \frac{\big(v_-\big)^2}{2} + \frac{\gamma}{\gamma-1}p_-\bigg) v_-\bigg] \ec \\
	C_2(\ep,\ov{\ep})&:= \frac{2}{(\rho_{M+} - \ep)\big(\mu_2(\ep) - \mu_1(\ep)\big)}\bigg[-\mu_2(\ep)\bigg(\frac{1}{\gamma-1}(p_M - \ov{\ep} - p_+)-\rho_+ \frac{\big(v_+\big)^2}{2}\bigg) \\ 
	&\qquad\qquad\quad\qquad +\mu_1(\ep)\frac{\gamma}{\gamma-1}(p_M - \ov{\ep}) - \bigg(\rho_+ \frac{\big(v_+\big)^2}{2} + \frac{\gamma}{\gamma-1}p_+\bigg)v_+\bigg] \ec \\
	\gamma_1(\ep,\ov{\ep})&:= \frac{1}{(\rho_{M-} + \ep)}\bigg[(\rho_{M-} + \ep)\frac{C_1(\ep,\ov{\ep})}{2} - \rho_- \big(v_-\big)^2 + p_M - \ov{\ep} - p_- \\
	&\qquad\qquad\quad\qquad - \mu_0(\ep) \big((\rho_{M-} + \ep)\mu_1(\ep) - \rho_- v_-\big)\bigg] \ec \\
	\gamma_2(\ep,\ov{\ep})&:= \frac{1}{(\rho_{M+} - \ep)}\bigg[(\rho_{M+} - \ep)\frac{C_2(\ep,\ov{\ep})}{2} - \rho_+ \big(v_+\big)^2 + p_M - \ov{\ep} - p_+ \\
	&\qquad\qquad\quad\qquad - \mu_2(\ep) \big((\rho_{M+} - \ep)\mu_1(\ep) - \rho_+ v_+\big)\bigg] \ed
\end{align*}
Note that these functions are well-defined because of the arguments above. More precisely, it holds that $\mu_0(\ep)-\mu_1(\ep)\neq 0$ and $\mu_2(\ep)-\mu_1(\ep)\neq 0$ for all $\ep \in(0,\ep_{\max}]$.

\begin{prop} \label{prop:limittostandard2}
	It holds that 
	\begin{align*} 
	\lim\limits_{(\ep,\ov{\ep})\to (0,0)} C_1(\ep,\ov{\ep}) &= \big(v_M\big)^2 \ec & \lim\limits_{(\ep,\ov{\ep})\to (0,0)} \gamma_1(\ep,\ov{\ep}) &= -\frac{\big(v_M\big)^2}{2} \ec \\
	\lim\limits_{(\ep,\ov{\ep})\to (0,0)} C_2(\ep,\ov{\ep}) &= \big(v_M\big)^2 \ec & \lim\limits_{(\ep,\ov{\ep})\to (0,0)} \gamma_2(\ep,\ov{\ep}) &= -\frac{\big(v_M\big)^2}{2} \ed
	\end{align*}
\end{prop}

\begin{proof}
	To prove this, we need the Rankine Hugoniot conditions of the self-similar solution in the energy equation \eqref{eq:full-selfsimilar-SS-en}, where we use \eqref{eq:incomp-EOS-pv} to replace $e$.
	
	We obtain that $\lim\limits_{(\ep,\ov{\ep})\to (0,0)} C_1(\ep,\ov{\ep}) = \big(v_M\big)^2$ and $\lim\limits_{(\ep,\ov{\ep})\to (0,0)} C_2(\ep,\ov{\ep}) = \big(v_M\big)^2$ by using Proposition~\ref{prop:limittostandard1} and \eqref{eq:full-selfsimilar-SS-en}.
	
	The fact that $\lim\limits_{(\ep,\ov{\ep})\to (0,0)} \gamma_1(\ep,\ov{\ep}) = -\frac{(v_M)^2}{2}$ and $\lim\limits_{(\ep,\ov{\ep})\to (0,0)} \gamma_2(\ep,\ov{\ep}) = -\frac{(v_M)^2}{2}$ can be shown analogously by using the Rankine Hugoniot conditions \eqref{eq:full-selfsimilar-SS-mom}.
\end{proof}

We continue the proof of the Theorem~\ref{thm:full-SS} by observing that the perturbations defined above indeed help to define numbers $\mu_0,\mu_1,\mu_2\in\R$, $\rho_i,p_i\in\R^+$, $\alpha_i,\beta_i,\gamma_i,\delta_i\in \R$ and $C_i\in\R^+$ (for $i=1,2$) as required by Proposition~\ref{prop:full-alg-eq}. 

\begin{prop} \label{prop:3algequa}
	Assume that there exist $\ep\in (0,\ep_{\max}]$, $,\ov{\ep}\in(0,p_M)$ which satisfy the following inequalities.
	\begin{itemize}
		\item Order of the speeds:
		\begin{equation} \label{eq:3order}
			\mu_0(\ep)<\mu_1(\ep)<\mu_2(\ep) \es
		\end{equation} 
		
		\item Subsolution conditions:
		\begin{align}
			C_1(\ep,\ov{\ep}) - \mu_1(\ep)^2 &> 0 \es \label{eq:3sc1}\\
			C_2(\ep,\ov{\ep}) - \mu_1(\ep)^2 &> 0 \es \label{eq:3sc2}\\
			\bigg(\frac{C_1(\ep,\ov{\ep})}{2} + \gamma_1(\ep,\ov{\ep})\bigg) \bigg(\frac{C_1(\ep,\ov{\ep})}{2} - \mu_1(\ep)^2 - \gamma_1(\ep,\ov{\ep})\bigg)  &> 0 \es \label{eq:3sc3} \\
			\bigg(\frac{C_2(\ep,\ov{\ep})}{2} + \gamma_2(\ep,\ov{\ep})\bigg) \bigg(\frac{C_2(\ep,\ov{\ep})}{2} - \mu_1(\ep)^2 - \gamma_2(\ep,\ov{\ep})\bigg)  &> 0 \es \label{eq:3sc4}
		\end{align}
		
		\item Admissibility condition on the left interface:
		\begin{align}
			&\mu_0(\ep) \Big((\rho_{M-}+\ep) s(\rho_{M-}+\ep,p_M-\ov{\ep}) - \rho_- s(\rho_-,p_-)\Big) \notag\\
			&\quad\leq (\rho_{M-}+\ep) s(\rho_{M-}+\ep,p_M-\ov{\ep}) \mu_1(\ep) - \rho_- s(\rho_-,p_-) v_- \es \label{eq:3adml}
		\end{align}
		
		\item Admissibility condition on the right interface: 
		\begin{align}
			&\mu_2(\ep) \Big(\rho_+ s(\rho_+,p_+) - (\rho_{M+}-\ep) s(\rho_{M+}-\ep,p_M-\ov{\ep})\Big) \notag\\
			&\quad\leq \rho_+ s(\rho_+,p_+) v_+ - (\rho_{M+}-\ep) s(\rho_{M+}-\ep,p_M-\ov{\ep}) \mu_1(\ep) \ed \label{eq:3admr}
		\end{align}
	\end{itemize}
 	Then 
 	\begin{align*}
 		\mu_0 &:= \mu_0(\ep)\ec &
 		\mu_1 &:= \mu_1(\ep)\ec &
 		\mu_2 &:= \mu_2(\ep)\ec \\
 		\rho_1 &:= \rho_{M-} + \ep \ec & & &
 		\rho_2 &:= \rho_{M+} - \ep \ec \\
 		& & p_1 &:= p_2 := p_M - \ov{\ep} \ec \\
 		& & \alpha_1 &:=\alpha_2 := 0 \ec \\
 		& & \beta_1 &:=\beta_2 := \mu_1(\ep) \ec \\
 		\gamma_1 &:= \gamma_1(\ep,\ov{\ep}) \ec & & & 
 		\gamma_2 &:= \gamma_2(\ep,\ov{\ep}) \ec \\
 		& & \delta_1 &:= \delta_2 := 0 \\
 		C_1 &:= C_1(\ep,\ov{\ep}) \ec & & & 
 		C_2 &:= C_2(\ep,\ov{\ep}) 
 	\end{align*}
 	satisfy \eqref{eq:full-order} - \eqref{eq:full-admr}.
\end{prop}

\begin{proof}
	The proof of Proposition~\ref{prop:3algequa} is a matter of straightforward calculation, where one has to recall that by assumption $u_-=u_+=0$.
\end{proof}

Hence in order to finish the proof of Theorem~\ref{thm:full-SS}, we have to find a pair of small parameters $(\ep,\ov{\ep})\in (0,\ep_{\max}]\times(0,p_M)$ such that the conditions \eqref{eq:3order} - \eqref{eq:3admr} are satisfied. 

We start with noting that we already have \eqref{eq:3order} fulfilled for all $\ep\in (0,\ep_{\max}]$.

Let us now investigate the subsolution conditions \eqref{eq:3sc1} - \eqref{eq:3sc4}. We start with the terms in the first parenthesis in \eqref{eq:3sc3} - \eqref{eq:3sc4}. We obtain by using that $\ov{\ep}\in(0,p_M)$
\begin{align*}
	\frac{C_1(\ep,\ov{\ep})}{2} + \gamma_1(\ep,\ov{\ep}) &= \frac{3-\gamma}{\gamma-1} \frac{ \ov{\ep}}{\rho_{M-} + \ep} - \frac{2\mu_1 (\ep)  \ov{\ep}}{(\rho_{M-} + \ep)\big(\mu_0(\ep)-\mu_1(\ep)\big)} + \frac{C_1(\ep,0)}{2} + \gamma_1(\ep,0) \\ 
	&\geq \frac{3-\gamma}{\gamma-1} \frac{\ov{\ep}}{\rho_{M-} + \ep} \underbrace{- \frac{2|\mu_1 (\ep)|  p_M}{(\rho_{M-} + \ep)\ \big|\mu_0(\ep)-\mu_1(\ep)\big|} + \frac{C_1(\ep,0)}{2} + \gamma_1(\ep,0)}_{=:R_1(\ep)} \ec \\ 
	\frac{C_2(\ep,\ov{\ep})}{2} + \gamma_2(\ep,\ov{\ep}) &= \frac{3-\gamma}{\gamma-1} \frac{\ov{\ep}}{\rho_{M+} - \ep} - \frac{2\mu_1 (\ep)  \ov{\ep}}{(\rho_{M+} - \ep)\big(\mu_2(\ep)-\mu_1(\ep)\big)} + \frac{C_2(\ep,0)}{2} + \gamma_2(\ep,0) \\ 
	&\geq \frac{3-\gamma}{\gamma-1} \frac{\ov{\ep}}{\rho_{M+} - \ep} \underbrace{- \frac{2|\mu_1 (\ep)|  p_M}{(\rho_{M+} - \ep)\ \big|\mu_2(\ep)-\mu_1(\ep)\big|} + \frac{C_2(\ep,0)}{2} + \gamma_2(\ep,0)}_{=:R_2(\ep)} \ec
\end{align*}	
where Propositions \ref{prop:limittostandard1} and \ref{prop:limittostandard2} together with the fact that $v_M=0$ imply that 
\begin{align*}
\lim\limits_{\ep\to 0 } R_1(\ep) &= 0 \ec & \lim\limits_{\ep\to 0 } R_2(\ep) &= 0 \ed
\end{align*}
Therefore $|R_1(\ep)|$ and $|R_2(\ep)|$ become arbitrary small if we choose $\ep$ small. Because $\gamma\in(1,3)$, there exists $\til{\ep}_1(\ov{\ep})\in(0,\ep_{\max}]$ for each $\ov{\ep}\in(0,p_M)$, such that 
\begin{equation}\label{eq:3T1}
\frac{C_1(\ep,\ov{\ep})}{2} + \gamma_1(\ep,\ov{\ep}) > 0 \qquad \text{ and } \qquad \frac{C_2(\ep,\ov{\ep})}{2} + \gamma_2(\ep,\ov{\ep}) > 0
\end{equation}
hold for all $\ep\in(0,\til{\ep}_1(\ov{\ep}))$. 

Similarly we handle terms in the second parenthesis in inequalities \eqref{eq:3sc3} - \eqref{eq:3sc4}. We obtain
\begin{align*}
\frac{C_1(\ep,\ov{\ep})}{2} - \mu_1(\ep)^2 - \gamma_1(\ep,\ov{\ep}) &= \frac{\ov{\ep}}{\rho_{M-}+\ep} + \underbrace{\frac{C_1(\ep,0)}{2} - \mu_1(\ep)^2 - \gamma_1(\ep,0)}_{=:R_3(\ep)}\ec \\
\frac{C_2(\ep,\ov{\ep})}{2} - \mu_1(\ep)^2 - \gamma_2(\ep,\ov{\ep}) &= \frac{\ov{\ep}}{\rho_{M+} - \ep} + \underbrace{\frac{C_2(\ep,0)}{2} - \mu_1(\ep)^2 - \gamma_2(\ep,0)}_{=:R_4(\ep)}\ed
\end{align*}
With the same arguments as above, we obtain that for each $\ov{\ep} \in(0,p_M)$ there exists $\til{\ep}_2(\ov{\ep})\in (0,\ep_{\max}]$ such that
\begin{equation}\label{eq:3T2}
	\frac{C_1(\ep,\ov{\ep})}{2} - \mu_1(\ep)^2 - \gamma_1(\ep,\ov{\ep}) > 0 \qquad \text{ and } \qquad \frac{C_2(\ep,\ov{\ep})}{2} - \mu_1(\ep)^2 - \gamma_2(\ep,\ov{\ep})
\end{equation}
hold for all $\ep\in(0,\til{\ep}_2(\ov{\ep}))$. 

Combining \eqref{eq:3T1} and \eqref{eq:3T2} we obtain \eqref{eq:3sc3} and \eqref{eq:3sc4} whereas summing together \eqref{eq:3T1} and \eqref{eq:3T2}, we obtain \eqref{eq:3sc1} and \eqref{eq:3sc2}.

To finish the proof we have to show that we can achieve that in addition the admissibility conditions \eqref{eq:3adml} and \eqref{eq:3admr} hold. Note that according to Proposition~\ref{prop:limittostandard1} the admissibility conditions \eqref{eq:3adml} and \eqref{eq:3admr} turn into the admissibility conditions of the self-similar solution \eqref{eq:full-selfsimilar-SS-adm-}, \eqref{eq:full-selfsimilar-SS-adm+} as $(\ep,\ov{\ep})\to (0,0)$. Since the latter are fulfilled strictly, we can choose $\ov{\ep}>0$ sufficiently small and also $\ep\in(0,\min\{\til{\ep_1}(\ov{\ep}),\til{\ep_2}(\ov{\ep})\})$ sufficiently small such that \eqref{eq:3adml} and \eqref{eq:3admr} hold. This finishes the proof of Theorem~\ref{thm:full-SS}.

\subsection{Sketches of the Non-Uniqueness Proofs for the Other Cases} \label{subsec:riemann-full-other-cases}

As in the isentropic case, we sketch how the non-uniqueness proof works in the other cases (2, 4, 6, 8, 11, 13, 15 and 17 in Table~\ref{tab:full-results}). 

\subsubsection{One Shock and One Rarefaction}

The non-uniqueness proof in the case where the self-similar solution contains one shock and one rarefaction (cases 6, 8, 15 and 17 in Table~\ref{tab:full-results})) is highly inspired by the corresponding isentropic case, see Subsection~\ref{subsec:riemann-isen-SR}. \name{Klingenberg}, the author and others \cite{KKMM20} show a ``smallness'' result, i.e. there exists an admissible fan subsolutions in the sense of Definition~\ref{defn:full-fansubs} and hence infinitely many solutions, if the rarefaction is ``small''. For ``large'' rarefactions the problem is solved using and auxiliary state and a patching approach.

\subsubsection{One Shock}

The paper \cite{KKMM20} by \name{Klingenberg}, the author and others covers the one-shock-case (2, 4, 11 and 13 in Table~\ref{tab:full-results}) as well. As for the isentropic Euler system, one uses the smallness result from the case above again, where now the auxiliary state is connected to $(\rho_+,\vu_+,p_+)$ by a shock.

\subsection{Other Results in the Context of the Riemann Problem} \label{subsec:riemann-full-other-results}

We finish by mentioning two other results in the context of the initial value problem for the full Euler system \eqref{eq:full-euler-pv-dens} - \eqref{eq:full-euler-pv-en} with Riemann initial data \eqref{eq:riemann-init-full}. 

\begin{itemize}
	\item If the self-similar solution contains two shocks, \name{Al Baba}, the author and others \cite{AKKMM20} present another approch apart from what we exhibited in Subsection~\ref{subsec:riemann-full-SS}. More precisely they also considered fan subsolutions where the underlying fan partitions consist of only three sets, similarly to what we considered for the isentropic Euler equations, cf. Definition~\ref{defn:isen-fanpart}. On the one hand this strategy works for all\footnote{Note once more, that in Subsection~\ref{subsec:riemann-full-SS} we had to assume that $\gamma<3$.} $\gamma>1$, but on the other hand it covers only a subclass of initial states for which $v_+ - v_-$ is sufficiently small. 
	
	\item As already mentioned several times, the question whether or not the self-similar solution is the unique admissible weak solution if it consists only of a contact discontinuity (case~10 in Table~\ref{tab:full-results}), is open. Recently it was shown by \name{Kang}, \name{Vasseur} and \name{Wang}~\cite{KanVasWan20} that for such initial data the vanishing viscosity limit of the Navier-Stokes-Fourier system yields no other than the self-similar solution. In other words the self-similar solution is unique in the class of solutions that are obtained as a vanishing viscosity limit from the Navier-Stokes-Fourier system. This does however not mean that the self-similar solution is unique in the class of admissible weak solutions in the sense of Definition~\ref{defn:aws-full-whole}.
\end{itemize}

%% file: DissertationCh0.tex
\chapter{Notation and Lemmas} \label{chap:not} 

\section{Sets} \label{sec:not-sets}

The basic sets of numbers are denoted as follows.
\begin{itemize}
	\item $\N:=\{1,2,...\}$ -- the set of natural numbers. In particular $0\notin\N$.
	\item $\N_0:=\N\cup \{0\}$. 
	\item $\R$ -- the set of real numbers.
	\item Intervals are denoted e.g. by $[a,b) := \{x\in\R\,|\,a\leq x < b\}$ (where $a<b$). The intervals $(a,b)$, $[a,b]$ and $(a,b]$ are defined analogously.
	\item $\R^+:=(0,\infty)$ -- the set of positive real numbers. 
	\item $\R^+_0 := [0,\infty)$ -- the set of non-negative real numbers. 
	\item $\R^-$ and $\R^-_0$ are defined analogously. 
\end{itemize}

\noindent Furthermore we use the following notation for sets in general. 
\begin{itemize}
	\item The empty set is denoted by $\emptyset$.
	\item For a subset $A$ of a set $B$ we write $A\subset B$. If $A\subset B$ and $A\neq B$, we also write $A\subsetneq B$. We avoid the symbol $\subsetnotused$. 
	\item The set difference of two sets $A$ and $B$ is denoted by $A\setminus B = \{x\in A\,|\, x\notin B\}$.
\end{itemize}

\section{Vectors and Matrices} \label{sec:not-vecmat}

\subsection{General Euclidean Spaces} 
For the $N$-dimensional Euclidean space we write $\R^N$ ($N\in\N$), whose elements are column vectors. We denote 
\begin{itemize}
	\item vectors and vector-valued functions by bold letters, e.g. $\vu\in\R^N$, 
	\item the zero vector by $\vz\in \R^N$,
	\item the $i$-th component of the vector $\vu$ by $[\vu]_i$ or simply $u_i$,
	\item the scalar product in $\R^N$ by $\cdot$, i.e. $\va\cdot \vb = \sum_{i=1}^{N} a_i b_i$, 
	\item the Euclidean norm of a vector $\vu\in\R^N$ by $|\vu|$, i.e. $|\vu| = \sqrt{\vu\cdot \vu} = \sqrt{\sum_{i=1}^{N} u_i^2} $,
	\item the $i$-th standard basis vector by $\ve_i$, whose $j$-th component is given by the Kronecker delta $[\ve_i]_j=\delta_{ij}$,
	\item the $(N-1)$-sphere by $\sphere^{N-1}:=\left\{\va\in\R^N\,\big|\,|\va|=1\right\}$, 
	\item the open ball in $\R^N$ with center $\vu\in \R^N$ and radius $r>0$ by 
	$$
	B_N(\vu,r):=\left\{\va\in \R^N\,\big|\,|\va-\vu|<r\right\}
	$$ 
	or simply $B(\vu,r)$ (if there is no doubt what $N$ is), 
	\item the interior of a set $A\subset \R^N$ by $\interior{A}$ and the closure by $\closure{A}$,
	\item the boundary of a set $A \subset \R^N$  by $\partial A:= \closure{A} \setminus \interior{A}$,
	\item the Lebesgue measure of a set $A \subset \R^N$ by $|A|$.
\end{itemize}	

\noindent We write $A\subsetcomp B$ if $A$ is bounded and $\closure{A} \subset \interior{B}$. 

\smallskip

\noindent For row vectors we use the notation $\vu^\trans$, where $\vu\in\R^N$. The $i$-th component of the row vector $\vu^\trans$ will be denoted by $[\vu]_i$ (or simply $u_i$), too.

\smallskip

\noindent For the space of real $M\times N$-matrices we write $\R^{M\times N}$. We denote
\begin{itemize}
	\item matrices and matrix-valued functions by special roman characters\footnote{To avoid confusion with the basic number sets, there will not be any matrix in this book denoted by $\N,\Z,\Q,\R,\C$.}, e.g. $\mF\in\R^{M\times N}$, 
	\item the $ij$-th component of the matrix $\mF$ by $[\mF]_{ij}$ or simply $F_{ij}$,
	\item the $j$-th column of the matrix $\mF$ by $\vF_j$, 
	\item the transpose of the matrix $\mF$ by $\mF^\trans$,
	\item the zero matrix by $\mZ$,
	\item the $N\times N$ identity matrix by $\id_N$ or simply $\id$ (if there is no doubt what $N$ is),
	\item the Frobenius product of matrices in $\R^{M \times N}$ by $:$, i.e. $\mA:\mB = \sum_{i=1}^M \sum_{j=1}^N A_{ij} B_{ij}$.
\end{itemize} 

\noindent The tensor product\footnote{One could also define $\otimes:\R^N\times \R^M \to \R^{N\times M}$, where in general $N\neq M$. However this is not needed in this book.} $\otimes:\R^N\times \R^N \to \R^{N\times N}$, $(\va,\vb) \mapsto \va\otimes \vb$ is defined by $[\va\otimes\vb]_{ij} = a_i b_j$.

\smallskip

\noindent Apart from the scalar product in $\R^N$ we will use $\cdot$ for the product of matrices and in particular for the product of a matrix with a vector, e.g. 
$$
	\cdot: \R^{N\times P} \times \R^{P\times M} \to \R^{N\times M},\ (\mA,\mB) \mapsto \mA\cdot \mB, \text{ where }(\mA\cdot \mB)_{ij} = \sum_{k = 1}^P A_{ik} B_{kj} \ec
$$
and in particular\footnote{Since the elements in $\R^N$ are treated as column vectors, one can identify $\R^{N\times 1}$ with $\R^N$.} for $M=1$
$$
	\cdot: \R^{N\times P} \times \R^P \to \R^N,\ (\mA,\vu) \mapsto \mA\cdot \vu, \text{ where }(\mA\cdot \vu)_i = \sum_{k = 1}^P A_{ik} u_k \ed
$$
Sometimes we simply write $\mA\vu$ instead of $\mA\cdot \vu$.

\smallskip

\noindent For quadratic matrices $\mA \in \R^{N\times N}$ we denote
\begin{itemize}
	\item the trace by $\tr (\mA) := \sum_{i = 1}^N A_{ii}$, 
	\item the determinant by $\det (\mA)$.
\end{itemize}
Note that $\tr (\vu \otimes \vu) = |\vu|^2$ for all $\vu\in \R^N$. Furthermore $\mA\mapsto \tr(\mA)$ is a linear map.

\smallskip

\noindent We deal with the following subsets of $\R^{N\times N}$. 
\begin{itemize}
	\item $\sym{N}:=\left\{\mA\in \R^{N\times N}\,\big|\,\mA^\trans = \mA\right\}$ -- the space of all symmetric $N\times N$-matrices.
	\item $\symz{N}:=\left\{\mA\in \sym{N}\,\big|\,\tr (\mA)=0\right\}$ -- the space of all symmetric traceless $N\times N$-matrices.
	\item $\gl{N}:=\left\{\mA\in \R^{N\times N}\,\big|\,\mA\text{ invertible} \right\}$ -- the group of all invertible $N\times N$-matrices (the \emph{general linear group}). If $\mT\in \gl{N}$ then $\mT^{-1}$ denotes the inverse of $\mT$. Note, that 
	$$
		\mT\in \gl{N} \ \Leftrightarrow\  \det(\mT)\neq 0\ed
	$$
	\item $\orth{N}:=\left\{\mA\in \gl{N}\,\big|\,\mA^{-1} = \mA^\trans\right\}$ -- the group of all orthogonal $N\times N$-matrices (the \emph{orthogonal group}). Note, that 
	$$
		\mA\in \orth{N} \ \Leftrightarrow\ \det(\mA)= \pm 1\ed
	$$
\end{itemize} 

\noindent For symmetric matrices $\mA\in \sym{N}$ we use additionally the following notation.
\begin{itemize}
	\item All symmetric matrices are diagonalizable. We donote the largest and smallest eigenvalue by $\lambda_{\max}(\mA)$ and $\lambda_{\min}(\mA)$, respectively. Note, that the mapping $\mA\mapsto \lambda_{\max}(\mA)$ is continuous. Furthermore the trace $\tr(\mA)$ is equal to the sum of all eigenvalues, whereas the determinant $\det(\mA)$ equals the product of all eigenvalues. 
	\item A symmetric matrix $\mA$ is \emph{positive (semi-)definite} if all eigenvalues are positiv (non-negative) or equivalently if $\vu^\trans\mA\vu > 0$ ($\geq 0$) for all $\vu\in \R^N$. Analogously negative (semi-)definiteness is defined.
	\item The matrix norm for symmetric matrices reads $\|\mA\|:=\max_{i=1,...,N} |\lambda_i|$, where the $\lambda_i\in \R$ are the eigenvalues of $\mA$. 
	\item We write $B_{\symz{N}}(\mA,r) := \big\{\mB\in \symz{N}\,\big|\,\|\mB-\mA\|<r\big\}$ for the ball in $\symz{N}$ with center $\mA\in \symz{N}$ and radius $r>0$.
\end{itemize}

\begin{lemma} \label{lemma:not-lmax-tr} 
	Let $\mA\in \sym{N}$. The following statements hold.
	\begin{enumerate}
		\item \label{item:lmax-tr.a} $\tr(\mA)\leq N\lambda_{\max}(\mA)$.
		\item \label{item:lmax-tr.b} If $N\lambda_{\max}(\mA)\leq \tr(\mA)$, then $\mA= \lambda_{\max}(\mA) \id_N =\frac{\tr(\mA)}{N} \id_N$.
	\end{enumerate}
\end{lemma}

\begin{proof}
	\begin{enumerate}
		\item As pointed out above, the sum of all eigenvalues of $\mA$ is equal to the trace $\tr(\mA)$. Hence $\tr(\mA)\leq N\lambda_{\max}(\mA)$. 
		\item The assumption together with \ref{item:lmax-tr.a} imply $N\lambda_{\max}(\mA) = \tr(\mA)$ and therefore all eigenvalues are equal to $\lambda_{\max}(\mA) = \frac{\tr(\mA)}{N}$. This means that there exists $\mT\in\gl{N}$ with 
		$$
		\mA = \mT \lambda_{\max}(\mA) \id \mT^{-1} = \lambda_{\max}(\mA) \id = \frac{\tr(\mA)}{N} \id \ed
		$$ 
	\end{enumerate}
	
\end{proof}

\subsection{The Physical Space and the Space-Time} \label{subsec:not-spacetime}

In this book we deal with two independent variables: Time $t\in\R$ and the spatial variable $\vx$ lying in the physical space $\R^n$, whose dimension is denoted by the natural number $n$. Principally we focus on the case $n\geq 2$.

We use bold letters for vectors in the space-time $\R^{1+n}$, too, e.g. $\veta\in \R^{1+n}$. Their components however are denoted with indices in $\{t,1,...,n\}$ for obvious reasons, for example $\veta = (\eta_t,\eta_1,...,\eta_n)^\trans$. Furthermore we write $\veta_\vx$ for the ``spatial'' component of $\veta\in \R^{1+n}$, i.e. $\veta_\vx:=(\eta_1,...,\eta_n)^\trans\in \R^n$. 

The components of matrices in $\R^{(1+n)\times (1+n)}$ are denoted similarly, e.g.
\begin{equation*}
	\mA = \left(\begin{array}{cccc}
		A_{tt} & A_{t1} & \cdots & A_{tn} \\
		A_{1t} & A_{11} & \cdots & A_{1n} \\
		\vdots & \vdots &   & \vdots \\
		A_{nt} & A_{n1} & \cdots & A_{nn}
	\end{array} \right) \in \R^{(1+n)\times (1+n)} \ed
\end{equation*}
Analogously the standard basis vectors of the space-time $\R^{1+n}$ are denoted by $\ve_t,\ve_1,...,\ve_n$.

Functions with values in the space-time $\R^{1+n}$ or in $\R^{(1+n)\times (1+n)}$ are treated similarly. In particular for $\vf:\R^{1+n}\to \R^{1+n}$, $(t,\vx)\mapsto \vf(t,\vx)$ we write $f_t$ for the $t$-th component of $\vf$ whereas the partial derivative of $\vf$ with respect to $t$ is denoted by $\partial_t \vf$, see Section~\ref{sec:not-functions}.

\begin{defn}
	\begin{itemize}
		\item A \emph{(spatial) domain} $\Omega\subset\R^n$ is an open and connected subset of the physical space $\R^n$.
		\item A \emph{space-time domain} $\Gamma\subset\R^{1+n}$ is an open and connected subset of the space-time $\R^{1+n}$.
	\end{itemize}
\end{defn}

\subsection{Phase Space}
The dependent variable $\vU$ - the \emph{state vector}, see Chapter~\ref{chap:conslaws} - lies in the phase space $\R^m$, whose dimension is denoted by the natural number $m$.

\section{Sequences}

For a sequence in a set $X$, i.e. a map $\N \to X$, we write as usual $(x_n)_{n\in \N} \subset X$.

\section{Functions} \label{sec:not-functions} 

Most of the functions which appear in this book belong to either one of the following types:
\begin{itemize}
	\item Functions of time $t\in\R$ and space $\vx\in\R^n$;
	\item Functions of the state vector $\vU\in\R^m$.
\end{itemize} 

In this book there appear also mappings $f:X\to \R$ where $(X,d)$ is a metric space. Such mappings are considered in Sections \ref{sec:not-lsc} and \ref{sec:not-baire}.

\subsection{Basic Notions}

\begin{itemize} 
	\item If two functions are equal we use the symbol $\equiv$, i.e. we write for example $f\equiv g$ if $f(t,\vx)=g(t,\vx)$ for all arguments $(t,\vx)\in \R^{1+n}$. 
	\item The support of the function $f$ is denoted by $\supp(f)$.
	\item For the characteristic function with respect to the set $S\subset \R^{1+n}$ we use the notation $\charf_S$, i.e.
	$$
	\charf_S : \R^{1+n} \to \R, \quad \charf_S(t,\vx) := \left\{ \begin{array}{ll}
		1 & \text{ if }(t,\vx)\in S \ec \\
		0 & \text{ else}\ed
	\end{array} \right. 
	$$
\end{itemize}

\subsection{Differential Operators}

For functions $f:A \to \R$, where $A\subset \R$, the derivative is denoted as usual with $f'$, the second derivative with $f''$, and so on. For conveniece we write $f^{(k)}$ for the $k$-th derivative.

\subsubsection{Functions of Time and Space} 
Let $\Gamma\subset \R^{1+n}$. For scalar functions $f:\Gamma \to \R$, $(t,\vx)\mapsto f(t,\vx)$ we work with the following differential operators:
\begin{itemize}
	\item The partial derivative with respect to time, denoted by $\partial_t f(t,\vx) := \frac{\partial}{\partial t} f(t,\vx)$.
	\item The partial derivative with respect to $x_i$ (the $i$-th component of $\vx$), denoted by \\$\partial_i f(t,\vx) := \frac{\partial}{\partial x_i} f(t,\vx).$ 
	\item Higher order partial derivatives, denoted by $\partial_{i_1 \cdots i_k}^k f(t,\vx) := \partial_{i_1}\cdots \partial _{i_k} f(t,\vx)$, where $k\in\N$ and $i_j\in\{t,1,...,n\}$ for each $j=1,...,k$. E.g. $\parfour{t}{t}{1}{2} f(t,\vx) = \partial_t \partial_t \partial_1 \partial_2 f(t,\vx)$.
	\item The gradient with respect to $\vx$, denoted by $\Grad f(t,\vx) := \big(\partial_1 f(t,\vx),\ldots,\partial_n f(t,\vx)\big)$ and taking values in $\R^{1\times n}$. In other words in this book the gradient of a scalar field is a row vector.
	\item The Laplacian with respect to $\vx$, denoted by $\Lap f(t,\vx) := \sum_{i = 1}^n \partial_i \partial_i f(t,\vx)$ and taking values in $\R$.
\end{itemize}

\noindent For vector-valued functions $\vf:\Gamma \to \R^M$, $(t,\vx)\mapsto \vf(t,\vx)$, the following differential operators appear in this book:
\begin{itemize} 
	\item The partial derivative with respect to time, denoted by $\partial_t \vf(t,\vx) := \frac{\partial}{\partial t} \vf(t,\vx)$ and taking values in $\R^M$.
	\item The gradient with respect to $\vx$, denoted by 
	$$
	\Grad \vf(t,\vx) := \left(\begin{array}{ccc}
		\partial_1 f_1(t,\vx) & \ldots & \partial_n f_1(t,\vx) \\
		\vdots & & \vdots \\
		\partial_1 f_M(t,\vx) & \ldots & \partial_n f_M(t,\vx) 
	\end{array}\right)
	$$ 
	and taking values in $\R^{M\times n}$. 
	\item The Laplacian with respect to $\vx$, denoted by 
	$$
	\Lap \vf(t,\vx) := \sum_{i = 1}^n \partial_i \partial_i \vf(t,\vx) = \left(\begin{array}{c}
		\sum_{i = 1}^n \partial_i \partial_i f_1(t,\vx) \\
		\vdots \\
		\sum_{i = 1}^n \partial_i \partial_i f_M(t,\vx) 
	\end{array}\right)
	$$ 
	and taking values in $\R^M$.
	\item If $M=n$, the divergence with respect to $\vx$, denoted by $\Div \vf(t,\vx) := \sum_{i=1}^n \partial_i f_i(t,\vx)$.
	\item If $M=1+n$, the divergence with respect to $(t,\vx)$, denoted by $$\Divtx \vf(t,\vx) := \partial_t f_t(t,\vx) + \Div \vf_\vx (t,\vx) = \partial_t f_t(t,\vx) + \sum_{i=1}^n \partial_i f_i (t,\vx) \ed$$
\end{itemize}

\noindent For matrix-valued functions $\mF:\Gamma \to \R^{M\times n}$, $(t,\vx)\mapsto \mF(t,\vx)$ we deal with the following differential operator:
\begin{itemize} 
	\item The row-wise divergence with respect to $\vx$, denoted by $\Div \mF :\R^{1+n}\to \R^M$ with $j$-th component $[\Div\mF]_j (t,\vx):= \sum_{i=1}^n \partial_i F_{ji}(t,\vx)$.
\end{itemize}

\noindent For matrix-valued functions $\mF:\Gamma \to \R^{M\times (1+n)}$, $(t,\vx)\mapsto \mF(t,\vx)$ the following differential operator occurs in this book:
\begin{itemize} 
	\item The row-wise divergence with respect to $(t,\vx)$, denoted by $\Divtx \mF :\R^{1+n}\to \R^M$ with $j$-th component $[\Divtx \mF]_j (t,\vx):= \partial_t F_{jt}(t,\vx) + \sum_{i=1}^n \partial_i F_{ji}(t,\vx)$.
\end{itemize}

\noindent Matrix-valued functions of the form $\mF:\Gamma \to \R^{M\times N}$ with $N\notin \{n,1+n\}$ do not appear in this book.

\subsubsection{Functions of the State Vector}

Let $\sO\subset \R^m$. For scalar functions $f:\sO \to \R$, $\vU\mapsto f(\vU)$ we deal with the following differential operators:
\begin{itemize}
	\item The partial derivative with respect to $U_i$, denoted by $\partial_{U_i} f(\vU) := \frac{\partial}{\partial U_i} f(\vU)$.
	\item The gradient with respect to $\vU$, denoted by $\Grad_{\vU} f(\vU) := \big(\partial_{U_1} f(\vU),\ldots,\partial_{U_m} f(\vU)\big)$ and taking values in $\R^{1\times m}$.
	\item The Hessian with respect to $\vU$, denoted by 
	$$
		\Hess_\vU f(\vU) := \left(\begin{array}{ccc}
			\partial_{U_1}\partial_{U_1} f(\vU) & \ldots & \partial_{U_1}\partial_{U_m} f(\vU) \\
			\vdots & & \vdots \\
			\partial_{U_m}\partial_{U_1} f(\vU) & \ldots & \partial_{U_m}\partial_{U_m} f(\vU) 
		\end{array}\right)
	$$ 
	and taking values in $\R^{m\times m}$. Note that $\Hess_\vU f$ is symmetric as soon as $f$ is twice continuously differentiable according to Schwarz's theorem.
\end{itemize}

\noindent For vector-valued functions $\vf:\sO \to \R^m$, $\vU\mapsto \vf(\vU)$ we work with the following differential operators:
\begin{itemize} 
	\item The partial derivative with respect to $U_i$, denoted by $\partial_{U_i} \vf(\vU) := \frac{\partial}{\partial U_i} \vf(\vU)$ and taking values in $\R^m$.
	\item The gradient with respect to $\vU$, denoted by 
	$$
		\Grad_\vU \vf(\vU) := \left(\begin{array}{ccc}
			\partial_{U_1} f_1(\vU) & \ldots & \partial_{U_m} f_1(\vU) \\
			\vdots & & \vdots \\
			\partial_{U_1} f_m(\vU) & \ldots & \partial_{U_m} f_m(\vU) 
		\end{array}\right)
	$$ 
	and taking values in $\R^{m\times m}$. 
\end{itemize}

\noindent Functions of the space vector $\vU$ with values in $\R^M$ with $M\notin\{1,m\}$ do not appear in this book.

\subsection{Function Spaces} 

Let us summarize the function spaces which appear in this book. As an example we consider functions defined on $\Gamma\subset \R^{1+n}$, an open subset of the space-time. Of course for functions defined on an open subset of $\R^N$ with any $N\in \N$ we use the same notation.

\begin{itemize}
	\item $C(\Gamma)$ -- the set of continuous functions $f:\Gamma \to \R$.
	\item $C^k(\Gamma)$, where $k\in \N$ -- the set of $k$-times continuously differentiable functions $f:\Gamma \to \R$.
	\item $C(\closure{\Gamma}) = \{f\in C(\Gamma)\,|\,f \text{ uniformly continuous on bounded subsets of }\Gamma\}$.
	\item $C^k(\closure{\Gamma}) = \{f\in C^k(\Gamma)\,|\,\text{all partial derivatives of } f \text{ up to }k\text{-th order}$\\ 
	\hphantom{$C^k(\closure{\Gamma}) = \{f\in C^k(\Gamma)\,|\,$}$\text{are uniformly continuous on bounded subsets of }\Gamma\}$. 
	
	\smallskip
	
	If $f\in C(\closure{\Gamma})$, then $f$ can be extended continuously to $\closure{\Gamma}$. Similarly if $f\in C^k(\closure{\Gamma})$, then each partial derivative up to $k$-th order can be extended continuously to $\closure{\Gamma}$. 
	\item $C^\infty(\Gamma) = \bigcap_{k=0}^\infty C^k(\Gamma)$.
	\item $C^\infty(\closure{\Gamma}) = \bigcap_{k=0}^\infty C^k(\closure{\Gamma})$.
	\item $\Cc(\Gamma) = \{f\in C^\infty(\Gamma)\,|\, \supp(f) \text{ is a compact subset of }\Gamma\}$. 
	
	\smallskip
	
	A function $f\in \Cc(\Gamma)$ can be extended to a function in $\Cc(\R^{1+n})$ by setting it equal to zero outside $\Gamma$. We still write $f$ for the extension. This fact is tacitly used several times in this book.
	\item $\Cc(\closure{\Gamma}) = \{f\in C^\infty(\closure{\Gamma})\,|\, \supp(f) \text{ is a compact subset of }\closure{\Gamma}\}$. 
	
	\smallskip
	
	Note that a function $f\in \Cc(\closure{\Gamma})$ does not need to vanish on the boundary $\partial \Gamma$. We will in particular consider the set $\Cc\big([0,T)\times \closure{\Omega}\big)$ with $T>0$ and an open set $\Omega\subset \R^n$. 
	Such functions may not vanish at $t=0$ as well as on the boundary of $\Omega$, whereas they do vanish at $t=T$.
	\item $L^p(\Gamma)$, where $1\leq p < \infty$ -- the set of Lebesgue measurable functions $f:\Gamma\to\R$ with 
	$$
		\|f\|_{L^p(\Gamma)} := \left(\iint_\Gamma |f(t,\vx)|^p \dx\dt \right)^{1/p} \ <\ \infty \ed
	$$
	\item $L^\infty(\Gamma)$ -- the set of Lebesgue measurable functions $f:\Gamma\to\R$ with 
	$$
		\|f\|_{L^\infty(\Gamma)} := \mathop{\text{ess}\sup}\limits_{(t,\vx)\in\Gamma} |f(t,\vx)| \ <\ \infty \ed
	$$
	\item $L^p_\loc(\Gamma) = \{f:\Gamma\to \R\,|\, f\in L^p(\Gamma_0) \text{ for all }\Gamma_0\subsetcomp \Gamma\}$. 
\end{itemize}

We say that a property holds \emph{almost everywhere} on a set $\Gamma\subset\R^{1+n}$ if there exists a set of zero Lebesgue measure $N\subset \R^{1+n}$ such that the property holds for all $(t,\vx)\in \Gamma\setminus N$. In this case we use the abbreviation \emph{a.e.}, i.e. we write that the property holds a.e. on $\Gamma$. Equivalently we say that the property holds for \emph{almost every (a.e.)} $(t,\vx)\in\Gamma$.

To be precise the Lebesgue spaces $L^p(\Gamma)$, $L^p_\loc(\Gamma)$ contain classes of functions rather than functions, where the corresponding equivalence relation relates two functions that coincide a.e. on $\Gamma$. Moreover if $\Gamma$ is bounded, then $L^p(\Gamma) = L^p_\loc(\Gamma)$ for all $1\leq p \leq \infty$.

The spaces $C(\Gamma;\R^M)$, $L^p(\Gamma;\R^M)$, ...\@ consist of all vector-valued functions $\vf:\Gamma\to \R^M$ whose components $f_i\in C(\Gamma),\ L^p(\Gamma)$ etc.\@ for all $i=1,...,M$. If $A\subset \R^M$ and $\vf\in C(\Gamma;\R^M),\ L^p(\Gamma;\R^M)$ etc.\@ with $\vf(\Gamma)\subset A$, we write $\vf\in C(\Gamma;A),\ L^p(\Gamma;A)$ etc.

We finish this subsection with the following lemma.

\begin{lemma} \label{lemma:not-local-property} 
	Let $\Gamma\subset \R^{1+n}$ open (not necessarily bounded) and $f \in L^\infty(\Gamma)$ with 
	$$
		\iint_{\Gamma_0} f(t,\vx) \dx\dt \geq 0 \qquad \text{ for all open and bounded subsets }\Gamma_0\subset \Gamma\ed
	$$
	Then $f(t,\vx) \geq 0 $ for a.e. $(t,\vx)\in \Gamma$.
\end{lemma}

\begin{proof} 
	Assume the contrary, i.e.
	\begin{align*}
		0 &< \left|\Big\{(t,\vx)\in \Gamma\,\Big|\, f(t,\vx) < 0\Big\} \right| \\
		&= \left|  \left\{(t,\vx)\in \Gamma\,\Big|\, f(t,\vx) < -1\right\} \cup \bigcup_{k=2}^\infty \left\{(t,\vx)\in \Gamma\,\Big|\, -\frac{1}{k-1} \leq f(t,\vx) < -\frac{1}{k}\right\} \right| \\
		&= \left|  \left\{(t,\vx)\in \Gamma\,\Big|\, f(t,\vx) < -1\right\}\right| + \sum_{k=2}^\infty \left|\left\{(t,\vx)\in \Gamma\,\Big|\, -\frac{1}{k-1} \leq f(t,\vx) < -\frac{1}{k}\right\} \right|\ed 
	\end{align*}
	Hence there exists $k\in\N$ such that $\big|\big\{(t,\vx)\in \Gamma\,\big|\, f(t,\vx) < -\frac{1}{k}\big\} \big|>0$. Let us write
	$$
		L:=\left\{(t,\vx)\in \Gamma\,\Big|\, f(t,\vx) < -\frac{1}{k}\right\} 
	$$
	for convenience. We may assume without loss of generality that $L$ is bounded. Indeed if this was not the case, then there must be at least a bounded subset of $L$ which still has positive measure since we can write
	\begin{align*}
		0 &< \left| L \right| \\
		&= \left| \Big( L \cap B_{1+n}(\vz,1) \Big) \cup \bigcup_{\ell=2}^\infty \Big( L \cap \big(B_{1+n}(\vz,\ell)\setminus B_{1+n}(\vz,\ell-1)\big) \Big) \right| \\
		&= \Big| L \cap B_{1+n}(\vz,1) \Big| + \sum_{\ell=2}^\infty \Big| L \cap \big(B_{1+n}(\vz,\ell)\setminus B_{1+n}(\vz,\ell-1)\big) \Big| \ed
	\end{align*}
	
	Now for each $\ep>0$ we find an open set $\Gamma_0\supset L$ such that $|\Gamma_0\setminus L| \leq \ep$, see e.g. \name{Elstrodt} \cite[Satz 7.1]{Elstrodt}. Let us set $\ep:= \frac{|L|}{2k\|f\|_{L^\infty}}$. Note that $\Gamma_0$ is bounded since $|\Gamma_0\setminus L| \leq \ep$ and $L$ is bounded. This leads to 
	\begin{align*}
		0 &\leq \iint_{\Gamma_0} f(t,\vx) \dx\dt \\
		&= \iint_{\Gamma_0\setminus L} f(t,\vx) \dx\dt + \iint_L f(t,\vx) \dx\dt \\
		&\leq \ep \|f\|_{L^\infty} - \frac{1}{k} |L| \\
		&= - \frac{|L|}{2k} \\
		&< 0\ec
	\end{align*}
	a contradiction. 
\end{proof}

\subsection{Integrability Conditions}

\begin{prop} \label{prop:not-int-cond}
	Let $\sO\subset \R^m$ open and simply connected, and $\vF \in C^1 (\sO;\R^m)$. Then there exists $\phi\in C^2(\sO)$ such that $\vF = \Grad_\vU \phi$ if and only if the integrability conditions 
	$$
	\partial_{U_i} F_j = \partial_{U_j} F_i\qquad \text{ for all }i,j=1,...,m
	$$ 
	hold on $\sO$.
\end{prop} 

Proposition~\ref{prop:not-int-cond} is standard. Its proof can be found in fundamental textbooks, e.g. \name{Amann}-\name{Escher} \cite[VIII 4.10 (a)]{AmaEsc2}.

\subsection{Boundary Integrals and the Divergence Theorem} 


If the boundary of a domain is sufficiently regular, one can define boundary integrals. We  briefly summarize some material on this topic. The content of this section can be found in many textbooks, e.g. \name{Feireisl}-\name{Novotn{\'y}} \cite[Section 4 in the chapter ``Notation, Definitions, and Function Spaces'']{FeiNov} and references therein.

\begin{defn} 
	A domain $\Omega\subset\R^n$ is called \emph{Lipschitz} if for each point $\vx\in\partial \Omega$ there exists $r>0$ and a Lipschitz mapping $\gamma:\R^{n-1}\to \R$ such that (after rotating and relabeling the coordinate axes if necessary) we have 
	\begin{align*}
		\Omega \cap B_n(\vx,r) &= \left\{\vy\in \R^n\,\big|\,\gamma(y_1,...,y_{n-1}) < y_n\right\} \cap B_n(\vx,r) \ec \\
		\partial\Omega \cap B_n(\vx,r) &= \left\{\vy\in \R^n\,\big|\,\gamma(y_1,...,y_{n-1}) = y_n\right\} \cap B_n(\vx,r) \ed
	\end{align*}
	A \emph{Lipschitz} space-time domain $\Gamma\subset\R^{1+n}$ is defined analogously.
\end{defn}

For Lipschitz domains $\Omega$ as well as for Lipschitz space-time domains $\Gamma$ we can define a surface measure (the $(n-1)$-dimensional \emph{Hausdorff measure}) and hence an integral on the boundary. We write
$$
\int_{\partial \Omega} f \dS_\vx \qquad \text{ or } \qquad \int_{\partial \Gamma} f \dS_{t,\vx}
$$
for those boundary integrals. 

Furthermore for Lipschitz (space-time) domains $\Omega$ (or $\Gamma$) there exists the \emph{outward pointing normal vector} $\vn(\vx)\in\R^n$ (or $\vn(t,\vx)\in\R^{1+n}$) for a.e. $\vx\in\partial \Omega$ (or $(t,\vx)\in \partial \Gamma$), where here ``a.e.'' is meant with respect to the surface measure on $\partial \Omega$ (or $\partial \Gamma$).

\renewcommand{\arraystretch}{1.2}

\begin{ex} 
	We call a space-time domain of the form $(t_0,t_1)\times \Omega \subset\R^{1+n}$, where $\Omega\subset \R^n$ is a bounded domain, a \emph{cylindrical space-time domain}, see Figure~\ref{fig:cylinder} for an example. Note that $\partial \Gamma$ consists of 3 sets 
	$$
		\partial \Gamma = \big(\{t_0\}\times \Omega \big) \cup \big( \{t_1\}\times \Omega \big) \cup \big( [t_0,t_1]\times \partial\Omega \big) \ed
	$$
	Furthermore $\Gamma$ is Lipschitz in $\R^{1+n}$ if $\Omega$ is Lipschitz in $\R^n$ and in this case the outward pointing normal vector $\vn(t,\vx)\in \R^{1+n}$ is given by 
	$$
		\vn(t,\vx) = \left\{\begin{array}{ll} 
			\bigg(\begin{array}{c} -1 \\ \vz \end{array} \bigg) & \text{ for }(t,\vx)\in \{t_0\}\times \Omega \ec \\
			\bigg(\begin{array}{c} 1 \\ \vz \end{array} \bigg) & \text{ for }(t,\vx)\in \{t_1\}\times \Omega \ec \\
			\bigg(\begin{array}{c} 0 \\ \vn_\Omega(\vx) \end{array} \bigg) & \text{ for }(t,\vx)\in (t_0,t_1)\times \partial\Omega \ec
		\end{array} \right.
	$$
	where $\vn_\Omega(\vx)\in\R^n$ is the normal vector which corresponds to $\Omega$. 
	
	Moreover for $f:\R^{1+n} \to \R$ and $\vf:\R^{1+n} \to \R^{1+n}$ we have  
	\begin{align}
		\int_{\partial\Gamma} f(t,\vx) \dS_{t,\vx} &= \int_\Omega f(t_1,\vx) \dx + \int_\Omega f(t_0,\vx) \dx + \int_{t_0}^{t_1} \int_{\partial\Omega} f(t,\vx) \dS_\vx \dt \ec \\ 
		\int_{\partial\Gamma} \vf(t,\vx)\cdot \vn(t,\vx) \dS_{t,\vx} &= \int_\Omega f_t(t_1,\vx) \dx - \int_\Omega f_t(t_0,\vx) \dx + \int_{t_0}^{t_1} \int_{\partial\Omega} \vf_\vx(t,\vx) \cdot \vn_\Omega(\vx) \dS_\vx \dt \ec \label{eq:not-boundary-integral-cylinder}
	\end{align}
	provided the integrals exist.
\end{ex}

\renewcommand{\arraystretch}{1.0}

\begin{figure}[tbp]
	\centering
	\includegraphics[width=0.45\textwidth]{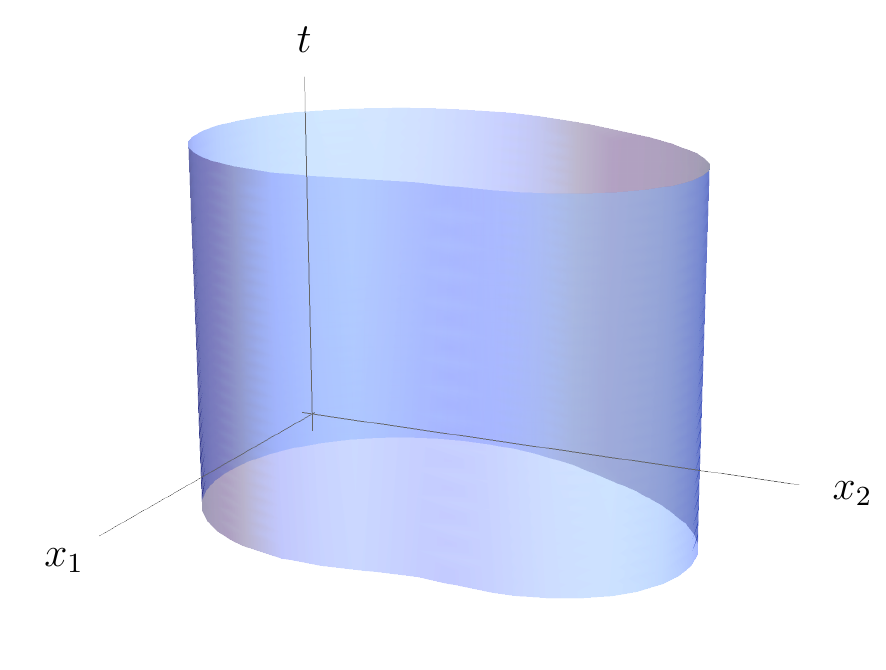}  
	\caption{An example of a cylindrical space-time domain where $n=2$.} 
	\label{fig:cylinder}
\end{figure} 

For bounded Lipschitz domains the following proposition is true.

\begin{prop} \label{prop:not-divergence} \emph{(Divergence Theorem)}
	Let the space-time domain $\Gamma\subset\R^{1+n}$ be bounded and Lipschitz and $\vf \in C^1(\closure{\Gamma};\R^{1+n})$. Then 
	\begin{equation} \label{eq:not-div-thm}
		\iint_\Gamma \Divtx \vf(t,\vx) \dx\dt = \int_{\partial \Gamma} \vf(t,\vx) \cdot \vn(t,\vx) \dS_{t,\vx} \ed
	\end{equation}
\end{prop} 

Note again that $\Divtx \vf = \partial_t f_t + \Div \vf_\vx$. Furthermore the right-hand side of \eqref{eq:not-div-thm} can be written as in \eqref{eq:not-boundary-integral-cylinder} if $\Gamma$ is cylindrical. 

For the proof of Proposition~\ref{prop:not-divergence} we refer to standard textbooks, e.g. \name{Ne{\v c}as} \cite[Theorem 1.1 in Section 3.1.2]{Necas} for a more general statement.


\subsection{Mollifiers} \label{subsec:not-mollification} 

\begin{defn}(See e.g. \name{Evans} \cite[Appendix C.5]{Evans}.)
	\begin{enumerate}
		\item The \emph{standard mollifier} $\phi\in C^\infty(\R;\R^+_0)$ is defined by 
		\begin{equation*}
			\phi(t) := \left\{ \begin{array}{cl}
				C \exp \left( \frac{1}{t^2 -1} \right) & \text{ for } |t|<1\ec \\
				0 & \text{ for }|t|\geq 1 \ec
			\end{array}\right.
		\end{equation*}
		where $C>0$ is such that 
		$$
			\int_{\R} \phi(t) \dt = 1\ed
		$$
		For $\delta>0$, define $\phi_\delta\in C^\infty(\R;\R^+_0)$ by 
		$$
			\phi_\delta (t) := \frac{1}{\delta} \phi\left(\frac{t}{\delta}\right) \ed
		$$
		It is a simple observation that 
		\begin{itemize}
			\item $\phi_\delta(t) = 0$ for $|t|\geq \delta$ and
			\item $\int_{\R} \phi_\delta(t) \dt = 1$.
		\end{itemize}
		\item For $f\in L^1_\loc (\R)$ the \emph{($\delta$-)mollification} $f_\delta$ of $f$ is defined as the convolution $f_\delta := \phi_\delta \ast f$, i.e. 
		$$
			f_\delta (t) = \int_\R \phi_\delta(t-\tau) f(\tau) \dtau = \int_{-\delta}^\delta \phi_\delta(\tau) f(t - \tau) \dtau \qquad \text{ for all }t \in \R \ed
		$$
	\end{enumerate}
\end{defn}	

\noindent Note that $f_\delta \in C^\infty (\R)$, see \cite[Theorem 7 in Appendix C.5]{Evans}.

\subsection{Periodic Functions} \label{subsec:not-periodic-functions}

At some point in this book we will deal with 1-periodic functions of the form $f:\R\to \R$,
\begin{equation} \label{eq:periodic-function}
f(t) = \left\{ \begin{array}{ll}
	\tau - 1 & \text{ if } t\in [0,\tau) + \Z \ec\\
	\tau & \text{ if } t\in [\tau,1) + \Z \ec
\end{array} \right.
\end{equation}
where $\tau\in (0,1)$. An example of such a function is depicted in Figure~\ref{fig:periodic}. Note that 
$$
\int_{0}^{1} f(t) \dt = \int_0^\tau (\tau-1) \dt + \int_\tau^1 \tau \dt = \tau (\tau-1) + (1-\tau) \tau = 0 \ec
$$
i.e. $f$ has zero mean. Such functions shall be mollified.

\begin{figure}[tbp]
	\vspace{0.5cm}
	\centering
	\includegraphics[width=0.5\textwidth]{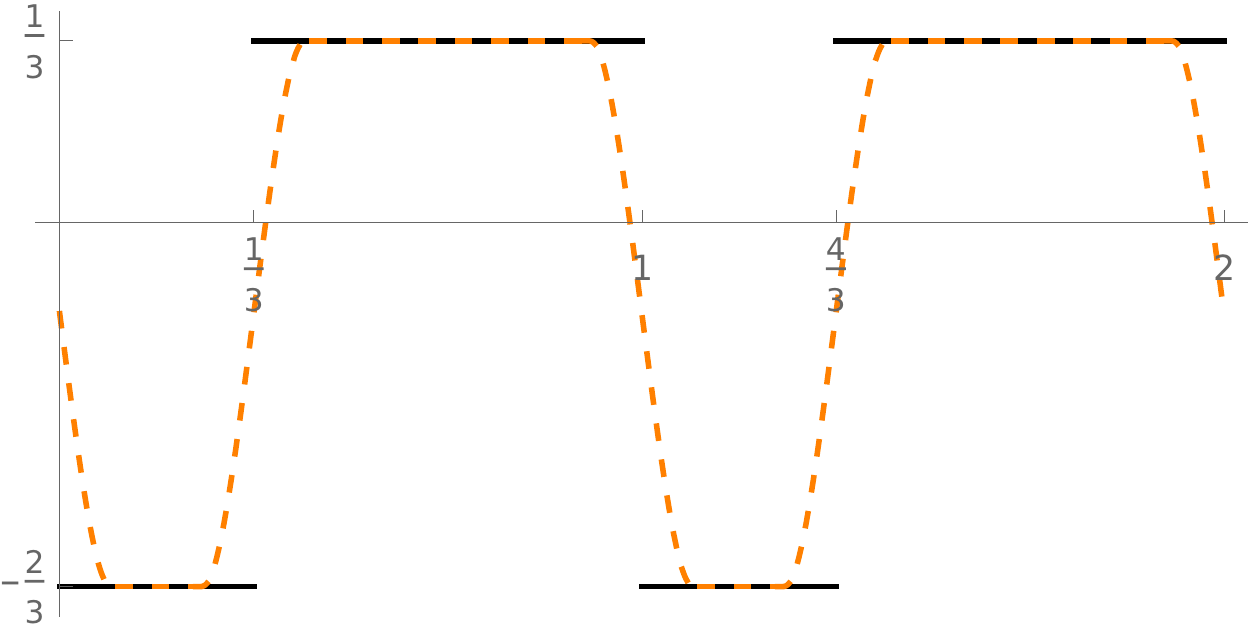}  
	\vspace{0.1cm}
	\caption{The function $f$ as in \eqref{eq:periodic-function} with $\tau=\frac{1}{3}$ (black) and its mollification $f_{\frac{1}{10}}$ (orange).} 
	\label{fig:periodic}
\end{figure} 

\begin{lemma} \label{lemma:periodic-mollification}
	Let $f$ as in \eqref{eq:periodic-function}. Its $\delta$-mollification $f_\delta$ (see Subsection~\ref{subsec:not-mollification}) has the following properties:
	\begin{itemize}
		\item $f_\delta$ is 1-periodic;
		\item $f_\delta$ has zero mean; 
		\item $f_\delta$ takes values in $[\tau-1,\tau]$;
		\item If $\delta <\min\left\{ \frac{\tau}{2},\frac{1-\tau }{2} \right\}$, then 
		\begin{align*}
			f_\delta (t) &= \tau - 1 & &\hspace{-2.5cm}\text{ for all }t\in [\delta, \tau- \delta] + \Z \ec \\
			f_\delta (t) &= \tau & &\hspace{-2.5cm}\text{ for all }t\in [\tau+\delta , 1-\delta] + \Z \ed
		\end{align*}
	\end{itemize} 
\end{lemma}

\begin{proof} 
	The simple proof is left to the reader. See also Figure~\ref{fig:periodic}. 
\end{proof} 

\begin{lemma} \label{lemma:not-periodic-primitive} 
	Let $f\in C^\infty(\R)$ be 1-periodic with zero mean. Then for each $k\in \N$ there exists $h\in C^\infty(\R)$ with the following properties:
	\begin{itemize}
		\item $h^{(i)}$ is 1-periodic for all $i=0,...,k$;
		\item $h^{(i)}$ has zero mean for all $i=0,...,k$;
		\item $h^{(k)} = f$.
	\end{itemize} 
	Note that the first two properties imply that all derivatives of $h$ are bounded.
\end{lemma}

\begin{proof} 
	It suffices to prove the claim for $k=1$ due to iteration. Define $h$ by 
	\begin{equation} \label{eq:not-periodic-primitive} 
		h(t) := \int_0^t f(s) \ds - \int_0^1 \int_0^\sigma f(s) \ds \dsigma \ed
	\end{equation}
	Note that these integrals exist because $f$ is continuous. Simple computations show that $h$ is 1-periodic, has zero mean and $h'=f$. The latter property implies that $h\in C^\infty(\R)$ since $f\in C^\infty(\R)$.
\end{proof}

\begin{rem} 
	For Lemma~\ref{lemma:not-periodic-primitive} it is crucial that $f$ has zero mean. Otherwise $h$ as defined in \eqref{eq:not-periodic-primitive} will not be bounded.
\end{rem}

\section{Convexity} \label{sec:not-convexity}

\subsection{Convex Sets and Convex Hulls}

\begin{defn} 
	\begin{itemize}
		\item The \emph{line segment} $[\vp,\vq]\subset \R^M$ between two points $\vp,\vq\in\R^M$ is defined by 
		$$
		[\vp,\vq] := \left\{ \vs\in \R^M \,\Big|\, \exists \tau\in[0,1]\text{ such that } \vs = \tau\vp + (1-\tau) \vq \right\}\ed
		$$
		\item A set $S\subset\R^M$ is called \emph{convex} if $\forall \vp,\vq \in S : [\vp,\vq]\subseteq S$.
		\item Let $S\subset \R^M$ be closed and convex. A point $\vs\in S$ is called \emph{exteme point} if there are no two points $\vp,\vq\in S\setminus\{\vs\}$ with $\vs\in[\vp,\vq]$. The set of all extreme points of $S$ is denoted by $\ext(S)$.
		\item The \emph{convex hull} $K^\co$ of a set $K\subset\R^M$ is the smallest convex set which contains $K$. 
	\end{itemize}
\end{defn} 

\begin{prop} \label{prop:convhull=convcombis}
	The convex hull $K^\co$ of a set $K\subset \R^M$
	\begin{enumerate}
		\item \label{item:convex-hull.a} as defined above is well-defined due to the fact that the intersection of convex sets is convex.
		\item \label{item:convex-hull.b} is equal to the set of all convex combinations of points in $K$, i.e.
		\begin{align}
			K^\co := \Bigg\{ \vp \in \R^M\,\Big|\,\exists &N\in\N, \exists(\tau_i,\vp_i)\in\R^+\times K\text{ for all }i=1,...,N\text{ such that } \label{eq:convex-hull} \\
			& \bullet \ \sum_{i=1}^N \tau_i = 1  \text{ and } \notag \\
			& \bullet \ \vp=\sum_{i=1}^N \tau_i \vp_i  \qquad \qquad \Bigg\} \ed \notag
		\end{align} 
	\end{enumerate}
\end{prop}

\begin{proof}
	\begin{enumerate}
		\item First of all note, that $\R^M$ is convex and contains $K$. \emph{A} smallest convex set which contains $K$ can hence be constructed by removing sets from $\R^M$ until there is no convex subset which contains $K$. This yields existence of \emph{a} smallest convex set which contains $K$. Such a set is unique, which can be shown as follows. Assume there would be two different smallest convex sets which contain $K$. The intersection of those two sets is again convex and contains $K$. But this is a contradiction to the smallness assumption of the two sets. Therefore the convex hull is well-defined. 
		\item We will not present all details here because the proof is quite simple\footnote{In addition to what is presented here, the claim can be easily deduced from Proposition~\ref{prop:KLambda=U}. Indeed setting $\Lambda=\R^M$ yields that the corresponding $\Lambda$-convex hull $K^\Lambda$ coincides with the convex hull $K^\co$. Furthermore each family of pairs $\big\{(\tau_i,\vp_i)\big\}_{i=1,...,N}$, where $(\tau_i,\vp_i)\in\R^+\times K$, satisfies the $H_N$-condition (see Definition~\ref{defn:hn}) if and only if $\sum_{i=1}^N\tau_i=1$, because of the fact that we have set $\Lambda=\R^M$. Hence $\sU$ as defined in \eqref{eq:U} coincides with the right-hand side of \eqref{eq:convex-hull}. Proposition~\ref{prop:KLambda=U} finally yields the claim.}. On the one hand one can simply show that the right-hand side of \eqref{eq:convex-hull} is convex and contains $K$. Hence \eqref{eq:convex-hull} holds with a ``$\subset$'' sign instead of ``$=$''. The other inclusion can be proven by induction over $N$.
	\end{enumerate}
\end{proof}

%

The following proposition will serve as a tool to compute the convex hull. For its proof we refer to the literature.

\begin{prop}\emph{(Minkowski's Theorem, see e.g. \name{Br{\o}ndsted} \cite[Theorem 5.10]{Brondsted})} \label{prop:minkowski}
	Let $S\subseteq\R^M$ be a compact convex set and let $K\subseteq S$. Then
	\begin{equation*}
	S=K^{\co} \ \Leftrightarrow\  \ext(S)\subset K \ed
	\end{equation*}
\end{prop} 

If a point lies in the interior of the convex hull of a set $K$, then it lies in the interior of the convex polytope spanned by finitely many points in $K$. This is the content of the following proposition.

\begin{prop} \label{prop:caratheodory-inner-points} 
	Let $K\subset \R^M$ and $\vp \in \interior{(K^\co)}$. Then there exists $N\in\N$ and points $\vp_1,...,\vp_N\in K$ such that $\vp\in \interior{(\{\vp_1,...,\vp_N\}^\co)}$.
\end{prop}

In order to prove Proposition~\ref{prop:caratheodory-inner-points}, we consider the following lemma.

\begin{lemma} \label{lemma:conv-polytopes}
	Let $N_1,N_2\in \N$ and $\vq_1,...,\vq_{N_1},\vp_1,...,\vp_{N_2} \in K \subset \R^M$ such that $\vq_1\in \{\vp_1,...,\vp_{N_2}\}^\co$. Then 
	$$
		\{\vq_1,...,\vq_{N_1} \}^\co \subset \Big( \{\vp_1,...,\vp_{N_2}\} \cup \{\vq_2,...,\vq_{N_1} \} \Big)^\co\ed
	$$
\end{lemma}

\begin{proof}
	Since $\vq_1\in \{\vp_1,...,\vp_{N_2}\}^\co$, there exist $\tau_i\in\R^+_0$ for $i=1,...,N_2$ such that $\sum_{i = 1}^{N_2} \tau_i =1$ and $\vq_1 = \sum_{i = 1}^{N_2} \tau_i \vp_i$ due to Proposition~\ref{prop:convhull=convcombis}. Let now $\vp\in\{\vq_1,...,\vq_{N_1} \}^\co$ arbitrary. Similarly there exist $\mu_i\in\R^+_0$ for each $i=1,...,N_1$ such that $\sum_{i = 1}^{N_1} \mu_i =1$ and $\vp = \sum_{i = 1}^{N_1} \mu_i \vq_i$. We simply deduce that
	\begin{align*}
		1&= \mu_1 + \sum_{i = 2}^{N_1} \mu_i = \sum_{i = 1}^{N_2} \mu_1\tau_i + \sum_{i = 2}^{N_1} \mu_i \qquad \text{ and } \\
		\vp &= \mu_1\vq_1 + \sum_{i = 2}^{N_1} \mu_i \vq_i =  \sum_{i = 1}^{N_2} \mu_1\tau_i \vp_i + \sum_{i = 2}^{N_1} \mu_i \vq_i\ec
	\end{align*}
	which shows -- again using Proposition~\ref{prop:convhull=convcombis} -- that $\vp\in\big( \{\vp_1,...,\vp_{N_2}\} \cup \{\vq_2,...,\vq_{N_1} \} \big)^\co$.
\end{proof}

\begin{proof}[Proof of Proposition~\ref{prop:caratheodory-inner-points}]
	Since $\interior{(K^\co)}$ is open, there exists a small $M$-dimensional open cube with center $\vp$ and which lies in $\interior{(K^\co)}$. Let $\vq_1,...,\vq_{2^M}\in K^\co$ be the corners of such a cube. Then $\vp\in \interior{(\{\vq_1,...,\vq_{2^M}\}^\co)}$.
	
	Because each of the points $\vq_i$ lies in $K^\co$, there exist $N_i\in \N$ and $\vp_{i,1},...,\vp_{i,N_i}\in K$ such that $\vq_i \in \{\vp_{i,1},...,\vp_{i,N_i}\}^\co$ accoring to Proposition~\ref{prop:convhull=convcombis}. Lemma~\ref{lemma:conv-polytopes} says that 
	$$
		\{\vq_1,...,\vq_{2^M} \}^\co \subset \bigg(\bigcup_{i=1,...,2^M} \{\vp_{i,1},...,\vp_{i,N_i}\} \bigg)^\co
	$$
	and hence 
	$$
		\vp\in \interior{(\{\vq_1,...,\vq_{2^M}\}^\co)} \subset \interior{\Bigg(\bigg(\bigcup_{i=1,...,2^M} \{\vp_{i,1},...,\vp_{i,N_i}\} \bigg)^\co\Bigg)} \ed
	$$
\end{proof}

\subsection{Convex Functions}

\begin{defn}
	Let $S\subset\R^M$ be a convex subset. A function $f:S\to \R$ is called \emph{convex} if
	$$
		f\big( \tau\vp + (1-\tau) \vq\big) \leq \tau f(\vp) + (1-\tau) f(\vq)
	$$
	for all $\vp,\vq\in\R^M$ and all $\tau\in[0,1]$.
\end{defn}

\begin{lemma} \label{lemma:convexity-kin-en}
	Let $N\in\N$. The function $\R^+ \times \R^N \to \R$, $(a,\vb)\mapsto \frac{|\vb|^2}{a}$ is convex.
\end{lemma}

\begin{proof}	
We have to show that for all $(a,\vb),(c,\vd)\in\R^+\times \R^N$ and all $\tau\in[0,1]$ it holds that 
\begin{equation} \label{eq:temp-not-lemmas-1}
	\frac{\big|\tau \vb + (1-\tau) \vd\big|^2}{\tau a + (1-\tau) c}  \leq \tau \frac{|\vb|^2}{a} + (1-\tau) \frac{|\vd|^2}{c} \ed
\end{equation}

Obviously we have
\begin{align*}
	0&\leq \tau(1-\tau)\big|a\vd-c\vb\big|^2 \\
	&=\tau(1-\tau)(a^2|\vd|^2 + c^2|\vb|^2) - 2 \tau (1-\tau) ac \vb\cdot \vd\ec
\end{align*}	
which is equivalent to 
\begin{equation*}
 	2 \tau (1-\tau) ac \vb\cdot \vd \leq \tau(1-\tau)(a^2|\vd|^2 + c^2|\vb|^2)\ed
\end{equation*}	
Adding $\tau^2 ac |\vb|^2 + (1-\tau)^2 ac |\vd|^2 $ on both sides we obtain
\begin{align*}
	&\tau^2 ac |\vb|^2 + (1-\tau)^2 ac |\vd|^2 + 2 \tau (1-\tau) ac \vb\cdot \vd \\
	&\leq \tau^2 ac |\vb|^2 + (1-\tau)^2 ac |\vd|^2 + \tau(1-\tau)(a^2|\vd|^2 + c^2|\vb|^2)\ed
\end{align*}	
This yields
\begin{equation*}
	ac \big|\tau  \vb + (1-\tau)  \vd\big|^2\leq \big(\tau a + (1-\tau)c\big)\big(\tau c |\vb|^2 + (1-\tau) a |\vd|^2\big) \ed
\end{equation*}	
Dividing by the positive expression $ac(\tau a + (1-\tau)c)$ leads to \eqref{eq:temp-not-lemmas-1}.
\end{proof}

\section{Semi-Continuity} \label{sec:not-lsc} 

\begin{defn}
	Let $(X,d)$ be a metric space. A map $f:X\to \R$ is called \emph{lower semi-continuous with respect to $d$} if for all $x\in X$ and all sequences $(x_k)_{k\in\N}\subset X$ with $x_k \mathop{\to}\limits^d x$ the inequality
	$$
	f(x) \leq \liminf\limits_{k\to \infty} f(x_k)
	$$
	holds.
\end{defn}

\section{Weak-$\ast$ Convergence in $L^\infty$} \label{sec:not-weakconv}

In this book we deal with weak-$\ast$ convergence in $L^\infty$. 

\begin{defn} \label{defn:not-weak-ast}
	Let $\Gamma\subset\R^{1+n}$. A sequence $(f_k)_{k\in \N}\subset L^\infty(\Gamma)$ converges \emph{weakly-$\ast$} to $f\in L^\infty(\Gamma)$ if 
	$$
		\iint_{\Gamma} f_k(t,\vx) \varphi(t,\vx) \dx\dt \ \to\ \iint_{\Gamma} f(t,\vx) \varphi(t,\vx) \dx\dt
	$$ 
	for all $\varphi\in L^1(\Gamma)$ as $k\to \infty$. In this case we write $f_k \mathop{\rightharpoonup}\limits^\ast f$ as $k\to \infty$. 
	
	A sequence of vector-valued functions $(\vf_k)_{k\in \N}\subset L^\infty(\Gamma;\R^M)$ converges \emph{weakly-$\ast$} to $\vf\in L^\infty(\Gamma;\R^M)$ if each component $([\vf_k]_i)_{k\in \N}$ converges weakly-$\ast$ to $[\vf]_i$ ($i=1,...,N$). 
\end{defn}

The topology which corresponds to the weak-$\ast$ convergence is called \emph{weak-$\ast$ topology}. For more details we refer to the textbooks by \name{Megginson} \cite[Section 2.6]{Megginson}, \name{Conway} \cite[Chapter V]{Conway} or \name{Alt} \cite[Chapter 6]{Alt}. We say that a topological property (like boundedness, closedness, compactness, etc.) holds \emph{weakly-$\ast$} if it holds with respect to the weak-$\ast$ topology. The topology which is induced by the norm $\|\cdot\|_{L^\infty}$ is also called \emph{strong topology}.

We need the following facts about the weak-$\ast$ topology:
\begin{itemize}
	\item A subset of $L^\infty$ is bounded (with respect to the strong topology) if and only if it is weakly-$\ast$ bounded, see \cite[Theorem 2.6.7]{Megginson}.
	\item The weak-$\ast$ closure of a weakly-$\ast$ bounded subset of $L^\infty$ is weakly-$\ast$ bounded, see \cite[Theorem 2.2.9 (i)]{Megginson}. Due to the fact above, we even have: The weak-$\ast$ closure of a bounded subset of $L^\infty$ is bounded.
	\item A bounded and weakly-$\ast$ closed subset of $L^\infty$ is weakly-$\ast$ compact, see \cite[Corollary 2.6.19]{Megginson}.
	\item Let $A\subset L^\infty$ be bounded. Then the relative weak-$\ast$ topology oa $A$ is metrizable, i.e. there exists a metric on $A$ which induces the relative weak-$\ast$ topology, see \cite[Corollary 2.6.20]{Megginson}.
\end{itemize}

It is well-known that periodic functions converge weakly-$\ast$ to their mean as the frequency grows, see e.g. \cite[U6.7]{Alt}. A version of this fact, which is tailored to our purposes, is the content of the following lemma.

\begin{lemma} \label{lemma:not-periodic-weak-convergence}  
	Let $f\in C^\infty(\R)$ be 1-periodic with zero mean. Let furthermore $\Phi\in \Cc(\R^{1+n})$, $\veta\in\R^{1+n}\setminus\{\vz\}$ and $C\in\R$. We define $f_k\in \Cc(\R^{1+n})$ by $f_k(t,\vx):= f\big(k(t,\vx)\cdot \veta\big)$. Then 
	$$
		C\Phi f_k \ \mathop{\rightharpoonup}\limits^{\ast} \ 0 \qquad \text{ in }L^\infty(\R^{1+n}) \text{ as }k\to \infty\ed
	$$
\end{lemma}

\begin{proof}
	We have to show that 
	$$
		\iint_{\R^{1+n}} C\Phi(t,\vx) f\big(k(t,\vx)\cdot \veta\big) \varphi(t,\vx) \dx\dt \ \to\ 0
	$$
	for all $\varphi\in L^1(\R^{1+n})$ as $k\to \infty$. Note that for each $\varphi\in L^1(\R^{1+n})$, the product $C\Phi\varphi$ lies in $L^1(\R^{1+n})$. Hence it remains to show 
	\begin{equation} \label{eq:8-temp-not}
		\lim_{k\to \infty} \iint_{\R^{1+n}} f\big(k(t,\vx)\cdot \veta\big) \varphi(t,\vx) \dx\dt = 0
	\end{equation}
	for all $\varphi\in L^1(\R^{1+n})$. Let us consider a $(1+n)$-dimensional rectangle $Q\subset \R^{1+n}$ with the property that one edge is parallel to $\veta$. Then $Q$ can be rotated such that this edge is afterwards parallel to $\ve_t$. Moreover there exist $a_t,a_1,...,a_n,b_t,b_1,...,b_n\in \R$ with $a_i < b_i$ for all $i=t,1,...,n$ such that 
	\begin{align}
		\iint_{Q} f\big(k(t,\vx)\cdot \veta\big) \dx\dt &= \int_{a_t}^{b_t}  \int_{a_1}^{b_1} \cdots  \int_{a_n}^{b_n} f\big(kt |\veta|\big) \dxcomp_n\cdots \dxcomp_1\dt \notag \\
		&= \prod_{i=1}^n (b_i-a_i) \int_{a_t}^{b_t}  f\big(kt |\veta|\big) \dt \notag \\
		&= \frac{\prod_{i=1}^n (b_i-a_i)}{k |\veta|} \int_{k|\veta| a_t}^{k|\veta| b_t}  f(t) \dt \ed \label{eq:7-temp-not}
	\end{align}
	Due to Lemma~\ref{lemma:not-periodic-primitive} there exists a bounded primitive $h$ of $f$. Using this in \eqref{eq:7-temp-not} we obtain 
	\begin{align*}
		\left|\iint_{Q} f\big(k(t,\vx)\cdot \veta\big) \dx\dt\right| &= \frac{\prod_{i=1}^n (b_i-a_i)}{k |\veta|} \Big|h\big(k|\veta| b_t\big) - h\big(k|\veta| a_t\big)\Big|  \ \to \ 0
	\end{align*}
	as $k\to \infty$. Hence \eqref{eq:8-temp-not} holds for all piece-wise constant $\varphi$ which are constant in rectangles as $Q$ above. Note that the set of such functions is dense in $L^1(\R^{1+n})$. 
	
	Obviously $\|f_k\|_{L^\infty} = \|f\|_{L^\infty}$ for all $k\in\N$. In other words $\|f_k\|_{L^\infty}$ is bounded. This together with the fact that \eqref{eq:8-temp-not} holds for all $\varphi$ in a dense subset of $L^1$, implies the claim, see \cite[Exercise 2.71]{Megginson} or \cite[U3.4]{Alt}. 
\end{proof}

\begin{lemma} \label{lemma:not-wls-convex} 
	Let $S\subset \R^M$ be a convex subset and $f:S\to \R$ a convex and continuous function. Let furthermore $\Gamma\subset\R^{1+n}$ open and bounded, and $(\vu_k)_{k\in\N}\subset L^\infty(\Gamma;S)$ be a sequence which converges weakly-$\ast$ in $L^\infty$ to $\vu\in L^\infty(\Gamma;S)$. Then 
	$$
		\iint_\Gamma f(\vu(t,\vx)) \dx\dt \leq \liminf\limits_{k\to\infty} \iint_\Gamma f(\vu_k(t,\vx)) \dx\dt \ed
	$$
\end{lemma} 

In other words the functional $\vu\mapsto \iint_\Gamma f(\vu) \dx\dt$ is lower semi-continuous with respect to the weak-$\ast$ topology. The statement of Lemma~\ref{lemma:not-wls-convex} can be found in \name{Tartar} \cite[Theorem 4]{Tartar79}. Since \name{Tartar} does not provide a proof, we give a detailed proof here.


\begin{proof}
	First of all note that the sequence $(\vu_k)_{k\in\N}$ is weakly-$\ast$ bounded since it converges. Hence it is strongly bounded, see the facts about the weak-$\ast$ topology above. This shows that $\vu$ and all $\vu_k$, $k\in\N$, take almost everywhere values in a bounded subset of $S$. Now for a given $\ep>0$ we partition this bounded set into finitely many subsets $S_1,...,S_N\subset S$ with the following property: For each $i=1,...,N$ there exist $\va_i\in \R^M$ and $b_i\in \R$ such that\footnote{Mind the small but crucial difference in the following two equations: \eqref{eq:1-temp-convex} holds only for $\vy\in S_i$, whereas \eqref{eq:2-temp-convex} is valid for all $\vy\in S = \bigcup_{i=1}^N S_i$.} 
	\begin{align}
		f(\vy) - \ep \leq \va_i \cdot \vy + b_i & & &\hspace{-2cm}\text{ for all }\vy\in S_i \text{ and } \label{eq:1-temp-convex} \\
		\va_i \cdot \vy + b_i &\leq f(\vy) & &\hspace{-2cm}\text{ for all }\vy\in S \ed \label{eq:2-temp-convex}
	\end{align}
	This is possible since $f$ is convex and continuous\footnote{Intuitively this seems to be clear. See \name{Ekeland}-\name{T{\'e}mam} \cite[Proposition 3.1 of Chapter 1]{EkeTem} for a rigorous proof.}.
	
	Now we partition $\Gamma$ into $N$ measurable subsets $\Gamma_1,...,\Gamma_N\subset \Gamma$ such that for all $i=1,...,N$ we have $\vu(t,\vx)\in S_i$ for a.e. $(t,\vx)\in\Gamma_i$. Thus \eqref{eq:1-temp-convex} implies 
	$$ 
		f(\vu(t,\vx)) - \ep \leq \va_i \cdot \vu(t,\vx) + b_i \qquad\text{ for a.e. } (t,\vx)\in \Gamma_i \ed
	$$
	Integrating over $\Gamma_i$ yields
	\begin{equation} \label{eq:3-temp-convex}
		\iint_{\Gamma_i} \Big( f(\vu(t,\vx)) - \ep \Big) \dx\dt \leq \iint_{\Gamma_i} \Big(\va_i \cdot \vu(t,\vx) + b_i \Big) \dx\dt\ed
	\end{equation}
	
	Moreover we obtain from \eqref{eq:2-temp-convex} that 
	$$ 
		\va_i \cdot \vu_k(t,\vx) + b_i \leq f(\vu_k(t,\vx)) \qquad\text{ for a.e. } (t,\vx)\in \Gamma \ed
	$$
	In particular, this holds a.e. on $\Gamma_i$ and hence integrating yields
	\begin{equation} \label{eq:4-temp-convex}
		\iint_{\Gamma_i} \Big( \va_i \cdot \vu_k(t,\vx) + b_i \Big) \dx\dt \leq \iint_{\Gamma_i} f(\vu_k(t,\vx)) \dx\dt  \ed
	\end{equation}
	
	The fact that $\vu_k\mathop{\rightharpoonup}\limits^\ast \vu$ implies
	\begin{equation} \label{eq:5-temp-convex}
		\lim\limits_{k\to\infty} \iint_{\Gamma_i} \Big( \va_i \cdot \vu_k(t,\vx) + b_i \Big) \dx\dt  = \iint_{\Gamma_i} \Big( \va_i \cdot \vu(t,\vx) + b_i \Big) \dx\dt \ed 
	\end{equation}
	
	Combining \eqref{eq:3-temp-convex}, \eqref{eq:4-temp-convex} and \eqref{eq:5-temp-convex}, we obtain 
	$$
		\iint_{\Gamma_i} \Big( f(\vu(t,\vx)) - \ep \Big) \dx\dt \leq \liminf\limits_{k\to\infty}\iint_{\Gamma_i} f(\vu_k(t,\vx)) \dx\dt \ed
	$$
	Finally summing over $i=1,...,N$ we find 
	$$
		\iint_{\Gamma} \Big( f(\vu(t,\vx)) - \ep \Big) \dx\dt \leq \liminf\limits_{k\to\infty}\iint_{\Gamma} f(\vu_k(t,\vx)) \dx\dt \ec
	$$
	which implies the claim, since $\ep>0$ was arbitrary.
\end{proof}

\section{Baire Category Theorem} \label{sec:not-baire} 

In this section we state a version of the famous \emph{Baire Category Theorem} taylored to our purposes. What is presented here can be found in several textbooks on topology, e.g. the one by \name{Waldmann} \cite[Chapter 7]{Waldmann}. Let us start with some definitions.

\begin{defn}
	Let $(X,\mathcal{T})$ be a topological space. A subset $M\subset X$ is called
	\begin{itemize}
		\item \emph{nowhere dense} if the interior of the closure is empty, i.e. $\interior{(\closure{M})}= \emptyset$,
		\item \emph{meager} if $M$ is the countable union of nowhere dense sets,
		\item \emph{residual} if the complement of $M$ is meager. 
	\end{itemize}
\end{defn}

The following is a simple observation. For a detailed proof we refer to \cite[Proposition~7.1.3~(iv)]{Waldmann}.

\begin{prop} \label{prop:intersection-residual}
	The intersection of countably many residual subsets of a topological space is residual.
\end{prop} 

We need Baire's Theorem in the following version. For the proof we refer to \cite[Theorem~7.2.1 and Proposition~7.1.5~(iv)]{Waldmann}.

\begin{prop}[Baire Category Theorem] \label{prop:Baire}
	Let $(X,d)$ be a complete metric space. Then every residual subset of $X$ is dense. 
\end{prop}

We finish this section with a corollary of some standard facts taylored to our purposes. 

\begin{prop} \label{prop:lsc-residual} 
	Let $(X,d)$ be a complete metric space and $f:X\to \R$ lower semi-continuous and taking values in a bounded interval of $\R$. Then the points of continuity of $f$ form a residual set in $X$. 
\end{prop}

\begin{proof}
	Under the assumption of the proposition, $f$ can be written as a pointwise supremum of an increasing sequence of continuous functions\footnote{In this case one also says that $f$ is a \emph{Baire-1 map}.} in $X$, see \name{Bourbaki} \cite[Proposition 11 in Section 2.7 of Chapter IX]{Bourbaki:Top2}\footnote{To be precise \name{Bourbaki} shows this under the assumption that $f\geq 0$. Instead of this assumption, we require that $f$ takes values in a bounded interval of $\R$. By adding a suitable constant to $f$, we can reduce our case to the one considered by \name{Bourbaki}.}. This implies the claim according to \cite[Proposition 7.3.2]{Waldmann}. 
\end{proof}